\documentclass[openany,USenglish]{book}

\input{header}
\author{Clemens~Kirisits, Bochra~Mejri, Sergei~Pereverzev, Otmar~Scherzer, Cong~Shi}
\title{Regularization Methods for Nonlinear Inverse Problems}
\subtitle{From Functional Analysis to Machine Learning Approaches}
\makeindex[intoc]

\begin{document}
\frontmatter
\maketitle
%\dedication{...\\ ...}
%\preface
\chapter*{Preface}
The focus of this book is on the analysis of regularization methods for solving \emph{nonlinear inverse problems}. Specifically, we place a strong emphasis on techniques that incorporate supervised or unsupervised data derived from prior experiments. This approach enables the integration of data-driven insights into the solution of inverse problems governed by physical models.

\emph{Inverse problems}, in general, aim to uncover the \emph{inner mechanisms} of an observed system based on indirect or incomplete measurements. This field has far reaching applications across various disciplines, such as medical or geophysical imaging, as well as, more broadly speaking, industrial processes where identifying hidden parameters is essential.

Solving inverse problems computationally is a highly challenging task.
The core difficulty lies in their inherent instability, which must be addressed using so called \emph{regularization methods}. These methods were first developed and analyzed by Andrey Nikolayevich Tikhonov. As fundamental as Tikhonov's work was, its results were primarily \emph{qualitative}. Specifically, regularization methods can exhibit arbitrarily slow convergence, and there was no objective framework to compare the efficiency of different methods. It was not until 1989 that the first paper providing quantitative error estimates for regularization methods for nonlinear ill-posed problems appeared: To this end Heinz Engl, Karl Kunisch and Andreas Neubauer introduced source conditions, which are conceptually similar to classical smoothness conditions in numerical analysis. Their work, being grounded in functional analysis, demonstrated how to overcome the issue of ``arbitrary slowness'', enabling a \emph{quantitative} evaluation of regularization methods in infinite dimensional settings. 

Up to the beginning of the 21st century \emph{data-driven} and \emph{machine learning methods} played only a minor role in the study of inverse problems. However, even before 2000, neural network ansatz functions had already been employed in regularization methods. In recent years the use of machine learning techniques has increased dramatically. Despite these developments, a comprehensive quantitative theory for \emph{data-driven} regularization methods applied to nonlinear inverse problems is still not available.

This book makes an attempt to close this gap. While much of the current literature positions machine learning or data-driven methods as competitors to physics-based approaches, our perspective is to integrate and analyze them in a hybrid framework. By leveraging the synergies between these approaches, we seek to enhance numerical reconstructions. Consequently, we focus on the study of
\bigskip\par
\fbox{
	\parbox{0.90\textwidth}{
		\begin{center}
			data-informed, physics-based inverse problems.
\end{center}}}
\bigskip\par
We are very grateful for our funding agencies. Without their support we would not have managed to work on this book:
This research was funded in whole, or in part, by the Austrian Science Fund
(FWF) 10.55776/P34981 -- New Inverse Problems of Super-Resolved Microscopy (NIPSUM),
SFB 10.55776/F68 ``Tomography Across the Scales'', project F6807-N36
(Tomography with Uncertainties), and 10.55776/T1160 ``Photoacoustic Tomography: Analysis and Numerics''. For open access purposes, the author has applied a CC BY public copyright license to any author-accepted manuscript version arising from this submission.
% CD-LAB MaMSi
The financial support by the Austrian Federal Ministry for Digital and Economic
Affairs, the National Foundation for Research, Technology and Development and the Christian Doppler
Research Association is gratefully acknowledged.

\bigskip\par
Our gratitude goes to all our collaborators:
Each of us worked with numerous coauthors, who have significantly shaped our understanding of the field. Special thanks are due to them. In particular, Otmar Scherzer is very grateful to Heinz Engl and Andreas Neubauer for their supervision on quantitative (convergence rates) results of regularization methods, which was an unexplored area for nonlinear inverse problems at the time.

\hspace*{\fill} Clemens Kirisits, Bochra Mejri, Sergei Pereverzev, Otmar Scherzer and Cong Shi\\
\hspace*{\fill}
Vienna and Linz, June 2026

\mainmatter

\tableofcontents

\chapter{Introduction}\label{cha:intro}
Throughout this book we are concerned with inverse problems consisting in the solution of an operator equation
\begin{equation} \label{eq:op}
	F[\x] = \y\,.
\end{equation}
Our primary focus is on the classical setting, treated, for instance, in \cite{EngHanNeu96,EngKunNeu89}, where $F : \X \to \Y$ may be nonlinear, and $\X$ and $\Y$ are Hilbert-spaces. These spaces are typically infinite-dimensional and consist of functions of several real variables. To address practical applications, the inverse problem is discretized, resulting in a finite-dimensional approximation.
For the solution to be physically meaningful, the discrete problem must provide an accurate approximation of the solution to the original equation (1.1). Ensuring this approximation is one of the central questions explored in this book.

A prototypical inverse problem is the inversion of the \emph{Radon-transform}. It is the mathematical basis of \emph{computerized tomography (CT)}, where one reconstructs material properties of the interior of the body, represented by a function $\x$, from measurements $\y$ of penetrating $X$-rays. See \autoref{sec:radon}, where we study this problem in greater detail, or \cite{BarSwi81,Eps07} for more details on the image formation process in CT.

Inverse problems frequently are \emph{ill-posed} in the sense of Hadamard \cite{Had07}\index{Problem!ill-posed}, meaning that at least one of the following properties is violated:
\begin{enumerate}
	\item For every $\y \in \Y$ there exists a solution of \autoref{eq:op},
	\item the solution is unique and
	\item it depends continuously on the data. \label{it:stab}
\end{enumerate}

For a long time ill-posed inverse problems were overlooked in the literature because they were considered physically irrelevant.
It was believed that the ill-posedness can always be ``cured'' with an adequate measurement setup.
While this might be true in certain situations, it is widely recognized nowadays that most practical and industrial inverse problems, for instance, in medical imaging \cite{FauSch23,KakSla01,Kuc13} and geophysical prospecting \cite{TarVal82,Tar05,OpeMinBroComMet15,VirOpe09}, are, in fact, ill-posed. See, for instance, \cite{Sch15} or \cite{Wan15,Wes03} for more applied texts. 

When solving inverse problems in practice, a lack of continuous dependence on the data, \autoref{it:stab}, can be particularly challenging. Methods which are able to deal with such instabilities are called \emph{regularization methods}\index{regularization}.
Aside from \emph{special purpose} methods such as the back projection\index{back projection} algorithms widely used in tomography (see \autoref{sec:radon} or \cite{Kuc13,Nat01,Qui06} as well as the references given in \cite{Sch11}), a number of \emph{multi-purpose} regularization methods for solving ill-posed problems have been developed and analyzed.

Arguably the best-known method for solving ill-posed problems is \emph{Tikhonov-regularization}\index{Tikhonov-regularization}.
It consists in approximating a solution of \autoref{eq:op} by a
minimizer $\xad$ of the functional
\begin{equation} \label{eq:Tik}	\boxed{
	\T_{\alpha,\y^\delta}[\x]:= \| F[\x]-\y^\delta \|_\Y^2+\alpha \| \x-\x^0 \|_\X^2,}
\end{equation}
where $\x^0\in\X$ typically unifies all available a priori information on the solution, $\y^\delta \in \Y$ are the measured noisy data and $\alpha$ is a positive parameter.
The functional
\begin{equation*}
	\x \mapsto \| \x-\x^0 \|_\X^2
\end{equation*}
is used to enforce stability of the computed minimizer and is therefore called \emph{regularization functional}\index{functional!regularization}. Both the \emph{fit to data} functional
\begin{equation} \label{eq:op_min}
	\x\mapsto \| F[\x]-\y^\delta \|_\Y^2
\end{equation}
and the regularization functional can be considered in a more general setting: Banach-space norms or semi-norms but also the Kullback-Leibler-divergence and related statistical quantities are common alternatives to Hilbert-space norms. See \cite{AmaHug91b,Egg93,EngLan93,Fri72,FriBur72,KulLei51,Luk88,ResAnd07,SchGraGroHalLen09,SchuKalHofKaz12} to mention but a few.

This book concentrates on the simpler Hilbert-space setting to introduce the basic concepts and ideas of data-driven regularization. We note that this also simplifies certain considerations in machine learning: For instance, in his seminal paper \cite{Wie32} on Tauber-Wiener theorems Norbert Wiener essentially proved a version of the \emph{universal approximation theorem}, later popularized by George Cybenko \cite{Cyb89}, for $L^1$ and $L^2$ functions on the real line. By contrast, George Cybenko formulated the theorem in the framework of Radon-measures. We also present $L^p$-formulations of the universal approximation theorem and link it with the Tauber-Wiener theorem (see \autoref{ex:lpdiscriminatory}). Likewise, the classical Paley-Wiener theorem (see \autoref{th:PaleyWiener}) was formulated in \cite{PalWie34} in an $L^2$ setting before being generalized to distributions by Laurent Schwartz. Again, the arguments are significantly easier in the Hilbert-space setting.

The regularization parameter $\alpha$ determines the trade off between stability and data fidelity of the regularization method. It is typically correlated with the noise level $\delta$ satisfying 
\begin{equation} \label{eq:datn}
	\| \y^\delta-\y\|_\Y \leq \delta\,.
\end{equation}
As a rule of thumb one might say that a high noise level requires a large regularization parameter. Consequently, in such a situation regularization is prioritized over data fidelity.

Starting from the fundamental work of Andrei Tikhonov \cite{Tik43,Tik63b,Tik63} a relatively complete regularization theory in
\emph{finite-dimensional} (see \cite{Han10}) and \emph{infinite-dimensional} Hilbert-spaces (see \cite{EngHanNeu96}) has been developed. Extensions to Banach-spaces are also well-known by now \cite{SchGraGroHalLen09,SchuKalHofKaz12}. Moreover, the relation between learning and finite dimensional Tikhonov-regularization has been mainstreamed in \cite{GooBenCou16}.

\emph{Iterative regularization methods}\index{algorithms!iterative} are a popular alternative to Tikhonov-regularization. They approximately solve \autoref{eq:op} by terminating an iteration of the general form
\begin{equation} \label{eq:iterative} \boxed{
	\x_{k+1}^\delta = \x_k^\delta - G[\x_k^\delta,\x_{k-1}^\delta,\ldots,\x_0;\y^\delta]}\,,
\end{equation}
at a certain iteration index $k_*$. A common choice for $G$ is
\begin{equation} \label{eq:steepest}
	 G[\x_k^\delta,\x_{k-1}^\delta,\ldots,\x_0;\y^\delta] = \mu_k^\delta F'[\x_k^\delta]^* \left( F[\x_k^\delta]-\y^\delta \right),
\end{equation}
where $\opd{\x}^*$ is the adjoint of the Fr\'echet-derivative $\opd{\x}$. The resulting iterations are \emph{gradient descent algorithms}. Among these we highlight the \emph{Landweber-method}\index{Landweber-method} where $\mu_k^\delta=1$ for all $k \in \N_0$. That is, the Landweber-iteration is a gradient descent algorithm without step size control.

Textbooks on numerical analysis commonly prove convergence of iterations like \eqref{eq:iterative} for well-posed problems. 
For ill-posed problems, however, the iteration must be \emph{stopped early}\index{early stopping} in order to overcome the instability of the operator equation. Note that this is analogous to having a strictly positive regularization parameter in $\alpha$ in \autoref{eq:Tik}. 
In contrast to the well-posed case one proves that, as $\delta \to 0$, it is possible to choose $k_* = k_*(\delta)$ (or $\alpha = \alpha(\delta)$) such that $\x_{k_*(\delta)}^\delta$ (or $\x^\delta_{\alpha(\delta)}$) converges to a solution of \autoref{eq:op}.

\bigskip
\fbox{
	\parbox{0.90\textwidth}{
        	It is a fundamental observation of regularization theory that without additional prior information on the solution of \autoref{eq:op} there \emph{cannot} exist \emph{quantitative} error estimates for the regularized solution. In fact the approximation can be \emph{arbitrarily poor} (see for instance \cite{Gro84}).
	}}
\bigskip\par
    This dilemma can be resolved in two ways:
    \begin{enumerate}
    	\item One option is to impose so-called \emph{source conditions} (see \autoref{eq:source_var} below), which are similar to standard closeness conditions in numerical analysis. 
    	In fact the application of regression provides examples on the equivalence of smoothness and source conditions \cite{Han02,HanSch01}.
\end{enumerate}

\bigskip
\fbox{
	\parbox{0.90\textwidth}{
    		Source conditions for \emph{nonlinear} inverse problems were first used by Heinz Engl, Karl Kunisch and Andreas Neubauer \cite{EngKunNeu89} in order to obtain \emph{quantitative} error estimates. These allow to objectively compare the efficiency of regularization methods, which is particularly relevant for evaluating data-driven regularization methods.
    	}}
    	
\begin{enumerate} \setcounter{enumi}{1}
    	\item Alternatively, in certain situations, it may be sufficient to compute features $\z$ derived from $\x$. In such cases $\z$ can sometimes be more accurately approximated than the solution of \autoref{eq:op} itself. Robert Anderssen was the first to consider the \emph{linear functional strategy} \cite{And86}. This aspect is also discussed in \autoref{ch:Aspri paper}.
    \end{enumerate}
    
With the rise of \emph{machine learning} new research directions have emerged in the field of regularization theory concerning the development of algorithms making use of data obtained from prior experience:
	\paragraph{Supervised parameter tuning of algorithms.} In some earlier studies (see \cite{HanNeuSch95,Sch98}), the step sizes $\mu_k^\delta$ of iterative regularization methods were manually adjusted to accelerate convergence. In contrast, modern data-driven algorithms automatically determine the optimal step size parameters based on the given descent directions or, more generally, the search directions $G[\x_k^\delta,\x_{k-1}^\delta,\ldots,\x_0;\y^\delta]$ by using \emph{labeled samples}\index{labeled samples} (see \cite[Sec.~4.9.1]{ArrMaaOktScho19}):
    \begin{equation} \label{eq:trainingsamplesy}
	    \mathcal{T}_\no:=\set{(\x^{(0)},\z^{(0)})} \cup \set{(\x^{(\ell)},\z^{(\ell)}): \ell = 1,\ldots,\no}
    \end{equation}
    or
    \begin{equation} \label{eq:trainingsamplesx}
    	\mathcal{T}_\no:=\set{(\z^{(0)},\y^{(0)})} \cup \set{(\z^{(\ell)},\y^{(\ell)}): \ell = 1,\ldots,\no}\;,
    \end{equation}
    where each pair $(\x^{(\ell)},\y^{(\ell)})$ is a solution pair of \autoref{eq:op} and $\z^{(\ell)}$ is a feature, for example, a classifier of some sort, of $\y^{(\ell)}$ or $\x^{(\ell)}$, respectively. We emphasize that $(\x^{(0)},\y^{(0)})$ is a distinctive pair of expert information, which is given particular importance. Additionally, we consider triples and quadruples of labeled samples (see \autoref{eq:sampled1}-\autoref{eq:sampled3} below).

    Similar strategies have been devised for the optimal selection of parameters, such as $\alpha$, as well as for determining optimal regularization functionals in Tikhonov-regularization (see for instance \cite{AfkChuChu21,ArrMaaOktScho19,ChePerXu15,ChiVitMolRosVil24,VitForNau22} and \autoref{ch:par}).

    \paragraph{Surrogate modeling.} In recent years the need for accurate physical models of the forward operator $F$ has been questioned. Instead, it has been suggested to replace $F$ by a surrogate operator $F_S$ derived from \emph{expert information} \index{expert information} consisting of selected training samples
    \begin{equation} \label{eq:expert_information}
    	\mathcal{S}_\no:=\set{(\x^{(0)},\y^{(0)}=\op{\x^{(0)}})} \cup \set{(\x^{(\ell)},\y^{(\ell)}=\op{\x^{(\ell)}}): \ell = 1,\ldots,\no}.
    \end{equation}
    This approach has been studied, for instance, in the context of learning linear operators and matrices \cite{AspFriKorSch21,AspKorSch20,MolMucSul22_report,YaXinXueDac12}, holomorphic operator learning \cite{KovLanMis21,LanLiStu23_report,LanMisKar22,LanStu23_report} and vector-valued regression (see \cite{brogat_2022,MeuSheMolGreLi24_report}).

    The main difference between expert information $\mathcal{S}_\no$ and labeled samples $\mathcal{T}_\no$ is that the former is directly related to the operator $F$ while the latter consists of solutions or data coupled with features. 
    The topic of function, functional and operator learning is treated in \autoref{ch:op_learning}.

    \paragraph{Detail representation.} Machine learning techniques, such as \emph{neural network functions}, have been used as ansatz functions for the elements of the spaces $\X$ and $\Y$, respectively, in order to better represent essential features of the solutions to be reconstructed (see \cite{AntHal21,JinMccFroUns17,LiSchwAntHal20,SchwAntHal19}). We discuss neural network functions in \autoref{ch:nn}.
    
    \paragraph{Unsupervised information} is used to increase the efficiency of regularization algorithms by using reconstructions from prior experiments
    \begin{equation} \label{eq:expert_information_discriminative}
    	\mathcal{U}:=\set{\x^{(0)}} \cup \set{\x^{(\ell)}: \ell = 1,\ldots,\no},
    \end{equation}
    but not the data from which these reconstructions were obtained. Such an approach is effective when the elements of $\mathcal{U}$ exhibit a high degree of similarity to the desired solution.

\bigskip
\fbox{
	\parbox{0.90\textwidth}{
	One objective of this book is to systematically analyze data-driven methods for solving nonlinear inverse problems using \emph{functional analytic techniques} and to make \emph{quantitative} comparisons with classical regularization methods.
	}
}

\section[The impact of data-driven techniques]{An illustration of the impact of data-driven techniques on regularization}
In \cite{NeuSch90} we investigated finite-dimensional approximations of Tikhonov-regularization, where the functional
\begin{equation} \label{eq:Tikdis}
	\boxed{ 
	\T_{\alpha,\y^\delta}^{{\tt n}}[\x]:=\norms{F_{\tt n}[\x]-\y^\delta}_\Y^2+\alpha\norms{\x-\x^0}_\X^2}
\end{equation}
is minimized over a finite-dimensional subspace $\X_{\tt m} \subset \X$ providing a minimizer $\x^{\alpha,\delta}_{{\tt m},{\tt n}}$. The operator $F_{\tt n}$ in \eqref{eq:Tikdis} is an approximation to $F$.
%\commentC{Use $\mathcal{T}_{\alpha,\y^\delta}^{{\tt n}}$ instead of $\mathcal{T}_{\alpha,\y^\delta}^{{\tt m},{\tt n}}$ because the functional is independent of $\X_{\tt m}$?}

Approximating the solution of \autoref{eq:op} using discretization and regularization goes back, at least, to \cite{GroNeu89,Nat77} for linear problems. For nonlinear inverse problems it has been studied in \cite{Neu89,NeuSch90,PoeResSch10}. Neural network ansatz functions have been employed to this end already in \cite{PogGir90} for the problem of denoising and in a more general setting, for example, in \cite{BurEng00,ObmSchwHal21}.

The results of \cite{Neu89,NeuSch90}, which are presented in \autoref{sec:fda}, show how ${\tt m}$, ${\tt n}$ and $\alpha$ have to be chosen in dependence of $\delta$ such that $\x^{\alpha,\delta}_{{\tt m},{\tt n}}$ converges to a solution of \autoref{eq:op}. The bottom line is that in typical applications the approximation space $\X_{\tt m}$ can be relatively coarse and the approximation $F_{\tt n}$ does not have to be too accurate, which aligns very well with the philosophy of learning surrogate operators mentioned above. We use the following example to demonstrate two ways of constructing an approximating operator.

\begin{example}[$c$-example] \label{ex:c_reconstruct} 
Consider the operator
	\begin{equation} \label{eq:cF}
		\begin{aligned}
			F: \mathcal{D}(F) := \set{\x \in L^2(0,1) : \x \geq 0} & \to L^2(0,1),\\
			\x &\mapsto F[\x] := \y
		\end{aligned}
	\end{equation}
	where $\y :=\y[\x] :[0,1] \to \R$ is the unique solution of
	\begin{equation} \label{eq:cproblem}
		\begin{aligned}
			-\y''(s) + \x(s)\y(s) & = {\tt f}(s) \quad s \in ]0,1[\,, \\
			\y(0) = 0 &= \y(1)\,,
		\end{aligned}
	\end{equation}
	\par\noindent
	for a given ${\tt f} \in L^2(0,1)$. 
%	It is notationally convenient to introduce the operator 
%	\begin{equation}\label{eq:Aop}
%		\begin{aligned}
%			A[\x] : L^2(0,1) &\to L^2(0,1)\,,\\
%			        {\tt f} &\mapsto \y 
%		\end{aligned}
%	\end{equation}
%	which maps the right hand side of the differential equation, \autoref{eq:cproblem} onto the solution (for a given parameter $\x$). The operator $A[\x]$ is invertible for all elements $\y \in V:=W^{2,2}(0,1) \cap W^{1,2}_0(0,1)$.
%	From \emph{Weyl's lemma} (see for instance \cite{Wlo87})\index{Lemma!Weyl} it follows that
%	\begin{equation}\label{eq:Aop2}
%	 c\norm{\tt f}_{L^2} \leq \norm{A[\x]}_{W^{2,2}} \leq C \norm{\tt f}_{L^2}\;.
%	\end{equation}
\end{example}
\paragraph{The finite element approach.}
The solution $\y$ of \autoref{eq:cproblem} is also the unique solution of the weak form
\begin{equation} \label{eq:fe}
	\inner{\y'}{{\tt v}'}_{L^2} + \inner{\x\y}{\tt v}_{L^2} = \inner{\tt f}{\tt v}_{L^2} \quad \text{ for all } {\tt v} \in W_0^{1,2}(0,1)\,.
\end{equation}
See \autoref{eq:sob_hom} for the definition of the Sobolev space $W_0^{1,2}(0,1)$.
Let $\Y_{\tt n}$ be the space of linear splines on a uniform grid of $({\tt n}+1)$ points in $[0,1]$ vanishing at $0$ and $1$.
%, that is 
%\begin{equation} \label{eq:splinespace}
%	\Y_{\tt n} = \spann \set{\Lambda_i : i=1,\ldots,\n -1}, 
%\end{equation}
%where $\Lambda_i$, $i=0,\ldots,\n$ is the linear spline, defined by
%\begin{equation}\label{eq:linearspline}
%	\Lambda_i \left( \frac{j}{\tt n} \right) = \delta_{ij}, \quad j=0,\ldots,\n,
%%	\Lambda_i (s) = \left\{ \begin{array}{rcl} {\tt n} \left( s - \frac{i-1}{\tt n} \right) & \text{ for } & s \in \left[\frac{i-1}{\tt n},\frac{i}{\tt n} \right]\,,\\ {\tt n} \left( \frac{i+1}{\tt n} -s \right) & \text{ for } & s \in \left[\frac{i}{\tt n},\frac{i+1}{\tt n} \right]\;.\end{array} \right. 
%\end{equation}
%where the boundary linear splines $\Lambda_0$ and $\Lambda_{\n}$ are excluded due to the homogeneous Dirichlet boundary conditions.
We approximate $\y$ by the unique solution $\y_\n = \sum_{i=1}^{\n-1} y_i \Lambda_i\in \Y_{\n}$ of
\begin{equation}\label{eq:fen}
	\inner{ \y_{\tt n}'}{{\tt v}'}_{L^2} + \inner{\x\y_{\tt n}}{\tt v}_{L^2} = \inner{\tt f}{\tt v}_{L^2}, \quad \text{ for all } {\tt v} \in \Y_{\tt n}, \;
\end{equation}
and define
\begin{equation}\label{eq:Fm}
	\begin{aligned}
	F_{\tt n}: \dom{\opo_\ttn} := \dom{\opo} &\to L^2(0,1)\;.\\
	\x &\mapsto F_{\tt n}[\x] := \y_{\tt n}
	\end{aligned}
\end{equation}
The approximate forward operators $F_{\tt n}$ satisfy
\begin{equation} \label{eq:m2_estimate}
	\| F[\x] - F_{\tt n}[\x] \|_{L^2} = \| \y - \y_{\tt n} \|_{L^2} \le C (1 + \| \x \|_{L^2}){\tt n}^{-2}.  %\mathcal{O}\left( (1 + \| \x \|_{L^2}){\tt n}^{-2}\right)\;.
\end{equation}
This is a simple special case of a fundamental finite element error estimate, based on C\'ea's lemma, interpolation error bounds and the Aubin-Nitsche trick. See, for instance, \cite{Cia02} and also \cite{ArnLog14,Zla68}. Estimates of this type are indispensable for investigating finite-dimen\-sional approximations of Tikhonov-regularization, cf.\ \autoref{as:fda}.

% (see, for instance, \cite{Cia02} and also \cite{ArnLog14,Zla68}) from which one obtains 
%\begin{equation*}
%	\| F[\x] - F_{\tt n}[\x] \|_{W^{1,2}} = \mathcal{O}\left({\tt n}^{-k}\right)\;,
%\end{equation*}
%if the solution $F[\x]$ of the elliptic partial differential equation belongs to $W^{k+1,2}$ and $F_{\tt n}[\x]$ is the solution of a finite element method of order $k$ on a triangulation with fineness proportional to $1/{\tt n}$. Together with the Aubin-Nitsche theorem (see \cite{Aub67,Nit68}) it follows that 
%	\begin{equation}\label{eq:de_an}
%		    	\| F[\x] - F_{\tt n}[\x] \|_{L^2} = \mathcal{O}\left({\tt n}^{-k-1}\right)\;.
%	\end{equation}
%\autoref{eq:m2_estimate} is a special case of this fundamental finite elements result.
%
%Estimates of this type are indispensable for estimating errors of finite-dimen\-sional approximations of Tikhonov-regularization, cf.\ \autoref{as:fda}. 

\paragraph{The data driven approach} computes an approximation of the operator $F$ from training samples. See \cite{AspKorSch20,AspFriKorSch21,KovLanMis21,LuJinPanZhaKar21,LanLiStu23_report,LanMisKar22,LanStu23_report,HerSchwZec24}, for instance. A central result in this context is the 
\emph{universal approximation theorem} for \emph{operators} from \cite{TiaHon95}\index{operator!universal approximation theorem}. It guarantees the existence of a coefficient vector
\begin{equation} \label{eq:D}
	\Gamma = \begin{pmatrix} \underbrace{{\alpha}_{j,k}}_{\in \R} & \underbrace{{w}_{j,k,l}}_{\in \R} & \underbrace{w_j}_{\in \R} & \underbrace{{\theta}_{j,k}}_{\in \R} & \underbrace{s_l}_{\in [0,1]}
		 & \underbrace{{\theta}_j}_{\in \R} &\end{pmatrix}_{\!\!\! \tiny \begin{array}{c} j=1,\ldots,N_j\\ k=1,\ldots,N_k\\ l=1,\ldots,N_l \end{array}}
\end{equation}
such that
\begin{equation} \label{eq:deeponet}
	\begin{aligned}
		D_\Gamma[\x](t) := \sum_{j=1}^{N_j} \sum_{k=1}^{N_k} \alpha_{j,k} \operatorname{\sigma}\left( \sum_{l=1}^{N_l} w_{j,k,l} \x(s_l)+\theta_{j,k}\right) \sigma (w_j t + \theta_j)		
	\end{aligned}
\end{equation}
approximates $F[\x](t)$ for all $t \in [0,1]$ and all $\x$ in a compact subset of $C[0,1]$ with a prescribed accuracy $\ve > 0$.
In practice $\Gamma$ has to be estimated from a system of nonlinear equations of the form
\begin{equation} \label{eq:learn}
	\begin{aligned}
		\y^{(\ell)}(t_\rho) 
		&= D_\Gamma[\x^{(\ell)}](t_\rho) 
	\end{aligned}
\end{equation}
where $(\x^{(\ell)},\y^{(\ell)})$, $\ell=1,\ldots,\no$, are expert information pairs and $t_\rho \in [0,1]$, $\rho=1,\ldots,Q$, are sampling points. Note that there are
\begin{equation*}
	N_jN_k(1+N_l)+N_l+N_jN_l \approx  N_jN_kN_l
\end{equation*}
coefficients and $\no Q$ equations. In view of universal approximation theorems for functions (see \cite{Cyb89,Hor91}) it seems counter-productive to constrain the amount of parameters. In several applications, however, this can be justified (see, for instance, \cite{TiaHon95} and \autoref{ta:nn}) and is strongly dependent on the choice of the activation function $\sigma$. In practice this means that the training process can be significantly simplified by reducing the variables of $\Gamma$ a-priori.
	
After the training process, that is, after a coefficient vector $\hat \Gamma$ has been determined, the approximate forward operator is given by
\begin{equation} \label{eq:deepoapp}
	F_{\tt n}[\x] := D_{\hat \Gamma} [\x]\;.   
\end{equation}
Based on the rate at which the operator in \eqref{eq:Fm} or \eqref{eq:deepoapp} approximates $F$ and on the noise level $\delta$ (recall \eqref{eq:datn}), the aim is to choose approximation spaces $\X_{\tt m}$ and a rate for $\alpha \to 0$ such that the error between $\x^{\alpha,\delta}_{{\tt m},{\tt n}}$ and a solution of \autoref{eq:op} vanishes in the limit. How this can be achieved is made precise in \autoref{th:NeuSch90} and \autoref{th:NeuSch90b}. The important difference between these two theorems is that in the latter smoothness assumptions are imposed on the solution of \autoref{eq:op} leading to convergence rates.

\bigskip
\fbox{
	\parbox{0.90\textwidth}{
		%\begin{center}
		This book employs the theory of approximate operators and source conditions to find optimal finite-dimensional approximations of regularized solutions. The essential difference between a \emph{finite-element-type} approximation and a neural network based surrogate operator (such as \autoref{eq:deepoapp}) is that for the former we can rely on results like \emph{C\'ea's lemma} to obtain estimates of the form
		\begin{equation}\label{eq:de_cia}
			\| F[\x] - F_{\tt n}[\x] \| = \mathcal{O}\left({\tt n}^{-p}\right),
		\end{equation}
		compare \eqref{eq:m2_estimate}. Here, ${\tt n}$ is inversely proportional to the fineness of the finite element triangulation. It is important to note that for neural network approximations such \emph{Sobolev estimates} are not available in general (see \autoref{ta:nn} and \autoref{tab:chal}). This appears to be one of the biggest obstacles for bridging the gap between neural network theory and regularization theory.
    }
}

\section{A guided tour through the book} \label{se:chall}
This book aims to develop an analysis of \emph{hybrid regularization methods}, which combine classical regularization theory with data-driven techniques (see \autoref{ch:Aspri paper}). The interplay between regularization theory and data-driven approaches is not a one-way interaction but a mutual exchange that drives progress in both areas. Examples of how concepts and methodologies from regularization theory have been applied in various machine learning contexts can be found in \cite{dinu2023addressing,GruHolLehHocZel24_report,zellinger2025binary,zellinger2021the}.

The focus of this book is on concepts rather than generality. Therefore we concentrate on simpler Hilbert-space settings. Generalization to more general function space settings (such as Banach-spaces) are conceptually analogous although technical difficulties can be anticipated. The relevant Hilbert-space regularization theory, which this book heavily builds on, has been developed in \cite{EngKunNeu89,EngHanNeu96} for variational and in \cite{Bak92,DeuEngSch98,HanNeuSch95,KalNeuSch08} for iterative regularization methods. What sets the cited work apart from much of the existing research on regularization is its focus on \emph{convergence rate results}. Building on these papers, convergence rate results have since been generalized in  \cite{Fle11,FleHof10,FleHof11,FleHofMat11,Hof06,HofKalPoeSch07,HofMat07,Hoh00,Sch01a,SchGraGroHalLen09}.

A difficulty in the \emph{quantitative} analysis of hybrid regularization methods is a mismatch of function spaces, summarized in \autoref{tab:chal}.\footnote{We use here the terminology Barron vector spaces and Barron energies in response to \cite{EMaWu22}:``It should be stressed that the terminologies ``space'' and ``norm'' are used in a loose way. For example, flow-induced norms are a family of quantities that control the approximation error. We do not take effort to investigate whether it is a real norm.''} Compare, for instance, the articles \cite{Bar93,BarCohDahDev08,Mha93,SieXu23,SieXu22}) with \cite{EngKunNeu89,EngHanNeu96,EngZou00,SchGraGroHalLen09}.
\begin{table}[h]
	\begin{center}
		\begin{tabular}{r||l}
			\hline
			Concept & Function space setting\\
			\hline
			\hline
			Neural networks & Barron vector spaces\\
			Regularization & Banach-and Hilbert-spaces,\\
			& in particular Sobolev spaces\\
			Wavelets & Besov spaces\\
			\hline
		\end{tabular}
	\end{center}
	\caption{\label{tab:chal} Prevailing function space settings in the respective theories. The most important definitions and properties are outlined in the Appendix in \autoref{sob_def}, \autoref{sec:barron} and \autoref{sec:besov}.}
\end{table}

In order to present a theory of hybrid regularization methods, we will need the following ingredients:
\begin{enumerate}	
	\item The classical theory of regularization methods for solving nonlinear inverse problems is reviewed in  \autoref{ch:classic} and lays the foundation for what comes after.
	
	\item We discuss three different classes of inverse problems in \autoref{sec:ill}:
	\begin{enumerate}
		\item Linear,
		\item compositions of linear ill-posed and nonlinear well-posed problem (see \autoref{se:decomp}) and
		\item general nonlinear inverse problems.
	\end{enumerate}
    Decomposed operators appear naturally when parametrizations with neural network functions are used for the functions
    $\x$ and $\y$ (see \autoref{eq:op2a}). It is important to note that in such situations convergence and rates of Newton's method for calculating neural network coefficients of a function $\x$ can be guaranteed (see \cite{SchHofNas23}).
	
	\item Neural network functions (see \autoref{ch:nn}) are central for many data driven approaches. They are compositions of \emph{activation functions} (denoted always by $\sigma$ throughout this book) and \emph{decision functions}\index{function!decision} (see \autoref{ch:nn}). These functions are building blocks of \emph{machine learning} methods.
	
	We adopt a \emph{generalized view} of \emph{neural network} functions (see \autoref{ch:nn}), which allows for versatile applications in regularization theory (see \autoref{ch:Aspri paper}). See \autoref{ta:nn} below, which shows a variety of such functions. Two historical research directions are followed for this purpose: The first historical detour starts with Norbert Wiener \cite{Wie32}, who approximated arbitrary functions by linear combinations of translations of \emph{one} activation function. This is a problem dating back to Alfred Tauber (see \cite{Tau97}). The second historical detour concerns wavelets (see Ingrid Daubechies \cite{Dau92} and Yves Meyer \cite{Mey93}). The ideas of wavelets and Wiener functions for approximating one dimensional functions can be generalized in various ways for approximating functions defined on a high dimensional Euclidean space $\R^n$  (see \autoref{ch:nn}): One of these generalizations leads to famous \emph{universal approximation theorem} developed by George Cybenko (see \cite{Cyb89}), Kurt Hornik \cite{Hor91}; see also the book of Walter Rudin \cite{Rud73}. The practical benefits for numerical implementations are that functions can be approximated by linear combinations of affine transformations of a \emph{single function}. \autoref{re:tw} shows the benefits of using only translations of a single function. However, in higher dimensions, this is too restrictive (see \autoref{ta:nn}).
	
	\item Instead
	of modeling the forward operator physically one can generate \emph{surrogate models} from expert data; recall \autoref{ex:c_reconstruct}. \index{surrogate model} Surrogate modeling of functions, functionals and operators is considered in \autoref{ch:op_learning}. We do not use them as surrogates but as priors in hybrid regularization methods (see \autoref{ch:Aspri paper}).
	
	For a comparison of the two approaches (constructive finite element approximations and neural network based forward operators) it is important to have quantitative convergence rates results, requiring both quantitative regularization results and error estimates for the approximate forward operator.

    \item While in \autoref{ch:par} we attempt to optimize regularization parameters, in \autoref{ch:Aspri paper} we combine learned operators with physically modeled operators. Applying the classical regularization to hybrid problems enlarges the possibilities of using prior information. See for instance \autoref{th:ip:rate-l}.

    \autoref{ch:op_learning} concerns all types of modern approximations of functions, functional and operators with neural network approximations, which can potentially replace classical function space approximations and operator approximations. In some sense, this makes the book different from parts of the literature, where data driven methods are employed as competitors to classical methods (see some discussions in \cite{GroKomLatScho24,HalNgu23,HerSchwZec24,LuJinPanZhaKar21}, to mention but a few).
	
	\item Machine learning provides opportunities to reinforce regularization as a \textit{data-driven approach}, for instance in the form of a posteriori parameter choice rules, such as the discrepancy principle, or balancing principles \cite{EngHanNeu96,LuPer13}. Data-driven techniques for optimal regularization parameter selection can be found for instance in
	\cite{CalCaoCarSchoVal16,ChiVitMolRosVil24,ChuEsp17,LosSchoVal16,LosSchoVal17,HabHol23_report,HabTen03,HolKunBar18,KunPoc13} and \autoref{ch:par}.
	
	\item The \emph{training process} of a neural network, such as the estimation of the coefficients $\Gamma$ in \autoref{eq:learn}, is highly demanding and is discussed in \autoref{ch:learning}.
\end{enumerate}

\autoref{fig:tg} shows a schematic overview of the contents of this book.
\begin{figure}[h]
	\begin{center}
		\includegraphics[width=.8\linewidth]{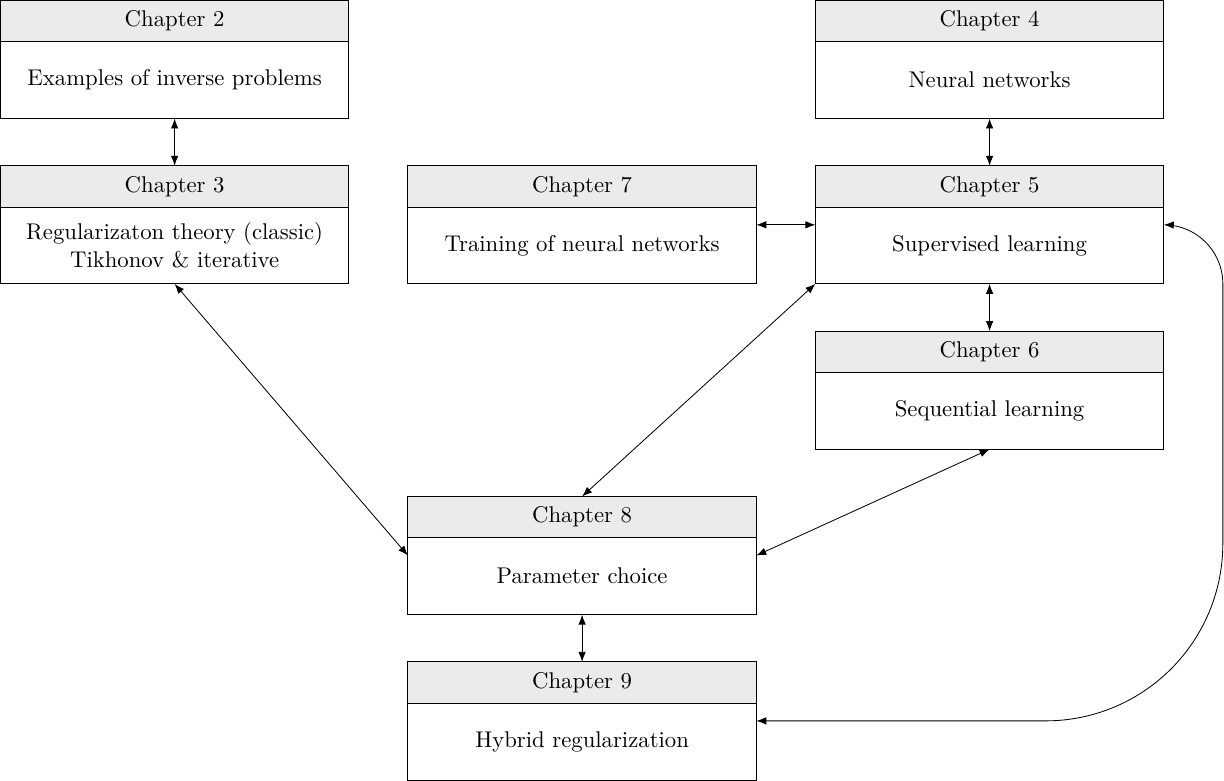}
		\caption{\label{fig:tg} There are strong dependencies between \autoref{ch:classic} and \autoref{ch:Aspri paper} as well as between \autoref{ch:op_learning} and \autoref{ch:Aspri paper}. \autoref{ch:Aspri paper} combines all the previously acquired information.}
	\end{center}
\end{figure}
Since the book combines many different areas, finding a unifying notation is very challenging. We have summarized our notation in \autoref{sec:not}.

The challenges of connecting data-driven approaches with regularization theory become evident through the simple examples originally presented by Heinz Engl, Karl Kunisch, and Andreas Neubauer in \cite{EngKunNeu89}. These examples serve as a foundation that we revisit and expand upon throughout this book.
\begin{description}
	\item{The  $c$-example:} 
		\begin{align*}
			\autoref{ex:c_reconstruct}
				&\Rightarrow \autoref{ss:c} \Rightarrow \autoref{ex:c-analysis} \Rightarrow \autoref{ex:c-rates} \\ 
				&\Rightarrow \autoref{ex:ac} \Rightarrow \autoref{ex:c_recon_cont}, \autoref{ex:c_recon_cont_II}
		\end{align*}
	\item{The $a$-example:}
		\begin{equation*}
			\autoref{a_reconstruct} \Rightarrow \autoref{ex:s_recon_cont}, \autoref{ex:a_recon_II}
		\end{equation*}
	The labels $c$ and $a$ refer to the variable names that were originally used in \cite{EngKunNeu89} to denote the parameter functions reconstructed in the respective inverse problems. In this book the notation is unified and all reconstructed parameters are usually denoted by $\x$.  
\end{description}
 \fbox{
	\parbox{0.90\textwidth}{The bottom line of these examples is that data-driven methodologies (in particular surrogate models) do not have the same powerful analysis available as for instance constructive (finite element) approaches (see \autoref{sec:appl}). However, the full strength of both approaches can be exploited if they are used in a hybrid manner (see \autoref{ch:Aspri paper}).
}}

\commentO{Unitary operator need to be unified.}

\chapter{Examples of inverse problems: from linear to nonlinear} \label{sec:ill}

In this book, we consider three different classes of ill--posed problems:
\begin{description}
	\item{{\bf Linear ill--posed problems}}, where the operator $\opo$ from \autoref{eq:op} is linear.
	\item{{\bf Nonlinear operators of categories $1$ and $2$:}} Here, the operator $\opo$ can be decomposed into a linear ill--posed operator $L$ and a well-posed nonlinear operator $\nopo$:
	\begin{description}
		\item{\bf Decomposition case 1} refers to the case $\opo = \nopo \circ L$, and
		\item{\bf Decomposition case 2} refers to the case $\opo = L \circ \nopo$.
	\end{description}
    These decomposition cases have been introduced in \cite{Nas87b,Hof94}.
    \index{decomposition case!1} \index{decomposition case!2}
    \item{{\bf General nonlinear problems:}} We consider two example problems that are neither of 1st nor 2nd decomposition cases. The first example is a one-dimensional \emph{inverse conductivity problem}, which can be formulated as an operator \autoref{eq:op}. The second problem concerns \emph{inpainting}, which \emph{cannot} be formulated as an operator equation. Still, regularization methods provide the basis to apply, in a second step, sophisticated painting techniques, such as \emph{texture synthesis}. 
\end{description}

\section{Inversion of the Radon transform} \label{sec:radon} is a \emph{linear} inverse problem:
\begin{enumerate}
	\item We investigate if we can learn the Radon operator from expert pairs $\mathcal{S}$, as defined in \autoref{eq:expert_information}. 
	Some methods which show the feasibility of this approach are discussed in \autoref{ss:lol} and \autoref{ss:Bio}. It is fascinating that methods from \autoref{ss:Bio} allow for the computation of the singular values of the Radon transform without any formal and physical knowledge of the operator $\opo$, but exclusively from expert pairs (see \cite{AspFriSch24_preprint}).
	\item With the methods from \autoref{ss:nol}, we can compute nonlinear approximations of $\opo$. It is a curious strategy, also feasible, to locally approximate a linear operator by a nonlinear operator (such as operator networks). However, we see that operator networks for linear operators are nothing more than a basis expansion with respect to an orthonormal basis (see \autoref{ss:nol}).
\end{enumerate}

The inversion of the \emph{Radon transform},\index{Radon transform} (see Radon's original work \cite{Rad17} and the reprint in \cite{Rad87}), is most likely the best studied linear inverse problem in the literature. This is due to its theoretical interest (see, for instance, the original papers  \cite{Fun13,Rad17,Rad17b,Cor63} and the historical comments \cite{Cor94}) and its relevance for medical imaging (see the original paper \cite{Hou73} and the mathematical survey of \cite{Kuc13}).\footnote{It is by now history that Allan Cormack and Godfrey Hounsfield got awarded the Nobel prize for the ``development of computer assisted tomography'' in 1979. On the occasion of the 75th anniversary of the Radon transform Allan Cormack stated that the Radon transform should be called Lorentz transform \cite{Cor94}.}

We consider reconstructing a function $\x: \R^2 \to \R$ given the values of integrals of $\x$ along lines. This problem is equivalent to determining $\x$ from its \emph{Radon transform}
\begin{equation}\label{eq:radon-transform}
	\begin{aligned}
		\y(t,\theta)&:=\mathcal{R}[\x](t,\theta) := \int_{L_{t,\theta}} \x(\vs) d \vs \\
		&=
		\int_{s=-\infty}^\infty \x \left(\gamma_{t,\theta}(s)\right) \norm{\frac{d}{ds} \gamma_{t,\theta}(s)}_2 ds \; \text{ for all } t \in \R, \theta \in [0,2\pi)\,,
	\end{aligned}
\end{equation}
where $\norm{\cdot}_2$ denotes the Euclidean norm and $\gamma_{t,\theta}$ is the parametrization of the line
\begin{equation}\label{eq:line}
	L_{t,\theta} = \set{\vs = \gamma_{t,\theta}(s)=\left(t  \vec{n}(\theta)+s \vec{n}(\theta)^\bot \right): s \in \R} .
\end{equation}
with $ \vec{n}(\theta) = \begin{pmatrix} \cos(\theta) \\ \sin(\theta) \end{pmatrix}$ and $ \vec{n}(\theta)^\bot = \begin{pmatrix} - \sin(\theta) \\ \cos(\theta) \end{pmatrix}$. 
This implies that 
\begin{equation}
	\label{eq:radon}
	\begin{aligned}
		\y(t,\theta)
		&=
		\int_{s=-\infty}^\infty \x \left(\gamma_{t,\theta}(s)\right) ds \; \text{ for all } t \in \R, \theta \in [0,2\pi)\;.
	\end{aligned}
\end{equation}
Because of the relation
\begin{equation*}
	\begin{aligned}
		\mathcal{R}[\x] (t,\theta+\pi) &= \int_{s=-\infty}^\infty \x \left(t\vec{n}(\theta+\pi)+s\vec{n}(\theta+\pi)^\bot \right) ds\\
		&=\int_{s=-\infty}^\infty \x \left(-t \vec{n}(\theta) - s  \vec{n}(\theta)^\bot \right) ds \\
		&= - \int_{\nu=\infty}^{-\infty} \x \left(-t \vec{n}(\theta) + \nu  \vec{n}(\theta)^\bot \right) d\nu 
		= \mathcal{R}[\x] (-t,\theta) \,,
	\end{aligned}
\end{equation*}
the Radon transform is completely determined from data for $t \in \R$ and $\theta \in [0,\pi[$. We call $(t,\theta) \in \R \times [0,\pi[$ the \emph{Radon-coordinates}\index{coordinates!Radon}.

\subsection*{Back projection}
Fundamental formulas for inverting the Radon transform are based on \emph{back projection}.
Suppose we select a point $\vs \in \R^2$, with $\vs\neq 0$, and represent it in polar coordinates:
$$\vs = \begin{pmatrix} s_1 \\ s_2 \end{pmatrix} = t  \vec{n}(\theta),$$
with some $t \in \R$ and $\theta \in [0,\pi[$. Note that the vector $ \vec{n}(\theta)$ is orthogonal to the line $L_{t,\theta}$ (see \autoref{eq:line}).
An intuitive way of recovering $\x$ from the values of its Radon transform consists in calculating the
average of all values of line integrals for lines passing through $\vs$. This is called the method of \emph{back projection}:\index{back projection}
\begin{definition}[Back projection] \label{de:back}
	Let \begin{equation*}
		\begin{aligned}
			{\tt z} : \R \times [0,\pi[ & \to \R,\\
			     (t,\theta) &\mapsto {\tt z}(t,\theta)
		\end{aligned}
	\end{equation*}
    be a function in Radon coordinates. The \emph{back projection} of ${\tt z}$ at a point $\vs \in \R^2$ is defined by
	\begin{equation*}
				\mathcal{B}[{\tt z}](\vs) := \frac{1}{\pi} \int_{\theta=0}^\pi {\tt z}(\inner{\vs}{\vec{n}(\theta)}_2,\theta) d\theta \,.
	\end{equation*}
\end{definition}
Note that the point $t  \vec{n}(\theta)$ with 
\begin{equation*}
t=\inner{\vs}{\vec{n}(\theta)}_2 = s_1\cos(\theta) +s_2\sin(\theta)
\end{equation*} 
is the projection of $\vs$ onto the normal vector
$ \vec{n}(\theta)$ of the line $L_{t,\theta}$ (see \autoref{fig:backprojection}).
\begin{figure}[h]
	\begin{center}
	\includegraphics[width=.4\linewidth]{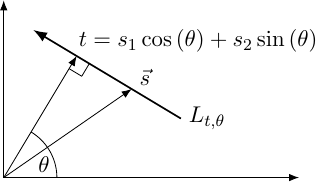}
	\caption{\label{fig:backprojection} Visualization of back projection}
	\end{center}
\end{figure}
In particular, the back projection of the Radon transform is given by:
\begin{equation*}
	\mathcal{B}[\mathcal{R}[\x]] (\vs) = \frac{1}{\pi} \int_{\theta=0}^\pi \mathcal{R}[\x] (\inner{\vs}{\vec{n}(\theta)}_2,\theta) d\theta \,.
\end{equation*}
Applying $\mathcal B$ to the Radon data $\mathcal{R}[\x]$ already provides a good approximation of $\x$ (see \cite{Qui01} and the elementary book \cite{Fee10}, where analytical examples are also computed).

\subsection*{The functional analytical setting of the Radon operator}
We consider the Radon transform for functions $\x \in L^2(\mathcal{B}_1(0))$, meaning that they are square integrable and have compact support in the unit disk $\mathcal{B}_1(0) \subseteq \R^2$.
\begin{definition}[Radon transform for functions supported in the unit disk] \label{de:radon}
	We consider the weighted $L^2$-space on ${\tt Z}:= (-1,1) \times (0,\pi)$ with weighted norm \index{Radon space!${\tt Z}$}
	\begin{equation*}	
		\norm{{\tt z}}_{L^2({\tt Z};{\tt w}^{-1})}^2= \int_{t=-1}^1 \int_{\theta=0}^\pi {\tt z}^2(t,\theta) {\tt w}^{-1}(t) d\theta \,dt \text{ for all } {\tt z} \in L^2({\tt Z};{\tt w}^{-1})
	\end{equation*}
	with weight function \index{weight function (Radon transform)!${\tt w}$}
	\begin{equation*}
		\begin{aligned}
			{\tt w}: [-1,1] &\to \R\,,\\
			     t &\mapsto \sqrt{1-t^2}\;.
		\end{aligned}
	\end{equation*}
	We consider the Radon transform as an operator between $L^2(\mathcal{B}_1(0))$ and $L^2({\tt Z};{\tt w}^{-1})$:	\index{Radon transform!${\mathcal R}$}
	\begin{equation} \label{eq:radon_disc}
		\begin{aligned}
			{\mathcal R}: L^2(\mathcal{B}_1(0)) &\to L^2({\tt Z};{\tt w}^{-1}).\\
			\x &\mapsto \mathcal{R}[\x]
		\end{aligned}
	\end{equation}
\end{definition}
Because 
\begin{equation*}
	\begin{aligned}
		~ & \norm{\mathcal{R}[\x]}_{L^2({\tt Z};{\tt w}^{-1})}^2 \\
		& = \int_{t=-1}^1 \int_{\theta=0}^\pi \abs{\mathcal{R}[\x](t,\theta)}^2 {\tt w}^{-1}(t) d\theta \,dt \\
		&=\int_{t=-1}^1 \frac{1}{\sqrt{1-t^2}} \int_{\theta=0}^\pi \left( \int_{s=-\sqrt{1-t^2}}^{\sqrt{1-t^2}} 1 \cdot \x \left(t \vec{n}(\theta) + s \vec{n}(\theta)^\bot \right)ds \right)^2  d\theta \,dt \\
		&\leq 2  \int_{\theta=0}^\pi \int_{t=-1}^1 \int_{s=-1}^1 \x^2 \left(t \vec{n}(\theta) + s \vec{n}(\theta)^\bot \right)ds dt d\theta\,,
	\end{aligned}
\end{equation*}

it follows by substituting
\begin{equation} \label{eq:subs}
	\vs = \vs(t,s) = t \vec{n}(\theta) + s \vec{n}(\theta)^\bot
\end{equation}
that (see also \cite{Dav81})
\begin{equation*}
	\norm{\mathcal{R}[\x]}_{L^2({\tt Z};{\tt w}^{-1})} \leq \sqrt{2 \int_{\theta=0}^\pi \int_{\vs \in \R^2} \x^2 (\vs) \,d\vs d\theta} = \sqrt{2\pi} \norm{\x}_{L^2(\mathcal{B}_1(0))}.
\end{equation*}
Thus, the Radon transform is bounded from $L^2(\mathcal{B}_1(0))$ into $L^2({\tt Z};{\tt w}^{-1})$. These function spaces serve as the basis for stable and convergent implementation of regularization methods. See for instance the more general (because applicable to nonlinear problems) results, \autoref{th:ip:stability}, \autoref{th:ip:convergence} and \autoref{th:ip:rate}.

In the following, we recall the expression of the adjoint operator of the Radon transform ${\mathcal R}$ defined for functions on the unit disk, which is the operator $\mathcal{R}^*: L^2({\tt Z};{\tt w}^{-1}) \to L^2(\mathcal{B}_1(0))$, which satisfies
\begin{equation*}
	\inner{\mathcal{R}[\x]}{{\tt z}}_{L^2({\tt Z};{\tt w}^{-1})} = \inner{\x}{\mathcal{R}^*[{\tt z}]}_{L^2(\mathcal{B}_1(0))}.
\end{equation*}
This operator should not be confused with the back projection operator, which has been defined without a function space setting. The adjoint operator is required, for instance, when applying gradient descent techniques (such as \autoref{eq:steepest}) to solve \autoref{eq:op}.
\begin{lemma}[Adjoint] \label{de:adjoint}
	For every ${\tt z} \in L^2({\tt Z};{\tt w}^{-1})$ and almost all $\vs \in \R^2$ the adjoint $\mathcal{R}^*$ of the Radon transform $\mathcal{R}$ defined in \autoref{eq:radon_disc} is given by
	\begin{equation*}
		\mathcal{R}^*[{\tt z}](\vs)= \int_{\theta = 0}^{\pi}
		{\tt z}(\xi(\vs,\theta),\theta) {\tt w}^{-1}(\xi(\vs,\theta))\, d \theta
	\end{equation*}
    with $\xi(\vs,\theta) = s_1 \cos(\theta) + s_2 \sin(\theta)$. 
\end{lemma}
\begin{proof} With this substitution $t = \xi(\vs,\theta)$ it follows
	\begin{equation*}
		\begin{aligned}
			~ & \inner{\mathcal{R}[\x]}{{\tt z}}_{L^2({\tt Z};{\tt w}^{-1})}\\
			& = \int_{\theta=0}^\pi \int_{t=-1}^1 \frac{1}{\sqrt{1-t^2}}\!\!\! \int_{s=-\sqrt{1-t^2}}^{\sqrt{1-t^2}} \x \left(t \vec{n}(\theta) + s \vec{n}(\theta)^\bot \right)ds \; {\tt z}(t,\theta)   \,dt d\theta \\
			& = \int_{\vs \in \R^2} \x (\vs) \;\int_{\theta=0}^\pi  {\tt z}(\xi(\vs,\theta),\theta)\frac{1}{\sqrt{1-\xi^2(\vs,\theta)}} \, d\theta d\vs \,,
		\end{aligned}
	\end{equation*}
	which is the assertion.
\end{proof}
In the following, we give a survey on the singular value decomposition (SVD) of the Radon transform for functions supported on the unit disk. The results are essentially copied from \cite{Nat01}, with the main difference being that we consider the adjoint $\mathcal{R}^*$ restricted to the range of $\mathcal{R}$ (in contrast to \autoref{de:back}).

 The singular value decomposition of the Radon transform has been computed in \cite{Dav81} (see also \cite{Nat01}) for the general case of images of $m$ variables. 
\begin{theorem}\cite[p 99]{Nat01} \label{th:singular_values} 
	The spectral decomposition (see \autoref{de:spectral}) of the Radon transform $\mathcal{R}$ is given by
	$$\set{({\tt u}_{k,l},{\tt v}_{k,l};\gamma_{k,l}): (k,l) \in \mathcal{I}}$$
	where
	\begin{enumerate}
		\item \label{it.1} $\mathcal{I} = \set{(k,l): k,l \in \N_0, l \leq k  \text{ and } l+k \text{ is even}}.$
		\item \label{it.2} $\gamma_k^2=\gamma_{k,l}^2 = \frac{2\pi}{k+1} > 0$ is independent of $l$,
		\item Let $\set{{\tt Y}_i : i \in \Z}$ denote the spherical harmonics (see \autoref{de:spherical_harmonics}) in $\R^2$\index{spherical harmonics} and let
		\begin{equation*}
			t \mapsto {\tt c}_k(t)=\frac{\sin \left((k+1)\arccos{(t)}\right)}{\sin(\arccos{(t)})}
		\end{equation*}
		denote the Chebyshev polynomials of the second kind\index{Chebyshev polynomials!second kind}. Moreover, let
		$$c_k^{-1} = \norm{ {\tt w} {\tt c}_k {\tt Y}_{k-2l}}_{L^2({\tt Z};{\tt w}^{-1})} = \norm{\sqrt{\tt w} {\tt c}_k}_{L^2(-1,1)}.$$
		Then the normalized eigenfunctions of $\mathcal{R}\mathcal{R}^*$ on the orthogonal complement of the nullspace of $\mathcal{R}^*$, $\mathcal{N}^\bot$, are given by
		\begin{equation}\label{eq:vkl_ukl}
			(t,\theta) \to {\tt v}_{k,l}(t,\theta):=c_k {\tt w}(t){\tt c}_k(t){\tt Y}_{k-2l}(\theta)  \text{ and } {\tt u}_{k,l} = \frac{1}{\gamma_k} \mathcal{R}^*[{\tt v}_{kl}]\;.
		\end{equation}
	\end{enumerate}
\end{theorem}
The singular value decomposition of the Radon transform has been computed in several works: see, for instance, \cite{Dav81,Maa87,Nat01} and \cite[Theorem 6.4]{NatWub01}, to mention but a few. The singular values are illustrated in \autoref{fig:ex_func_u} and \autoref{fig:ex_func_v}, respectively.  
\begin{remark}\label{ex:sets_Ek}
	\begin{enumerate}
		\item The existence of a singular value decomposition with $\gamma_k \to 0$ implies, in particular, that the Radon transform is compact.
		\item From \autoref{th:singular_values}, \autoref{it.1} and \autoref{it.2}, we see that in general, $\gamma_k$ has multiplicity higher than one. The sets $E^k$ associated to a spectral value $\gamma_k$ are spanned by the spectral functions ${\tt v}_{kl}$ with $l$ such that $(k,l) \in \mathcal{I}$. For example, taking $k=0,\ldots,7$ and $l=0,\ldots,k$, with $k+l$ is even (meaning that $(k,l) \in \mathcal{I}$ in \autoref{it.1}), we have
		\begin{equation}\label{eq:Ek_example}
			\begin{aligned}
				E^1&=\spann\set{{\tt v}_{0,0}},\quad E^2=\spann\set{{\tt v}_{1,1}},\quad E^3=\spann\set{{\tt v}_{2,0},{\tt v}_{2,2}},\\
				E^4&=\spann\set{{\tt v}_{3,1},{\tt v}_{3,3}}, \quad
				E^5=\spann\set{{\tt v}_{4,0},{\tt v}_{4,2}, {\tt v}_{4,4}},\\
				E^6&=\spann\set{{\tt v}_{5,1},{\tt v}_{5,3}, {\tt v}_{5,5}}, \quad E^7 =\spann\set{{\tt v}_{6,0},{\tt v}_{6,2}, {\tt v}_{6,4}, {\tt v}_{6,6}}, \\
				E^8&=\spann\set{{\tt v}_{7,1},{\tt v}_{7,3}, {\tt v}_{7,5}, {\tt v}_{7,7}}.
			\end{aligned}
		\end{equation}
		It is an interesting fact, stated in \cite{Nat01}, that the adjoint of the Radon transform, $\mathcal{R}^*$, has a nullspace. In fact, the nullspace of $\mathcal{R}^*$ is given by
		\begin{equation*}
			\mathcal{N} = \set{{\tt v}_{k,l}: k \in \N_0, l \in \set{0,1,\ldots,k}, \text{ satisfying } l+k \text { is odd}}.
		\end{equation*}
		Moreover, $\mathcal{R}\mathcal{R}^*$ has the same nullspace, as shown in \cite[p. 99]{Nat01}. This is why we concentrate our considerations on data in $\mathcal{N}^\bot$, the complement space of $\mathcal{N}$, and excludes the nullspace of $\mathcal{R}^*$.
	\end{enumerate}
\begin{figure}[h!]
	\centering
	\includegraphics[width=\textwidth]{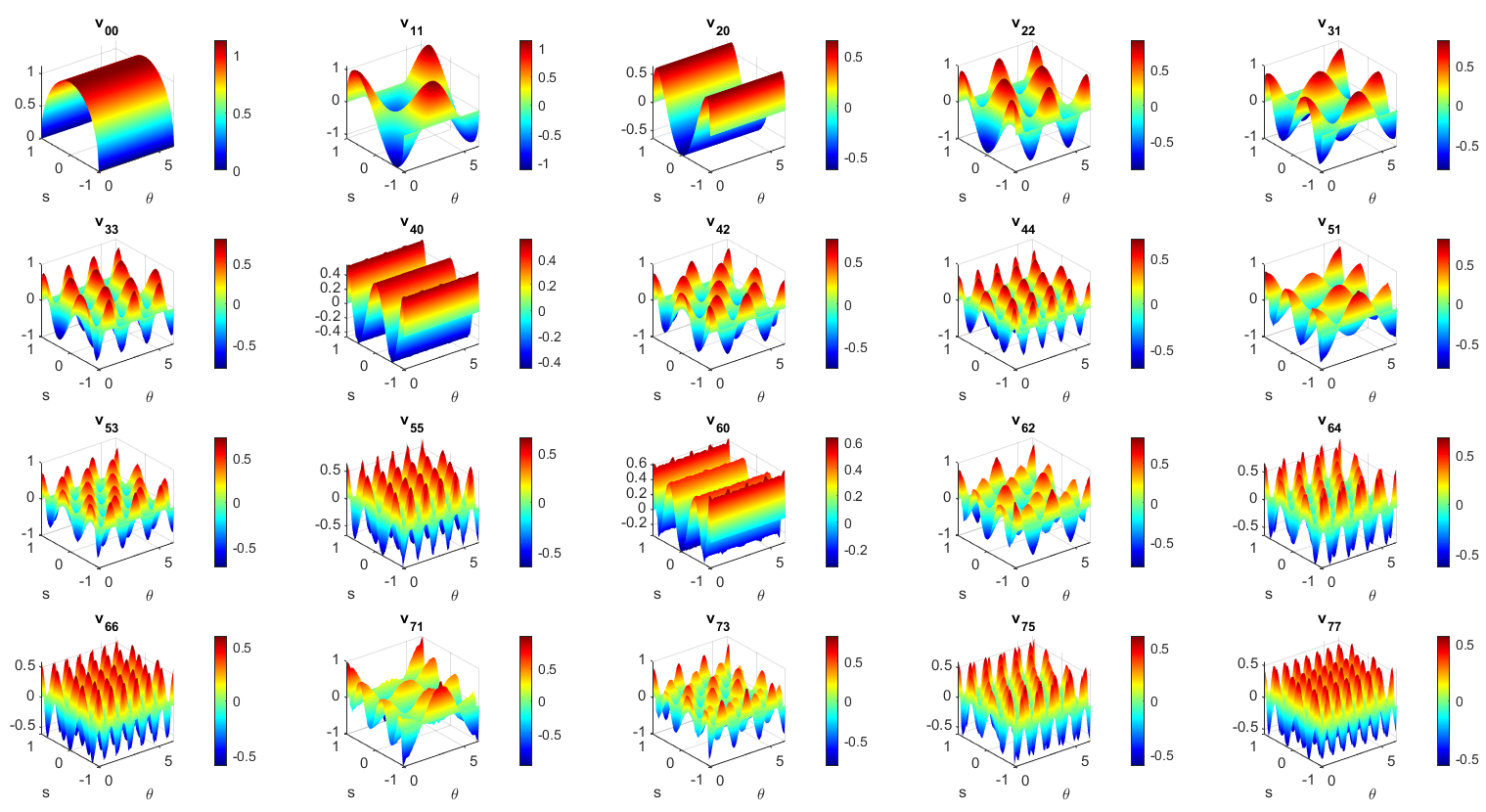}
	\caption{This plot visualizes the first twenty normalized singular functions ${\tt v}_{kl}$, defined in \autoref{eq:vkl_ukl}, for $k =0,1,\ldots,7$, $l=0,1,\ldots,k$ and $l+k$ even.}
	\label{fig:ex_func_v}
\end{figure}
\begin{figure}[h!]
\centering
\includegraphics[width=\textwidth]{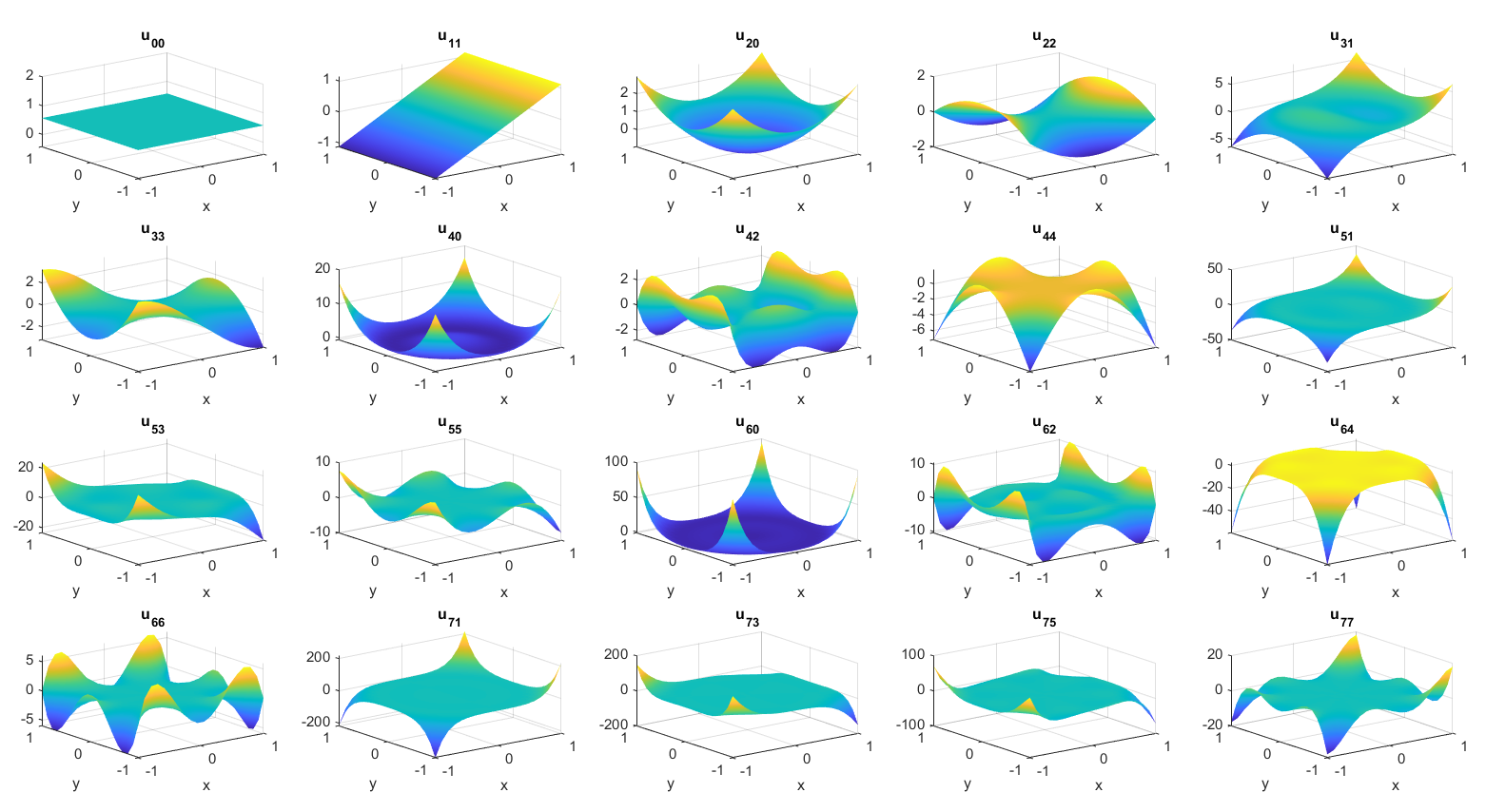}
\caption{This plot visualizes the first twenty singular functions ${\tt u}_{k,l}$, defined in \autoref{eq:vkl_ukl}, for $k =0,1,\ldots,7$, $l=0,1,\ldots,k$ and $l+k$ even.}
\label{fig:ex_func_u}
\end{figure} 
\end{remark}
Because of the analytical knowledge of its singular value decomposition the Radon transform it is interesting test example for learning properties, or even learning the complete Radon transform, from expert pairs (see \cite{AspFriSch24_preprint}).

\section{Decomposition cases}\label{se:decomp} The operators considered in this section have a weak form of nonlinearity. They can be decomposed into a linear ill-posed problems and a well--conditioned nonlinear problem. The way how they are composed, classifies them into first and second decomposition category, respectively. A recent survey containing several aspects of decomposition is \cite{Fle18}. 

\subsection{First decomposition case} \label{ss:decomp1} This decomposition case appears naturally, when the measurements are features of data generated by a linear operator: That is
\begin{equation} \label{eq:op2}
     \op{\x} = \nop{L \x} = \y\;.
\end{equation}
For instance, $\nopo$ could be an operator that maps a given function $\z = L\x$ onto a vector $\y = \vp \in \R^{\noc_*}$, which is a parametrization of $\z$, such as for instance the parameters $\vp$ of a shallow affine linear network (see \autoref{eq:classical_approximation}). Such a vector $\vp$ is called \emph{feature vector} of $\z$.\index{feature vectors} 

\begin{example}[Schlieren tomography] \label{ex:schlieren}
A practical example of the first decomposition case is \emph{Schlieren tomography}\index{Schlieren tomography} (see \cite{HanZan91,SchGraGroHalLen09,Fle11a}) where 
	\begin{equation}\label{eq:schlieren}
		\begin{aligned}
			\opo: \dom{\opo} \subseteq L^2(\mathcal{B}_1(0)) &\to L^2({\tt Z};{\tt w}^{-1})\,,\\
			\x &\mapsto \op{\x} = \left(\int_{L(\theta,d)} \x(s) \,ds \right)^2
		\end{aligned}
	\end{equation}
	is defined on the space of square-integrable functions on $\R^2$ and maps into some subspace of a weighted Sobolev space (see \autoref{de:radon}). Therefore Schlieren-tomography consist in inversion of the squared of the Radon-transform (see \cite{PitGreLuKin94,HanZan91,SchGraGroHalLen09}). Note that we have to ensure that elements of $\dom{\opo}$ are mapped into $L^2$, which requires some regularization. Moreover, note, that opposed to the standard Radon-transform (see \cite{Nat01,NatWub01}) the evaluated function $\x$ can have negative parts. This is a consequence of the fact that \autoref{eq:schlieren} is derived from an expansion in the Raman-Nath regime (see \cite{Boy20}).
\end{example}

\subsection{Second decomposition case} \label{ss:decomp2} This decomposition case has been studied in the context of \emph{state space regularization} \cite{ChaKun93} and in the context of the \emph{degree of ill--posedness} in \cite{GorHof94,Hof99,Hof94,HofSch94,HofTau97} as well as for defining the \emph{degree of nonlinearity} \cite{HofSch94}. Another ansatz leading to the second decomposition case is by parametrization of the elements of $\X$. Particularly, for regularization methods, finite dimensional approximations of images $\x \in \X$ with neural network functions have been studied in \cite{ObmSchwHal21} and for iterative regularization techniques in \cite{SchHofNas23}.

\begin{example}[Neural network representation] \label{ex:nn_second_decomp}
We consider the family of shallow affine linear neural network as defined in \autoref{eq:classical_approximation} and we 
use them as ansatz functions for solving a linear operator equation: Therefore we aim to solve
\begin{equation} \label{eq:op2a}
	\op{\x} = L\left(\nop{\vp}\right)= \y\,,
\end{equation}
where $\nop{\vp} = \Psi[\vp]$ as in \autoref{eq:classical_approximation}.

Related to \autoref{eq:op2} and \autoref{eq:op2a}, there are two central algorithmic issues:
\begin{enumerate}
	\item Given $\vp$ calculate $\x$, which is a straight forward evaluation (see \autoref{eq:classical_approximation} and \autoref{eq:DNN}).
	\item The computation of the parametrization $\vp$ given $\x$ is more delicate. The standard way consists in solving a minimization problem, to determine
	\begin{equation*}
		\argmin_{\vp \in \R^{\noc_*}} \norm{\x - \Psi[\vp]}_\X^2\;.
	\end{equation*}
    This minimization problem can be solved with a Gauss-Newton method and has been studied in \cite{SchHofNas23}. 
\end{enumerate}
\end{example}

\section{Examples of general nonlinear inverse problems}
In \autoref{ex:c_reconstruct} we have considered one instance of a nonlinear inverse problem. It gave us a first impression of how machine learning techniques can be used in regularization theory. Below, in \autoref{ss:c}, we provide further details regarding the well-posedness of the direct problem and the ill-posedness of the inverse problem. A further example concerns an \emph{inverse conductivity problem}\index{inverse conductivity problem} and is discussed in \autoref{a_reconstruct}. Both have been analyzed in \cite{ColKun86,EngKunNeu89,NeuSch90}.

All function spaces in \autoref{ss:c} and \autoref{a_reconstruct} consist of real valued functions on the unit interval $\Omega = ]0,1[$. and we avoid explicit mention of $\Omega$ whenever convenient.

\subsection{$c$-example}\label{ss:c}
%\subsubsection*{The c-example}
%We return to the c-example from the introduction, \autoref{cha:intro}. After showing well-posedness of the direct problem and ill-posedness of the inverse problem, we verify that the associated operator $\opo$ satisfies \autoref{as:ip:weakly-closed} so that the results of \autoref{sec:mns} and \autoref{sec:tikhonov} apply.
%\begin{example}[c-example]
	
\subsubsection*{The direct problem} %Since the set $\mathcal{C} = \{ \x \in L^2 : \x \ge 0 \; \text{a.e.} \}$ has empty interior %\commentS{This could be a nice exercise.}
	%in $L^2$, we introduce its closed \emph{$\epsilon$-fattening}
	%\begin{align*}
	%	\mathcal{C}_\epsilon  = \{ \x + \z : \x \in \mathcal{C}, \; \z \in L^2, \; \| \z \|_{L^2} \le \epsilon \},
	%\end{align*}	
	%where $0 < \epsilon < 1$. For $\x \in \mathcal{C}_\epsilon$ and $\f \in L^2$ %define the linear differential operator
	%\begin{align*}
	%	A_\x \y := -\y'' + \x \y
	%\end{align*}
	%and 
	Recall the boundary value problem \index{problem!$c$}
	\begin{equation} \label{eq:c-bvp}
		\begin{aligned}
			A_\x \y	&= {\tt f} \quad \text{in } \Omega \,, \\
			\y(0) =  \y(1)	&= 0 \,,
		\end{aligned}
	\end{equation}
	where $A_\x \y := -\y'' + \x\y$ and $\f \in L^2$. In \autoref{ex:c_reconstruct} we assumed the coefficient $\x$ to belong to $\mathcal{C} = \{ {\tt c} \in L^2 : {\tt c} \ge 0 \; \text{a.e.} \}$. However, since $\mathcal{C}$ has empty interior in $L^2$, it will be useful later to slightly enlarge this set (see, for instance, \autoref{as:ip:rates}). Therefore, we introduce its closed \emph{$\epsilon$-fattening}
	\begin{align*}
		\mathcal{C}_\epsilon  = \{ {\tt c}  + \z : {\tt c}  \in \mathcal{C}, \; \z \in L^2, \; \| \z \|_{L^2} \le \epsilon \},
	\end{align*}	
	where $0 < \epsilon < 1$. The weak formulation of \eqref{eq:c-bvp} consists in finding $\hat \y \in W_0^{1,2}$ such that
	\begin{equation} \label{eq:c-bvp-weak}
		\inner{\hat \y'}{{\tt v}'}_{L^2} + \inner{\x\hat \y}{\tt v}_{L^2} = \inner{\tt f}{\tt v}_{L^2} \quad \text{ for all } {\tt v} \in W_0^{1,2}.
	\end{equation}
	Note that the estimate
	\begin{align}\label{eq:poincare}
		\norm{\y}_{L^2} \le \norm{\y}_{L^\infty} \le \norm{\y'}_{L^2},
	\end{align}
	which is valid for all $\y \in W_0^{1,2}$, implies that $\y \mapsto \norm{ \y' }_{L^2}$ is equivalent to the standard $W^{1,2}$ norm on the subspace $W_0^{1,2}$. Analogously, $\y \mapsto \norm{\y''}_{L^2} + \norm{\y'}_{L^2}$
	%\begin{align*}
		%\norm{ \y }_H := \norm{\y''}_{L^2} + \norm{\y'}_{L^2}.
	%\end{align*}
	is equivalent to the standard $W^{2,2}$ norm on $W^{2,2} \cap W_0^{1,2}$.
	\begin{proposition}\label{thm:c-uniqueness}
		For every $\x \in \mathcal{C}_\epsilon$ and every $\f \in L^2$ problem \eqref{eq:c-bvp-weak} has a unique solution $\hat \y = \hat \y_{\x,\f} \in  W^{1,2}_0$. It satisfies
		\begin{align}\label{eq:c-H1-est}
			(1-\epsilon) \norm{\hat \y'}_{L^2} \le \norm{\f}_{L^2}.
		\end{align}
	\end{proposition}
	\begin{proof}
		The bilinear form given by the left-hand side of \eqref{eq:c-bvp-weak} is continuous due to \eqref{eq:poincare}. We prove coercivity next. Let $\tilde \x \in \mathcal{C}$ be such that $\norm{\x - \tilde \x}_{L^2} \le \epsilon.$ Then, the Cauchy-Schwarz inequality and \eqref{eq:poincare} imply
		\begin{align*}
			\inner{\x \y}{\y}_{L^2}
				&= \inner{\left( \x - \tilde \x + \tilde \x \right) \y}{\y}_{L^2} \ge \inner{\left( \x - \tilde \x \right) \y}{\y}_{L^2} \ge - \norm{\x - \tilde \x}_{L^2} \norm{\y^2}_{L^2} \\
				&\ge - \epsilon \norm{\y^2}_{L^2} \ge - \epsilon \norm{\y}^2_{L^\infty} \ge - \epsilon \norm{\y'}^2_{L^2}.
		\end{align*}
		Adding $\inner{\y'}{\y'}_{L^2}$ to both sides of the inequality and recalling that $\epsilon < 1$ shows coercivity. The Riesz representation theorem (or the Lax-Milgram theorem) now guarantees existence of a unique $\hat \y$ solving \eqref{eq:c-bvp-weak}.
		%\commentS{I think the existence of a unique weak solution follows from the Lax-Milgram-Lemma. The Lax-Milgram proof uses Riesz representation but the coercivity is not needed for Riesz, I think the coercivity is used for the uniqueness. But I would have to double check.}
		Estimate \eqref{eq:c-H1-est} follows from combining \eqref{eq:c-bvp-weak} for ${\tt v} = \hat \y$ with the coercivity estimate.
	\end{proof}	
	The next proposition shows that $\hat \y_{\x,\f}$ depends continuously, as an element of $W^{2,2}$, on both $\x$ and $\f$.
	\begin{proposition}\label{thm:c-H2}
		The weak solution $\hat \y = \hat \y_{\x,\f}$ belongs to $W^{2,2}$ and satisfies
		\begin{align}\label{eq:c-strong-ae}
			-\hat \y'' + \x \hat \y = \f \quad \text{a.e.\ in } \Omega.
		\end{align} 		
		Moreover, there is a constant $C = C(\epsilon)$ such that
		\begin{align}\label{eq:c-H2-est}
			\norm{\hat \y_{\x,\f} - \hat \y_{{\tt w},{\tt g}}}_{W^{2,2}} \le C \left( 1 + \norm{{\tt w}}_{L^2} \right) \left(\norm{\f - {\tt g}}_{L^2} + \norm{\f}_{L^2} \norm{\x - {\tt w}}_{L^2} \right)
		\end{align}
		for all $\f$, ${\tt g} \in L^2$ and all $\x, {\tt w} \in \mathcal{C}_\epsilon.$
	\end{proposition}
	\begin{proof}
		%First, it can be shown along the lines of \cite[§6.3.1, Thm.~1]{Eva10} that $\hat \y \in W^{2,2}_{\mathrm{loc}}$, but note that it is assumed there that $\x \in L^\infty.$ Using \eqref{eq:poincare}, however, their argument can be modified to yield the desired result.
		The weak formulation implies that
		\begin{align}\label{eq:c-y''}
			\inner{\hat \y'}{{\tt v}'}_{L^2} = -\inner{\x \hat \y -\f}{\tt v}_{L^2}
		\end{align}
		holds for all ${\tt v} \in C_c^\infty$. Thus, $\x \hat \y -\f$ is the weak derivative of $\hat \y'$ and since $\x \hat \y -\f$ belongs to $L^2$, $\hat \y$ belongs to $W^{2,2}$. Moreover, integration by parts of the left-hand side of \eqref{eq:c-y''} and applying the fundamental lemma of the calculus of variations yields \eqref{eq:c-strong-ae}.
		
		Turning to \eqref{eq:c-H2-est}, we have
		\begin{align*}
			 \norm{\y''}_{L^2} - \norm{\x}_{L^2}\norm{\y'}_{L^2} \le \norm{\y'' - \x \y}_{L^2}
		\end{align*}
		for every $\y \in W^{2,2} \cap W^{1,2}_0$. On the other hand,
		\begin{align*}
			(1-\epsilon) \norm{\y'}^2_{L^2} \le \inner{\y'}{\y'}_{L^2} + \inner{\x \y}{\y}_{L^2} = \inner{\y}{-\y'' + \x\y}_{L^2} \le \norm{\y'}_{L^2} \norm{-\y'' + \x\y}_{L^2}.
		\end{align*}
		Multiplying this inequality with $\frac{1 + \norm{\x}}{(1-\epsilon)\norm{\y'}}$ and adding it to the previous one shows that there is a constant $C'$ such that
		\begin{align*}
			\norm{\y}_{W^{2,2}} \le C' (1 + \norm{\x}_{L^2}) \norm{-\y'' + \x\y}_{L^2}
		\end{align*}
		holds for all $\y \in W^{2,2} \cap W^{1,2}_0$. Plugging in $\hat \y_{\x,\f}$ for $\y$ and using \eqref{eq:c-strong-ae} gives
		\begin{align} \label{eq:c-H2-est1}
			\norm{\hat \y_{\x,\f}}_{W^{2,2}} \le C' (1 + \norm{\x}_{L^2}) \norm{\f}_{L^2}.
		\end{align}
		The function $\z = \hat \y_{\x,\f} - \hat \y_{{\tt w},{\tt f}}$ satisfies
		\begin{align*}
			A_{\tt w} \z = A_{\tt w} \hat \y_{\x,\f} - \f = A_{\tt w} \hat \y_{\x,\f} - A_\x \hat \y_{\x,\f} = ({\tt w} - \x) \y_{\x,\f} \in L^2
		\end{align*}
		as well as the boundary conditions. Thus, $\z$ is the unique solution of the boundary value problem \eqref{eq:c-bvp} with inhomogeneity  $({\tt w} - \x) \y_{\x,\f}$ and coefficient ${\tt w}$. Applying \eqref{eq:c-H2-est1}, \eqref{eq:poincare} and \eqref{eq:c-H1-est} gives
		\begin{align}\label{eq:c-H2-est2}
			\norm{\z}_{W^{2,2}} \le C'(1 + \norm{{\tt w}}_{L^2}) \norm{({\tt w} - \x) \y_{\x,\f}}_{L^2} \le C''(1 + \norm{{\tt w}}_{L^2}) \norm{{\tt w} - \x}_{L^2} \norm{\f}_{L^2}.
		\end{align}
		Finally, we write
		\begin{align*}
			\norm{\hat \y_{\x,\f} - \hat \y_{{\tt w},{\tt g}}}_{W^{2,2}} = \norm{\z + \hat \y_{{\tt w},{\f}} - \hat \y_{{\tt w},{{\tt g}}}}_H  \le \norm{\z}_{W^{2,2}} + \norm{\hat \y_{{\tt w},{\f - {\tt g}}}}_{W^{2,2}}
		\end{align*}
		and apply \eqref{eq:c-H2-est2} to the first term on the right and \eqref{eq:c-H2-est1} to the second one. This gives \eqref{eq:c-H2-est} and the proof is finished.
	\end{proof}

\subsubsection*{An inverse problem} Fix arbitrary $\f \in L^2$ and $0<\epsilon<1$, set $\X = \Y = L^2$ and define the nonlinear operator
	\begin{align}\label{eq:c-op}
		\opo :\dom{\opo} = \mathcal{C}_\epsilon \subset \X \to \Y, \quad \x \mapsto \op{\x} = \hat{\y}_{\x,\f}.
	\end{align}
	First, we state a sufficient condition for unique solvability of the operator equation $\op{\x} = \y$.
	\begin{proposition}\label{thm:c-existence-uniqueness}
		Suppose $\y \in W^{2,2} \cap W^{1,2}_0$, $\y \neq 0$ a.e.\ in $\Omega$ and $\x:=\frac{\f + \y''}{\y} \in \mathcal{C}_\epsilon$. Then, $\x$ is the unique solution of $\op{\x} = \y$.
	\end{proposition}	
	\begin{proof}
		The function $\x$ is defined such that $-\y'' + \x\y = \f$ holds a.e. in $\Omega.$ Hence, $\op{\x} = \y.$ For every other solution $\tilde \x$ we have $-\y'' + \x\y = \f = -\y'' + \tilde \x\y$, which implies $\x = \tilde \x$ a.e.\ since $\y \neq 0$ a.e.
	\end{proof}
	The following instance of \eqref{eq:c-bvp}, which is due to \cite{EngKunNeu89}, shows that this inverse problem is ill-posed. Specifically, a sequence $\y_k \to \y$ is constructed such that $\opo^{-1}[\y_k]$ and $\opo^{-1}[\y]$ are uniquely defined but $ \lim_k \opo^{-1}[\y_k] \neq \opo^{-1}[\y]$.
	%\begin{example}
	
		Set $\f \equiv 16$ and $\y(s) = 8s(1-s)$. Then, according to \autoref{thm:c-existence-uniqueness}, the unique solution of the inverse problem is $\x \equiv 0$. For $k \ge 2$ define $\y_k = \y + {\tt e}_k$, where ${\tt e}_k$ is symmetric with respect to $s=1/2$, meaning that ${\tt e}_k(s) = {\tt e}_k(1-s)$ for $0 \le s\le 1$, and
		\begin{align*}
			{\tt e}_k(s) = k^{-\frac54} (2s)^{2k} - 4k^{-\frac14}s, \quad 0 < s \le \frac12.
		\end{align*}
		We have $\norm{{\tt e}_k}_{L^2} \le \norm{{\tt e}_k}_{L^\infty} = |{\tt e}_k(1/2)| = |k^{-\frac54} - 2k^{-\frac14}| \to 0$ and therefore $\y_k \to \y$.
		
		Next, we verify that $\y_k$ meets the conditions of \autoref{thm:c-existence-uniqueness}. Note, first, that $\y_k$ is a polynomial and that it vanishes at the endpoints. Moreover, the function
		\begin{align*}
			\xk := \frac{\f + \y''_k}{\y_k} = \frac{{\tt e}''_k}{\y + {\tt e}_k},
		\end{align*}
		lies in $\mathcal{C}$ because it is bounded and both numerator and denominator are nonnegative. Thus, $\xk$ is the only function satisfying $\op{\xk} = \y_k$. However, it can be shown that $\norm{\xk}_{L^2} \to \infty$, showing the ill-posedness the inverse problem. %\commentC{Possible exercise: Work out the details of this example, i.e.\ show ${\tt e}_k \to 0$, $\xk \in \mathcal{C}$ and $\norm{\xk}_{L^2} \to \infty$.}		
%\end{example}

\subsection{$a$-example: an inverse conductivity problem} \label{a_reconstruct}
\index{problem!$a$}
This is the $a$-example mentioned already in the introduction. For given $\nu > 0$ let $\opo$ be defined as follows
\begin{equation} \label{eq:aex} 
	\begin{aligned}
	\opo: \dom{\opo} := \set{\x \in W^{1,2} : \x \geq \nu > 0} &\to L^2\,,\\
	    \x &\mapsto \op{\x} := \y
	\end{aligned}
\end{equation}
where $\y:=\y[\x]$ is the unique solution of %(note that $\x$ and $\y$ are functions of the variable $s \in ]0,1[$)
\begin{equation} \label{eq:aex2}
	\begin{aligned}
		-(\x \y')' & = \f \quad \text{in } \Omega,\\
		\y(0) = \y(1) & = 0,
	\end{aligned}
\end{equation}
for a given $\f \in L^2$. %This problem is also known as \emph{Darcy's flow}\index{Darcy's flow} (see for instance \cite{Stu10}).

Analogously to \autoref{ex:c_reconstruct} we turn to the variational formulation. Let $\Y_{\tt n}$ be
as in \autoref{eq:Fm} and let $\y_{\tt n}[\x] \in \Y_{\tt n}$ be the unique solution of
\begin{equation} \label{eq:afem}
	\inner{\x \y_{\tt n}'}{{\tt v}'}_{L^2} = \inner{\tt f}{\tt v}_{L^2} \;\text{for all}\; 
	{\tt v} \in \Y_{\tt n}.
\end{equation}
$\opo_{\tt n}$ will then be defined as
\begin{equation} \label{eq:afm}
	\begin{aligned}
		\opo_{\tt n}: \dom{\opo} &\to L^2\;.\\
		\x &\mapsto \opo_{\tt n}[\x] := \y_{\tt n}[\x]
	\end{aligned}
\end{equation}
It is known (see \cite{Cia78}) that $\opo_{\tt n}$ approximates $\opo$ with the rate
\begin{equation} \label{eq:rate_a}
	\begin{aligned}
		\norm{\opo_{\tt n}[\x]-\op{\x}}_{L^2} &= \norm{\y_{\tt n}[\x] - \y[\x]}_{L^2} \\
		&= (1+\norm{\x}_{W^{1,2}}) \mathcal{O}\Bigl(
		{\tt n}^{-2} \Bigr),
	\end{aligned}
\end{equation}
%Note, that the rate is with respect to $\ttn$. The uniform order constant with respect to $\ttn$ is multiplied by  $(1+\norm{\x}_{W^{1,2}(0,1)})$,
which is uniform on bounded subsets of $W^{1,2}$. Note that ${\tt n}^{-1}$ is one half of the width of the support of the spline basis functions.

\begin{remark}
	The terminologies $a$-, $c$- inverse problems are motivated from the classical ordinary differential example equation
	\begin{equation*}
		a \x''+ b \x'+ c \x = {\tt f}\,,
	\end{equation*}
	which is called $a,b,c$-problem.
	\index{problem!$a,b,c$}
\end{remark}

\section{Inpainting}\label{ss:inpainting} \index{inpaiting}
Let $\Omega$ be an open set in $\R^2$ with $\bar \Omega \subset ]0,1[^2$. We consider the problem of \emph{naturally continuing} pixel intensities given by a function $\yd: ]0,1[^2 \backslash \Omega \to \R$ into $]0,1[^2$. The goal therefore is to reconstruct $\x :]0,1[^2 \to \R$ satisfying $\x\big|_{]0,1[^2 \backslash \Omega} \approx \y$ given $\yd$, which we assume to be noisy. We start by formulating the according operator
\begin{equation} \label{eq:inpaint}
	\begin{aligned}
			\opo: \X &\to \Y\,,\\
		\x &\mapsto \x \big|_{]0,1[^2 \backslash \Omega}
	\end{aligned}
\end{equation}
restricting a function $\x$ defined on $]0,1[^2$ onto $]0,1[^2 \backslash \Omega$. Common choices are $\X = W^{s,2}(]0,1[^2)$ (see \autoref{sob_def}) with $s=1,2$ and $\Y = L^2(]0,1[^2 \backslash \Omega)$. Some appetizers of methods and algorithms for inpainting are \cite{BalBerCas01,ChaShe05,ChaSheVes03,ChaShe05b,EseShe02,GroSch03,KokMorFitRay95,Mas02,MasMor98,OliBowMckCha01,Sap01} among many others. 
The extrapolation of data is based on physical models like partial differential equations, but does not resemble \emph{expected aesthetics}\index{expected aesthetics}, meaning that expected image information, like patterns from the neighborhood of the inpainting domain are inferred. In general, pure inpainting algorithms produce artifacts if the inpainting domain is surrounded by textured regions.

The operator $\opo: W^{s,2}(]0,1[^2) \to L^2(]0,1[^2 \backslash \Omega)$ is linear, however the constraints imposed on \autoref{eq:op} are typically nonlinear, and resemble the fact that the inpainting should look natural. We present one typical constraint and assume for the sake of simplicity that $\y = \yd$, that is, we have access to noise-free data in the painted domain. A common assumption is that for every $l \in \R$ the boundary of the level set 
    \begin{equation*} 
       \Omega_l := \set{\vs \in ]0,1[^2 \backslash \Omega: \x(\vs) \leq l}
    \end{equation*} (more precisely these are called sublevel sets)\index{level sets}\index{sublevel sets} has $1$-dimensional Hausdorff-measure (see \cite{EvaGar15}) and for $\vs \in \partial \Omega \cap \Omega_l$ the tangent to $\Omega_l$ is not parallel to $\partial \Omega$. This means that the level sets can locally be extended from the painted domain $]0,1[\backslash \Omega$ into $\Omega$.

\subsection*{Texture synthesis and inpainting}
Natural images contain texture, such as shading, noise, feature patterns, to mention but a few. Textures have to be interpolated in the inpainting region to make the overall look natural. The task is performed by \emph{texture synthesis}.\index{texture synthesis} There one makes the assumption that for every point $\vs \in \Omega$ there exists a point $\tilde{\vs} \in ]0,1[^2 \backslash \Omega$ such that the according \emph{textures}\index{image!texture} in neighboring \emph{patches} of the two points are identical; patches is a term from computer vision which refers to a vector of image information: For instance we can assemble a patch from the image in a neighborhood of a pixel by making a wavelet expansion of the image in the local neighborhood at different scale.\index{patch!image} 

\emph{Texture synthesis algorithms} use a sample of the available image data and fill the inpainting domain with texture, making a realistic image visualization, see for instance \cite{WeiLev00,Ash01,PagLon98,EfrLeu99b,WieLiSylVivTur09_report}. In the recent year improvements of these algorithms have be achieved in two ways: First of all by using by neural networks (see \autoref{ch:nn}) for inpainting (see for instance \cite{ZhaZhaLiZhoLin23}) by texture synthesis with a bigger database of images; Note that the above mentioned algorithms only use the texture information from the original image on which texture synthesis needs to be performed. 

\subsection*{A texture synthesis inpainting algorithm}
The following algorithm for inpainting with texture synthesis has been proposed in \cite{Gro04,Gro05} and consists of five steps:
\begin{algorithm}
\caption{A texture synthesis algorithm for inpainting}
\label{alg:AI}
\begin{algorithmic}[1]
	\State Filtering the image data and decomposing it into geometry and texture;
	\State Inpainting with variational or partial differential equation (PDE) based methods of the geometry part;
	\State Postprocessing of the geometry inpainting;
	\State Segmentation of the inpainted geometry image, for instance with level set methods (see \autoref{se:lsm});
	\State Synthesizing texture for each segment, for instance with \emph{Markov-models} (see \autoref{se:hmm}).\index{Markov!model}
\end{algorithmic}
\end{algorithm}
It is important to note that texture synthesis \emph{cannot} be formulated as a constraint of the operator equation \autoref{eq:op}. Therefore only partial problems can be solved with regularization methods as considered in this book: For instance finding an adequate partitioning of the inpainting domain can be found with regularization methods. \emph{Aesthetic} texture is inpainted in the partitioned inpainting domain via Markov-models (see \autoref{se:hmm}).

\section{Open research questions}
\begin{opq} \label{opq:AubNit} \autoref{ex:c_reconstruct} and \autoref{a_reconstruct} use the famous Aubin-Nitsche trick (see \cite{Aub67,Nit68}, which derives $L^2$-estimates for finite element approximations of a partial-differential equations, which are one order of discretization (to be precise $\tt{n}^{-2}$, see \autoref{eq:rate_a} and \autoref{eq:m2_estimate})) better than standard Cea's-lemma (see \cite{Cia02,ArnLog14,Zla68}) for $W^{1,2}$-estimates. The Aubin-Nitsche-trick is very important for inverse problems applications because it is assumed that the measurement errors are in $L^2$ and thus we require estimates $\norm{F_{\tt n}[\x]-F[\x]}_{L^2}$ in a neighborhood of the solution $\xdag$ of the inverse problem (that is of the solution of \autoref{eq:op}). In statistics (see for instance \cite{Nic23}) it is often assumed that the noise is $W^{-1,2}$, the dual of the Sobolev-space $W^{1,2}$. The open question is whether the Aubin-Nitsche-trick can be further generalized to the dual norm?  
\end{opq}

\section{Further reading}
Over the last decades a huge amount of examples of inverse problems have been studied. Therefore only an incomplete selection can be provided here. A highly studied nonlinear inverse problem is \emph{electrical impedance tomography} (see \cite{BarBro90,Nac88,SylUhl87,Uhl09}), which solves \emph{Calderon's problem} \cite{Cal80}. \index{Calderon's problem}\index{electrical impedance tomography}A similar amount of attention has been attracted by \emph{optical tomography}\index{optical tomography} (see \cite{Arr99,ArrScho09}) and \emph{inverse scattering} (see \cite{CakCol14,CakColHad16,ColKre19}).\index{inverse scattering} Relatively recently the area of \emph{coupled physics imaging} gained a lot of popularity as well \cite{ArrSch12,BalSch10,Bal12,BalUhl12,GilLevScho18,Kuc12}. The Radon transform and its spherical transform continue to be the focus of a sizable amount of research (see, for instance, \cite{ElbSchShi17,Kuc13,KucKun08,Nat12}), in part because of applications in photoacoustics \cite{KruLuiFanApp95,KruMilReyKisReiKru00,RivZhaAna06,Wan08,XuWan02a,XuFenWan02}\index{photoacoustics}. \emph{Diffraction tomography}\index{diffraction tomography} has been studied in \cite{Kak79,KakSla01,Lau02,Pan98,Pan00,TsiDev93} and has important applications in \emph{ultrasound tomography} and \emph{microscopy}\index{microscopy} (see \cite{BroBroZibAzh02,DurLitBabChaAze05,Gre83,NatWub01,NatWue95,SunChoFanBadDas09,WanMatAniLiDurAna15}). Nonlinear inverse problems related to \emph{elastography}\index{elastography} have been considered in \cite{BarGok04,BarObe07,Doy12,HubSheNeuSch18,KenKenSam14,KraSheHubLiuDre22,LarSam17,ManOliDreMahKru01,McLRen06a,MclRen06,OphCesPonYazLi91,PelFroInsHer04,PanGaoTaoLiuBai14,Schm98,SetGoeDhoBruSlo11,SheKraHubDreSch20}. In \cite{KalNguSch21} a series of nonlinear inverse problems for parabolic partial differential equations are considered. Further information on the decomposition of operators in the context of inverse problems can be found in \cite{EngHofZei93,Fle18}.

In many perspectives the inpainting \autoref{alg:AI} shares similarities with \emph{natural language models}\index{natural language models} (see for instance \cite{Cho56}) and other \emph{artificial intelligence algorithms}\index{Artificial intelligence!AI} (AI). On a methodological side these are Markov-models (see also \cite{Ros95}), with trained weights. Curiously, \emph{language models} which do not obey the context have been implemented around the year 2000 to write fake scientific papers. \emph{SCIgen}\index{SCigen} is a paper generator that uses \emph{context-free} grammar. See \url{https://pdos.csail.mit.edu/archive/scigen/}
In view of the development in language models one could consider comparing the hybrid inpainting algorithm combining texture synthesis and partitioning of the inpainting domain a \emph{natural} inpainting algorithm.

Texture is often modeled as having finite energy in negative Besov space. This idea goes back to Yves Meyer \cite{Mey01}. His ideas have been used to decompose images into cartoon and texture. See, for instance, \cite{AujAubBlaCha03,AujAubBlaCha05,AujCha05,AujGilChaOsh06,ChaMor25_report,OshSolVes03,VesOsh04,VesOsh03}).

\chapter[Regularization theory]{Regularization of nonlinear inverse problems} \label{ch:classic}
In this chapter we review the classical theory of variational and iterative regularization methods. This forms the theoretical basis of the hybrid regularization approaches presented in \autoref{ch:Aspri paper}, which combine physical and data driven modeling for the forward operator. 

\section{Tikhonov regularization} \label{se:Tik}
This section is devoted to quadratic Tikhonov regularization in Hilbert spaces, which, among variational regularization techniques, is arguably the easiest to analyze. However, the complications due to the nonlinear operator
\begin{equation*}
	\opo: \mathcal{D}(F) \subset \X \to \Y
\end{equation*}
are significant. Specifically, we approximate solutions of
\begin{equation} \label{eq:Op} 
	\op{\x} = \y
\end{equation}
by minimizers of
\begin{equation} \label{eq:tik}
	\T_{\alpha,\y^\delta}[\x] = \norm{F[\x]-\y^\delta}_\Y^2 + \alpha \norm{\x-\x^0}_\X^2\;.
\end{equation}
where the initial guess $\x^0$, noisy datum $\y^\delta$ and regularization parameter $\alpha > 0$ are given. We set $\mathcal{T}_{\alpha,\yd}[\x] = +\infty$ if $\x \notin \dom{\opo}$.

Next, in \autoref{sec:mns} we narrow down the set of solutions that are effectively approximated by minimizers of $\T_{\alpha,\y^\delta}$. A summary of the most fundamental notions and goals of Tikhonov regularization is presented in \autoref{sec:varcon}. In \autoref{sec:tikhonov} we outline the \emph{qualitative} analysis, which is relatively standard and follows classical textbooks such as \cite{EngHanNeu96,Mor84,TikArs77,TikGonSteYag95}. The \emph{quantitative} analysis is summarized in \autoref{ss:cr}. \autoref{sec:fda} is concerned with the finite-dimensional approximation of Tikhonov regularization. Finally, in \autoref{sec: Polregres} we present a brief stochastic version of the preceding deterministic analysis. %These results, however, are central for the comparison of different approximation strategies, for instance, finite element and machine learning approximations.

\subsection{Minimum-norm solutions}\label{sec:mns}

%On the one hand, the following definition can be motivated by the lack of uniqueness of solutions of \autoref{eq:Op} and the availability of an initial guess $\x^0$. 
In general, solutions of \autoref{eq:Op} are not unique. At the same time we assume to have an initial guess $\x^0$. Therefore, \autoref{de:xMNS} below is clearly a natural one. On the other hand, it is evident from the definition of $\T_{\alpha,\yd}$ that its minimizers $\xad$ are forced to be close to $\x^0$. Consequently, as will be shown in \autoref{th:ip:convergence}, if the $\xad$ converge to a solution of \autoref{eq:Op} for $\alpha, \delta \to 0$, then it must be an $\x^0$-minimum-norm solution. 
\begin{definition} \label{de:xMNS}
	A solution $\xdag \in \X$ of \autoref{eq:Op} is called $\x^0$-\emph{minimum-norm solution} if it satisfies
	\begin{equation*} %\label{eq:xmns}
		\norm{\xdag-\x^0} = \inf \set{ \norm{\x-\x^0}: \x \in \dom{\opo} \text{ and } F[\x]=\y}\;.
		\index{Solution!Minimal Norm}
	\end{equation*}
\end{definition}
Note that in general $\x^0$-minimum-norm solutions need not exist, and if they do, they need not be unique. However, if \autoref{eq:Op} has a solution and the operator is weakly sequentially closed, then existence of an $\x^0$-minimum-norm solution is guaranteed. %Weak sequential closedness is also the primary assumption on the operator $\opo$ for the qualitative analysis in \autoref{sec:tikhonov}. We therefore highlight it. \commentS{I would leave out the last two sentences.}
% While $\x^0$-minimum-norm solutions still are not unique in general, 
%
%We can apply these results of the Tikhonov-functional, if we use properties of $F$, which enforce weak lower semi-continuity of the Tikhonov-functional:
%For proving stability and convergence  is \emph{weak closedness}\index{operator!weakly closed}, which can be thought of as a weak form of continuity.
\begin{definition}\label{de:wc})
A mapping $\opo: \mathcal{D}(F) \subset \X \to \Y$ is \emph{weakly sequentially closed}, if for every sequence $(\xk) \subset \mathcal{D}(F)$ we have:
		\begin{equation*}
			\text{If }\xk \rightharpoonup_{\X} \x \text{ and } \op{\x_k} \rightharpoonup_{\Y} \tilde{\y}, \text{ then } \x \in \mathcal{D}(F) \text{ and } F[\x]=\tilde{\y}\;.
		\end{equation*}
\end{definition}
From now on we mostly omit the subscripts in the norms, inner products, etc., as the spaces are easily identified from the context. For later reference we highlight the assumptions required for the existence of $\x^0$-minimum-norm solutions.
\begin{assumption}
	\label{as:ip:weakly-closed}
		\autoref{eq:Op} has a solution and $\opo$ is weakly sequentially closed. % \commentS{I would subdivide this into a definition and an assumption. Here the def of weak sequantial closedness and the assumption of Equation 3.1 having a solution and F being weakly sequentially closed is combined. The previous sentence tries to mix both but in my opinion the mix does not make it more readible.} That is, for every sequence $(\xk) \subset \mathcal{D}(F)$ we have:
%		\begin{equation*}
%			\text{If }\xk \rightharpoonup_{\X} \x \text{ and } \op{\x_k} \rightharpoonup_{\Y} \tilde{\y}, \text{ then } \x \in \mathcal{D}(F) \text{ and } F[\x]=\tilde{\y}\;.
%		\end{equation*}
\end{assumption}
%The following definition is natural, since solutions of \autoref{eq:op} are typically not unique and we assume to have an initial guess $\x^0$.
\begin{lemma}
	\label{le:ip:existence}
	Let \autoref{as:ip:weakly-closed} be satisfied. Then there exists an $\x^0$-minimum-norm solution.
\end{lemma}
\begin{proof}
	The set $\opo^{-1}[\y]$ is non-empty by assumption. Therefore, it contains a sequence $(\xk)$ to such that
	\begin{equation*}
		\norm{\xk-\x^0} \to c := \inf \set{ \norm{\x-\x^0}: \x \in \mathcal{D}(F) \text{ and } F[\x]=\y}\;.
	\end{equation*}
	Since this sequence is bounded, it has a weakly convergent subsequence (see \autoref{th:weak_convergence}) with limit $\tilde{\x}$.
	Weak lower semi-continuity of the norm implies $\norm{\tilde{\x} - \x^0} \le c$, while weak sequential closedness of $\opo$ gives $\op{\tilde{\x}} = \y$.
	Therefore, $\tilde{\x}$ is an $\x^0$-minimum-norm solution.	
\end{proof}

\subsection{Basic assumptions and notation} \label{sec:varcon}
Before we start with the analytical considerations we summarize some notation and terminology, which is used throughout this book:
%\begin{notation} \label{no:general} \mbox{}
\begin{description}
	\item{\bf Spaces and operators:}
\begin{itemize}
	\item $\X$ and $\Y$ are separable Hilbert spaces.
	\item $\opo: \mathcal{D}(F) \subset \X \to \Y$ is a linear or nonlinear mapping.
\end{itemize}	
	\item{\bf Data and noise level:}
	%\commentC{Denote exact data by e.g.\ $\y^\dagger$ so that $\y$ can be used for generic elements of $\Y$?}
	\begin{itemize}
	\item $\y, \y^\delta \in \Y$  are the unperturbed and noisy data, respectively.
	\item $\delta \geq 0$ is the noise level, that is,
		\begin{equation}\label{eq:noise-level}
			\norm{\y - \yd}_{\Y} \le \delta
		\end{equation}
\end{itemize}	
\item{\bf Priors and solution:}
\begin{itemize}
	\item $\x^0, \x_0 \in \X$ %\commentS{I would unify the notation.}
		are prior guesses of a solution of \autoref{eq:Op}.
	\item $\xdag$ is an $\x^0$-minimum-norm solution, see \autoref{de:xMNS}. %\commentS{I am not a huge fan of referencing forward. I would prefer to put the chapter of the $x^0$-MNS before the defintion of the Tikhonov functional. Also later it would be nicer to announce convergence at a point where the notion of the object that the Tikhonov functional is converging to is already introduced. }
\end{itemize}	
\item{\bf Regularization:}
\begin{itemize}
	\item $\alpha > 0$ is the regularization parameter balancing data fidelity with regularity (i.e.\ closeness to $\x^0$) of minimizers. %\commentS{I would rather add that after the definition of the Tikhonov functional.}
	\item $\xad$ is a regularized solution, i.e.\ a minimizer of $\mathcal{T}_{\alpha,\y^\delta}$. 
	\item $\xa$ is a regularized solution for exact data, i.e.\ a minimizer of $\mathcal{T}_{\alpha,\y}$.
\end{itemize}
\end{description}
%\end{notation}
\noindent
The following five results are of main interest for Tikhonov regularization.
\begin{itemize}
	\item{\emph{Existence}} (\autoref{th:ip:well-posedness}): For every $\alpha > 0$
	and every $\tilde{\y} \in \Y$ the functional $\mathcal{T}_{\alpha,\tilde{\y}}$ has a minimizer.
	\index{regularization!existence}
	\index{regularization!well-posedness}\index{Well-Posedness}
	This means that the method of Tikhonov regularization is \emph{well-defined}.
	\item{\emph{Stability}} (\autoref{th:ip:stability}) ensures that for $\alpha>0$ fixed, the regularized solution
	$\xad$ depends continuously on $\y^\delta$.
	\index{regularization!stability}\index{Stability}
	\item{\emph{Convergence}} (\autoref{th:ip:convergence}) guarantees that as $\delta \to 0$ the regularization parameter $\alpha$ can be chosen such that $\xad$ converges to a solution of \autoref{eq:op}.
	\index{regularization!convergence}
	\item{\emph{Convergence rates}} (\autoref{th:ip:rate}) quantify the speed of convergence of $\xad$ in terms of the noise level $\delta$.
	\index{regularization!convergence rates}
	\item{\emph{Stability estimates}} (\autoref{th:ip:qual_rates}) provide a bound on the distance between $\xad$
	and $\xa$ depending on $\delta$.
	\index{regularization!stability estimates}
\end{itemize}
%In the \emph{qualitative regularization analysis} that handles existence, stability and convergence, the following two fundamental results of functional analysis are used, which are based on \emph{weak convergence} arguments (see \autoref{sec:weak}).

\subsection{Analysis of Tikhonov regularization} \label{sec:tikhonov}
In the following we provide a review of the fundamental results concerning well-definedness, stability and convergence of Tikhonov regularization. For the proofs of this section we refer to standard references, in particular to \cite{EngHanNeu96,SchGraGroHalLen09}.
%\commentO{These proofs should be added!}
\begin{theorem}[Existence]
	\label{th:ip:well-posedness}
	Let \autoref{as:ip:weakly-closed} hold. Then, $\T_{\alpha, \tilde{\y}}$ has at least one minimizer for every $\alpha > 0$, $\x^0 \in \X$, and $\tilde{\y} \in \Y$.
\end{theorem}
\begin{proof}
	Let $(\x_k) \subset \mathcal{D}(F)$ be a minimizing sequence, that is, $\T_{\alpha, \tilde{\y}}(\x_k) \to \inf \T_{\alpha, \tilde{\y}}$. Then, the sequence $(\T_{\alpha, \tilde{\y}}(\x_k))$ is bounded and since $ \alpha \norm{\x_k - \x^0}^2 \le \T_{\alpha, \tilde{\y}}[\x_k]$, the sequence $(\x_k)$ is bounded as well.
	% for arbitrary $\epsilon > 0$ and large enough $k$ we have $\T_{\alpha, \tilde{\y}}[\x_k] \le \inf \T_{\alpha, \tilde{\y}} + \epsilon$.
	%\commentS{Is the $\varepsilon$ really necessary here? I think we only need that the sequence $\mathcal{T}_{\alpha,\tilde{y}}[x_k]$ is bounded which directly follows from the convergence. The $\varepsilon$-argument might be a little bit misleading if one reads this for the first time.}
	%On the other hand, $\T_{\alpha, \tilde{\y}}[\x_k] \ge \alpha \norm{\x_k - \x^0}^2$, showing that the sequence is bounded.
	Due to \autoref{th:weak_convergence} we can extract a subsequence $(\x_{k_j})$ that converges weakly to an $\hat \x \in \X.$
	
	Similarly, the sequence $(\op{\x_{k_j}}) \subset \Y$ is also bounded
	%, because $\norm{\op{\x_{k_j}} - \tilde \y}^2 \le \inf \T_{\alpha, \tilde{\y}} + \epsilon$ \commentS{Here I have the same remark as above, I think the $\varepsilon$ is not really needed.} for large enough $j$.
	and has a subsequence that converges weakly to a $\hat \y \in \Y$. The weak closedness of $\opo$ now implies that $\hat \x \in \dom{\opo}$ and that $\op{\hat \x} = \hat \y$. In fact, the same argument shows that every subsequence of $(\op{\x_{k_j}})$ has a another subsequence weakly converging to $\op{\hat \x}$. Thus, $\op{\x_{k_j}} \rightharpoonup \op{\hat \x}$. 

Since the norms are weakly lower semicontinuous, cf.\ \autoref{th:ls}, we have
	\begin{align*}
		\T_{\alpha,\tilde y}[\hat \x] \le \liminf_{j \to \infty} \T_{\alpha,\tilde y}[\x_{k_j}] = \inf \T_{\alpha,\tilde y},
	\end{align*}
and we have shown that $\hat \x$ minimizes $\T_{\alpha, \tilde{\y}}$.
\end{proof}
%%While knowledge of the noise level, as stated in \autoref{eq:datn}, is essential for a convergence analysis (see for instance~\cite{Bak84,EngGre94}), it is important to note that data driven parameter choice strategies, in general do not take this information into account. See \autoref{ch:par}. \commentS{I would connect this comment more to the current topic (existence and stability of Tikhonov). It is not clear to me why this is noted here.}

\begin{theorem}[Stability]
	\label{th:ip:stability}
	Let \autoref{as:ip:weakly-closed} hold.
	Consider sequences $(\xk) \subset \X$ and $(\y_k) \subset \Y$ where $\y_k \to \tilde{\y}$ and $\xk$ is a minimizer of $\mathcal{T}_{\alpha,\y_k}$, $k\in\N$.
	Then $(\xk)$ has a strongly convergent subsequence and the limit of every such subsequence is a minimizer of $\mathcal{T}_{\alpha,\tilde{\y}}$.
\end{theorem}
%\commentC{Weak convergence $\y_k \rightharpoonup \tilde{\y}$ should be enough.}
\begin{proof}
%	By definition of a Tikhonov regularized solution we have for some $\x \in \dom{F}$, which we assume to be non-empty in \autoref{as:ip:weakly-closed},
%	\begin{equation} \label{eq:sth1}
%		\begin{aligned}
%			\mathcal{T}_{\alpha,\y_k}[\xk] &= \norm{F[\xk]-\y_k}^2 + \alpha \norm{\xk-\x^0}^2\\
%			&\leq \norm{F[\x]-\y_k}^2 + \alpha \norm{\x-\x^0}^2\;.
%		\end{aligned}		
%	\end{equation}
	Take an arbitrary $\x \in \dom{F}$. Then $\T_{\alpha,\y_k}[\xk] \le \T_{\alpha,\y_k}[\x]$. Together with the boundedness of $\y_k$ it follows that $(\xk)$ and $(F[\xk])$ are bounded as well. As in the proof of \autoref{th:ip:well-posedness}, \autoref{th:weak_convergence} and the weak closedness of $\opo$ imply the existence of a subsequence, denoted again by $(\xk)$, such that $\xk \rightharpoonup \tilde \x \in \dom{\opo}$ and $\op{\xk} \rightharpoonup \op{\tilde \x}$. The weak lower semi-continuity of the norms (see \autoref{th:ls}) gives
%	Because $\y_k \to \tilde{\y}$, we deduce from this estimate that $(\xk)$ and $(F[\xk])$ are uniformly bounded. Therefore (see \autoref{th:weak_convergence}) a subsequence of $(\xk)$ exists, denoted again by $(\xk)$ for the simplicity of notation, and elements $\tilde{\x} \in \X$ and $\tt z \in \Y$ such that
%	\begin{equation*}
%		\xk \rightharpoonup \tilde{\x} \text{ and } F[\xk] \rightharpoonup {\tt z}\;.
%	\end{equation*}
%	From \autoref{as:ip:weakly-closed} to follows that ${\tt z}=F[\tilde{\x}]$.
%
%	\begin{equation} \label{eq:contr}
%		\norm{\tilde{\x}-\x^0} \leq \liminf \norm{\xk-x^0} \text{ and }
%		\norm{F[\tilde{\x}]-\yd} \leq \liminf \norm{F[\xk]-\y_k},
%	\end{equation}
%	
%	Moreover, \autoref{eq:sth1} implies that for every $\x \in \dom{F}$ (note that $\xk$ is a minimizer of the Tikhonov functional with data $\y_k$)
	\begin{align} \label{eq:sthe}
		%\begin{aligned}
			\mathcal{T}_{\alpha,\tilde \y}[\tilde{\x}]
				%&= \norm{F[\tilde{\x}]-\yd}^2 + \alpha \norm{\tilde{\x}-\x^0}^2 \\
            	&\leq \liminf_{k \to \infty} \T_{\alpha,\y_k}[\xk] 
            		\leq \limsup_{k \to \infty} \T_{\alpha,\y_k}[\xk] 
            		\leq \limsup_{k \to \infty} \T_{\alpha,\y_k}[\x]
					= \lim_{k \to \infty} \T_{\alpha,\y_k}[\x] \notag \\
				&= \T_{\alpha,\tilde \y}[\x] . 
		%\end{aligned}		
	\end{align}
	Since $\x \in \dom{F}$ was arbitrary, $\tilde{\x}$ is a minimizer of $\mathcal{T}_{\alpha,\tilde \y}$. Thus we have shown that the sequence of minimizers $(\xk)$ has a weakly convergent subsequence and the limit of every such subsequence is a minimizer of $\mathcal{T}_{\alpha,\tilde{\y}}$. It remains to show strong convergence.
	
	Letting $\x = \tilde{\x}$ in \autoref{eq:sthe} gives
    \begin{equation*} %\label{eq:stheii}
		\lim_{k \to \infty} \T_{\alpha,\y_k}[\xk] = \T_{\alpha,\tilde \y}[\tilde \x].		
	\end{equation*}
	This implies $\norm{\xk - \x^0} \to \norm{\tilde \x - \x^0}$, because
	\begin{align*}
		\limsup_{k \to \infty} \alpha \norm{\xk - \x^0}^2
			&= \limsup_{k \to \infty} \left( \T_{\alpha,\y_k}[\xk] - \norm{\op{\xk} - \y_k}^2 \right) \\
			&\le \limsup_{k \to \infty} \T_{\alpha,\y_k}[\xk] + \limsup_{k \to \infty} \left( - \norm{\op{\xk} - \y_k}^2 \right) \\
			&= \T_{\alpha,\tilde \y}[\tilde \x] - \liminf_{k \to \infty} \norm{\op{\xk} - \y_k}^2 \\
			&\le \T_{\alpha,\tilde \y}[\tilde \x] - \norm{\op{\tilde \x} - \tilde \y}^2 \\
			&= \alpha \norm{\tilde \x - \x^0}^2.
	\end{align*}
	Combining convergence of norms with weak convergence yields strong convergence:
	\begin{align*}
		\norm{\xk - \tilde \x}^2
			&= \norm{\xk - \x^0 - ( \tilde \x - \x^0) }^2 \\
			&= \norm{\xk - \x^0}^2 - 2 \langle \xk - \x^0, \tilde \x - \x^0 \rangle + \norm{\tilde \x - \x^0}^2 \to 0.
	\end{align*}

	This ends the proof.
%	Now, assume that $(x_k)$ does not converge (strongly) to $\tilde{\x}$, then
%	\begin{equation*}
%		c:=\limsup \norm{\xk -x^0} > \norm{\tilde{\x}-x^0}
%	\end{equation*}
%	and there exists a subsequence, again denoted by $(x_k)$, such that
%	\begin{equation*}
%		\x_k \rightharpoonup \tilde{\x},\quad F[\x_k] \rightharpoonup F[\tilde{\x}] \text{ and } \norm{\xk-\x^0} \to x\;.
%	\end{equation*}
%	Therefore, we get from \autoref{eq:stheii}
%	    \begin{equation*}
%			\lim_{k \to \infty} \norm{F[\xk]-\y_k}^2 \leq
%			\norm{F[\tilde{\x}]-\yd}^2 + \alpha \left(\norm{\tilde{\x}-\x^0}^2-c\right) < \norm{F[\tilde{\x}-\tilde \y}^2 \,,
%	\end{equation*}
%	which is a contradiction to m\autoref{eq:contr}.
\end{proof}

The following theorem shows that minimization of $\mathcal{T}_{\alpha, \y^\delta}$ indeed is a regularization method: An exact solution of \autoref{eq:Op} can be recovered in the limit $\delta \to 0$, if $\alpha = \alpha(\delta)$ is chosen appropriately.
\begin{theorem}[Convergence]
	\label{th:ip:convergence}
	Let \autoref{as:ip:weakly-closed} hold. Assume $\alpha: (0,\infty) \to (0,\infty)$ satisfies
	\begin{equation} \label{eq:noise_cond}
		\boxed{\
			\alpha(\delta)\to 0
			\quad \text{and} \quad \frac{\delta^2}{\alpha(\delta)} \to 0\,,
			\text{ as } \delta \to 0\;.\
		}
	\end{equation}
	Consider sequences $(\delta _k),(\alpha_k), (\xk), (\y_k)$ where $\delta_k \to 0$, $\alpha_k = \alpha(\delta_k)$, $\norm{\y-\y_k} \leq \delta_k$ and $\xk \in \argmin \mathcal{T}_{\alpha_k,\y_k}$, $k\in\N$. Then $(\xk)$ has a strongly convergent subsequence and the limit of every such subsequence is an $\x^0$-minimum-norm solution. If there is only one $\x^0$-minimum-norm solution $\xdag$, then
	\begin{align*}
		\boxed{ \lim_{k \to \infty }\xk = \xdag.}	
	\end{align*}		
\end{theorem}
\begin{proof}
	Let $\xdag$ be a $\x^0$-minimum-norm solution. % and set $\alpha_k = \alpha(\delta_k)$. 
	%\commentS{I would define this already in the statement of the proof. The $\mathcal{T}_{\alpha(\delta_k),y_k}$ above is a bit bulky anyway.}
	Then, 
	\begin{align}\label{eq:conv-est}
		\alpha_k \norm{\xk - \x^0}^2 \le \T_{\alpha_k, \y_k}[\xk] \le \T_{\alpha_k, \y_k}[\xdag] \le \delta_k^2 + \alpha_k \norm{\xdag - \x^0}^2.
	\end{align}
	Using \eqref{eq:noise_cond} it follows that $(\xk)$ is bounded. As the sequence $(\op{\xk})$ is also bounded, we conclude from the weak closedness of $\opo$, as we did in the previous two proofs, that $\op{\x_{k_j}} \rightharpoonup \op{\tilde \x}$ for a suitable subsequence and some $\tilde \x \in \dom{\opo}$.
	%, so that we can extract a subsequence converging weakly to an $\tilde \x \in \X$. As the sequence $(\op{\xk})$ is also bounded, the weak closedness of $\opo$ ensures that $\tilde \x \in \dom{\opo}$ and $\op{\x_{k_j}} \rightharpoonup \op{\tilde \x}$ for a suitable subsequence. %\commentS{I would subdivide the two arguments. 1. $F[x_k]$ is bounded therefore there exists a weakly converging subsequence $F[x_{k_j}]$. 2. Due to the fact that $F$ is weakly closed, this weak limit is equal to $F[\tilde{x}]$.}
	
	Next, $\tilde \x$ is a solution of \autoref{eq:Op}, because
	\begin{align*}
		\norm{\op{\tilde \x} - \y}^2
			&\le \liminf_{j \to \infty} \norm{\op{\x_{k_j}} - \y}^2 \le \liminf_{j \to \infty} \left( 2 \norm{\op{\x_{k_j}} - \y_{k_j}}^2 + 2 \delta_{k_j}^2 \right) \\
			&\le \liminf_{j \to \infty} \left( 2 \T_{\alpha_{k_j}, \y_{k_j}}[\x_{k_j}] + 2 \delta_{k_j}^2 \right)
	\end{align*}
	and the right-hand side vanishes due to \eqref{eq:conv-est} and \eqref{eq:noise_cond}. Estimate \eqref{eq:conv-est} also implies that
	\begin{align*}
		\limsup_{j \to \infty} \norm{\x_{k_j} - \x^0}^2 \le \norm{\xdag - \x^0}^2,
	\end{align*}
	and consequently
	\begin{align*}
		\norm{\tilde \x - \x^0}^2
			&\le \liminf_{j \to \infty} \norm{\x_{k_j} - \x^0}^2 \le \limsup_{j \to \infty} \norm{\x_{k_j} - \x^0}^2 \\
			&\le \norm{\xdag - \x^0}^2 \le \norm{\tilde \x - \x^0}^2.
	\end{align*}
	We conclude that $\tilde \x$ is an $\x^0$-minimum-norm solution and, using the same argument as in the proof of \autoref{th:ip:stability}, that $(\x_{k_j})$ converges in the norm.
	
	If there is only one $\x^0$-minimum-norm solution, then $\tilde \x = \xdag.$ Moreover, every subsequence of $(\x_k)$ has another subsequence converging to the same limit. Therefore, $\xk \to \xdag.$
\end{proof}

\begin{example}[$c$-example]\label{ex:c-analysis}
We return to the c-example from \autoref{ss:c} and check whether the results of this section apply, that is, we verify \autoref{as:ip:weakly-closed}. First, \autoref{thm:c-existence-uniqueness} provides a sufficient condition for existence of a solution of $\op{\x} = \y.$ The weak sequential closedness of $\opo$ is an immediate consequence of the following two properties: (i) weak sequential closedness of $\dom{\opo}$ and (ii) weak-to-weak sequential continuity of $\opo$. The former holds, because $\dom{\opo} = \mathcal{C}_\epsilon$ is convex (being the Minkowski sum of two convex sets) and closed in $L^2$. The latter follows from the proposition below.
\begin{proposition} Let $(\x_k) \subset \dom{\opo}$, $ \x_k \rightharpoonup \x$, then $\op{\x_k} \rightharpoonup \op{\x}$ in $W^{2,2}.$
\end{proposition}
\begin{proof}
	By virtue of being weakly convergent, the sequence $(\x_k)$ is bounded in $L^2$. Estimate \eqref{eq:c-H2-est} implies that the sequence $(\y_k) := (\op{\xk})$ is bounded in $W^{2,2}.$ Hence, we can extract a subsequence $(\y_{k_j})$ converging weakly in $W^{2,2}$ to some ${\tt w} \in W^{2,2} \cap W^{1,2}_0$. Letting $j \to \infty$ in the weak formulation
	\begin{equation*}
		\inner{\y'_{k_j}}{{\tt v}'}_{L^2} + \inner{\x_{k_j} \y_{k_j}}{\tt v}_{L^2} = \inner{\tt f}{\tt v}_{L^2}
	\end{equation*}
	gives
	\begin{equation*}
		\inner{{\tt w}'}{{\tt v}'}_{L^2} + \inner{\x {\tt w}}{\tt v}_{L^2} = \inner{\tt f}{\tt v}_{L^2}.
	\end{equation*}
	showing that ${\tt w} = \op{\x}$. The argument above shows that, in fact, every subsequence of $(\y_k)$ has another subsequence converging weakly in $W^{2,2}$ to $\op{\x}$. This proves the claim.
	\end{proof}
\end{example}

\subsection{Convergence rates} \label{ss:cr}

Because the convergence of regularized solutions can be arbitrarily slow, it is important to identify conditions that guarantee a rate of convergence and to determine that rate. The first convergence rates result for Tikhonov regularization of nonlinear inverse problems was established in \cite{EngKunNeu89}. Since then various extensions have been found; see, for instance, \cite{Hoh00,Neu89,Sch01a}. Source conditions, such as \autoref{eq:source_var}, are central for obtaining results of this kind.

\begin{assumption} \label{as:ip:rates} \mbox{}
\begin{enumerate}
	\item \label{it:mns} %There is an $\x^0$-minimum-norm solution $\xdag$ in the interior of $\dom{\opo}$.
		%There is an $\x^0$-minimum-norm solution $\xdag$ such that $\mathcal{B}_\rho(\xdag) \subset \dom{\opo}$ for a $\rho > 2\norm{\xdag - \x^0}$.
		There is an $\x^0$-minimum-norm solution $\xdag$ and a $\rho > \norm{\xdag - \x^0}$ such that $\mathcal{B}_\rho(\x^0) \subset \dom{\opo}$ .
	\item \label{it:lip} $\opo$ is Fr\'{e}chet differentiable at $\xdag$ and there is an $L\ge 0$ such that
%  		\begin{equation} \label{eq:lipschitz}
%  			\norm{\opo'[\xdag] - \opo'[\x]} \leq L\norm{\xdag - \x} \quad \text{for all } \x \in \operatorname{int}\mathcal{D}(F),
%		\end{equation}
		\begin{equation} \label{eq:lipschitz}
  			\norm{\op{\x} - \op{\xdag} - \opd{\xdag} \left( \x - \xdag \right)} \le \frac{L}{2} \norm{\x - \xdag}^2 %\qquad \text{for all } \x \in \mathcal{B}_\rho(\x^0).
		\end{equation}
		for all $\x \in \mathcal{B}_\rho(\x^0)$.
		%where the norm on the left is the operator norm (see \autoref{de:op_norm}).
	\item \label{it:sourcecon} There is an $\omega \in \Y$ satisfying the \emph{source condition}\index{source condition}
		\begin{equation} \label{eq:source_var}
			\boxed{\xdag - \x^0 = \opo'[\xdag]^*\omega \;, \quad L\norm{\omega} < 1}\;,
		\end{equation}
		where $\opo'[\xdag]^*$ is the adjoint of $\opo'[\xdag]$.
\end{enumerate}
\end{assumption}

\begin{theorem}[Convergence rates] \label{th:ip:rate}
Let Assumptions \ref{as:ip:weakly-closed} and \ref{as:ip:rates} hold. % and let $\mathcal{D}(F)$ be convex.
If, as $\delta \to 0$, inequality %\commentS{I would rather write if (3.3) holds for all $\delta>0$..}
\eqref{eq:noise-level} holds and
\begin{equation} \label{eq:order}
	\boxed{\alpha = \alpha(\delta) \sim \delta\,,}
\end{equation}
%\commentS{Not sure if this is correct. Below one can see that $\frac{\delta^2}{\alpha}$ is supposed to converge to $0$, which does not hold if the above is assumed.}
which means that there are $c,C >0$ such that $c\delta \leq \alpha \leq C \delta$ for sufficiently small $\delta$, then
\begin{equation} \label{eq:delta_rate}
	\boxed{\norm{\xad - \xdag} = \mathcal{O} ( \sqrt{\delta} ).}
\end{equation}
\end{theorem}
\begin{proof}
	We show first that $\xad \in \mathcal{B}_\rho(\x^0)$ for sufficiently small $\delta$. From the definition of $\xad$ we obtain
	\begin{equation} \label{eq:conv-rate-pf}
		\begin{aligned} \hspace{2em}&\hspace{-2em} 
			\norm{\op{\xad} - \yd}^2 + \alpha \norm{\xad - \xdag}^2  \\
				&=		\T_{\alpha,\yd}[\xad] + \alpha \left( \norm{\xad - \xdag}^2 - \norm{\xad - \x^0}^2 \right)  \\
				&\le	\T_{\alpha,\yd}[\xdag] + \alpha \left( \norm{\xad - \xdag}^2 - \norm{\xad - \x^0}^2 \right) \\
				&\le	\delta^2 + \alpha \left( \norm{\xdag - \x^0}^2 + \norm{\xad - \xdag}^2 - \norm{\xad - \x^0}^2 \right)  \\
				&=		\delta^2 + 2\alpha \left\langle \xdag - \x^0, \xdag - \xad \right\rangle.
		\end{aligned}
	\end{equation}
	It follows that
	\begin{align*}
		\norm{\xad - \xdag}^2	&\le \frac{\delta^2}{\alpha} + 2\left\langle \xdag - \x^0, \xdag - \xad \right\rangle,
	\end{align*}
	and after completing the square
	\begin{align*}
		\norm{\xad - \x^0}	&\le \frac{\delta}{\sqrt{\alpha}} + \norm{\x^0 - \xdag}.
	\end{align*}
%	Hence
%	\begin{align*}
%		\norm{\xad - \xdag}	&\le \frac{\delta}{\sqrt{\alpha}} + 2\norm{\x^0 - \xdag},
%	\end{align*}	
	This together with \eqref{eq:order} implies $\xad \in \mathcal{B}_\rho(\x^0)$ for sufficiently small $\delta$. Therefore, inequality \eqref{eq:lipschitz} holds for $\x = \xad$.
	
	We now resume the estimate \eqref{eq:conv-rate-pf} using the source condition \eqref{eq:source_var}, the Cauchy-Schwarz inequality and \eqref{eq:lipschitz}
	\begin{align*} \hspace{2em}&\hspace{-2em}
	\norm{\op{\xad} - \yd}^2 + \alpha \norm{\xad - \xdag}^2 \\
			&\le \delta^2 + 2\alpha \left\langle \xdag - \x^0, \xdag - \xad \right\rangle \\
			&=		\delta^2 + 2\alpha \left\langle \omega, \opd{\xdag} \left( \xdag - \xad \right) \right\rangle \\
			&=		\delta^2 + 2\alpha \left\langle \omega, \op{\xad} - \op{\xdag} - \opd{\xdag} \left( \xad - \xdag \right) + \op{\xdag} - \yd - \op{\xad} + \yd \right\rangle \\
		%\intertext{and apply the Cauchy-Schwarz inequality as well as \eqref{eq:lipschitz}, yielding}
			&\le 	\delta^2 + 2\alpha \norm{\omega} \left( \frac{L}{2} \norm{\xad - \xdag}^2 + \delta + \norm{\op{\xad} - \yd} \right).
	\end{align*}
	This inequality can be expressed as
%	\begin{align*} \hspace{5em}&\hspace{-5em} 
%		\norm{\op{\xad} - \yd}^2 + \alpha \left( 1 - L \norm{\omega} \right) \norm{\xad - \xdag}^2 \\ &\le \delta^2 + 2 \alpha \delta \norm{\omega} + 2\alpha \norm{\omega} \norm{\op{\xad} - \yd}.
%	\end{align*}
%	This inequality is of the form $a^2 + d^2 \le c^2 + ba$, for $a,b,c,d\ge 0$, which implies $d \le b + c$. Equivalently,
%	\begin{align*}
%		\norm{\xad - \xdag} \le \frac{2\alpha \norm{\omega} + \sqrt{\delta^2 + 2 \alpha \delta \norm{\omega}}}{\sqrt{\alpha \left( 1 - L \norm{\omega} \right)}}.
%	\end{align*}
	\begin{align*}
		\left( \norm{\op{\xad} - \yd} - \alpha \norm{\omega} \right)^2 + \alpha \left( 1 - L \norm{\omega} \right) \norm{\xad - \xdag}^2 \le \left( \delta + \alpha \norm{\omega} \right)^2.
	\end{align*}
	%\commentS{The first square in the above inequality should be remove. There is no term of fourth order in the inequalities above.}
	Dropping the first term and recalling that $L \norm{\omega} < 1$, it implies
	\begin{align*}
		\norm{\xad - \xdag} \le \frac{\delta + \alpha \norm{\omega} }{\sqrt{\alpha \left( 1 - L \norm{\omega} \right)}}.
	\end{align*}
	Using \eqref{eq:order}, the assertion follows.
\end{proof}
\begin{remark} \label{rem:rates}$ $
\begin{enumerate}
%\begin{remark}\label{re:par-choice}
	%The regularization parameter $\alpha$ is chosen \emph{a priori} depending on the noise level $\delta$. This means that it is determined before the regularized solution is actually computed. Alternatively \emph{a-posteriori}, \emph{data-driven} and \emph{heuristic} regularization parameter selection criteria can be used, which are discussed in \autoref{ch:par}.
	\item Equations \eqref{eq:noise_cond}and \eqref{eq:order}, but also \eqref{eq:ps1}, \eqref{eq:ps2} and \eqref{eq:eta} below, present \emph{a-priori parameter choice rules}. They answer the question of how to choose $\alpha$ asymptotically in order to obtain convergence (rates) in the vanishing noise limit. Such rules are primarily of theoretical interest, not only because of their asymptotic nature but also because they depend on information that is typically not available in practice, such as a bound on the noise or smoothness of $\xdag$. This must be contrasted with the practical problem of selecting $\alpha$ for a given $\yd$ to obtain a good reconstruction $\xad$, ideally one minimizing $\norm{\xad - \xdag}$. The results of \autoref{se:Tik} only address the former. For \emph{a-posteriori}, \emph{data-driven} and \emph{heuristic} parameter selection strategies see \autoref{ch:par} and the references given there.
%\end{remark}
%\commentS{I would move this remark. Personally, I am not a huge fan of referencing forward. I think it would be a nice remark for after the convergence rates.}
	\item Note that \autoref{th:ip:rate} in particular establishes convergence $\xad \to \xdag$ without extracting a subsequence and without assuming uniqueness of $\xdag$. Compare this to \autoref{th:ip:convergence}. Therefore, it follows from \autoref{th:ip:rate} that there can be only one $\xdag$ satisfying \autoref{as:ip:rates}. %If, on the other hand, it is known a priori that there is only one $\x^0$-minimum-norm solution, then it already follows from \autoref{th:ip:convergence} that $\xad \to \xdag$. In this case the $\rho$ of \autoref{as:ip:rates} can be chosen arbitrarily. 
	\item \label{it:rem:lipschitz} Item \ref{it:lip} of \autoref{as:ip:rates} requires that the first order Taylor remainder of $\opo$ around $\xdag$ is $\mathcal{O}\big( \norm{\xdag - \x}^2 \big)$. This is the case, for instance, if $\opo$ is Fr\'{e}chet differentiable in $\mathcal{B}_\rho (\x^0)$ and the map $\x \mapsto \opd{\x}$ satisfies a Lipschitz-condition at $\xdag$, meaning that
  		\begin{equation*} 
  			\norm{\opd{\xdag} - \opd{\x}} \leq L\norm{\xdag - \x}
		\end{equation*}
		for all  $\x \in \mathcal{B}_\rho (\x^0)$, where the norm on the left is the operator norm between $\X$ and $\Y.$
	%\item It is sufficient, not only for \autoref{th:ip:rate} but also for \autoref{th:NeuSch90} and \autoref{th:NeuSch90b} below, if \eqref{eq:lipschitz} only holds locally around $\xdag$, that is, for $\x \in \mathcal{B}_\rho (\xdag) \subset \dom{\opo}$ where $\rho > 2 \norm{\xdag - \x^0}$ instead of everywhere in $\dom{\opo}$. We skip these details, which can be found in \cite{EngHanNeu96,Neu89,NeuSch90}, in order to simplify the proofs.
\end{enumerate}
%Although there may be more than one $\x^0$-minimum-norm solution $\xdag$, \autoref{th:ip:rate} implies that the $\xdag$ satisfying \autoref{as:ip:rates} must be unique.
\end{remark}
Next we state a stability estimate first derived in \cite{Sch93}. The theorem below is a slight generalization of \cite{ResSch06}, where in \autoref{eq:source_var} $2 L \norm{\omega} < 1$ was assumed to hold. %This makes the assumptions consistent with those of \autoref{th:ip:rate}.
We refer to \cite{SchGraGroHalLen25} for a proof.

%The theorem below is the Hilbert space version of the more general result from \cite{SchGraGroHalLen25}. This is a slight generalization of \cite{ResSch06}, where in \autoref{eq:source_var} $2 L \norm{\omega}_\Y < 1$ was assumed to hold. The slightly more general result makes the assumption consistent with the ones in \autoref{th:ip:rate}.

\begin{theorem}[Stability estimates]
	\label{th:ip:qual_rates}
	Let the assumptions of \autoref{th:ip:rate} hold. In addition, assume that $F$ is Fr\'{echet} differentiable in $\mathcal{B}_\rho(\x^0)$ and that
	\begin{gather}\label{eq:lipschitz2}
  		\norm{\op{\x} - \op{\tilde \x} - \opd{\tilde \x} \left( \x - \tilde \x \right)} \le \frac{L}{2} \norm{\x - \tilde \x}^2
	\end{gather}
	holds for all $\x, \tilde \x \in \mathcal{B}_\rho(\x^0)$. Then, for $\delta$ sufficiently small, we have
	\begin{equation*}
		\norm{\xad-\xa}^2 \leq \frac{4}{1-L \norm{\omega}}
		\frac{\delta^2}{\alpha}\;.
	\end{equation*}
	In particular, for $\alpha \sim \delta$, we have the
	stability estimate
	\begin{equation*}
		\norm{F[\xad]-F[\xa]}=\mathcal{O}(\delta)
		\qquad \text{and} \qquad
		\norm{\xad - \xa} =\mathcal{O}(\sqrt{\delta})\;.
	\end{equation*}
\end{theorem}
%\commentC{Add proof?}

\begin{remark} \label{opq:stability}
	Pairs of elements $\x$ and $\tilde \x$ which satisfy \eqref{eq:lipschitz2} and
	\begin{gather} \label{eq:sourcecon2}
		\tilde \x - \x = \opo'[\tilde \x]^*\omega
	\end{gather}
	for an $\omega \in \Y$ with $L\norm{\omega} < 1$, also satisfy
	\begin{equation*}
		\norm{\tilde \x-\x}_\X \leq C\sqrt{\norm{\op{\tilde \x} - F[\x]}}\,
	\end{equation*}
for an appropriate constant $C$. In other words, \eqref{eq:lipschitz2} and \eqref{eq:sourcecon2} implicitly characterize Hölder stability (see \cite{HofYam10,Hoh00,Sch01a}).
%Pairs of elements $\x$ and $\xdag$, which meet source conditions, i.e., they satisfy
%\begin{equation*}
%	\xdag - \x = \opo'[\xdag]^*\omega \;, \quad L\norm{\omega} < 1\;,
%\end{equation*}
%satisfy
%\begin{equation*}
%	\norm{\xdag-\x}_\X \leq C\sqrt{\norm{\op{\xdag} - F[\x]}}\,,
%\end{equation*}
%with an appropriate constant $C$. In other words, the source condition is an implicit characterization of Hölder-stability (see \cite{HofYam10,Hoh00,Sch01a}).
\end{remark}

\begin{example}[$c$-example]\label{ex:c-rates}
In order to investigate whether $\opo$, as defined in \eqref{eq:c-op}, meets \autoref{as:ip:rates} it is helpful to regard $A_\x : \y \mapsto -\y'' + \x\y $ as an operator
\begin{align*}
	A_\x : \dom{A_\x} = W^{2,2} \cap W^{1,2}_0 \subset L^2 \to L^2.
\end{align*}
Recall that all function spaces consist of real valued functions on the unit interval. First, we collect a few basic properties of $A_\x$.
\begin{proposition}\label{thm:A_x}
	$A_\x$ is a densely defined self-adjoint operator with a self-adjoint inverse.
\end{proposition}
\begin{proof}
	$A_\x$ is densely defined, because $C_c^\infty \subset \dom{A_\x}$ and $C_c^\infty$ is dense in $L^2.$ It is invertible due to Propositions \ref{thm:c-uniqueness} and \ref{thm:c-H2}. Integrating by parts twice yields
	\begin{align*}
		\langle A_\x \y , \z \rangle_{L^2} = \langle  \y , A_\x \z \rangle_{L^2}
	\end{align*}
	for all $\y, \z \in \dom{A_\x}$. Thus, $A_\x $ is self-adjoint. Self-adjointness of the inverse follows from that of $A_\x$. See, for instance, \cite[Thm.\ 1.8]{Schm12}.
\end{proof}
\begin{proposition}\label{thm:c-frechet}
	The Fr\'echet derivative of $\opo$ at $\x \in \operatorname{int} \dom{\opo}$ is given by
	\begin{align*}
		\opd{\x} \h = - A_\x^{-1} (\h \op{\x}), \quad \h \in L^2.
	\end{align*}
\end{proposition}
The next proposition shows that the operator-valued mapping $\x \mapsto \opd{\x}$ satisfies a Lipschitz-condition implying item \ref{it:lip} of \autoref{as:ip:rates}. Recall \autoref{rem:rates}, \ref{it:rem:lipschitz}.
\begin{proposition}
	Let $\x \in \operatorname{int}\dom{\opo}$. Then there is a constant $L = L(\epsilon,\x,\f)$ such that
	\begin{align*}
		\norm{\opd{\tilde \x} - \opd{\x}} \le L \norm{\tilde \x - \x}
	\end{align*}
	for all $\tilde \x \in \operatorname{int}\dom{\opo}$.
\end{proposition}
\begin{proof}
	The proof is a straightforward consequence of \autoref{thm:c-H2} and \autoref{thm:c-frechet}.
\end{proof}
\begin{proposition}\label{thm:c-adjoint}
	The adjoint of $\opd{\x}$ is given by
	\begin{align*}
		\opd{\x}^* \h = - \op{\x} A_\x^{-1} \h, \quad \h \in L^2,
	\end{align*}
	and there is a constant $C = C(\epsilon)$ such that
	\begin{align}\label{eq:c-adjoint-est}
		\norm{\opd{\x}^* \h}_{W^{2,2}} \le C (1 + \norm{\x}_{L^2})^2 \norm{\f}_{L^2} \norm{\h}_{L^2}.
	\end{align}
\end{proposition}
\begin{proof}
	The first claim is a direct consequence of \autoref{thm:c-frechet} using the self-adjointness of $A_\x^{-1}$, recall \autoref{thm:A_x}. Regarding the second one, note that there is a $C'$ such that $\norm{\y\z}_{W^{2,2}} \le C' \norm{\y}_{W^{2,2}} \norm{\z}_{W^{2,2}}$. In conjunction with \eqref{eq:c-H2-est1} this gives
	\begin{align*}
		\norm{\opd{\x}^* \h}_{W^{2,2}}
			&= \norm{\op{\x} A_\x^{-1} \h}_{W^{2,2}} 
				= \norm{\hat \y_{\x,\f} \hat \y_{\x,\h}}_{W^{2,2}}
				\le C' \norm{\hat \y_{\x,\f}}_{W^{2,2}}  \norm{\hat \y_{\x,\h}}_{W^{2,2}} \\
			& \le C'' (1 + \norm{\x}_{L^2})^2 \norm{\f}_{L^2} \norm{\h}_{L^2},
	\end{align*}
	finishing the proof.
\end{proof}
	We conclude this section by presenting a simple instance where the source condition is satisfied. As in \autoref{ss:c} let $\f \equiv 16$ and $\y(s) = 8s(1-s)$, so that $\x \equiv 0$ is the solution of $\op{\x} = \y$. Since it is unique, $\xdag = \x \equiv 0$ is the $\x^0$-minimum-norm solution for every $\x^0 \in \X$. Due to \autoref{thm:c-adjoint} the source condition \eqref{eq:source_var} requires the existence of an $\omega \in L^2$ such that $\x^0 = \y A_0^{-1} \omega$. Such an $\omega$ exists if $\x^0 / \y \in \dom{A_0} = W^{2,2} \cap W^{1,2}_0$. This holds, for instance, if $\x^0 = \beta \phi$ where $\beta > 0$ and $\phi \in C_c^\infty$. The norm bound in \eqref{eq:source_var} becomes
	\begin{align*}
		\frac{1}{L} > \norm{\omega}_{L^2} = \norm{A_0(\x^0/\y)}_{L^2} = \beta \norm{(\phi/\y)''}_{L^2}
	\end{align*}
	and is surely met for sufficiently small $\beta.$
	\end{example}

\subsection{Finite-dimensional approximation} \label{sec:fda}
For the numerical realization of Tikhonov regularization the typically infinite-dimensional space $\X$ is approximated by a sequence of finite-dimensional subspaces $(\X_\ttm)_{\ttm \in \N}$.
Simultaneously, $\opo$ is approximated by a sequence of operators $(F_\ttn)_{\ttn \in \N}$. Moreover, we account for the practical limitation that it might not be possible to solve the resulting minimization problems exactly. Therefore, defining 
\begin{equation*} %\label{eq:tikdis}
	\T_{\alpha,\y^\delta}^{{\tt n}}[\x]:=\norms{\opdis{\ttn}{\x}-\y^\delta}^2+\alpha\norms{\x-\x^0}^2,
\end{equation*}
we approximate $\x^0$-minimum-norm solutions by elements $\xademn \in \X_\ttm$ satisfying
%\begin{multline} \label{eq:Tik_approx}
%	\norm{F_\ttn[\xademn] - \y^\delta}^2 + \alpha\norm{\xademn-\x^0}^2 \\
%		\leq \norm{\opdis{\ttn}{\x} -\y^\delta}^2 + \alpha\norm{\x-\x^0}^2 + \eta, \qquad \text{for all } \x \in C_\ttm,
%\end{multline}
\begin{equation} \label{eq:Tik_approx}
	\T_{\alpha,\y^\delta}^{{\tt n}}[\xademn] \le \T_{\alpha,\y^\delta}^{{\tt n}}[\x] + \eta, \qquad \text{for all } \x \in  \X_\ttm,
\end{equation} 
where $\eta >0$ and $\T_{\alpha,\y^\delta}^{{\tt n}}[\x] = + \infty$ if $\x \notin \dom{\opo_\ttn}$. %Note that such $\xademn$ always exist under the stated assumptions.In this section we study the question of convergence of (near) minimizers of $\T_{\alpha,\y^\delta}^{{\tt n}}$ within the spaces $\X_\ttm$ to solutions of \autoref{eq:Op}.
As will be seen below, in order to obtain a convergent scheme the approximations $\X_\ttm$ %\commentS{I would either only mention one space $X_m$ or the whole sequence $(X_m)_{m \in \N}$. I don't really like the in between notation.}
and $F_\ttn$ must be appropriately balanced with $\eta$, $\alpha$ and $\delta$.

The results of this section are due to \cite{Neu89,NeuSch90}.
%Since the results of this section are less standard than the previous ones, we include their proofs below.

\begin{assumption} \label{as:fda} \mbox{}
\begin{itemize}
	\item The sequence $(\X_\ttm)$ consists of nested finite-dimensional subspaces of $\X$ with a dense union, that is,
	\begin{equation} \label{eq:nested}
		\X_1 \subset \X_2 \subset \quad \ldots \quad \text{and} \quad
		{\overline{\bigcup_{\ttm \in \N} \X_\ttm}}= \X\;.
	\end{equation}
	The orthogonal projector onto $\X_\ttm$ is denoted by $P_\ttm$ and
	\begin{equation*}
		 C_\ttm := \mathcal{D}(F) \cap \X_\ttm \neq \emptyset, \quad \text{for all} \quad \ttm \in \N.
	\end{equation*}
	\item For every ${\ttn} \in \N$ the map $F_\ttn: \X \to \Y$ has domain $\dom{F_\ttn}=\mathcal{D}(F)$. %\commentS{has ... as a domain.}
	Furthermore, there is a non-negative sequence $\nu_{\ttn} \to 0$ and a constant $C(\rho)>0$ for every $\rho > 0$ such that
	\begin{equation} \label{eq:uniform_approx}
		\norm{\opdis{\ttn}{\x}  - F[\x]} \leq C(\rho) \nu_{\ttn} \quad \text{for all} \quad
		\x \in \mathcal{D}(F) \cap \mathcal{B}_\rho(\xdag), \, \ttn \in \N.
	\end{equation}
\end{itemize}
\end{assumption}
%\commentC{Adding the assumption that $\opo_\ttn$ is weakly closed, should let us get rid of $\eta$ everywhere.}
The next theorem provides sufficient conditions for the convergence of $\xademn$ to an $\x^0$-minimum-norm solution.
\begin{theorem}[Convergence of finite-dimensional approximations] \label{th:NeuSch90}
	%Let $\xdag \in \operatorname{int} \mathcal{D}(F)$ be an $\x^0$-minimum-norm solution.
	Let Assumptions \ref{as:ip:weakly-closed}, items \ref{it:mns} and \ref{it:lip} of \autoref{as:ip:rates} and \ref{as:fda} hold.
	Assume that, as $\ttm,{\ttn} \to \infty$ and $\delta,\eta \to 0$, the regularization parameter $\alpha = \alpha(\ttm,{\ttn},\delta,\eta)$ is such that
		\begin{equation} \label{eq:ps1}
			\boxed{\alpha \to 0, \quad \frac{\delta^2}{\alpha} \to 0, \quad \frac{\eta}{\alpha} \to 0, \quad \frac{\nu_{\ttn}^2} {\alpha} \to 0}
		\end{equation}
		and 
		\begin{equation} \label{eq:ps2}
			\frac{\gamma_\ttm \norm{(I-P_\ttm)\xdag} +
				\frac{L}{2}\norm{(I-P_\ttm)\xdag}^2}{\sqrt{\alpha}} \to 0\,,
		\end{equation}
	    where
		\begin{equation} \label{eq:gamma}
			\gamma_\ttm := \norm{(I-P_\ttm) \opo'[\xdag]^*}\;.
		\end{equation}
	Consider sequences $\ttm_k,\ttn_k \to \infty$, $\delta_k,\eta_k \to 0$, $\alpha_k := \alpha(\ttm_k,{\ttn}_k,\delta_k,\eta_k)$, and $(\y_k) \subset \Y$ with $\norm{\y-\y_k} \le \delta_k$. Then the sequence $(\x_k)$, where $\x_k := \x_{\ttm_k,{\ttn}_k}^{\alpha_k,\delta_k,\eta_k}$ satisfies \eqref{eq:Tik_approx} for $\y^{\delta_k} = \y_k$, has a convergent subsequence and the limit of every convergent subsequence is an $\x^0$-minimum-norm solution. If, in addition, the $\x^0$-minimum-norm solution $\xdag$ is unique, then
	\begin{equation*}
		\boxed{
		\lim_{k \to\infty} \x_k = \xdag.}
	\end{equation*}
\end{theorem}
\begin{proof}
%	Let $\rho>0$ be such that $\mathcal{B}_\rho(\xdag) \subset \mathcal{D}(F)$. For sufficiently large $\ttm \in \N$,
%	\begin{equation} \label{eq:proj}
%		P_\ttm\xdag \in C_\ttm \cap \mathcal{B}_\rho(\xdag),
%	\end{equation}
%	because $P_\ttm \xdag \to \xdag$ for $\ttm \to \infty$.
	
	Since $P_\ttm \xdag \to \xdag$ for $\ttm \to \infty$ and $\xdag \in \mathcal{B}_\rho (\x^0)$, we know that $P_\ttm \xdag \in \mathcal{B}_\rho (\x^0)$ as well for sufficiently large $\ttm \in \N$.
	Now,
	% since the whole line segment connecting $\xdag$ and $P_\ttm \xdag$ lies in $\mathcal{D}(F)$,
	\eqref{eq:lipschitz} gives %the following error bound for the linear approximation of $\opo$ around $\xdag$
	%\begin{equation}\label{eq:lin-approx}
	%	\norm{\op{P_\ttm \xdag} -\op{\xdag} - \opo'[\xdag](P_\ttm \xdag-\xdag)} \le \frac{L}{2} \norm{\xdag - P_\ttm \xdag}^2,
	%\end{equation}
	%hence
	\begin{equation}\label{eq:lin-approx2}
		\norm{\op{P_\ttm \xdag} -\op{\xdag}} \le \gamma_\ttm \norm{(I-P_\ttm)\xdag} + \frac{L}{2} \norm{\xdag - P_\ttm \xdag}^2.
	\end{equation}
	For large enough $k \in \N$, \eqref{eq:Tik_approx}, \eqref{eq:lin-approx2} and \eqref{eq:uniform_approx} imply
	\begin{equation}\label{eq:long-est}
		\begin{aligned}
		\hspace{1em}&\hspace{-1em} \norm{\opdis{\ttn_k}{\x_k} - \y_k }^2 + \alpha_k \norm{\x_k-\x^0}^2 \\
			&\leq \norm{\opo_{\ttn_k}[P_{\ttm_k} \xdag]-\y_k}^2 + \alpha_k \norm{P_{\ttm_k} \xdag - \x^0}^2 + \eta_k \\
			&\leq \Bigl(\norm{\opo_{\ttn_k}[P_{\ttm_k} \xdag] - F[P_{\ttm_k} \xdag]} + \norm{\opo[P_{\ttm_k} \xdag] - \op{\xdag}} + \delta_k \Bigr)^2 \\
				&\quad {} + \alpha_k \norm{P_{\ttm_k} \xdag - \x^0}^2 + \eta_k \\
			&\leq \Bigl(  C(\rho) \nu_{\ttn_k} + \gamma_{\ttm_k} \norm{(I-P_{\ttm_k})\xdag} + \frac{L}{2} \norm{(I-P_{\ttm_k})\xdag}^2 +  \delta_k \Bigr)^2 \\
				&\quad {} + \alpha_k \norm{P_{\ttm_k} \xdag - \x^0}^2 + \eta_k.
		\end{aligned}
	\end{equation}
	Dividing by $\alpha_k>0$ and letting $k \to \infty$ while using \eqref{eq:ps1} and \eqref{eq:ps2} yields
	\begin{equation} \label{eq:help_conv}
		\opdis{\ttn_k}{\x_k} \to \y \quad \text{and} \quad \limsup_{k \to \infty} \norm{\x_k-\x^0} \leq \norm{\xdag-\x^0}.
	\end{equation}
	%\commentS{If one would like to make this more readible one could add an in between step here. From the above one directly sees that $\norm{F_{n_k}[x_k]-y_k} \| \to 0$. The desired convergence $F_{n_k} \to y $ follows from the first observation and the fact that $y_k \to y$.} 
	\autoref{eq:uniform_approx} now implies that $(\x_k)$ has a subsequence $(\x_{k_j})$ such that
	\begin{equation}\label{eq:help_conv2}
		\op{\x_{k_j}} \to \y \quad \text{and} \quad \x_{k_j} \rightharpoonup \bar{\x}.
	\end{equation}
	Since $\opo$ is weakly sequentially closed, we have $\bar \x \in \mathcal{D}(F)$ and $\op{\bar \x} = \y$.
	Using \autoref{eq:help_conv} and the fact that $\xdag$ is an $\x^0$-minimum-norm solution we obtain
	\begin{equation}\label{eq:help_conv3}
		\begin{aligned}
		\norm{{\bar{\x}}-\x^0}^2
			& =  \lim_{j \to \infty} \inner{\bar{\x}-\x^0}{\x_{k_j} - \x^0}
				\leq \norm{{\bar{\x}}-\x^0} \limsup_{j \to \infty}  \norm{\x_{k_j} - \x^0} \\
			& \leq \norm{{\bar{\x}}-\x^0} \norm{\xdag-\x^0}
				\leq \norm{{\bar{\x}}-\x^0}^2 \;.
		\end{aligned}
	\end{equation}
	\par\noindent
	It follows that $\norm{{\bar{\x}}-\x^0} = \norm{\xdag-\x^0} = \lim \norm{\x_{k_j} - \x^0},$ from which we conclude that ${\bar{\x}}$ is an $\x^0$-minimum-norm solution and $\x_{k_j} \to \bar{\x}$. Thus we have proved the first claim.
	
	If $\xdag$ is unique, then $\bar{\x} = \xdag.$
	In this case, we can actually conclude from \eqref{eq:help_conv} that every subsequence of $(\x_k)$ has another subsequence which satisfies \eqref{eq:help_conv2} and \eqref{eq:help_conv3} and therefore converges to $\xdag$.
	Hence, the sequence $(\x_k)$ itself converges to $\xdag$.
\end{proof}

\begin{remark}\label{rem:compact-gamma}
	The sequence $(\gamma_\ttm)$, defined in \eqref{eq:gamma}, is bounded by the uniform boundedness principle \commentS{I think it would be nice to include the corresponding theorem (banach-steinhaus) in the appendix.}. In the next theorem we assume that $\gamma_\ttm \to 0$. This condition is met, for instance, if $\opo$ is compact, meaning that $\opo$ is continuous and $\overline{ \op{B}}$ is a compact subset of $\Y$ for every bounded $B \subset \X.$ In this case, $\opd{\xdag}$ is compact as well, see \cite[Prop.~8.2]{Dei85b}, and precomposition with a compact operator turns pointwise into uniform convergence \commentS{This statement would also be nice to formulate in the appendix. Reference: Dirk Werner.}. Hence $\gamma_\ttm \to 0.$
\end{remark}

\begin{theorem}[Convergence rates for finite-dim.~approximations]\label{th:NeuSch90b} \mbox{}
	Let Assumptions \ref{as:ip:weakly-closed}, \ref{as:ip:rates} and \ref{as:fda} hold. In addition, as $\delta \to 0$ and $\ttm \to \infty$, suppose that \eqref{eq:noise-level} holds, that
	%$\opo$ is compact, $\mathcal{D}(F)$ is convex and
	\begin{equation}\label{eq:proj-rate}
		\gamma_\ttm \to 0, \qquad \norm{(I-P_\ttm)\x^0} = \mathcal{O}(\gamma_\ttm),
	\end{equation}
	and that $\eta = \eta(\delta,\ttm)$, $\ttn = \ttn(\delta,\ttm)$, and $\alpha=\alpha(\delta,\ttm)$ are such that
	\begin{equation} \label{eq:eta}
		\boxed{
			\begin{aligned}
				\eta		= \mathcal{O}(\delta^2 + \gamma_\ttm^4),\qquad
				\nu_{\ttn}	= \mathcal{O}(\delta + \gamma_\ttm^2), \qquad
				\alpha	\sim \delta + \gamma_\ttm^2.  %\max \set{\delta,\gamma_\ttm^2} \\
			\end{aligned}
	        }
	\end{equation}
	Then
	\begin{equation} \label{eq:approx_rate}
		\boxed{
		\norm{\xademn - \xdag} = \mathcal{O}\left( \sqrt{\delta} + \gamma_\ttm \right)\;.}
	\end{equation}
\end{theorem}
\begin{proof}
	We start with deriving three auxiliary estimates \eqref{eq:aux-est}, \eqref{eq:aux-est2}, \eqref{eq:aux-est3}. First, the source condition (\autoref{it:sourcecon} of \autoref{as:ip:rates}) and \eqref{eq:proj-rate} give
	\begin{equation}\label{eq:aux-est} \tag{AE1}
		\begin{aligned}
			\norm{\xdag - P_\ttm \xdag}
				&\leq \norm{(I - P_\ttm)(\xdag - \x^0)} + \norm{(I - P_\ttm)\x^0} \\
				&\leq \frac{1}{L}\norm{(I - P_\ttm)\opo'[\xdag]^*} + \mathcal{O}(\gamma_\ttm) = \mathcal{O}(\gamma_\ttm).
		\end{aligned}
	\end{equation}
	Second, using \eqref{eq:lin-approx2} and \eqref{eq:aux-est} we have %we estimate $\norm{\opdis{\ttn}{P_\ttm \xdag} - \y^\delta }$ in the same way $\norm{\opdis{\ttn_k}{P_{\ttm_k} \xdag} - \y_k}$ was estimated in \eqref{eq:long-est} to obtain
		\begin{align}\label{eq:aux-est2}
			\hspace{1em}&\hspace{-1em}
			\norm{\opdis{\ttn}{P_\ttm \xdag} - \y^\delta } \le \norm{\opdis{\ttn}{P_\ttm \xdag} - \op{P_\ttm \xdag} } + \norm{\op{P_\ttm \xdag} - \op{\xdag}} + \norm{\y - \y^\delta } \notag \\
				&\leq \mathcal{O}(\nu_\ttn) + \mathcal{O}(\gamma_\ttm^2) + \delta \leq \mathcal{O}(\gamma_\ttm^2 + \delta). \tag{AE2}
		\end{align}
%where $\norm{\op{P_\ttm \xdag} - \op{\xdag}}$ could be estimated as follows:
%\begin{align*}
%  \norm{\op{P_\ttm \xdag} - \op{\xdag}} &\leq \gamma_\ttm \norm{(I - P_\ttm) \xdag} + \frac{L}{2}\norm{(I - P_\ttm) \xdag}^2 = \mathcal{O}(\gamma_\ttm^2).
%\end{align*}
	In order to derive the third estimate we show first that $\xademn$ eventually lies in $\mathcal{B}_\rho (\x^0)$. The definition of $\xademn$ in \eqref{eq:Tik_approx}, \eqref{eq:aux-est2} as well as the asymptoic behavior of $\eta$ in \eqref{eq:eta} imply
	\begin{align*}
		\alpha \norm{\xademn - \x^0}^2
			&\le \T_{\alpha,\y^\delta}^{{\tt n}}[\xademn] \le \T_{\alpha,\y^\delta}^{{\tt n}}[P_\ttm \xdag] + \eta \\
			&\le \mathcal{O}(\gamma_\ttm^4 + \delta^2) + \alpha \norm{P_\ttm \xdag - \xdag + \xdag - \x^0}^2.
	\end{align*}
	Consequently,
	\begin{align*}
		\norm{\xademn - \x^0} \le \mathcal{O} \left( \frac{\gamma_\ttm^2 + \delta}{\sqrt{\alpha}} \right) + \norm{P_\ttm \xdag - \xdag} + \norm{\xdag - \x^0},
	\end{align*}
	and we deduce from \eqref{eq:eta} and \eqref{eq:aux-est} that $\xademn \in \mathcal{B}_\rho (\x^0)$ for sufficiently large $\ttm, \ttn$ and sufficiently small $\delta, \eta$.
	
	Now, turning to the third auxiliary estimate, we use the self-adjointness of $P_\ttm$, the source condition and the Cauchy-Schwarz inequality to obtain
	{\allowdisplaybreaks
		\begin{align*}
			\hspace{1em}&\hspace{-1em}
			\inner{\xademn - P_\ttm \xdag}{\x^0 - P_\ttm \xdag} = \inner{P_\ttm \xdag - \xademn}{\xdag - \x^0} \notag \\
				&= 	\inner{P_\ttm \xdag - \xademn}{\opo'[\xdag]^*\omega} \notag \\
				&= 	\left\langle \opo'[\xdag](P_\ttm - I) \xdag + \op{\xdag} - \y^\delta + \y^\delta - \opdis{\ttn}{\xademn} + \opdis{\ttn}{\xademn} - \op{\xademn} \right. \notag \\
					&\quad \left. {}+ \op{\xademn} - \op{\xdag} - \opo'[\xdag](\xademn-\xdag), \omega \right\rangle \notag \\
				&\le 	\left( \norm{\opo'[\xdag](P_\ttm - I) \xdag} + \delta + \norm{\opdis{\ttn}{\xademn} - \y^\delta} + \norm{\opdis{\ttn}{\xademn} - \op{\xademn}} \right. \notag \\
					&\quad \left. {}+ \norm{\op{\xademn} - \op{\xdag} - \opo'[\xdag](\xademn-\xdag)} \right) \norm{\omega}. \notag
		\end{align*}
}
%where $\norm{\opo'[\xdag](P_\ttm - I) \xdag}$ could be estimated as follows:
%\begin{align*}
%\norm{\opo'[\xdag](P_\ttm - I) \xdag} &= \norm{\opo'[\xdag](I - P_\ttm)^2 \xdag}\\
%&\leq \norm{\opo'[\xdag](I - P_\ttm)}\norm{(I - P_\ttm)\xdag}\\
%&\leq \gamma_\ttm \norm{(I - P_\ttm) \xdag} = \mathcal{O}(\gamma_\ttm^2).
%\end{align*}
	Combining this inequality with \eqref{eq:aux-est}, \eqref{eq:uniform_approx}, the asymptotic behavior of $\nu_\ttn$ in \eqref{eq:eta} as well as \eqref{eq:lipschitz}, we get
	\begin{multline} \label{eq:aux-est3} \tag{AE3}
		\inner{\xademn - P_\ttm \xdag}{\x^0 - P_\ttm \xdag} \\
			\le \left( \mathcal{O}(\gamma_\ttm^2 + \delta) + \norm{\opdis{\ttn}{\xademn} - \y^\delta} + \frac{L}{2} \norm{\xademn - \xdag}^2 \right) \norm{\omega}.
	\end{multline}

	We now turn to the main estimate of this proof. Recalling that $\xademn \in \X_\ttm$ while $P_\ttm \xdag-\xdag\in \X_\ttm^\perp$ we can expand in the following way
	{\allowdisplaybreaks
	\begin{align*}
		\hspace{1em}&\hspace{-1em}
		\norm{\opdis{\ttn}{\xademn} - \y^\delta }^2 + \alpha \norm{\xademn-\xdag}^2 \\
			&=	\norm{\opdis{\ttn}{\xademn} - \y^\delta }^2 + \alpha \norm{\xademn-P_\ttm \xdag}^2 + \alpha \norm{P_\ttm \xdag-\xdag}^2 \\
			&\leq	\norm{\opdis{\ttn}{\xademn} - \y^\delta }^2 + \alpha \norm{\xademn - \x^0}^2 + 2\alpha \inner{\xademn - \x^0}{\x^0 - P_\ttm \xdag} \\
				&\quad {}+ \alpha \norm{\x^0 - P_\ttm \xdag}^2 + \alpha \mathcal{O}(\gamma_\ttm^2), \\
		\intertext{where we have also used \eqref{eq:aux-est}. The definition of $\xademn$ in \eqref{eq:Tik_approx} and the fact that $P_\ttm \xdag \in C_\ttm$ for large $\ttm$ (cf.\ the proof of \autoref{th:NeuSch90}), give }
			&\leq	\norm{\opdis{\ttn}{P_\ttm \xdag} - \y^\delta }^2 + \alpha \norm{P_\ttm \xdag - \x^0}^2 + \eta + 2\alpha \inner{\xademn - \x^0}{\x^0 - P_\ttm \xdag} \\
				&\quad {}+ \alpha \norm{\x^0 - P_\ttm \xdag}^2 + \alpha \mathcal{O}(\gamma_\ttm^2) \\
			&=	\norm{\opdis{\ttn}{P_\ttm \xdag} - \y^\delta }^2 + \eta + 2\alpha \inner{\xademn - P_\ttm \xdag}{\x^0 - P_\ttm \xdag} + \alpha \mathcal{O}(\gamma_\ttm^2). \\
		\intertext{Taking into account \eqref{eq:aux-est2}, \eqref{eq:aux-est3} and the asymptotic behaviour of $\eta$ in \eqref{eq:eta} we obtain}
			&\leq	\mathcal{O}(\gamma_\ttm^4 + \delta^2 + \alpha \gamma_\ttm^2) + \mathcal{O}(\alpha \delta) + 2 \alpha \norm{\omega} \norm{\opdis{\ttn}{\xademn}- \y^\delta }\\
				&\quad {} + \alpha L \norm{\omega} \norm{\xademn- \xdag}^2.
	\end{align*}
	}
	In total we have shown that
	\begin{multline*}
		\left( \norm{\opdis{\ttn}{\xademn} - \y^\delta } - \alpha \norm{\omega}\right)^2 + \alpha(1 - L \norm{\omega}) \norm{\xademn-\xdag}^2 \leq \\ \le \mathcal{O}(\gamma_\ttm^4 + \delta^2 + \alpha \gamma_\ttm^2) + \mathcal{O}(\alpha \delta) + \alpha^2 \norm{\omega}^2.
	\end{multline*}
	Noting that the right-hand side is a $\mathcal{O}(\alpha^2)$, due to \eqref{eq:eta}, the claim follows.
%	\begin{align*}
%		\norm{\xademn-\xdag} \le \sqrt{\frac{ \mathcal{O}(\gamma_\ttm^4 + \delta^2 + \alpha \gamma_\ttm^2) + \mathcal{O}(\alpha \delta) + \alpha^2 \norm{\omega}^2}{\alpha(1 - L \norm{\omega})}}.
%	\end{align*}
%	The claim now follows using \eqref{eq:eta}.
%
%
%	In total we have shown that
%	\begin{multline*}
%		\norm{\opdis{\ttn}{\xademn} - \y^\delta }^2 + \alpha(1 - L \norm{\omega}) \norm{\xademn-\xdag}^2 \leq \\ \le \mathcal{O}(\gamma_\ttm^4 + \delta^2 + \alpha \gamma_\ttm^2) + \mathcal{O}(\alpha \delta) + 2 \alpha \norm{\omega} \norm{\opdis{\ttn}{\xademn} - \y^\delta}.
%	\end{multline*}
%	This inequality is of the form $a^2 + b^2 \le c^2 + ad$, for $a,b,c,d\ge 0$, which implies $b \le c + d$. That is,
%	\begin{multline*}
%		\sqrt{\alpha(1 - L \norm{\omega})} \norm{\xademn-\xdag}
%			\leq \sqrt{\mathcal{O}(\gamma_\ttm^4 + \delta^2 + \alpha \gamma_\ttm^2) + \mathcal{O}(\alpha \delta)} + 2 \alpha \norm{\omega},
%	\end{multline*}
%	and the claim follows using \eqref{eq:eta}.
\end{proof}

\begin{remark} $ $
	\begin{itemize}
		\item Note that the asymptotic conditions \eqref{eq:ps1} and \eqref{eq:ps2} in \autoref{th:NeuSch90} can always be met by simply letting $\alpha \to 0$ sufficiently slowly. In this sense the situation is similar to \autoref{th:ip:convergence}. The parameter restrictions are, however, markedly different in \autoref{th:NeuSch90b}. There, $\alpha$, $\eta$ and $\ttn$ all must be chosen depending on the key quantity $\gamma_\ttm^2 + \delta$. While $\eta = \mathcal{O}(\gamma_\ttm^4 + \delta^2)$ puts an asymptotic bound on the accuracy with which the finite-dimensional minimization problems must be solved, the next assumption, $\nu_\ttn = \mathcal{O}(\gamma_\ttm^2 + \delta)$, is central because it establishes a connection between the two discretizations $\X_\ttm$ and $\opo_\ttn$. Ignoring noise, this assumption essentially says that the desired accuracy of the approximation space dictates the required accuracy of the approximate forward operator.
		\item In order to obtain convergent approximations or convergence rates it is of course not necessary to resort to \emph{near} minimizers. Exact minimizers $\x_{\ttm,\ttn}^{\alpha,\delta,0}$, where $\eta = 0$, can be used, provided they exist. In this case, the assumptions on the asymptotic behaviour of $\eta$ in \eqref{eq:ps1} and \eqref{eq:eta} are met automatically. We employ near minimizers primarily to model situations where exact minimization is impractical or infeasible.
	\end{itemize}
\end{remark}

%In the following we continue with \autoref{ex:c_reconstruct} and \autoref{a_reconstruct} and apply \autoref{th:NeuSch90} and \autoref{th:NeuSch90b}, respectively: 
\begin{example}[$c$-example] \label{ex:ac}
Below we describe an approximation scheme for the $c$-example that is entirely based on linear splines and fulfills the requirements of this section. First we note that the error between a function $\y \in W^{2,2}$ and its linear spline interpolant on a grid of mesh size $h>0$, $\mathcal{I}_h \y$, can be bounded by
\begin{align}\label{eq:interp-error}
	\norm{(\y - \mathcal{I}_h \y)^{(k)}}_{L^2} \le C h^{2-k} \norm{\y''}_{L^2}, \quad k=0,1.
\end{align}
%\begin{align}\label{eq:interp-error}
%	\norm{\y - \mathcal{I}_h \y}_{L^2} \le C h^2 \norm{\y''}_{L^2}.
%\end{align}
See \cite{BreSco08} for instance.

Let $\X_\ttm$, $\ttm \in \N$, be the space of linear splines on a uniform grid of $2^{\ttm}+1$ points on $\bar{\Omega} = [0,1]$. These spaces are nested, dense in $L^2 = \X$ and have nonempty intersection with $\dom{\opo}=\mathcal{C}_\epsilon$. Thus, the first item in \autoref{as:fda} is met.

Turning to the second item, let $\Y_\ttn \subset W^{1,2}_0$ be the space of linear splines on a uniform grid of $\ttn+1$ points which vanish at the endpoints. That is 
\begin{equation} \label{eq:splinespace}
	\Y_{\tt n} = \spann \set{\Lambda_i : i=1,\ldots,\n -1}, 
\end{equation}
where $\Lambda_i$ is the linear spline satisfying
\begin{equation}\label{eq:linearspline}
	\Lambda_i \left( \frac{j}{\tt n} \right) = \delta_{ij}, \quad j=0,\ldots,\n.
%	\Lambda_i (s) = \left\{ \begin{array}{rcl} {\tt n} \left( s - \frac{i-1}{\tt n} \right) & \text{ for } & s \in \left[\frac{i-1}{\tt n},\frac{i}{\tt n} \right]\,,\\ {\tt n} \left( \frac{i+1}{\tt n} -s \right) & \text{ for } & s \in \left[\frac{i}{\tt n},\frac{i+1}{\tt n} \right]\;.\end{array} \right. 
\end{equation}
%The boundary splines $\Lambda_0$ and $\Lambda_{\n}$ are excluded due to the homogeneous Dirichlet boundary conditions.
The operators $\opo_\ttn : \dom{\opo_\ttn} := \dom{\opo} \to \Y$ are defined by $\opdis{\ttn}{\x} := \y_\ttn$, where $\y_\ttn \in \Y_\ttn$ is the unique solution of the Galerkin equation
\begin{equation*}
	\inner{\y'_\ttn}{{\tt v}'}_{L^2} + \inner{\x\y_\ttn}{\tt v}_{L^2} = \inner{\tt f}{\tt v}_{L^2} \quad \text{ for all } {\tt v} \in \Y_\ttn,
\end{equation*}
recall the weak formulation \eqref{eq:c-bvp-weak}. It is well-known that
\begin{align}\label{eq:nn2}
	\norm{\op{\x} - \opdis{\ttn}{\x}}_{L^2} \le C (1 + \norm{\x}_{L^2}) \ttn^{-2},
\end{align}
showing that \eqref{eq:uniform_approx} is satisfied with $\nu_\ttn = \ttn^{-2}$. Thus, \autoref{th:NeuSch90} applies.

In order to apply \autoref{th:NeuSch90b} we need to additionally verify \eqref{eq:proj-rate}. That $\gamma_\ttm$ tends to zero follows from the compactness of $\opo$, which is a consequence of estimate \eqref{eq:c-H2-est1}. Also recall \autoref{rem:compact-gamma}. However, here we would like to obtain an explicit upper bound for $\gamma_\ttm$. In order to do so we use, in this order, the fact that $P_\ttm$ is an orthogonal projection, the interpolation error bound \eqref{eq:interp-error} and the estimate for the adjoint in \eqref{eq:c-adjoint-est}
\begin{align*}
	\norm{(I-P_\ttm)\opd{\xdag}^*{\tt g}}_{L^2}
		&\le \norm{(I-\mathcal{I}_{2^{-\ttm}})\opd{\xdag}^*{\tt g}}_{L^2}
			\le C 2^{-2\ttm} \norm{\big(\opd{\xdag}^*{\tt g}\big)''}_{L^2} \\
		&\le C' 2^{-2\ttm} \big(1 + \norm{\xdag}_{L^2}\big)^2 \norm{\f}_{L^2} \norm{{\tt g}}_{L^2}.
\end{align*}
Dividing by $\norm{{\tt g}}_{L^2}$ yields directly
\begin{align*}
	\gamma_\ttm \le C' 2^{-2\ttm} \big(1 + \norm{\xdag}_{L^2}\big)^2 \norm{\f}_{L^2} = \mathcal{O}\big( 2^{-2\ttm} \big).
\end{align*}
The second requirement in \eqref{eq:proj-rate} is met for every $\x^0 \in W^{2,2}$ due to \eqref{eq:interp-error}. Therefore, if $\eta$, $\ttn$ and $\alpha$ are chosen according to \eqref{eq:eta}, we obtain from \autoref{th:NeuSch90b} that
\begin{align*}
	\norm{\xdag - \xademn}_{L^2} = \mathcal{O} \big( \sqrt{\delta} + 2^{-2\ttm} \big).
\end{align*}

        \item{\bf a-example:} We approximate the operator $\opo$ with a linear finite element method, analogously as in \autoref{ex:ac} and make use of \autoref{eq:nn2}, which also holds here. Again we take the space of linear splines $\X_{\ttm}$ to approximate element $\x \in \dom{\opo} \subseteq W^{1,2}(0,1)$. 

        If $\xdag \in \dom{\opo} \subseteq W^{2,2}(0,1)$, then we get from \cite[Lemma 3.1.]{SwaVar72} that 
        \begin{equation} \label{eq:estimateh1}
        	\norm{(I-P_\ttm) \xdag}_{W^{1,2}} \leq C \ttm ^{-1} \norm{\xdag}_{W^{2,2}} \text{ and } \gamma_\ttm = \mathcal{O}(\ttm^{-1})\,.
        \end{equation}
        Applying \autoref{th:NeuSch90b} then yields the optimal parameter choice
        \begin{equation*}
        	\eta = \mathcal{O}(\delta^2 + \ttm^{-4}),\; \ttm \sim \ttn \text{ and } \alpha \sim \max \set{\delta,\ttm^{-2}}
        \end{equation*}
        and get
        \begin{equation*}
        	\norm{\x_{\ttm,\ttn} -\xdag}_{W^{1,2}} = \mathcal{O}(\sqrt{\delta} + \ttn^{-1})\;.
        \end{equation*}
%	\end{description}
\end{example}
\commentO{
\begin{example}[Schlieren tomography] \label{ex:schlieren_discrete}
We use a piecewise constant approximation of $\x$ in $\R^2$. For the evaluation of the Radon-operator, we discretize the line by the Bresenham-algorithm (see \cite{Bre65}) to approximate the line, where the Radon-transform is evaluated by summation. Then the error between $\opo_\ttn$ and $\opo$ can be estimated. We generalize this idea to linear splines approximating $\x$.
\end{example}}

\subsection{Tikhonov regularization in a stochastic setting}
\label{sec: Polregres}
%\commentC{Turn this subsection into remark?}
In a stochastic setting the inverse problem for $\opo$ can be formulated as follows. The basic terms from probability theory are recalled in \autoref{ap:probability_theory}.

Let $(\Omega, \mathfrak{F}, \mathbb{P})$ be a probability space and let $\mathfrak{S}_\X$ and $\mathfrak{S}_\Y$ be $\sigma$-algebras on $\X$ and $\Y$, respectively. Consider random variables (meaning that they are measurable functions between a probability space and a measure space)
\begin{equation*}
	\begin{aligned}
		{\textsf x}: (\Omega, \mathfrak{F}, \mathbb{P}) \to (\X,\mathfrak{S}_\X)\,,
		\quad {\textsf x}^0: (\Omega, \mathfrak{F}, \mathbb{P}) \to (\X,\mathfrak{S}_\X)\,,\\
		{\textsf y}: (\Omega, \mathfrak{F}, \mathbb{P}) \to (\Y,\mathfrak{S}_\Y)\,, \quad
		{\textsf n}: (\Omega, \mathfrak{F}, \mathbb{P}) \to (\Y,\mathfrak{S}_\Y)\;,
	\end{aligned}
\end{equation*}
which are related by means of the operator $\opo$ in the following way
\begin{equation} \label{eq:op_eps}
	{\textsf y} = \op{{\textsf x}} + {\textsf n} \quad \text{almost surely}\;.
\end{equation}
To be precise, \emph{almost surely}\index{almost surely} in \autoref{eq:op_eps} means that there exists a subset $\tilde{\Omega} \in \mathfrak{F}$ (meaning it is measurable) and $\mathbb{P}(\Omega \backslash \tilde{\Omega}) =0$, which satisfies
\begin{equation}\label{eq:as}
	{\textsf y}(\xi) = \op{{\textsf x}(\xi)} + {\textsf n}(\xi) \text{ for all } \xi \in \tilde{\Omega}\;.
\end{equation}

As a consequence of the above notations, for each $\alpha > 0$, we define the regularized solution $\textsf{x}_\alpha: \Omega \to \X$ as

\begin{equation} \label{eq:Tikhonov_stoch}
	\textsf{x}_\alpha(\xi) \in \argmin_{\x \in \X} \Bigl\{ \norm{\op{\x}-{\textsf y}(\xi)}^2 + \alpha \norm{\x-\textsf{x}^0(\xi)}^2 \Bigr\}\;.
\end{equation}
%for almost all $\xi \in \Omega$.
Note that the existence of minimizer in \autoref{eq:Tikhonov_stoch} can be guaranteed almost surely by \autoref{th:ip:well-posedness}, while the measurability of $\textsf{x}_\alpha(\xi)$ as a function of $\xi$ can be ensured under some additional assumptions, such as the ones required, for example, in extensions of the Filippov's implicit function lemma (see \cite{Fil62,McsWar67}).
%\commentO{References ok?. Do we really need them here}
For the sake of simplicity, below we just assume this measurability, which allows a treatment of $\textsf{x}_\alpha(\xi)$  as the corresponding random variable.

The results of \cite{ChiVitMolRosVil24} provide an analog of \autoref{th:ip:rate} for this setting. 
\begin{theorem}[Convergence rates, stochastic] \label{th:Tikhonov_risk} Let \autoref{as:ip:weakly-closed}, \autoref{it:lip} of \autoref{as:ip:rates} as well as the following conditions hold:
	\begin{itemize}
		\item $\dom{\opo}$ is convex and has nonempty interior.
		\item The triple of random variables ${\textsf x} = {\textsf x}^*$, ${\textsf y}$ and ${\textsf n}$ satisfy \autoref{eq:op_eps} and ${\textsf x}^*$ almost surely belongs to the interior of $\dom{\opo}$.
		\item $\norm{\textsf{x}^* - \textsf{x}^0} \le 1$ holds almost surely in $\Omega$ and there is a $\Y$-valued random variable $\omega$ satisfying the \emph{probabilistic source condition}\index{source condition!probabilistic}
		\begin{equation} \label{eq:source_var_st}
			\boxed{{\textsf x}^* - \textsf{x}^0 = \opd{{\textsf x}^*}^*\omega, \quad L\norm{\omega} \leq 1, \quad \text{almost surely.}}
		\end{equation}
	In long this means that
			\begin{equation*} {\textsf x}^*(\xi) - \textsf{x}^0(\xi) = \opd{{\textsf x}^*(\xi)}^*\omega(\xi), \quad L\norm{\omega(\xi)} \leq 1, \quad \text{almost surely,}
	\end{equation*} 
	as explained in \autoref{eq:as}.
		\item The random variables ${\textsf x}^*$, ${\textsf n}$ are independent, and ${\textsf n}$ satisfies
		\begin{equation} \label{eq:exp_eps}
			\mathbb{E}\norm{{\textsf n}}^2 = \int_{\Omega} \norm{{\textsf n}(\xi)}^2 d\mathbb{P}(\xi) \le \delta^2 .\footnote{This condition replaces \autoref{eq:datn} in the deterministic setting.}
		\end{equation}
	\end{itemize}
	Then, for all $\alpha > 0$,
	\begin{equation}\label{eq:bound_chirinos}
		\mathbb{E} \left[\norm{{\textsf x}^* - \textsf{x}_\alpha}^2 \right] \le C \left(\alpha + \frac{\delta^2}{\alpha}\right)\;.
	\end{equation}
%\commentO{Is there a result about the variance as well?}
\end{theorem}

\begin{remark}
	Note that the boundedness condition $\norm{\textsf{x}^* - \textsf{x}^0} \le 1$ can be weakened, and it is formulated here just as in \cite{ChiVitMolRosVil24} to avoid additional technicalities. Another formalism for lifting convergence rates results from the deterministic to the stochastic setting was suggested in \cite{EngHofKin05}. There quantitative bounds for the Prokhorov metric were established under rather general source conditions.
\end{remark}

\section{Iterative regularization methods} \label{sec:bascon}
The basic assumptions on the operator $\opo$ for analysing Tikhonov regularization and iterative regularization methods are significantly different. The reason is that Tikhonov regularization is a concept, where the method for computing a \emph{global minimizer} is not specified.
In contrast, iterative regularization methods are methods providing an explicit formula for each iterate $\xkd$. For nonlinear inverse problems iteration methods like the Landweber-iteration, as already introduced in \autoref{eq:steepest} with step-sizes $\mkd=1$ for all $k \in \N_0$, will in general not converge globally but locally.

We review below local convergence results of iterative regularization, which require imposing \emph{structural}
conditions on $\opo$, which are summarized in the following. We also emphasize that in the classical setting $\x^0 = \x_0$ (note that independent of $\delta$ we use $\x_0^\delta = \x_0$) in \autoref{eq:steepest}. Make benefits out of the flexibility of choosing different $\x_0$ and $\x^0$ will be discussed in \autoref{ch:Aspri paper} below.

Two ingredients are important for a successful convergences analysis:
\begin{enumerate}
	\item Structural conditions on $\opo$ guarantee that in the neighborhood of a desired solution it is the only solution, and it can be \emph{seen} from every point (see \autoref{fig:cone}). This avoids being trapped with a gradient descent method. Being able to always keep an eye on the solution is the basis of the analysis of iterative method in finite dimensions (see \autoref{th:deupot92}), so these are quite natural generalizations for operator equations.
	\item \emph{Early stopping:} The iteration at $k_* \in \N$, when for the first time during the iteration
	\begin{equation} \label{eq:disc}
		\norm{\op{\x_{k_*}}-\y^\delta}_\Y \leq \tau \delta < \norm{\op{\xkd}-\y^\delta}_\Y \text{ for all } 0 \leq k < k_*	
	\end{equation}
	where $\tau > 1$ is a fixed parameter. As an approximation of the solution of \autoref{eq:op} one uses $\x_{k_*}^\delta$ (see \cite{KalNeuSch08}). \index{early stopping} \index{discrepancy principle!Morozov}
	We mention that Morozov's discrepancy principle \cite{Mor66b,Mor93} -- with $\tau>1$ -- has been
	applied successfully by Vainikko \cite{Vai80} to the regularization of linear
	ill-posed problems via Landweber-iteration. In \cite{DefMol87}, Defrise \& De Mol used a
	different technique to study the stopping rule \autoref{eq:disc} for $\tau>2$. We mention that this parameter choice is an \emph{a-posteriori} strategy (see \autoref{sec:post}) because the regularization parameter $k_*$ is determined during the iteration. \index{stopping criterion!a-posteriori}
\end{enumerate}

\subsection{Basic assumptions}
In order to perform a convergence analysis of iterative methods in \emph{infinite dimensions} we make the following assumptions on $\opo$:
\begin{enumerate}
	\item We require that $\opo$ from \autoref{eq:op} is properly scaled. For our analysis we assume that the operator norm of $\opo'[\x]$ can locally uniform be estimated by one:
	\begin{equation}\label{eq:scal}
		\norm{\opo'[\x]}\leq 1\,,\qquad \x\in \overline{\mathcal{B}_{2\rho}(\x_0)} \subset \mathcal{D}(F)\,,
	\end{equation}
	where $\overline{\mathcal{B}_{2\rho}(\x_0)}$ denotes a closed ball of radius $2\rho$ around $\x_0$.
	\item
	In addition to this scaling property, we need the following local condition, originally introduced in \cite{HanNeuSch95} and named \emph{tangential cone condition} (see also \cite{Sch95}):\index{tangential cone condition}
	\begin{equation}\label{eq:nlc}
		\boxed*{
			\begin{aligned}
			\norm{\op{\x}-\op{\tilde{\x}}-\opo'[\x](\x-\tilde{\x})}_\Y \leq \eta\norm{\op{\x}-\op{\tilde{\x}}}_Y,\qquad \eta<\frac{1}{2}\,,\\
			\x,\tilde{\x}\in\overline{\mathcal{B}_{2\rho}(\x_0)}\subset\mathcal{D}(F)\,.			
			\end{aligned}
		}
	\end{equation}
	Both conditions together are strong enough to ensure local convergence to a solution of
	\autoref{eq:op} if it is solvable in $\overline{\mathcal{B}_\rho(\x_0)}$. They also guarantee
	that all iterates $\xkd$, $0\leq k\leq k_*$, remain in $\mathcal{D}(F)$, which makes the
	iteration method well defined.
\end{enumerate}
Condition \autoref{eq:nlc} essentially requires that from a point $(\x,F[\x])$ on the graph of $\opo$ the point $(\tilde{\x},\op{\tilde{\x}})$ is viewable and vice versa (see \autoref{fig:cone}). It is like walking in the mountains where one should always be able to keep an eye on the destination. In a local valley one might get trapped. To avoid the trapping one needs prior information such as a map, which will be realized with hybrid regularization (see \autoref{ch:Aspri paper}).
\begin{figure}[h]
	\begin{center}
		\includegraphics[width=.8\linewidth]{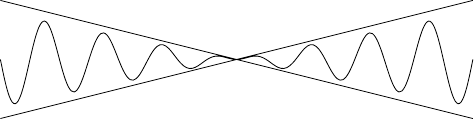}
		\caption{\label{fig:cone} Visualization of \autoref{eq:nlc}: The graph of the function $\opo(x)-\opo(x^\dagger)$, where $(x^\dagger,\opo(x^\dagger))$ is the central point in the middle of the graph. The area bounded by the two lines marks the region where \autoref{eq:nlc} is valid with $\tilde{x}$ and $x^\dagger$. However, the condition should be valid for all iteration points, and this is only possible in a very small neighborhood of $x^\dagger$. We recall our notation $x$ is a number and $\opo(\cdot)$ is a function, while in the general context we consider functions $\x$ and an operator $\op{\cdot}$.}
	\end{center}
\end{figure}
From \autoref{eq:nlc} it follows immediately with the triangle inequality that
\begin{equation}\label{eq:inc}
	\frac{1}{1+\eta}\norm{\opo'[\x](\tilde{\x} - \x)}_\Y \leq \norm{\op{\tilde{\x}}-F[\x]}_\Y \leq \frac{1}{1-\eta}
	\norm{\opo'[\x](\tilde{\x} - \x)}_\Y
\end{equation}
for all $\x,\tilde{\x}\in\overline{\mathcal{B}_{2\rho}(\x_0)}$. Thus, condition \autoref{eq:nlc} seems to be rather
restrictive. However, the condition is quite natural as the following argument
shows: if $\opd{\cdot}$ is Lipschitz-continuous and $\x,\tilde{\x}\in\mathcal{D}(F)$, then the
error bound
\begin{equation}\label{eq:tayl}
	\norm{\op{\tilde{\x}}-\op{\x}-\opo'[\x][\tilde{\x}-\x]}_\Y \leq c\norm{\tilde{\x}-\x}_\X^2
\end{equation}
holds for the Taylor approximation of $\opo$. For ill-posed problems, however, it
turns out that this estimate carries too little information about the local
behavior of $\opo$ around $x$ to draw conclusions about convergence of iterative methods.
For several examples one can even prove the stronger condition
\begin{equation}\label{eq:snlc}
	\norm{F[\x]-\op{\tilde{\x}}-\opo'[\x](\x-\tilde{\x})}_Y \leq c\norm{\x - \tilde{\x}}_\X \norm{F[\x]-\op{\tilde{\x}}}_\Y.
\end{equation}
Provided $\norm{\x-\tilde{\x}}_\X$ is sufficiently small, this implies condition
\autoref{eq:nlc}.

First, we recall a result about existence of a unique $\x_0$-minimum-norm solution, which already follows from conditions like \autoref{eq:nlc}:

\begin{proposition}\label{pr:pmns}
	Let $\opo$ satisfy:
	\begin{equation}\label{eq:p1}
		\begin{aligned}
			\norm{F[\x]-\op{\tilde{\x}}-\opo'[\x](\x - \tilde{\x})}_\Y \leq c(\x,\tilde{\x})\norm{F[\x]-\op{\tilde{\x}}}_\Y, \\*[1ex]
			\x,\tilde{\x}\in\overline{\mathcal{B}_\rho(\x_0)}\subset\mathcal{D}(F)\,,
		\end{aligned}
	\end{equation}
	for some $c(\x,\tilde{\x})\geq 0$, where
	\begin{equation*}
	c(\x,\tilde{\x})<1 \text{ if } \norm{\x - \tilde{\x}}\le \ve\;.
	\end{equation*}
	\begin{enumerate}
		\item Then for all $x\in\overline{\mathcal{B}_\rho(\x_0)}$
		\begin{equation*}
			M_\x:=\{\tilde{\x}\in\overline{\mathcal{B}_\rho(\x_0)}\,:\,\op{\tilde{\x}}=F[\x]\}
             =\x+\nsp{\opo'[\x]}\cap\overline{\mathcal{B}_\rho(\x_0)}
		\end{equation*}
		and $\nsp{\opo'[\x]}=\nsp{\opo'(\tilde{\x})}$ for all $\tilde{\x}\in M_\x$. Moreover,
		\begin{equation*}
			\nsp{\opo'[\x]}\supseteq \{t(\tilde{\x}-\x)\,:\,\tilde{\x}\in M_\x, t\in\R\},
		\end{equation*}
		where instead of $\supseteq$ equality holds if $\x\in \mathcal{B}_\rho(\x_0)$.
		\item If $F[\x]=\y$ is solvable in $\overline{\mathcal{B}_\rho(\x_0)}$, then a unique  $\x_0$-minimum-norm solution exists. It is
		characterized as the solution $\xdag$ of $F[\x]=\y$ in $\overline{\mathcal{B}_\rho(\x_0)}$ satisfying the
		condition
		\begin{equation}\label{eq:orthc}
			\xdag-\x_0\in\nsp{\opo'[\xdag]}^\bot.
		\end{equation}
	\end{enumerate}
\end{proposition}
For a proof see \cite[Proposition 2.1]{KalNeuSch08}.

\subsection{Landweber-method}
The Landweber-method\index{Landweber-method} is defined as follows:
\begin{itemize} \label{alg:Landweber}
	\item Initialize $\x_0^\delta=\x_0$.
	\item Until Morozovs's stopping criterion\index{discrepancy principle!Morozov} \autoref{eq:disc}
	becomes active, that is $k=0,1,\ldots,k_*-1$ update
	\begin{equation} \label{eq:landweber}
		\xkpd = \xkd - \opd{\xkd}^* (\op{\xkd} - \y^\delta)\;.
	\end{equation}
	\item The approximating solution is given by $\x_{k_*}^\delta$.
\end{itemize}
The Landweber-iteration is a gradient descent method with fixed step length.
More sophisticated gradient descent methods, such as \emph{steepest descent}, \emph{minimal error methods} and \emph{conjugate gradient methods} have been analyzed in \cite{NeuSch95,Sch96}.

The Landweber-iteration is well-investigate. For the convergence and stability analysis we refer to \cite{HanNeuSch95,KalNeuSch08}. As any other regularization method, the convergence of $\x_{k_*}^\delta$ to $\xdag$ can be arbitrarily slow for $\delta \to 0$.

\subsection{Convergence and stability of the Landweber-iteration}
\begin{proposition} \label{lwmon}
	Assume that the conditions \autoref{eq:scal} and \autoref{eq:nlc} hold and that \autoref{eq:op}
	has a solution $\x^*\in\mathcal{B}_\rho(\x_0)$. If $\xkd\in\mathcal{B}_\rho(\x^*)$, a sufficient
	condition for $\xkpd$ to be a better approximation of $\x^*$ than $\xkd$ is that
	\begin{equation} \label{eq:resest}
		\norm{\op{\xkd}-\y^\delta}_\Y >2\,\frac{1+\eta}{1-2\eta}\,\delta\,.
	\end{equation}
	Moreover, it then holds that $\xkd,\xkpd\in\mathcal{B}_\rho(\x^*)\subseteq \mathcal{B}_{2\rho}(\x_0)$.
\end{proposition}

In view of this proposition, the number $\tau$ in the stopping rule \autoref{eq:disc}
should be chosen subject to the following constraint depending on $\eta$, with
$\eta$ as in \autoref{eq:nlc}:
\begin{equation} \label{eq:lwtau}
	\tau>2\,\frac{1+\eta}{1-2\eta}>2\,.
\end{equation}
From the proof of \autoref{lwmon} we can easily extract an inequality
that guarantees that the stopping index $k_*$ in \autoref{eq:disc} is finite and hence
well defined.

\begin{corollary} \label{corest}
	Let the assumptions of \autoref{lwmon} hold and let $k_*$ be chosen
	according to the stopping rule \autoref{eq:disc}, \autoref{eq:lwtau}. Then
	\begin{equation} \label{eq:ksest}
		k_*(\tau\delta)^2<\sum_{k=0}^{k_*-1}\norm{\op{\xkd}-\y^\delta}_\Y^2\leq \frac{\tau}
		{(1-2\eta)\tau-2(1+\eta)}\norm{\x_0-\x^*}_\X^2.
	\end{equation}
	In particular, in the case of exact data, it follows that
	\begin{equation} \label{eq:sumest}
		\sum_{k=0}^\infty \norm{\y-\op{\xk}}_\Y^2<\infty\,.
	\end{equation}
\end{corollary}

\begin{theorem} \label{lwced}
	Assume that the conditions \autoref{eq:scal} and \autoref{eq:nlc} hold and that \autoref{eq:op}
	is solvable in $\mathcal{B}_\rho(\x_0)$. Then the nonlinear Landweber-iteration applied to exact data $\y$
	converges to a solution of \autoref{eq:op}. If
	\begin{equation*}
		\mathcal{N}(\opo'[\xdag]) \subseteq \mathcal{N}(\opo'[\x]) \text{ for all } \x\in \mathcal{B}_\rho(\xdag)\,,
	\end{equation*}
	then $\xk$ converges to $\xdag$ as $k\to\infty$.
\end{theorem}

\begin{theorem} \label{lwcnd}
	Let the assumptions of \autoref{lwced} hold and let $k_*=k_*(\delta,\y^\delta)$
	be chosen according to the stopping rule \autoref{eq:disc}, \autoref{eq:lwtau}. Then the
	Landweber-iterates $\x_{k_*}^\delta$ converge to a solution of \autoref{eq:op}. If $\mathcal{N}(\opo'[\xdag])
	\subseteq \mathcal{N}(\opo'[\x])$ for all $\x \in \mathcal{B}_\rho(\xdag)$, then $\x_{k_*}^\delta$ converges to $\xdag$ as
	$\delta\to 0$.
\end{theorem}

\subsection{Convergence rates of the Landweber-iteration}
As stated above the rate of convergence of $\xk \to \x^*$ as $k \to \infty$ (with precise data) or
$\x_{k_*}^\delta \to \x^*$ as $\delta \to 0$ (with perturbed data) will, in general, be
arbitrarily slow. As for Tikhonov regularization we can prove convergence rates, under the source condition \autoref{eq:source_var}. However, in contrast to Tikhonov regularization, assumption \autoref{eq:source_var}
(with $\norm{\omega}$ sufficiently small) is not enough to obtain convergence rates for Landweber-iteration.
In \cite{HanNeuSch95} rates were proven under the additional assumption that $\opo$ satisfies
\begin{equation} \label{eq:sc1}
	\boxed{
	\opo'[\x]=R_\x \opo'[\xdag]\text{ and } \norm{R_\x-\id}\leq c\norm{\x-\xdag}_\X,\qquad \x\in
	\mathcal{B}_{2\rho}(\x_0)\,,}
\end{equation}
where $\set{R_\x\,:\,\x\in\mathcal{B}_{2\rho}(\x_0)}$ is a family of bounded linear operators $R_\x:\Y \to \Y$ and $c$ is a positive constant. Again $\norm{R_\x-\id}$ denotes the operator norm of $R_\x-\id$.
\begin{remark}\label{re:deland} From \autoref{eq:sc1} it follows that
		\begin{equation} \label{eq:decomp3}
			\begin{aligned}
				F[\x] - \op{\xdag} &= \int_0^1 \opd{\xdag+t(\x-\xdag)}(\x-\xdag)dt \\
                &= \int_0^1 R_{\xdag + t(\x-\xdag)} \opd{\xdag}(\x-\xdag) dt\;.
			\end{aligned}
		\end{equation}

	In other words $\opo$ satisfies a condition, which is similar as the first decomposition case (see \autoref{ss:decomp1}).
	Actually, for operators, which satisfy \autoref{eq:decomp3}, we can prove similar results as for operators, which satisfy the first decomposition case.
\end{remark}

Unfortunately, the conditions above are not always satisfied (see
\cite[Example~4.3]{HanNeuSch95}). To enlarge the applicability of the results, in
this section we consider instead of the Landweber-iteration \autoref{eq:steepest} (with $\mkd=1$) the following slightly modified iteration method,
\begin{equation} \label{eq:lwm}
	\xkpd = \xkd - G^\delta[\xkd]^*(\op{\xkd}-\y^\delta)\,, \quad k\in \N_0\,,
\end{equation}
where $G^\delta[\x]:=G[\x,\y^\delta]$, and
\begin{equation*}
	G: \mathcal{D}(F) \subseteq \X \times \Y \to L(\X,\Y)
\end{equation*}
is a continuous operator. The iteration will again be stopped according to the discrepancy principle \autoref{eq:disc}.

To obtain local convergence and convergence rates for the modification \autoref{eq:lwm}
we need the following assumptions:

\begin{assumption} \label{as:asslw}
	Let $\rho$ be a positive number such that $\mathcal{B}_{2\rho}(\x_0)\subseteq \mathcal{D}(F)$.
	\begin{enumerate}
		\item \autoref{eq:op} has an minimum norm solution $\xdag$ in $\mathcal{B}_\rho(\x_0)$.
		\item There exist positive constants $c_1, c_2, c_3$ and linear operators
		$R_\x$ such that for all $x\in \mathcal{B}_\rho(\xdag)$ the following estimates hold:
		\begin{equation} \label{eq:nlc1}
			\norm{F[\x]-\op{\xdag}-\opo'[\xdag](\x-\xdag)}_\Y \leq c_1 \norm{\x-\xdag}_\X \norm{F[\x]-\op{\xdag}}_\Y\,,
		\end{equation}
	    satisfying
		\begin{equation} \label{eq:nlc2} \begin{aligned}
			G^\delta[\x]=R_\x G^\delta[\xdag] \text{ with }
			\norm{R_\x-\id} \leq c_2\|\x-\xdag\|_\X\,,\\
			\norm{\opo'[\xdag]-G^\delta[\xdag]} \leq c_3\delta\,.
			\end{aligned}
		\end{equation}
		\item Moreover, $\opo$ has to be scaled appropriately
		\begin{equation} \label{eq:om}
			\norm{\opo'[\xdag]} \leq 1\,.
		\end{equation}
	\end{enumerate}
\end{assumption}
Note that, if instead of \autoref{eq:nlc1} the slightly stronger condition \autoref{eq:snlc}
holds in $\mathcal{B}_{2\rho}(\x_0)$, then the existence of a unique minimum norm solution $\xdag$ is guaranteed (see \cite{KalNeuSch08}).

The next theorem states that under \autoref{as:asslw} the modified Landweber-iteration
\autoref{eq:lwm} converges locally to $\xdag$ if $\tau$ is chosen properly:

\begin{theorem}(Theorem 2.13 in \cite{KalNeuSch08}) \label{lwrat2} Let $\xkd$ be the iterates of \autoref{eq:lwm}.
	Assume that \autoref{as:asslw} holds and let $k_*=k_*(\delta,\y^\delta)$ be chosen
	according to \autoref{eq:disc} with $\tau>2$. If $\xdag-\x_0$ satisfies the source condition
	\autoref{eq:source_var}, then it follows that
	\begin{equation*}
		k_* = \mathcal{O} (\delta^{-1})
	\text{ and }
		\norm{\x_{k_*}^\delta-\xdag}_\X = \mathcal{O} (\sqrt{\delta})\;.
	\end{equation*}
\end{theorem}

\subsection{Iteratively regularized Gauss-Newton-method (\IRGN)} \label{se:itrgn}
In this section we deal with the iteratively regularized Gauss-Newton-method\index{Gauss-Newton-method!iteratively regularized!\IRGN}
\begin{equation} \label{eq:itrgn}
	\begin{aligned}
		\xkpd = & \xkd - (\opd{\xkd}^*\opd{\xkd}+\alpha_k \id)^{-1}\bigl( \opd{\xkd}^*(\op{\xkd}-\y^\delta) \\
		& \qquad \quad + \alpha_k(\xkd-\x_0)\bigr)\,,
	\end{aligned}
\end{equation}
where $\alpha_k$ is a sequence of positive numbers tending towards zero.
This method was introduced in \cite{Bak92}.

Note that the approximate solution $\xkpd$ minimizes the Tikhonov functional for a linear operator $\opd{\xkd}$
\begin{equation*}
	\mathcal{T}_{\footnotesize{\alpha,\y^\delta + \opd{\xkd}\xkd}}[\x]
     =\norm{\op{\xkd}-\y^\delta-\opd{\xkd}(\x-\xkd)}_\Y^2+\alpha_k\norm{\x-\x_0}_\X^2.
\end{equation*}

This means that $\xkpd$ minimizes the Tikhonov functional where the nonlinear
operator $\opo$ is linearized around $\xkd$ (cf.~\cite[Chapter 10]{EngHanNeu96}).

We emphasize that for a fixed number of iterations the process \autoref{eq:itrgn} is a
stable method if $\opo'$ is Lipschitz-continuous.

\subsection{Convergence and stability of \IRGN}
A convergence and stability analysis, which does not making use of an a prior source condition for $\xdag$, such as \autoref{eq:source_var}, requires structural properties on $\opo$ as \autoref{rg1} in Case 1 below. This is similar as for the Landweber-iteration, \autoref{as:asslw}.

Under the source condition \autoref{eq:source_var} the structural conditions are not necessary, instead a standard Lipschitz-continuity of $\opo'$ is required. Here the analysis is different than to the Landweber-iteration.

The relevant conditions are as follows (see \cite{KalNeuSch08}):

\begin{assumption} \label{as:assrgn}
	The following general assumptions hold:
	\begin{enumerate}
		\item There exists $\rho>0$ such that $\mathcal{B}_{2\rho}(\x_0) \subset \mathcal{D}(F)$.
		\item \autoref{eq:op} has an a minimum norm solution  $\xdag$ in $\mathcal{B}_\rho(\x_0)$.
		\item $\opo$ is Fr{\'e}chet-differentable in $\mathcal{D}(F)$.
	\end{enumerate}
	There exists $r>1$ such that the sequence $(\alpha_k)_{k \in \N_0}$ in \autoref{eq:itrgn} satisfies
	\begin{equation} \label{alpb}
		\alpha_k>0\,,\; 1\leq \frac{\alpha_k}{\alpha_{k+1}}\leq r\,,\; \lim_{k \to \infty}\alpha_k=0\;.
	\end{equation}
	In addition either one of the following two cases holds:
	\begin{description}		
		\item{Case 1: the source condition \autoref{eq:source_var} is not satisfied:}
		The Fr\'echet-derivative $\opo'$ satisfies the following structural conditions
		\begin{equation} \label{rg1} \begin{aligned}
				\opo'[\tilde{\x}] &= R[\tilde{\x},\x]\opo'[\x]+Q[\tilde{\x},\x]\,,  \\
				\norm{\id-R[\tilde{\x},\x]} &\leq c_R\,, \\
				\norm{Q[\tilde{\x},\x]} &\leq c_Q\norm{\opo'[\xdag](\tilde{\x}-\x)}_\Y\,,
		\end{aligned} \end{equation}
		for $\x, \tilde{\x}\in\mathcal{B}_{2\rho}(\x_0)$, where $c_R$ and $c_Q$ are nonnegative constants.
		\item{Case 2: $\xdag-\x_0$ satisfies the source condition \autoref{eq:source_var}.}
		The Fr\'echet-derivative $\opo'$ is Lipschitz-continuous in
		$\mathcal{B}_{2\rho}(\x_0)$: That is, there exists some $L>0$ such that
		\begin{equation} \label{Lip}
			\norm{\opo'[\tilde{\x}]-\opo'[\x]}\leq L\norm{\tilde{\x}-\x}_\X,\qquad \x, \tilde{\x}\in\mathcal{B}_{2\rho}(\x_0)\;.
		\end{equation}
	\end{description}
\end{assumption}

\begin{remark}
	Similarly as in \autoref{re:deland} we can interpret \autoref{rg1} as a perturbation of the first decomposition case.
\end{remark}

\subsection{Convergence rates of \IRGN}
So far we considered the discrepancy principle \autoref{eq:disc} for early stopping. Originally in \cite{Bak92} an \emph{a--priori stopping rule} for \IRGN \index{stopping rule!a--priori}:
\begin{definition}[A--priori stopping] \label{de:apriori} For given noise level $\delta > 0$ and sequence $\set{\alpha_k:k \in \N}$ defining \IRGN we stop the iteration after $k_*=k_*(\delta)$ steps depending whether the
	source condition \autoref{eq:source_var} is satisfied or not:
	\begin{description}
		\item{Case 1:} We choose $\eta > 0$ and determine $k_*=k_*(\delta)$ such that
		\begin{equation}\label{stoprgb}
			k_*(\delta) \to \infty \text{ and } \eta \geq \frac{\delta}{\sqrt{\alpha_{k_*}}} \to 0
				\text{ as } \delta\to 0\;.
		\end{equation}
		\item{Case 2:} Then we choose some $\eta > 0$ and determine
		$k_*$ a-priori such that
              \begin{equation} \label{stoprga}
		        \eta \alpha_{k_*} \leq \delta < \eta\alpha_k
		        \text{ for all } 0 \leq k <k_*\;.
	          \end{equation}
    \end{description}
\end{definition}
Note that, due to \autoref{alpb}, this guarantees that
$k_*(\delta)<\infty$, if $\delta>0$ and that $k_*(\delta)\to\infty$ as
$\delta\to 0$. In the noise free case ($\delta=0$) we can set $k_*(0):=\infty$
and $\eta:=0$.

We will show in the next theorem that this a--priori stopping rule yields convergence and
convergence rates for the iteratively regularized Gauss-Newton-method \autoref{eq:itrgn} provided that $\norm{\omega}_\Y$, $c_R$,
$c_Q$, and $\eta$ are sufficiently small.

\begin{theorem} \label{rgrat1}
	Let \autoref{as:assrgn} hold and let $k_*=k_*(\delta)$ be chosen according
	to \autoref{stoprga} or \autoref{stoprgb} (depending on $\mu$). Moreover, we assume that $\x_0$ is close to $\xdag$ (the precise formulas can be found in \cite[Theorem 4.12]{KalNeuSch08}.
	Then
	\begin{equation*}
		\|\x_{k_*}^\delta-\xdag\|_\X=\left\{\begin{array}{ll} o(1) & \text{Case 1}, \\
			\mathcal{O}(\sqrt{\delta}) \qquad & \text{Case 2}. \end{array}\right.
	\end{equation*}
	For the noise free case $(\delta=0)$ we obtain that
	\begin{equation*}
		\|\xk-\xdag\|_\X=\left\{\begin{array}{ll} o(1) & \text{Case 1}\,, \\
			o(\sqrt{\alpha_k}) & \text{Case 2}\,, \end{array}\right.
    \text{ and }
		\norm{\op{\xk}-\y}_\Y=\left\{\begin{array}{ll} o(\sqrt{\alpha_k}) &\text{Case 1}\,, \\
			\mathcal{O}(\alpha_k) & \text{Case 2}\,. \end{array}\right.
	\end{equation*}
\end{theorem}
In \cite{Bak92} local convergence has been proved under a stronger source condition than
\autoref{eq:source_var}, which allowed to even prove the rate (compare with the rates in \autoref{rgrat1})
\begin{equation*}
	\|\xk-\xdag\|_\X=\mathcal{O}(\alpha_k)\,.
\end{equation*}
In \cite{BlaNeuSch97,KalNeuSch08} it was shown that convergence rates can be obtained under weak conditions.
Here this result is omitted for the purpose of unification of the book.

\begin{theorem} \label{rgrat2}
	Let \autoref{as:assrgn} hold. For Case 2 we assume on top that \autoref{rg1} holds. Let $k_*=k_*(\delta)$ be
	chosen according to \autoref{eq:disc} with $\tau>1$. Moreover, we assume $\x_0$ and $\xdag$ are close (the precise numbers can be found in \cite[Theorem 4.13]{KalNeuSch08}). Then
	\begin{equation*}
		\|\x_{k_*}^\delta-\xdag\|_\X=\left\{\begin{array}{ll}
			o(1) & \text{Case 1}\,, \\
			\mathcal{O}(\sqrt{\delta}) & \text{Case 2}\,. \end{array}\right.
	\end{equation*}
\end{theorem}

Summarizing the convergence rates results in a table gives a better impression of the challenges (see \autoref{tab:crates}):

\begin{table}[h]
	\begin{tabular}{r || c | c | c}
		Iterative method & Rate & Source condition &  Structural property $\opo$\\
		\hline
		\hline
		Landweber-& $o(\delta)$ & --- & \autoref{eq:nlc}\\
		& $\sqrt{\delta}$ &\autoref{eq:source_var} & \autoref{eq:sc1}\\
		\hline
		\IRLI  & $o(\delta)$ & --- & \autoref{rg1}\\
		& $\sqrt{\delta}$ &\autoref{eq:source_var} & \autoref{eq:lipschitz}\\
		\hline
		\IRGN  & $o(\delta)$ & --- & \autoref{rg1}\\
		& $\sqrt{\delta}$ &\autoref{eq:source_var} & \autoref{eq:lipschitz}\\
		\hline\hline
	\end{tabular}
	\caption{\label{tab:crates} The method, the used source conditions on $\x_0=\x^0$ and the essential structural properties on $\opo$ giving a rate. One sees that stronger source conditions allow to decrease the strength of the structural properties condition for \IRLI and \IRGN.}
\end{table}	

The proofs are rather technical and omitted here. It is however important to note that structural properties of $\opo$ can be avoided if $\xdag-\x_0$ satisfies a source condition.

\section{Open research questions}
The theory of iterative regularization methods is quite developed but not complete. One open research question concerns whether it is possible to prove convergence of iterative regularization methods if only approximations of the operator $\opo$ are available.
\begin{opq} \label{opq:FFN} The open challenge concerns the derivation of error estimates, like for Tikhonov regularization (see \autoref{th:NeuSch90} and \autoref{th:NeuSch90b}) for iterative regularization methods under error perturbation of $\opo$. Therefore, we aim for estimates for $\|\xkd - \x_k^{\delta,{\ttn}}\|_\X$, where $\x_k^{\delta,{\ttn}}$ are the iterates of the Landweber-iteration with an approximating operator $\opo_\ttn$:
	\begin{equation} \label{eq:landweber_m}
		\x_{k+1}^{\delta,{\ttn}} = \x_{k}^{\delta,{\ttn}} - F_\ttn'[\x_{k}^{\delta,{\ttn}}]^*\left( F_\ttn[\x_{k}^{\delta,{\ttn}}] - \y^\delta \right)\;.
	\end{equation}
	$F_\ttn$ could be a constructive finite element approximation or a trained operator (see \autoref{ex:c_reconstruct} for motivation). Neither for Landweber-nor the iteratively regularized Landweber-iteration such estimates could be derived so far. The only methods which provide an analysis are \emph{multi-scale methods} (see \cite{Sch98b,BerHooFauSch16,HooQiuSch12,HooQiuSch15,Kal06,Kal08}) and \emph{frequency swapping methods} \cite{BerHooFauSch16,BerHooFraVesZha17}.
\end{opq}

\begin{opq} It is interesting to note that convergence rates results for the Landweber-iteration are proven with source conditions together with even stronger conditions than the tangential cone condition \autoref{eq:nlc}, which is used to prove pure
	convergence (see \cite{HanNeuSch95,KalNeuSch08}). For the iteratively regularized Gauss-Newton-\cite{Bak92} and iteratively regularized Landweber-iteration \cite{Sch98} this behavior is opposite, namely the conditions on $\opo$ get standard (weak) Fr{\`e}chet-differentiability when the prior $\x^0=\x_0$ satisfies a source condition and are strongest, when no source condition holds. The question is whether the theory of both classes of methods can be unified? The comparison of different iterative methods are summarized in \autoref{tab:crates}.

\end{opq}

\section{Further reading and outlook}
A relatively complete \emph{deterministic regularization theory} of variational methods for solving nonlinear ill--posed problems is documented in the books \cite{Mor84,Gro84,Hof86,Bak92,BakGon94,EngHanNeu96,KalNeuSch08,SchGraGroHalLen09,SchuKalHofKaz12,Kab12,SchGraGroHalLen09,VesLeG16}. In the context of \emph{early vision} we mention \cite{BerPogTor88}. The theory of regularization methods for solving linear ill--posed problems in Hilbert spaces in summarized for instance in \cite{Gro84,Gro11,Han17}, which present the technical analysis in a very appealing manner. An excellent exposition developing regularization methods for particular inverse imaging examples is \cite{BerBoc98}.

The first analysis of Tikhonov regularization for solving nonlinear ill--posed problems was published simultaneously in \cite{EngKunNeu89,SeiVog89}. The importance of the prior paper is that it provides the first \emph{quantitative results}.
Later this have been extended in \cite{EngHanNeu96}. While the theory in the above referenced work is valid for Hilbert space $\X$ and $\Y$, the books \cite{SchGraGroHalLen09,SchuKalHofKaz12} formulate some according results in Banach spaces and even beyond.

Wavelets as ansatz functions in regularization methods have been considered in \cite{FreSchn98,DauDefMol04,MaaRie97}, typically with Besov space norm regularization (see \autoref{sec:besov}).
\begin{remark}\label{re:ObmSchwHal21} Neural network ansatz functions together with Tikhonov regularization have been investigated, for instance, in \cite{BurEng00} for solving inverse problems and in \cite{PogGir90} for denoising. Later, they were given the name \emph{synthesis networks}\index{synthesis networks} in \cite{ObmSchwHal21}. There, elements $\x \in \X$ are parametrized by means of a map $\Psi: \Xi \subset \R^{\dimlimit} \to \X$.
	Then, instead of $\T_{\alpha,\y^\delta}$, the functional
	\begin{equation} \label{eq:Tiksyn}
		\vp \mapsto \norm{\op{\Psi[\vp]}-\y^\delta}_{\Y}^2+\alpha\norm{\vp-\vp^0}_2^2
	\end{equation}
	is minimized.
	The concept of reparametrization in regularization is much older and was studied already in \cite{ChaKun93}.
	Existence, stability and convergence follow immediately if the composition $\opo \circ \Psi$ satisfies the general \autoref{as:ip:weakly-closed} with the space $\X$ replaced by $\Xi \subseteq \R^{\dimlimit}$.
	Examples of parametrizations are given in \autoref{de:affinenns}.
	
	In the context parametrized representations of functions one often looks for solution, which can be represented by a few parameters. This is achieved by \emph{sparsity regularization}\index{regularization!sparsity}, such as
	\begin{equation} \label{eq:Tiksynl1}
		\vp \mapsto \norm{\op{\Psi[\vp]}-\y^\delta}_{\Y}^2+\alpha\norm{\vp}_{1}\,.
	\end{equation}
	This method has been studied extensively in the literature. See for instance \cite{Can06,Can08,CanRomTao06b,DauDefMol04,ChaPoc11} - a particular emphasize is on fast reconstruction.
	Convergence rates of sparsity regularization based on source conditions were established in \cite{GraHalSch11,Gra10b,Gra20}. These papers, in particular link source conditions and the \emph{restricted isometry property}\index{restricted isometry property} of \cite{CanRomTao06b}.
    Algorithms for solving the problem are LASSO (see for instance \cite{Tib96b}), ISTA, which is proximal gradient method (see for instance \cite{DauDefMol04,ComPes07}) and FISTA (see for instance \cite{BecTeb09}).
\end{remark}

The theory of \emph{iterative Tikhonov regularization}\index{Tikhonov regularization!iterative} is well-understood for linear inverse problems (see \cite{Gro84}). It is well-known for linear inverse problems that iterated Tikhonov regularization can overcome the \emph{saturation effect} of Tikhonov regularization: The best convergence rates for Tikhonov-regularization are $\mathcal{O}(\delta^{2/3})$, a rate which has also been proven for nonlinear inverse problems in \cite{Neu89}. Source conditions in weak form have been developed in \cite{EngZou00}. Tikhonov regularization for nonlinear inverse problems have been analyzed in \cite{Sch93b}. This method can overcome the saturation result of Tikhonov regularization, where $\mathcal{O}(\sqrt{\delta})$ is optimal (see \cite{Gro83,Gro84}). To overcome the saturation is also the topic of iterative Bregman regularization for learning (see for instance \cite{GruHolLehHocZel24_report}).

Convergence rates for finite dimensional approximations of Tikhonov regularization in Banach spaces have been published in \cite{PoeResSch10}.

Instead of evaluating $\opo$ exactly one can formulate a constraint $\y = F[\x]$ and solve for a constrained Tikhonov functional. It is often referred to as an \emph{all-at-once} approach\index{all-at-once approach} (see for instance \cite{Kal16,KalNguSch21}) and bilevel optimization \cite{ChaKun94a,ItoKun00,ItoKun99}. The all-at-once approach also allows for convergence rates results.

The main interest of this chapter was in an analysis of regularization methods in an infinite dimensional setting. Iterative methods for solving \emph{linear} ill-posed problems have been extensively studied in the infinite dimensional setting, see for instance \cite{Han95,Gro84,EngHanNeu96}. For nonlinear ill-posed problems the literature is much sparser: We refer to the books \cite{KalNeuSch08,BakGon89,BakKok04}. A Levenberg-Marquardt iteration has been analyzed in \cite{Han97}. More on gradient descent methods can be found in \cite{ReaJin20}. \cite{Ram99} considered \autoref{eq:lwm} where  $\opo$ is approximated during the iteration by operators $F_\ttn$ which are possibly easier to evaluate. For the analysis of iterative methods for solving finite dimensional linear inverse problems we refer to \cite{Han10}. The analysis of iterative regularization methods for solving linear inverse problems in infinite dimensional Hilbert spaces has been developed for instance in \cite{Gro84,Han95,Gro11,Gil77}. Regularization theory of iterative methods for solving nonlinear inverse problems has been developed in \cite{Bak92,HanNeuSch95,Sch96,Bla96,BlaNeuSch97,BakKok04,KalNeuSch08,SchGraGroHalLen09,SchuKalHofKaz12}.
More recent references on Landweber-types methods in Banach-spaces are published in \cite{Jin25,HuaJinLuZha25}, containing also steepest descent and minimal error methods, which have been proposed in an Hilbert space setting in \cite{Sch96}.

A recent trend also concerns the analysis of dynamical flow for solving linear inverse problems, which probably goes back to \cite{Sho67}. In contrast to \emph{scale space methods}\index{scale space}, which are often used for filtering (see for instance \cite{AlvGuiLioMor93,AlvLioMor92,Wei97a,Wei98}) regularization methods lead to \emph{inverse scale space methods}\index{inverse scale space}, introduced in \cite{SchGro01}, which has been extended from the Hilbert space norm setting to the Bregman setting in
\cite{Gil18,BurGilOshXu06,ShiOsh08b,LieNor07b,BurOshXuGil05,BurGilMoeEckCre16}.

More recent advances are \cite{BotDonElbSch22,JinHua24,GroSch00,Tau94,HocHonOst09,JinWan18,JinLu14,ZhaHof20,KalNeuRam02}. Several of this work originated from \cite{Nes05,AttChbPeyRed18,AttPeyRed16,Nes83,Pol64}. A convergence analysis for a continuous
analogue of the iteratively regularized Gauss-Newton-method was derived in \cite{KalNeuRam02}.
	
\emph{Source conditions} like \autoref{eq:source_var}, \autoref{eq:source_var_l} play a significant role in the quantitative analysis of regularization methods and occur in various forms (see \cite{Hof93,HofKalPoeSch07,HofMat07,Neu97,Mat04,HeiHof09,HofSch98,SchGraGroHalLen09,EngKunNeu89,Fle11a,HohWei17a}). Source conditions related to machine learning have been considered for instance in \cite{BenBubRatRic24}.

\emph{Over-smoothing} is a phenomenon, which appears when the solution $\x_* \notin \X$, in other words the regularization norm is smoother than the solution $\x_*$; Note that in this case we do not have a minimum-norm solution. In this case convergence can still be proven if the operator $\opo$ satisfies Hilbert scales. See for instance \cite{Nat84,HofHofMatPla21,HofKinMat19,HofMat20}. Similar analysis is also performed with \emph{weak noise}\index{weak noise} (see \cite{EggLarNas09,EggLarNas12}).

Machine learning techniques have been introduced to learn the step sizes $\set{\mu_k: k \in \N_0}$ of gradient descent methods (see \cite{AdlOkt17,AdlOkt18}, where more general methods have been considered)
\begin{equation} \label{eq:gan_landweber}
	\xkpd = \xkd - \mu_k \opd{\xkd}^* (\op{\xkd} - \y^\delta)\;.
\end{equation}

Variational methods for solving \emph{discrete inverse problems} have been studied in (see \cite{Han10}). Bayesian formulations (see \cite{KaiSom05,CalSom18,CalSom23}) result in discrete inverse problems although of course can solve statistical challenges. Data driven ideas are included, for instance to derive adequate regularization and fit-to-data functionals, such as
\begin{equation*}
	\mathcal{T}_{\y^\delta,E}[\x] := \norms{F[\x] -\y^\delta}_\Y^2 + \norms{\x-\x^0}_E^{2}
\end{equation*}
where $E$ is some learned semi-norm (see for instance \cite{Stu10,KaiSom05}). It is interesting to note that when regularization methods are generalized from functions to random-variables the same convergence rates conditions apply see \cite{EngHofKin05} and \autoref{sec: Polregres}.

Variational regularization also plays a big role in \emph{data assimilation} (see \cite{NakPot15}).

We emphasize in this book on \emph{nonlinear} inverse problems in Hilbert spaces: The restriction to Hilbert spaces is certainly a limitation, but it has been chosen on purpose to limit the complexity of regularizations methods to the nonlinearity of the operator. The extension of classical regularization theory to Banach spaces has been considered for instance in \cite{SchuKalHofKaz12,SchGraGroHalLen09,Res05,BonKazMaaSchoSchu08} and in particular for sparsity regularization and compressed sensing\index{compressed sensing} (see for instance \cite{GraHalSch08,BonBreLorMaa07,CanRomTao06,CanRomTao06b,DauDefMol04,Don06,Ram08,Lor08,Tro06} to mention but a few).

\chapter[Neural network functions]{Neural network functions and approximations} \label{ch:nn}
In this chapter we define generalized neural network functions and derive approximation results, which can be used as basis for efficient implementations of regularization methods (see \autoref{ex:c_reconstruct}). 

Early approximation theoretical results of neural networks where for \emph{shallow neural networks}\index{neural network!shallow}, while 
recently it became popular to study \emph{deep neural networks}\index{neural network!deep}:
\begin{definition}[Deep and shallow networks]\label{de:network}
	Let $\sigma: \R^{\hat{m}} \to \R^n$ be a function and let $A_\rho: \R^m \to \R^{\hat{m}}$, where $\rho$ is a parametrization vector; See \autoref{ta:nn}, where we see several examples and one also sees that the vector can be chosen from a countable or uncountable family of elements, or with mixed components. A \emph{neural network ansatz function} $N$ \index{neural network!ansatz function} is composed of an activation function $\sigma$ and a transformation $\rho$, typically an affine linear function (see for instance \autoref{eq:classical_approximation}).
	\begin{table}[H]
		\begin{center}
			\begin{tabular}{r|l}
				\hline
				activation function $\sigma$ & transformation $A_\rho$\\
				\autoref{ss:act} & \autoref{se:nnf} \\
				\hline
			\end{tabular}
			\caption{\label{ta:nn2} A neural network is the composition of an activation function and a transformation. Several options will be discussed in the following subsections.}
		\end{center}  
	\end{table} 
	\begin{itemize}
		\item The family of \emph{shallow} neural network functions consists of functions
		\begin{equation*}
			\vs \in \R^m \mapsto N(\vs) := \sum_{i \in \N} a_i \sigma(A_{\rho_i}(\vs)) \in \R^n\,,
		\end{equation*}
		presumably the limit exists.
		\item A \emph{deep} neural network is a composition of shallow neural networks:
		\begin{equation*}
			\vs \in \R^m \mapsto  N_{\tilde{m}}\left(N_{\tilde{m}-1}\left(\ldots N_1(\vs)\right)\right)\;.
		\end{equation*}
	\end{itemize}
\end{definition}
The starting point of mathematics of neural networks is difficult to be traced back to a single paper. However, a few ground breaking publications on shallow neural networks (in the spirit of \autoref{de:network}) stand out:
\begin{enumerate}
	\item Norbert Wiener \cite{Wie32} made a significant and inspiring contribution: He answered the following research question:\footnote{Norbert Wiener is considered the father of \emph{cybernetics}\index{cybernetics}, which has its origins in interdisciplinary interactions between numerous disciplines around 1940. Initial developments included already neural networks, information theory and control (see \cite{RosWieBig43}).}
	\bigskip\par
	\fbox{
			\parbox{0.90\textwidth}{
	\begin{center}
		When can an arbitrary univariate function from $\R$ to $\R$ or $\C$ be approximated by linear combinations of \emph{translates} of a \emph{single function} $\sigma$ with arbitrary accuracy?
	\end{center}}}
\bigskip\par
    These answers to such questions are commonly called \emph{Tauber-Wiener-theorems}.\index{Theorem!Tauber-Wiener}
  \bigskip\par
    	\fbox{
    	\parbox{0.90\textwidth}{
    		\begin{center}
    The choice of the function $\sigma$ is crucial in mathematical theory and therefore we try to give a rather complete picture on usable and reasonable functions.
    	\end{center}}}
\bigskip\par
    \item \emph{Wavelet bases} (see \cite{Dau88,Dau92,Mey93}) approximate a function with arbitrary precision by	
    \bigskip\par
    \fbox{
    	\parbox{0.90\textwidth}{
    \begin{center}
    	\emph{scalings} and \emph{translations} of a single function.
    \end{center}}}
\bigskip\par
    In wavelet theory there exist two approaches:	

    \begin{enumerate}
    	\item \emph{The continuous wavelet transform} (see \cite{Mey93}), is a continuous transforms of the function $\x$ into a function of all possible scalings and translations (which is an uncountable set), and the
    	\item \emph{discrete wavelet transform} (see \cite{Dau92}), which maps the function $\x$ into a function of dyadic scalings and translations (which is a countable representation).\index{wavelet transform!continuous}\index{wavelet transform!discrete}
    \end{enumerate}
\end{enumerate}
The relevance for efficient implementation of neural network and wavelet based algorithms (for instance on hardware architecture \cite{BocFonKut23}) is the possibility of approximation of an arbitrary function $\x$ by linear combinations of scalings and translations of a \emph{single} function $\sigma$.\footnote{It is the personal experience of an author with Pentium III and Pentium IV processors that computation with Pentium IV was slower but it had other computational benefits (see \url{https://retrocomputing.stackexchange.com/questions/14890/is-it-true-that-pentium-iii-was-faster-than-its-successor-pentium-4}). As a consequence look-up tables where used instead of function evaluations. Clearly, if only function needs to be evaluated it is a clear storage benefit.}
\begin{figure}%[h]
	\begin{center}
		\includegraphics[width=.8\linewidth]{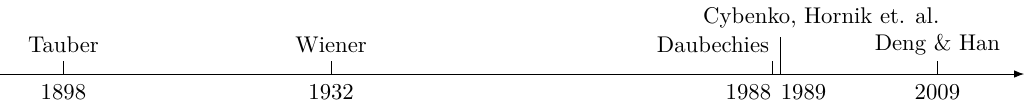}
		\caption{\label{fig:timeline} Historical timeline of important publications on neural networks and wavelets. In addition to the ground-breaking references on neural networks and wavelets we add Deng \& Han \cite{DenHan09}, whose work allows for proving convergence rates for neural network approximations. A result which is required for quantitative analysis of regularization methods.}
	\end{center}
\end{figure}
Below we use the following three observations as a guide star:
\begin{itemize}
	\item In \cite{Wie32} conditions on the expanding function $\sigma$ are formulated. These conditions are formulated in terms of the Fourier-transform or expansion (see \autoref{Wiener_Tauberian} below). In the literature the expanding functions are called \emph{Tauber-Wiener} functions. Neural networks are defined via \emph{activation functions} (again denoted by $\sigma$). They are not Tauber-Wiener-functions, because they are not $L^1$-integrable. However, in the literature they are used synonymously. Is this identification justified? It actually is, because a Tauber-Wiener $\sigma$ is used to approximate arbitrary functions on the unbounded domain $\R$, while neural network papers approximates functions on bounded domains (see \cite{Cyb89}). Therefore, commonly, compactified activation functions can be used and they are Tauber-Wiener-functions as we show
	(see \autoref{ss:act}).
    \item Another aspect is the dimensionality: The Tauber-Wiener and the universal approximation theorems show that one can approximate an arbitrary function $\x$ with a given accuracy with a \emph{finite} linear combination of scaling and shifts of a single function, which forms an \emph{uncountable} family. this is in contrast to discrete wavelets, which approximate $\x$ with a linear combination from a \emph{countable} family of function. The overall question is therefore, how many shifts and scalings are actually needed?
    \item Tauber-Wiener theorems and wavelet approximations can be generalized in various ways to multi-variate functions. We discuss several possibilities.
\end{itemize}

\section{Tauber-Wiener theorems} %(see \autoref{ch:fourier})
Tauber-Wiener-theorems (see \cite{Wie32}) give answers to the question whether a function $\x \in L^1(\R)$, $L^2(\R)$, respectively, can be approximated by linear combinations of translations of a given function $\sigma \in L^2(\R)$.\footnote{Wiener-did not distinguish between real and complex valued functions. We concentrate on real valued functions for the sake of simplicity of presentation.} See \autoref{ch:fourier} for background information on the Fourier-transform.
\begin{theorem}[Tauber-Wiener-theorem] \index{Theorem!Wiener's} \label{Wiener_Tauberian}
	\begin{enumerate}
		\item Let %the Tauber-Wiener-function (defined in \autoref{de:tw} below) satisfy
		$\sigma \in L^2(\R)$. Then the following statements are equivalent:
		\begin{itemize}
			\item For every function $\x \in L^2(\R)$ and every $\ve >0$ there exists $\noc \in \N$ and a set
		          $\set{(\alpha_k,\theta_k) : k=1,\ldots,\noc}$ such that the function
		          \begin{equation} \label{eq:wiener_property}
			         s \in \R \mapsto \x_\noc(s) = \sum_{k=1}^\noc \alpha_k \sigma(s + \theta_k)
		          \end{equation}
	              satisfies
	              \begin{equation*}
	              	 \norm{\x-\x_\noc}_{L^2(\R)} < \ve.
	              \end{equation*}
            \item The \emph{Lebesgue measure} of the zeros of the Fourier-transform of $\sigma$ vanishes
            \index{Lebesgues-measure}, i.e.,\index{meas!Lebesgues-measure}
            \begin{equation} \label{eq:zeros_fourier_I}
            	\text{meas}\set{s : \mathcal{F}[\sigma](s) = 0} = 0\;.
            \end{equation}

        \end{itemize}
            \item Let $\sigma \in L^1(\R)$. The following statements are equivalent:
        \begin{itemize}
        \item For every function $\x \in L^1(\R)$ and every $\ve >0$ there exists $\noc \in \N$ and a set
        $\set{(\alpha_k,\theta_k) : k=1,\ldots,\noc}$ such that the function $\x_\noc$ defined in \autoref{eq:wiener_property}
        satisfies
        \begin{equation*}
        	\norm{\x - \x_\noc}_{L^1(\R)} < \ve.
        \end{equation*}
        \item The Fourier-transform of $\sigma$, $\mathcal{F}[\sigma]$, satisfies
        \begin{equation} \label{eq:zeros_fourier_II}
        	\set{s : \mathcal{F}[\sigma](s) = 0} = \emptyset\;.
        \end{equation}
	    \end{itemize}
	\end{enumerate}
\end{theorem}
\begin{proof}
	See Theorem I and II in \cite{Wie32}.
\end{proof}
Note that the sampling of the approximating function in \autoref{eq:wiener_property} is dependent on $\x$ and $\ve$.

\begin{remark}
	The Tauber–Wiener-theorem (Theorem~\ref{Wiener_Tauberian}) applies to all $\x \in L^1(\R)$ and $L^2(\R)$, including those with compact support or rapid decay. However, the structure of the Fourier-transform in these two special classes simplifies the vanishing conditions.
\begin{itemize}
	\item For a nonzero compactly supported function $\sigma \in L^2(\R)$, its Fourier-transform $\mathcal{F}[\sigma]$ is real-analytic (in fact, entire of exponential type). Consequently, its zero set in $\R$ consists only of isolated points and thus has zero Lebesgue measure. Hence, for $L^2(\R)$, every nonzero
	compactly supported $\sigma$ automatically satisfies the Tauber–Wiener-property: for every $\x \in L^2(\R)$ and every $\ve > 0$, there exists $\noc \in \N$ and coefficients $(\alpha_k,\theta_k)$ such that $\norm{\x - \sum_{k=1}^{\noc} \alpha_k \sigma(\cdot + \theta_k)}_{L^2(\R)} < \ve$. In contrast, for $L^1(\R)$ with compact support, the stronger condition $\mathcal{F}[\sigma](\xi) \neq 0$ for all $\xi \in \R$ is required; a compactly supported $\sigma \in L^1(\R)$ whose Fourier-transform has any real zero (e.g., $\sigma = \mathbf{1}_{[-1,1]}$) fails to generate a dense set of translates in $L^1(\R)$.
	\item For rapidly decaying (Schwartz) functions, the situation is reversed: the Fourier-transform is also Schwartz, and it may vanish on intervals of positive measure. For example, take $\sigma$ such that $\mathcal{F}[\sigma]$ is a nonzero smooth bump function supported on $[-1,1]$. Then $\mathcal{F}[\sigma]$ vanishes on a set of positive measure, so by the $L^2$ version of the theorem, the translates of $\sigma$ are not dense in $L^2(\R)$. In the $L^1$ case, the same $\sigma$ fails the condition because $\mathcal{F}[\sigma](\xi) = 0$ for $|\xi| > 1$. Thus, while compactly supported functions are always "good" for $L^2$-approximation, Schwartz functions can be "bad" in both $L^1$ and $L^2$ if their Fourier-transform is compactly supported. This highlights the interplay between the smoothness/decay of $\sigma$ and the size of the zero set of $\mathcal{F}[\sigma]$ in the Wiener–Tauberian theory.
\end{itemize}
\end{remark}
\begin{definition}[Tauber-Wiener-function] \label{de:tw} A function $\sigma:\R \to \R$ is called \emph{Tauber-Wiener-function}\index{function!Tauber-Wiener}\footnote{In \cite{Wie32} functions have also been considered complex valued.} if it satisfies \autoref{eq:zeros_fourier_I}.
\end{definition}

\begin{remark} \label{re:tw} For a function $\x_\noc$ as defined in \autoref{eq:wiener_property} the Fourier-transform is given by
	\begin{equation}\label{eq:mot}
		\mathcal{F}[\x_\noc](\omega) = \sum_{k=1}^\noc \alpha_k \mathcal{F}[\sigma(\cdot+\theta_k)](\omega) =
		\mathcal{F}[\sigma](\omega) \left(\sum_{k=1}^\noc \alpha_k \exp^{\i \omega \theta_k} \right)\;.
	\end{equation}
    If $\mathcal{F}[\sigma]$ vanishes at most on a set of measure $0$, we can divide by $\mathcal{F}[\sigma]$ and get
    \begin{equation}\label{eq:explicit}
    \frac{\mathcal{F}[\x_\noc](\omega)}{\mathcal{F}[\sigma](\omega)} = \sum_{k=1}^\noc \alpha_k \exp^{\i \omega \theta_k} \text{ almost everywhere}
    \end{equation}
    with respect to the Lebesgue measure. Note that the coefficient
    \begin{equation*}
    	\set{\alpha_k=\alpha_k^\ve : k=1,\ldots,M=M(\ve)}
    \end{equation*}
    of \autoref{eq:explicit}, depend on the specified accuracy $\ve$. How the coefficients behave when $\ve \to 0$ is commonly refered to \emph{Tauber-Wiener-theorems}.\index{Theorem!Tauber-Wiener} Important work in this area is for instance \cite{HarLit21,Har31,Wie32,Har49} dating back to \cite{Tau97}. Different variants have been considered with functions defined on bounded intervals and on $\R$.
\end{remark}
\begin{remark}
	The fight hand side of \autoref{eq:explicit} is nothing else than a \emph{non uniform discrete Fourier-transform}
	(see for instance \cite{FesSut03b,LeeGre05,PloPotSteTas18,PotSte03}). This means that the the Tauber-Wiener theorem 
	guarantees that $\frac{\mathcal{F}[\x_\noc]}{\mathcal{F}[\sigma]}$ provides the existence of a non equidistant Fourier-series. 
\end{remark}
In the following we define energies of functions, which are computed from the coefficients of Tauber-Wiener-approximations (see \autoref{eq:wiener_property})

\begin{definition}[$\mathcal{L}^1$-energy] \label{de:ell1}
	Let $\x, \sigma \in L^p(\R)$, with $p=1,2$, then, by making use of \autoref{eq:wiener_property} we define
	\begin{equation} \label{eq:normve}
		\norms{\x}_{\mathcal{L}^1}^{\ve,p} := \inf_{\alpha_k,\theta_k \in \R} \set{ \sum_{k=1}^\noc \abs{\alpha_k} : \norm{\sum_{k=1}^\noc \alpha_k \sigma(\cdot+\theta_k)-\x(\cdot)}_{L^p} \hspace*{-0.02\textwidth} < \ve}\;.
	\end{equation}
	Taking the limit $\ve \to 0+$ (note that the $\norms{\tt g}_{\mathcal{L}^1}^{\ve,p}$ is monotonically increasing), we define the $\mathcal{L}^1$-energy\index{energy!$\mathcal{L}^1$}
	\begin{equation} \label{eq:norm}
		\norms{\x}_{\mathcal{L}^1}^{p} := \lim_{\ve \to 0+} \norms{\x}_{\mathcal{L}^1}^{\ve,p}\;.
	\end{equation}
\end{definition}
Because of the conceptual similarity to the original definition of the $\mathcal{L}^1$-energy
(see \autoref{de:ell1space} below) we use the same name here. The only difference in the definition is that we use explicitly the $\ve$-accuracy.

We also explain the similarity of the $\mathcal{L}^1$-energy with the definition of energies on \emph{Barron-vector spaces}\index{Barron!vector space} (see \autoref{sec:barron}) and the definition of Hardy-spaces (see \cite{CoiWei77} and \cite{Ste70}), respectively. We define the following energies of Barron-type:
\begin{definition}[Barron-type energy]\label{de:barron_energies}
First, we define two Barron-type energies on sets of generalized functions, distributions (see \autoref{sec:distribution}), respectively:
\begin{itemize}
   \item For a function $\x: \Omega \subseteq \R^m \to \R$, which admits the representation
    \begin{equation}\label{eq:barron_rep_TW}
    \x(s) = \lim_{\noc \to \infty} \x_\noc(s) \text{ with }
    \x_\noc(s):= \sum_{k=1}^\noc \alpha_k^\noc	\sigma(s + \theta_k^\noc)\,,
    \end{equation}
    we define the \emph{multiplicative Barron-energy}\index{Barron!multiplicative energy} as the limes inferior of all Tauber-Wiener-approximations (see \autoref{eq:wiener_property}) that converge to $\x$ in a distributional setting:
\begin{equation} \label{eq:barron_norm_TW}
\norm{\x}_{\mathcal{B}_p}^p := \inf_{\x_\noc \to \x} \sum_{k=1}^\noc \abs{\alpha_k^\noc}^p (1+\abs{\theta_k^\noc})^p\;.
\end{equation}
The $1$ in $(1+\abs{\theta_k^\noc})^p$ is used to avoid that for small values $\theta_k^\noc$, $k=1,\ldots,\noc$ the multiplicative Barron-energy becomes small. This is consistent with the standard Barron-energy (see \autoref{eq:barron_norm}), where the vector $\vw$ is a one dimensional constant $1$. In this case the Barron-energy is approximating the $\mathcal{L}^1$-norm defined in \autoref{de:ell1space}.
\item The \emph{additive Barron-energy}\index{Barron!additive energy} is defined by
\begin{equation} \label{eq:barron_norm_TW2}
	\norm{\x}_{\mathcal{B}_p}^p := \inf_{\x_\noc \to \x} \sum_{k=1}^\noc \left(\abs{\alpha_k^\noc}^p + \abs{\theta_k^\noc}^p \right)\;.
\end{equation}
\item
The \emph{Barron-vector space}\index{Barron!vector space} consists of all tempered distributions (see \autoref{sec:distribution}, for which the Barron-energy is finite. The additive and multiplicative Barron-energies can define two possible different spaces.
\end{itemize}
\end{definition}

\begin{remark} In \autoref{eq:normve} we are looking for an absolutely convergent sequences of Fourier-coefficients, which the quotient of two functions can attain. Norbert Wiener's result actually prevents us from divergent series (see \cite{Har49}).
\end{remark}

\begin{figure}[H]
	\centering
		\begin{subfigure}[t]{0.45\linewidth}
		\begin{tikzpicture}[samples=100,domain=0:60,scale=.65,every node/.style={scale=0.65}]]
			\pgfmathdeclarefunction{sinc}{1}{%
			\pgfmathparse{abs(#1)<0.001 ? int(1) : int(0)}%
			\ifnum\pgfmathresult>0 \pgfmathparse{1}\else\pgfmathparse{sin(#1 r)/#1}\fi%
		}
		\pgfmathdeclarefunction{sincd}{1}{%
			\pgfmathparse{abs(#1)<0.01 ? int(1) : int(0)}%
			\ifnum\pgfmathresult>0 \pgfmathparse{0}\else\pgfmathparse{(#1*cos(#1 r)-sin(#1 r))/(#1*#1)}\fi%
		}
		\draw[very thin,color=gray] (-0.1,-1.1) grid (6.1,3.1);
		\draw[thick,->] ( -0.2, 0    ) -- ( 6.2, 0   );
		\draw[thick,->] (    0, -1.2 ) -- ( 0  , 3.2 );
		\draw[color=black,smooth] plot (\x/10,{sqrt(2*pi)*sinc(\x/(2*pi))})
			node[right,yshift=2mm] {$\Re(\mathcal{F}[\sigma](\omega))$};
		\draw[color=red,smooth] plot (\x/10,{sqrt(2*pi)*sincd(\x/(2*pi))})
			node[right,yshift=-1mm] {$\Im(\mathcal{F}[\sigma](\omega))$};
	\end{tikzpicture}
	\caption{\label{fi:fcla} The common zeros of $\Re(\mathcal{F}[\sigma](\omega))$ and $\Im(\mathcal{F}[\sigma](\omega))$ are empty such that the shifts of a compactified linear activation function (see \autoref{eq:ReLU_general}) form a basis of $L^2(\R)$. Here $\mathcal{F}[\sigma]$ (see \autoref{eq:Fourierheaviside}) is plotted for parameters $b=c=1$.}
	\end{subfigure}
	\hfill
	\begin{subfigure}[t]{0.45\linewidth}
		\centering
		\begin{tikzpicture}[samples=100,domain=0:60,xscale=.65,yscale=1.5*.65,every node/.style={scale=0.65}]
		    \pgfmathsetmacro{\c}{1}
			\pgfmathsetmacro{\a}{2}
			\pgfmathsetmacro{\b}{1}
			\pgfmathsetmacro{\M}{max(-\c, -\b/\a)} % Calculate M

			\draw[very thin,color=gray] (-0.1,-1.1) grid (6.1,1.1);
			\draw[thick,->] ( -0.2, 0    ) -- ( 6.2, 0   );
			\draw[thick,->] (    0, -1.2 ) -- ( 0  , 1.2 );

			\draw[color=black,smooth] plot (\x/10,{(cos(\x*\c r) - cos(\x*\M r))/sqrt(2*pi)})
				node[right,yshift=2mm] {$\Re(\mathcal{F}[\sigma](\omega))$};

			\draw[color=red,smooth] plot (\x/10,{(-sin(\x*\c r) + sin(\x*\M r))/sqrt(2*pi)})
				node[right,yshift=-1mm] {$\Im(\mathcal{F}[\sigma](\omega))$};
		\end{tikzpicture}
		\caption{\label{fi:fchs}$\Re(\mathcal{F}[\sigma](\omega))$ and $\Im(\mathcal{F}[\sigma](\omega))$ from  \autoref{eq:fchs} with  $c=1$ and $b/a=0.5$ have common zeros such that linear combinations form a basis of $L^1(\R)$. }
	\end{subfigure}\\%
	\begin{subfigure}[t]{0.45\linewidth}
		\centering
		\begin{tikzpicture}[samples=100,domain=1:60,xscale=.65,yscale=1.2*.65,every node/.style={scale=0.65}]
			\pgfmathsetmacro{\c}{1} % Example value for c
			\pgfmathsetmacro{\b}{1} % Example value for b
			\pgfmathsetmacro{\a}{2} % a is given as 2

			%\draw[very thin,color=white!80!gray,ystep=0.1] (-0.1,-2.3) grid (8.1,4.1);
			%\draw[very thin,color=gray               ] (-0.1,-2.3) grid (8.1,4.1);
			\draw[very thin,color=gray] (-0.1,-1.1) grid (6.1,1.1);
			\draw[thick,->] ( -0.2, 0    ) -- ( 8.2, 0   );
			\draw[thick,->] (    0, -1.4 ) -- ( 0  , 1.2 );
			\draw[color=black,smooth] plot (\x/10,{4*sin(\x/2 r)*(\x + 2*\x*cos(\x r) - sin(\x r))/sqrt(2*pi)/((\x)^2)})
				node[right,yshift=-3mm] {$\Re(\mathcal{F}[\sigma](\omega))$};
			\draw[color=red,smooth] plot (\x/10,{4*cos(\x/2 r)*(-\x + 2*\x*cos(\x r) - sin(\x r))/sqrt(2*pi)/((\x)^2)})
				node[right,yshift=3mm] {$\Im(\mathcal{F}[\sigma](\omega))$};
		\end{tikzpicture}
		\caption{\label{fi:fcReLU} $\Re(\mathcal{F}[\sigma](\omega))$ and $\Im(\mathcal{F}[\sigma](\omega))$ from  \autoref{eq:fcReLU} with  $c=1$ and $b/a=0.5$ have no common zeros such that linear combinations form a basis of $L^2(\R)$. }
	\end{subfigure}
	\hfill
	\begin{subfigure}[t]{0.45\linewidth}
		\centering
		\begin{tikzpicture}[samples=100,domain=1:60,xscale=.65,yscale=2*.65,every node/.style={scale=0.65}]
			\pgfmathsetmacro{\a}{1} % a is given as 1
		    \pgfmathsetmacro{\b}{0} % b is given as 0

			\draw[very thin,color=gray] (-0.1,-0.11) grid (6.1,1.1);
			\draw[thick,->] ( -0.2, 0    ) -- ( 6.2, 0   );
			\draw[thick,->] (    0, -0.12 ) -- ( 0  , 1.2 );

			\draw[color=black,smooth] plot (\x/10,{exp(-(\x/10)^2/2)})
				node[right,yshift=2mm] {$\Re(\mathcal{F}[\sigma](\omega))$};

			\draw[color=red,smooth] plot (\x/10,{0})
				node[right,yshift=-1mm] {$\Im(\mathcal{F}[\sigma](\omega))$};
	    \end{tikzpicture}
	    \caption{\label{fi:fgaussion} $\Re(\mathcal{F}[\sigma](\omega))$ and $\Im(\mathcal{F}[\sigma](\omega))$ from  \autoref{eq:fgaussian} with  $b=0$ and $a=1$ have no common zeros such that linear combinations form a basis of $L^2(\R)$. }
	\end{subfigure} \\
	\begin{subfigure}[t]{0.45\linewidth}
	\centering
	\begin{tikzpicture}[xscale=.65,yscale=.65/3,every node/.style={scale=0.65}]
		\draw[very thin,color=gray] (-0.1,-1) grid (10.1,7);
		\draw[thick,->] ( -0.2,  0 ) -- ( 10.2, 0   );
		\draw[thick,->] (    0, -1 ) -- ( 0  , 7.1 );

		% Plot the real part
		\draw[color=black] plot[smooth,mark=none] coordinates {
	(0,0)
	(0.1,-7.590692511467564e-15) (0.2,-1.9461837236409147e-15) (0.3,-2.772131118649356e-15)
	(0.4,1.8953304386965303e-15) (0.5,-1.6198176232009453e-15) (0.6,2.1197927559271068e-15)
	(0.7,8.873393632454936e-16) (0.8,0.) (0.9,2.1603835292930007e-15)
	(1.,0.) (1.1,6.311799773971112e-16) (1.2,-4.0457184797122157e-16)
	(1.3,-2.8813472126679593e-16) (1.4,-1.9700024432046923e-16) (1.5,6.734540025530212e-16)
	(1.6,1.438895085889114e-16) (1.7,-4.3040603879193195e-16) (1.8,1.8392024750226415e-16)
	(1.9,1.571848644886075e-16) (2.,1.1514500889770288e-16) (2.1,3.9362827890718494e-16)
	(2.2,-3.9247704701194444e-16) (2.3,2.875071541866546e-16) (2.4,-2.0476163665167858e-16)
	(2.5,-4.8124083130216156e-17) (2.6,-2.991169008678222e-17) (2.7,1.0225442857471595e-16)
	(2.8,0.) (2.9,-2.9874763921045595e-16) (3.,0.) (3.1,1.7729243969429453e-16)
	(3.2,9.324338544027555e-17) (3.3,2.7891203858062985e-16) (3.4,3.4052609824031657e-16)
	(3.5,-2.910258622073653e-16) (3.6,-2.487211785270229e-17) (3.7,2.1256607292101197e-17)
	(3.8,-6.358331634092484e-17) (3.9,2.484141306994394e-16) (4.,-5.307591469852438e-17)
	(4.1,0.) (4.2,-3.101343274053646e-16) (4.3,1.6565747548362486e-16) (4.4,7.830969077392353e-17)
	(4.5,1.8149499882834693e-16) (4.6,1.861345915979749e-16) (4.7,-1.4140206459316842e-16)
	(4.8,-1.8127094213386866e-16) (4.9,-2.5820112124671134e-17) (5.,0.)
	(5.1,-1.6973171441133803e-16) (5.2,-3.223529629642784e-17) (5.3,-2.7549444658298108e-17)
	(5.4,2.82536965939936e-16) (5.5,-2.0122188515316708e-17) (5.6,4.2992866717343497e-17)
	(5.7,4.4091901924126345e-17) (5.8,2.512168443699471e-17) (5.9,1.0734947939611612e-16)
	(6.,1.2844262141519094e-16) (6.1,-1.568166972671887e-17) (6.2,3.886614647832768e-16)
	(6.3,0.) (6.4,4.405024858726766e-16) (6.5,2.3424755333359504e-16) (6.6,-3.145943299595761e-16)
	(6.7,-4.644008318859943e-16) (6.8,-5.22228511797999e-16) (6.9,-4.853678734888834e-16)
	(7.,-5.111257278818486e-16) (7.1,-7.954153216601446e-16) (7.2,-8.915286139918804e-17)
	(7.3,3.4286960150669136e-16) (7.4,8.628067409186637e-16) (7.5,9.008627197415212e-16)
	(7.6,8.561158171769256e-16) (7.7,6.503710193812338e-16) (7.8,2.6054554029805653e-16)
	(7.9,-1.484477321050325e-17) (8.,-5.709094913232725e-16) (8.1,-8.836769775281127e-16)
	(8.2,8.895868347898259e-16) (8.3,-5.22687641701614e-16) (8.4,-3.2487830902596847e-16)
	(8.5,9.255089941254042e-17) (8.6,2.3729199117590086e-16) (8.7,-8.449928248736505e-16)
	(8.8,7.779031115161147e-16) (8.9,6.912472371890526e-16) (9.,2.1098741900606378e-16)
	(9.1,3.606348873020188e-17) (9.2,1.6951635648138055e-16) (9.3,-6.058399919942638e-16)
	(9.4,-7.006814894259036e-16) (9.5,-6.685706011864418e-16) (9.6,3.452971101795244e-16)
	(9.7,1.4052540036333442e-16) (9.8,2.1617651870266741e-16) (9.9,6.055767704577998e-16)
	(10.,6.622856267026505e-16)
};
\node[right] at (10, 0.5) {$\Re(\mathcal{F}[\sigma](\omega))$};

% Plot the imaginary part
\draw[color=red] plot[smooth,mark=none] coordinates {
	(0,0)
	(0.1,2.3663855991043676) (0.2,4.434438591074961) (0.3,5.9546808519610845)
	(0.4,6.766365782076449) (0.5,6.821159281467363) (0.6,6.186564785116135)
	(0.7,5.028948705649815) (0.8,3.5799048058723826) (0.9,2.0927684901833543)
	(1.,0.7977551779009409) (1.1,-0.135851104739 70885) (1.2,-0.6237385979749052) (1.3,-0.6730904761572879)
	(1.4,-0.37144085365525087) (1.5,0.13880914596612454) (1.6,0.6942391887082351)
	(1.7,1.1453297036647203) (1.8,1.3853939130806485) (1.9,1.3681242118814612)
	(2.,1.1110377933386593) (2.1,0.6852013760428599) (2.2,0.19445323920328278)
	(2.3,-0.2506629166062622) (2.4,-0.5582636375829698) (2.5,-0.6743519000930475)
	(2.6,-0.5921507734837889) (2.7,-0.35019003195349346) (2.8,-0.020344008080910102)
	(2.9,0.31083907234596314) (3.,0.5628931772616103) (3.1,0.6794669829518191)
	(3.2,0.639719144643949) (3.3,0.46088441580246153) (3.4,0.19194015277503484)
	(3.5,-0.09967222811318785) (3.6,-0.344839065767523) (3.7,-0.4890256851819846)
	(3.8,-0.5042663849903829) (3.9,-0.39434016411995776) (4.,-0.19227610203770468)
	(4.1,0.04885393381765418) (4.2,0.2693883817195809) (4.3,0.41743707477761294)
	(4.4,0.460752990408652) (4.5,0.3934484197202417) (4.6,0.2362087121248139)
	(4.7,0.03023422927918518) (4.8,-0.17340027847087966) (4.9,-0.32627475625716734)
	(5.,-0.39401259533599386) (5.1,-0.3638881904064395) (5.2,-0.24700495769046735)
	(5.3,-0.07473766913898909) (5.4,0.10950265306102533) (5.5,0.26105756813798275)
	(5.6,0.34471335109825235) (5.7,0.3427632073472652) (5.8,0.2586177431454417)
	(5.9,0.11525962616369069) (6.,-0.0509148689665602) (6.1,-0.19924775906136535)
	(6.2,-0.2946870197900714) (6.3,-0.3159957123406355) (6.4,-0.2603982764363121)
	(6.5,-0.14366742828972895) (6.6,0.004228135279315555) (6.7,0.14674063782402008)
	(6.8,0.24970677173372988) (6.9,0.28947521831057466) (7.,0.2582785700495455)
	(7.1,0.16564337853543262) (7.2,0.03561301114104122) (7.3,-0.09943455722336794)
	(7.4,-0.2067989033272053) (7.5,-0.2613290132110299) (7.6,-0.25127101203225666)
	(7.7,-0.1807426232821225) (7.8,-0.06832876024128655) (7.9,0.05776450078798545)
	(8.,0.16674615813526172) (8.1,0.23272141579056393) (8.2,0.24080901621774817)
	(8.3,0.19045788966092336) (8.4,0.09523364109166996) (8.5,-0.020789129297538927)
	(8.6,-0.1290803332476027) (8.7,-0.20363973759199366) (8.8,-0.22721118149833722)
	(8.9,-0.19524716020128832) (9.,-0.11671992859620726) (9.1,-0.011654333891738701)
	(9.2,0.09395060335527704) (9.3,0.17453029729276967) (9.4,0.21111978209739599)
	(9.5,0.19579425065976108) (9.6,0.13335067227339098) (9.7,0.03988278248334751)
	(9.8,-0.061342235837487935) (9.9,-0.14565494520857875) (10.,-0.19297632594972003)
};
\node[color=red,right] at (10, -0.5) {$\Im(\mathcal{F}[\sigma](\omega))$};
	\end{tikzpicture}
	\caption{$\Re(\mathcal{F}[\sigma](\omega))$ and $\Im(\mathcal{F}[\sigma](\omega))$ from  \autoref{eq:fcTanh} with $c=5$. Since the Fourier-transform of functions with compact support are analytic, and zero sets of analytic functions are zero measure, we can see that the zero set of this Fourier-transform has zero measure.}
	\end{subfigure}
	\caption{\label{fig:alles} Examples of Tauber-Wiener-functions. All of them have compact support and therefore are analytic functions.}
\end{figure}

The approximation theories of neural networks are linked with Tauber-Wiener-approximation results by proving that activation functions of neural networks are Tauber-Wiener-functions.

More general, we define Tauber-Wiener-functions of higher order:
	\begin{definition} \label{de:twho} (Tauber-Wiener-function higher order)\index{function!Tauber-Wiener!higher order} A function $\sigma:\R \to \R$ is called Tauber-Wiener-function of order $k \in \N_0$ if
	\begin{equation*}
		\text{meas} \left( \bigcup_{l=0}^k \set{s:\mathcal{F}[\sigma^{l}](s)=0} \right) =0\;.
	\end{equation*}
    \end{definition}

\subsection{Tauber-Wiener-and activation functions} \label{ss:act}
Here, we show that \emph{activation functions} of neural networks
(as defined in \cite{HinDenYuDahMoh12,KriSutHin17,AlElr19,Cyb89,ElfUchDoy18}) are indeed Tauber-Wiener-functions (see \autoref{de:tw}) after compactification. This is a simple consequence of the \emph{Paley-Wiener-Theorem}, which (in its forward direction) states that the Fourier-transform of a compactly supported function (or distribution) admits an entire holomorphic extension (see \cite{Yos95}). The converse characterization—relating growth conditions to the size of the support—is given in \autoref{th:PaleyWienerYosida}.
	
We consider approximation of functions by translations of a Tauber-Wiener-function $\sigma$, which are $L^1$, $L^2$-functions, where the zeros are at most of measure $0$: We consider two classes of activation functions, which are \emph{ridge} and \emph{radial} activations.

{\bf Ridge activation functions:} \index{function!activation!ridge}
		\begin{itemize}
			\item{\bf The linear activation function} \index{function!activation!linear} is defined as
			\begin{equation}\label{eq:linear_act}
				s \in \R \mapsto \sigma(s):=as +b\;.
			\end{equation}
			The Fourier-transform only exists in a distributional sense (see \autoref{sec:distribution} and is given by
			\begin{equation*}
				\omega \in \R \mapsto  \mathcal{F}[\sigma](\omega) = \sqrt{2\pi} \left(b \delta(\omega) + {\i} a \delta'(\omega) \right)\,,
			\end{equation*}
		    where $\delta$ denotes the $\delta$-distribution and $\delta'$ denotes its derivative. (see \autoref{eq:delta_dis}).		
	    	We cannot apply Tauber-Wiener-theorems, because the function $\sigma$ is neither in $L^1(\R)$ nor in $L^2(\R)$.
	    	
	    	The {\bf compactified linear activation function.} \index{function!activation!linear compactified} With a compactification of a linear activation function we obtain a Tauber-Wiener-function:
	    	\begin{equation} \label{eq:ReLU_general}
	    		s \in \R \mapsto  \sigma(s) := (as +b) \chi_{[-c,c]}(s) \;.
	    	\end{equation}
    	    Note that the Fourier-transform of a function with compact support is analytic and the Fourier-transform can only have zeros of measure zero or vanish identically (see \cite{Lan67}).
    	
    	    The Fourier-transform is computed using the convolution theorem (see \autoref{th:Fconv}) and the computation rules for the $\delta$-distribution (see \autoref{th:distribution_rechenregeeln}) and reads as (see \autoref{fi:fcla})
    	    \begin{equation} \label{eq:fcla}
    	    	\begin{aligned}
    	    	\omega \in \R \mapsto  \mathcal{F}[\sigma](\omega)
    	    	=  \sqrt{2 \pi} \left( \frac{b}{c} \sinc \left(\frac{\omega}{2 \pi c} \right) - {\i} \frac{a}{c} \sinc' \left(\frac{\omega}{2 \pi c} \right) \right)\;.
    	    	\end{aligned}
    	    \end{equation}
    	    \item The \emph{Heaviside-activation function}\index{function!activation!Heaviside} is defined by
            \begin{equation} \label{eq:heaviside}
	             s \in \R \mapsto \sigma(s) := \chi(as+b) %\chi_{\set{at+b > 0:t \in \R}}(s)
	             = \chi_{\set{t > -\frac{b}{a}}}(s)
            \end{equation}
            where $\chi$ is the Heaviside-function and $\chi_{\set{t > -\frac{b}{a}}}$ is the characteristic function of the set $\set{t > -\frac{b}{a}}$. Its Fourier-transform is given by
            \begin{equation}\label{eq:Fourierheaviside}
        	 \begin{aligned}
        		\omega \in \R \mapsto  \mathcal{F}[\sigma](\omega) =
        		\frac{\exp^{\i \omega \frac{b}{a}}}{\sqrt{2 \pi}}  \left( \frac{1}{\i \omega} + \pi \delta(\omega) \right)\;.
             \end{aligned}
            \end{equation}
            Again the Heaviside-activation function needs to be considered a distribution. After compactification with sufficiently large $c > 0$ we get
            \begin{equation}\label{eq:chs}
            	s \in \R \mapsto \sigma(s) := \chi_{\left]-\frac{b}{a},c\right[}(s)\;.
            \end{equation}
            An therefore the Fourier-transform is given by
            \begin{equation}\label{eq:fchs}
        	\begin{aligned}
        		\omega \in \R \mapsto  \mathcal{F}[\sigma](\omega)
        		= \frac{1}{\sqrt{2\pi}} \left( \exp^{-{\i} \omega c} - \exp^{\i \omega \frac{b}{a}}\right)\;.
        	\end{aligned}
        \end{equation}
        The real and imaginary part of $\sigma$ have common zeros, such that $\sigma$ vanishes on some points of the real line. The Tauber-Wiener-theorem \autoref{Wiener_Tauberian} therefore allows to approximate arbitrary functions in $L^2$ with arbitrary accuracy but not in $L^1$.
        \item \emph{Sigmoid activations} play an important role in \emph{universal approximation theorems} (see \autoref{le:LinearApproximation} below):
	    A continuous function $\sigma:\R \to \R$ is called \emph{sigmoidal}\index{function!sigmoid} if
	    	\begin{equation}\label{de:sigmoid}
	    		\sigma(s) \to \left\{ \begin{array}{rcl} 1 & \text{ as } & s \to \infty\\
	    			0 & \text{ as } & s \to -\infty
	    		\end{array}\right. \;.
	    	\end{equation}		
	    	We call $\sigma$ \emph{antisymmetric sigmoidal} \index{function!activation!sigmoid antisymmetric} if
	    	\begin{equation} \label{eq:sigmoid_anti}
	    		\sigma(s)-\frac{1}{2} = - \left(\sigma(-s)-\frac{1}{2}\right).
	    	\end{equation}
	    	The characteristic function of the interval $[0,\infty[$ is antisymmetric. Examples of sigmoidal functions are
	    	\begin{enumerate}
	    		\item The \Sig function is defined as:
	    		\begin{equation}\label{de:sigmoid_f}
	    			\sigma(s) =\frac{1}{1+\exp^{-s}}.
	    		\end{equation}
	    		The Fourier-transform of \Sig is given by
            \begin{equation}\label{eq:Fouriersigmoid}
        	 \begin{aligned}
        		\omega \in \R \mapsto  \mathcal{F}[\sigma](\omega) =
        		\sqrt{\frac{\pi}{2}}\left(\delta(\omega)  -  \i \operatorname{csch}(\pi\omega)\right).
             \end{aligned}
            \end{equation}

	    		\item The \HS function is defined as
	    		\begin{equation}\label{de:hardsigmoid}
	    			\sigma(s) = \left\{ \begin{array}{rcl} 1 & \text{ if }\,\, s \geq c\\
	    				0 & \text{ if }\,\, s \leq -c\\
	    				\frac{s}{2c} +\frac{1}{2} & \text{ otherwise }
	    			\end{array}\right.
	    		\end{equation}
	    		where $c>0$ is a threshold. The Fourier-transform is given by
	    		\begin{equation*}
	    			\begin{aligned}
	    				\omega  \in \R \mapsto  \mathcal{F}[\sigma](\omega)
	    				= &\frac{\exp^{\i \omega c}}{2 \pi}  \left( \frac{1}{\i \omega} + \pi \delta(\omega) \right) +\\
	    				  & \qquad + \frac{1}{2 \i \omega} \left(1+\frac{1}{\i c \omega}\right)\left( \e^{\i\omega c} -\e^{-\i\omega c}\right)\;.
	    			\end{aligned}
	    		\end{equation*}
	    	\end{enumerate}   	
	\item The \ReLU (rectified linear unit) function is defined as \index{function!activation!rectified linear unit \ReLU}
	    \begin{equation} \label{eq:ReLU}
	    		\begin{aligned}
	    			s \in \R \mapsto  \ReLU(s) &:= \max\set{(0,as +b)}\\
	    			&= (as +b) \chi_{[-\frac{b}{a},\infty[}(s) \;.
	    		\end{aligned}
	    \end{equation}
    	The Fourier-transform is given by
    \begin{equation*}
    \begin{aligned}
    \omega \in \R \mapsto  \mathcal{F}[\ReLU](\omega)
    =  \frac{a \exp^{\i \omega \frac{b}{a}}}{\sqrt{2\pi}}\left(-\frac{1}{\pi\omega^2}+\i\delta'(\omega)\right)\;.
    \end{aligned}
    \end{equation*}
    The Fourier-transform of the compactified \ReLU \index{function!activation!compactified \ReLU}
    \begin{equation} \label{eq:cReLU}
    	s \in \R \mapsto \sigma(s) := \chi_{]-c,c[}(s) \ReLU(s)
    \end{equation}
    is given by
    	\begin{equation} \label{eq:fcReLU}
    		\begin{aligned}
    			&\omega \in \R \mapsto  \mathcal{F}[\sigma](\omega) = \\
    	     			&
    	  \frac{\exp^{-\i \omega \left(c + \frac{b}{a}\right)}}{\sqrt{2\pi}\omega^2} \left( a+\i \omega(2b+ac) - \exp^{2\i c \omega}(a+\i \omega(2b-ac))\right)\;.
        \end{aligned}
    	 \end{equation}
    	 \end{itemize}
  \bigskip\par\noindent
        {\bf Radial activation functions} \index{function!activation!radial}
        \begin{itemize}
        	\item The {\bf Gaussian activation}\index{function!activation!Gaussian} is defined by
        	\begin{equation} \label{eq:gaussian}
        		s \in \R \mapsto \sigma(s):=\exp^{-\frac{(s-b)^2}{2a^2}}.
        	\end{equation}
        	where $\alpha >0$.
            This function is in $L^2(\R)$ and the Fourier-transform is given by
            	\begin{equation}\label{eq:fgaussian}
            		\omega \in \R \mapsto  \mathcal{F}[\sigma](\omega):=\abs{a} \exp^{-\frac{\omega}{2} (a^2 \omega + 2 \i b)}.
            	\end{equation}
            The Fourier-transform has no zeros on the real line and thus the Tauber-Wiener-theorem allows to approximate arbitrary functions in $L^1$ and $L^2$.
            \item{\bf $\tanh$ activation:} \index{function!activation!$\tanh$} The Fourier-transform of $\tanh$ is given by
            \begin{equation*}
            	\omega \in \R \mapsto - \i \sqrt{\frac{\pi}{2}} \text{csch} \left( \frac{\pi \omega}{2}\right)\;.
            \end{equation*}
            For all $p \geq 1$ this function is not in $L^p(\R)$. The compactified tanh function
            	\begin{equation} \label{eq:cTanh}
            		s \in \R \mapsto \sigma(s):=\chi_{]-c,c[}(s) \tanh(s)
            	\end{equation}
            	has the Fourier-transform
            	\begin{equation} \label{eq:fcTanh}
            		\begin{aligned}
            			\mathcal{F}[\sigma](\omega) &=
            			\frac{1}{2} \exp^{-\frac{\pi \omega}{2}} \left( B\left(-\exp^{2c}; 1 - \frac{\i \omega}{2}, 0\right) + B\left(-\exp^{2c}; -\frac{\i \omega}{2}, 0\right)\right. \\
            			&- \exp^{\pi \omega} \left( B\left(-\exp^{2c}; 1 + \frac{\i \omega}{2}, 0\right) + B\left(-\exp^{2c}; \frac{i \omega}{2}, 0\right) \right) \\
            			&\left.- 2\i \pi \left(1 + \coth\left(\frac{\pi \omega}{2}\right)\right) \right)\,,
            		\end{aligned}
            	\end{equation} with the $\beta$-function
            		\begin{equation} \label{eq:beta}
            			B(z; a, b) := \int_{0}^{z}u^{a-1}(1-u)^{b-1}du.
            		\end{equation}
            		and $_2F_1(a, b; c; z)$ is the hypergeometric function. With the Tauber-Wiener-theorem $\sigma$ can be used to approximate functions in $L^1$ and $L^2$, respectively.
        \end{itemize}

In the following we consider generalizations of Tauber-Wiener, wavelets, sigmoidal and activation functions in higher dimensions.

\subsection{Extensions to higher dimensions}
The Tauber-Wiener-theorem gives a complete answer on approximating univariate function $\x:\R \to \R$ by linear combinations of shifts of a single functions $\sigma:\R \to \R$. The generalization to higher dimensions can be performed in various ways, leading to the following research questions:
\begin{itemize}
\item Can every function $\x : \R^m \to \R$ be approximated by linear combinations of \emph{affine linear transformations}  of a function $\sigma:\R \to \R$?\index{Translations!affine linear} This research question is the subject of a class of theorems that are referred to as the universal approximation theorems (see \cite{Cyb89,Hor91,Bis95} \autoref{le:LinearApproximation} below). That is, for a given continuous function
$\x: \R^m \to \R$ we can find a function
\begin{equation} \label{eq:translation_affine}
\x_\noc(\vs) = \sum_{k=1}^\noc \alpha_k \sigma (\vw_k^T \vs + \theta_k)\,,
\end{equation}
with appropriate parameters $\vw_k \in \R^m$, $\theta_k \in \R$ and $\alpha_k \in \R$, $k=1,\ldots,\noc$, that approximates an arbitrary function $\x:\R^m \to \R$ on a given \emph{compact domain}. Linear combinations of compositions of affine linear combinations $\vs \mapsto \vw^T \vs + \theta$ with  \emph{activation functions} $\sigma$\index{function!activation} are called \emph{affine linear neural network functions}.\index{neural network!affine linear}
The \emph{uncountable} number of ansatz functions in neural network functions might be stated as the
\emph{neural network mystery:}\index{mystery!neural networks} The family of neural networks used in machine learning publications to approximate an arbitrary function is \emph{uncountable} (see \cite{Cyb89,Hor91}). However, approximation rates are proven based on best-worst-case scenarios with a \emph{countable} number of neural networks (see, for instance, \cite{ShaCloCoi18} and \autoref{co:ConvR} below). In other words the stronger results are proven with less ansatz functions. This explains why \emph{sparsity} is such an important tool in neural networks theory, which can be achieved with compressed sensing techniques as we explained already in \autoref{eq:Tiksynl1}.

\item A different generalization of the Tauber-Wiener-theorems \autoref{Wiener_Tauberian} to higher dimensions is: Can every function $\x : \R^m \to \R$ be approximated by linear combinations of \emph{standardized affine linear translations} of a function $\sigma$? \index{Translations!standardized affine linear} Standardized refers to affine linear translations, where for all $k=1,\ldots,\noc$ the Euclidean norm is one, i.e., $\norm{\vw_k}_2=1$. Of course all other norms $\norm{\cdot}_p$,
$p=[1,\infty]$ could be used for standardization (see \autoref{ta:nn}).
\item Let $\sigma :\R^m \to \R$. Then we can ask the question, whether an arbitrary function $\x:\R^m \to \R$ can be approximated by
\begin{equation} \label{eq:translation_affine_n}
\x_\noc(\vs) = \sum_{k=1}^\noc \alpha_k \sigma (O_k^T \vs + \theta_k)?
\end{equation}
Here $O_k \in \R^{m\times m}$ is an orthogonal matrix, $\theta_k \in \R^m$, $\alpha_k \in \R$, $k=1,\ldots,\noc$ and $\sigma$ is the indicator of a rectangle. It has been shown in \cite{Rud73,Cyb89} that an arbitrary function $\x \in L^1(]-1,1[^m)$ can be approximated with \autoref{eq:translation_affine_n}. Note that here $\sigma:\R^m \to \R$ is multi-variate function. A typical multi-variate function is the \SM \index{function!softmax} where
\begin{equation}\label{de:softmax}
	(\sigma(\vs))_i = \frac{\exp^{s_i}}{\sum_{j=1}^{m}\exp^{s_j}}\,,\quad i=1,\ldots,m\;.
\end{equation}
\item $\x \in L^1(\R^m)$ can also be approximated by the function
\begin{equation} \label{eq:translation_Rudin}
	\x_\noc(\vs) = \sum_{k=1}^\noc \alpha_k \sigma (w_k \vs + \vec{\theta}_k)\,,
\end{equation}
where $w_k \in \R$, $\vec{\theta}_k \in \R^m$ and $\alpha_k \in \R$, $k=1,\ldots,\noc$. See \cite[Chap.9.7]{Rud73}.
\end{itemize}
\begin{remark}
	We consider the approximating function $\x_\noc:\R \to \R$ from \autoref{eq:translation_affine} and apply the Fourier-transformation, which gives
	\begin{equation*}
		\begin{aligned}
		\mathcal{F}[\x_\noc](\omega) &= \sum_{k=1}^\noc \alpha_k \frac{1}{\sqrt{2\pi}} \int_{s=-\infty}^\infty \sigma(w_ks+\theta_k) \exp^{-\i \omega s}\,ds\\
		&\underbrace{=}_{t=w_k s + \theta_k}
		\sum_{k=1}^\noc \frac{\alpha_k}{w_k} \exp^{\i \frac{\omega}{w_k} \theta_k} \frac{1}{\sqrt{2\pi}} \int_{t=-\text{sgn}(w_k)\infty}^{\text{sgn}(w_k)\infty} \sigma(t) \exp^{-\i \frac{\omega}{w_k} t}\,dt\\
		&=
		\sum_{k=1}^\noc \frac{\alpha_k \exp^{\i \frac{\omega}{w_k} \theta_k}}{\abs{w_k}}
		\mathcal{F}[\sigma] \left(\frac{\omega}{w_k}\right)\;.
		\end{aligned}
	\end{equation*}
    For a wavelet $\sigma$ and a dyadic decomposition of the scale-shift space $(w,\theta)$ the above formula gives the wavelet approximation (see \cite{Dau92}). In contrast to the Tauber-Wiener-expansion (see \autoref{Wiener_Tauberian}) we cannot simply divide the left hand side by $\mathcal{F}[\sigma]$ to derive an explicit expansion for $\mathcal{F}[\x_N]/\mathcal{F}[\sigma]$ (compare with \autoref{eq:explicit}).
\end{remark}

\section{Neural network functions} \label{se:nnf}
We introduce general neural network functions below: We start the discussion with classical, affine linear, neural network functions, which are then successively generalized. Essentially there are two paths to generalize classical shallow neural networks:
\begin{enumerate}
	\item Deep neural networks (see \autoref{eq:DNN} below) with more than one internal layer.
	\item Alternative, more complex functions than affine linear functions, can be considered inside the activation function. Here the transformation part in \autoref{ta:nn2} is more complex.
\end{enumerate}
While the first approach has attracted a lot of attention in the literature, the second approach has been only considered sparsely, although it yields to the same approximation rates (see \cite{Mha93}, where this possibility was discussed first, and \cite{FriSchShi24}).

We define now generalized neural networks. For that purpose we introduce the following notation:
\begin{notation}
	Let $\degree \geq m \in \N$.
	Row vectors %\footnote{In numerical analysis row vectors consist of lines. See for instance \cite{Han02}.}
	in $\R^\degree$ and $\R^m$ are denoted by
	\begin{equation*}
		\vw = (w_1,w_2,\ldots,w_\degree)^T \text{ and }
		\vs = (s_1,s_2,\ldots,s_m)^T, \text{ respectively.}
	\end{equation*}
    We indicate with the different dimension $m$ and $\hat{m}$ that the \emph{weight vector} can have a different notation than the space dimension, which is the dimension of $\Omega_\X \subseteq \R^m$. We apply the obtained approximations also to elements $\Omega_\Y \subseteq \R^n$, with obvious notational adaptations.
\end{notation}
	
\begin{definition}[Affine linear neural network functions] \label{de:affinenns} Let $\noc \in \N$ be fixed.
\begin{itemize}
		\item A \emph{shallow} affine linear neural networks ({\bf ALNN}) is a function (here $\degree=m$)\index{neural network!affine linear!ALNN}
		\begin{equation}\label{eq:classical_approximation}
			\vs \in \R^m \to {\tt p}(\vs) := \Psi[\vp](\vs) := \sum_{k=1}^{\noc} \alpha_k\sigma\left(\vw_k^T \vs +\theta_k \right) \in \R,
		\end{equation}
		with $\alpha_k, \theta_k \in \R$, $\vw_k \in \R^m$; $\sigma$ is an activation function, such as the {\bf Sigmoid}- (see \autoref{de:sigmoid})\commentO{Reference with example section the $\ReLU{}^k$} or \SM-function (see \autoref{de:softmax}) respectively).
		Moreover,
		\begin{equation} \label{eq:p}
			\begin{aligned}
			\vp =
			(\alpha_k,\vw_k,\theta_k)_{k=1}^\noc \in \R^\dimlimit
			\text{ with } \dimlimit = (m+2)\noc
			\end{aligned}
		\end{equation}
		denotes the according parametrization of ${\tt p}$. Note ${\tt p}$ is a function, which is parametrized by a vector $\vp$.
		In this context the set of {\bf ALNN}s is given by
		\begin{equation*}
			\bP := \set{{\tt p} = \Psi[\vp] \text{ of the form \autoref{eq:classical_approximation}}: \vp \in \R^\dimlimit}.
		\end{equation*}
		\item \emph{Deep} affine linear neural network functions ({\bf DNN}s) are defined as follows: \index{neural network!affine linear!deep!DNN} An $(L+2)$-layer network looks as follows:
		\begin{equation}\label{eq:DNN}
				\begin{aligned}
					\vs \in \R^m \to {\tt p}(\vs) := \Psi[\vp](\vs) := \vec{a}^T \vec{\sigma}_L\left( {\tt p}_{L} \left( \ldots \left( \vec{\sigma}_1 \left({\tt p}_{1}(\vs) \right) \right)\right) \right),
				\end{aligned}
			\end{equation}
			where we define iteratively
			\begin{equation*}
				\begin{aligned}
					{\tt p}_l: \R^{\noc_{l-1}} &\to \R^{\noc_l}\\
					\vs &\mapsto (\vw_{1,l},\vw_{2,l},\ldots,\vw_{\noc_l,l})^T \vs +(\theta_{1,l},\theta_{2,l},\ldots,\theta_{\noc_l,l})^T
				\end{aligned}
			\end{equation*}
			with $\vec{a} \in \R^{\noc_L}, \theta_{j,l} \in \R$ and $\vw_{j,l} \in \R^{\noc_{l-1}}$ for all $l=1,\ldots,L$ and all $j=1,\ldots,\noc_l$. Here, for $l=0,1,\ldots,L$
			\begin{itemize}
				\item $\noc_l$ denotes the number of possible coefficients in $l$-th internal layer and $\noc_0=m$;
				\item $\vec{\sigma}_l:\R^{\noc_{l}} \to \R^{\noc_l}$ denotes an activation function (see \autoref{ss:act}) to layer $l$, which operates componentwise, that is for an activation function $\sigma:\R \to \R$
				\begin{equation*}
					\vec{\sigma}_l(\vs) = (\sigma (s_i))_{i=1}^{\noc_l} \text{ for all } \vs \in \R^m\;.
				\end{equation*}
			\end{itemize}
		    Moreover, we denote the parametrizing vector by
			\begin{equation} \label{eq:pd}
				\vp = [\vec{a}; \vw_{1,1},\ldots,\vw_{\noc_L,L}; \theta_{1,1},\ldots,\theta_{\noc_L,L}] \in \R^\dimlimit.
		\end{equation}
	\end{itemize}
\end{definition}
\begin{remark}
	A shallow linear neural network\index{neural network!shallow!linear} can also be called a \emph{3-layer network}.\index{neural network!3-layer} The notation 3-layer network is consistent with the literature because input-and output are counted on top as layers. Therefore,
	a general $(L+2)$-layer network has only $L$ \emph{internal} layers.\index{layer!internal}
\end{remark}

This affine linear neural network functions can be generalized by using higher order polynomials:
\begin{definition}[Shallow generalized neural network function] \index{neural network!shallow generalized} \label{de:gnn}
	Let $m,\degree,\noc \in \N$ and let
	\begin{equation}\label{eq:hj}
	    \vfm_j: \R^m \to \R^\degree, \quad j=1,\ldots,\noc
    \end{equation}
	be vector-valued functions, called \emph{network generating functions}. \index{function!network generating} Neural network functions associated to a family of network generating functions are defined by
	\begin{equation}\label{eq:general}
		\vs \to {\tt p}(\vs) = \Psi[\vp](\vs) := \sum_{j=1}^{\noc} \alpha_j\sigma\left(\vw_j^T\vfm_j(\vs) +\theta_j\right)
	\end{equation}
    with
    \begin{equation*}
    	\vp = (\alpha_j, \theta_j, \vw_j)_{j=1}^\noc \in \R^{(\degree+2)\noc}\;.
    \end{equation*}
    Again we denote the parametrization vector with $\vp$ and the set of all such functions with $\bP$.
	We call
	\begin{equation}\label{eq:GeneralFunctions}
		\mathcal{D}: = \set{ \vs \to \vw^T \vfm_j(\vs) +\theta: \vw \in \R^{\degree}, \theta \in \R, j=1,\ldots,\noc}
	\end{equation}
	the set of \emph{decision functions}\index{function!decision} associated to \autoref{eq:general}.
	The composition of $\sigma$ with a decision function is called \emph{neuron}.\index{neuron}
\end{definition}
\begin{remark}
	Classical decision functions, that means {\bf ALNN's}, are widely documented (see for instance \cite{HinDenYuDahMoh12,KriSutHin17,AlElr19,Cyb89,ElfUchDoy18}), while generalized neural network functions have only been discussed sparsely (see for instance \cite{TsaTefNikPit19,FanXioWan20,FriSchShi24,FriSchShi25}.
	\end{remark}

The versatility  of the network defined in \autoref{eq:general} is due to the possibility to flexibly choose the functions $\vfm_j$, $j=1,\ldots,\noc$.

We give a few examples of generalized decision function below and split them into three categories:

\begin{definition}[Neural networks with generalized decision functions] \label{de:quadratic}
	\begin{description}
		\item{General quadratic neural networks {\bf (GQNN)}}: Let \index{neural network!general quadratic} \index{GQNN}
		$$\degree = m+1, \quad \vfm_j(\vs) = \begin{pmatrix} {\tt h}_{j,1}(\vs)\\ \vdots\\ {\tt h}_{j,m}(\vs)\\ {\tt h}_{j,m+1}(\vs) \end{pmatrix}= \begin{pmatrix} s_1\\ \vdots\\ s_m\\ \vs^T A_j \vs \end{pmatrix}
		\text{ for all } j=1,\ldots,\noc.$$
		That is, $\vfm_j$ is a graph of a quadratic function. Inserting this choice into \autoref{eq:general} we get:
		\begin{equation}\label{eq:quadratic_approximation_fixed}
			\begin{aligned}
				\vs \to {\tt p}(\vs) = \Psi[\vp](\vs) &:=
				\sum_{j=1}^{\noc} \alpha_j \sigma \left( \vw_j^T \vs + w_{j,m+1} \vs^T A_j \vs +\theta_j\right) \\
				&\text{ with }  \alpha_j, \theta_j \in \R, \vw_j \in \R^{m+1}, A_j \in \R^{m \times m} .
			\end{aligned}
		\end{equation}
		$\vp$ denotes the parametrization of ${\tt p}$:
		\begin{equation} \label{eq:pq}
			\begin{aligned}
				\vp = (\alpha_j,\vw_j,A_j,\theta_j)_{j=1}^\noc \in \R^\dimlimit %\\
				\text{ with }
				\dimlimit = (m^2+m+3)\noc.
			\end{aligned}
		\end{equation}
		\item{Constrained quadratic neural networks {\bf (CQNN)}} are networks, where the entries of the matrices $A_j$, $j=1,\ldots,\noc$ are constrained:\index{neural network!constrained quadratic {\bf CQNN}}
			For the choice $A_j=\id$, the identity matrix, we get:
			\begin{equation}\label{eq:radial_approximation}
				\begin{aligned}
					\vs \to {\tt p}(\vs) = \Psi[\vp](\vs) &:=
					\sum_{j=1}^{\noc} \alpha_j \sigma \left(\vw_j^T \vs + w_{j,m+1} \norm{\vs}_2^2 +\theta_j\right) \\
					& \text{ with } \alpha_j, \theta_j \in \R, \vw_j \in \R^{m+1}\;.
				\end{aligned}
			\end{equation}
		%\end{enumerate}
		\item{Radial quadratic neural networks \RQNN{}s:} \index{neural network!radial quadratic!\RQNN}
		For $w_{j,m+1} \neq 0$ the argument in \autoref{eq:radial_approximation} can be rewritten as
		\begin{equation}\label{eq:nus}
			\begin{aligned}
				w_{j,m+1} \norm{\vs}_2^2 + \vw_j^T\vs  + \theta_j
				= \kappa_j + w_{j,m+1} \norm{\vs - \vr_j}_2^2
			\end{aligned}
		\end{equation}
		with
		\begin{equation} \label{eq:kappa}
			\vr_j = - \frac{1}{2w_{j,m+1}} \vw_j \text{ and } \kappa_j = \theta_j - \frac{\norm{\vw_j}_2^2}{4 w_{j,m+1}}.
		\end{equation}
		We call the set of functions from \autoref{eq:radial_approximation}, which satisfy the constraint
		\begin{equation} \label{eq:circle}
			w_{j,m+1} < 0 \text{ and } \kappa_j > 0 \text{ for all } j=1,\ldots,\noc,
		\end{equation}
		\RQNN{}s, \emph{radial neural network functions}\index{radial neural network}. That is, an \RQNN reads as follows
		\begin{equation}\label{eq:rqnn}
		\begin{aligned}
			\vs \to {\tt p}(\vs) = \Psi[\vp](\vs) &:=
			\sum_{j=1}^{\noc} \alpha_j \sigma \left( \kappa_j - \nu_j \norm{\vs - \vr_j}_2^2 \right) \\
			& \text{ with } \nu_j, \kappa_j > 0.
		\end{aligned}
	    \end{equation}
	    We see that the level sets of $\Psi[\vp]$ are circles and so \RQNN{}s are \emph{radial basis functions}
		(see for instance \cite{Buh03} for a general overview).
		\item{Sign based quadratic {\bf (SBQNN)} and cubic neural networks {\bf (CUNN)}}: Let $\degree=m$, which means that we have as many network generating functions ${\tt h}_j$, $j=1,\ldots,\degree$ as the space dimension $m$.\index{neural network!sign based quadratic} \index{neural network!cubic} \index{SBQNN} \index{RQNN}
		\begin{enumerate}
			\item {\bf (SBQNN)}: Let, independent of $j=1,\ldots,\noc$,
			\begin{equation} \label{eq:sgn_distance}
				{\tt h}_{i} (\vs)= \text{sgn}(s_i) s_i^2 \text{ for } i=1,\ldots,m\;.
			\end{equation}
			With this family of network generating functions the family of neural network functions reads as follows:
			\begin{equation}\label{eq:generalf}
				\begin{aligned}
					\Psi[\vp](\vs) := \sum_{j=1}^{\noc} \alpha_j \sigma
					\left(\sum_{i=1}^m w_{j,i}\text{sgn} (s_i) s_i^2+\theta_j\right) \\
					\text{ with } \alpha_j, \theta_j \in \R \text{ and } \vw_j \in \R^m.
				\end{aligned}
			\end{equation}
			We call these functions \emph{signed based squared neural networks}\index{neural network!signed squared}. Note, here $\vp \in \R^\dimlimit$ with $\dimlimit=(m+2)\noc$.
			\item {\bf (CUNN)}: Let
			\begin{equation} \label{eq:cube}
				{\tt h}_{i} (\vs)= s_i^3 \text{ for } i=1,\ldots,m.
			\end{equation}
			We obtain the family of functions
			\begin{equation}\label{eq:cubef}
				\begin{aligned}
					\Psi[\vp](\vs) := \sum_{j=1}^{\noc} \alpha_j \sigma
					\left(\sum_{i=1}^m w_{j,i} s_i^3+\theta_j\right) \\
					\text{ with } \alpha_j, \theta_j \in \R \text{ and } \vw_j \in \R^m.
				\end{aligned}
			\end{equation}
			We call these functions \emph{cubic neural networks}. Again $\vp \in \R^\dimlimit$ with $\dimlimit=(m+2)\noc$.
		\end{enumerate}	
	\end{description}
\end{definition}

\subsection{Universal approximation theorems} \label{sec:uat}
We start our convergence review with the classical \emph{universal approximation theorem} from \cite{Cyb89}. For this, we need the definition of discriminatory functions, which is considered broader than in the standard literature:\index{function!discriminatory}
\begin{definition}[Discriminatory function] \label{de:disc_function}
	Let $\Omega \subseteq \R^m$ be Lebesgue measurable:
	\begin{itemize}
		\item A continuous function $\sigma : \R \to \R$ is called \emph{discriminatory with respect to measures} if for every measure $\mu$ on $\Omega$, which satisfies
		\begin{equation} \label{eq:discrimination}
			\int_\Omega \sigma (\vw^T \vs + \theta)\,d\mu(\vs)=0 \quad
			\text{ for all } \vw \in \R^m, \theta \in \R,
		\end{equation}
		it follows that $\mu$ is equal to zero. The original definition of \cite{Cyb89} is for $\Omega = [0,1]^m$.
	\item Let $p=1,2$ with associated dual $p_*=\infty,2$, respectively. A function $\sigma \in L^p(\R)$ is called \emph{discriminatory with respect to $L^p(\R)$}
	 if for every function $g \in L^{p_*}(\Omega)$, which satisfies
	 \begin{equation} \label{eq:discriminationp}
	 	\int_\Omega \sigma (\vw^T \vs + \theta)\,g(\vs) d \vs=0 \quad
	 	\text{ for all } \vw \in \R^m, \theta \in \R
	 \end{equation}
	 it follows that $g$ is equal to zero. A detailed treatment of these duality results appears in Brezis's classic textbook on functional analysis \cite{Bre11b}.
	 \end{itemize}	
\end{definition}

\begin{example}\label{ex:sigmoid}
	Let $\Omega$ be connected and $\text{meas}(\Omega) > 0$.  Moreover, let $\sigma$ be a linear activation function (see \autoref{eq:linear_act}), then from \autoref{eq:discrimination} it follows that for $g \in L^{p_*}(\Omega)$ which satisfies
	\begin{equation} \label{eq:discriminationb}
		\int_{\Omega} (a(\vw^T \vs + \theta) + b)g(\vs)\,d \vs=0 \quad \text{ for all } \vw \in \R^m, \theta \in \R\,
	\end{equation}
	which is equivalent to
	\begin{equation} \label{eq:discriminationc}
		\vw^T \int_{\Omega} \vs g(\vs) \,d\vs + \theta \int_{\Omega} g(\vs)\,d \vs=0 \quad \text{ for all } \vw \in \R^m, \theta \in \R \;.
	\end{equation}
	However, \autoref{eq:discriminationc} is satisfied if the zero-th and first moments of $g$ vanish. Clearly, there exists such non-trivial functions. 
	The same argument applies to all polynomials $\sigma$. However, compactified functions have the discriminatory property. In particular all Tauber-Wiener-functions are discriminatory.
	
	Note that every non-polynomial function is \emph{discriminatory} (this follows from the results in \cite{LesLinPinScho93}). Therefore the activation functions in \autoref{de:affinenns}
	are discriminatory.
\end{example}
In the following, we show the relation to the Tauber-Wiener-theory:
\begin{example}\label{ex:lpdiscriminatory}
	Let $\Omega = \R$. Then $\sigma \in L^2(\R)$ is discriminatory in $L^2(\R)$ if from  \autoref{eq:discrimination}, which reads particularly for $\Omega=\R$,
	\begin{equation} \label{eq:discriminationae}
		\int_\R \sigma (w s + \theta) g(s)\,d s=0 \quad \text{ for all } w \in \R, \theta \in \R \text{ and } g \in L^2(\R)
	\end{equation}
	it follows that $g \equiv 0$.
	
	In particular $\sigma$ is discriminatory if from the weaker condition
		\begin{equation} \label{eq:discriminationaep}
		\int_\R \sigma (s + \theta) g(s)\,d s=0 \quad \quad \text{ for all } w \in \R, \theta \in \R
	\end{equation}
	it follows that $g \equiv 0$.
    Using the Fourier-theorem it follows that
    \begin{equation} \label{eq:discriminationaepf}
    	\int_\R \mathcal{F}[\sigma](\omega) \e^{\i \omega \theta} \mathcal{F}[g](\omega)\,d \omega=0 \quad \quad \text{ for all } w \in \R, \theta \in \R.
    \end{equation}
Since the family $\e^{\i \cdot \theta}$ is a basis, we conclude that $\mathcal{F}[\sigma](\omega) \mathcal{F}[g](\omega) = 0$ and if $\sigma$ is Tauber-Wiener, then $\mathcal{F}[g]\equiv 0$.
\end{example}

\begin{theorem}[Universal approximation theorem (UAT) \cite{Cyb89}]\label{le:LinearApproximation}\index{Theorem!Universal approximation}\index{UAT}
	Let $\sigma:\R \to \R$ be a continuous discriminatory function. Then, for
	every function $\x \in C([0,1]^m)$ and every $\ve >0$, there exists a function
	\begin{equation}\label{eq:Linear}
		\x_\noc(\vs) = \sum_{j=1}^{\noc}\alpha_j \sigma(\vw_j^T\vs + \theta_j)
	\end{equation}
    with $\noc \in \N$, $\alpha_j$, $\theta_j \in \R$, $\vw_j \in \R^m$.
	satisfying
	\begin{equation*}
		\abs{\x_\noc(\vs)-\x(\vs)} < \ve  \text{ for all } \vs\in [0,1]^m.
	\end{equation*}
\end{theorem}

It is obvious that general quadratic neural networks ({\bf GQNN}s) and matrix constrained neural networks ({\bf CQNN}) are more versatile than affine linear works, and thus they satisfy the universal approximation property. For radial quadratic neural networks (\RQNN{}s) we prove a convergence rates result in \autoref{sec:rqnn} and this serves as an alternative to the universal approximation results. For sign based quadratic ({\bf SBQNN}s) and cubic ({\bf CUNN}s) networks we need a different analysis:

\begin{theorem}[Generalized universal approximation theorem (GUAT)] \label{th:general}
	Let $\sigma:\R \to \R$ be a continuous discriminatory function
	and assume that $\vfm : [0,1]^m \to \R^{\hat{m}}$ is injective (this in particular means that $m \leq \hat{m}$) and continuous.
	
	Then for every $g\in C([0,1]^m)$ and every $\ve>0$ there exists some function
	\begin{equation*}
		G_\ve^\vfm(\vs) := \sum_{j=1}^\noc \alpha_j\sigma\left(\vwm_j^T\vfm(\vs) +\theta_j\right) \text{ with } \alpha_j, \theta_j \in \R \text{ and } \vwm_j \in \R^{\hat{m}}
	\end{equation*}
	satisfying
	\begin{equation*}
		\abss{G_\ve^\vfm(\vs) - g(\vs)} < \ve \text{ for all }\vs\in [0,1]^m.
	\end{equation*}
\end{theorem}
\begin{proof}
	We begin the proof by noting that since $\vs \to \vfm(\vs)$ is injective, the inverse function on the range of $\vfm$ is well-defined, and we write
	$$\vfm^{-1}:\vfm([0,1]^m) \subseteq \R^{\hat{m}} \to [0,1]^m \subseteq \R^m\;.$$
	Since $\vfm$ is continuous and $[0,1]^m$ is compact also $\vfm^{-1}$ is continuous and $\vfm([0,1]^m)$ is compact.
	
	Applying the Tietze–Urysohn–Brouwer extension theorem (see \cite{Kel55}) to the continuous function $g \circ \vfm^{-1} : \vfm([0,1]^m) \to \R$, this can be extended continuously to $\R^{\hat m}$. This extension will be denoted by $g^*:\R^{\hat m} \to \R$.
	
	We apply \autoref{le:LinearApproximation} to conclude that there exists $\alpha_j, \theta_j \in \R$ and $ \vwm_j \in \R^{\hat{m}}$, $j=1,\ldots,\noc$ such that
	\begin{equation*}
		G^*(\vz) := \sum_{j=1}^\noc \alpha_j \sigma(\vwm_j^T \vz +\theta_j) \text{ for all } \vz \in \R^{\hat{m}}, \theta_j \in \R,
	\end{equation*}
	which satisfies
	\begin{equation}\label{eq:1}
		\abs{G^*(\vz)-g^*(\vz)} <\ve \text{ for all } \vz \in \R^{\hat m}.
	\end{equation}
	Then, because $\vfm$ maps into $\R^{\hat m}$ we conclude, in particular, that
	\begin{equation*}
		\begin{aligned}
			%G_\ve^\vfm(\vs):=
			G^*(\vfm(\vs)) &=\sum_{j=1}^\noc \alpha_j \sigma(\vwm_j^T \vfm(\vs)+\theta_j) \text{ and } \\
			\abs{G^*(\vfm(\vs)) - g(\vs)}%=\abs{G^*(\vfm(\vs)) - g(\vfm^{-1}(\vfm(\vs)))}
			&=\abs{G^*(\vfm(\vs)) - g^*(\vfm(\vs))} <\ve.
		\end{aligned}
	\end{equation*}
	Therefore $G_\ve^\vfm(\cdot):= G^*(\vfm(\cdot))$ satisfies the claimed assertions.
\end{proof}
We notice that the variability in $\vw_j$, $j=1,\ldots,\noc$ is the key to bring our proof and the universal approximation theorem in context. That is, if $\vw_j$, $\theta_j$, $j=1,\ldots,\noc$ are allowed to vary over the full spaces $\R^m$, $\R$, respectively. \RQNN{}s are constrained to $\theta_j < \norm{\vw_j}_2^2/(4 w_{j,m+1})$ in \autoref{eq:circle} and thus
	 \autoref{th:general} does not apply. Interestingly \autoref{thm:ConvR} below applies and allows to approximate arbitrary functions in $\mathcal{L}^1(\R^m)$ (even with rates). So the universal approximation theorem is subject to generalizations.
	
\begin{corollary}[UATs for {\bf SBQNN}s and {\bf CUNN}s] \label{th:qdf}
	Let $\sigma: \R \to \R$ be discriminatory and Lipschitz continuous. All families of neural network functions from \autoref{de:quadratic} satisfy the universal approximation property on $[0,1]^m$.
\end{corollary}
\begin{proof} The proof follows from \autoref{th:general} and noting that the generating network functions $\vfm_j$, $j=1,\ldots,\noc$ defined in \autoref{de:quadratic} are injective.
\end{proof}

We summarize universal approximation theorems for neural network functions extending the table from \cite{Cyb89}:

\begin{table}
	\begin{center}
		\begin{tabular}{r|c|c|l}
			\hline
			Neural network type & Function space & Reference & Our notation \\
			\hline
			\hline
			$\sigma(\vw^T\vs + \theta)$ & $C([0,1]^m)$ & \cite{Cyb89} & {\bf ALNN}\\
			$\vw \in \R^m$, $\theta \in \R$ & & &\\
			$\sigma$ continuous sigmoid & $L^2(]0,1[^m)$ & \cite{CarDic89} & \\
			\hline
			$\sigma(\vw^T\vs + \theta)$ & $C([0,1]^m)$ & \cite{Fun89,HorStiWhi89} &{\bf ALNN}\\
			$\vw \in \R^m$, $\theta \in \R$ & & &\\
			$\sigma$ monotonic sigmoid & & &\\		
			\hline
			$\sigma(\vw^T\vs + \theta)$ & $C([0,1]^m)$ & \cite{Jon90}&{\bf ALNN}\\
			$\vw \in \R^m$, $\theta \in \R$ & & &\\
			$\sigma$ sigmoid & & &\\	
			\hline
			$\sigma(\vw^T\vs + \theta)$ & $L^1(]0,1[^m)$ & \cite{Cyb89}&{\bf ALNN}\\
			$\vw \in \R^m$, $\theta \in \R$ & & &\\
			$\sigma \in L^1(\R)$, $\int_\R \sigma(s)ds \neq 0$ & & &\\	
			\hline
			$\sigma(w^Ts + \theta)$& $W^{1,2}(]0,1[)$ & \autoref{th:rates_relu} & \ReLU\\
			$w \in \R$, $\theta \in \R$ & & &\\
			$\sigma=\ReLU$ & & &\\	
			\hline
			$\sigma(U \vs + \vec{\theta})$ & $L^1(]0,1[^m)$ & \cite{Cyb89} &\\
			$U \in \R^{m \times m}$ orthogonal, $\vec{\theta} \in \R^m$ & & &\\
			$\sigma =\chi_{[0,1]^m}$ & & &\\
			\hline
			$\sigma(t \vs + \vec{\theta})$ & $L^1(]0,1[^m)$ & \cite{Rud73} &\\
			$t \in \R$, $\vec{\theta} \in \R^m$ & & &\\
			$\sigma \in L^1(\R^m)$, $\int_{\R^m} \sigma(\vs) \neq 0$ & &  &\\	
			\hline
			$\sigma \left( \vec{w}^T \vs + \vs^T A \vs +\theta\right)$ & $C([0,1]^m)$ & \cite{FriSchShi25} & {\bf CQNN}\\
			$\vw \in \R^m$, $\theta \in \R$ & & &\\
			$\sigma$ continuous sigmoid & &  &\\
			\hline
			$\sigma \left( \vec{w}^T \vs + \xi \vs^T A \vs +\theta\right)$ & $C([0,1]^m)$ & \cite{FriSchShi25} & {\bf MCNN}\\
			$\vw \in \R^m$, $\theta, \xi \in \R$ & & &\\
			$\sigma$ continuous sigmoid & &  &\\
			\hline	
			$\sigma(\xi \norm{\vs}^2 + \vw^T \vs  + \theta)$ & $C([0,1]^m)$ & \cite{FriSchShi25} & \RQNN\\
			$\vw \in \R^m$, $\theta, \xi \in \R$ & & &\\
			$\sigma$ continuous sigmoid & &  &\\
			\hline	
			$\sigma (\vw^T (\text{sgn} (\vs) \vs^2) +\theta)$ & $C([0,1]^m)$ & \cite{FriSchShi25} & {\bf SBQNN}\\
			$\vw \in \R^m$, $\theta \in \R$ & & &\\
			$\sigma$ continuous sigmoid & &  &\\
			\hline
			$\sigma(\vw^T \vs^3+\theta)$  & $C([0,1]^m)$ & \cite{FriSchShi25} & {\bf CUNN}\\
			$\vw \in \R^m$, $\theta \in \R$ & & &\\
			$\sigma$ continuous sigmoid & &  &\\		
			\hline
			$s \mapsto 2^{j/2}\sigma(2^j s + k)$, $k \in \ZN$  & $W^{2,\xi}(\R)$ & \cite{BasBoi98} & {\bf DWF}\\	
			multiresolution analysis &&& \\
			with respect to $j \in \N_0$  &  & & \\
			\hline
		\end{tabular}
	\end{center}
	\caption{\label{ta:nn} The table shows possible approximations of arbitrary functions with shallow neural networks.}
\end{table}

\section{Convergence rates}\label{sec:rates}
Convergence rates results for neural networks in function spaces (typically Sobolev spaces), which are required in the analysis of regularization methods remain widely open. We present a few results, which have been derived so far, in particular, in the wavelet context. Moreover, some results for noisy data can also be found in \cite{LiLuMatPer24}.

\subsection{\ReLU networks}
Let \ReLU be the normalized function from \autoref{eq:ReLU} with $a=1$ and $b=0$. That is
\begin{equation*}
	s \in \R \mapsto \ReLU(s) = \max \set{0,s}\;.
\end{equation*}
Then for the linear spline functions defined in \autoref{eq:linearspline} we find
\begin{equation} \label{eq:lambda_ReLU}
	\begin{aligned}
		\frac{1}{\m} \Lambda_i(s) &= \ReLU \left(s-\frac{i-1}{\m} \right) -2 \ReLU\left(s-\frac{i}{\m} \right) \\
		&\quad + \ReLU \left(s-\frac{i+1}{\m} \right) \text{ for } i=0,\ldots,\m \text{ and } s \in [0,1]\;.
	\end{aligned}
\end{equation}
It follows that
\begin{equation*}
	{\tt R}_\m = \spann \set{ \ReLU \left(s-\frac{j}{\m}\right) : j=0,\ldots,\m}
\end{equation*}
contains all linear splines on a uniform grid of $\m+1$ points on $[0,1]$. 

\begin{remark}
	A consequence of this is that we get convergence rates results for \ReLU approximation functions from splines. It is therefore not surprising that differential equations can be solved efficiently with \ReLU-networks.
\end{remark}

\begin{theorem} \label{th:rates_relu}
%	The set
%	\begin{equation*}
%		{\tt R}_\m = \spann \set{\m \ReLU \left(s-\frac{j}{\m}\right) : j=0,\ldots,\m-1}
%	\end{equation*}
%	is a superset of the linear spline space $\Y_\m$ defined in \autoref{eq:splinespace}.
%	Note that $s \to \ReLU(s-1)$ has no support in $[0,1]$. Thus it is neglected in the definition of ${\tt R}_\m$, we just sum up to $\m-1$.
	%Thus, in particular we can approximate
	Each $\x \in W^{2,2}(0,1)$ can be approximated by a \ReLU-network such that 
	\begin{equation} \label{eq:rate_relu}
		\norm{\x-\x_{\m}}_{L^2(0,1)} = \mathcal{O}(\m^{-2}) \text{ and } \norm{\x-\x_{\m}}_{W^{1,2}(0,1)} = \mathcal{O}(\m^{-1})\;.
	\end{equation} 
\end{theorem}
\begin{proof} \autoref{eq:rate_relu} follows immediately from the results used in \autoref{ex:ac}, in particular \autoref{eq:interp-error}.
\end{proof}

\begin{remark} We emphasize that for approximation of an arbitrary function we only need \ReLU functions with slope $a=1$ and dilations on a uniform grid. At the same time, when we make use of \emph{adaptivity}, that is, when we allow for discretization points which are non-uniform, then we can get a better rate $o(\m^{-1})$, than \autoref{eq:rate_relu} (see \cite{FrePop69}). In this situation we also need to adapt the slopes of the \ReLU network. In fact in general the weights are $\frac{1}{s_{i+1}-s_i}$ in this situation. 	
\end{remark}

\subsection{Wavelets}
A motivation for \emph{wavelets}\index{wavelets} stems from the \emph{windowed Fourier-transform}, \index{Fourier-transform!windowed} which was historically the first concept of a time-frequency analysis.
The idea consists in making locally around a given time instance a Fourier-analysis.

\subsection*{Windowed Fourier-transform}

\begin{definition}\label{de:wft} For a given function $g \in L^2(\R)$ (can be real or complex valued), called the \emph{window function},\index{function!window} the \emph{windowed Fourier-transform} of a function $f \in L^2(\R)$ (can be real or complex valued) is defined as follows
	\begin{equation} \label{eq:wft}
		\begin{aligned}
			(\omega,s) \in \R^2 \mapsto {\mathcal W}[{\tt f}](\omega,s) &:= \int_\R {\tt f}(t) \overline{{\tt g}(t-s)} \exp^{-\i \omega t}\,dt \\
			&= \inner{\tt f}{{\tt g}^{(\omega,s)}}_{L^2},
		\end{aligned}
	\end{equation}
	where $\inner{\cdot}{\cdot}_{L^2}$ denotes the complex $L^2$-inner product on $\R$ and
	\begin{equation} \label{eq:gpq}
		t\in \R \mapsto {\tt g}^{(\omega,s)}(t) := {\tt g}(t-s)\exp^{\i \omega t}.
	\end{equation}
	The function ${\tt g} = {\tt g}^{(0,0)}$ is called \emph{centered} at $(0,0)$ if \index{window!centered}
	\begin{equation} \label{eq:cwf}
		\int_\R t \abs{{\tt g}(t)}^2 dt = 0 \text{ and } \int_\R \omega \abs{\mathcal{F}[{\tt g}](\omega)}^2 d\omega = 0.
	\end{equation}
\end{definition}
The windowed Fourier-transform describes the local frequency spectrum of a function $f$ at $s \in \R$.
\begin{lemma}\label{le:iso}
	The windowed Fourier-transform is a mapping from $L^2(\R)$(real or complex valued) to $L^2(\R^2;\C)$ because it satisfies
	\begin{equation*}
		\norm{{\mathcal W}[{\tt f}]}_{L^2(\R^2;\C)}^2 = 2\pi \norm{\tt g}_{L^2}^2\norm{\tt f}_{L^2}^2\;.
	\end{equation*}
\end{lemma}
\begin{proof}
	Let $s \in \R \mapsto {\tt h}(s) := \ol{{\tt g}(-s)}$ and $s \mapsto {\tt j}(s) := {\tt f}(s) \exp^{-\i \omega s}$, then it follows from the isometry of the Fourier-transform, \autoref{eq:fourier_isometry} and \autoref{eq:conf_F} that
	\begin{equation} \label{eq:w}
		\begin{aligned}
			\ & \int_{\R^2} \abs{{\mathcal W}[{\tt f}](\omega,s)}^2\,ds\,d\omega
			=
			\int_{\R^2} \abs{ \int_\R {\tt f}(t) \ol{{\tt g}(t-s)} \exp^{-\i \omega t} \,dt }^2 ds\,d\omega \\
			= &
			\int_{\R^2} \abs{ \int_\R {\tt h}(s-t) {\tt f}(t) \exp^{-\i \omega t} \,dt }^2 ds\,d\omega =
			\int_{\R^2}
			\abs{ h*{\tt j}(s) }^2 ds\,d\omega \\
			= &
			\int_{\R^2}
			\abs{ \mathcal{F}[h*j](\hat{\omega}) }^2 d \hat{\omega} \,d\omega
			=
			2\pi\int_{\R^2}
			\abs{ \mathcal{F}[{\tt h}](\hat{\omega})\mathcal{F}[{\tt j}](\hat{\omega})}^2 d\hat{\omega} \,d\omega \;.
		\end{aligned}
	\end{equation}
	From the definition of the Fourier-transform it follows that
	\begin{equation*}
		\abs{\mathcal{F}[{\tt j}](\hat{\omega})} = \abs{\frac{1}{\sqrt{2\pi}} \int_\R {\tt f}(s) \exp^{- \i (\omega+\hat{\omega})s} \,ds} = \abs{\mathcal{F}[{\tt f}](\hat{\omega}+\omega)}\,,
	\end{equation*}
	and
	\begin{equation*}
		\begin{aligned}
			\abs{\mathcal{F}[{\tt h}](\omega)}^2 &= \frac{1}{2\pi} \abs{ \int_\R \ol{{\tt g}(-s)} \exp^{-\i \omega s} d s }^2
			= \frac{1}{2\pi} \abs{ \int_\R \ol{{\tt g}(s)} \exp^{\i \omega s} }^2 d s \\
			&= \frac{1}{2\pi} \abs{ \int_\R {\tt g}(s) \exp^{- \i \omega s} }^2 = \abs{\mathcal{F}[{\tt g}](\omega)}^2\;.
		\end{aligned}
	\end{equation*}
	Plugging the last two identities in \autoref{eq:w} we get
	\begin{equation*}
		\begin{aligned}
			\int_{\R^2} \abs{\mathcal{F}[{\tt h}](\hat{\omega}) \mathcal{F}[{\tt j}](\hat{\omega}) }^2 d\hat{\omega} d\omega
			&= \int_{\R^2} \abs{\mathcal{F}[{\tt g}](\hat{\omega})}^2 \abs{\mathcal{F}[{\tt f}](\hat{\omega}+\omega)}^2 d\hat{\omega} \,d\omega\\
			&= \int_\R \abs{\mathcal{F}[{\tt g}](\hat{\omega})}^2 d\hat{\omega} \int_\R \abs{\mathcal{F}[{\tt f}](\omega)}^2 \,d\omega\;.
		\end{aligned}
	\end{equation*}
	Using the isometry of the Fourier-transform and plugging the following identity in \autoref{eq:w} we get the assertion.
\end{proof}
For practical application a discretization is employed:
\begin{equation*}
	\omega = j \omega_0, s = k s_0\;
\end{equation*}
\begin{definition}
	The discrete window functions are defined as \index{window function!discrete}
	\begin{equation} \label{eq:gmn}
		t \mapsto {\tt g}_{j,k}(t) := {\tt g}^{(j\omega_0,ks_0)}(t) = {\tt g}(t-ks_0)\exp^{\i  j \omega_0 t}, \quad (j,k) \in \ZN^2.
	\end{equation}
	The \emph{discrete windowed Fourier-transformation} \index{windowed Fourier-transformation!discrete} is defined as
	\begin{equation} \label{eq:dwft}
		\begin{aligned}
			\mathcal{W}_d: L^2(\R) &\to l^2(\ZN^2).\\
			{\tt f} &\mapsto \left(
			\inner{{\tt g}_{j,k}}{\tt f}_{L^2}=:c_{j,k}[{\tt f}] \right)_{(j,k) \in \ZN^2}\;.
		\end{aligned}
	\end{equation}
\end{definition}

\begin{remark}
	For a fixed frequency $\omega_0$, $t \mapsto {\tt g}^{(\omega_0,s)}(t) = {\tt g}(t-s)\e^{\i \omega_0 t}$ are $L^2$-discriminatory (see \autoref{de:disc_function}) if from $\mathcal{W}[{\tt f}](\omega_0,\cdot) = 0$ it follows that ${\tt f} \equiv 0$.
\end{remark}

\subsection*{Wavelet transformation}
Instead of the window functions defined in \autoref{eq:gpq} we use now \emph{analyzing wavelet}\index{wavelet!analyzing}, where shifts (in this case $b$) and scale (in this case $a$) are linked. The main idea stays the same, namely to analyze a function locally in time and frequency, although frequency has a different meaning.
\begin{definition}
	Let $\psi \in L^2(\R)$ and $a \neq 0$. The family of functions
	\begin{equation} \label{eq:ana_wavelet}
		\psi^{(a,b)}(t) := \abs{a}^{-\frac{1}{2}} \psi\left( \frac{t-b}{a} \right)
		\text{ for } a \neq 0
	\end{equation}
	is called \emph{continuous wavelet family}.\index{wavelet family!continuous}
\end{definition}
The factor $\abs{a}^{-\frac{1}{2}}$ is introduced to preserve norm properties:
\begin{lemma} \label{le:norm_wavelets} For every $a \neq 0$ and $b \in \R$ we have
	\begin{equation*}
		\int_\R \abs{\psi^{(a,b)}(t)}^2 dt = \int_\R \abs{\psi(t)}^2 dt\;.
	\end{equation*}
\end{lemma}
\begin{proof}
We compute directly using change of variable:
\begin{equation*}
\begin{aligned}
\int_{\R} |\psi^{(a,b)}(t)|^2 dt
&= \int_{\R} \left| |a|^{-1/2} \psi\!\left( \frac{t-b}{a} \right) \right|^2 dt \\
&= |a|^{-1} \int_{\R} \left| \psi\!\left( \frac{t-b}{a} \right) \right|^2 dt \\
&= |a|^{-1} \int_{\R} |\psi(u)|^2 \cdot |a| \, du \qquad (u = (t-b)/a,\; dt = |a|\,du) \\
&= \int_{\R} |\psi(u)|^2 du = \int_{\R} |\psi(t)|^2 dt.
\end{aligned}
\end{equation*}
Thus $\|\psi^{(a,b)}\|_{L^2} = \|\psi\|_{L^2}$ for every $a \neq 0$ and $b \in \mathbb{R}$.
\end{proof}
\begin{remark}
	The support of the functions $\set{\psi^{(a,b)} : a \neq 0, b \in \R}$ is increasing with respect to $a$.
	The function $\psi^{(a,b)}$ (see \autoref{eq:cwf}) is centered at $t=b$ and scaled by the factor $a$, which means that
	\begin{equation*}
		\int_\R t \abs{\psi^{(a,b)}(t)}^2 dt = b \text{ and }
		\int_\R \omega \abs{\mathcal{F}[\psi^{(a,b)}](\omega)}^2 d \omega = \frac{\omega_0}{a}\;.
	\end{equation*}
\end{remark}
In analogously with the windowed Fourier-transformation we define the \emph{continuous wavelet transform}\index{wavelet tranform!continuous}:
\begin{definition}[Continuous wavelet transform] \label{de:cwt}
	\begin{equation} \label{eq:cwt}
		\Psi[\tt f](a,b):=\Psi(b,a) := \int_\R  {\tt f}(t) {\overline{\psi^{(a,b)}}}(t)dt =
		\inner{{\tt f}}{\psi^{(a,b)}}_{L^2}\;.
	\end{equation}
\end{definition}
\begin{remark}
	We note that the wavelet family $\psi^{(a,b)}$ is a composition of an activation function $\Psi$ and an affine linear function. That is it is an \ALNN. Moreover, we show below that there exists a countable family of wavelets, with which we can approximate arbitrary functions $\x \in L^2(\R)$. This leads to the discrete wavelet transform.
\end{remark}
As for the windowed Fourier-transform we defined the discrete Wavelet transformation.
However, now we use consider an \emph{exponential grid}:
\begin{equation*}
	a = a_0^j \text{ with } a_0 > 0 \text{ and } b = k b_0 a_0^j \text{ for } j,k \in \ZN\;.
\end{equation*}
With this grid we get the discrete wavelets from \autoref{eq:ana_wavelet}:
\begin{equation} \label{eq:dwt}
	\psi_{j,k}(t) = \psi^{(a_0^j,k b_0 a_0^j)}(t) = a_0^{-\frac{j}{2}} \psi(a_0^{-j}t - k b_0) \text{ for } j,k \in \ZN\;.
\end{equation}
Now, we define the discrete wavelet transformation: \index{wavelet tranformation!discrete}
\begin{definition}
	The \emph{discrete wavelet transformation} $\Psi_d: L^2(\R) \to l^2(\ZN^2)$ is defined as follows:
	\begin{equation} \label{eq:dwt1}
		\begin{aligned}
			\Psi_d[f](j,k) &:= c_{(j,k})[{\tt f}]:= \inner{{\tt f}}{\psi_{j,k}}_{L^2} \\
			& =
			a_0^{-\frac{j}{2}}
			\int_\R  {\tt f}(t) {\overline{\psi (a_0^{-j}t - k b_0)}}\,dt\;.
		\end{aligned}
	\end{equation}
\end{definition}
If the family $\psi_{j,k}$ is an orthonormal basis of $L^2(\R)$, then we can reconstruct a function from the wavelet coefficients. The construction of such has been made in the fundamental work of Ingrid Daubechies and Yves Meyer \cite{Dau88,Dau92,Mey91,Mey93,Mey93b}.

\subsection*{Wavelet filter coefficients}
Now, we construct filter coefficients of real, compactly supported, orthonormal wavelets. We restrict attention to $a_0=2$ and $b_0=1$ in \autoref{eq:dwt}. We follow Daubechies' original construction (see \cite{Dau88,Dau92}) of orthonormal wavelets, which is based on the existence of a scaling function:
%
%{\color{orange} I would like to recommend to write the definition of scaling function as follows:
%	
%	Item (i): \(\phi(t) = \sum_{k \in \mathbb{Z}} c_k \phi(2t-k)\) for some coefficient \(c_k\) (comment: This coefficient also can be real or complex, as we discuss before. But I am accustomed to real. Just choose what you want in the book.)
%	
%	Item (ii): Place the second item about MRA. It seems fine for me.
%	
%	Item (iii): \(\{\phi_{0,k}(t) = \phi(t-k): k \in \mathbb{Z}\}\) is orthonormal w.r.t \(V_0\).
%	
%	\vspace{0.5\baselineskip}
%	I recommend to write this as above since finding orthonorml systems (via \(\psi_{j,k}\), which also implies the same property to \(\phi_{j,k}\)) in L2 is the ultimate goal of the wavelet theory (in my opinion). So at the end of the day, this first item would be right. But I think before then, it would be better to write like this. But this is just my opinion. So, depends on you.}
%	
%	\vspace{0.5\baselineskip}
%{\color{orange}Oh Item (i) is already in (4.76)...
%	I truely recommend to write scaling function with scaling relation (Eqn 4.76). But I am not sure what you prefer.}

\begin{definition}[Scaling function]\label{de:scaling}
	A \emph{scaling function}\index{scaling function} $\phi: \R \to \R$ or $\C$, satisfies
	\begin{enumerate}
		\item for every $j \in \ZN$ the functions
		\begin{equation} \label{eq:scalf}
			t \mapsto \phi_{j,k}(t) := 2^{-j/2}\phi(2^{-j}t-k) \text{ for all } k\in \ZN\,
		\end{equation}
		are orthonormal with respect to $L^2({\R})$. $j$ is called the \emph{scale dimension} of $\phi_{j,k}$.\index{scale dimension}
		\item The sets
		\begin{equation}\label{eq:vm}
			V_j:=\overline{\spann\set{\phi_{j,k} :k\in\ZN}} \text{ for all } j\in\ZN
		\end{equation}
		form a multiresolution analysis on $L^2({\R})$\index{Multiresolution analysis}: That is
		\begin{equation}\label{eq:mra}
			V_j \subset V_{j-1} \text{ for all } j\in\ZN\,,
		\end{equation}
		with
		\begin{equation*}
			\bigcap_{j\in\ZN} V_j = \set{ 0 } \mbox{ and }
		    \overline{\bigcup_{j\in\ZN}V_j} = L^2({\R}).
		\end{equation*}
	\end{enumerate}
\end{definition}
The wavelet space is defined based on the scaling function:
\begin{definition}[Wavelet function]\label{de:wavelet}
	The wavelet spaces $\set{W_j : j \in \ZN}$ according to a scaling function $\phi$ are the orthogonal
	complements of $V_j$ in $V_{j-1}$: That is
	\begin{equation*}
		W_j := V_j^\perp \cap V_{j-1}.
	\end{equation*}
	The \emph{mother wavelet} $\psi$ is chosen such that the set
	\begin{equation*}
		t \mapsto \psi_{j,k}(t) := 2^{-j/2}\psi(2^{-j}t-k) \text{ for } k\in\ZN,
	\end{equation*}
	 form an orthonormal basis of $W_j$.
	\index{space!wavelet} \index{mother wavelet}
\end{definition}
\begin{remark}
Since $\phi=\phi_{0,0}\in V_0\subset V_{-1}$ (see \autoref{eq:mra}),
the scaling function $\phi$ must satisfy the \emph{dilation equation}
\begin{equation}\label{eq:dil}
	\phi(t) = \sum_{k \in \ZN} h_k\phi(2t-k) \text{ for almost all } t \in \R\,,
\end{equation}
where the sequence~$\set{h_k: k \in \ZN}$ is called \emph{filter sequence}. \index{filter sequence}
\end{remark}

Following \cite{Dau88,Dau92} the construction of
orthonormal wavelet functions is reduced to the design of the
corresponding filter sequence $\set{h_k : k \in \ZN}$ in \autoref{eq:dil}. Moreover,
one assumes that the mother wavelet $\psi$ satisfies
\begin{equation}\label{e:dil-w}
	\psi(t) = \sum_{k \in \ZN} (-1)^k h_{1-k}\phi(2t-k) \text{ for almost all } t \in \R\;.
\end{equation}
In particular this choice guarantees that the wavelet $\psi$ and the scaling
function $\phi$ are orthogonal in $L^2(\R)$.

In orthogonal wavelet theory the desired properties on the scaling function and wavelets are:
\begin{enumerate}
	\item For fixed integer $N \ge 1$ the scaling function $\phi$ has
	support in the interval $[1-N,N]$. This in particular holds when
	the filter coefficients satisfy
	\begin{equation}\label{e:hk.1}
		h_k=0 \text{ for }k<1-N \text{ and } k>N.
	\end{equation}
	This means that for all $t \in \R$,
	\begin{equation}\label{e:dil_N}
		\phi(t) = \sum_{k=1-N}^N h_k\phi(2t-k) \text{ and } \psi(t) = \sum_{k=1-N}^N (-1)^k h_{1-k}\phi(2t-k)
	\end{equation}
	\item The existence of a scaling function $\phi$ satisfying
	\autoref{eq:dil} requires that
	\begin{equation}\label{e:hk.2}
		\sum_{k \in \ZN} h_k=2\;.
	\end{equation}
	\item In order to impose orthonormality of the integer translates of
	the scaling function $\phi$: That is
	\begin{equation*}
	   \int\limits_{\R} \phi(t-l)\phi(t) dt =\delta_{0,l},
	\end{equation*}
	the filter coefficients $\set{h_k : k \in \ZN}$ have to satisfy
	\begin{equation} \label{e:hk.3}
		\sum_{k \in \ZN} h_k h_{k-2l} = 2\delta_{0,l} \text{ for }l=0,\dots,N-1.
	\end{equation}
	\item The wavelet $\psi$ is postulated to have $N$ vanishing
	moments, i.e.,
	\begin{equation}
		\label{vanish}
		\int\limits_{\R} t^l\psi(t) dt =0 \text{ for } l=0,\dots,N-1\,,
	\end{equation}
	which requires the filter sequence to satisfy
	\begin{equation}\label{e:hk.4}
		\sum_{k \in \ZN} (-1)^k h_{1-k} k^l = 0\text{ for } l=0,\dots,N-1\;.
	\end{equation}
\end{enumerate}
\begin{example}
	\begin{enumerate}
		\item $N=1$: The real valued \emph{Haar-wavelet} is given by\index{Haar wavelet}
	\begin{equation*}
		(h_0,h_1) =
		\left( \frac{1}{\sqrt{2}}, \frac{1}{\sqrt{2}}\right)\;.
	\end{equation*}
	The Haar-functions were already introduced in 1910 (see \cite{Haa10}).
	\item
	There exist four Daubechies wavelets of order $2N=4$. The two with real valued coefficients are given by
\begin{equation*}
	(h_{-1},h_0,h_1,h_2) =
	\left( \frac{1+{\sqrt 3}}{4}, \frac{3+{\sqrt 3}}{4},
	\frac{3-{\sqrt 3}}{4}, \frac{1-{\sqrt 3}}{4} \right)
\end{equation*}
and
\begin{equation*}
	(h_{-1},h_0,h_1,h_2) =
	\left( \frac{1-{\sqrt 3}}{4}, \frac{3-{\sqrt 3}}{4},
	\frac{3+{\sqrt 3}}{4}, \frac{1+{\sqrt 3}}{4} \right).
\end{equation*}
One real valued Daubechies $3$ wavelet is given by
\begin{equation*}
	\begin{aligned}
	~ & (h_{-2},h_{-1},h_0,h_1,h_2,h_3) \\
	= &
\biggl( {\frac{1 + {\sqrt{10}} +
		{\sqrt{5 + 2{\sqrt{10}}}}}{16}},
{\frac{5 + {\sqrt{10}} +
		3{\sqrt{5 + 2{\sqrt{10}}}}}{16}},
{\frac{5 - {\sqrt{10}} +
		{\sqrt{5 + 2{\sqrt{10}}}}}{8}}, \biggr. \\
	&\quad \biggl. {\frac{5 - {\sqrt{10}} -
		{\sqrt{5 + 2{\sqrt{10}}}}}{8}},
{\frac{5 + {\sqrt{10}} -
		3{\sqrt{5 + 2{\sqrt{10}}}}}{16}},
{\frac{1 + {\sqrt{10}} -
		{\sqrt{5 + 2{\sqrt{10}}}}}{16}}\,\biggr)\;.
	\end{aligned}
\end{equation*}
\end{enumerate}
\end{example}
\begin{remark}
The scaling functions and wavelets are elements of the neural network functions,
\begin{equation*}
\begin{aligned}
\Phi &:= \set{ t \mapsto a \phi(w t + \theta) : a,w,\theta \in \R}\,,\\
\Psi &:= \set{ t \mapsto a \psi(w t + \theta) : a,w,\theta \in \R}\,,
\end{aligned}
\end{equation*}
respectively.

Since the Daubechies scaling functions and wavelets are compactly supported, they are Taubner-Wiener functions (see \autoref{Wiener_Tauberian} and recall the Paley-Wiener theorem). In particular, every function $\x$ can be approximated by shifts $s \mapsto \phi(s+\theta)$ and $s \mapsto \psi(s+\theta)$, respectively.
\end{remark}

\subsection*{Smoothness and convergence rates of wavelets}
With the filter coefficients $\set{h_k: k \in \ZN}$ satisfying \autoref{e:hk.1},  \autoref{e:hk.2}, \autoref{e:hk.3} and \autoref{e:hk.4}, we associate the trigonometric polynomial
\begin{equation} \label{eq:tau}
	\xi \in [-\pi,\pi) \mapsto \tau(\xi) := \sum_{k \in \ZN} h_k \e^{-\i k\xi}\;.
\end{equation}
This polynomial satisfies
\begin{equation}\label{lowpass_cond}
	\abs{\tau(\xi)}^2 = \cos^{2N} (\xi/2) {\tt r}(\xi) \text{ for all } \xi \in [-\pi,\pi)
\end{equation}
where
\begin{equation*}
	\xi \mapsto {\tt r}(\xi) := \sum_{k=0}^{N-1} r_k \cos(k\xi)
\end{equation*}
is a trigonometric polynomial with real coefficients $\set{r_k: k=0,\ldots,N-1}$. In this case the filter coefficients of the scaling function $\set{h_k:k \in \ZN}$ are real, which we aim for.
Associated with the trigonometric polynomial ${\tt r}$, there is the operator
\begin{equation}\label{eq:Tr}
	\begin{aligned}
		T_{\tt r} : C[-\pi,\pi) &\to C[-\pi,\pi)\\
		{\tt u} & \mapsto \left(\xi \in [-\pi,\pi) \to {\tt r}(\xi/2) {\tt u}(\xi/2) + {\tt r}(\pi - \xi/2) {\tt u} (\pi - \xi/2) \right).
	\end{aligned}
\end{equation}
Evaluation of the operator $T_{\tt r}$ for each of the functions in the set
\begin{equation*}
\set{\xi \in [-\pi,\pi) \mapsto {\tt e}_l(\xi):= \cos(l\xi): 0 \leq l \leq N-1}\,,
\end{equation*}
we get
\begin{equation*}
\begin{aligned}
	T_{\tt r} [{\tt e}_l](\xi) & = \sum_{k=0}^{N-1} r_k (1+ (-1)^{k+l}) \cos(k\xi/2) \cos(l\xi/2)\\
	& = \sum\limits_{\substack{0 \leq k \leq N-1 \\ k+l \text{ is even}}} r_k \left( \cos\left(\frac{(k+l)\xi}{2}\right) + \cos \left(\frac{(k-l)\xi}{2}\right) \right).
\end{aligned}
\end{equation*}
Then we define the matrix
\begin{equation} \label{eq:A_r}
\begin{aligned}
    A_{\tt r} &= (A_{m,l})_{0 \leq m,l \leq N-1}
\text{ with } \\
	A_{m,l} &= \begin{cases}
		r_{2m - l} + r_{2m + l} & \text{if } 0 \leq 2m-l \leq N-1 \text{ and } 0 \leq 2m + l \leq N-1,\\
		r_{2m - l} & \text{if } 0 \leq 2m-l \leq N-1 \text{ and } 2m + l \notin [0,N-1],\\
		r_{2m + l} & \text{if } 0 \leq 2m+l \leq N-1 \text{ and } 2m - l \notin [0, N-1],\\
		0 & \text{otherwise.}
	\end{cases}
\end{aligned}
\end{equation}
\begin{theorem}\label{s_N^*}
	Let $N$ be the degree of a Daubechies wavelet (see \autoref{e:hk.1}).
    Moreover, let $\tau$ be as in \autoref{eq:tau} and let ${\tt r}$ be the trigonometric
    polynomial defined in \autoref{lowpass_cond} with the associated matrix
    $A_{\tt r}$ defined in \autoref{eq:A_r}. Then for every
    \begin{equation} \label{eq:csob}
        s < s_N^* := N - \log_4 \rho(A_{\tt r})\,,
    \end{equation}
    where $\rho(A_{\tt r})$ denotes the spectral norm of $A_{\tt r}$,
    the scaling function $\phi$ and wavelet $\psi$ are in $W^{s,2}(\R)$.
    $s_N^*$ is called \emph{critical Sobolev index}\index{critical Sobolev index}
\end{theorem}
The proof of this theorem can be found in \cite{Coh03}.
\begin{table}\label{table:wavelet}
	\begin{center}
		\begin{tabular}{|r||c|c|c|c|c|c|c|c|c|c|}
			\hline
			N & 1 & 2 & 3 & 4 & 5 & 6 & 7 & 8 & 9 &10\\
			\hline
			Sobolev $s_N^*$ & 0.5 & 1 & 1.42 & 1.78 & 2.10 & 2.39 & 2.66 & 2.91 &3.16 &3.40 \\
			%H\"older $s_N^*$ & & 0.55 &0.92 &1.28 &1.6 &1.89 &2.16 &2.42 &2.67 &2.9\\
			\hline
		\end{tabular}\caption{\label{ta:smoothwavelet} The critical Sobolev index defined in \autoref{s_N^*}. The Daubechies scaling function $\phi$ and wavelet $\psi$ are elements of the Sobolev space $W^{s,2}$ with $s < s_N^*$.}
	\end{center}
\end{table}
%\end{theorem}
Finally, we recall a convergence rates result:
\begin{theorem} \label{thm: Sobolev-error-estimate}
	Let $\phi \in W^{s,2}(\R)$ be a scaling function of Daubechies $N$-wavelets and let $0 < s < t < s_N^*$, then for all $\x \in W^{t,2}(\R)$ we have
	\begin{equation} \label{eq:rates_wavelets}
		\norm{\x - P_j\x} _{W^{s,2}} \leq 2^{-j(t-s)} \norm{\x}_{W^{t,2}}.
	\end{equation}
	Here $P_j$ denotes the projection of $\x$ onto the space $V_j$ as defined in \autoref{eq:vm}.
\end{theorem}

One can apply this theorem to $a$- and $c$-problem (see \autoref{a_reconstruct} and \autoref{ss:c}), respectively.
\commentO{For the moment we have to leave this out:
\begin{example}
	Let ${\tt m} \in \Z$ and let $\X_{\tt m} := V_{\tt m}$ be the space of linear combination of Daubechies-$N=2$-wavelets (see \autoref{eq:vm}). We verify the conditions of \autoref{th:NeuSch90} and \autoref{th:NeuSch90b} with 
	$\opo_{\tt n}=\opo$ for the $c$-problem (see \autoref{ss:c}) to simplify the considerations. This, in particular means that the conditions taking into account $\nu_\n$ (for instance in \autoref{eq:ps1} and following) are always satisfied.
    For Daubechies-$2$-wavelets, the critical Sobolev-smoothness is $s_2^* = 1$ (see \autoref{ta:smoothwavelet}), showing (see \autoref{eq:rates_wavelets}) that for arbitrary but fixed $0 < t< 1$, for all $\z \in W^{t,2}$ we have the $L^2 = W^{s=0,2}$ estimate
    $$\norm{\z - P_{\tt m}\z}_{L^2} \leq C 2^{-{\tt m}t}\norm{\z}_{W^{t,2}}\;.$$ 
    This shows, in particular, that for $\xdag \in W^{t,2}$ we have 
    $$ \norm{(I-P_\m) \xdag} = \mathcal{O}(2^{-{\tt m}t})\;.$$
    Moreover, since the range of $\opo'[\xdag]$ is in $W^{2,2}$ \commentO{Check} we also have  that $\gamma_{\tt m}$ as defined in \autoref{eq:gamma} satisfies \commentO{Here the index $t$ is different. I think it is $t=1$}
    $$\gamma_{\tt m} = \mathcal{O}(2^{-{\tt m}t})\;.$$
    Therefore \autoref{eq:ps2} in \autoref{th:NeuSch90} requires to choose $\alpha$ such that 
    \begin{equation} \label{eq:ps2a}
      4^{-{\tt m}t}/\sqrt{\alpha} \to 0\,,
    \end{equation}
    for obtaining a convergent regularization method.     
    For the convergence rates result in \autoref{th:NeuSch90b} we require that \commentO{$\x^0 \in W^{2,2} \cap W_0^{1,2}$???} and a choice (see \autoref{eq:eta}) \commentO{The Sobolev spaces are not the same in the previous example and the one here}
    $$\alpha \sim \max\{\delta, \gamma_{\tt m}^2\}$$  
    yields $\norm{\x_{\tt m,\tt n}^{\alpha,\delta,\eta} - \xdag} = \mathcal{O}(\sqrt{\delta} + 2^{-{\tt m}t})$. Balancing $2^{-{\tt m}t} \sim \sqrt{\delta}$ gives ${\tt m} \sim \frac{1}{t}\log_2(1/\sqrt{\delta})$ and the optimal rate is $\mathcal{O}(\sqrt{\delta})$.
\end{example}}

\commentO{This is definitely not part of the Arxive:
\begin{example}
Take $\X_{\tt m} = V_{\tt m}$ the space of Daubechies $N=3$ wavelets ($6$ coefficients) at scale ${\tt m}$ (see \autoref{eq:vm}). Applying \autoref{th:NeuSch90} and \autoref{th:NeuSch90b} with $\opo_{\tt n} = \opo$ to the $a$-problem, where $\opo: \x \mapsto \y$ satisfies $-(\x \y')' = {\tt f}$ with $\y(0)=\y(1)=0$. This means that the operator evaluation of $\opo$ is performed exactly and not by a finite element method. 
For Daubechies $N=3$, $s_3^* \approx 1.42$, so \commentO{Where is $\gamma$ defined}$\gamma_{\tt m} = \mathcal{O}(2^{-1.42{\tt m}})$.  \autoref{th:NeuSch90} requires $\alpha \gg \gamma_{\tt m}^2$, and \autoref{th:NeuSch90b} with $\alpha \sim \max\{\delta, \gamma_{\tt m}^2\}$ yields $\norm{\x_{\tt m}^{\alpha,\delta} - \xdag} = \mathcal{O}(\sqrt{\delta} + 2^{-1.42{\tt m}})$. Balancing $2^{-1.42{\tt m}} \sim \sqrt{\delta}$ gives ${\tt m} \sim \frac{1}{2.83}\log_2(1/\delta)$ and the optimal rate is $\mathcal{O}(\sqrt{\delta})$.
\end{example}}

\subsection{Radial quadratic neural nets (\RQNN{}s)} \label{sec:rqnn}
In the following we review convergence rates of \RQNN{}s (as defined in \autoref{eq:radial_approximation})
in the $\mathcal{L}^1$-norm, defined in \autoref{eq:spaceL1} below, which is a space where the norm is defined via the wavelets coefficients and is not the same as the standard Lebesgue space $L^1$.
To be precise we specify a subset of \RQNN{}s with which we can already approximate arbitrary functions with a rate.
This is a much finer result than the standard universal approximation result, \autoref{th:general}.

However, so far only for \RQNN's quantitative approximation result can be provided, but not for other neural network. One reason for that is that the constraint \autoref{eq:circle}, below, implies that the level-sets of the neurons are compact and not unbounded, as they are for instance for shallow affine linear neurons. Therefore we require a different analysis (and different spaces) than in the most advanced and optimal convergence rates results for affine linear neural networks as for instance in \cite{SieXu22,SieXu23}. In fact the presented analysis is closer to an analysis of compact wavelets.

The convergence rate are expressed in the following norm:
\begin{definition} \label{de:ell1space}
	Let ${\tt U}$ be a family of functions (probably uncountable) in $L^2(\R^m)$.
 For a function $\x \in L^2(\R^m)$ which can be expresses as
\begin{equation*}
\x = \sum_{\mathtt{g} \in {\tt C}} c_{\mathtt{g}} \mathtt{g} \text{ almost everywhere,}
\end{equation*}
where ${\tt C} \subseteq {\tt U}$ is countable,
the \textbf{$\mathcal{L}^1$-norm} of $\x$ is defined as:
\begin{equation} \label{eq:l1norm}
\norm{\x}_{\mathcal{L}^1} := \inf \set{ \sum_{\mathtt{g} \in {\tt C}} |c_{\mathtt{g}}| \ : \ \x= \sum_{\mathtt{g} \in {\tt C}} c_{\mathtt{g}} \mathtt{g} }\;.
\end{equation}
Moreover,
\begin{equation} \label{eq:spaceL1}
\mathcal{L}^1(\R^m) = \set{\x \in L^2(\R^m): \norm{\x}_{\mathcal{L}^1} < \infty}\;.
\end{equation}
Let $\Omega \subseteq \R^m$. We say $\x \in \mathcal{L}^1(\Omega)$ if there exist an extension $\hat{\x} : \R^m \to \R$ such that $\hat{\x} \in \mathcal{L}^1(\R^m)$.
\end{definition}
\begin{remark} We add several remarks on this definition:
	\begin{enumerate}
		\item $\mathcal{L}^1$ \emph{does not} refer to the common $L^1$-function space of absolutely integrable functions and depends on the choice on the family ${\tt U}$. There is only a notational similarity of the spaces. For more properties and details on this space see \cite[Remark 3.11]{ShaCloCoi18}.
		\item In \autoref{eq:norm} we have used the uncountable family
		\begin{equation*}
		   {\tt U} = \set{\sigma(\cdot+\theta) : \theta \in \R}
		\end{equation*}
	    to define the $\mathcal{L}^1$-norm.
	    \item If ${\tt U} = \set{\psi_n: n \in \N}$ is an orthonormal basis of $L^2(\R^m)$, then we have for
	    $\x \in L^2(\R^m)$:
	     \begin{equation*}
	     	\x = \sum_{n \in \N} c_n \psi_n \text{ almost everywhere}.
	     \end{equation*}
	     Here $c_n = \inner{\x}{\psi_n}_{L^2}$.
         And the $\mathcal{L}^1$-norm is the standard $\ell^1$-norm:
         \begin{equation*}
         	\norm{\x}_{\mathcal{L}^1} = \sum_{n \in \N} \abs{c_n} = \norm{(c_n)_{n \in \N}}_{\ell^1}\;.
         \end{equation*}
        \item The $\mathcal{L}^1$-norm is always used, when the set ${\tt U}$ is redundant.
	\end{enumerate}

\end{remark}
The following definitions are quite general and allow for proving convergence rates for wavelets generated from neural network.
\bigskip\par
\fbox{
		\parbox{0.90\textwidth}{
				\begin{center}
					The convergence rates for approximating arbitrary functions are obtained with a countable set of neural network ansatz functions while the family of neural network functions is itself uncountable.
					The family of general neural networks (linear combinations of an arbitrarily shifted and scaled functions) is redundant.
						\end{center}}}
\bigskip\par

In the following we construct a family of \RQNN's (see \autoref{eq:rqnn}) which forms a frame. For this purpose
we require the following technical concepts formulated by Deng \& Han \cite{DenHan09}:
\begin{enumerate}
	\item \emph{Approximation to identity} (AtI)
	which satisfies
	\item the \emph{double Lipschitz condition}.\index{double Lipschitz condition}
\end{enumerate}
The workflow is as follows: First construct a family of functions which forms an AtI and a double Lipschitz condition. From such a wavelet family can be constructed which forms a frame. In turn from an expansion of a function with respect to the frame we get approximation rates.

\begin{definition}[Approximation to the identity (AtI)\cite{DenHan09}] \label{def:wavelet} 
	\index{Approximation to the identity!AtI}
	A family of \emph{symmetric kernel functions}
	\begin{equation*}
	   \set{S_k: \R^m \times \R^m  \to \R : k\in \ZN}
	\end{equation*}
	is said to be an AtI if there exist a quintuple $(\ve ,\zeta,C,C_\rho,C_A)$ of positive numbers satisfying
	\begin{equation} \label{eq:quin}
		0 < \ve  \leq \frac{1}{m}, \; 0 < \zeta \leq \frac{1}{m} \text{ and } C_A < 1
	\end{equation}
	such that the following three conditions are satisfied for all $k \in \ZN$:
	\begin{enumerate}
		\item \label{it1}
		$\abs{S_k(\vs,\vt)} \leq C\frac{2^{-k\ve }}{\left(2^{-k}+ C_\rho\norms{\vs - \vt}_2^m \right)^{1+\ve}}$ for all $\vs,\vt \in \R^m$;
		\item \label{it2}
		$\abs{S_k(\vs,\vt) - S_k(\vr,\vt)} \leq C \left( \frac{C_\rho\norms{\vs - \vr}_2^m }{2^{-k}+ C_\rho\norms{\vs - \vt}_2^m }\right)^{\zeta}\frac{2^{-k\ve }}{\left(2^{-k}+ C_\rho\norms{\vs - \vt}_2^m \right)^{1+\ve }}$ \\
		for all triples $(\vs,\vr,\vt) \in \R^m \times \R^m \times \R^m$ which satisfy
		\begin{equation} \label{eq:rest_item2}
			C_\rho\norms{\vs - \vr}_2^m  \leq C_A \left(2^{-k}+ C_\rho\norms{\vs - \vt}_2^m \right);
		\end{equation}
		\item \label{it4} The following normalization property holds:
		\begin{equation} \label{eq:it4}
			\int_{\R^m} S_k(\vs,\vt) d\vt = 1 \text{ for all } \vs \in \R^m\;.
		\end{equation}
	\end{enumerate}
	Moreover, we say that the AtI satisfies the \emph{double Lipschitz condition}\index{double Lipschitz condition} if there exists a triple $(\tilde{C},\tilde{C}_A,\zeta)$ of positive constants with
	$\tilde{C}_A < \frac12,$ such that for all $k \in \ZN$
	\begin{equation}\label{eq:DoubleLipschitzCondition}
		\begin{aligned}
			&\abs{S_k(\vs,\vt) - S_k(\vr,\vt) - S_k(\vs,\vec{u}) + S_k(\vr,\vec{u})}\\
			\leq & \tilde{C} \left( \frac{C_\rho\norms{\vs - \vr}_2^m }{2^{-k}+C_\rho\norms{\vs - \vt}_2^m} \right)^\zeta
			\left( \frac{C_\rho \norms{\vt - \vec{u}}_2^m}{2^{-k}+C_\rho\norms{\vs - \vt}_2^m }\right)^\zeta \cdot \\&
			\qquad \cdot
			\frac{2^{-k \ve }}{(2^{-k}+C_\rho\norms{\vs - \vt}_2^m )^{1+\ve }}
		\end{aligned}
	\end{equation}
	for all quadruples $(\vs,\vr,\vt,\vec{u})\in \R^m \times \R^m \times \R^m \times \R^m$ which satisfy
	\begin{equation} \label{eq:rest_item3}
		C_\rho \max\set{\norms{\vs - \vr}_2^m,\norms{\vt - \vec{u}}_2^m} \leq \tilde{C}_A \left( 2^{-k}+C_\rho\norms{\vs - \vt}_2^m\right).
	\end{equation}
    Associated to an AtI there are \emph{wavelets}
    \begin{equation}\label{eq:Sb}
    	\begin{aligned}
    		\wcinn{k}: \R^m \times \R^m &\to \R\;.\\
    		(\vs,\vt) &\mapsto 2^{-k/2} (\scinn{k}(\vs,\vt)- \scinn{k-1}(\vs,\vt)).
    	\end{aligned}
    \end{equation}
    The name wavelets is attribute to the fact that they are constructed by scaling and differences of shifts but it is not an actual orthonormal wavelet as in \autoref{de:wavelet}.
\end{definition}
Below we construct a frame out of an AtI generated wavelet family related to \RQNN{}s defined in \autoref{eq:rqnn}:

\begin{definition}[\RQNN wavelets] \label{de:cf}
	Let $r > 0$ be a fixed constant and let $\sigma$ be a discriminatory function (see \autoref{de:disc_function})  satisfying
	\begin{equation*}
		\abs{\int_{\R^m}\sigma(r^2-\norm{\vs}_2^2)d\vs} < \infty\;.
	\end{equation*}
	Then we define
	%the normalized radial quadratic neural networks (\RQNN's), which are linear (note that in contrast to the general definition $\kappa_j=r^2$ and $w_{j,m+1}=1$)
	\begin{equation}\label{eq:varphi}
		\begin{aligned}
			\varphi: \R^m &\to \R\,,\\
			\vs &\mapsto C_m \sigma(r^2 - \norm{\vs}_2^2),
		\end{aligned}
	\end{equation}
	where $C_m$ is a normalizing constant such that $\int_{\R^m}\varphi(\vs)d\vs=1$. With $\varphi$ we define for all $k \in \ZN$ the \emph{circular scaling functions}
		\begin{equation}\label{eq:S}
			\begin{aligned}
				\scinn{k}: \R^m \times \R^m &\to \R\\
				(\vs,\vt) &\mapsto 2^k \varphi(2^{k/m}(\vs-\vt))
			\end{aligned}
		\end{equation}
	If $k \in \ZN$ and $\vt \in \R^m$ are considered parameters, we write
	\begin{equation}\label{eq:Sa}
		\begin{aligned}
			\vs \to \scinn{k,\vt}(\vs) := \scinn{k}(\vs,\vt) \text{ and }
			\vs \to \wcinn{k,\vt}(\vs) := \wcinn{k}(\vs,\vt)\;.
		\end{aligned}
	\end{equation}
	We consider the following countable subset of \RQNN{}s (as defined in \autoref{eq:rqnn}), called the \RQNN-\emph{scaling functions}\index{scaling function!\RQNN}
	\begin{equation} \label{eq:fpcinnd}
		\mathbb{S} := \set{S_{k,\vt}: k \in \ZN \text{ and } \vt \in 2^{-k/m}\ZN^m}\,,
	\end{equation}
	and the associated wavelet space, which we call \emph{\RQNN-wavelets}\index{wavelet!\RQNN} (see \autoref{eq:Sb})  	
	\begin{equation} \label{eq:fpwinnd}
		\mathbb{W} :=
		\set{\wcinn{k,\vt} := 2^{-k/2} \left(\scinn{k,\vt}-\scinn{k-1,\vt}\right) :
			k \in \ZN\,, \vt \in 2^{-k/m}\ZN^m}.
	\end{equation}
\end{definition}

The approximation to the identity (AtI) in \autoref{def:wavelet} is the basis
to derive quantitative $\mathcal{L}^1$-estimates for arbitrary functions in $\mathcal{L}^1(\R^m)$ with \RQNN-wavelets:
	
\begin{theorem}[$\mathcal{L}^1$-convergence of \RQNN-wavelets] \label{thm:ConvR}
	Let $\sigma:\R \to \R$ be monotonically increasing, twice differentiable, and satisfy for $i=0,1,2$
	\begin{equation} \label{eq:sigma0}
		\abs{\sigma^i (r^2-t^2)} \leq C_\sigma (1+\abs{t}^m)^{-1-{(2i+1)/m}} \text{ for all } t \in \R,
	\end{equation}
	where $\sigma^i$ denotes the $i$-th derivative of $\sigma$. Then the \RQNN-wavelets, $\mathbb{W}$, defined in \autoref{eq:fpwinnd}, is a frame (see for instance \cite{Chr16}).
			
	Moreover, for every function $\x \in \mathcal{L}^1(\R^m)$ (see \autoref{de:ell1space}) and every
	$\noc \in \N$, there exists a function $\x_\noc$, which is spanned by $\noc$ elements of $\mathbb{W}$ satisfying
	\begin{equation} \label{eq:app_error}
		\norm{\x - \x_\noc}_{L^2} \leq \norm{\x}_{\mathcal{L}^1} (\noc+1)^{-1/2}.
	\end{equation}
\end{theorem}
The proof is a consequence of the result from \cite{ShaCloCoi18}.

Using \autoref{thm:ConvR} we are able to formulate the main result of this subsection: For this purpose we write down the scaling functions and wavelets related to \RQNN{}s with fixed $\kappa_j = r^2 = 1$ in \autoref{eq:rqnn}, for the sake of simplicity of presentation.
Then,
\begin{itemize}
	\item for $k \in \ZN$ and $\vec{k} \in \ZN^m$ the elements of $\mathbb{S}$, defined in \autoref{eq:fpcinnd} are the functions
	\begin{equation*}
		\vs \mapsto S_{k,2^{-k/m} \vec{k}}(\vs) = 2^k C_m \sigma\left(1 - 2^{k/m}\norms{\vs - 2^{-k/m}\vec{k}}_2^2\right)\;.
	\end{equation*}
    \item The wavelets functions of $\mathbb{W}$ are given by
    \begin{equation*}
    	\begin{aligned}
    		\vs \mapsto \wcinn{k,2^{-k/m}\vec{k}}(\vs)
    		& = 2^{\frac{k}{2}}C_m \bigl( \sigma(1 -2^{2k/m}\norms{\vs-2^{-k/m}\vec{k}}_2^2) \bigr.\\
    		& \qquad \bigl.
    		-2^{-1} \sigma(1-2^{(2k-2)/m}\norms{\vs-2^{-k/m}\vec{k}}_2^2)
    		\bigr)\;.
    	\end{aligned}
    \end{equation*}
\end{itemize}
Only linear combinations of the above mentioned wavelet functions are required to approximate function in $\mathcal{L}^1$.

\begin{corollary}[$\mathcal{L}^1$-convergence of \RQNN{}s] \label{co:ConvR}
	Let the activation function $\sigma$ satisfying the assumptions of \autoref{thm:ConvR}.
	Then, for every function $\x \in \mathcal{L}^1(\R^m)$ (see \autoref{de:ell1space}) and every
	$\noc \in \N$, there exists an index set
	\begin{equation*}
		I_\noc \subset \set{(k,\vec{k}) \in \ZN \times \ZN^m}
	\end{equation*}
    of cardinality $\noc$ and coefficients $(\alpha_{\vec{j}})_{\vec{j} \in I_\noc}$ such that
	\begin{equation} \label{eq:rqnnpsi}
		\begin{aligned}
			\Psi[\vp](\vs) := & \sum_{(k,\vec{k}) \in I_\noc} \alpha_{(k,\vec{k})} \sigma(1-p_{k}\norms{\vs-\theta_{k,\vec{k}}}_2^2)\\
			& \quad 
			-\frac{1}{2} \sum_{(k,\vec{k}) \in I_\noc} \alpha_{(k,\vec{k})} \sigma(1-q_{k}\norms{\vs-\theta_{k,\vec{k}}}_2^2)\,,
		\end{aligned}
	\end{equation}
	with
	\begin{equation} \label{eq:pp}
		\theta_{k,\vec{k}} = 2^{-k/m}\vec{k},\; p_{k} = 2^{2k/m},\; q_{k} = 2^{(2k-2)/m}
	\end{equation}
	satisfies
	\begin{equation}
		\label{eq:app_error II}
		\norm{\x - \Psi[\vp]}_{L^2} \leq \norm{\x}_{\mathcal{L}^1} (\noc+1)^{-1/2}\;.
	\end{equation}
\end{corollary}
For the proof of \autoref{co:ConvR} see \cite{FriSchShi24}.
\begin{remark}
	\begin{itemize}
		\item We point out that the parameter $\vp \in \R^{1+m}$ used to approximate $\x$ has less freedom than used in the universal approximation theorem, where, for instance for \ALNN's (see \autoref{eq:p}) $\vp \in (\R \times \R \times \R^m)_{j=1}^\noc$.
        \item Note that the result is actually sharp. \autoref{eq:app_error} and \autoref{eq:app_error II} are consistent with the result from \cite[Theorem 2.1]{BarCohDahDev08}.
	\end{itemize}
\end{remark}
In the following we define \RQNN-wavelets with a bounded scale $\ttm$:
\begin{definition} The \emph{scale $\ttm$} \RQNN-wavelet basis functions are elements of the set:\index{wavelets!scale $\ttm$ basis functions}
	\begin{equation} \label{eq:wfm}
		\mathbb{W}_\ttm := %\text{span}
		\set{ \vs \to \psi_{k,\vt}(\vs): k \in \set{-\ttm,\ldots,\ttm}, \vt \in 2^{-k/m}\ZN^m }\;.
	\end{equation}
    Associated is the finite dimensional space of \RQNN{}s
    \begin{equation} \label{eq:xfm}
    	\X_\ttm := \overline{\spann (\mathbb{W}_\ttm)}\;.
    	\end{equation}
    Moreover, $P_\ttm:=P_{\X_\ttm}$ denotes the $L^2$-orthogonal projector onto $\X_\ttm$.
    We call $\X_\ttm$ the space of \RQNN-wavelets of scale $\ttm$.\index{wavelets!scale $\ttm$!$\X_\ttm$}
\end{definition}
In the following we define functions, which can be approximated by functions in $\X_\ttm$: 
\begin{definition} \label{de:representer}
Let $\x \in \mathcal{L}^1$.
%\begin{itemize}
	%\item 
	We call $\x$ \emph{countably representable}\index{function!countably representable} if for every 
	$\noc \in \N$ there exists $\ttm$ such that 
	\begin{equation}\label{eq:lfunm}
		\x_\noc = \sum_{i=1}^\noc c_{{\tt g}_i}^\noc {\tt g}_i^\noc \in \X_\ttm
	\end{equation}
	and that $$\norm{\x_\noc - \x}_{L^2}\to 0\;.$$ 
	Without loss of generality we assume that the sequence is monotonically 
	\begin{equation} \label{eq:monom}
		\norm{\x_1 - \x}_{L^2} \geq \norm{\x_2-\x}_{L^2} \geq \ldots
	\end{equation}
	Moreover, we assume that 
	\begin{equation} \label{eq:monoII}
		\X_\ttm \subseteq \X_{\ttm +1}\;.
	\end{equation}
	%\item We call $\x$ \emph{exactly countably representable}\index{function!exact countably representable} if there exists 
	%$\ttm, \noc \in \N$ such that \begin{equation}\label{eq:lfunmdag}
		%\x = \sum_{i=1}^\noc c_{{\tt g}_i}^\noc {\tt g}_i^\noc \in \X_\ttm\;.
	%\end{equation}
%\end{itemize}

\end{definition}
Note that this result holds for bounded open sets $\Omega$ and $\R^m$. Therefore we do not specify the domain in the spaces and norms. %\commentO{Maybe write always $\Omega$}

We then define a function, which maps $\noc$ onto the smallest ${\tt m} \in \N$ such that
\begin{equation} \label{eq:Nfuncresult}
	\x_\noc \in \X_\ttm\;.
\end{equation}
\begin{definition} \label{de:mono} Let $\x \in \mathcal{L}^1$ be representable and let $\x_\noc$ as in \autoref{eq:monom} approximating $\x$ with respect to the $L^2$-norm. The monotonic increasing envelope of the function $\noc \to \ttm$, is denoted by %\commentO{Make an example of a monotonic envelope}
\begin{equation} \label{eq:Nfunc}
	\begin{aligned}
		\ttm: \N &\to \N\;.\\
		\noc &\mapsto \ttm(\noc)
	\end{aligned}
\end{equation}%
Because of the monotonicity this function can be inverted, %\commentO{Make an example of the inverse}, 
and it is given by 
\begin{equation} \label{eq:Nfunci}
	\begin{aligned}
		\noc: \N &\to \N\,,\\
		\ttm &\mapsto \noc(\ttm)
	\end{aligned}
\end{equation}
where
\begin{equation*}
	\noc(\ttm) := \max \set{\hat{\noc}: \ttm(\hat{\noc}) \leq \ttm}
\end{equation*}
\end{definition}

These mappings are central for proving convergence rates results of regularization methods. It is used in particular in \autoref{ex:c_recon_cont} and \autoref{ex:c_recon_cont_II}.

\section{Deep neural networks}
\label{sec:deep}
In the previous sections we considered approximations of functions with shallow neural networks. In the following we highlight benefits of deep neural network approximations in certain applications. For this purpose we first use the Heaviside-function as an activation function, which will be approximated by a sigmoid function for numerical purposes at a later stage.

\subsection{A Heaviside-network for polygons}

We consider a Heaviside-network as defined in \autoref{eq:heaviside} in $\R^m$. This means that each Heaviside-activation function is the characteristic function of a halfspace (see \autoref{fig:motiv_eg_2}). The boundaries of a convex polyhedral domain, such as the triangle in \autoref{fig:motiv_eg_2}, are defined by the zero-level sets of the affine functions ${\tt a}_j: \vs \mapsto \vw_j^T \vs + \theta_j$ for $j=1, \ldots, n_1$ where $n_1$ denotes the number of neurons in the first hidden layer. The signs of the weights $\vw_j$ and biases $\theta_j$ are chosen such that the intersection of the positive half-spaces matches the polyhedron. Therefore, by appropriately choosing the network's weights and activation thresholds, we obtain the characteristic function of the polyhedron. In the following, we present a theorem to illustrate how this characteristic function of the polyhedron can be constructed using a single-hidden-layer neural network.

%Now let us consider a one layer Heaviside-network
%	\begin{equation} \label{eq:neto}
%	\x_{k} (\vs)  := \x[\vec{a}_k,{\bf W}_k,\vec{b}_k] := \sum_{j=1}^{n_1} (\vec{a}_k)_j \sigma ((\vw_{k})_j^T\vs + (\vec{\theta}_k)_j)\,,
%\end{equation}
%where $(\vec{a}_k)_j, (\vec{b}_k)_j \in \R$ denoting the $j$-th elements of the vectors $\vec{a}_k, \vec{b}_k \in \R^{n_1}$, respectively, and  $(\vw_{k})_j \in \R^2$ is the $j$-th column of the matrix ${\bf W}_k := \left((\vw_{k})_1 \ (\vw_{k})_2 \cdots (\vec{w}_{k})_{n_1}\right) \in \R^{2 \times n_1}$.

\begin{theorem} \label{th:CV_connex} 	
	Let $\set{\mathcal{P}_k: k=1,\ldots, \mathcal{S}} \subseteq \Omega$ be a collection of non empty, open, connected polygons such that the boundaries of any two polygons intersect in only finitely many points, and the boundary of each polygon is a simple closed curve. For every $k = 1,\ldots, m$, define
	$\mathcal{P}_k^1 := \mathcal{P}_k$ and $\mathcal{P}_k^{-1} := \Omega\backslash\mathcal{P}_k$, and for every $\mult \in \set{-1,1}^\mathcal{S} $, define \footnote{the sets $P_\mult$ can be empty or disconnected.}
	\begin{equation} \label{eq:pmult}
		P_\mult = P_{(\iota_1,\iota_2,\ldots,\iota_\mathcal{S} )}  := \bigcap_{k=1}^\mathcal{S}  \mathcal{P}_k^{\iota_k}.
	\end{equation}
	Then, there exists a family of one layer networks $\set{\network_k: k=1,\ldots,\mathcal{S} } \subseteq \nnset$
		\begin{equation} \label{eq:neto}
		\x_{k} (\vs)  := \x[\vec{a}_k, {\bf W}_k,\vec{\theta}_k] := \sum_{j=1}^{n_1} (\vec{a}_k)_j \sigma ((\vw_{k})_j^T\vs + (\vec{\theta}_k)_j)\,,
	\end{equation}
	 such that for every $\mult \in \{-1,1\}^\mathcal{S} $,
	\begin{equation}\label{eq:GammaP}
		P_\mult = \bigcap_{k=1}^\mathcal{S}  \set{\vs \in \Omega : \mult_k \x_{k} (\vs) > 0}\;.
	\end{equation}
	Moreover, if $\vs \in \partial P_\mult$, then there exists some $k \in \{1,\ldots, \mathcal{S} \}$ such that $\vs \in \{\x_{k} = 0\}$, and for every $k \in \{1,\ldots, \mathcal{S} \}$ and $\vs \in \{\x_{k} = 0\}$, there exists
	some $\mult \in \{-1,1\}^\mathcal{S}$ such that $\vx \in \partial P_\mult$.
\end{theorem}

\begin{example}
\autoref{fig:motiv_eg_2} traces the step-by-step computation of a two-hidden-layer network segmenting a triangle. First, three first-layer neurons create binary half-planes, whose overlaps partition the domain into seven subregions with unique activation codes. The second layer then averages these triplets per subregion, yielding scores of $1/3$, $2/3$, or $1$. Finally, an output threshold isolates the target triangle by mapping the maximum score to $1$ and all other regions to $0$. This process illustrates how hierarchical architectures systematically compose complex, bounded shapes from simple linear cuts.
\end{example}

\begin{figure*}[h]
	\centering   \includegraphics[width=0.8\linewidth]{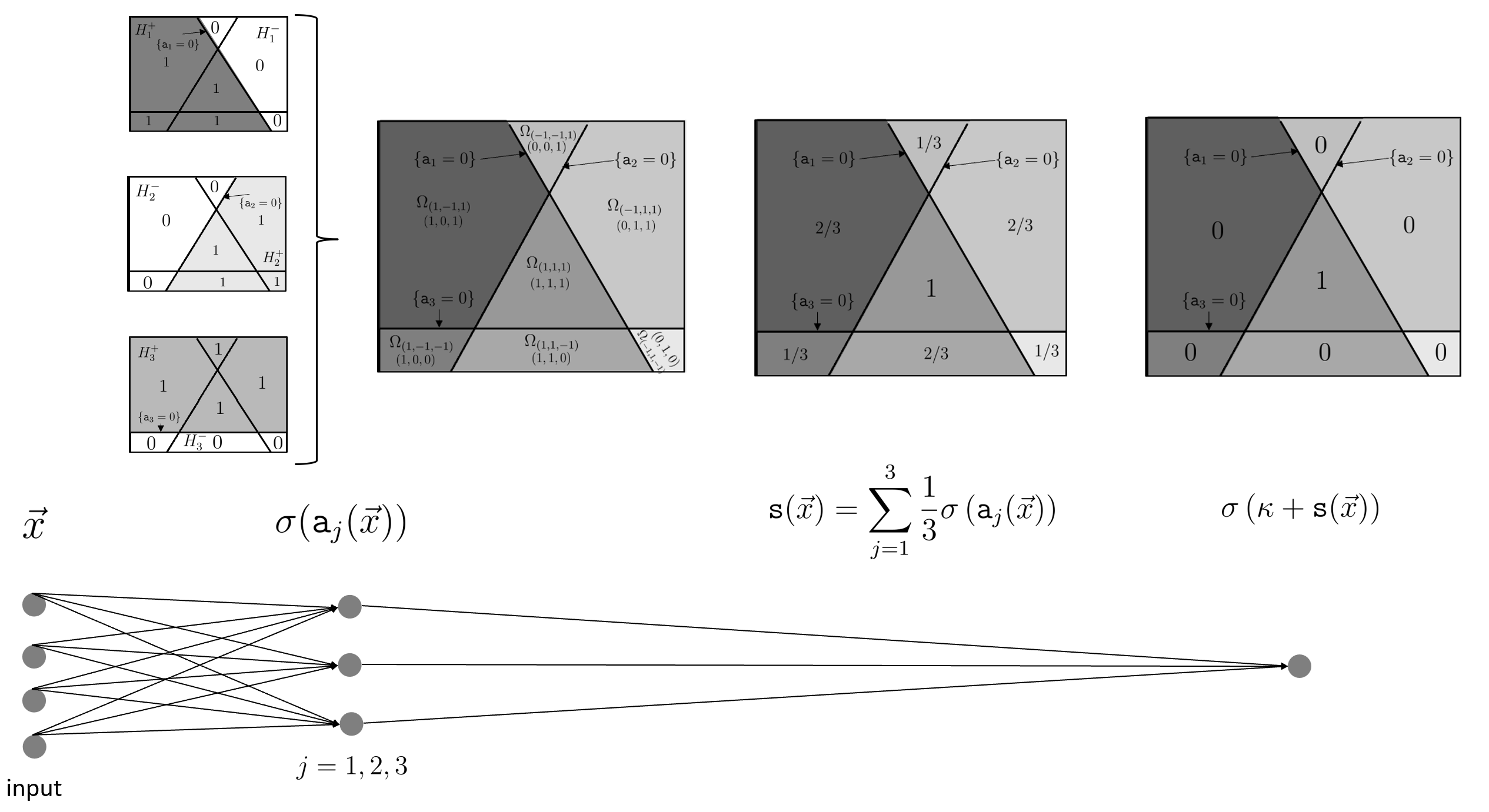}
	\caption{\textbf{(1)} Binary partitions from three neurons in the first layer. \textbf{(2)} Their activation regions.
		\textbf{(3)} Averaged triplets per subregion.
		\textbf{(4)} Final triangle detection via thresholding.}\label{fig:motiv_eg_2}
\end{figure*}

%\textbf{(1)} Each line $H_j^0 = \{{\tt a}_j = 0\}$, for $j=1,2,3$, divides the image domain into 2 half planes $H_j^1$ and $H_j^{-1}$. The half planes containing the triangle $T$ (shown in gray) attain the value $1$ and the complementary half planes (shown in white) attain the value $0$. The three images each depict a binary partition (gray and white regions), corresponding to the activation of a single neuron in the first layer of a neural network. Taken together, these binary regions represent the activations of three neurons in the first layer, represented by $\sigma({\tt a}_j)$, $j = 1, 2, 3$, where ${\tt a}_j > 0$ inside the gray region and ${\tt a}_j(\vx) < 0$ outside the gray region, as in the standard level set formulation.
%\textbf{(2)} Seven subregions $\Omega_{\iota}$ for $\iota \in (-1,1)^3$ and their triplets of activation values $(\sigma({\tt a}_1), \sigma({\tt a}_2), \sigma({\tt a}_3))$.
%\textbf{(3)} The average of these activations triplets are taken over each subregion $\Omega_\iota$, resulting in values $1/3,2/3,1$. This corresponds to the linear combination $\network_c = \sum_{j=1}^3 \frac{1}{3}\sigma({\tt a}_j)$.
%\textbf{(4)} The function $\sigma(\kappa + \network_c)$ attains the value $1$ on the triangle $\Omega_{(1,1,1)}$. On the complement of $\Omega_{(1,1,1)}$, $\sigma(\kappa + \network_c)$ is zero almost everywhere.

\subsection{Level set methods} \label{se:lsm}

 Level set methods, originally introduced by Osher \& Sethian \cite{OshSet88}, have proven a powerful approach to solve image segmentation tasks, because the implicit representation of segments allows to identify objects of complex topologies, such as, for instance, segments consisting of multiple connected domains. Building upon this foundational framework, Chung \& Vese \cite{ChuVes05,ChuVes09} extended the Chan-Vese-model to a multiphase level set framework capable of simultaneously segmenting multiple objects. Brox \& Weickert \cite{BroWei04} proposed a variational minimization strategy within a level set framework that robustly optimizes both the number of desired segments. Kang et al \cite{KanSanYip11,KanMar14,KanShaSte14} have complemented level set methods which allow for automatic determination of the number of objects.

There exist two categories of level set algorithms:
\begin{itemize}
	\item{\emph{Unparametrized}} level set algorithms consider the input images and the level set functions, from which segments and classes are detected, pixel wise. All functions are represented as a vector of pixel values.
	\item{\emph{Parametrized}} algorithms (see, for instance, \cite{YanFucJueSch06,FeiFucJueSchYan08,AghKilMil11,OzsKilStuSaiMil25}) use ansatz-functions (such as splines or finite elements) to represent the level sets and input images. This means, in particular, that the segments and classes are specified \emph{double implicitly}: as the zero level sets of the representing functions (first), which in turn are represented via the parametrization (second).
\end{itemize}

The different concepts can be compared with \emph{finite difference} and \emph{finite element} methods for solving differential equations: The latter represents the solution of the differential equation implicitly by finite element parametrizations, whereas in finite difference methods the function values of the solution are given as a vector associated with solution values on the grid points. In fact, \emph{hybrid} level set algorithms combining the advantages of both parametrized and unparametrized strategies have also been developed (see \cite{FucJueSchYan09b}).

%\begin{figure}[H]
%	\centering
%	\includegraphics[width=2.5in]{levelset/reladiagram.png}
%	\caption{Commutative diagram illustrating the relationships among the classical, parametrized, and smooth Chan-Vese-models. The classical (unparametrized) model can be approximated by the parametrized non smooth model, which in turn can be approximated by the smooth variant. Moreover, the smooth parametrized model can also approximate the smooth unparametrized model, which itself provides a smooth approximation of the classical model. }\label{fig:rela}
%\end{figure}

\section{Chan-Vese Regularization}

Chan \& Vese \cite{ChaVes01,VesCha02} proposed a level set formulation of the functional defined in \eqref{eq:CV_og} to minimize it numerically. The formulation utilizes that a piecewise constant function $\multiphase$ can be represented as a \emph{multiphase level set function}. We give the definition for multiphase level set functions below.
\begin{definition}[\textbf{Multiphase level set function}] \label{lsf}
	Let $\mathcal{S} \in \N$. For every $k=1,\ldots,\mathcal{S}$, let $\levelfun_k \in C^0(\overline{\Omega})$ be a continuous \emph{level set function}, which, in particular, implies that
	\begin{equation}\label{eq:GammaI0}
		\begin{aligned}
			\mathcal{L}_k^1 = \set{\vs \in \Omega : \levelfun_k(\vs) > 0},
			\mathcal{L}_k^{-1} = \set{\vs \in \Omega : \levelfun_k(\vs) < 0},\\
			\partial \mathcal{L}_k^1 = \mathcal{L}_k^0 =  \{\vs \in \Omega : \levelfun_k(\vs) = 0\}, \text{ with } \mathcal{H}^1(\partial \mathcal{L}_k^1) < \infty. \nonumber
		\end{aligned}
	\end{equation}
\end{definition}

\begin{definition}[\textbf{Multiphase level set function of degree $2^\mathcal{S}$}] \label{lsfm}
	Let $\mathcal{S} \in \N$. A multiphase level set function \emph{of degree }$2^\mathcal{S}$ is a function of the form
	\begin{equation}\label{eq:multiphaselevelset}
		\begin{aligned}
			\multiphase_\levelfun(\vs) = \hspace*{-0.03\textwidth} \sum_{\mult = (\mult_1,\ldots,\mult_\mathcal{S}) \in \{-1,1\}^\mathcal{S}} \hspace*{-0.02\textwidth} c_\mult \prod_{k=1}^\mathcal{S} \sigma(\mult_k \levelfun_k(\vs))\;.		
		\end{aligned}
	\end{equation}
\end{definition}

%\subsection{Segmentation}\label{segm}

The Chan-Vese model is designed to approximate a given function $\f:\Omega \to \R$
by a piecewise constant
\begin{equation}\label{eq:CV_piecewise}
	\multiphase(\vs) = \sum_{\mult \in \{-1,1\}^\mathcal{S}} c_\mult \chi_{L_\mult}(\vs)\;
\end{equation}
which minimizes the Chan-Vese-functional
\begin{equation}\label{eq:CV_og}
	\begin{aligned}
		\mathrm{CV}_{\tt f}(\{(c_\mult,L_\mult)\}_{\mult \in\{-1,1\}^\mathcal{S}}) = &\sum_{\mult \in \{-1,1\}^\mathcal{S}}\int_{L_\mult} \left( \multiphase - \f \right)^2 \mathrm{d}\vs \\ + \mu \hspace*{-0.01\textwidth} \sum_{\mult \in \{-1,1\}^\mathcal{S}} \mathcal{H}^1(\partial L_\mult) &+ \nu \mathcal{H}^2\left(\bigcup_{\mult \in \{-1,1\}^\mathcal{S}} L_\mult\right),
	\end{aligned}
\end{equation}
where $\mathcal{H}^n$, for $n=1,2$, denotes the $n$-dimensional Hausdorff measure.

We consider the Vese-Chan segmentation of multiple objects, where the level set functions $\set{\levelfun_k : k = 1,\ldots,\mathcal{S}}$ (see \eqref{eq:multiphaselevelset}) are parametrized by one layer  Heaviside-networks. Since polygons approximate arbitrary sets, one layer  Heaviside-networks are an adequate ansatz for numerical minimization of the Chan-Vese-level set functional.

\begin{definition}[\textbf{Multiphase function for polygons}] \label{lsfm1}
	Let $m \in \N$. A multiphase Heaviside-function for polygons is a function of the form
	\begin{equation} \label{eq:para}
		\vs \in \Omega \subseteq \R^2 \mapsto \multiphase_\network(\vs):=\sum_{\mult \in  \{-1,1\}^\mathcal{S}} c_\mult \prod_{k=1}^\mathcal{S} \sigma(\mult_k\x_k (\vs)).\;
	\end{equation}
\end{definition}

By definition, the multiphase function for polygons $\multiphase_\network$, defined in \eqref{eq:para}, is piecewise constant, and its corresponding superlevel sets
\begin{align}\label{eq:GammaIs}
	S_\mult = S_{(\mult_1,\ldots,\mult_\mathcal{S})} = \bigcap_{k=1}^\mathcal{S} \set{\vs \in \Omega : \mult_k \x_k (\vs) > 0} = \bigcap_{k=1}^\mathcal{S} \mathcal{S}_k^{\mult_k}
\end{align}
are polygons. The main question of interest is the converse direction: given polygonal sets $\{P_\mult :\mult \in \{-1,1\}^\mathcal{S} \}$, we seek to construct a family of \emph{one layer} networks $\{\x_k : k=1,\ldots,\mathcal{S}\}$, whose zero level sets coincide with the boundaries of $\{P_\mult :\mult \in \{-1,1\}^\mathcal{S} \}$ . Consequently, the induced multiphase function $\multiphase_\network$ is piecewise constant on this polygonal partition (compare with \eqref{eq:CV_piecewise}), that is,
\begin{equation} \label{eq:sn}
	\multiphase_\network(\vs) =  \sum_{\mult \in \{-1,1\}^\mathcal{S}} c_\mult \chi_{P_\mult}(\vx) \text{ for almost all } \vs \in \Omega\,,
\end{equation}
and $P_\mult = S_\mult$ for all $\mult \in \{-1,1\}^\mathcal{S}$, where $S_\mult$ is as defined in \eqref{eq:GammaIs}.
This construction is formalized in \autoref{th:CV_connex}.

As a consequence of \autoref{th:CV_connex}, we naturally use one layer Heaviside-neural networks $\{\x_k : k=1,\ldots,\mathcal{S}\}$ to parametrize the level set functions $\{\levelfun_k : k=1,\ldots,\mathcal{S}\}$ of the objects to be segmented, and define the parametrized Chan-Vese-functional as follows:
\begin{definition}
The one layer  neural network parametrized Chan-Vese-functional is defined by
\begin{equation}\label{eq:paraChanVese_l}
	\begin{aligned}
		&\mathrm{CV}^{\mathrm{LS}}_{\tt f}((c_\mult)_\mult,(\vec{a}_k,{\bf W}_k,\vec{\theta}_k)_{k=1}^\mathcal{S}) \\
		:= &\sum_{\mult \in \{-1,1\}^\mathcal{S}}\int_{\Omega} \left( c_\mult - \f(\vs) \right)^2\prod_{k=1}^\mathcal{S}\sigma(\mult_k\x_k(\vs)) \ \mathrm{d}\vs \\
		&+ \mu \sum_{k=1}^\mathcal{S} \mathcal{H}^1(\{\x_k=0\}) + \nu \mathcal{H}^2\left(\bigcup_{k=1}^\mathcal{S} \{\x_k > 0\}\right).
	\end{aligned}
\end{equation}
\end{definition}
The one layer  neural network parametrized Chan-Vese-functional can be used for \emph{multi-polygonal} segmentation: Instead of minimizing Chan-Vese functional over arbitrary level set functions to determine the segments of the function ${\tt f}$, we constrain the optimization of the functional to piecewise constant functions on polygonal partitions of the domain.

Moreover, if we consider two layer neural network, the above result can be extended. We first give the definition of two layer Heaviside-neural network:

\begin{definition}
A \emph{two layer Heaviside-neural network with $n_1$ neurons in the first layer and $n_2$ neurons in the second layer} is defined as
\begin{equation}\label{eq:t}
	\begin{aligned}
		\vs \in \R^2 &\mapsto \nettwo(\vs) : = \nettwo[{\bf A},\vec{c}, {\bf D}, \vec{\theta}, {\bf W}](\vs)\\
		&:= \sum_{i=1}^{n_2} c_i \sigma\left( \sum_{j=1}^{n_1} a_{ij} (\sigma \circ {\tt a}_j)(\vs) + d_{ij} \right)\;.
	\end{aligned}
\end{equation}
${\bf A} = (a_{ij}) \in \R^{n_2 \times n_1}$, $\vec{c} =(c_i) \in \R^{n_2}$, ${\bf D} = (d_{ij}) \in \R^{n_2 \times n_1}$ and $\vec{\theta}=(\theta_j)\in \R^{n_1}$, ${\bf W} = (\vw_j) \in \R^{2 \times n_1}$ (note that affine linear function ${\tt a}_j$ depends on ${\bf W}$ and $\vec{b}$) are the parametrization of ${\tt t}$.
\end{definition}

With this definition, we show that every piecewise constant functions on polygonal sets $\multiphase_\network$ as defined in \eqref{eq:para} coincides pointwise with a two layer Heaviside-neural network $\nettwo \in \nntwo$:

\begin{theorem} \label{thrm:lsnn}
	Let $m \in \N$ be fixed. Then, every multiphase Heaviside-function for polygons $\multiphase_\network$, defined in \eqref{eq:para}, can be represented as a two layer Heaviside-network
	\begin{equation} \label{eq:lsnn}
		\begin{aligned}
			\vs \in \Omega \mapsto \multiphase_\network(\vs) =
			\hspace{-0.3cm}\sum_{\mult \in \{-1,1\}^\mathcal{S}} \hspace{-0.3cm} c_\mult \sigma\left(\kappa + \sum_{j=1}^{n_1}\frac{1}{n_1} \sigma(\mult_j \affine_j(\vs))\right) \in \nntwo,
		\end{aligned}
	\end{equation}
	where $\kappa \in \R$ is a constant.
\end{theorem}

This chapter established the theoretical foundations for implementing the Chan-Vese level set model using neural networks. The Chan–Vese model is implemented via a level set formulation, where the level set function is crucially represented by a neural network — its output defines the implicit surface and its zero-level set encodes the segmentation boundary. This neural network parameterization replaces traditional grid-based evolution with a learnable continuous function, offering a flexible and mesh-free framework for geometric segmentation. A more comprehensive theoretical analysis, alongside concrete numerical algorithms and diverse empirical examples, are presented in \cite{SchShiVu25}.

%The classical (unparametrized) model can be approximated by the parametrized non smooth model, which in turn can be approximated by the smooth variant. Moreover, the smooth parametrized model can also approximate the smooth unparametrized model, which itself provides a smooth approximation of the classical model.

\section{Open research questions}
The open research questions concern proving Tauber-Wiener theorems in Sobolev-spaces rather than in $L^1$ and $L^2$.

\begin{opq} \label{opq:tw1} The idea used in \autoref{re:tw} to approximate a function $\x_\noc$ is applied to the derivative
	\begin{equation*}
		\mathcal{F}[\x_\noc'](\omega) = \sum_{k=1}^\noc \alpha_k \mathcal{F}[\sigma'(\cdot+\theta_k)](\omega) =
		\mathcal{F}[\sigma'](\omega) \sum_{k=1}^\noc \alpha_k \exp^{\i \omega \theta_k}\;.
	\end{equation*}
	If in addition also $\mathcal{F}[\sigma']$ vanishes at most on a set of measure $0$, we can divide by it and get
	\begin{equation*}
		\frac{\mathcal{F}[\x_\noc'](\omega)}{\mathcal{F}[\sigma'](\omega)} = \sum_{k=1}^\noc \alpha_k \exp^{\i \omega \theta_k} \text{ almost everywhere}\;.
	\end{equation*}
	If $\sigma$ has compact support, which was often assumed above (see \autoref{fig:alles}), then all derivatives have compact support, and all Fourier transforms of derivatives of $\sigma$ are analytic functions by the Paley-Wiener theorem. In other words the sets of zeros of all derivatives have measure zero. Consequently, the Tauber-Wiener \autoref{Wiener_Tauberian} guarantees that also $\x_\noc'$ can be approximated with \autoref{eq:wiener_property} when $\sigma$ is replaced by $\sigma'$. Applying \autoref{Wiener_Tauberian} therefore shows that an arbitrary function $\x_\noc$ can be approximated with respect to the Sobolev-norm $W^{1,2}$ (see \autoref{sob_def}). This is an important observation for regularization theory as discussed in the introduction \autoref{se:chall}.
	
	The research question however concerns the characterization of \emph{Tauber-Wiener functions of higher order} to generalize the Tauber-Wiener theorem to Sobolev-spaces:
    We note that for a higher order Tauber-Wiener function,
		\begin{equation*}
			\mathcal{F}[\sigma^{k}](\omega) = (\i \omega)^k \mathcal{F}[\sigma](\omega)
		\end{equation*}
		in an appropriate sense (distributional, almost everywhere,...).
		
		According to \autoref{Wiener_Tauberian} we can approximate
		$\x_\noc$ with respect to Sobolev-norm $W^{k,2}(\R)$ but not with respect to the $W^{k,1}$-norm if $k \geq 1$, because the Fourier transform of a derivative has always a zero at $\omega=0$. Note that in order to approximate function in $L^1$ it is required in \autoref{Wiener_Tauberian} that the zeros of the Fourier transform of $\sigma$ are empty.
		
		The two research questions are:
		\begin{enumerate}
			\item Can a function $\x_\noc \in W^{k,1}(\R)$ be approximated with a representation \autoref{eq:wiener_property}.
			\item Can a Tauber-Wiener functions of higher order approximate a function in $W^{\xi,1}$ ($\xi$ must not be equal to $k$) and still have some zeros at the real line?
		\end{enumerate}
\end{opq}
%
%\begin{opq} \label{eq:opg_barron} We conjecture that the \emph{multiplicative Barron energy} defined in \autoref{eq:barron_norm_TW} is equivalent to the $\mathcal{L}^1$-norm defined in \autoref{de:ell1space}.
%\end{opq}
%

\begin{opq}
	The curious observation is that in the literature convergence rates of approximating arbitrary functions with neural networks are obtained already from a countable family of neural network function. This is actually also the basis of discrete wavelet theory. Interestingly the Tauber-Wiener theorems allow for approximation of univariate functions with uncountable basis functions while wavelets require countable ones. So the question is what are the actually required parameters of scaling and translation? Note that \cite{TiaHon95} already proved that over a family of functions to approximate the scaling and translation parameters can be chosen constant.
	This implies a significant reduction of parameters for the training (see \autoref{ch:learning}). It would be ideal to know, depending on $\sigma$ and classes of functions to approximate, the minimum number of neural network parameters.
\end{opq}
%
%\begin{opq}The function $\noc(\ttm)$ in \autoref{eq:Nfunci} is well-defined but is a theoretical concept. The open research question is whether there can be found upper bounds for this number by constraining the investigated functions.
%\end{opq}

\section{Further reading}
In the field of machine learning, there have been various attempts to model neural networks modeled based on the \emph{Kolmogorov–Arnold representation} (see \cite{Kol36,LinUnb93,PolPol21})\index{Kolmogorov–Arnold representation} and some extensions \cite{Spr65,Lor62}), which replaces the universal approximation theorems \autoref{sec:uat} in this context. The \emph{Kolmogorov-Arnold} representation states that every continuous function $\x:[0,1]^m \to \R$
can be represented as a superposition of continuous univariate functions: That is
\begin{equation*}
	\x (\vs) = \sum_{q=1}^{2m+1} \Phi_q\left( \sum_{p=1}^{m}\phi_{q, p}(s_p)\right)
\end{equation*}
where $\Phi_q : \R \to \R$ and $\phi_{q, p}: [0, 1] \to \R$ are continuous univariate functions. Mathematically this is extremely exciting because the Kolmogorov-Arnold theorem is related to \emph{Hilbert's thirteenth problem}\index{Hilbert's 13th problem} \cite{Vit04}. However, \cite{GirPog89} state that it is irrelevant for neural network computations, because $\Phi_q$ depends on $\x$. Nevertheless, the connection between neural networks and the Kolmogorov–Arnold representation has been explored in several studies, see or example \cite{LinUnb93,PolPol21,Afz25,CerCerKolMalPlo61,Arn63}. The research in this field is partly motivated to better understand neural networks (see \cite{Bis24}).

For more background on wavelets we refer to \cite{Wal04,Mey93b,Dau92,Mal09}: From the reference list there many more references can be found. Extensions of wavelets, such as curvelets \cite{CanDon99,MaPlo10} and shearlets (see \cite{Kut23}) can been considered for machine learning. A recent survey on wavelets and machine learning is \cite{DauDevDymFaiKov23}.

The relation between \ReLU-networks and splines has been discussed extensively in the literature (see \cite{BalBar18}). As explored by Lewandowski \cite{Lew25}, the decision boundaries formed by ReLU networks correspond to tropical hypersurfaces, which are unions of polyhedral segments (see \cite{AlfBibHamGaaGha23}). This correspondence enables the use of tropical geometry to study the shape and complexity of decision regions in classification tasks. Complementary results by Koutschan and collaborators demonstrate that piecewise linear functions computed by ReLU networks can be analyzed using methods from algebraic geometry and computer algebra \cite{KouMosPonSchi25}. In particular, Koutschan et al. show that every continuous piecewise linear function on $\R^n$ can be represented as a linear combination of maxima of at most $n+1$ affine-linear functions, establishing fundamental bounds on the expressive power of such networks. This algebraic perspective, combined with the tropical view of neural networks \cite{ZhaNaiLim18_report}, unifies spline theory, tropical geometry, and the polyhedral analysis of network decision regions. The Tauber-Wiener theorem has been generalized to quantum harmonic analysis (see \cite{Wer84,FulLueWer26}). 

The computation of wavelets coefficients can be realized with computer algebra software (see \cite{ChyPauSchSchoZim01}), which allows for generalizations of the wavelets construction (see \cite{PauSchScho03}).

\chapter[Supervised learning]{Supervised function, functional and operator learning} \label{ch:op_learning}
We recall the terminology from the introduction, \autoref{cha:intro}:
 We assume that $\x^{(\ell)}$ and $\y^{(\ell)} = \op{\x^{(\ell)}}$ are 
      \emph{noiseless}, meaning that
	  $\x^{(\ell)}$ solves \autoref{eq:op} with right hand side $\y^{(\ell)}$.
	  We recall that $$
	     \mathcal{S}_\no:= \set{(\x^{(\ell)},\y^{(\ell)}=\op{\x^{(\ell)}}): 
          \ell = 0,1,\ldots,\no} $$ (see \autoref{eq:expert_information}) is the 
      available expert information.
As in \autoref{eq:trainingsamplesx} and \autoref{eq:trainingsamplesy}, we call the sets
\begin{equation} \label{eq:sampled1}
\begin{aligned}
	\mathcal{T}_\no :=\set{(\x^{(\ell)},\z_\y^{(\ell)}): \ell = 0, 1, \ldots, \no} \text{ and }\\
	\mathcal{T}_\no :=\set{(\z_\x^{(\ell)},\y^{(\ell)}): \ell = 0, 1, \ldots, \no}
\end{aligned}
\end{equation}
\emph{labeled expert pairs}\index{labeled expert!pairs}. We think of $\z_\x^{(\ell)}$ and $\z_\y^{(\ell)}$ as \emph{features} of $\x^{(\ell)}$ and $\op{\x^{(\ell)}}$, respectively. For instance we can think of $\x^{(\ell)}$ as digital images and $\z_\x^{(\ell)}$ as digits classifying the images.\index{classifiers}
To be consistent, noisy data is considered labeled data $\z_\x^{(\ell)}$ and $\z_\y^{(\ell)}$, respectively.

We also consider the \emph{labeled expert triple}\index{labeled expert!triple}, 
combining expert and feature data:
\begin{equation}  \label{eq:sampled2}
\begin{aligned}
	\mathcal{T}_\no :=
	\set{(\x^{(\ell)},\z_\y^{(\ell)},\y^{(\ell)}): \ell = 0, 1, \ldots, \no} \text{ and }\\
	\mathcal{T}_\no := \set{(\z_\x^{(\ell)},\x^{(\ell)},\y^{(\ell)}): \ell = 0, 1, \ldots, \no}
\end{aligned}
\end{equation}
and the \emph{labeled expert quadruple}\index{labeled expert!quadruple}
\begin{equation}  \label{eq:sampled3}
	\begin{aligned}
		\mathcal{T}_\no&:=
		\set{(\z_\x^{(\ell)},\x^{(\ell)},\z_\y^{(\ell)},\y^{(\ell)}): \ell = 0, 1, \ldots, \no}\;.
	\end{aligned}
\end{equation}

The elements $\x^{(\ell)}$, $\ell=1,\ldots,\no$ (without the information $\y^{(\ell)}$) are called \emph{expert inputs}. The elements $\y^{(\ell)}$, $\ell=1,\ldots,\no$ are called \emph{expert outputs}. Note that expert inputs are \emph{unsupervised training elements} as defined in
\autoref{eq:expert_information_discriminative}
\begin{equation*}
	\mathcal{U}_\no:=\set{\x^{(\ell)}: \ell = 0, 1, \ldots,  \no}.
\end{equation*}
We emphasize on some inconsistencies in the notation:
\begin{itemize}
	\item In contrast to the original definition of $\mathcal{S}$ in \autoref{eq:expert_information} we add here the subscripts $\no$, resulting in $\mathcal{S}_\no$, and for labeled samples $\mathcal{T}_\no$. The subscript denotes the number of samples. In the introduction this was avoided, because there $\no$ was fixed. Here we also consider convergence for $\no \to \infty$, meaning that we consider $\no$ as a variable.
	\item The notation above is tuned to operator learning. In the context of function and functional regression $\x^{(\ell)}$ and $\y^{(\ell)}$ are vectors or numbers, which will be taken into account below. This means that for function regression we consider expert vectors $\vx^{(\ell)}$, $\ell = 0, 1,  \ldots,\no$, and $\vy^{(\ell)}$, $\ell = 0, 1, \ldots,  \no$, and samples
	$\vz_\vx^{(\ell)}, \vz_\vy^{(\ell)}$, $\ell = 0, 1, \ldots, \no$. For univariate functions and functionals the expert data are
	$y^{(\ell)}, z_y^{(\ell)}$, $\ell = 0, 1, \ldots, \no$.
	\item For the samples and expert data we start the summation at either $0$ or $1$. It should be clear from the context what is meant. For instance in iterative methods we use $\x^{(0)}$ as prior, which is taken from the expert pairs $\mathcal{S}_\no$ and then we consider the rest of $\mathcal{S}_\no$ for operator learning or function regression, respectively. In Tikhonov regularization we denote the prior with $\x^0$ (without brackets) and it is used unsupervised.
\end{itemize}

\section{Scenarios of learning}
We study below three scenarios for function, functional and operator learning, as well as for \emph{supervised learning}.
\begin{definition} \label{de:operator_learning}
	Given are expert pairs $\mathcal{S}_\no$ as defined in \autoref{eq:expert_information}:
	\begin{description}
		\item{Function learning:} Let $\no < \infty$ and $(\vx^{(\ell)},\vy^{(\ell)}) \in \R^{m} \times \R^{n}$ for all $\ell=1,\ldots,\no$, then finding a function relation $\opo: \R^m \to \R^n$, which satisfies $F(\vx^{(\ell)}) = \vy^{(\ell)}$ for all $\ell=1,\ldots,\no$ is called function learning.\index{learning!function}
	    \item{Functional learning:} Let $\X$ be a (probably infinite dimensional) Hilbert-space:
	    and $(\x^{(\ell)},\vy^{(\ell)}) \in \X \times \R^n$ for all $\ell=1,\ldots,\no$, then finding a functional relation $\opo: \X \to \R^n$, which satisfies $\op{\x^{(\ell)}} = \vy^{(\ell)}$ is called functional learning.
        Note, $\no$ can be finite or infinite and $\X$ can be finite or infinite dimensional. In the case $\X = \R^m$ this is equivalent to function learning. \index{functional learning}
	    \item{Operator learning:} Let $\X$ and $\Y$ be (probably infinite dimensional) Hilbert-spaces:
	    	Finding an operator relation $\opo: \X \to \Y$, which satisfies $\op{\x^{(\ell)}} = \y^{(\ell)}$ is called operator learning. \index{operator learning}
        \item \emph{Supervised learning}\index{supervised learning}, refers to function, functional and operator learning of labeled data $\mathcal{T}_\no$. In particular finding an approximate relation $G: \X \times \Y \to \Z_\X \times \Z_\Y$, which satisfies
        \begin{equation*}
    	    G[\x^{(\ell)},\y^{(\ell)}] \approx (\z_\x^{(\ell)},\z_\y^{(\ell)})\;.
        \end{equation*}
        is called \emph{feature learning}.\index{feature learning} Therefore the mapping $G$ provides sample data from input data $\x$ and $\y$.
        Note that in this relation single components can be considered constant, such as the feature operator $G: \X  \to \Z_\Y$, which is, for instance, the composition of $\opo$ and a classification for the data $\y$
        \begin{equation*}
        	G[\x^{(\ell)}] = C[ \op{\x^{(\ell)}}] =: \z_\y^{(\ell)}\;.
        \end{equation*}
        We emphasize that with the above mentioned inconsistencies of notation, in function regression $\x^{(\ell)}$, $\y^{(\ell)}$ and $\z^{(\ell)}$ can be vectors, in which case they must be interpreted as $\vx^{(\ell)}$, $\vy^{(\ell)}$ and $\vz^{(\ell)}$.
    \end{description}
\end{definition}
In view of \autoref{de:operator_learning} functional and vector valued regression, discussed in \autoref{sec: Polregres}, \autoref{sec:vRKHS}, can be used as tools for supervised functional and operator learning.
Linear function, functional and operator learning is widely studied in the literature (see for instance the early references on \emph{black box} theory \cite{Pap62}) \index{theory!black box} or in \emph{multi linear regression} (see for instance \cite{YaXinXueDac12}).

We study below several variants of supervised and operator learning and start with the simplest case of \emph{supervised learning} of finite dimensional linear operators.

\section[Learning finite dimensional linear functions]{Supervised learning of finite dimensional linear functions} \label{sec:slol}
We consider learning a linear function $\opo:\R^m \to \R^n$ (that is a matrix) which satisfies
\begin{equation} \label{eq:Fx=y}
	\opo \vx^{(\ell)} = \vy^{(\ell)} \text{ for all } \ell=1,\ldots,\no.
\end{equation}
All along this section we assume that $\opo$ is injective, meaning in this case that $\opo$ has no nullspace. A consequence of this assumption is that $n \geq m$.
\begin{definition} \label{def:Un}
We denote the spans of the first $\ttm \leq \no$ training data by
\begin{equation}
	\X_\ttm := \spann\set{\vx^{(\ell)}: \ell=1,\ldots,\ttm} \text{ and } \Y_\ttm := \spann\set{\vy^{(\ell)}: \ell=1,\ldots,\ttm}.
\end{equation}
We assume that $\text{dim}(\X_\ttm) = \ttm$. That is, the images $\vx^{(\ell)}$ are linearly independent. Since we also assume that $\opo$ is injective also the data $\vy^{(\ell)}$ are linearly independent.

Orthogonal projection operators onto $\X_\ttm$ and $\Y_\ttm$ are denoted by $P_{\X_\ttm}$ and $P_{\Y_\ttm}$, respectively.
\end{definition}
This families of sets $\set{\mathcal{S}_\ttm: \ttm \leq \no}$, $\set{\X_\ttm : \ttm \leq \no}$, $\set{\Y_\ttm:\ttm \leq \no}$ are nested, that is
\begin{equation}\label{ass_2}
	\X_\ttm \subset \X_{\ttm+1}, \quad \Y_\ttm \subset \Y_{\ttm+1}, \quad \mathcal{S}_\ttm \subseteq \mathcal{S}_{\ttm+1} \text{ for all } \ttm \leq \no-1.
\end{equation}An often used methodology for function, functional and operator learning is \emph{orthonormalization} of the expert inputs (see \cite{YaXinXueDac12,AspKorSch20}).

In the following we highlight an orthonormalization strategy for supervised linear operator learning. For this we make the following assumptions on the expert information:
\begin{assumption}[Uniform boundedness]\label{ass_1}
	There exist constants $c_\x,C_\x>0$ such that
    	\begin{equation*}
    		c_\x \leq \norms{\vx^{(\ell)}}_{\R^m} \leq C_\x \text{ for all } \ell \in  \set{1,\ldots,\no}.
    	\end{equation*}
        Hence with no loss of generality we assume in this section that
        \begin{equation*}
        	\norms{\vx^{(\ell)}}_{\R^m} =1 \text{ for all } \ell \in  \set{1,\ldots,\no}.
        \end{equation*}
\end{assumption}

\begin{remark}\label{prop:data_lin_indep} Since $\opo$ is a bounded operator and the inputs $\set{\vx^{(\ell)}: \ell=1,\ldots,\no} \subseteq \R^m$ are uniformly bounded, the outputs $\set{\vy^{(\ell)}:\ell=1,\ldots,\no} \subseteq \R^n$ are also uniformly bounded.
\end{remark}
In the following we present the general approach to learn a linear operator $\opo:\R^m \to \R^n$ as used in \cite{YaXinXueDac12}: We represent the linear operator as a matrix $\opo \in \R^{m \times n}$.
Let $\no < \infty$, then we define the matrices
\begin{equation*}
	Y^{(\no)} = \begin{pmatrix} \vy^{(1)}, \ldots, \vy^{(\no)} \end{pmatrix} \in \R^{n \times \no} \text{ and }
	X^{(\no)} = \begin{pmatrix} \vx^{(1)}, \ldots, \vx^{(\no)} \end{pmatrix} \in \R^{m \times \no}.
\end{equation*}
Therefore, we aim to identify the matrix $\opo \in \R^{m \times n}$, which minimizes the function
\begin{equation*}
	\opo \in \R^{m \times n} \mapsto \norms{Y^{(\no)} - \opo X^{(\no)}}_2^2.
\end{equation*}
The optimality condition provides us with an explicit solution
\begin{equation} \label{eq:right_inverse}
	\opo = Y^{(\no)} X^{(\no)}{}^T (X^{(\no)} X^{(\no)}{}^T)^\dagger
\end{equation}
where $(X^{(\no)} X^{(\no)}{}^T)^\dagger$ is the Moore-Penrose inverse of $X^{(\no)} X^{(\no)}{}^T$ (see \autoref{not:inverse}).
This is a common estimator in statistics (see \cite{CraMas13,HorKid15}). The typical approach to determine $\opo$ is via a spectral decomposition (see \cite{YaXinXueDac12}).

An alternative approach, presented below, is based on orthonormalizing the vectors $\vy^{(\ell)}$ with a Gram-Schmidt-orthononormalization.

\subsection{Gram-Schmidt for matrix learning} \label{ss:GS}
We start with the expert data $\vy^{(1)},\ldots,\vy^{(\no)} \in \R^n$ and orthonormalize them iteratively.
Let
\begin{equation}\label{eq:sigmave}\begin{aligned}
		\vec{\sigma} : \R^n \backslash \set{0} & \to \R^n \\
		\vy &\mapsto \frac{\vy}{\norms{\vy}_2}
	\end{aligned} \quad \text{ and } \quad \begin{aligned}
	\vec{\sigma}_\ve : \R^n \backslash \set{0} & \to \R^n.\\
	\vy &\mapsto \frac{\vy}{\sqrt{\norms{\vy}_2^2+\ve^2}}
\end{aligned}
\end{equation}
Note that $\vec{\sigma}$ and $\vec{\sigma}_\ve$ can be seen as a \emph{high dimensional antisymmetric} activation function (such as \autoref{eq:sigmoid_anti} multiplied by a factor $2$).

Then the Gram-Schmidt-algorithm computes the orthonormalized vectors as follows: \index{Gram-Schmidt!orthonormalization}
\begin{equation} \label{eq:gs}
	\begin{aligned}
		\ul{\vy}^{(\ell)} &:= \vec{\sigma} \biggl(\underbrace{
			\vy^{(\ell)} - \sum_{i=1}^{\ell-1}\inner{\vy^{(\ell)}}{\ul{\vy}^{(i)}}_2 \ul{\vy}^{(i)}
		}_{=:a_\ell(\vy^{(\ell)})} \biggr) \text{ for all } \ell=1,\ldots,\no.
	\end{aligned}
\end{equation}
Moreover, we denote the according images $\ul{\vx}^{(\ell)}$ of orthonormalized expert data $\ul{\vy}^{(\ell)}$ as the solutions of
\begin{equation} \label{eq:gsuu}
		\opo \ul{\vx}^{(\ell)} = \ul{\vy}^{(\ell)} \text{ for all } \ell=1,\ldots,\no.
\end{equation}
Note that in general $\ul{\vx}^{(\ell)}$, $\ell =1,\ldots,\no$ are \emph{not} orthonormal.

In actual implementations we use the \emph{regularized Gram-Schmidt-algorithm},\index{Gram-Schmidt!regularized} which computes for all $\ell=1,\ldots,\no$
\begin{equation} \label{eq:gsII}
	\ul{\vy}^{(\ell)} = \vec{\sigma}_\ve (a_\ell(\vy^{(\ell)})).
\end{equation}
Note that this method only provides an approximation of the orthonormalized data.

Since $\ul{\vy}^{(\ell)}$, $\ell=1,\ldots,\no$, are orthonormal, it follows that
\begin{equation} \label{eq:coordinates}
	Y^{(\no)} = \ul{Y}^{(\no)} R^{(\no)},
\end{equation}
where $\ul{Y}^{(\no)}$ is the matrix consisting of orthonormalized vectors
and
\begin{equation}\label{eq:R_n}
	R^{(\no)} :=
	\begin{pmatrix}
		\inner{\vy^{(1)}}{\ul{\vy}^{(1)}}_2 & \inner{\vy^{(2)}}{\ul{\vy}^{(1)}}_2 & \cdots & \inner{\vy^{(\no)}}{\ul{\vy}^{(1)}}_2 \\
		0 			& \inner{\vy^{(2)}}{\ul{\vy}^{(2)}}_2 & \cdots & \inner{\vy^{(\no)}}{\ul{\vy}^{(2)}}_2 \\
		0			&	0       & \ddots 	& \vdots \\
		0 &  0 & \cdots & \inner{\vy^{(\no)}}{\ul{\vy}^{(\no)}}_2
	\end{pmatrix}.
\end{equation}
Since $\opo$ is linear we get from \autoref{eq:coordinates} and the assumption that $\vy^{(\ell)}$, $\ell =1,\ldots\no$ are linearly independent that
\begin{equation*}
	\ul{X}^{(\no)} = X^{(\no)} (R^{(\no)})^{-1}\;.
\end{equation*}
This means that $\ul{\vx}^{(\ell)}$, $\ell=1,\ldots,\no$ can be explicitly calculated from $\vy^{(\ell)}$, $\ell =1,\ldots,\no$
(note that we assume that $\opo$ is injective).

We compute a solution $\vx^\dag \in \R^m$ of the linear equation
\begin{equation} \label{eq:linear}
	\opo \vx = \vy
\end{equation}
with $\vy \in \range{\opo} \subseteq \R^n$ as follows. By basis expansion we get
\begin{equation} \label{eq:recon}
	\vy = \sum_{\ell=1}^\no \inner{\vy}{\ul{\vy}^{(\ell)}}_2 \ul{\vy}^{(\ell)}\,,
\end{equation}
and consequently a solution of \autoref{eq:linear} is given by
\begin{equation} \label{eq:recon_x}
	\vx^\dag = \sum_{\ell=1}^\no \inner{\vy}{\ul{\vy}^{(\ell)}}_2 \ul{\vx}^{(\ell)}\;.
\end{equation}
This is in fact the best approximating solution of an arbitrary vector $\vy \in \R^n$ in $\X_\ttm$, and therefore we denote it by $\vx^\dag$. For more details we refer to \cite{AspKorSch20,AspFriKorSch21}.

\begin{remark} \label{re:GS_just}
	Note that $a_\ell(\vy)$ as in \autoref{eq:gs} is a linear function in $\vy$. In other word Gram-Schmidt can be implemented as affine linear neural network (see \autoref{de:affinenns} above) with precomputed weights $w^{(l,k)} = - \inner{\vy^{(l)}}{\vy^{(k)}}_2$.
\begin{figure}[h]
	\begin{center}
		\includegraphics[scale=0.7]{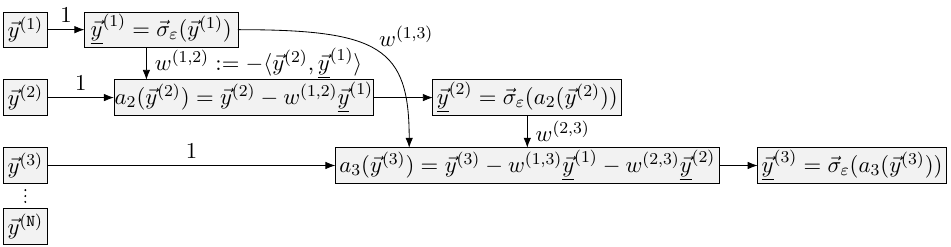}
		\caption{\label{fig:zymlk} The neural network structure of the Gram-Schmidt- orthonormalization. We use here the function $\vec{\sigma}_\ve$ from \autoref{eq:sigmave} as activation function (replacing the original $\vec{\sigma}$ in Gram-Schmidt. Note that the network structure has $\no$-layers, which is very unusual for deep neural networks.}
	\end{center}
\end{figure}
	We note, however, that this is not a standard neural network (see, for instance, \cite{Bis95}) where
	$\sigma_\ve:\R \to \R$ is evaluated for each component of $\R^n$. Such standard network cannot be used here because we must ensure that every $a_\ell(\vy^{(\ell)})$, $\ell=1,\ldots,\no$,(defined in \autoref{eq:gs}) is a linear combination of $\vy^{(j)}$, $j=1,\ldots,\ell-1$.
	
	There are efficient alternatives to Gram-Schmidt, which are, for instance, block based (see, e.g., \cite {CarLunRozTho22}). These can be reinterpreted again as deep neural networks \cite{QiuWanLuZhaDu12b}. Gram-Schmidt breaks down if and only if one of the vectors $a_\ell(\vy^{(\ell)})$ becomes zero, or in other words, if the vectors $\vy^{(\ell)}$, $\ell=1,\ldots,\no$ are linearly dependent. We excluded this by a general assumption that all images are linearly independent.
\end{remark}

\subsection{Bi-orthonormalization for matrix learning} \label{ss:Bio}
We proceed now opposite to \autoref{ss:GS} and orthonormalize the \emph{expert images} $\vec{x}^{(1)},\ldots,\vec{x}^{(\no)}$, which are denoted by $\vec{\ol{x}}^{(\ell)}$, $\ell=1,\ldots,\no$. We use overline to denote the orthonormalized expert images, while in \autoref{ss:GS} we used underline to denote the orthonormalized expert data.
Accordingly the data of orthonormalized expert images is also denoted with overline.

Since for $\vec{x} \in \X_\no$ we have the representation
\begin{equation*}
	\vx = \sum_{\ell=1}^\no \inner{\vx}{\vec{\ol{x}}^{(\ell)}}_2 \vec{\ol{x}}^{(\ell)},
\end{equation*}
it follows that
\begin{equation}\label{eq:gsi}
	\vy = \opo \vx = \sum_{\ell=1}^\no \inner{\vx}{\vec{\ol{x}}^{(\ell)}}_2 \opo \vec{\ol{x}}^{(\ell)} = \sum_{\ell=1}^\no \inner{\vx}{\vec{\ol{x}}^{(\ell)}}_2 \vec{\ol{y}}^{(\ell)}.
\end{equation}

\begin{remark} If instead of an exact Gram-Schmidt the approximate Gram-Schmidt from \autoref{fig:zymlk} is used, we get
	\begin{equation} \label{eq:gsi2}
		\begin{aligned}
			\vec{\ol{y}}^{(\ell)} &:= F \vec{\ol{x}}^{(\ell)} \approx \opo \sigma_\ve (a_\ell(\vx^{(\ell)})) \\
			& = \frac{1}{\norms{\sigma_\ve \left( a_\ell(\vx^{(\ell)}) \right)}_2}  \opo a_\ell(\vx^{(\ell)})
			\text{ for all } \ell=1,\ldots,\no.
		\end{aligned}
	\end{equation}
	Recall that the linear function $a_\ell$ is defined in \autoref{eq:gs} for the vectors $\vy \in \R^n$, while here it is the linear function of Gram-Schmidt for vectors $\vx \in \R^m$.
\end{remark}
In the following we explain a two-step orthonormalization procedure based on the following notation:
\begin{definition} Let $\vec{\ol{x}}^{(\ell)}$, $\ell=1,\ldots,\no$, be the orthonormalized training images and $\vec{\ol{y}}^{(\ell)}$, $\ell=1,\ldots,\no$ the according images as defined in \autoref{eq:gs} and \autoref{eq:gsi}, respectively. We denote:
	\begin{equation}\label{eq:pcaa}
		\begin{aligned}
			{\ol X} := (\vec{\ol{x}}^{(1)},\ldots,\vec{\ol{x}}^{(\no)}) \in \R^{{m \times \no}},\;
			{\ol Y} = (\vec{\ol{y}}^{(1)},\ldots,\vec{\ol{y}}^{(\no)}) \in \R^{{n \times \no}} \\
			\text{ and } A:={\ol Y} {\ol Y}^T \in \R^{n \times n}.
		\end{aligned}
	\end{equation}
	Note that by our general assumptions, the rank of each of the three matrices is always $\no$, which satisfies $\no \leq m \leq n$.
\end{definition}
Now, we are able to formulate the matrix learning algorithm. In fact we are not learning the matrix itself but it's singular values and functions (see \autoref{alg:ON2}) first, from which we can calculate the solution of the matrix equation \autoref{eq:linear}. The well-definedness of the algorithm is proven in \cite{AspFriSch25}:
\begin{algorithm}
	\caption{A two step orthonormalization}
	\label{alg:ON2}
	Input is expert information $(\vx^{(\ell)},\vy^{(\ell)})$, $\ell=1,\ldots,\no$.
	\begin{algorithmic}[1]
		\State Orthonormalize $\vx^{(\ell)}$ to get $\vec{\ol{x}}^{(\ell)}$, $\ell=1,\ldots,\no$;
		\State Compute $\vec{\ol{y}}^{(\ell)}$ from \autoref{eq:gsi2}, $\ell=1,\ldots,\no$;
		\State Calculate the principal component analysis (PCA) of $A=\ol{Y} \ol{Y}^T$ (this is the second orthonormalization procedures), which delivers the singular values $\psi^{(\ell)}$, $\ell=1,\ldots,\no$ of and singular vectors of $\opo$. See \cite{AspFriSch25} for a proof of this statement.
		\State Calculate the least squares solution from \autoref{eq:linear} using the principal vectors (for details see \cite{AspFriSch25}).
	\end{algorithmic}
\end{algorithm}
In \cite{AspFriSch25} this technique has been applied to compute the singular values of the Radon-transform (see \cite{Kuc13,Nat01}), represented in \autoref{fig:ex_func_u} and \autoref{fig:ex_func_v}. One challenge with the Radon-transform is its multiple eigenvalues (see \cite{Dav81,Nat01}).

\section{Learning nonlinear functions}\label{sec:lnf}
The task of approximating nonlinear functions is also called \emph{function regression}.\index{function regression} Among many other approaches neural networks (see \autoref{ch:nn}) are used for approximating nonlinear functions
\begin{equation*}
	\x:\R^m \to \R\;.
\end{equation*}
We concentrate on here uni-variate (i.e., scalar-valued) functions. For multi-variate functions $\x:\R^m \to \R^n$ every component is approximated separately.

\subsection{Reproducing kernel Hilbert-spaces (RKHS)} \label{ss:rkhs}
Reproducing kernel Hilbert-spaces are a widely used technique to approximate functions (see \cite{Scholkopf_2001}, \cite{Per22}).
\begin{definition}(RKHS) A Hilbert-space $\X$ of functions $\x: \Omega \subseteq \R^m \to \R$ is a reproducing kernel Hilbert-space if for every $\vx \in \Omega$ there exists a function $K_\vx \in \X$ such that\index{reproducing kernel Hilbert-spaces}\index{RKHS}
	\begin{equation} \label{eq:kak0}
		\inner{\x}{K_\vx}_\X = \x(\vx)\;.
	\end{equation}
	The function $K: \Omega \times \Omega \to \R$, which is defined for $\vx, \vx^{\prime} \in \Omega$ as
	\begin{equation*}
	   K(\vx, \vx^{\prime}):= \inner {K_\vx}{K_{\vx^{\prime}}}_\X\,,
	\end{equation*}
	is called \emph{reproducing kernel}\index{reproducing kernel} of the space $\X$.
\end{definition}
One of the most popular approaches to the RKHS-approximation of a function $\x:\Omega \to \R$ from its possibly noisy values $y^{(\ell)} \in \R$ at points $\vx^{(\ell)} \in \Omega$, $\ell=1,\ldots,\no$, is penalized least squares regression with the RKHS penalty. In this approach, the approximation of $\x$ in RKHS $\X$ is defined as the minimizer $\x_{\alpha}$ of the penalized least squares problem
\begin{equation} \label{eq:rkhs_ls}
\frac{1}{\no} \sum_{\ell =1}^{\no} \left( \x(\vx^{(\ell)}) - y^{(\ell)} \right)^2 + \alpha \norms{\x}_\X^2	\\ \rightarrow \min.
\end{equation}
The \emph{representer theorem} \index{representer theorem} of Kimeldorf \& Wahba \cite{KimWah70} tells us that the minimizer
$\x_{\alpha}$ lies in the finite-dimensional subspace of the infinite-dimensional RKHS $\X$ spanned by $K_{\vx^{(\ell)}}$, $\ell= 1,\ldots,\no$, and admits the following representation:
\begin{equation} \label{eq:rkhs_representer}
\x_{\alpha}(\vx)= \sum_{\ell = 1}^{\no} c^{(\ell)} K_{\vx^{(\ell)}}(\vx) \text{ for all } \vx\in \Omega,
\end{equation}
where the vector $\vec{c}= (c^{(1)}, \ldots, c^{(\no)})^T $ is given as
\begin{equation}\label{eq:c_in_rep}
\vec{c} = \no^{-1}(\alpha \id  + \no^{-1} \mathcal{K} )^{-1}Y^{(\no)} ,
\end{equation}
with the Gram-matrix associated to the kernel $K$
\begin{equation*}
  \mathcal{K}_{\ell,j} = K(\vx^{(\ell)},\vx^{(j)}) = \inner {K_{\vx^{(\ell)}}}{K_{\vx^{(j)}}}_\X, \quad 1 \le \ell,j \le \no,
\end{equation*} and
\begin{equation*}
Y^{(\no)} = (y^{(1)}, \ldots, y^{(\no)})^T \in \R^{\no}.
\end{equation*}

\subsection{Neural networks for function learning}
We describe here the strategy for approximation with \emph{shallow affine linear neural networks} (\ALNN's) (see \autoref{de:affinenns}). The objective is to find a vector of coefficients
\begin{equation} \label{eq:pl}
	\vp = (\alpha_j,\vw_j,\theta_j)_{j=1}^{\noc} \in \R^\dimlimit \text{ with } \dimlimit = (m+2)\noc, \\
\end{equation}
such that the neural network function
\begin{equation}\label{eq:classical_approximation_n}
	\vx \in \R^m \to \Psi[\vp](\vx) := \sum_{j=1}^{\noc} \alpha_{j}\sigma\left(\vw_{j}^T \vx +\theta_j  \right)
\end{equation}
approximates $\x:\R^m \to \R$ uniformly on bounded subsets of $\R^m$.

These coefficients are determined by computing an optimal vector
\begin{equation} \label{eq:training_fun}
	\argmin_{\vp \in \R^{\dimlimit}} \sum_{\ell=1}^\no | \Psi[\vp](\vx^{(\ell)}) - \vy^{(\ell)}|^2\,,
\end{equation}
which provides the best matching operator $\Psi[\vp]$ from the expert information $\mathcal{S}_\no$, defined in \autoref{eq:expert_information}.
This procedure is called \emph{training} of a neural network and discussed in greater detail in \autoref{ch:learning}.\index{neural network!training}
Instead of shallow networks as in \autoref{eq:classical_approximation} one can also use deep networks as defined in \autoref{eq:DNN} below. For the sake of simplicity of presentation we concentrate in the presentation and analysis on shallow networks.

\subsection{Deep neural networks for composite function learning}
The goal is to approximate a function ${\tt h} = {\tt g} \circ {\tt f}$. When the function ${\tt f}$ maps a high-dimensional space $\X$ onto a space $\Z$, it might be beneficial to approximate it with neural networks as described in \autoref{ch:nn}. If also ${\tt g}:\Z \to \Y$ is approximated by a neural network the resulting network is deep. Neural neural networks are therefore ideal candidates for approximating composite functions and the number of layers is therefore directly related to the number of compositions. In \cite{ConPhi12} the authors argue that if ${\tt g}: \X \to \Z$ and
${\tt f}:\Z \to \Y$, where $\X$ is high-dimensional and $\Z, \Y$ are low-dimensional, then Lancos-based methods for approximating ${\tt f} \circ {\tt g}$ are very efficient. This is intuitive because ${\tt g}$ is compressing the data. In view of this the layers of deep networks could be understood as optimal decompositions of a function, where every composite is approximated by a neural network. 

\section{Linear functional learning} \label{sec:lfl}
Let $\X$ be an infinite dimensional Hilbert-space. We are given expert information $\mathcal{S}_\no$ consisting of pairs of functions and vectors $(\x^{(\ell)},\vy^{(\ell)}) \in \X \times \R^n$ for all $\ell=1,\ldots,\no$. The goal is to compute a \emph{linear functional} $F: \X \to \R^n$, which satisfies $\opo \x^{(\ell)} = \vy^{(\ell)}$ for all $\ell=1,\ldots,\no$. In general, in this case the operator $\opo$ has a null-space - for instance if $\X \subseteq \R^{\hat{n}}$ with $\hat{n} \geq n$ and $\opo$ is linear. The reconstruction formula \autoref{eq:right_inverse} applies analogously here, when we use as function space $\X_\no$ the subspace of $\X$ spanned by the functions $\x^{(1)},\ldots, \x^{(\no)}$. The only difference is that the operator maps from $\X$ to $\R^n$. For instance the Moore-Penrose inverse $(\X^{(\no)}{}^t\X^{(\no)})^\dagger$ can be represented as an element of $\X^{m \times m}$, which is a matrix of functions (see \cite{AspKorSch20,AspFriKorSch21}).

\section{Nonlinear functional learning} \label{se:nful}
Nonlinear functional learning was considered in \cite{TiaHon95}. The goal is to compute a \emph{nonlinear functional} $\opo: \X \to \R$, from a function space $\X$ of functions $\x:\Omega \to \R$, which satisfies $\op{\x^{(\ell)}} = y^{(\ell)}$. It was shown in \cite{TiaHon95} that $\opo$ can be approximated by
\begin{equation} \label{eq:tiahon95}
		\Psi[\vp][\x] = \sum_{k=1}^{N_k} \alpha_{k}\sigma\left( \sum_{l=1}^{N_l} w_{k,l} \x(\vs_l) +\theta_k \right)\;.
\end{equation}
%
%\commentS{Suggestion for notation: $d$ instead of $N_k$ and $N$ instead of $N_l$. In the operator section below those can then be transformed into $d_j$ and $N_j$.}
Here
\begin{equation} \label{eq:pl_func}
		\vp = (\alpha_k,w_{k,l},\theta_k,\vs_l)_{k,l=1}^{N_k,N_l}  \in \R^\dimlimit
		\text{ with } \dimlimit = 2N_k + N_k N_l +m N_l\,,
\end{equation}
where $\vs_l$, $l=1,\ldots,N_l$ are the locations where the function $\x : \Omega \to \R$ is sampled. Training is performed with the goal to determine the optimal vector $\vp$ minimizing the functional
\begin{equation} \label{eq:training}
	\vp \in \R^\dimlimit \to \sum_{\ell=1}^\no \abs{ \Psi[\vp][\x^{(\ell)}] - y^{(\ell)}}^2\;.
\end{equation}
Note the difference to \autoref{eq:training_fun}, where $\Psi[\vp]$ is evaluated at a vector $\x^{(\ell)}$, while here $\Psi[\vp][\x^{(\ell)}]$ is a function evaluation.

\section{Linear operator learning} \label{ss:lol}
When we aim for learning a linear operator $\opo: \X \to \Y$ we also need to make an assumption that the expert information is  sufficiently rich in the sense that $\bigcup_{\ttm \in \N} \X_\ttm$ is dense in $\X$. Recall that $\X_\ttm$ is the subspace spanned by the first $\ttm$ training images (see \autoref{ass_2}).
\begin{assumption}[Denseness]\label{ass_3}
	We assume that the input subspaces $\X_\ttm$ are nested and dense in $\X$, that is
	\begin{equation*}
		\overline{\lim_{\ttm \to \infty} \X_\ttm} = \overline{\bigcup_{\ttm \in \N} \X_\ttm} = \X.
	\end{equation*}
\end{assumption}
As a consequence of the previous assumptions, we have
\begin{proposition} \label{pr:dense_Y}
	By \autoref{ass_3} the subspaces $\Y_\ttm = \opo \X_\ttm$ are dense in $\overline{\range{\opo}}$, that is,
	\begin{equation*}
		\overline{\bigcup_{\ttm \in \N} \Y_\ttm} = \overline{\range{\opo}}.
	\end{equation*}
\end{proposition}
The consequence of \autoref{pr:dense_Y} is that only data in the closure of the range of $\opo$ can be approximated by elements in $\X$.

\emph{Linear operator learning} performs similarly to supervised linear functional and function learning (see \autoref{sec:lfl} and \autoref{sec:slol}, respectively): Let $\X, \Y$ be infinite-dimensional Hilbert-spaces and we are given expert information
\begin{equation*}
	(\x^{(\ell)},\y^{(\ell)}) \in \X \times \Y \text{ for } \ell \in \N.\footnote{Note that we aim to learn an operator between infinite-dimensional Hilbert-spaces, such that theoretically we require an infinite amount of training data.}
\end{equation*}
The goal of operator learning is to compute a \emph{linear operator} $\opo: \X \to \Y$, which satisfies $\opo \x^{(\ell)} = \y^{(\ell)}$. The difference to the strategies proposed in \autoref{sec:lfl} and \autoref{sec:slol} is that we have to consider a multi scale approach with respect to the size of the expert information $\no \to \infty$. With this approach we aim to learn the operator $\opo$ restricted to
\begin{equation*}
  \X_\no = \set{\x^{(\ell)} : \ell =1,\ldots,\no}
\end{equation*}
first and let $\no \to \infty$ afterwards. The solutions $\x_{\X_\no}^\dagger$ of \autoref{eq:op} constraint to $\X_\no$ converge to $\xdag$ for $\no \to \infty$ solving \autoref{eq:op} (see \cite{AspKorSch20,AspFriKorSch21}). However, this strategy has some convergence issues, which have been pointed out already in \cite{Sei80} in a different context. Error propagation is currently studied in \cite{HanSch25_report}. The mentioned convergence issues can be overcome with a bi-orthonormalization strategy as outlined above in vector spaces (see \autoref{ss:Bio}), which can be carried over analogously to a function space setting.

\section{Nonlinear operator learning} \label{ss:nol}
Let $\opo: \X \to \Y$ be a nonlinear operator between Hilbert-spaces of functions $\X$ and $\Y$. These functions are specified by
$\x:\Omega_\X \subseteq \R^m \to \R$, and $\y:\Omega_\Y \subseteq \R^n \to \R$, respectively.
The idea of learning a nonlinear operator is to approximate the operator $\opo$ by a linear combination of neural network functional approximation, as defined in \autoref{eq:tiahon95}, and a classical neural network approximation for functions ( see \autoref{eq:classical_approximation_n}, again with $n=1$). That is
\begin{equation} \label{eq:op_approx}
	\begin{aligned}
		 \op{\x}(\vt) &\approx \Psi^\opo[\vp](\vt) = \sum_{j=1}^{N_j} \alpha_{j}[\x] \sigma\left(\vw_j^T \vt +\theta_j \right) \text{ for all } \vt \in \R^n.
	\end{aligned}
\end{equation}
%\commentS{Suggestion for notation: $d$ instead of $N_j$.}
For the further motivation we recall the \emph{linear operator approximation} \autoref{eq:gsi}, which tells us that in the situation of a linear operator the coefficients are just evaluations of an inner product and can be considered a quadrature integration formula of the infinite dimensional setting. Since the orthonormal basis functions according to $\ol{\x}^{(\ell)}$ have in general very similar smoothness  (compare for instance \autoref{fig:ex_func_u} and \autoref{fig:ex_func_v}) we can choose the coefficients approximating $\alpha_j$ independent of $j$. Therefore, we use \emph{functional} evaluations, defined as follows (see \autoref{eq:tiahon95})
\begin{equation} \label{eq:approxII}
	\begin{aligned}
		\alpha_j[\x] = \Psi[\vp][\x] = \sum_{k=1}^{N_k} \alpha_{j,k} \sigma\left( \sum_{l=1}^{N_l} w_{j,k,l} \x(\vs_l) +\theta_{j,k} \right)\,,
	\end{aligned}
\end{equation}
where $N_k$ and $N_l$ are independent of $j$.
%
%\commentS{Suggestion for notation: $d_j$ instead of $N_k$ and $N_j$ instead of $N_l$. Does the number of function evaluations inside of $\sigma$ depend on $j$? Based on Chen-Chens statement I would guess so but maybe one can get rid of the dependence. If not then also the evaluation points $\vs_l$ would depend on j?}
Plugging \autoref{eq:approxII} in \autoref{eq:op_approx} yields
\begin{equation} \label{eq:op_approx_II}
	\begin{aligned}
		 \op{\x}(\vt) &\approx \opo_\ttn[\x](\vt) := \sum_{j=1}^{N_j} \underbrace{\sum_{k=1}^{N_k} \alpha_{j,k} \sigma\left( \sum_{l=1}^{N_l} w_{j,k,l} \x(\vs_l) +\theta_{j,k} \right)}_{=\alpha_j[\x]} \sigma\left(\vw_j^T \vt + \theta_j  \right)\,,
	\end{aligned}
\end{equation}
%
%\commentS{In total I suggest the notation
%\begin{align*}
%		 \op{\x}(\vs) &\approx \opo_\ttn[\x](\vs) :=\sum_{j=1}^d \sum_{k=1}^{d_j} \alpha_{j,k} \sigma \left( \sum_{l=1}^{N_j} w_{j,k,l}\x ( \vs_l) + \theta_{j,k} \right) \sigma ( \vw_j^T \vs + \theta_j )
%\end{align*}
%If one wants to avoid the double usage of $\theta$ one could use $\phi_j$ instead of $ \theta_j$.
%}
which we refer to as an \emph{\deepONet}.
The $\ttn$ in \eqref{eq:op_approx_II} is identified with the triple $(N_j,N_k,N_l)$.
In \cite{LuJinPanZhaKar21} so-called \emph{DeepONets} were introduced to efficiently learn operators.
%Their suggested network structure is sketched in \autoref{fi:deeponet}.
%\begin{figure}[H]
%	\centering
%	\includegraphics[width=0.5\textwidth]{fig35/image3.pdf}
%	\caption{The DeepONet structure \cite{LuJinPanZhaKar21}.}
%	\label{fi:deeponet}
%\end{figure}

\begin{theorem}[{\bf Approximation of operators (Theorem 5, \cite{TiaHon95})}]\label{th:operator_apx}
	Let $\sigma = \tanh$.\footnote{We need particular properties of the $\sigma$, such as anti-symmetricity and that it goes to $\pm 1$ for $s \to \pm \infty$. The $\tanh$-function satisfies them.}
	The following assumptions hold:
	\begin{itemize}
		\item $\Omega_\X, \Omega_\Y$ are compact subsets of $\R^m, \R^n$, respectively;
		\item $(C(\Omega_\X),\norms{\cdot}_\infty)$ is the Banach- spaces of continuous function defined on $\Omega_\X$ and
		\begin{equation*}
			\opo: \dom{\opo} \subseteq W^{s,2}(\Omega_\X) \to L^2(\Omega_\Y)
		\end{equation*}
		is a continuous operator with respect to the Sobolev- space  (see \autoref{sob_def}) $(W^{s,2}(\Omega_\X),\norms{\cdot}_{W^{s,2}(\Omega_\X)})$;
	    \item $\dom{\opo}$ is a compact set of $(C(\overline{\Omega_\X}),\norms{\cdot}_\infty)$.
	\end{itemize}
	Then, for every $\ve >0$, there exist
	\begin{itemize}
		\item positive integers $N_j$, $N_k$, $N_l$,
		\item numbers $c_i^k$, $\zeta_k$, $\xi_{i,j}^k \in \R$,
		\item vectors $\vw_k \in \Omega_\Y$, $\vs_j \in \Omega_\X$ for $i=1,\ldots N_i$, $k=1,\ldots N_k$, $j=1,\ldots N_j$,
	\end{itemize}
	such that for all $\x \in \dom{\opo}$
	\begin{equation}\label{eq:operator apx}
		\norms{\op{\x} - \opo_\ttn[\x]}_{L^2(\Omega_\Y)} < \ve,
	\end{equation}
    where $\opo_\ttn[\x]$ is defined in \autoref{eq:op_approx_II}.
\end{theorem}
The proof requires an adaptation of indices and the embedding theorem for Sobolev-spaces (see \autoref{th:ap:sob_emb}) from \cite{TiaHon95} (see \cite{SchVuYan25_report}).

\begin{remark}
	Following \cite{TiaHon95} it becomes clear that the number of free parameters can be significantly reduced: In particular we can choose
	$$w_{j,k,l} = w_{k,l}, \; \theta_{j,k} = \theta_k\,,$$
	which results in the simplified approximation:
	\begin{equation} \label{eq:op_approx_III}
		\begin{aligned}
			\op{\x}(\vt) &\approx \sum_{j=1}^{N_j} \sum_{k=1}^{N_k} \alpha_{j,k} \sigma\left( \sum_{l=1}^{N_l} w_{k,l} \x(\vs_l) +\theta_k \right)\sigma\left(\vw_j^T \vt +\theta_j  \right)\;.
		\end{aligned}
	\end{equation}
\end{remark}

\subsection{Learning and regularizing with operator valued reproducing kernels}\label{sec:vRKHS}
In this section, we discuss operator learning with \emph{vector valued reproducing kernel Hilbert-spaces}
(vRKHS)\index{vector valued reproducing kernel Hilbert-spaces}\index{vRKHS}. These spaces generalize \emph{reproducing kernel Hilbert-spaces} (see \autoref{ss:rkhs}) and are tools for approximating vector valued functions and operators.
\begin{definition}(vRKHS, \cite{CarVitToi06}) Let $\X$ and $\Y$ be Hilbert-spaces and let $\mathcal{L}(\Y;\Y)$ denote the Banach- space of bounded linear operators on $\Y$. $\mathcal{L}(\Y;\Y)$ is equipped with the operator norm (see \autoref{de:op_norm}).
	
	Moreover, we assume that the potentially \emph{nonlinear} mapping
	\begin{equation} \label{eq:kak}
		\mathfrak{K}: \X \times\X \rightarrow \mathcal{L}(\Y;\Y),
	\end{equation}
	is
	\begin{itemize}
		\item \emph{symmetric}, which means that
		\begin{equation*}
			\mathfrak{K}[\x,\x'] = \mathfrak{K}[\x',\x]	\text{ for all } \x,\x' \in \X\;.
	    \end{equation*}
		\item for all  $\x,\x' \in \X$ the operator $\mathfrak{K}[\x,\x']$ is positive-semidefinite on $\Y$, i.e.,
		\begin{equation*}
			\inner{{\tt h}}{\mathfrak{K}[\x,\x']{\tt h}}_{\Y} \geq 0 \text{ for all } \tt h \in \Y\;.
		\end{equation*}
	\end{itemize}
	For every $\x \in X$ and $\tt h \in \Y$ we define an $\Y$-valued function
	\begin{equation*}
		\begin{aligned}
			\mathfrak{K}_{\x} {\tt h}: \X &\to \Y\;. \\
			\x' &\mapsto \mathfrak{K}[\x',\x]{\tt h}
		\end{aligned}
	\end{equation*}
	Then the closure of the linear span
	\begin{equation*}
		\mathcal{G}_{\mathfrak{K}} = \overline{\set{\mathfrak{K}_{\x}{\tt h}:\X \to \Y: \x \in \X, {\tt h} \in \Y}}^{\mathcal{G}_{\mathfrak{K}}}\,,
	\end{equation*}
	of all such functions with respect to the norm induced by the inner product
	$$\langle\mathfrak{K}_{\x}{\tt h},\mathfrak{K}_{\x'}{\tt h}'\rangle_{\mathcal{G}_{\mathfrak{K}}}:=
	\inner{{\tt h}}{\mathfrak{K}[\x,\x']{\tt h}'}_{\Y}$$
	is called the vector valued (vRKHS) reproducing-kernel Hilbert-space generated by the operator-valued kernel $\mathfrak{K}$.
\end{definition}

For supervised expert data
\begin{equation*}
\mathcal{S}_\no = \set{(\x^{(\ell)},\y^{(\ell)}) \in \X \times \Y : \ell=1,\ldots,\no}
\end{equation*}
the considered vRKHS-method consists in computing a regularized function
\begin{equation} \label{eq:vrkhs_ls}
	\opo_\alpha := \argmin_{\opo \in \mathcal{G}_{\mathfrak{K}}} \left(
	\frac{1}{\no} \sum_{\ell=1}^{\no} \norms{ \op{\x^{(\ell)}} - \y^{(\ell)}}_{\Y}^2 + \alpha \norms{\opo}^2_{\mathcal{G}_{\mathfrak{K}}}	\right)\,,
\end{equation}
which can be seen as an operator analog of the RKHS-method, defined by \autoref{eq:rkhs_ls}.
\begin{example} \label{ex:gene} (see \cite{meunier2024optimal}) Let
	\begin{equation} \label{eq:kernel_scal}
		\mathfrak{K}[\x,\x']:= K[\x,\x']\id_{\Y} \text{ for all } \x,\x'\in \X,
	\end{equation}
	where $\id_{\Y}$ is the identity operator in $\Y$, and $K: \X \times\X \rightarrow \R$ is  a scalar-valued symmetric positive-semidefinite kernel function. Examples of such function can be found in Table 1 of \cite{Gretton2016} and include in particular
	radial basis kernels $$K[\x,\x'] = \kappa (\norms{\x-\x'}_{\X})$$
	defined for $s \in \R$
	\begin{itemize}
		\item $\kappa (s) = \exp^{-\frac{s^2}{\sigma^2}}$ (Gaussian kernels),
		\item $\kappa (s) = \exp^{-\frac{s}{\sigma^2}}$ (exponential kernels),
		\item $\kappa (s) = (1 + \frac{s^2}{\sigma^2})^{-1}$ (Cauchy kernels),
		\item $\kappa (s) = (1 + s^{\theta})^{-1}$, $0< \theta < 2$ (generalized $t$-student kernels),
		\item $\kappa (s) = (s^2 + \sigma^2)^{-1/2}$ (inverse multiquadric kernels).
	\end{itemize}
	The family of kernels given by \autoref{eq:kernel_scal} allows for a closed-form solution of \autoref{eq:vrkhs_ls}. Namely, from \cite[Proposition 1]{meunier2024optimal} we know that for vRKHS $\mathcal{G}_{\mathfrak{K}}$ generated by the kernels from \autoref{eq:kernel_scal} the solutions $\opo_{\alpha}[\x]$ of \autoref{eq:vrkhs_ls} are given as
	\begin{equation}\label{eq: closed_form}
		\opo_{\alpha}[\x] = \sum_{\ell = 1}^{\no} \y^{(\ell)} {\tt c}_{\ell}[\x] ,
	\end{equation}
	where the vector $\vec{\tt c}[\x]= ({\tt c}_{1}[\x],\ldots, {\tt c}_{\no}[\x])^T $ admits the representation
	\begin{equation}\label{eq:c}
		\vec{\tt c}[\x] = \no^{-1}(\alpha \id + \no^{-1} \mathcal{K} )^{-1} Y^\no[\x] ,
	\end{equation}
	with the Gram-matrix associated to the scalar-valued kernel $K[\x,\x']$,
    \begin{equation*}
    	\mathcal{K}_{\ell,j} = K[\x^{(\ell)},\x^{(j)}] \text{ with } 1 \le \ell,j \le \no
    \end{equation*}	
	and
	\begin{equation*}
	  Y^\no[\x]= (K[\x,\x^{(1)}],\ldots, K[\x,\x^{(\no)}])^T\;.
	\end{equation*}
	
	One can see the similarities between penalized least squares regressions in RKHS and vRKHS by comparing \autoref{eq:rkhs_ls} - \autoref{eq:c_in_rep} with \autoref{eq:vrkhs_ls} -  \autoref{eq:c}.
\end{example}
\begin{remark}
Note that learning with vRKHS also gives a tool for function-on-function regression, which currently attracts attention in several ways, including sensor technology, data measurement, and engineering security. For example, recently in \cite{WitNeuMenGer25}, a function-on-function regression framework for long-term monitoring of changes in structural conditions has been discussed in the context of Structural Health Monitoring. It is worth to observe that in the framework of  \cite{WitNeuMenGer25} a model describing the effect of temperature  $\x = \x(t)$ has also the form of \autoref{eq: closed_form}, while the functions ${\tt c}_{\ell}[\x]$ are chosen a priori, but not in a data-driven manner (see the formula (6) in \cite{WitNeuMenGer25}). From this perspective, learning with vRKHS potentially allows for more flexibility.
\end{remark}

%\subsection{Deep neural networks for decomposition cases}
%We consider a decomposition operator as defined in \autoref{sec:ill}.
%\commentO{ToDo}

%Recall in \autoref{sec:ill}, a nonlinear operator $F$ decomposed in two ways:
%
%\begin{align*}
%	\text{Case 1:} &\quad F = \nopo \circ L \quad \text{(linear ill-posed $L$ first)} \\
%	\text{Case 2:} &\quad F = L \circ \nopo \quad \text{(linear ill-posed $L$ last)}
%\end{align*}

\subsection{Backpropagation} \label{ss:backpropagtion}
More general then calculating neural network parametrization of a function is the problem of calculating
neural network parameters of an input-output relation. The algorithms, which have been developed in the machine learning community (see for instance \cite{MinPap88,Hoc91,WilZip95,HocSchm97,Schm93}) are named \emph{backpropagation algorithms}. This topic falls in the category of function and operator learning, which have been discussed in \autoref{sec:lnf} and \autoref{ss:nol}. We split the learning problem of finding the parameters in \autoref{eq:op_approx_II} into two steps. This means, we solve the coupled system:
	\begin{equation}\label{eq:learn3}
		\begin{aligned}
		\opo_1(\vp)
		&:=\sum_{j=1}^{N_j} \beta_j^{(\ell)} \sigma (\vw_j^T \vt_\rho +\theta_j)\,,\\
		\opo_2(\vp) &:=
		\left(\beta_j^{(\ell)}\right)_{
			\tiny{\begin{array}{c}
					j=1,\ldots,N_j\\
					\ell=1,\ldots,\no
				  \end{array}}} \\
			  &= \left(\sum_{k=1}^{N_k} \alpha_{j,k} \sigma \left( \sum_{l=1}^{N_l} w_{j,k,l} \x^{(\ell)}(\vs_l)+\theta_{j,k}\right) \right)_{
			  \tiny{\begin{array}{c}
			  j=1,\ldots,N_j\\
			  \ell=1,\ldots,\no
		  \end{array}}}\;.
			  \end{aligned}
\end{equation}
Now, the parameters (assuming that $\vec{s}_l$, $l = 1,\ldots,\no$ is fixed)
\begin{equation*}
	\begin{aligned}
	\vp = (\underbrace{\alpha_{j,k}}_{\R},\underbrace{w_{j,k,l}}_{\R},\underbrace{\vw_j}_{\in \R^m},\underbrace{\theta_{j,k}}_{\in \R},\underbrace{\theta_j}_{\in \R})_{\tiny \begin{array}{c} j=1,\ldots,N_j \\ k=1,\ldots,N_k \\ l=1,\ldots,N_l \end{array}} \in \R^\dimlimit \\
		 \text{ with } \dimlimit=N_jN_K(2+N_l) + (m+1)N_j\;.
	\end{aligned}
\end{equation*}
are determined in two steps, by first determining 
$$ \vp_1  = (\underbrace{\alpha_{j,k}}_{\R},\underbrace{w_{j,k,l}}_{\R},\underbrace{\theta_{j,k}}_{\in \R})_{\tiny \begin{array}{c} j=1,\ldots,N_j \\ k=1,\ldots,N_k \\ 
		\ell =1,\ldots,\no \end{array}} \in \R^{N_jN_K(2+N_l)}\,,$$
and then by determining
$$ \vp_2  = (\underbrace{\vw_j}_{\in \R^m},\underbrace{\theta_j}_{\in \R})_{\tiny \begin{array}{c} j=1,\ldots,N_j \\ \end{array}} \in \R^{(m+1)N_j}\;.$$

In an implementation $\dimlimit$ should match the number of measurements, which is $Q\no$.

%\commentO{Correct:
%Backprojection means that we solve $F_2(\vp)$ and then $F_1(\vp)$. The analysis is the same as for shallow? \cite{Elm90,RumHinWil86}}

\section{Applications in regularization theory} \label{sec:appl}
In this section we apply \autoref{th:NeuSch90} and \autoref{th:NeuSch90b} to the $c$-example, \autoref{ex:c_reconstruct}, and the $a$-example from \autoref{a_reconstruct}.

\begin{example}[$c$-example with \RQNN{}s: convergence] \mbox{} \label{ex:c_recon_cont}
	We use \RQNN{}s to approximate functions $\x \in \X=L^2(0,1)$. In this case we have
	\begin{equation*}
		\X_\ttm = \overline{\spann \left(\mathbb{W}_\ttm \right)}\,,
	\end{equation*}
	see \autoref{eq:wfm}, \autoref{eq:xfm}.  
	
	Moreover, let $\opo_\ttn$ be the finite element operator as in \autoref{ex:ac}.
	
	If $\xdag \in \mathcal{L}^1(\R)$ can be extended to a countably representable function, see \autoref{de:representer}, then the approximation rate \eqref{eq:app_error} can be used to show 
	\begin{equation} \label{eq:para_selectII}
		\frac{\gamma_\ttm \xi_\ttm+ \frac{L}{2}\xi_\ttm^2}{\sqrt{\alpha}} \to 0,
	\end{equation}
	where $\xi_\ttm = \norms{(I-P_\ttm)\xdag}$, provided $\alpha \to 0$ sufficiently slowly. Note that $\ttm = \ttm({\tt M})$ as in \autoref{eq:Nfunc}. 
	
	Thus \autoref{th:NeuSch90} yields convergence of (sub-)sequences of approximate minimizers $\x_{\ttm, \ttn}^{\alpha,\delta,\eta}$ of the functionals
	\begin{align*}
		\mathcal{T}^{\ttm,\ttn}_{\alpha,\y^\delta}[\x] =  \norm{\opdis{\ttn}{\x} -\yd}^2 + \alpha\norm{\x-\x^0}^2
	\end{align*}
	to $\x^0$-minimum-norm solutions of \autoref{eq:op}.
	Summarizing, this example shows that: %we therefore have shown that the
	\bigskip\par%\noindent
	\fbox{
		\parbox{0.90\textwidth}{
			%\begin{center}
		 The {\bf finite element approximation} $\opo_\ttn$ defined on the space of \RQNN {\bf neural network functions} $\X_\ttm$ can be used for a convergent discretization of Tikhonov regularization for solving \autoref{ex:c_reconstruct}.
	%\end{center}
	}}
\end{example}

\begin{remark} \label{re:space_conj}
		The condition $\xdag \in \mathcal{L}^1(\R)$ is less technical than it appears at first glance. We recall the relation to the $\ell^1$-norm (see \autoref{de:ell1space}), if the basis ${\tt U}$ is orthonormal. One difficulty in combining regularization theory and data driven methods comes from the fact that the information about a data constructed surrogate operator can be contained in somehow unpredictable components of the data. We observed this in connection with \autoref{alg:ON2} with which one can calculate singular values of an operator from expert data: We have observed that we \emph{cannot} compute the singular vales in an order fashion, from first to last. For instance, singular vectors corresponding to low singular values might not be contained in the expert data. The same difficulties appear here and are resembled in the condition $\xdag \in \mathcal{L}^1(\R)$.
	\end{remark}
\begin{example} [$c$-example with \RQNN{}s: rates]\label{ex:c_recon_cont_II}
In order to obtain a rate for the convergence $\x^{\alpha,\delta,\eta}_{\ttm ,\ttn} \to \xdag$ we need to verify additional conditions, mainly, the source condition in \autoref{as:ip:rates} and the rate $\norm{(I - P_\ttm)\x^0} = \mathcal{O}(\gamma_\ttm)$.
%and impose further rate restrictions on the various involved quantities, as stated in \autoref{th:NeuSch90b}. Concerning \autoref{as:ip:rates}, which includes the source condition,
Concerning the former we can rely on the results of \cite{EngKunNeu89}, recall \autoref{ex:c-rates}. In particular, we point out that an $\omega \in \Y$ satisfying $\xdag - \x^0 = \opd{\xdag}^*\omega$ exists, if $(\xdag-\x^0)/\op{\xdag} \in W^{2,2} \cap W^{1,2}_0$. %Convexity of $\dom{\opo}$ is clear.
\end{example}

Next, we analyze the $a$-example from \autoref{a_reconstruct} further. We emphasize that in \autoref{ex:c_recon_cont} we have used a neural network ansatz for the functions in $\X$ and a finite element ansatz for the operator. In the following example we do the opposite.
\begin{example}[$a$-example with an \deepONet: convergence] \mbox{} \label{ex:s_recon_cont}
	We continue with \autoref{a_reconstruct} with $\Omega_\X=\Omega_\Y=]0,1[$ and $\opo$ as in \autoref{eq:aex},
	\begin{equation*}
		\begin{aligned}
			\opo: \dom{\opo} := \set{\x \in W^{1,2}(]0,1[) : \x \geq \rho > 0} \to L^2(]0,1[), \quad \x \mapsto \y[\x]
		\end{aligned}
	\end{equation*}
	where $\y[\x]:[0,1] \to \R$ is the unique solution of
	\begin{equation*}
		\begin{aligned}
			-(\x\y')'(s)  & = {\tt f}(s) \quad \text{ for all } s \in ]0,1[\,, \\
			\y(0) = \y(1)  & = 0 \;.
		\end{aligned}
	\end{equation*}
	Concerning the applicability of \autoref{th:NeuSch90} to this inverse problem we refer to \cite{EngKunNeu89}.
	
	Now, we let $\X_\ttm$ be the space of linear splines on a uniform grid of $\ttm + 1$ points in $[0, 1]$ vanishing at 0 and 1, while $\opo_\ttn$ is a suitable sequence of \deepONet s, as defined in \autoref{eq:op_approx_II}, approximating $F$ in a neighborhood of $\xdag$ according to \autoref{th:operator_apx}. That is,
		\begin{equation} \label{eq:nun}
			\norms{\opo_{\tt n}[\x]-\op{\x}} \to 0
		\end{equation}
		uniformly with respect to all $\x$ in the neighborhood, so that a sequence $\nu_\ttn \to 0$, as required by \autoref{as:fda}, exists.
	% the space of linear splines on $]0,1[$ with equidistant grid, that is
	%\begin{equation*}
		%\X_\ttm = \set{{\tt g} \in C[0,1]: {\tt g} \text{ is linear in } \left]\frac{i-1}{\ttm},\frac{i}{\ttm} \right[, i=1,\ldots,\ttm, {\tt g}(0)={\tt g}(1)=0}\;.
	%\end{equation*}
	%We know from \cite{NeuSch90} that
	%\begin{enumerate}
		%\item $\opo$ is compact and Fr{\'e}chet-differentiable and that \autoref{eq:lipschitz} holds.
		%\item $P_\ttm \xdag \in C_\ttm$, where $P_\ttm: \X \to \X_\ttm$ is the orthogonal projector onto $\X_\ttm$.
		%\item The operator $\opo_\ttn$ is the \deepONet as defined in \autoref{eq:op_approx_II}. Note that according to \autoref{th:operator_apx}
		%\begin{equation} \label{eq:nun}
		%	\boxed{\nu_\ttn:=\norms{\opo_{\tt n}[\x]-\op{\x}} \to 0}
		%\end{equation}
		%locally on bounded subsets of $W^{1,2}(]0,1[)$.
	%\end{enumerate}
	As in \autoref{ex:c_recon_cont} we have $\xi_\ttm, \gamma_\ttm \to 0$. Therefore, if $\alpha \to 0$ sufficiently slowly, we obtain convergence of (sub)sequences of the $\x_{\ttm, \ttn}^{\alpha,\delta,\eta}$ to $\x^0$-minimum-norm solutions of \autoref{eq:op}.
	% Then we make the parameter selections \autoref{eq:para_selectII}
	% with $\eta_\ttn$ from \autoref{eq:nun}. With such a parameter choice
	% every sequence $(\x_{\ttm_k,\ttn_k}:=\x_{\ttm_k,\ttn_k}^{\alpha_k,\delta_k,\eta_k})_{k \in \N}$, where
	% $\delta_k \to 0$, $\eta_k \to 0$, $\ttn_k \to \infty$, $\ttm_k \to \infty$
	% and $\alpha_k := \alpha(\ttm_k,\ttn_k,\delta_k,\eta_k)$ as $k \to \infty$ and
	% $\x_{\ttm_k,\ttn_k}$ is a solution of \autoref{eq:Tik_approx}, has a convergent subsequence. The limit of every convergent
	% subsequence is an $\x^0$-- minimum norm solution. If, in addition the $\x^0$-- minimum norm solution $\xdag$ is
	% unique, then
	% \begin{equation*}
	% 	\lim_{\delta \to 0, \eta \to 0 \atop \ttm\to\infty, \ttn \to\infty}
	% 	\x_{\ttm,\ttn} = \xdag.
	% \end{equation*}
\end{example}

\bigskip
\fbox{
	\parbox{0.90\textwidth}{
	An \deepONet $\opo_\ttn$ defined on the space of {\bf linear splines} $\X_\ttm$ can be used for a convergent discretization of Tikhonov regularization.
}}
\bigskip

\begin{example}[$a$-example: rates] \mbox{} \label{ex:a_recon_II} Assuming that the source condition is satisfied, the details of which can be found in \cite{EngKunNeu89},
the main obstacle in obtaining convergence rates for the approximation procedure elaborated in \autoref{ex:s_recon_cont} seems to be connected to requirement
\begin{equation*}
	\nu_\ttn = \mathcal{O}(\gamma_\ttm^2 + \delta)
\end{equation*}
from \autoref{th:NeuSch90b}, since the rate of $\nu_\ttn \to 0$ is unknown in general. This suggests that, generally speaking, learned operators do not yet provide the same quality of results as constructive finite element methods when applied in regularization theory.
\end{example}

\bigskip%\par%\noindent
\fbox{
\parbox{0.9\textwidth}{At present neural network surrogates do not provide the same quality of approximation as splines and wavelets for functions and finite element methods for operators. On the other hand, if used in a hybrid manner, the analysis in \autoref{ch:Aspri paper} shows improvements over classical results.
}}
\bigskip\par\noindent

\section{Open research questions}
The neural network approximation theory is not as advanced that we can apply it directly in regularization theory. While the classical regularization analysis for finite dimensional approximations (see \autoref{ch:classic}) is based on the well-developed theory of splines (see for instance \cite{DeB78}) and finite element methods (see for instance \cite{Hac18,Cia78}) Sobolev- space estimates are not as common for neural networks (see for instance \autoref{ta:nn}), but instead Besov- space approximations are used (see \cite{Tri78,Tri92,Tri06}). Resulting from this we reviewed several open challenges:
\begin{opq} \label{opq:approx1}
	The current literature on approximations of functions $\x$ as surveyed in \autoref{ta:nn} is only of partial use in inverse problems theory, where typically approximation properties in Sobolev-, Bochner- or Banach- spaces are required (see for instance \cite{NeuSch90,PoeResSch10,NeuSch97,DicMaa96}). Can the approximation results for neural networks be extended to Sobolev- or Bochner- spaces (see \cite{Tri83,Tri92,Tri06}?
\end{opq}
\begin{opq} \label{opq:approxs}
	The two treated examples \autoref{ex:c_recon_cont} and \autoref{ex:s_recon_cont} can be applied for neural network basis function or neural network operator approximations, respectively. Can the theory be extended to allow for two neural network approximations (operators and function spaces) in parallel?
\end{opq}

\begin{opq} \label{re:317}
	Every element from $\Psi(\dom{\Psi})$ is represent only by the vectors $\vp \in \R^\dimlimit$ aside from the elements with \emph{obvious symmetries} such as formulated in \autoref{eq:antisym}. These ``mirrored'' elements
	are a set of measure zero in $\bP$. We conjecture that this corresponds to the set of measure zero as stated in \cite{Lam22}.
\end{opq}

\begin{opq} \label{re:3:17a}
	\autoref{eq:ca_dd} requires that all components of the vector $\vec{\alpha}$ are non-zero. This means in particular that for ``sparse solutions'', with less than $\dimlimit=(m+2) \noc$ non zero coefficients,
	convergence is not guaranteed by the existing results, because of a locally degenerating submanifold. The conjecture is that Newton's method is still converging and convergence can also be proven.
\end{opq}

\begin{opq} \label{co:newton1}
	Let $Q\no=\dimlimit$ (meaning that in \autoref{eq:op_disc} we have as many equations as unknowns). We conjecture that Newton's method, \autoref{eq:newton_invert}, is locally convergent if the points $(\vec{\beta}_j \neq 0, \theta_j)$ are pairwise different and the vectors $\vw_j$ are linearly independent.
\end{opq}
Open research questions are also related to practical realizations:
\begin{opq} \label{co:gs1}
With a smooth approximation $\vec{\sigma}_\ve$, the linear dependence of the training images does not result in a break down of the algorithm described in \autoref{fig:zymlk}. In fact, all available training pairs can be used. However, the linear dependence assumption cannot be avoided for the analysis (see \cite{AspFriSch25}). The research question concerns, whether an approximate Gram-Schmidt-methods provides approximations of the orthonormalized training vectors $\vy^{(\ell)}$, $\ell = 1,\ldots,\no$. The second research question in this context is whether block Gram-Schmidt can be written as a deep neural network.
\end{opq}

\section{Further reading}
Level set methods for solving inverse problems have been studied in \cite{Bur01,BurOsh05,OshSan01,San96,Set99}. Monotonic sigmoid functions can be used in level set methods. The approximation of arbitrary level set functions with neural network functions has been studied in \cite{SchShiVu25}.

Classification with generalized neural networks has been studied in \cite{FriSchShi25}.

Operator learning has been performed in many different situations, leading for instance to \emph{black box theory}\index{black box theory} (see for instance \cite{Pap62}). An early reference on back-propagation algorithms in neural networks and physiological motivation is \cite{Ros61}, which appear under the name of perceptrons\index{perceptron} (see \cite{MinPap88}).
	
In \cite{PinPet23} both \emph{discrete} forward and inverse problem are approximated and neural networks are learned for both. It is claimed that deep neural networks can be used to stably solve high-dimensional, noisy, nonlinear inverse problems. Implementing the approach from \cite{PinPet23} one learns networks for the forward and inverse operator separately, and accordingly the coefficient vectors for forward and inverse problem are uncorrelated. A counter argument against this strategy is that the inverse operator $\opo^{-1}$ is not continuous, and therefore the learned inverse operator can only be regularizing if the number of coefficients of the trained inverse network are small relative to the noise level.

In \cite{MooCyrOhmSieTum25} Jacobi iterations are learned for determining the largest singular value and it was shown that these are neural networks.

The relation between tomography and reproducing kernels has ben discussed recently in \cite{YunPan25}.

Learning in vRKHS has many applications (see for instance \cite{MicPon05,CarVitToi06,meunier2024optimal} and the references therein.

\chapter{Sequential Learning} \label{ch:sequentiallearning}

\emph{Sequential learning} addresses the challenge of making predictions or extracting structure from data that arrive in a temporal or logical order, where each observation may depend on preceding ones. \emph{Markov-models}\index{Markov-models} provide a natural framework for such problems by capturing the probabilistic dependencies between successive states. The forward problem in a Markov model is straightforward: given the model parameters, we compute the likelihood of an observed sequence. The inverse problem, however, is more challenging: given only the observed sequence, we wish to recover the underlying parameters that generated it. \emph{Markov-models}\index{Markov-models}, which we study now, are a mathematical foundation of \emph{sequential learning}.\index{learning!sequential}

This chapter will introduce the foundational concepts of Markov models and hidden Markov models, framing each as a general inverse problem. We develop the necessary mathematical machinery, including forward operators, inverse problems, parameter estimation, and so on.

\section{Homogeneous Markov-models} \label{ss:stationaryMc}
Below we introduce Markov-models. We motivate their definition and analysis by two well studied examples for the better understanding:
\begin{itemize}
	\item the \emph{weather example}\index{example!weather} (see \autoref{ex:weather}) and
	\item the \emph{mugs and vase examples}\index{example!mugs and vase} (see \autoref{ex:cupsandvase}).
\end{itemize}
Before we study these examples we comment on the notation used here.
\begin{notation}
	In \autoref{de:prob_joint} and \autoref{de:prob_cond} we defined the \emph{joint} and \emph{conditional} probabilities, which are recalled here:
	 For two discrete random variables
		${\textsf T}_1: \Omega_1 \to \Alpha_1$, ${\textsf T}_2:\Omega_2 \to \Alpha_1$
		\begin{itemize}
			\item the joint probability is defined as
		\begin{equation*}
			\prob_{{\textsf T}_1,{\textsf T}_2}(\alpha_1,\alpha_2) =
			\prob\bigl( \set{   ({\textsf t}_1,{\textsf t}_2): ({\textsf T}_1,{\textsf T}_2)=(\alpha_1,\alpha_2)
			} \bigr)\,,		
		\end{equation*}
		which is defined in the product space $({\textsf t}_1,{\textsf t}_2) \in \Omega_1 \times \Omega_2$.	
		\item The conditional probability is defined as
		\begin{equation*}
			\prob_{{\textsf T}_1|{\textsf T}_2}(\alpha_1|\alpha_2) =
			\left\{ \begin{array}{ll}
				\frac{\prob_{{\textsf T}_1,{\textsf T}_2}(\alpha_1,\alpha_2)}{\prob_{{\textsf T}_2}(\alpha_2)}& \text{ if } \; \prob_{{\textsf T}_2}(\alpha_2)>0\,,\\
				0 & \text{ if } \; \prob_{{\textsf T}_2}(\alpha_2)=0\;.
				\end{array}
			\right.
		\end{equation*}
		\end{itemize}
		This notation has been used already in \cite{SchGraGroHalLen09} and is different to standard notation in Bayesian analysis:
		For instance there they use for the joint probability the notation
		\begin{equation*}
		   \prob(({\textsf T}_1,{\textsf T}_2)=(\alpha_1,\alpha_2)) \text{ or simply }\;
		   \prob(\alpha_1,\alpha_2)\;.
		\end{equation*}
		In the context of sequential learning, the new notation seems quite intuitive, because we aim to determine the probabilities considered as functions of all possible inputs. An example is the parameter vector $\vec{\lambda}$ in \autoref{eq:lambda1}, which contains all evaluations of conditional probabilities and probabilities of the state spaces, respectively. In other words identifying $\vec{\lambda}$ is a parameter estimation problem or an inverse problem, as one wishes to call it, for the functions $\prob_{\observ|\staterv}$ and $\prob_{\staterv}$, respectively.
\end{notation}

We start with the Markov-weather model:
\begin{example}[Weather model (see \cite{Lue79})] \label{ex:weather} We assume that weather has three observable states, namely
	sonny, cloudy and rainy, which we denote by $\Omega_\obse := \set{s, c, r}$. 
	We make the assumption that the weather is \emph{only depending} on the previous day but not on the other preceding days.
	
	We use the probability space $\Omega_\obse$ (see \autoref{de:ap:prob_space}) with the $\sigma$-algebra $\mathfrak{F} =2^{\Omega_\obse}$ and the probability measure
	\begin{equation*}
		\prob : \mathfrak{F} \to [0,1]\;.
	\end{equation*}

	In addition we are given the probabilities that weather changes over one day according to the following probability table:
	\begin{center}
		\begin{tabular}{c|ccc}
			$p_{i,j}$ & s & c & r\\
			\hline
			s & $\frac{1}{2}$ & $\frac{1}{2}$& $0$\\
			c & $\frac{1}{2}$ & $\frac{1}{4}$ & $\frac{1}{4}$ \\
			r & $0$ & $\frac{1}{2}$ & $\frac{1}{2}$\\
		\end{tabular}
	\end{center}
	The entries of the according matrix
    \begin{equation} \label{eq:weather_matrix}
	\probm = \begin{pmatrix}
		p_{i,j}
	\end{pmatrix}_{i,j=1,2,3} = \begin{pmatrix}
	\frac{1}{2} & \frac{1}{2} & 0\\
	\frac{1}{2} & \frac{1}{4} & \frac{1}{4} \\
	0 & \frac{1}{2} & \frac{1}{2} \end{pmatrix}\;.
	\end{equation}
	are \emph{conditional probabilities} of the observation random variable
	\begin{equation}\label{eq:state:rvII}
		\observ_k:\Omega_\obse \to \R, \quad k \in \N_0\,,
	\end{equation}
	where $\observ_k(s)=1$, $\observ_k(c)=2$ and $\observ_k(r)=3$. In practice we identify $\Omega_\obse = \set{s, c, r}$ and $\set{1,2,3}$, and use them synonymously.
	This means that
	\begin{equation*}
		p_{i,j} := %p_{\obse_i,\obse_j} :=
        \prob_{\observ_{k+1}|\observ_k}(\obse_i|\obse_j),
	\end{equation*}
	provides the probability that on the next day (``k+1'') we observe the weather condition $\obse_i \in \Omega_\obse$ if we observe the weather condition $\obse_j \in \Omega_\obse$ on day $k$. For instance that a weather condition changes from sunny to cloudy is
	$$\prob_{\observ_1|\observ_0}(s|c)=1/2\;.$$
	We investigate now the weather over multiple days: We denote by
	\begin{equation*}
		\vp_k = \begin{pmatrix}
			\prob_{\observ_k}(s)\\
			\prob_{\observ_k}(c)\\
			\prob_{\observ_k}(r)
		\end{pmatrix}\,,
	\end{equation*}
	the vector where each row denotes the probability that at day $k \in \N_0$ the weather is sunny, cloudy, rainy, respectively.
	We investigate the long term behavior of $\vp_{k}$:	
	\begin{equation} \label{eq:mrf}
		\vp_{k+1}^T = \vp_{k}^T \probm.
	\end{equation}
	This iteration is called a \emph{Markov-process}. \index{Markov-process}
    In this example $\probm$ is not dependent on $k \in \N_0$. We call \autoref{eq:mrf} therefore \emph{homogeneous}.
\end{example}

In the example we have been considering a Markov-model where we have three different states $s,c,r$. In general a Markov-model is defined as follows:
\begin{definition}[Homogeneous Markov-model]\label{de:mc}
	A homogeneous \emph{Markov-model}\index{Markov-model!homogeneous} $(\Omega_\obse,\probm)$ of order ${\tt K}$ consists of ${\tt K}$ different states $\Omega_\obse=\set{\obse_1,\ldots,\obse_{\tt K}}$ and conditional probabilities
	\begin{equation} \label{eq:uebergangsmatrix}
		\probm = \begin{pmatrix} p_{i,j} \end{pmatrix}_{i,j \in \set{1,\ldots,{\tt K}}}.
	\end{equation}
	It is common to abbreviate the Markov-model by $\probm$ instead of $(\Omega_\obse,\probm)$. 
	
	A homogeneous \emph{Markov-process}\index{Markov-process} is an iteration of the form \autoref{eq:mrf} in which $\probm$ is independent of $k$. 
\end{definition}
\begin{remark}
For a homogeneous Markov-model we have
\begin{equation*}
\prob_{\observ_{k+1}|\observ_k}(\obse_{k+1}|\obse_k)=
\prob_{\observ_{k+1}|\observ_k,\observ_{k-1}, \ldots, \observ_1}
(\obse_{k+1}|\obse_k,\obse_{k-1}, \ldots, \obse_1)\,,
\end{equation*} which means that only the previous observation influences the current observation.
\end{remark}

Now, we define probability vectors and stochastic matrices:
\begin{definition}[Probability vector and stochastic matrix]
	A vector $\vp \in \R^{\tt K}$ is called \emph{probability vector}\index{probability vector} if
	\begin{equation} \label{eq:randomv}
		p_i \in [0,1] \text{ and } \sum_{i=1}^{\tt K} p_i=1.
	\end{equation}
	A matrix $\probm$ that satisfies
	\begin{equation} \label{eq:smatrix}
		p_{i,j} \in [0,1] \text{ for all }i,j=1,\ldots,{\tt K} \text{ and } \sum_{j=1}^{\tt K} p_{i,j}=1 \text{ for all } i=1, \ldots, {\tt K}
	\end{equation}
	is called \emph{stochastic matrix}.\index{matrix!stochastic}
\end{definition}
For instance, the matrix $\probm$ in \autoref{ex:weather} is a stochastic matrix.
\begin{lemma}
	Let $\vp$ a probability vectors and $\probm$ a stochastic matrix, then $\vp^T P$ is also a probability vector.
	In particular, this means that the iterates of the Markov-process defined in \autoref{eq:mrf} are probability vectors.
\end{lemma}
\begin{proof} First, we note that all components of
	\begin{equation*}
		\vp^T P = \left(\sum_{i=1}^{\tt K} p_i p_{i,j} \right)_{j=1,\ldots,{\tt K}}
	\end{equation*}
	are non negative. Therefore, because of \autoref{eq:smatrix} we get
	\begin{equation*}
		\sum_{j=1}^{\tt K} (\vp^T P)_j = \sum_{j=1}^{\tt K} \sum_{i=1}^{\tt K} p_i p_{i,j} = \sum_{i=1}^{\tt K} p_i
		\underbrace{\sum_{j=1}^{\tt M} p_{i,j}}_{=1} = 1,
	\end{equation*}
	which gives the assertion.
\end{proof}
\begin{definition}\label{eq:randomvariable}
	$\vp \in \R^{\tt K}$ is the probability of the state vector $\vec{\obse} \in \Omega_\obse^{\tt K}$.
	A \emph{Markov-model}\index{process!Markov} (as in \autoref{eq:mrf}) is formally written as
	\begin{equation} \label{eq:mrf_sloppy}
		\vec{\obse}_{k+1}^T = \vec{\obse}_{k}^T \probm\;.
	\end{equation}
	Note however, that the probabilities $\vp$ of $\vec{\obse}$ are evolved and not $\vec{\obse}$ itself.
\end{definition}

The basis of analyzing Markov-models is the theorem of Frobenius-Perron:
\begin{theorem}[Frobenius-Perron] \label{th:Frobenius}\index{Theorem!Frobenius-Perron}
	Let $A \in \R^{{\tt K} \times {\tt K}}$ satisfying $a_{i,j} >0$ for all $i,j \in \set{1,\ldots,{\tt K}}$. Let $\sigma(A)$ denote the spectrum of the matrix $A$. That is, $\sigma(A)$ consists of all $\lambda \in \C$ such that there exists a vector $0 \neq \vx \in \R^{\tt K}$ such that $A\vx = \lambda \vx$ (see \cite{GolVan96}). Then, there exists $0 < r \in \R$, such that
	\begin{itemize}
		\item $\max\set{ \abss{{\vl}} : {\vl} \in \sigma(A)} = r$,
		\item and the eigenvalue $r$ has multiplicity $1$.
	\end{itemize}
	$r$ is called Frobenius-Perron-number.\index{Frobenius-Perron-number}
\end{theorem}
\begin{proof}
	See for instance \cite{Lue79}.
\end{proof}
From this theorem it follows with arguments from linear algebra (see \cite[Sec. 6.2]{Lue79}):
\begin{proposition}
	A stochastic matrix has an eigenvalue one of single multiplicity. Therefore, the absolute values of all other eigenvalues are strictly smaller than $1$.
\end{proposition}
In the following we define special Markov-models:
\begin{definition}[regular Markov-model]\index{Markov-model!regular} \label{de:Mcr}
	A Markov-model $(\Omega_\obse,\probm)$ (see \autoref{de:mc}) is called \emph{regular} if there exists $m \in \N$ such that $\probm^m > 0$. Note that this means that every matrix entry of $\probm^m$ is greater than $0$.
\end{definition}
Regular Markov-models have nice properties:
\begin{theorem}\label{th:markov}
	Let $\probm$ be a regular Markov-model.
	\begin{enumerate}
		\item Then there exists a unique probability vector $\vp > 0$ (means it is componentwise greater than $0$), such that
		\begin{equation*}
			\vp^T \probm = \vp^T\;.
		\end{equation*}
		\item Let $$\vec{e}_i = \begin{pmatrix} 0 \\ \vdots \\ 1 \\ \vdots \\ 0 \end{pmatrix},\qquad i=1,\ldots,{\tt K}$$ be the canonical basis in $\R^{\tt K}$. Then for every $i=1,\ldots,{\tt K}$
		\begin{equation*}
			\vp^T = \lim_{k \to \infty} \vec{e}_i^T \probm^k\;.
		\end{equation*}
		\item Let
		\begin{equation*}
			\ol{\probm} = \begin{pmatrix} \vp^T \\ \vdots \\ \vp^T \end{pmatrix} \in \R^{{\tt K} \times {\tt K}},
		\end{equation*}
		then
		\begin{equation*}
			\ol{\probm} = \lim_{k \to \infty} \probm^k.
		\end{equation*}
	\end{enumerate}
\end{theorem}
\begin{proof} see \cite[p 231]{Lue79}.
\end{proof}
We continue with the weather example \autoref{ex:weather}.
\begin{example} \label{ex:wetterII}	With the Markov-model $\probm$ from \autoref{eq:weather_matrix} we get
	\begin{equation*}
		\probm^{16} = \begin{pmatrix} 0.4 & 0.4 & 0.2\\0.4 & 0.4 & 0.2\\0.4 & 0.4 & 0.2
		\end{pmatrix}= \begin{pmatrix} \vp^T \\ \vdots \\ \vp^T \end{pmatrix} = \ol{\probm} .
	\end{equation*}
	Therefore we have on average 40\% sunny and cloudy days, respectively, and 20\% rainy days.
	In particular the weather model from \autoref{ex:weather} is a regular Markov-model (see \autoref{de:Mcr}). However, we note that we cannot apply the standard theorem of Frobenius-Perron, \autoref{th:Frobenius}, because the matrix $\probm$ from \autoref{eq:weather_matrix} has a zero entry.
\end{example}

\section{Non homogeneous Markov-models} \label{ss:nsmp}
We generalize the idea of homogeneous Markov-models to non homogeneous ones.
\begin{definition}[Non homogeneous Markov-model]\label{de:mcns}
	Let $(\Omega_\obse,\probm_k)_{k \in \N_0}$ be a family of \emph{Markov-models} of order ${\tt K}$. Every Markov-model consists of the ${\tt K}$ different states $\Omega_\obse=\set{\obse_1,\ldots,\obse_{\tt K}}$ and for every $k \in \N_0$ the conditional probabilities are given by
	\begin{equation} \label{eq:uebergangsmatrix_ns}
		\probm_k = \begin{pmatrix} p_{i,j;k} \end{pmatrix}_{i,j \in \set{1,\ldots,{\tt K}}} \text{ for all } k \in \N_0.
	\end{equation}
	The probabilities tell us the probability that the process is at state $\obse_i$ at instance $k+1$ if the process is at state $\obse_j$ at instance $k \in \N_0$.
	
	A \emph{non homogeneous Markov-process}\index{Markov-process!non homogeneous} is an iteration of the form
	\begin{equation} \label{eq:mrf_ns}
		\vp_{k+1}^T = \vp_{k}^T \probm_k\;.
	\end{equation}
\end{definition}
Again, the entries of $\probm_k$ are conditional probabilities of two random vectors
$\observ_{k}$ and $\observ_{k+1}$. That is
\begin{equation*}
p_{i,j;k} = \probm_{\observ_{k+1}|\observ_k;k}(\obse_i|\obse_j),
\end{equation*}
which means that on the next day (``$k+1$'') we have the weather condition $\obse_i \in \Omega_\obse$ if today (``$k$'') we observe the weather condition $\obse_j \in \Omega_\obse$. One could for instance model the influence of weather by \emph{Geoengineering}\index{Geoengineering} (the term has been created in 1977, see \cite{Schn10}):
\begin{example}
	We continue with the weather \autoref{ex:weather}. On the first day we do not employ Geoengineering but on the second, which could be modeled as follows:
		\begin{equation} \label{eq:weather_matrix_GE}
		\probm_0 = \begin{pmatrix}
			\frac{1}{2} & \frac{1}{2} & 0\\
			\frac{1}{2} & \frac{1}{4} & \frac{1}{4} \\
			0 & \frac{1}{2} & \frac{1}{2} \end{pmatrix} \text{ and } 	
			\probm_1 = \begin{pmatrix}
			\frac{1}{4} & \frac{1}{2} & \frac{1}{4}\\
			\frac{1}{2} & \frac{1}{4} & \frac{1}{4} \\
			\frac{1}{4} & \frac{1}{4} & \frac{1}{2} \end{pmatrix}
			\;.
	\end{equation}
	The second matrix models the fact of Geoengineering to produce more rain.
\end{example}
The modeling of non homogeneous Markov-model is quite challenging due to the amount of data that is needed (all matrices $\probm_k$, $k \in \N_0$).

\emph{Hidden Markov-models}, as discussed below, model the conditional probabilities $\probm_k$ in dependence of $k \in \N_0$ from hidden state variables. The hidden states influence the observations but, in general, not vice versa.

\section{Hidden Markov-models} \label{se:hmm}\index{Markov-model!hidden}
In \autoref{ss:stationaryMc} and \autoref{ss:nsmp} we considered the iteration of observations (like the weather conditions). Now, we consider a potential complicated interaction between observations and hidden states. We motivate the ideas with the mugs and vases examples, which have been considered in \cite{Por88}:
\begin{example} \label{ex:cupsandvase}
  One has a mug (see \autoref{fig:cupsmug1}), denoted by ``0'', filled with balls marked with ``1'' and ``2''. One draws from the mug a ball and the number of the ball tells us from which vase we draw next. The outcome is either a white or black ball. The color is noted. The observation of the experiment is the color of the ball drawn from the vase. The complete experiment consists of two subsequent random experiments, namely drawing from the mug first and then drawing from the according vase; see \autoref{fig:cupsmug1} and \autoref{fig:cupsmug2}.
  	\begin{figure}[h!]
  	\centering
  	\includegraphics[width=\textwidth]{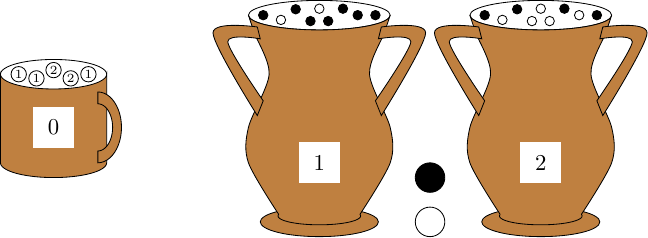}
  	\caption{We have a mug, named MUG(0), containing balls numbered with $1$ and $2$. Moreover, we have two vases, named Vase(1) and Vase(2), where each of them contains black and white balls.}
  	\label{fig:cupsmug1}
  \end{figure} 	
  In \autoref{fig:cupsmug2} a simple hidden Markov-model is shown.
  \begin{figure}[h!]
  	\centering
  	\includegraphics[width=\textwidth]{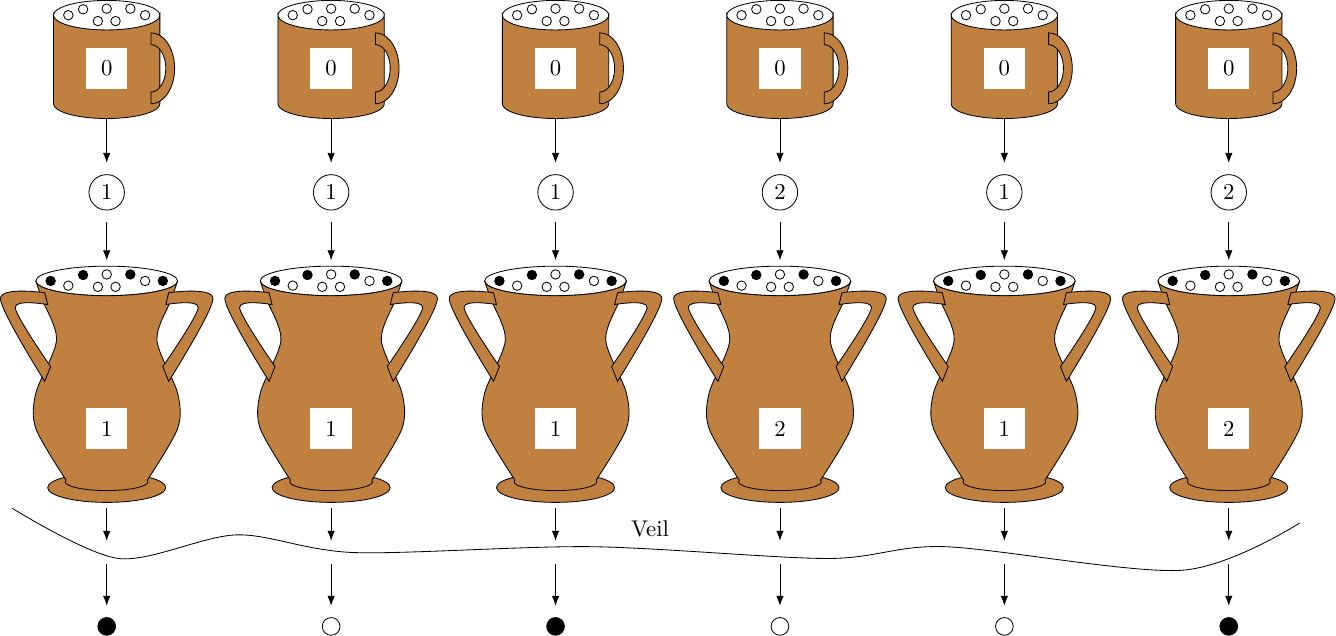}
  	\caption{In the game we draw a ball from the mug. This ball has either the value $1$ or $2$. Then we draw from the vase with the number of the drawn ball from the mug. The observation is a black or white ball. The veil symbolizes that we observe the outcome of two stochastic experiments for drawing from the mug and the vases. Finally, the only relevant observation are the relative counts of observed black and white balls (see \autoref{eq:game2}). In this game we have a hidden state, which are the sequence of drawn states ``1'' and ``2''. This state is not interacting with other states. Therefore, we call it $1$-state hidden-Markov- model. In \autoref{fig:cupsmug3} we consider interactions of hidden states, which gives us a $2$-state hidden Markov-model. The term veil refers to the fact that the observations are a composition of two random experiments, where the single ones cannot be observed.}
  	\label{fig:cupsmug2}
  \end{figure}

  We denoted by
  \begin{equation} \label{eq:st}
  	\begin{aligned}
  		\staterv : \Omega_\state := \set{1,2} &\to \R,\\
  		              \state &\mapsto \state
  	\end{aligned}
  \end{equation}
  the random variable (see \autoref{de:ap:random_variable}) of the \emph{state space}, which describes drawing from the mug. Similarly, we denote the random variable of the  \emph{observation space} by
  \begin{equation} \label{eq:o}
  	\begin{aligned}
  		\observ: \Omega_\obse:=\set{b,w} &\to \R\,,\\
  		\obse &\mapsto \left\{\begin{array}{rcl}
  			1 &\text{ if } & \obse=b\\
  			2 &\text{ if } & \obse=w
  		\end{array} \right.
  	\end{aligned}
  \end{equation}
  As in \autoref{ex:ap:dice} we use $b$, $w$ synonymously for $1$,$2$, respectively.

  From Bayes' theorem (see \autoref{th:Bayes}) it follows that
  \begin{equation}\label{eq:ex_conc}
  	\begin{aligned}
  	\prob_\observ(b)
  	=& \quad \prob_{\observ|\staterv}(b|1)
  	         \prob_\staterv(1) +
  		     \prob_{\observ|\staterv}(b|2)
  		     \prob_\staterv(2)\;.
  	\end{aligned}
  \end{equation}
  In other words, given $\prob_\staterv(1), \prob_\staterv(2)$ and $\prob_{\observ|\staterv}(b|1)$, $\prob_{\observ|\staterv}(b|2)$ we can determine the average count of black ball observations. The parameter vector
  \begin{equation}\label{eq:lambda}
  	\vl = (\prob_\staterv(1),\prob_\staterv(2),\prob_{\observ|\staterv}(b|1),\prob_{\observ|\staterv}(w|1),\prob_{\observ|\staterv}(b|2),\prob_{\observ|\staterv}(w|2))
  \end{equation}
  is a complete descriptor for the observable random vector $\observ$. We are using the \emph{reduced} parameter vector
  \begin{equation}\label{eq:lambdar}
  	\vl_r = (\prob_\staterv(1),\prob_{\observ|\staterv}(b|1),\prob_{\observ|\staterv}(b|2))\,,
  \end{equation}
  which contains all information of $\vl$, because
  \begin{equation*}
  \prob_\staterv(2)=1-\prob_\staterv(1), \prob_{\observ|\staterv}(b|i)=1-\prob_{\observ|\staterv}(w|i)
  \text{ for } i=1,2.
  \end{equation*}
  \end{example}
  In the following we consider a repeated experiment:
  \begin{example}
  We repeat the experiment $\no$-times, see \autoref{fig:cupsmug2}. We use the random variables
  \begin{equation*}
  	\begin{aligned}
  		\vec{\observ} := (\observ_1,\ldots,\observ_\no) : \Omega_\obse^\no \to \R^\no \text{ and }
  		\vec{\staterv}:= (\staterv_1,\ldots,\staterv_\no) : \Omega_\state^\no \to \R^\no\,,
  	\end{aligned}
  \end{equation*}
  which maps every $\no$-times repeated experiment to an observation and a hidden state (a vector in $\R^\no$), respectively.

  In this example the states $\state_i$ and $\state_{i+1}$ are independent from each other, which simplifies the computations. In the forthcoming \autoref{ex:lambda_conc2} this assumptions does not hold. Applying Bayes' theorem, since every observation is independent, we get for a given observation $\vec{\obse} \in \set{b,w}^\no$,
  \begin{equation} \label{eq:game}
  	\begin{aligned}
  		\prob_{\vec{\observ}} (\vec{\obse})
  		= & \prod_{\ell=1}^{\no} \prob_{\observ_\ell} (\obse_\ell)
  		=  \prod_{\ell=1}^{\no}
  		    \biggl( \prob_{\observ_\ell,\staterv_\ell}(\obse_\ell,1) + \prob_{\observ_\ell,\staterv_\ell}(\obse_\ell,2) \biggr) \\
  		= & \prod_{\ell=1}^{\no}\biggl(
  		\prob_{\observ_\ell|\staterv_\ell}(\obse_\ell|1)
  		\prob_{\staterv_\ell}(1) +
  		\prob_{\observ_\ell|\staterv_\ell}(\obse_\ell|2)
  		\prob_{\staterv_\ell}(2) \biggr)
%  		\\
%  		= & \prod_{\ell=1}^{\no}\bigl( p_1({\obse}_\ell) \prob_\staterv(1) + p_2({\obse}_\ell) \prob_\staterv(2) \bigr)
  		\;.
  	\end{aligned}
  \end{equation}
  Here $\prob_{\observ_\ell,\staterv_\ell}$ denote the joint probability and $\prob_{\observ_\ell|\staterv_\ell}$ denotes the conditional probability (see \autoref{de:prob_cond}).

  Each term in the product of the right hand side is either
  \begin{equation*}
  	\prob_{\observ|\staterv}(b|1) \prob_\staterv(1) + \prob_{\observ|\staterv}(b|2) \prob_\staterv(2) \quad \text{ or } \quad \prob_{\observ|\staterv}(w|1) \prob_\staterv(1) + \prob_{\observ|\staterv}(w|2) \prob_\staterv(2) \,,
  \end{equation*}
  such that we get from \autoref{eq:game} with
  \begin{equation*}
  	k_b = \abs{\set{{\obse}_\ell=b}} \quad \text{ and } \quad k_w = \abs{\set{{\obse}_\ell=w}}
  \end{equation*}
  that
  \begin{equation} \label{eq:game2}
  	\begin{aligned}
  	\prob_{\vec{\observ}} (\vec{\obse}) =&
  		\bigl( \prob_{\observ|\staterv}(b|1) \prob_\staterv(1) + \prob_{\observ|\staterv}(b|2) \prob_\staterv(2) \bigr)^{k_b} \\
  		&\qquad
  		\bigl( \prob_{\observ|\staterv}(w|1) \prob_\staterv(1) + \prob_{\observ|\staterv}(w|2) \prob_\staterv(2) \bigr)^{k_w}\;.
  	\end{aligned}
  \end{equation}
\end{example}
In the following we make the example concrete to get a feeling on the dimensions:
\begin{example} \label{ex:lambda_conc}
	To make it concrete, we assume that in the mug we have 40\% balls marked with ``1'' and 60\% balls marked with ``2'' and we assume that Vase(1) has 20\% black and Vase(2) has 50\% black. Therefore, for the particular example \autoref{fig:cupsmug2} we have
	\begin{equation} \label{eq:lambda1}
		\begin{aligned}
			\vl &= (\prob_\staterv(1),\prob_\staterv(2),\prob_{\observ|\staterv}(b|1),\prob_{\observ|\staterv}(w|1),\prob_{\observ|\staterv}(b|2),\prob_{\observ|\staterv}(w|2))\\
			&= (0.4,0.6,0.2,0.8,0.5,0.5) \text{ and }\\
			\vl_r &= (\prob_\staterv(1),\prob_{\observ|\staterv}(b|1),\prob_{\observ|\staterv}(b|2)) = (0.4,0.2,0.5)\,,
		\end{aligned}
	\end{equation}
	In thus case it follows from \autoref{eq:ex_conc} that
	\begin{equation}\label{eq:ex_conc_conc}
		\prob_\observ(b) = \prob_{\observ|\staterv}(b|1) \prob_\staterv(1) + \prob_{\observ|\staterv}(b|2) \prob_\staterv(2) = 0.2 \cdot 0.4 + 0.5 \cdot 0.6 = 0.38.
	\end{equation}
	In the example of \autoref{fig:cupsmug2} $\no=6$. The \emph{sampling balls} are $\vec{\state}=(1,1,2,1,2,1)$ (state spaces) and the observation vector is $\vec{\obse}=(b,w,b,w,w,b)$.
	Since we observe 50\% black and 50\% white balls drawn, that
	\begin{equation} \label{eq:game2s} \begin{aligned}
			 \prob_{\vec{\observ}}(b,w,b,w,w,b)
			&=
			\left( 0.2 \cdot 0.4 + 0.5\cdot 0.6 \right)^3
			\left( 0.8 \cdot 0.4 + 0.5\cdot 0.6\right)^3 \\
			&= 0.0131\;.
			\end{aligned}
	\end{equation}
    We have calculated the probability that the sequence $b,w,b,w,w,b$ is observed (in this order).
\end{example}
Now, we put \autoref{ex:cupsandvase} in perspective with inverse problems of the form \autoref{eq:op}.
\begin{example}
We define the \emph{random forward operator}\index{forward operator!random}, which maps a given vector of
probabilities $\vl_r \in [0,1]^3$ onto the set of pairs, consisting of possible observations and associated probabilities, 
 $$\mathcal{O}:= \bigcup_{\vl_r \in [0,1]^3} \mathcal{O}_{\vl_r} \text{ where }  \mathcal{O}_{\vl_r} = \set{\left(\set{b,w}^\no,\prob_{\vl_r}^\no\right)}\;.$$
 Note that for every $\set{b,w}$-sequence we also have a probability. In practice the probability is not given and defined from observations. Thus it is a conditional probability.
Then, in formulas, the operator reads as follows
  \begin{equation} \label{eq:random_op}
  	\begin{aligned}
  		\opo_r: [0,1]^3 &\to 2^\mathcal{O}\;.\\
  		         \vl_r &\mapsto \mathcal{O}_{\vl_r}
  	\end{aligned}	
  \end{equation}
  This means that if we have given the parameter $\vl_r$, which consists of the numbers of balls in the mug marked with $1$ and $2$ and the numbers of black and white balls in the vases, then we repeat the game and record the observed drawings and probabilities.

  Moreover, we define the \emph{reduced random observation operator}\index{forward operator!random, observed}
  \begin{equation}\label{eq:roex1}
  	\opo_{r,r}^{\vec{\obse}}: [0,1]^3 \to [0,1]\,,
  \end{equation}
  which maps a reduced parameter $\vl_r$ onto the relative observations of black ($b$) balls (the observation of $w$ (white) has probability one minus the probability of observing $b$). Note, that here the outcome is not a random variable but an observed probability. In other words, it maps the reduced parameter $\vl_r$ onto the \emph{expected probability} of observing the black balls.\index{probability!expected} 

  In a real game the conditional probabilities
  \begin{equation}\label{eq:cond1}
  	\prob_{\observ|\staterv}(b|1) := \prob_{\observ_\ell|\staterv_\ell}(b|1) \text{ and } \prob_{\observ|\staterv}(b|2):=\prob_{\observ_\ell|\staterv_\ell}(b|2)\,,
  \end{equation}
  and the probabilities
  \begin{equation} \label{eq:cond2}
  	\prob_\staterv(1):=\prob_{\staterv_\ell}(1) \text{ and } \prob_\staterv(2):=\prob_{\staterv_\ell}(2)
  \end{equation}
  are \emph{not} known. We assume the probabilities to be independent of $\ell = 1,\ldots,\no$ and note that the $0$ in $p_{0,i}$ refers to MUG(0).

  The \emph{random inverse problem}\index{inverse problem!random} consists in finding from the expected probability and the observation $\obse$ the reduced parameter vector $\vl_r$.
\end{example}
In the following we summarize the essential some essential terms in an abstract from, which have been used the previous example:
\begin{definition}[Prior and posterior probability] \label{de:pcp} \index{probability!prior}\index{probability!posterior}
	We denote the set of \emph{observables}\index{observable} by
	$$\Omega_\obse = \set{\obse_1,\ldots,\obse_{\tt O}}$$
	and the set of attainable \emph{states}\index{states} by
	$$\Omega_\state = \set{\state_1,\ldots,\state_{\tt M}}\,.$$
	Let $\state \in \Omega_\state$ be a state of $\staterv_\ell$ and let $\vec\observ =(\obse_1,\ldots,\obse_\no)$ by a vector observation random variable, consisting again of $\no$ identical random variables, and let $\obse$ by a state of $\observ_\ell$.
	\begin{itemize}
		\item Then the quantity
	          \begin{equation} \label{eq:prior}
		         \prob_{\staterv_\ell}(\state)
	          \end{equation}
	          denotes the \emph{prior probability} of state $\state$ at instance $\ell \in \set{1,\ldots,{\tt M}}$.
	    \item The \emph{conditional probability} of an observation $\obse$ at instance $\ell$ is defined as
	          \begin{equation} \label{eq:cpob}
	          	\prob_{\observ_\ell|\staterv_\ell}(\obse|\state)\;.
	          \end{equation}
	          Note that we assume that ${\tt M} \leq \no$ and that we expand the states to a vector of size $\no$, if there is no hidden state, which means that $\prob_{\observ_\ell|\staterv_\ell}(\obse|\state)=1$, which makes \autoref{eq:cpob} well defined.
	    \item The \emph{posterior probability} of state $\state$ at $\ell \in \set{1,\ldots,\no}$ is given by
	          \begin{equation}\label{eq:posterior}
		         \gamma_\ell(\state) :=
		         \prob_{\staterv_\ell | \vec{\observ}} (\state|\vec{\obse})\;.
	          \end{equation}
\end{itemize}
\end{definition}
\begin{example} \label{ex:markov_practice}
	In \autoref{ex:cupsandvase}
	$$\Omega_\obse = \set{b,w}, \Omega_\state = \set{1,2}\;.$$
	We use the two probability spaces $(\Omega_\state,2^{\Omega_\state},\prob_\state)$ and $(\Omega_\obse,2^{\Omega_\obse},\prob_\obse)$, where $\prob_\state$ is the prior probability (see \autoref{eq:prior}) and (see \autoref{eq:game})
	\begin{equation*}
		\begin{aligned}
		 \prob_\obse(\obse) := \prob_\observ(\obse)\;.
		\end{aligned}
	\end{equation*}
	Moreover, we have
	\begin{equation*}
	    \prob_\state(\state) := \prob_\staterv(\state)\;.
	\end{equation*}
	It follows from Bayes' theorem and from \autoref{eq:prior} and \autoref{eq:cond1} that
		\begin{equation} \label{eq:concrete_ex}
			\begin{aligned}
				\gamma_\ell(\state)
				&=\frac{\prob_{\vec{\observ},\staterv_\ell}(\vec{\obse},\state)}
				{\prob_{\vec{\observ}}(\vec{\obse})} =\prob_{\observ_\ell,\staterv_\ell}(\obse,\state)
				\frac{\prod_{t \neq \ell} \prob_{\observ_t}(\obse_t)}{\prob_{\vec{\observ}}(\vec{\obse})}
				= \prob_{\observ_\ell|\staterv_\ell}(\obse_\ell|\state)  \frac{\prob_{\staterv_\ell}(\state)}{\prob_{\observ_\ell}(\obse_\ell)} \\
				&= \frac{\prob_{\observ_\ell|\staterv_\ell}(\obse_\ell|\state)\prob_{\staterv_\ell}(\state)}
				{\prob_{\observ_\ell|\staterv_\ell}(\state_\ell|1)\prob_{\staterv_\ell}(1) +
				 \prob_{\observ_\ell|\staterv_\ell}(\state_\ell|2)\prob_{\staterv_\ell}(2)}\;.
			\end{aligned}
		\end{equation}
		In the computation of \autoref{eq:concrete_ex} we used that the states $\state_\ell$ can be observed independently. We will derive later more general formulas, when the $\state$ need to be considered dependent (see \autoref{eq:ps}.
		
In particular we also have
\begin{equation*}
	\gamma_\ell(1) + \gamma_\ell(2) =1\;.
\end{equation*}
Thus, if the true parameter $\vl$ is known, we can make with $\gamma_\ell$ an educated guess whether a ball from the mug with a mark $1$ or $2$ is chosen. In practice $\vl$ is not know and has to be determined from maximum-likelihood estimation. \index{maximum-likelihood estimation}

For the particular \autoref{ex:lambda_conc} we get for $\state = 1$ and $\ell=6$, meaning that $\obse_6=b$, that:
\begin{equation*}
	\begin{aligned}
		\gamma_6(1)
		&=\frac{\prob_\staterv(1) \prob_{\observ|\staterv}(b|1)}{\prob_{\observ|\staterv}(b|1) \prob_\staterv(1) + \prob_{\observ|\staterv}(b|2) \prob_\staterv(2)} =
		\frac{0.4 \cdot 0.2}{0.2\cdot 0.4 + 0.5 \cdot 0.6}
		=0.2105 \text{ and } \\
		\gamma_6(2)
		&=\frac{\prob_\staterv(2) \prob_{\observ|\staterv}(b|2)}{\prob_{\observ|\staterv}(b|1) \prob_\staterv(1) + \prob_{\observ|\staterv}(b|2) \prob_\staterv(2)} =
		\frac{0.6 \cdot 0.5}{0.2\cdot 0.4 + 0.5 \cdot 0.6}
		=0.7895\;.
	\end{aligned}
\end{equation*}
$\gamma_6(2)$ is higher than $\gamma_6(1)$, which mans that at the $6$ draw we have with high probability a state ``2''. This coincides with the experiment \autoref{fig:cupsmug2}.
\end{example}
We continue with a more complicated game with mugs and vases:
\begin{example} \label{ex:lambda_conc2}
This game is also considered in \cite{Por88}. The further complication, in relation to \autoref{ex:cupsandvase}, (see \autoref{fig:cupsmug4}) is that we have now three mugs, which influence the choice of the vases (see \autoref{fig:cupsmug3}): MUG(0) is only used once to initialize the process. According to the drawn ball from the mug we choose \emph{both} the vase and the mug, where we draw next. This is an \emph{order one} or \emph{2-state} hidden Markov-model.\index{Markov-model!order one}\index{Markov-model!2-state} In this model one has an interaction between two hidden states.
  	\begin{figure}[h!]
	\centering
	\includegraphics[width=\textwidth]{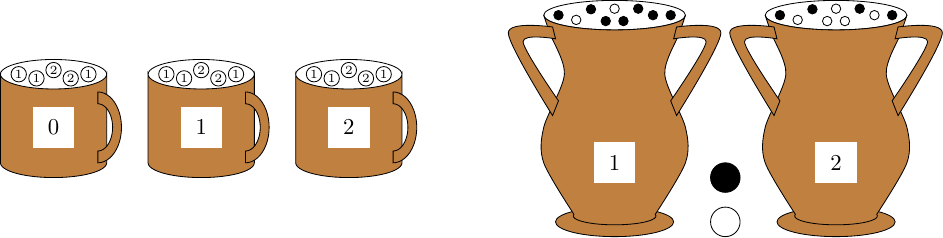}
	\caption{Here we draw from three mugs. The drawing from the first mug influences the next drawing from the mug and thus the game has a memory.}
	\label{fig:cupsmug4}
\end{figure} 	
\begin{figure}[h!]
	\centering
	\includegraphics[width=\textwidth]{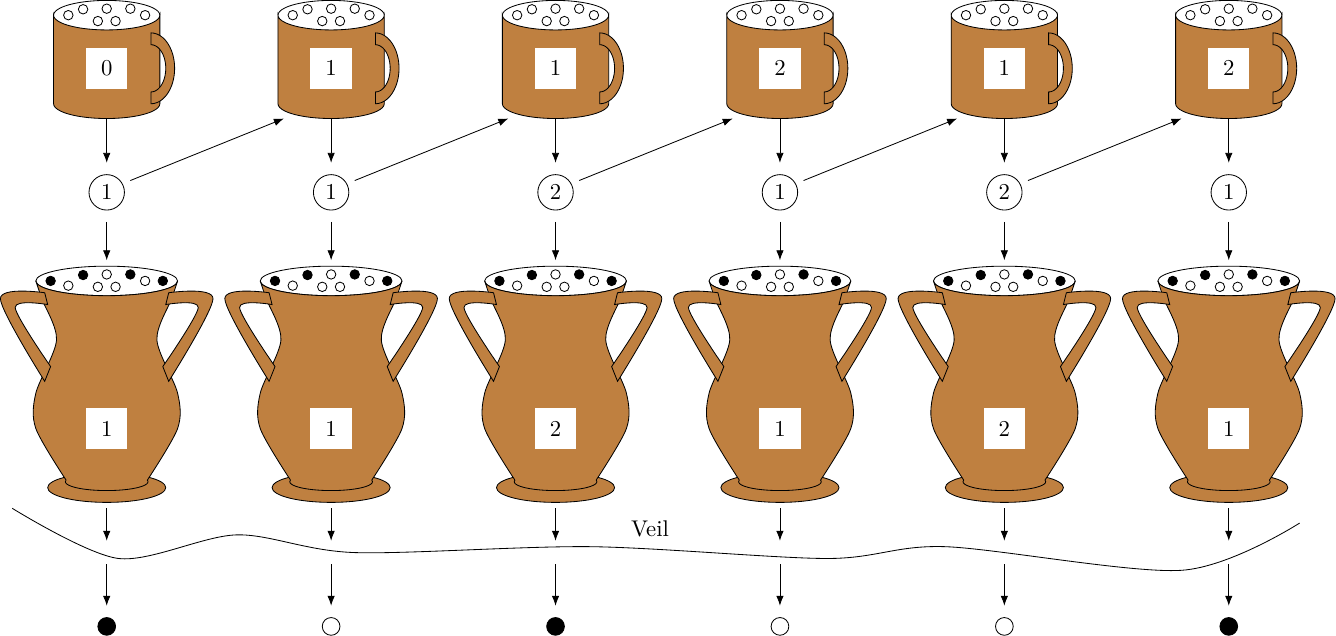}
	\caption{
	\label{fig:cupsmug3} The observations $\vec{\obse}=(b,w,b,w,w,b)$ are white or black balls from vase one or two. The hidden state spaces, $\vec{\state} = (1,1,2,1,2,1)$, are the numbers drawn from the mugs.}
\end{figure}
We have a state random variable
\begin{equation*}
	\vec{\staterv} : \Omega_\state^\no = \set{1,2}^\no \to \R^\no
\end{equation*}
which models drawing from the mug. $\staterv_1$ is special, because in the game it is fixed that one pulls from MUG(0). In the next mug drawings one draws from either MUG(1) or MUG(2). The random observations are independent, because we assume that drawn ball from the vase is put back.

The output vector $\vec{\obse}$ and the hidden states $\vec{\state}$ (see \autoref{fig:cupsmug3}) are related as in \autoref{eq:ex_conc}: In this case we have a parameter vector
\begin{equation} \label{eq:lambda_hm}
	\begin{aligned}
	\vl =& (\prob_{\staterv}^0(1),\prob_\staterv^0(2),\prob_\staterv^1(1),\prob_\staterv^1(2),\prob_\staterv^2(1),\prob_\staterv^2(2),\\
	& \qquad ,\prob_{\observ|\staterv}(b|1),\prob_{\observ|\staterv}(w|1),\prob_{\observ|\staterv}(b|2),\prob_{\observ|\staterv}(w|2))\,,
	\end{aligned}
\end{equation}
which contains the probabilities $\prob_\staterv^\rho(i)$ to draw a ball marked with $i=1,2$ from MUG($\rho$).
Note that for $i=1,2$ and $\rho=1,2$
\begin{equation*}
	\prob_\staterv^0(1) = \prob_{\staterv_1}^0(1)\,, \quad
	\prob_\staterv^\rho(i) = \prob_{\staterv_\ell}^\rho(i) \text{ for } \ell=2\,\ldots,\no
\end{equation*}
and $j=b,w$ and $i=1,2$
\begin{equation*}
	\prob_{\observ|\staterv}(j|i) = \prob_{\observ_\ell|\staterv_\ell}(j|i) \text{ for } \ell=2\,\ldots,\no
\end{equation*}
denote the probabilities of drawing a black or white ball from VASE($i$), $i=1,2$.
In other word the probabilities are independent of the number of the experiment and therefor $\vl$ completely determines the experiment. For further considerations we denote
\begin{equation} \label{eq:further}
	\vec{p}_0 = (\prob_\staterv^0(1),\prob_\staterv^0(2)), \quad \probm = \begin{pmatrix}
		\prob_\staterv^1(1) & \prob_\staterv^1(2) \\
		\prob_\staterv^2(1) & \prob_\staterv^2(2)
	\end{pmatrix}\,,
\end{equation}
which are a probability vector and a stochastic matrix (see \autoref{eq:randomv} and \autoref{eq:smatrix}).
The game is completely described by the reduced parameter vector
\begin{equation} \label{eq:lambda_hmr}
	\vl_r = (\prob_\staterv^0(1),\prob_\staterv^1(1),\prob_\staterv^2(1),\prob_{\observ|\staterv}(b|1),\prob_{\observ|\staterv}(b|2))\;.
\end{equation}
As in the previous example let
 $$\mathcal{O}:= \bigcup_{\vl_r \in [0,1]^5} \mathcal{O}_{\vl_r} \text{ where }  \mathcal{O}_{\vl_r} = \set{\bigl(\set{b,w}^\no,\prob_{\vl_r}^\no\bigr)}\;.$$
We define the set-valued random forward operator as
\begin{equation} \label{eq:random_op2}
\begin{aligned}
	\opo_r: [0,1]^5 &\to \mathcal{O}\,,\\
	\vl_r &\mapsto \mathcal{O}_{\vl_r}
\end{aligned}	
\end{equation}
and the \emph{reduced random observation operator}
\begin{equation}\label{eq:roex2}
\opo_{r,r}^\obse: [0,1]^5 \to [0,1]\,,
\end{equation}
which maps onto the expected probability for observing black balls.
\end{example}
The games in \autoref{ex:cupsandvase} and \autoref{ex:lambda_conc2} differ by the dependencies of the states.
In fact in \autoref{ex:cupsandvase} the states are not dependent on each other. In \autoref{ex:lambda_conc2} the state at the $k$-th experiment is dependent on the $k-1$-th state. More generally, we define
\begin{definition} \label{de:hMmstates}
	An \emph{$m$-state hidden Markov-model} selects a successor state of $m-1$--predecessor states.\index{Markov-model!$m$-state hidden}
\end{definition}
In view of this \autoref{ex:cupsandvase} is an example of a $1$-state hidden Markov-model, while \autoref{ex:lambda_conc2} is an example a $2$-state hidden Markov-model.

We systematically visualize a general Markov-model of order one in \autoref{fig:cupsmug5} (would correspond to $\no$ mugs):
\begin{figure}[h!]
	\centering
	\includegraphics[width=0.6\textwidth]{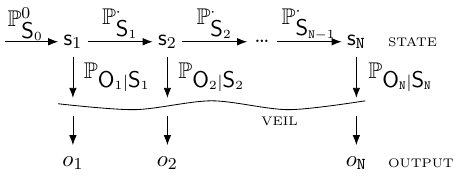}
	\caption{
		\label{fig:cupsmug5} Simulation of a Markov-model with ${\tt M}$ hidden states $\state_i$, $i=1,\ldots,{\tt M}$ and $\no$ observations $\obse_i$ for $i=1,\ldots,\no$. We assume that we have ${\tt M} \leq \no$ hidden states.}
\end{figure}
\begin{remark}
	In \autoref{fig:cupsmug2} no state is connected, thus it is a one layer Markov-model. In \autoref{fig:cupsmug3} two neighboring states are connected, and thus it is called two layer Markov-model. In general, when ${\tt M}$-layers are connected, it is called ${\tt M}$-layer Markov-model.\index{Markov-model!$\state$-layer} We assume that (see \autoref{fig:cupsmug5})\index{Markov model!two states hidden}\index{Markov model!order one}
	\begin{equation*}
		{\tt M} \leq \no\;.
	\end{equation*}
	We can point out the main difference between the games described in \autoref{fig:cupsmug2} and \autoref{fig:cupsmug3}, respectively. While in the game from \autoref{fig:cupsmug2} the states are independent, they are influenced from all prior states in \autoref{fig:cupsmug3}.
\end{remark}
\begin{convention}
	In the following we take ${\tt M}=\no$ and we justify this by increasing the number of hidden states artificially. Of course this is not a practical way for computations.
\end{convention}

In the following we define the conditional properties of states of a repeated experiment: We recall that $\vl$ is given and fixed. Moreover, we make the following assumption, which we have used implicitly in the Cups and Vases \autoref{ex:cupsandvase} and \autoref{ex:lambda_conc2}.
\begin{assumption}\label{ass:obscure}
	The joint conditional probability $\prob_{\observ_\ell,\staterv_\ell|\staterv_{\ell-1}}$ satisfies
	\begin{equation}
		\label{eq:obscure}
		\prob_{\observ_\ell,\staterv_\ell|\staterv_{\ell-1}}(\obse,\state|\stater) =
		\prob_{\observ_\ell|\staterv_\ell}(\obse|\state) \cdot \prob_{\staterv_\ell,\staterv_{\ell-1}}(\state|\stater) \text{ for all } \ell=2,\ldots,\no\;.
	\end{equation}
\end{assumption}
Note that by Bayes' theorem \autoref{th:Bayes} it follows that for $\ell=1,\ldots,\no$:
\begin{equation} \label{eq:obscure1}
	\begin{aligned}
		\prob_{\observ_\ell,\staterv_\ell|\staterv_{\ell-1}}(\obse,\state|\stater) &=
		\frac{\prob_{\observ_\ell,\staterv_\ell,\staterv_{\ell-1}}(\obse,\state,\stater)}
		     {\prob_{\staterv_\ell,\staterv_{\ell-1}}(\state,\stater)} \cdot
		     \frac{\prob_{\staterv_\ell,\staterv_{\ell-1}}(\state,\stater)}{\prob_{\staterv_{\ell-1}}(\stater)} \\
		     &=
		     \frac{\prob_{\observ_\ell,\staterv_\ell,\staterv_{\ell-1}}(\obse,\state,\stater)}
		     {\prob_{\staterv_\ell,\staterv_{\ell-1}}(\state,\stater)} \cdot
		     \prob_{\staterv_\ell,\staterv_{\ell-1}}(\state|\stater)\;.
	\end{aligned}	
\end{equation}
\autoref{eq:obscure} and \autoref{eq:obscure1} match, if $\observ_\ell$ is only depended on the $\state_\ell$, which is the case in the two examples, \autoref{ex:cupsandvase} and \autoref{ex:lambda_conc}.

\begin{definition}
 The \emph{conditional probability} between successive instances $\ell \in \set{2,\ldots,\no-1}$ of $\vec{\staterv}$ is defined by
	\begin{equation} \label{eq:cp}
		p_{\state_{\ell-1},\state_\ell} := \prob_{\staterv_\ell|\staterv_{\ell-1}}(\state_\ell|\state_{\ell-1})\;.
	\end{equation}
	In a Markov-model, we define conditional probabilities between arbitrary instance as
		\begin{equation}\label{eq:ps}
		\begin{aligned}	
			\prob_{\staterv_\no}(\state_\no) &:=
			\prob_{\staterv_\no|\staterv_{\no-1}}(\state_\no|\state_{\no-1}) \cdots
			\prob_{\staterv_2|\staterv_1}(\state_2|\state_1)
			     \cdot \prob_{\staterv_1}(\state_1) \\
			     &= \prob_{\staterv_1}(\state_1) \prod_{\ell=2}^\no p_{\state_{\ell-1},\state_\ell}\;.
		\end{aligned}
	\end{equation}
\end{definition}

We compute now prior, posterior and conditional probabilities as defined in \autoref{de:pcp} when the states of the experiment are dependent (as it is for instance in \autoref{ex:lambda_conc2}). Note, however,
that now states are dependent. For the calculation we use the following lemma:
\begin{lemma}\label{le:dependent}
	Let be given a Markov-model with parameter vector
	\begin{equation} \label{eq:lambda_recon}
		\vl = \bigl( \vp, \probm \bigr)\,,
		\end{equation}
		where
		$\vp = (\prob_{\staterv_1}(\state))_{\tiny \state \in \Omega_\state}$ is a prior probability, which is independent of the rest of the hidden Markov-model; see for instance \autoref{ex:lambda_conc}, where we draw first from MUG(0) to initialize the experiment. MUG(0) is never used afterwards again in the experiment. Moreover,
	\begin{equation*}
		\probm = \begin{pmatrix} p_{i,j} \end{pmatrix} = \begin{pmatrix} p_{\state_i,\state_j} \end{pmatrix}
		\in [0,1]^{{\tt M} \times {\tt M}}
	\end{equation*} denotes a stochastic matrix, which denote the probabilities of transferring a state $\observ_\ell$ to $\observ_{\ell+1}$.
    Then the observations have the probabilities
    	\begin{equation} \label{eq:induction}
    	\begin{aligned}
    		\prob_{\vec{\observ}}(\vec{\obse})
    		&= \sum_{\vec{\state} \in \Omega_\state)^\no} \prob_{\observ_1|\staterv_1}(\obse_1|\state_1) (\prob_{\staterv_1}(\state_1))^\no \prod_{\ell=2}^\no \left(  \prob_{\observ_\ell|\staterv_\ell}(\obse_\ell|\state_\ell)  \prod_{i=2}^\ell p_{\state_{\ell-1},\state_\ell}\right)\;.
    	\end{aligned}
    	\end{equation}
\end{lemma}
\begin{proof}
	For $\no \in \N_0$  we have
	\begin{equation*}
		\begin{aligned}
			\prob_{\vec{\observ}}(\vec{\obse}) &= \sum_{\vec{\state} \in (\Omega_\state)^\no} \prob_{\vec{\observ},\vec{\staterv}}(\obse,\state) = \sum_{\vec{\state} \in (\Omega_\state)^\no} \prod_{\ell=1}^\no \prob_{\observ_\ell,\staterv_\ell}(\obse_\ell,\state_\ell)\\
			&= \sum_{\vec{\state} \in \Omega_\state)^\no} \prob_{\observ_1|\staterv_1}(\obse_1|\state_1) \prob_{\staterv_1}(\state_1) \prod_{\ell=2}^\no \bigl( \prob_{\observ_\ell|\staterv_\ell}(\obse_\ell|\state_\ell) \prob_{\staterv_{\ell}}(\state_\ell)\bigr)\\	
			&= \sum_{\vec{\state} \in \Omega_\state)^\no} \prob_{\observ_1|\staterv_1}(\obse_1|\state_1) \prob_{\staterv_1}(\state_1) \prod_{\ell=2}^\no \left( \left( \prob_{\observ_\ell|\staterv_\ell}(\obse_\ell|\state_\ell)  \prod_{i=2}^\ell p_{\state_{\ell-1},\state_\ell}\right)\prob_{\staterv_1}(\state_1) \right)\;.
		\end{aligned}
	\end{equation*}
\end{proof}
\begin{remark}\label{re:critical}
	The Markov-model is specified by a parameter vector $\vec{\lambda}$. This is determined from repeated experiments: That is, we determine $\vec{\lambda}$ from input states $\vec{\state}$ and observation $\vec{\obse}$ and $\prob_{\vec{\observ}}(\vec{\obse}_\ell)$.
	In real experiments, a times series $\observ$ is, in general, \emph{not} generated by a hidden Markov-model. A theoretical justification of a maximum likelihood model is therefore not given.
	However in practice the likelihood model works reasonably well.
\end{remark}
In the following we turn to the inverse problem of recovering a parameter $\vl$ (see \autoref{eq:lambda}).
We have a $\no$-long observation sequence and we assume that we have an ${\tt M}$-state hidden-Markov-model.
From these observations we aim to recover the parameter vector $\vl$, which determines the hidden states.
\begin{algorithm}[H] \caption{Markov-model parameter estimation} \label{alg:mmpe}
	The inverse problem of reconstructing $\lambda$ from an $\no$-long observation vector $\vec{\obse}$ is done as follows:
	\begin{enumerate}
		\item Compute $\prob_{\vec{\observ}}^\vl(\vec{\obse})$ for all $\vl$ from \autoref{eq:lambda_recon}.
		\item Find the model $\vec{\hat{\lambda}}$ that maximizes $\prob_{\vec{\observ}}^\vl(\vec{\obse})$. That is the parameter vector, which maximizes the likelihood.
		\item Estimate the hidden state sequence $\vec{\state}$ according to an observation $\vec{\obse}$ and $\vec{\hat{\lambda}}$.
	\end{enumerate}
\end{algorithm}
There have proposed several algorithms for maximum likelihood estimation (that is the 2nd step of algorithm), which we do not review here. In the following we are concerned with the third item of \autoref{alg:mmpe}.

\section{Language models}
We consider exemplarily the \emph{Raleigh-language} model (see for instance \cite{Jel76}). \index{Model!Raleigh language} This model has a vocabulary of $250$ words, which are organized into word classes (boxes) at different states. A simplified version of the Rayleigh-language model is shown in \autoref{fig:grammer1}. Our model only consists of six layers, while the 
original model is more diverse: 
\begin{itemize}
	\item Layer 1 (Determiners): Box 1A = \{one\}, Box 1B = \{each\}.
	\item Layer 2 (Adjectives): Box 2A = \{bad, black, gentle, great, primary, ...\}, Box 2B = \{distant, eager, kind, large, new, ...\}.
	\item Layer 3 (Nouns): Box 3A = \{condition, duration, general, private, ...\}, Box 3B = \{division, part, period, power, ...\}.
	\item Layer 4 (Verbs): Box 4A = \{considered, created, gave, ...\}, Box 4B = \{contributed, disturbed, forgot, ...\}.
	\item Layer 5 (Article): Box 5A = \{the\}.
	\item Layer 6 (Concrete nouns): Box 6A = \{bus, campaign, food, gun, motion, ...\}.
\end{itemize}
This model is a hidden Markov-model, where the observables are the words and the states are the different branches (denoted by circles). The analogy with the mugs-and-vases model is as follows: The boxes of words corresponds to the vases, which are chosen randomly after the path ,  corresponds to selecting the states from the mugs, has been selected. Exemplarily we see that at the first state we have two choices of the paths, and this corresponds to a mug with two balls. After this path has been chosen a word from the according box is chosen, corresponding to a draw from a vase. 

The difference to the mugs-and-vases model is that the intersection of the boxes can be empty, in which cases the states are not hidden, but in fact the words indicate, which branch/state has been selected. In this case one can model the world alignment with a standard Markov-model.

The most significant challenge is to identify the language model itself. On the other hand finding the most probable sentence from a model is already challenging and done with the Viterbi-algorithm \cite{Vit67}\index{algorithm!Viterbi}. 

\begin{figure}[h!]
	\centering
	\includegraphics[width=1\textwidth]{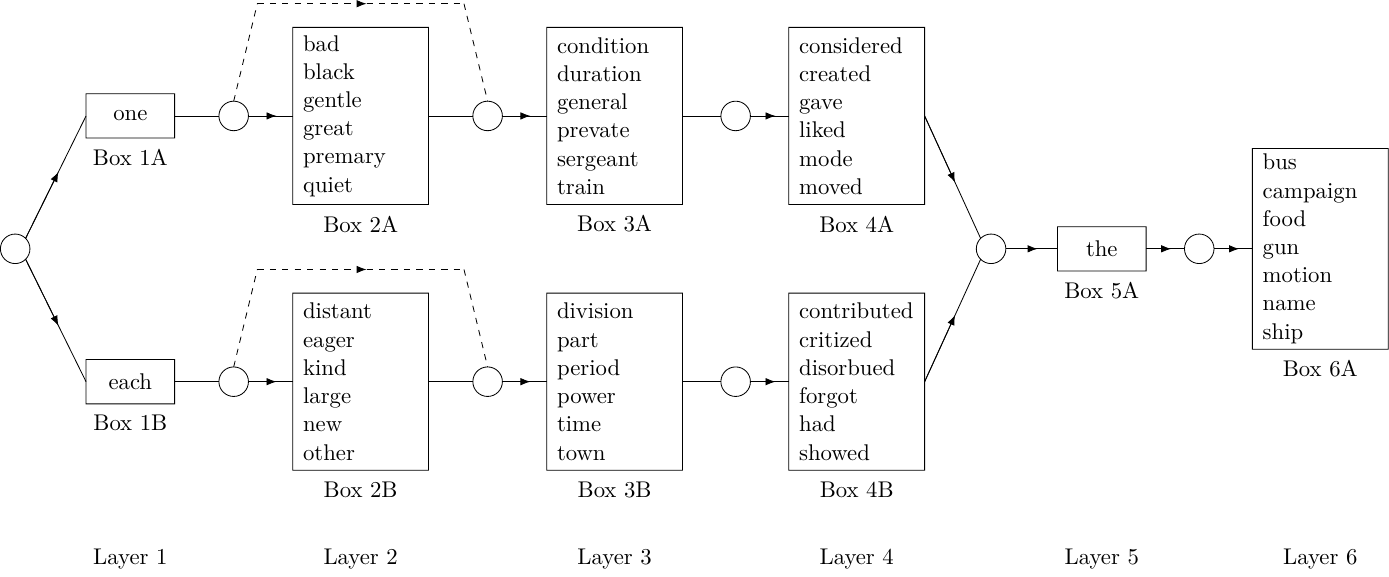}
	\caption{A simplified Rayleigh-language model. Compared it with the original one in \cite{Jel76}, which is more diverse. In this example there are only six layers.
		\label{fig:grammer1}}
\end{figure}

For this language model we need a more general definition of prior and posterior probabilities then in \autoref{de:pcp}:

\begin{definition}[Prior and posterior probability] \label{de:pcpg} \index{probability!prior}\index{probability!posterior}
The index $\ell= 1,\ldots,{\tt M}$ denotes the layers (in the example ${\tt M}=6$) and $c_\ell \in {A,B}$ denotes the box from which we choose the word. We denote the set of \emph{observable words}\index{words!observable} by 
    $$ \Omega_\obse := (\Omega_\obse^1,\ldots,\Omega_\obse^{\tt M}),$$
    where
	$$\Omega_\obse^\ell  = \set{\obse_1^\ell,\ldots,\obse_{{\tt O}_\ell}^\ell}\,, \; \ell = 1,\ldots,{\tt M}$$ 
	and the set of attainable \emph{states}\index{states} by 
	$$ \Omega_\state := (\Omega_\state^1,\ldots,\Omega_\state^{\tt M}),$$
	where
	$$\Omega_\state^\ell = \set{\state_1^\ell,\ldots,\state_{{\tt M}_\ell}^\ell}\,, \; \ell = 1,\ldots,{\tt M}\;.$$
	The difference to \autoref{de:pcp} is that the states are different and for every state we can have a different observation (word).
	
	Let $\state \in \Omega_\state^\ell$ be a state of the random variable $\staterv_\ell$ and let $\vec\observ =(\obse_1,\ldots,\obse_\no)$ by a vector observation random variable, consisting of $\no$ random variables, and let $\obse$ by a state of $\observ_\ell$. Note that in contrast to \autoref{de:pcp} the observation random variables are not identical. With this notation the prior, conditional and posterior distributions can be defined analogously as in 
	\autoref{eq:prior}, \autoref{eq:cpob} and \autoref{eq:posterior}.
\end{definition}
%
%We define the random variable of the state space:
%\begin{equation*}
%	\staterv_l: \Omega_\state := \{\text{all possible paths through the 6 layers}\} \to \R,
%\end{equation*}
%\commentO{I see no difference in making a paths and words. You can make it analogous to first cup and vase model. for one level first you pull the path and then work. In this sense it is more or less the same model as the first cups and vase example. I suggest that you write it like that, than it should be again after the Cups and Vase example, and you now exactly how to formulate.}
%
%\commentO{Now, I would like to see an example. If you know the probabilities, then you get most probable sentences. Now, if we have learned the model (let us assume we did it), what is then the question to answer? Let us assume we have a sentence with this words, can we check that it is grammatically correct. Probably not. So what can we do with trained model? Here we have given all words....}
%and the random variable of the observation space:
%\begin{equation*}
%	\observ_l: \Omega_\omega:=\{\text{all 250 words in vocabulary}\} \to \R. 
%\end{equation*}
%\commentO{I removed the random variables of the boxes. I have only states. Please check}{\color{blue}Yes, we can formulate like this.}

From the Markov property, the probability of a sentence $\vec{\obse} = (\obse_1,\obse_2,\ldots,\obse_6)$ is given by 
(see \autoref{eq:induction}):
\begin{equation}\label{eq:grammermarkov}
	\prob_{\vec{\observ}}(\vec{\obse}) = \prob_{\staterv_1}(\state_1) \prod_{\ell=2}^{6} \prob_{\staterv_\ell|\staterv_{\ell-1}}(\state_\ell|\state_{\ell-1}) \prod_{l=1}^{6} \prob_{\observ_\ell|\staterv_\ell}(\obse_\ell|\state_\ell)\;.
\end{equation}
Note that $\prob_{\observ_1|\state_1}=1$ because there is no observation \commentO{Check}. %{\color{blue} to be more specific, $\prob_{\observ_1|\staterv_1}(\obse_1|\state_1)=1$? }
The parameter vector for the model is:
\begin{equation*}
	\vl = (\prob_{\staterv_1}(\state_1), \prob_{\staterv_\ell|\staterv_{\ell-1}}(\state_\ell|\state_{\ell-1}),\prob_{\observ_\ell|\staterv_\ell}(\obse_\ell|\state_\ell))_{\ell=1,\ldots,6},
\end{equation*}
which means that it has 13 parameters.
%
%For example, the sentence ``one great general gave the food'' has the probability:
%$$\prob(\text{``one great general gave the food''}) = \left(\prod_{l=1}^{6} \prob(c_l \mid c_{l-1})\right)  \left(\prod_{l=1}^{6} \prob(\obse_l \mid c_l)\right),$$
%where $c_0$ denotes the start state, $\prob(c_1 \mid c_0) = \prob_{\staterv_1}(\state_1)$, and
%\begin{align*}
%	c_1 &= \text{Box 1A}, \quad \obse_1 = \text{one}\\
%	c_2 &= \text{Box 2A}, \quad \obse_2 = \text{great}\\
%	c_3 &= \text{Box 3A}, \quad \obse_3 = \text{general}\\
%	c_4 &= \text{Box 4A}, \quad \obse_4 = \text{gave}\\
%	c_5 &= \text{Box 5A}, \quad \obse_5 = \text{the}\\
%	c_6 &= \text{Box 6A}, \quad \obse_6 = \text{food}
%\end{align*}
%The first example (\autoref{fig:grammer1}) demonstrates two branches fusing into one, while the second (\autoref{fig:grammer2}) shows a single branch splitting into two. Both have six layers, each with one or two boxes, and their corresponding Markov model is presented in \autoref{eq:grammermarkov}. 
%
%\begin{figure}[h!]
%	\centering
%	\includegraphics[width=1\textwidth]{grammer2/image.pdf}
%	\caption{Branching network: Single-to-Dual architecture for a simplified Raleigh language model.
%		\label{fig:grammer2}}
%\end{figure} 

\section{Recovery of the states of a hidden Markov-model} \label{se:rhMm}

Since $\vec{\staterv}$ has ${\tt M}^\no$ possible states, the computation of $\prob{\vec{\observ}}(\vec{\obse})$ with \autoref{eq:induction} becomes untractable if ${\tt M}$ or $\no$ get large. In such a situation one can use the following computations:

\begin{definition}[Forward- and backward calculation] \label{de:foba}
For fixed $\ell = 1,\ldots,\no$ we define the
\begin{description}
	\item{\emph{forward calculation}}: Let
	\begin{equation*}
		\vec{\observ}_\ell^+ := (\observ_1,\ldots,\observ_\ell)\;.
	\end{equation*}
	Then for every $\state \in \Omega_\state$ the forward calculation is defined by
	\begin{equation} \label{eq:fc}
		\begin{aligned}
			\alpha_\ell(\state) &:= \prob_{\vec{\observ}_\ell^+,\staterv_\ell} (\vec{\obse}_\ell^+,\state)\;.
		\end{aligned}
	\end{equation}
	\item{The \emph{backward calculation}:}\index{backward calculation} Let
	\begin{equation*}
		\vec{\observ}_\ell^- := (\observ_{\ell+1},\ldots,\observ_\no)\;.
	\end{equation*}
	Then the backward calculation is defined as
	\begin{equation} \label{eq:bc}
		\begin{aligned}
			\beta_\ell(\state) &:= \prob_{\vec{\observ}_\ell^-|\staterv_\ell}( \vec{\obse}_\ell^-|\state)\;.
		\end{aligned}
	\end{equation}
\end{description}

\end{definition}
The terms forward-and backward calculations date back to \cite{Jel76,BahJelMer83}.
According to Bayes' theorem it follows from \autoref{eq:prior} and \autoref{eq:cp} that
\begin{equation} \label{eq:foba1}
	\begin{aligned}
		\alpha_1(\state)
		&= \prob_{\observ_1,\staterv_1}(\obse_1,\state)  =
		   \prob_{\observ_1|\staterv_1}(\obse_1|\state)  \prob_{\staterv_1}(\state)\;.
\end{aligned}
\end{equation}
and therefore we see by summing up over all possible states $\stater \in \Omega_\state$ that
\begin{equation} \label{eq:fobana}
	\begin{aligned}
		\alpha_\ell(\state) &= \sum_{\stater \in \staterv}
		\prob_{\vec{\observ}_\ell^+,\staterv_\ell,\staterv_{\ell-1}}(
		\vec{\obse}_\ell^+,\state,\stater)\;.
	\end{aligned}
\end{equation}
Then by using Bayes' theorem and \autoref{ass:obscure} we get
\begin{equation} \label{eq:foban}
\begin{aligned}
\alpha_\ell(\state) &= \sum_{\stater \in \staterv}
		\prob_{\vec{\observ}_{\ell-1}^+,\staterv_{\ell-1}} (\vec{\observ}_{\ell-1}^+,\stater)
		\prob_{\observ_\ell,\staterv_\ell|\staterv_{\ell-1}}(\obse_\ell,\state_\ell|\stater)\\
		&= \sum_{\stater \in \staterv} \alpha_{\ell-1}(\stater)
		\prob_{\observ_\ell|\staterv_\ell}(\obse_\ell|\state)
		\prob_{\staterv_\ell|\staterv_{\ell-1}}(\state|\stater)\\
		&= \sum_{\stater \in \staterv} \alpha_{\ell-1}(\stater) p_{\stater,\state} \prob_{\observ_\ell|\staterv_\ell}(\obse_\ell|\state)\;.
\end{aligned}
\end{equation}
We emphasize that from \autoref{eq:fc} it follows that
\begin{equation*}
	\prob_{\vec{\observ}}(\vec{\obse}) = \sum_{\state \in \staterv} \alpha_\ell(\state) \text{ for all } \ell=1,\ldots,\no\;.
\end{equation*}
By analogous computations as in \autoref{eq:foban} we see that for $\state \in \staterv$
\begin{equation*}
	\begin{aligned}
		\beta_\no(\state) =1 \text{ and }
		\beta_\ell(\state) = \sum_{\stater \in \staterv} p_{\state,\stater} \prob_{\observ_{\ell+1}|\staterv_\ell}(\obse_{\ell+1}|\stater) \beta_{\ell+1}(\stater) \text{ for all } \ell=\no-1,\ldots,1\;.
	\end{aligned}
\end{equation*}
This is a recursive formula, which requires only a linear amount of elementary computations to calculate
the conditional probabilities of the states $\prob_{\vec{\observ}_\ell^-|\staterv_\ell}( \vec{\obse}_\ell^-|\state)$ and $\prob_{\vec{\observ}_\ell^+,\staterv_\ell} (\vec{\obse}_\ell^+,\state)$ at the next iterate.

\begin{remark}
	We observe that since
\begin{equation} \label{eq:helper1}
	\begin{aligned}
	\prob_{\vec{\observ},\staterv_\ell}(\vec{\obse},\state)
	= &
	\prob_{\vec{\observ}_\ell^+,\staterv_\ell,\vec{\observ}_{\ell+1}^-} (\vec{\obse}_\ell^+,\state,\vec{\obse}_{\ell+1}^-)\\	
	= &
	\prob_{\vec{\observ}_\ell^+,\staterv_\ell} (\vec{\obse}_\ell^+,\state) \cdot
	\prob_{\vec{\observ}_{\ell+1}^-|\staterv_\ell,} (\vec{\obse}_{\ell+1}^-|\state)\\
	=&
	\alpha_\ell(\state) \beta_\ell(\state)
	\end{aligned}
\end{equation}
We calculate now some \emph{posterior probabilities}\index{probability!posterior} (see also \autoref{eq:concrete_ex}) by making use of \autoref{eq:helper1}:
\begin{equation*}
	\begin{aligned}
		\gamma_\ell(\state)
		:= &\prob_{\staterv_\ell|\vec{\observ}} (\state|\vec{\obse})
		=  \frac{\prob_{\vec{\observ},\staterv_\ell}(\vec{\obse},\state)}
		     {\prob_{\vec{\observ}}(\vec{\obse})}
		= \frac{\alpha_\ell(\state) \beta_\ell(\state)}{\prob_{\vec{\observ}}(\vec{\obse})}\;.
	\end{aligned}
\end{equation*}
$\gamma$ is the posterior distribution and gives an estimate on the probability of the states $\state$ at instance $\ell =1,\ldots,\no$, that is of $\prob_{\staterv^{\ell}}$.
We see from \autoref{eq:foban} that in a backprojection formula we sum up over all paths that pass through $\state$.
This is the similarity with the backprojection algorithm in tomography (see \autoref{de:back}).
\begin{figure}[h]
	\begin{center}
		\includegraphics[scale=0.7]{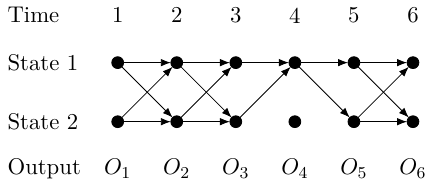}
		\caption{\label{fig:bakcprojection} All random-variables $\vec{\staterv}_\ell$, where $\vec{\staterv}_\ell = \state$, are summed up in a backprojection algorithm. This is similar as in the backprojection algorithm in tomography (see \autoref{de:back}).}
	\end{center}
\end{figure}
\end{remark}

\begin{remark}
	It is important to note that in general an observed \emph{time series}\index{time series} will not be generated by a hidden Markov-model. Justification of these models rest on the success in applications.
\end{remark}

\section{Open research questions}

\begin{opq} \label{co:mr1}
	We reconsider the mugs and vase examples, \autoref{ex:lambda_conc} and \autoref{ex:lambda_conc2}:
	So far we do not know whether the restricted operators $\opo_{r,r}^{\obse}$ in \autoref{eq:roex1} in
	\autoref{ex:lambda_conc} and \autoref{eq:roex2} in \autoref{ex:lambda_conc2} are injective.
	In the above question we consider one particular observation $\obse$, which in particular means that the
	$\abs{\obse}=\no$ is given. The same question can be asked for $\no \to \infty$ and assuming that we have available the particular distribution of drawn black and white balls.
\end{opq}

\section{Further reading}

Training of a Markov-model with \autoref{alg:mmpe} requires two algorithmic solutions: Finding the hidden state sequence, which can by performed with \emph{dynamic programming}\index{dynamic programming} (\emph{Bellman's algorithm}\index{algorithm!Bellham's} \cite{Bel58b}), that is finding the optimal probability parameter $\vl$. Bellham discussed in \cite{Bel57b} the asymptotic of Markov-models, which are integral equations\commentO{Explain more}. The first solution of the parameter estimation problem for $\vl$ in \autoref{alg:mmpe} appeared in \cite{BahJelMer83}.
An interesting introduction to speech recognition is \cite{Red76}, where the challenges are highlighted.

\commentO{ToDo}

\commentO{\cite{BenFra96,Die02}}

\chapter{Training of neural networks} \label{ch:learning}
\emph{Training}\index{training of neural networks} is concerned with finding parameters in a neural network which fit sampled training sets or expert pairs. To make it concrete we study the problem of training an \deepONet first. This represents a realistic scenario for training and it reveals the challenges arising due to the huge amount of parameters to train. A consequence of this discussion is the necessity of reducing the parameters a priori. We also emphasize that training is a highly nonlinear problem and we study the problem of computing parameters of a neural network approximating a given function $\x$ (see \autoref{sec:cond} below). While the standard numerical methods for training networks are gradient descent algorithms, we analyze Newton-type methods. The reason is twofold: Firstly, when we solve an originally linear inverse problem on a set of functions parametrized by neural networks, then it becomes a nonlinear inverse problem of the second decomposition case as discussed in \autoref{ss:decomp2}, and for such problems Newton-type methods are easier to analyze (see for instance \cite{Bla96,KalNeuSch08}). A second observation of \cite{SchHofNas23}, obtained from numerical experiments, is that computing parameters $\vp$ of a neural network representing a function $\x$ is more efficient with a Newton-type method.

\section{Training an \deepONet}
We start with the example of training an \deepONet to demonstrate the complexity of this approach.
\begin{definition}{\bf (Training an \deepONet)} consists in determining a high-dimensional vector
	\begin{equation*}
		C = \begin{pmatrix} \underbrace{\alpha_{j,k}}_{\in \R} & \underbrace{w_{j,k,l}}_{\in \R} & \underbrace{\vw_j}_{\in \R^n} & \underbrace{\theta_{j,k}}_{\in \R} & \underbrace{\vs_l}_{\in \R^m} & \underbrace{\theta_j}_{\in \R} &
		\end{pmatrix}_{\tiny \begin{array}{c} j=1,\ldots,N_j\\ k=1,\ldots,N_k\\ l=1,\ldots,N_l \end{array}} \in \R^\dimlimit
	\end{equation*}
	(repeated from \autoref{eq:D}), which satisfies
	\begin{equation} \label{eq:learn2}
		\begin{aligned}
			\y^{(\ell)}(\vt_\rho) %&= D_C[\x](\vt_\rho) \\
			&= \sum_{j=1}^{N_j} \underbrace{\sum_{k=1}^{N_k} \alpha_{j,k} \sigma \left( \sum_{l=1}^{N_l} w_{j,k,l} \x^{(\ell)}(\vs_l)+\theta_{j,k}\right)}_{=:\beta_j^{(\ell)}} \sigma (\vw_j^T \vt_\rho +\theta_j)\;.
		\end{aligned}
	\end{equation}
\end{definition}
\begin{remark}
	Note, that in the above determination the general setting with \emph{vector} sampling points $\vs_j \in \R^m$, $j=1,\ldots,N_j$ and $\vt_\rho \in \R^n$, $\rho=1,\ldots,Q$ is used.
	
	In \autoref{eq:learn2} only the function evaluations $\y^{(\ell)}(\vt_\rho)$ and $\x^{(\ell)}(\vs_l)$ are used as expert information, while the general setting of this book is based on the expert information of \emph{expert function pairs} $\mathcal{S}_\no = \set{(\x^{(\ell)},\y^{(\ell)}): \ell =1,\ldots,\no}$ (see \autoref{eq:expert_information}). That is, the information of the whole functions is available and not just sample variables.
\end{remark}

\section{Training with Newton's method} \label{se:newtonlearning}
For illustration purposes we simplify \autoref{eq:learn2} as follows:
\begin{simple}
	The sampling points
	\begin{equation*}
		\set{ \vt_\rho : \rho =1,\ldots,Q} \text{ and } \set{\vs_j:j=1,\ldots,N_j}
	\end{equation*}
	have been pre-determined. For instance we might assume an equidistant rectangular sampling both for the functions $\x$ and $\y$. For starting the discussion we make a further assumption that $w_{j,k,l}$ and $\theta_{j,k}$ with $j=1,\ldots,N_j$, $k=1,\ldots,N_k$ and $l=1,\ldots,N_l$ are also specified a-priori. This means that some difficult parameters to be computed, $\beta_j^{(\ell)}$ for $j=1,\ldots,N_j$ and $\ell=1,\ldots,\no$, are already known.
	
	After this simplification we get from \autoref{eq:learn2} the equation
	\begin{equation} \label{eq:op_disc}
		\opo(\vp)= \vy \text{ where } \vy = (\y^{(\ell)}(\vt_\rho))_{\tiny \begin{array}{c}\ell=1,\ldots,\no\\ \rho=1,\ldots,Q\end{array}}\,,
	\end{equation}
	which has $Q\no$ equations and where the nonlinear function $\opo$ (recall that we denote a function evaluation with $(.)$-brackets and a functional/operator evaluation with $[\cdot]$) is defined as follows
	\begin{equation*}
		\opo(\vp)
		:= \left( \sum_{j=1}^{N_j} \beta_j^{(\ell)} \sigma (\vw_j^T \vt_\rho +\theta_j) \right)_{\tiny \begin{array}{c}\ell=1,\ldots,\no\\ \rho=1,\ldots,Q\end{array}}\;.
	\end{equation*}
	Here
	\begin{equation*}
		\vp = (
		\underbrace{\beta_{j}^{(\ell)}}_{\in \R},\underbrace{\vw_j}_{\in \R^m},\underbrace{\theta_j}_{\in \R})_{\tiny \begin{array}{c} \ell=1,\ldots,\no\\ j=1,\ldots,N_j \\ \ell =1,\ldots,\no \end{array}} \in \R^\dimlimit \text{ with } \dimlimit=N_j(m+1+\no)\;.
	\end{equation*}
\end{simple}
For an implementation of Newton's method this number should match the number of unknown $\dimlimit=Q\no$.

For analyzing convergence of a Newton's method for solving \autoref{eq:op_disc} we need to calculate the partial derivatives of $\opo$ first:
\begin{equation*}
	\begin{aligned}
		\frac{\partial \opo}{\partial \beta_{j}^{(\ell)}} (\vp) & = \sigma (\vw_j^T \vt_\rho +\theta_j)\,,\;\;\;\ell=1,\ldots,\no\;,
		j=1,\ldots,N_j\,,\\
		\frac{\partial \opo}{\partial \vw_{j}} (\vp) & = \beta_j^{(\ell)} \sigma' (\vw_j^T \vt_\rho +\theta_j) \vt_\rho \in \R^n \,,\;\;\;j=1,\ldots,N_j,\\
		\frac{\partial \opo}{\partial \theta_{j}} (\vp) & = \beta_j^{(\ell)} \sigma' (\vw_j^T \vt_\rho +\theta_j) \,,\;\;\;j=1,\ldots,N_j\;.
	\end{aligned}
\end{equation*}
Convergence of Newton's method is unproved so far: The problem consists in verifying critical assumptions of convergence theorems for Newton's method, such as \autoref{th:deupot92}: One important assumption there is that $\opo'(\vp)$ has full rank in a neighborhood of the solution $\vec{p}^\dag$. We refer to \autoref{co:newton1}, where we conjecture that the Fr{\`e}chet-derivative $\opo'$ of $\opo$ has \emph{locally} full rank. Other assumptions from \autoref{th:deupot92}, like Lipschitz-continuity of $\opo'$, are rather trivial if $\sigma$ is Fr{\`e}chet-differentiable.

\section{Newton methods for computing parametrizations of a shallow network} \label{sec:cond}
Here we consider the seemingly easy problem of computing the parametrization $\vp \in \R^\dimlimit$ of a neural network function $\x$ as defined in \autoref{eq:classical_approximation}. That is, we aim to solve the equation\footnote{Note that here $m=n$.}
\begin{equation} \label{eq:nnpara}
	\Psi[\vp] = \x\;.
\end{equation}
Here $\Psi$ is the functional of an affine linear neural network (\ALNN) as defined in \autoref{eq:classical_approximation}.
For convergence of Newton's method we verify the Newton-Mysovskii conditions for $\Psi$ as defined in \autoref{th:deupot92dg}. To do so, we calculate the first and second derivatives of $\Psi$ with respect to (see \autoref{eq:classical_approximation})
\begin{equation*}
	\vp := (\alpha_\ell,\vw_\ell,\theta_\ell)_{\ell=1}^{\noc} \in \R^{(m+2)\noc} = \R^\dimlimit.
\end{equation*}
Note, that the main difference to the \autoref{se:newtonlearning} is that the operator $\opo$ maps finite dimensional vectors $\vp$ onto functions $\x$ of an infinite dimensional space, and thus it may not have full range in general. This has to be taken into account in an analysis.

\begin{lemma} \label{le:sigma}
	Let $\sigma : \R \to \R$ be a two times differentiable function with uniformly bounded
	function values and first, second order derivatives, such as $\tanh$. Then, the derivatives of $\Psi$ (defined in \autoref{eq:classical_approximation}) with respect to the coefficients $\vp$ are given by the following formulas, where we use the matrix notation
	\begin{equation*}
		w_{\ell,i} \text{ is the }i-\text{th component of the vector } \vw_\ell \text{ and } W=(w_{\ell,i})_{\tiny \begin{array}{c} i=1,\ldots,m\\ \ell=1,\ldots,\noc \end{array}}\;:
	\end{equation*}
	\begin{itemize}
		\item Derivative with respect to $\alpha_\ell$, $\ell=1,\ldots,\noc$:
		\begin{equation}\label{eq:ca_d1}
			\begin{aligned}
				\frac{\partial \Psi}{\partial \alpha_\ell}[\vp](\vs) &= \sigma \left(\vw_\ell^T \vs +\theta_\ell \right) \text{ for } \ell=1,\ldots,\noc.
			\end{aligned}
		\end{equation}	
		\item Derivative with respect to $w_{\ell,i}$ where $\ell=1,\ldots,\noc$, $i=1,\ldots,m$:
		\begin{equation}\label{eq:ca_d2}
			\begin{aligned}
				\frac{\partial \Psi}{\partial w_{\ell,i}} [\vp](\vs) = \sum_{j=1}^{\noc} \alpha_j\sigma' \left(\vw_j^T\vs +\theta_j \right) \delta_{\ell=j}s_i 
				= \alpha_\ell \sigma' \left(\vw_\ell^T \vs +\theta_\ell \right) s_i \;.
			\end{aligned}
		\end{equation}
		\item Derivative with respect to $\theta_\ell$ where $\ell=1,\ldots,\noc$:
		\begin{equation}\label{eq:ca_d3}
			\begin{aligned}\frac{\partial \Psi}{\partial \theta_\ell} [\vp](\vs)&= \sum_{j=1}^{\noc} \alpha_j\sigma' \left(\vw_j^T\vs +\theta_j \right) \delta_{\ell=j}
				= \alpha_\ell \sigma' \left(\vw_\ell^T \vs +\theta_\ell \right).
			\end{aligned}
		\end{equation}
	\end{itemize}
	All derivatives are functions in $\X = L^2((0,1)^m)$.
	In particular, we see that
	\begin{equation} \label{eq:DG}
		\begin{aligned}
			D \Psi[\vp](\vs)  \vh
			= \begin{pmatrix}
				\frac{\partial \Psi}{\partial \vec{\alpha}}[\vp](\vs) &
				\frac{\partial \Psi}{\partial W}[\vp](\vs) &\frac{\partial \Psi}{\partial \vec{\theta}}[\vp](\vs)
			\end{pmatrix}^T\vh\\
			\text{ for all } \vh = \begin{pmatrix} \vh_{\vec{\alpha}} & {\vec h}_W & \vh_{\vec{\theta}}
			\end{pmatrix}^T \in \R^\dimlimit \text{ and } \vs \in \R^m.
		\end{aligned}
	\end{equation}
	The convergence of Newton's method also requires 2nd derivatives to $\Psi$, which are written down here for the sake of completeness: Let $j_1,j_2=1,\ldots,\noc$, $k_1,k_2 = 1,\ldots,m$, then it follows that
	\begin{equation}\label{eq:ca_dd}
		\begin{aligned}
			\frac{\partial^2 \Psi}{\partial \alpha_{j_1} \partial \alpha_{j_2}}(\vs) &= 0,\\
			\frac{\partial^2 \Psi}{\partial \alpha_{j_1} \partial w_{j_2,k_1}}(\vs) &= \sigma' \left(\sum_{i=1}^m w_{j_2,i} x_i +\theta_{j_1} \right) x_{k_1} \delta_{j_1=j_2}, \\
			\frac{\partial^2 \Psi}{\partial \alpha_{j_1} \partial \theta_{j_2}} (\vs)&= \sigma' \left(\sum_{i=1}^m w_{j_2,i} x_i +\theta_{j_1} \right) \delta_{j_1=j_2}, \\
			\frac{\partial^2 \Psi}{\partial w_{j_1,k_1} \partial w_{j_2,k_2}} (\vs)&=
			\alpha_{j_1} \sigma'' \left(\sum_{i=1}^m w_{j_2,i} x_i +\theta_{j_1} \right) x_{k_1}x_{k_2} \delta_{j_1=j_2}, \\
			\frac{\partial^2 \Psi}{\partial w_{j_1,k_1} \partial \theta_{j_2}}(\vs) &=
			\alpha_{j_1} \sigma'' \left(\sum_{i=1}^m w_{j_2,i} x_i +\theta_{j_1} \right) x_{k_1} \delta_{j_1=j_2}, \\
			\frac{\partial^2 \Psi}{\partial \theta_{j_1} \partial \theta_{j_2}}(\vs) &= \alpha_{j_1} \sigma'' \left(\sum_{i=1}^m w_{j_2,i} x_i +\theta_{j_1} \right)\delta_{j_1=j_2},
		\end{aligned}
	\end{equation}
	where $\delta_{a=b} = 1$ if $a=b$ and $0$ else, that is the Kronecker-delta.
\end{lemma}
\begin{proof}
	The proof follows from elementary calculations.
\end{proof}
The notation of directional derivatives with respect to parameters might be confusing. Note, that for instance
$\frac{\partial \Psi}{\partial \theta_s} [\vp](\vs)$ denotes a directional derivative of the functional $\Psi$ with respect to the variable $\theta_s$ and this derivative is a function, which depends on $\vs \in ]0,1[^m$.

\begin{remark}
	\begin{itemize}
		\item 	In particular \autoref{eq:ca_dd} shows that if $\sigma$ is twice differentiable that
		\begin{equation} \label{eq:d2}%\vh^T
			\begin{aligned}
				\vh^T D^2 \Psi [\vp](\vs) \vh
				\text{ is continuously dependent on } \vp \text{ for fixed }\vs.
			\end{aligned}
		\end{equation}Here we use the abbreviation $\vh = \begin{pmatrix} \vh_{\vec{\alpha}} & \vh_W & \vh_{\vec{\theta}} \end{pmatrix}^T$.
		\item
		We emphasize that under the assumptions of \autoref{le:sigma} the linear space (for fixed $\vp$) is a subspace of $L^2((0,1)^m)$,
		\begin{equation*}
			\range{D\Psi[\vp]} = \set{D\Psi [\vp] \vh : \vh = \begin{pmatrix} \vh_{\vec{\alpha}} & \vh_W & \vh_{\vec{\theta}} \end{pmatrix} \in \R^\dimlimit} \subseteq L^2((0,1)^m).
		\end{equation*}
		\item In order to prove convergence of the Gauss-Newton method\footnote{We call it Gauss-Newton method, because the linearization does not have full rank, and the inverse operator in Newton's method has to be replaced by a Moore-Penrose inverse.}  (defined in \autoref{eq:newton}) by applying \autoref{th:deupot92}, we have to prove that $\Psi$ is a Lipschitz-continuous immersion. The only non-trivial property of the Newton convergence condition is that
		\begin{equation} \label{eq:co1}
			\partial_{h_k} \Psi[\vp], \quad k=1,\ldots,\dimlimit
		\end{equation}
		are linearly independent functions. That this is in fact true remains open as a conjecture so far (see \autoref{re:317}).
	\end{itemize}
\end{remark}
In the following we survey some results on linear independence with respect to the coefficients $\vp = (\alpha_j,\vw_j,\theta_j)_{j=1,\ldots,\noc} \in \R^\dimlimit$, of the \ALNN basis functions (see \autoref{eq:classical_approximation})
\begin{equation*}
	\vs \in \R^m \mapsto \sigma \left(\vw_j^T \vs  + \theta_j \right)\,,\quad j=1,\ldots,\noc,
\end{equation*}
which equals the functions $\vs \to \frac{\partial \Psi}{\partial \alpha_j}[\vp](\vs)$, that are the derivatives of $\Psi$ with respect to the first $\noc$ variables of the neural network operator. This is actually only a necessary but not sufficient condition for convergence of Newton's method.

For analyzing convergence of a Newton's method for solving \autoref{eq:op_disc} we need to calculate the partial derivatives of $\opo$ first:
\begin{equation*}
	\begin{aligned}
		\frac{\partial \opo}{\partial \beta_{j}^{(\ell)}} (\vp) & = \sigma (\vw_j^T \vt_\rho +\theta_j)\,,\;\;\;\ell=1,\ldots,\no,\;
		j=1,\ldots,N_j \\
		\frac{\partial \opo}{\partial \vw_{j}} (\vp) & = \beta_j^{(\ell)} \sigma' (\vw_j^T \vt_\rho +\theta_j) \vt_\rho \in \R^n \,,\;\;\;j=1,\ldots,N_j,\\
		\frac{\partial \opo}{\partial \theta_{j}} (\vp) & = \beta_j^{(\ell)} \sigma' (\vw_j^T \vt_\rho +\theta_j) \,,\;\;\;j=1,\ldots,N_j\;.
	\end{aligned}
\end{equation*}
Convergence of Newton's method is unproved so far: The problem consists in verifying critical assumptions of convergence theorems for Newton's method, such as \autoref{th:deupot92}: One important assumption there is that $\opo'(\vp)$ has full rank in a neighborhood of the solution $\vec{p}^\dag$. We refer to \autoref{co:newton1}, where we conjecture that the Fr{\`e}chet-derivative $\opo'$ of $\opo$ has \emph{locally} full rank. Other assumptions from \autoref{th:deupot92}, like Lipschitz-continuity of $\opo'$, are rather trivial if $\sigma$ is Fr{\`e}chet-differentiable.

\section{Newton methods for computing parametrizations of a deep network} \label{sec:cond_deep}
We consider solving \autoref{eq:op_disc} where $\opo$ is the deep-network structure from \autoref{sec:deep} (see also \autoref{de:network}), that is 
\begin{equation*}
	\opo(\vp)
	:= \sigma \left( \sum_{j=1}^{\tt M} \beta_j \sigma (\vw_j^{(\ell)}{}^T \vt_\rho +\theta_j^{(\ell)}) \right) \in \R^Q\;.
\end{equation*}
Here $\beta_j$, $j=1,\ldots,{\tt M}$ is given, that means we have a customized network. Then the parametrization vector 
is given by  
\begin{equation*}
	\vp = (\underbrace{\vw_j^{(\ell)}}_{\in \R^m},\underbrace{\theta_j^{(\ell)}}_{\in \R})_{\tiny \begin{array}{c} \ell=1,\ldots,\no\\ j=1,\ldots,{\tt M}  \end{array}} \in \R^\dimlimit \text{ with } 
	\dimlimit=2{\tt M}(m+1)\;.
\end{equation*}
Note that in the context of the example shown in \autoref{fig:motiv_eg_2} the parameter $\vp$ is vector of the parameters of the three lines bounding the triangle.  
With the abbreviations
\begin{equation*}
	\begin{aligned}
		z_j^{(\ell)} &= \beta_j \sigma (\vw_j^{(\ell)}{}^T \vt_\rho +\theta_j^{(\ell)})\,,\\
		\tau_j^{(\ell)} &= \vw_j^{(\ell)}{}^T \vt_\rho +\theta_j^{(\ell)}\,,
	\end{aligned}
\end{equation*} 
we calculate the derivatives of $\opo$ with respect to $\vp$ for $j=1,\ldots,{\tt M}$ and $\ell =1,\ldots,\no$: 
\begin{equation*}
	\begin{aligned}
		\frac{\partial F}{\partial \vw_j^{(\ell)}}(\vp) &= \sigma'(z_j^{(\ell)})  \beta_j \sigma'(\tau_j^{(\ell)})\vec{t}_\rho \\
		\frac{\partial F}{\partial \theta_j^{(\ell)}}(\vp) &= \sigma'(z_j^{(\ell)})  \beta_j \sigma'(\tau_j^{(\ell)})
	\end{aligned}
\end{equation*}

%\commentO{ToDO: Checken und interpretieren}

These derivative expressions highlight how the depth and width of the network influence the sensitivity of $\opo$ to parameter variations. The product structure $\sigma'(z_j^{(\ell)})  \beta_j \sigma'(\tau_j^{(\ell)})\vec{t}_\rho$ and $\sigma'(z_j^{(\ell)})  \beta_j \sigma'(\tau_j^{(\ell)})$ indicates that gradients can vanish or explode depending on the activation regime, directly affecting the conditioning of the Newton system. Understanding this structure is therefore essential for designing robust initialization and regularization strategies.

\section{Linear independence of activation functions and its derivatives}
The universal approximation theorems \cite{Cyb89,HorStiWhi89,Hor91} do not allow to conclude
that neural networks functions, as in \autoref{eq:classical_approximation}, are linearly independent. In fact the basis representation with neural networks might not be unique. Linear independence is a non-trivial aspect:
We recall a result from \cite{Lam22} from which linear independence of a shallow neural network operator,
as defined in \autoref{eq:classical_approximation}, can be deduced for a variety of activator functions.
Similar results on linear independence of shallow network functions based on sigmoid activation functions have been stated in \cite{TamTat97,Gua03}, but the discussion in \cite{Lam22} raises questions on the completeness of the proofs.
In \cite{Lam22} it is stated that all activation functions from the \emph{Pytorch library} \cite{PasGroMasLerBra19} are linearly independent with respect to \emph{almost all} parameters $\vw$ and $\theta$.
\begin{theorem}[from \cite{Lam22}] \label{th:lam_main} For the following activation functions,  \HS (\autoref{de:hardsigmoid}), \ReLU (\autoref{eq:ReLU}),  \Sig (\autoref{de:sigmoid}),  $\tanh$, and the \emph{PyTorch} functions \SELU (see \cite{PasGroMasLerBra19}),
	the according \ALNN{}s (defined in \autoref{eq:classical_approximation})
	formed by \emph{randomly generated} vectors $(\vw_j,\theta_j)$, $j=1,\ldots,\noc$ are linearly independent.
	%\index{\HSh}\index{\HS}\index{\HT}\index{\HSw}\index{\LReLU}\index{\PReLU}\index{\ReLU}\index{\ReLUs}\index{\RReLU}
	%\index{\SoS}\index{\Th}\index{\LS}\index{\SP}\index{\TS}\index{\CELU}\index{\ELU}\index{\SELU}{\index{\Sig}}
\end{theorem}
% unsed activation functions, but appears in the exercise book: \HSh (\autoref{de:hardshrink}), \HSw (\autoref{de:hardswish}), \LReLU \autoref{de:leakyrelu}, \ReLUs (\autoref{de:relu6}), \SoS (\autoref{de:softshrinkage}), \LS (\autoref{de:logsigmoid}), \SP (\autoref{de:softplus}), , \CELU (\autoref{de:celu}), \ELU (\autoref{de:elu}),

\begin{remark} \label{re:nonunique}
	\begin{enumerate}
		\item For the analysis of Newton's method linear independence \emph{for almost all} parameters is not sufficient. Instead we require that for all parameters $\vp$ in a \emph{local} neighborhood of the solution $\vp^\dagger$ the according neural network functions are linearly independent. So there is a lack of results in neural network theory on linear independence of neural network functions for proving convergence of Gauss-Newton methods. However, in order to apply neural network theory we need in fact weaker results: The discrepancy can in fact the visualized as follows:
		\bigskip\par
		\fbox{
			\parbox{0.90\textwidth}{
				\begin{center}
					local everywhere (Gauss-Newton) $\Leftrightarrow$ global almost everywhere independent neural networks (see \cite{Lam22}).
		\end{center}}}
		\item \autoref{th:lam_main} states that the functions $\frac{\partial \Psi}{\partial \alpha_s}$ (taking into account \autoref{eq:ca_d1}) are linearly independent for \emph{almost all} parameters $(W,\vec{\theta}) \in \R^{\noc \times m} \times \R^{\noc}$. Note, that $W = (\vw_j)_{j=1}^\noc$ summarizes all vectorial weights.
		In other words, the first block of the matrix is $D \Psi$ in \autoref{eq:DG} consists of functions, which are linearly independent for almost all parameters $(W,\vec{\theta})$.
		For our Gauss-Newton convergence results to hold we need on top that the functions $\frac{\partial \Psi}{\partial w_{s,t}}$ and $\frac{\partial \Psi}{\partial \theta_s}$
		from the second and third block (see \autoref{eq:ca_dd}) are linearly independent within the blocks, respectively, and also across the blocks of the matrix $W$. So far this has not been proven but can be conjectured (see \autoref{re:317}).
		\item For $\sigma = \tanh$ we also have \emph{obvious symmetries} in $\Psi$ because
		\begin{equation} \label{eq:antisymmetric}
			\begin{aligned}
				\sigma' \left(\vw_j^T \vs +\theta_j \right) &= \sigma' \left(-\vw_j^T \vs -\theta_j \right) \\
				& \text{ for every } \vw_j \in \R^m, \theta_j \text{ and } j \in \set{1,\ldots,\noc}\;.
			\end{aligned}
		\end{equation}
		Consequently, for the function $\Psi$ from \autoref{eq:classical_approximation} we have according to \autoref{eq:ca_d3} that
		\begin{equation}\label{eq:antisym}
			\begin{aligned}
				\frac{\partial \Psi}{\partial \theta_s} [\vec{\alpha},W,\vec{\theta}](\vs)
				&=  \alpha_s \sigma'(\vw_j^T \vs + \theta_j)
				=  \alpha_s \sigma'(-\vw_j^T \vs - \theta_j) \\ &=
				\frac{\partial \Psi}{\partial \theta_s}[\vec{\alpha},-\vw,-\vec{\theta}](\vs),
			\end{aligned}
		\end{equation}
		and consequently $\frac{\partial \Psi}{\partial \theta_s}[\vec{\alpha},\vw,-\vec{\theta}]$ and
		$\frac{\partial \Psi}{\partial \theta_s}[\vec{\alpha},-\vw,-\vec{\theta}]$ are linearly dependent.
	\end{enumerate}
\end{remark}
The convergence results of Gauss-Newton method can be extended to 2nd decomposition case operator (see \autoref{ss:decomp2}).
\begin{theorem}[2nd decomposition case] \label{th:newtonNN} Let $L:\X = L^2((0,1)^m) \to \Y$ be a linear, bounded operator with trivial nullspace and dense range. Moreover, let $\opo = L \circ \Psi$, where $\Psi: \dom{\Psi} \subseteq \R^{(m+2)\noc} \to \X$ is a shallow neural network operator generated by an activation function $\sigma$.
	Let $\vp^0 \in \dom{\Psi}$ be the starting point of the Gauss-Newton iteration \autoref{eq:newton} and
	let $\vp^\dagger \in \dom{\Psi}$ be a solution of \autoref{eq:sol}, which satisfy \autoref{eq:h}.
	Then the Gauss-Newton iterations are locally, that is if $\vp^0$ is sufficiently close to $\vp^\dagger$, and quadratically converging.
\end{theorem}
The proof can be found in \cite{SchHofNas23}.
\begin{remark}
	We have shown that a nonlinear operator equation, where the operator is a composition of a linear bounded operator and a shallow neural network operator (2nd decomposition case \autoref{ss:decomp2}), can be solved with a Gauss-Newton method with guaranteed local convergence in the parameter space. So far, the proof is not complete, because it contains assumptions on the linear independence of neural network basis functions.
\end{remark}

\section{Further reading}
Newton's method has been extensively studied in the literature across various contexts and involving different types of inverses, such as the Moore-Penrose inverse, outer inverses, and others. See for instance \cite{DeuHei79,DeuHoh91,DeuHoh93,DeuPot92,Hae86,KanAki64,NasChe93,Ort68,OrtRhe70,Schw79}.
In machine learning gradient descent methods are typically used for parameter learning. However, our experience in calculating parametrizations $\vp$ of neural network functions, which is a learning process on its own, clearly indicates a preference for Newton's methods (see \cite{SchHofNas23}). The theory of \cite{Bla96,DeuEngSch98,Kal97,Kal98b,KalNeuSch08} applies to Newton-type methods for solving inverse problems if the nonlinear operator satisfies the 2nd decomposition case (see \autoref{ss:decomp2}).
Newton's method is quite noise sensitive and therefore it requires some stabilization, such as iterative regularization \cite{BakGon89,BakGon94,Bak92} or trust region methods (see \cite{Han97}).
Currently, we are extending the theory of linear operator learning to decomposition cases (see \cite{LiSch25_report}).

\chapter{Regularization parameter selection} \label{ch:par}
From the regularization community there have been developed a-priori, a-posteriori, heuristic and learned parameter choice  criteria, both in the deterministic and stochastic setting. For the sake of simplicity of presentation we concentrate on parameter choice of Tikhonov regularization, that is of $\alpha$, although the strategies apply to other regularization techniques in a similar manner. We differ between data-driven and deterministic regularization parameter choice. The data-driven prior uses in addition prior experience.  

\section{Parameter selection (deterministic)}
We start with a review on deterministic parameter choice strategies, meaning that the regularization parameter $\alpha$ is determined \emph{without} prior experience from experiments with training samples.

\subsection{A-priori regularization parameter selection} \label{sec:priori}
A-priori strategies determine the regularization parameter from knowledge of the noise-level $\delta$ and properties of the desired solution, typically the minimum norm solution $\xdag$ (defined in \autoref{de:xMNS}). Examples of such strategies have been outlined before in \autoref{eq:noise_cond} and \autoref{eq:order}. There exist a variety of such results in the literature on nonlinear inverse problems (see \cite{EngHanNeu96,SchGraGroHalLen09}).

\subsection{A-posteriori regularization parameter selection} \label{sec:post}
The most famous a-posteriori principle is \emph{Morozov's discrepancy principle}\index{stopping criterion!Morozov}, (see \cite{Mor84}) where the regularization parameter in Tikhonov regularization is chosen as the largest parameter $\alpha$ such that 
\begin{equation}\label{eq:Moro}
	\norms{\op{\xad}-\yd}_\Y \leq \tau \delta
\end{equation}
for some specified $\tau > 1$. This is the same strategy as implemented for iterative regularization algorithms in \autoref{eq:disc}. Compared to a-priori strategies the implementation is more demanding since regularized solutions for many parameters have to be calculated. Other a-posteriori strategies which have been analyzed in the literature for nonlinear inverse problems are for instance \cite{SchEngKun93}, which is based on the Engl-Gfrerer principle for linear ill--posed problems \cite{EngGfr88} or \cite{LuPerRam07}, which is based on Lepski-type balancing principle \cite{MatPer03}.

\section{Heuristic parameter selection without noise information}
We consider regularization parameter selection for Tikhonov regularization (as defined in \autoref{eq:Tik}). The technique can be applied to other regularization methods as well, but we omit it here for the sake of simplicity of presentation. There have been proposed quite a number of such techniques: We concentrate here on two:
\begin{enumerate}
	\item Let $\opo :\X \to \R^n$ be a linear operator with finite-dimensional range. The \emph{generalized cross-validation}\index{generalized cross-validation} \cite{Wah90} consists in selection the global minimum $\hat{\alpha}$ of the functional
	\begin{equation}
		\label{eq:cgv}
		\psi(\alpha,\yd) = \frac{1}{\rho(\alpha)} \norms{\opo \xad -\yd}_\Y\,,
	\end{equation}
    where 
    \begin{equation*}
    	\rho(\alpha) = \frac{\alpha}{n} \text{tr}\left((\alpha I + \opo^* \opo)^{-1}\right)\,,
    \end{equation*}
    where $\text{tr}$ denotes the trace of the operator. Note that this parameter criterion has been developed originally for linear operators $\opo :\R^m \to \R^n$. The difficulty in carrying over the theory to nonlinear problems is already formal that a surrogate for $\rho$ needs to be found.
    \item It was observed numerically in \cite{Han92,HanOle93} that the graph 
    \begin{equation*}
    	\set{\left(\norms{\op{\xad}-\yd}_\Y,\norms{\xad}_X \right): \alpha \in \R_+}
    \end{equation*}
    has a characteristic kink, literally an $L$, at the optimal regularization parameter $\alpha$. This is commonly called the $L$-\emph{curve method}.\index{stopping criterion!$L$-curve method}
\end{enumerate}
Various kinds of \emph{heuristic strategies} have been proposed in the literature (see \cite{HanRau96,Reg21,Luk93,KinRai19}).
Against all these methods \emph{Bakushinskii's veto}\index{Bakushinskii's veto} \cite{Bak84} is directed, which states that estimates of the noise-level are required to implement a stable regularization method. A detailed analysis of the Bakushinskii veto for the $L$-curve criterion can be found in \cite{EngGre94}. Since the arguments are for an infinite dimensional setting the counter argument is that the selection criteria are implemented for finite dimensional inverse problems, which in turn can be counter argued that this only can provide a solution of the infinite dimensional problem if the projection error of the solution on the finite dimensional space is known (see \cite{Sch98b}). 

\section{Parameter selection (data driven)} \label{ss:psl}
In the following we discuss parameter selection criteria based on supervised data $\mathcal{S}_\no$, as defined in \autoref{eq:expert_information}\footnote{As stated above we are flexible to consider the elements of $\mathcal{S}_\no$ starting at index $0$ or $1$.}. But we even allow the items $\x^{(\ell)},\y^{(\ell)}$ of the training data $\mathcal{S}_\no$ to be samples of the random variables ${\textsf x}$, ${\textsf y}$ related by \autoref{eq:op_eps}, such that $\y^{(\ell)}$ is not necessary the exact free term of the equation $\y^{(\ell)} = \op{\x^{(\ell)}}$.

We follow two approaches:
\begin{enumerate}
	\item Empirical risk minimization and 
	\item regularization functional learning.
\end{enumerate}

\subsection{Empirical risk minimization}
In empirical risk minimization a regularization parameter $\alpha \in (0,\infty)$ is learned from training data $\mathcal{S}_\no$ and this parameter is used afterwards in applications. 
 \begin{definition}\label{de:emprisk}
	Let $R:\X \times \X \to \R \cup \set{+\infty}$ be a function, which is called the \emph{loss function}\index{loss function}. Then the \emph{empirical risk}\index{empirical risk} related to $\mathcal{S}_\no$ is defined as 
	\begin{equation} \label{eq:emp_risk}
		{\mathcal{\Hat{R}}_{\alpha,\no}} = \frac{1}{\no}\sum_{\ell = 1}^{\no} R[\x^{(\ell)},\xa^{(\ell)}],
	\end{equation}
	where $\xa^{(\ell)}$ is a minimizer of the functional $\x \mapsto \mathcal{T}_{\alpha,\y^{(\ell)}}[\x]$. 

The \emph{empirical risk minimization principle}\index{empirical risk!minimization principle} consists in determining a regularization parameter $\hat{\alpha}$, which minimizes the empirical risk 
\begin{equation} \label{eq:erm}
	\hat{\alpha}  = \argmin_{\alpha \in \mathcal{A}_q} \mathcal{\Hat{R}}_{\alpha,\no}\,,
\end{equation}
from an admissible set of distinct regularization parameters
\begin{equation} \label{eq:admissible}
	\mathcal{A}_q = \set{\alpha_j: j=1,\ldots,q}\;.% \text{ with } 0<\alpha_1<\alpha_2<\ldots<\alpha_q\;.
\end{equation} 
\end{definition}
This strategy has been also discussed, for example, in \cite{HabTen03,AfkChuChu21} and 
the performance of the regularization parameter choice rule as in \autoref{eq:erm} has been analyzed in \cite{ChiVitMolRosVil24} in a stochastic setting. 

In particular, if the loss function is defined via squared norm $\norm{\cdot}_\X^2$, and the set $\mathcal{A}_q$ consists of a geometric sequence, such that $\alpha_j = \alpha_1 (\frac{\alpha_q}{ \alpha_1})^{\frac{j-1}{q-1}}$ and the noise level $\delta$ in \autoref{eq:exp_eps} is within the interval $(\alpha_1,\alpha_q)$, then under the assumptions of \autoref{th:Tikhonov_risk} and some additional, mostly technical assumptions, the main result of \cite{ChiVitMolRosVil24} states that with probability at least $1-\eta$
\begin{equation*} 
	\mathbb{E} \left[\norm{{\textsf x}^* - \textsf{x}_{\hat{\alpha}}}^2 \right] \le C \left(\delta \left(\frac{\alpha_q}{ \alpha_1}\right)^{\frac{1}{q-1}} + \frac{1}{\no} \log \left(\frac{q}{\eta} \right) \right)\;.
\end{equation*}
%Note, that in this result the solution of \autoref{eq:op} is \commentO{I do not understand this theorem: The empirical risk is not a random variable. So I have to assume that $\y$ is already a random variable. How does this go together?}

%\commentO{I do not understand the following two sentences. There must be much more explanation on the stochastic aspects}
This means that for sufficiently large number $\no$ of examples in supervised data set $\mathcal{S}_\no$ the regularization parameter choice rule $\alpha = \hat{\alpha}$ allows for automatic achievement of expected accuracy that corresponds to the optimal order $\mathcal{O}(\delta)$ in \autoref{eq:bound_chirinos} (see \cite{ChiVitMolRosVil24}). 

\subsection{Regularization functional parameter}
Her the strategy assumes that for every pair $(\x^{(\ell)},\y^{(\ell)}$, $\ell = 0, 1,\ldots\no$, we have predetermined the optimal regularization parameter $\alpha^{(\ell)}$, $\ell=0,1,\ldots,\no$ (for instance by visual inspection comparing $\x^{(\ell)}$ and $\x_\alpha^{(\ell)}$).
Afterwards one learns a functional $\alpha: \Y \to [0,\infty[$ with the methods developed in \autoref{se:nful}, which is used to choose the regularization parameter for arbitrary data $\y \in \Y$.

\begin{example} \label{ex:rkh}
For instance, in the present context, the tool presented in \autoref{ex:gene} can be simplified allowing the following form of the learned functional

\begin{equation}\label{eq: closed_form_parameter }
		{\alpha} = {\alpha}[\y] = \sum_{\ell = 1}^{\no} {\alpha}^{(\ell)} {\tt c}_{\ell}[\y] ,
	\end{equation}
where the vector $\vec{\tt c}[\y]= ({\tt c}_{1}[\y],\ldots, {\tt c}_{\no}[\y])^T $ admits the representation
	\begin{equation}\label{eq:c_for_parameter}
		\vec{\tt c}[\y] = \no^{-1}(\lambda \id + \no^{-1} \mathcal{K} )^{-1} Y^\no[\y] ,
	\end{equation}
and, similar to \autoref{ex:gene}, the matrix $\mathcal{K}$ is formed by the values of the chosen scalar-valued kernel $K = K[\y,\y']$,
\begin{equation*}
    	\mathcal{K}_{\ell,j} = K[\y^{(\ell)},\y^{(j)}] \text{ with } 1 \le \ell,j \le \no,
    \end{equation*}	
while the vector $Y^\no[\y]$ depends on previously unseen input $\y$ as follows

\begin{equation*}
	  Y^\no[\y]= (K[\y,\y^{(1)}],\ldots, K[\y,\y^{(\no)}])^T\;.
	\end{equation*}
From \autoref{eq: closed_form_parameter } and  \autoref{eq:c_for_parameter} one can see that the learned functional ${\alpha}[\y]$ also depends on a regularization parameter $\lambda$ that needs to be turned properly. But this turning issue can be addressed with the use of expert information in the way described, for example, in \cite{dinu2023addressing}.
%\commentO{In the example we should be much more detailed, especially for computing the kernel etc.}
\end{example}
    
  \begin{remark}
  	In the linear, finite-dimensional setting, when $\X, \Y$ are Euclidean spaces and the forward operator is 
  	defined by a matrix $\opo$, i.e., $\op{\x}= \opo \x$, the strategy was explored in \cite{VitForNau22}. 
  	If $\Y$ is a Hilbert space, polynomial or nonlinear functional data regression \cite{WanChiMul16}, \cite{HolPer24} can be employed as a tool to learn the regularization parameter.
  \end{remark}

\section{Further reading}  
The methods of parameter learning from \autoref{ss:psl} for Tikhonov regularization can potentially be used for a variety of regularization methods, such as iterative methods. In many publications parameter selection criteria have only been applied for regularization methods for solving linear inverse problems. An exception is \cite{ChiVitMolRosVil24}, which applies to nonlinear inverse problems. A recent survey on learning methods is \cite{HabHol23_report}.

Bilevel optimization for parameter learning is a natural strategy for learning regularization parameters. It consist in subsequently optimizing for the parameter and the regularized solution. It has been considered for instance in \cite{KunPoc13,ChuEsp17,HolKunBar18,CalCaoCarSchoVal16,LosSchoVal16,LosSchoVal17} and also for learning of regularization functionals \cite{HabTen03}.

\chapter{Hybrid regularization}
\label{ch:Aspri paper}
In this chapter, we consider and analyze \emph{data-informed} regularization methods, which consist of using learned operators \index{regularization methods!data informed} as priors for solving \emph{physics informed} inverse problems.\index{inverse problems! physics informed}. The trust is in the physics-based model, and convergence is therefore proven to a solution of \autoref{eq:op}.
\section{Motivation}
A typical convergence result of a gradient descent method for solving a function equation\footnote{In this example, $F$ is a function. Thus, we use brackets $(\cdot)$ for the arguments, and $x \in \R$ is the argument.} $F(x)=0$ guarantees local convergence (see, for instance, \cite{QuaSacSal07}); this means that a good initial guess $x_1$ is required to assure convergence of Newton's iterations to the desired solution. A classical deterministic analysis of Newton's method cannot admit a possible second guess $x_2$ (see, for instance, \autoref{th:deupot92}). A potential remedy to this problem could be achieved through a stochastic implementation and an analysis based on averaging of initial guesses. However, this still does not employ the expert information $y_i= F(x_i)$, $i=1,2$.  The linear interpolation of the expert information $(y_i \approx f_i,x_i)$, $i=1,2$, takes into account all available expert information. In the example in \autoref{fig:example}, the gradient descent method, starting from the zero of the linear interpolation, converges to the expected local solution. Neither $x_1$ nor $x_2$ is a good initial guess, but the zero of the linear interpolation indeed is. This example shows the reasonability of the approach.
\begin{figure}[bht]
	\begin{center}
	\includegraphics[width=0.7\linewidth]{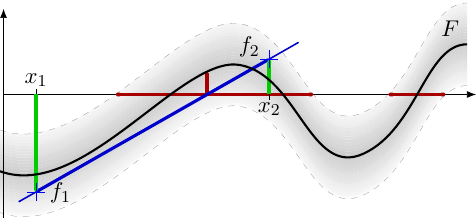}%
	\caption{The mathematical function $F$. The blue line is the linear interpolation of expert information, for which the zero is close to the most meaningful zero of $F$. The image shows an uncertainty region, which contains ``good'' expert information.}
	\label{fig:example}
	\end{center}
\end{figure}

\section[Regularization with surrogate operators]{Data-driven Tikhonov regularization complemented with surrogate operators} 
\label{sec:ddtr}
We continue with \autoref{se:Tik} and further analyze Tikhonov regularization. Now we make use of a learned operator $P:\X \to \Y$ from expert pairs $\mathcal{S}_\no$ (defined in \autoref{eq:expert_information}), which acts as a prior in the Tikhonov regularization method. 

\bigskip\par
\fbox{
	\parbox{0.90\textwidth}{
		\begin{center}
			In this section, we assume that the learned operator $P$ does not fully explain the data and is used only to support the reconstruction; that is, we consider a \emph{physics-informed} and \emph{data-driven} regularization method. 
	\end{center}}}
\bigskip\par
We therefore investigate the properties of minimizers of a \emph{hybrid Tikhonov functional}
\begin{equation} \label{eq:hybrid_tik}
	\mathcal{T}_{\alpha,\lambda,\y^\delta}[\x] = \norms{F[\x]-\y^\delta}^2 + \lambda \underbrace{\norms{P[\x]-\y^\delta}^2}_{\text{prior}}  +
	\alpha \underbrace{\norms{\x - \x^0}^2}_{\text{regularizer}}\;.
\end{equation}
For the sake of simplicity of notation, we leave out the specification of the norms and inner products whenever we think it will not disturb the precision of presentation.

In addition to \autoref{as:ip:weakly-closed} we assume now that the prior operator $P$ satisfies analogous conditions as $\opo$:
\begin{assumption} \label{as:ip:weakly-closed:l}
	The operator $P: \dom{F} \subset \X \to \Y$ is sequentially closed with respect to the weak topologies on $\X$ and $\Y$.
\end{assumption}
Analogously as in \autoref{se:Tik} the following results, extending \autoref{th:ip:well-posedness}, \autoref{th:ip:stability} and \autoref{th:ip:convergence} can be proven:
\begin{theorem}[Hybrid regularization with surrogate operators] \label{th:ip:well-posedness_l}
	Let $\opo$, $\dom{F}$, $\X$ and $\Y$  satisfy \autoref{as:ip:weakly-closed} and assume that $P$ satisfies \autoref{as:ip:weakly-closed:l}.
	
	Assume that $\y^\delta \in \Y$, and $\x^0 \in \X$.
	Then,
	\begin{enumerate}
		\item for every $\tilde{\y} \in \Y$ and every $\alpha >0$ and $\lambda >0$, $\mathcal{T}_{\alpha,\lambda,\tilde{\y}}$ attains a minimizer.
		\item Let $\y_k \in \Y$ converge to $\y^\delta$ and let $\x_k \in \argmin \; \mathcal{T}_{\alpha,\lambda,\y_k}$.
		Then $(\x_k)$ has a convergent subsequence. Every convergent
		subsequence converges to a minimizer of $\mathcal{T}_{\alpha,\lambda,\y^\delta}$.	
		\item Assume that \autoref{eq:op} has a solution in $\dom{F}$ and that
		$\alpha: (0,\infty) \to (0,\infty)$ satisfies \autoref{eq:noise_cond}. Moreover, we assume
		that $\lambda: (0,\infty) \to (0,\infty)$ is of the same order as $\alpha$, i.e., 
		\begin{equation*}
			\alpha \sim \lambda 
		\end{equation*} which means that there exist $c,C > 0$ such that 
		\begin{equation*}
		c \alpha(\delta) \leq \lambda(\delta) \leq C \alpha(\delta)\;.
		\end{equation*}
		Let the sequence $(\delta_k)$ of positive numbers converge to $0$,
		and assume that the data $\y_k:=\y^{\delta_k}$ satisfy $\norm{\y-\y_k} \leq \delta_k$.
	Let 
	\begin{equation*} 
		\xk \in \argmin \; \mathcal{T}_{\alpha(\delta_k),\lambda(\delta_k),\y_k}\;.
	\end{equation*}
	Then, 
	\begin{enumerate}
			\item $(\xk)_{k \in \N}$ has a convergent subsequence. The limit $\ol{\x}$ is a solution of \autoref{eq:op}.
            \item Let $\alpha = \lambda$. Then $\ol{\x}$ is a $(P,\x^0)$-solution, meaning that
            \begin{equation} \label{eq:mnslopo}
            	\begin{aligned}
            	~& \norms{\ol{\x}-\x^0}^2 + \norms{\lop{\ol{\x}}-\y}^2 \\
            	=& \min\set{\norms{\hat{\x}-\x^0}^2 + \norms{\lop{\hat{\x}}-\y}^2 : \hat{\x} \text{ solves } \autoref{eq:op}}.
            	\end{aligned}
            \end{equation}
            The $(P,\x^0)$-solution will again be denoted by $\xdag$.
            If in addition the $(P,\x^0)$-minimum norm solution is unique, then $\xk \to
        \xdag$.
	\end{enumerate}\end{enumerate}
\end{theorem}
This theorem shows potential benefits to classical results: For instance, the minimum-norm solution might be visually much closer to the desired solution because it respects prior information.

\subsection{Convergence rates} results are more difficult in the hybrid setting because one has to take into account the prior operator $P$:
\begin{theorem}[Convergence rates] \label{th:ip:rate-l} We assume that the assumptions from \autoref{th:ip:rate} hold:
	\begin{itemize}
		\item $\opo$, $\dom{F}$, $\X$ and $\Y$ satisfy \autoref{as:ip:weakly-closed};
		\item $\emptyset \neq \dom{F}$ is open and convex.
		\item $\y^\delta \in \Y$ satisfies \autoref{eq:datn}. 
		\item Moreover, let $\opo$ be Fr\'{e}chet differentiable with Lipschitz continuous derivative in a neighborhood of $\xdag$ (i.e. it satisfies \autoref{eq:lipschitz}). The Lipschitz constant of $\opo'$ is denoted again by $L$.
	\end{itemize}
    A variation to the assumptions in \autoref{th:ip:rate} is that 
    \begin{itemize}
    	\item $P$ satisfies \autoref{as:ip:weakly-closed:l} and has a Lipschitz continuous derivative.
	    \item Let $\xdag$ be an $(P;\x^0)$ minimum norm solution and that there exist $\omega\in \Y$ satisfying the \emph{learned source condition}\index{source condition!learned}
		\begin{equation} \label{eq:source_var_l}
			\boxed{\xdag - \x^0 + P'[\xdag]^*(\lop{\xdag}-\y) = \opd{\xdag}^*\omega \text{ with } L\norm{\omega}_\Y \leq 1\;.}
		\end{equation}
	 Here, as in \autoref{eq:lipschitz} $L$ denotes the Lipschitz constant of $\opd{\x}$ in a neighborhood of $\xdag$.
 \end{itemize}
	Then for the regularization parameter choice $\alpha =\lambda \sim \delta$ we have the following convergence rates result
	\begin{equation} \label{eq:delta_rate-l}
		\boxed{\norms{\xad - \xdag} = \mathcal{O}(\sqrt{\delta}).}
	\end{equation}
\end{theorem}
\begin{remark} The proof is analogous to the proof of \autoref{th:ip:rate}. Note that the term $P'[\xdag]^*(\lop{\xdag}-\y)$ only provides a difference to the standard source condition, \autoref{eq:source_var} if $P[\xdag] \neq \y$. This is what we mean when we say that the learned operator $P$ \emph{cannot} fully explain the data. In the next subsection, the situation is different: We think, for instance, of feature constraints, in which case the prior can explain the feature data $\z$. Since the feature operator is not injective, the explanation of features does not mean that one obtains the solution.
\end{remark}

\section[Regularization with feature operators]{Data-driven Tikhonov regularization with feature operators} 
\label{sec:ddtrfo}
A more general view of data-driven Tikhonov regularization as discussed in \autoref{eq:hybrid_tik} concerns determining 
a learned operator $P:\X \to \Z$, which acts as a prior for \emph{features}, for instance, learned classifiers. We therefore investigate the properties of minimizers of a \emph{hybrid Tikhonov functional}
\begin{equation} \label{eq:hybrid_tik_f}
	\mathcal{T}_{\alpha,\lambda,\y^\delta,\z^\delta}[\x] = \norms{F[\x]-\y^\delta}^2 + \lambda \underbrace{\norms{P[\x]-\z^\delta}^2}_{\text{constraint}}  +
	\alpha \underbrace{\norms{\x - \x^0}^2}_{\text{regularizer}}\;.
\end{equation}
In addition to \autoref{as:ip:weakly-closed} we assume now that the prior operator $P:\dom{F} \subseteq \X \to \Z$
maps into the feature space $\Z$. In addition to data $\y^\delta$, we assume that we also have a feature function $\z^\delta$.

The main difference to \autoref{th:ip:well-posedness_l} is that the data can probably explain the features. That is, there exists $\ol{\x}$ satisfying 
\begin{equation} \label{eq:explain}
	\lop{\ol{\x}} = \ol{\z}\;.
\end{equation}

In the case \autoref{eq:explain} does not hold, the analysis concerning well-posedness, stability, convergence, and rates is analogous to that in \autoref{sec:ddtr}, and thus is omitted here. With a feature operator, the analysis is different:

Analogously as in \autoref{se:Tik} the following results, extending \autoref{th:ip:well-posedness}, \autoref{th:ip:stability} and \autoref{th:ip:convergence} can be proven:
\begin{theorem}[Hybrid regularization with features] \label{th:ip:well-posedness_l_general}	
	Let $\opo$, $\dom{F}$, $\X$ and $\Y$  satisfy \autoref{as:ip:weakly-closed} and assume that $P$ satisfies \autoref{as:ip:weakly-closed:l}. 
	
	Assume that $\alpha > 0$, $\tilde{\y} \in \Y$, $\tilde{\z} \in \Z$ and $\x^0 \in \X$.
	Then
	\begin{enumerate}
		\item $\mathcal{T}_{\alpha,\lambda,\tilde{\y},\tilde{\z}}$ attains a minimizer.
		\item Let $(\y_k,\z_k) \in \Y \times \Z$ converge to $(\y^\delta,\z^\delta)$ and let $\x_k \in \argmin \; \mathcal{T}_{\alpha,\lambda,\y_k,\z_k}$.
		Then $(\x_k)$ has a convergent subsequence. Every convergent
		subsequence converges to a minimizer of $\mathcal{T}_{\alpha,\lambda,\y^\delta,\z^\delta}$.	
		\item Assume in addition that \autoref{eq:op} has a solution in $\dom{F}$ and that
		$\alpha: (0,\infty) \to (0,\infty)$ satisfies \autoref{eq:noise_cond}. We take\footnote{Remember, that in \autoref{sec:ddtr} we have chosen $\alpha=\lambda \to 0$ for $\delta \to 0$.} 
		\begin{equation*} 
			\lambda: (0,\infty) \to (0,\infty) \equiv 1\;. 
		\end{equation*}
		Let the sequence $(\delta_k)$ of positive numbers converge to $0$,
		and assume that the data $\y_k:=\y^{\delta_k}$, $\z_k:=\z^{\delta_k}$ satisfy 
		$\norms{\y-\y_k}_\Y \leq \delta_k$ and $\norms{\z-\z_k}_\Z \leq \delta_k$.
		Moreover, let 
		\begin{equation*} 
			\xk \in \argmin \; \mathcal{T}_{\alpha(\delta_k),\lambda(\delta_k),\y_k,\z_k}.
		\end{equation*}
	    Then 
		\begin{enumerate}
			\item $(\xk)_{k \in \N}$ has a convergent subsequence. 
			\item The limit $\ol{\x}$ is a minimum-norm solution of \autoref{eq:op}.
			If in addition the $\x^0$-minimum norm solution is unique, then $\xk \to
			\xdag$.
		\end{enumerate}
	   \end{enumerate}
\end{theorem}
\begin{remark}
The parameter $\lambda$ controls the relative weight of the feature equation~\eqref{eq:explain} with respect to the data misfit term~\eqref{eq:hybrid_tik_f}. In Theorem~\eqref{th:ip:well-posedness_l_general}, we consider the special case $\lambda \equiv 1$ and assume the same noise level $\delta$ for the measured data $\y^\delta$ and the feature data $\z^\delta$. More generally, if the measured and feature data are affected by different noise levels $\delta_\y$ and $\delta_\z$, respectively, the convergence analysis requires
\[
\frac{\delta_\y^2+\lambda(\delta)\delta_\z^2}{\alpha(\delta)}
\to 0,
\qquad \text{as } \delta \to 0.
\]
If $\lambda(\delta)\to 0$, the method approaches a standard data-driven Tikhonov reconstruction. Conversely, if $\lambda(\delta)\to\infty$, the feature equation~\eqref{eq:explain} behaves asymptotically as a dominant constraint.
\end{remark}
Now, we prove convergence rate results for feature priors:
\subsection{Convergence rates}
Convergence rates can be proven very similar to \autoref{th:ip:rate}. The only difference is that instead of \autoref{eq:op} we consider the product equation 
\begin{equation*}
	\begin{aligned}
		F[\x] = \y \text{ and }
		P[\x] = \z\,,
	\end{aligned}
\end{equation*}
together with the product measure of the error:
\begin{equation} \label{eq:er2}
	\norms{\y^\delta-\y}_\Y^2 + \norms{\z^\delta-\z}_\Z^2 \leq \delta^2\;.
\end{equation}

\begin{theorem}[Convergence rates] \label{th:ip:rate-lf} We assume that the assumptions from \autoref{th:ip:rate} hold:
	\begin{itemize}
		\item $\opo$, $\dom{F}$, $\X$ and $\Y$ satisfy \autoref{as:ip:weakly-closed};
		\item $\emptyset \neq \dom{F}$ is open and convex;
		\item $\y^\delta \in \Y$ satisfies \autoref{eq:datn};
		\item Moreover, let $\opo$ be Fr\'{e}chet differentiable with Lipschitz continuous derivative in a neighborhood of $\xdag$ (i.e. it satisfies \autoref{eq:lipschitz}). The Lipschitz constant of $\opo'$ is denoted again by $L_\opo$.
		\item $P$ satisfies \autoref{as:ip:weakly-closed:l} and has a Lipschitz continuous derivative. $L_P$ denotes the Lipschitz constant of $P$.
		\item $\xdag$ is an $\x^0$ minimum norm solution and there exist $\omega\in \Y$ satisfying \autoref{eq:source_var}
		\begin{equation*} 
			\xdag - \x^0  = \opd{\xdag}^*\omega_1  + P'[\xdag]^*\omega_2\text{ with } 
			L_\opo\norms{\omega_1}_\Y +L_P\norms{\omega_2}_\Z\leq 1\;.
		\end{equation*}
	\end{itemize}
	Then for the regularization parameter choice $\alpha \sim \delta$ and $\lambda \sim 1$, we have the following convergence rates result
	\begin{equation*} 
		\boxed{\norms{\xad - \xdag} = \mathcal{O}(\sqrt{\delta}) \text{ and } \norm{\op{\xad}-\y} = \mathcal{O}(\delta)\,,\norm{P[\xad]-\y} = \mathcal{O}(\delta).}
	\end{equation*}
\end{theorem}
We summarize the difference between the two comments, \autoref{sec:ddtr} Tikhonov regularization with complemented surrogate operators and \autoref{sec:ddtrfo} with feature operators: In the latter, we assume that we have information on priors, which do not explain the model. For instance, we think of a particular prior, which is a digit in an image. That is the problem of letter identification in an inverse problem. If the letter is recognized, then the constraint will be achieved. So this is a situation that falls in the category of this subsection. On the other hand, if we learn an operator from finite samples of expert data, it will not be sufficient to fully explain an infinite-dimensional physical model, and this falls into the category treated in \autoref{sec:ddtr}. 

In the mapping $P[\x] =\z$, input features in $\x$ (e.g., a tissue sequence surrounding a region in an ultrasound image) capture structural and textural patterns, while output features in $\z$ represent higher-level semantics such as a diagnosis or segmentation label or learned latent representation. The function or operator $P$ must learn to compress and transform these low-level spatial features into the target feature space, meaning only discriminative input patterns that consistently predict changes in $\z$ are useful. Understanding which input features are invariant versus predictive is key to designing an effective and interpretable model.

Another example of feature operator, where $\x$ denotes a satellite image and $\z$ represents a land-cover segmentation map. The image $\x$ contains low-level spectral and spatial information, such as texture, vegetation indices, and reflectance patterns, while $\z$ encodes higher-level semantic categories including urban areas, forests, water bodies, and agricultural regions. The feature operator $P$ transforms the raw image into a semantic representation by identifying patterns that are predictive of the land-cover classes. As a result, only discriminative image characteristics that reliably distinguish between different terrain types contribute significantly to $\z$. Incorporating the feature relation $P(\x)=\z$ into the reconstruction process can reduce ambiguities and promote solutions that are consistent with both the measured data and the expected semantic structure of the scene.

\section{Data-driven iterative regularization methods} \label{sec:datadriven}
In the following, we consider regularized gradient descent and Newton-type methods in infinite-dimensional Hilbert spaces, taking into account expert information. The developed theory is based on the original references \cite{Bak92,HanNeuSch95,Bla96,DeuEngSch98,KalNeuSch08}. The additional use of expert information has only been considered recently in \cite{AspBanOekSch20}, and this is considered here. Generalizations to the Banach space setting (see \cite{SchuKalHofKaz12}) are formally similar, but technically more complicated, and thus omitted here. In the following, we start with a review of established iterative regularization methods, with the goal of motivating the inclusion of expert and feature operators. The complexity of the usable prior information is gradually increased. We emphasize that for these new kinds of iterative regularization methods, we consider the analysis preliminary, starting with \cite{AspBanOekSch20}. The main goal of this section is to point out potential ways of making use of prior information, and to break up the traditional structure of iterative methods of \emph{one step} into \emph{multi-step} iterations, which reveal synergies with advanced machine learning methods, like
\emph{score-based diffusion}. In a nutshell, these techniques optimize simultaneously for the physical model and the expert information.

\subsection{Motivation}
The best analyzed iterative regularization method is the Landweber-method \index{Landweber-method} (see \cite{HanNeuSch95,KalNeuSch08}), where Given an initial guess $\x_0 \in \X$ the iterations 
\begin{equation}\label{eq:landweber_iter}		
	\xkpd = \xkd - \opd{\xkd}^*\left( \op{\xkd}-\y^\delta \right),\qquad k \in \N_0
\end{equation}
are computed.
As mentioned in the introduction, the regularizing effect comes from early stopping:
The regularized solution is $\x_{k_*(\delta)}$, where $k_*$ is determined from the discrepancy principle \autoref{eq:disc}.

We recall that the Landweber method is the gradient descent method (see \autoref{eq:steepest} explained in \autoref{cha:intro}) with fixed step size $\lambda_k\equiv 1$.

The \emph{iteratively regularized Landweber} method \index{Landweber-method!iteratively regularized!\IRLI} as developed in \cite{Sch98} (see also \cite{KalNeuSch08}), consists in computing the iterative updates via
\begin{equation}\label{eq:ir_landweber_iter}		
	\xkpd = \xkd - \opd{\xkd}^*\left( \op{\xkd}-\y^\delta \right) - \mkd(\xkd - \x^0)\;.
\end{equation}
Two different priors can be used: Namely, the start of the iteration $\x_0$ and $\x^0$.
$\mkd > 0$ are non-negative numbers, which are chosen based on theoretical considerations (see \cite{Sch98}). The structure of \IRLI motivates the development of data-driven methods below.
We note that the iteration \autoref{eq:ir_landweber_iter} can be rewritten as
\begin{equation}\label{eq:ir_landweber_iter_opt}
	\xkpd = \xkd - \tfrac{1}{2}\partial_\x \left( \norms{F[\x]-\y^\delta}^2 + \mkd \norm{\x-\x^0}^2 \right)[\xkd]\;,
\end{equation}
where $\partial_\x$ denotes the \emph{functional derivative} \index{derivative!functional} at $\xkd$ with respect to the elements $\x \in \X$ of the functional
\begin{equation*}
	\x \in \X \mapsto \norms{F[\x]-\y^\delta}^2 + \mkd \norm{\x-\x^0}^2\;.
\end{equation*}
In fact, it has been shown in \cite{Sch98} that multiple prior information $\x_0$ and $\x^0$ can increase the applicability but also can slow down the iteration speed. In our terminology $\x_0$ and $\x^0$ are unsupervised information (see \autoref{eq:expert_information_discriminative}) that means without taking into account information on $\op{\x^0}$ and $\op{\x_0}$, respectively. In the following, we study more sophisticated unsupervised iterative regularization methods and move forward to supervised methods.
\bigskip\par\noindent
\IRLI has been motivated from the \emph{iteratively regularized Gauss-Newton-method} (\IRGN) \index{Gauss-Newton-method!iteratively regularized!\IRGN}from \cite{Bak92}:
\begin{equation}\label{eq:ir_gauss_iter}	
	\begin{aligned}	
	\xkpd &= \xkd - \left(\mkd \id + \opd{\xkd}^*\opd{\xkd}\right)^{-1} \\
	& \qquad \qquad \left( \opd{\xkd}^* \left( \op{\xkd}-\y^\delta \right) + \mkd(\xkd - \x^0)\right)\;,
	\end{aligned}
\end{equation}
which shares with \IRLI the additional stabilizing term. From theoretical considerations, the choice $\mkd = 2^k$ is commonly used in theoretical considerations. Essential is the exponential behavior of this sequence.

Note that \IRGN is based on a Gauss-Newton method
\begin{equation}\label{eq:ir_gauss}		
	\xkpd = \xkd - \left(\opd{\xkd}^*\opd{\xkd}\right)^{-1} \left( \opd{\xkd}^* \left( \op{\xkd}-\y^\delta \right)\right)\,,
\end{equation}
which is nothing but Newton's equation for the normal \autoref{eq:op_min}
\begin{equation*}
	\opd{\x}^*(F[\x]-\y)=0
\end{equation*}
with noise-free data. \IRGN was invented and analyzed in \cite{Bak92} (see also \cite{BakKok04,BakGon94}).

It is instructive to filet \IRLI and \IRGN into two iteration steps, which will motivate novel two-step iteration methods. These formulations are closer to optimization methods used in machine learning applications.

\subsection{2-step \IRLI} \label{ss:2sIRLI}
The \emph{2-step} \IRLI \index{Landweber-method!iteratively regularized!two step}
is defined as follows:
\begin{itemize} \label{alg:two-step}
	\item Initialize $\x_0^\delta:=\x_0$.
	\item Until the Morozov's stopping criterion (see \autoref{eq:disc}) is satisfied perform two iteration steps:
	\begin{enumerate}
		\item model driven update: $\xkppd = \xkd - \opd{\xkd}^*( \op{\xkd} - \y^\delta )$;
		\item data driven update: $\xkpd = \xkppd - \mkd(\xkppd - \x^0)$.
	\end{enumerate}
\end{itemize}

In the considered \emph{2-step} \IRLI method, the model-driven step performs a gradient descent update of the current iterate $\xkd$ using the forward model, starting from an initial guess $\x_0$. The second step incorporates prior information by regularizing the solution toward the regularization point $\x^0$, thereby enhancing the stability of the method. 

To account for the interaction between the two steps, we consider a ball centered at $\x_c = \frac{1}{2}(\x^0 + \x_0)$ with radius $r = \norms{\x^0 - \x_0}$, which is expected to contain the iterates $\xkd$. Accordingly, the assumptions on the operator $\opo$, \autoref{eq:scal} and \autoref{eq:nlc}, are assumed to hold within the ball $\mathcal{B}_{2r}(\x_c)$ from this point onward.

The \emph{2-step} \IRLI can be written in the form of a convex combination as follows: 
	\begin{equation} \label{eq:lwos}
		\xkpd = (1-\mkd) \left(\xkd - \opd{\xkd}^*( \op{\xkd} - \y^\delta )\right) +\mkd \x^0\;,
	\end{equation}
where $\mkd\in (0,1)$ is a weight that controls the interpolation between the model-driven and data-driven updates. When $\mkd$ is small, the iteration is dominated by the model-driven term, and when $\mkd$ is close to 1, the update is dominated by the prior information $\x^0$.  An analysis can be performed along the same lines as \cite{KalNeuSch08} for \IRLI.

\begin{remark}
The \emph{2-step} \IRLI from \autoref{eq:lwos} can be interpreted as a fixed-point iteration for the optimality condition of the nonlinear least-squares problem \autoref{eq:op_min}, given by 
\begin{equation}\label{eq:fixed-point}
	\x = \x-\opd{\x}^*( F[\x] - \y )\;.
\end{equation}
Since fixed-point iterations only converge if $\y^\delta=\y$ is noiseless, we continue with this setting. Then, the \emph{2-step}
\IRLI takes the form
\begin{equation*}
	\xkpd = (1-\mkd)\left(\xkd - \opd{\xkd}^*( \op{\xkd} - \y )\right)+\mkd\; \x^0.
\end{equation*}
The associated fixed-point operator is generally \emph{not} contractive, which prevents the use of classical fixed-point theory. Instead, convergence is proven under structural conditions such as the tangential cone condition~\cite{HanNeuSch95}. Such classical iterations have been considered before by \cite{BroPet67,Bro67}.
\end{remark}

\subsection{Unsupervised \IRLI} \label{sec:wUIRLI}
Unsupervised information $\mathcal{U}=\set{\x^{(\ell)}:\ell=0,1,\ldots,\no}$ can be included in the \IRLI-method (as defined in \autoref{eq:ir_landweber_iter}) in various ways:
\begin{description}
	\item{The \emph{weighted unsupervised} \IRLI} reads as follows \index{Landweber-method!iteratively regularized!weighted unsupervised}
	\begin{equation} \label{eq:wUIRLI}
		\xkpd = \xkd - \opd{\xkd}^*\left( \op{\xkd}-\y^\delta \right) - \mkd \left( \xkd - \frac{1}{\no+1} \sum_{\ell=0}^\no \x^{(\ell)}\right)\;.
	\end{equation}	
	This method has been investigated in \cite{AspSch25}. Analogously as in \autoref{ss:2sIRLI} we define the 2-step weighted unsupervised \IRLI:
	\begin{equation} \label{eq:wUIRLI2} \begin{aligned}
			\xkppd &= \xkd - \opd{\xkd}^*\left( \op{\xkd}-\y^\delta \right), \\
			\xkpd &= \xkppd -\mkd \left(\xkppd - \frac{1}{\no+1} \sum_{\ell=0}^\no \x^{(\ell)}\right)\;.
		\end{aligned}
	\end{equation}
	\item{The \emph{cylic unsupervised} \IRLI} \index{Landweber-method!iteratively regularized!cyclic unsupervised} reads as follows
	\begin{equation} \label{eq:cUIRLI}
		\xkpd = \xkd - \opd{\xkd}^*\left( \op{\xkd}-\y^\delta \right) - \mkd (\xkd-\x^{(r(k))})\;,
	\end{equation}
	where $r(k) = k-\no \lfloor \frac{k}{\no} \rfloor$ with $\lfloor s \rfloor$ the greatest integer below $s$.
	Or its two-step variant:
	\begin{equation} \label{eq:cUIRLI2} \begin{aligned}
				\xkppd &= \xkd - \opd{\xkd}^*\left( \op{\xkd}-\y^\delta \right), \\
				\xkpd &= \xkppd -\mkd (\xkppd-\x^{(r(k))})\;.
		\end{aligned}
	\end{equation}
	\item{The \emph{randomized unsupervised} \IRLI} \index{Landweber-method!iteratively regularized!randomized unsupervised} reads as follows
	\begin{equation} \label{eq:rUIRLI}
		\xkpd = \xkd - \opd{\xkd}^*\left( \op{\xkd}-\y^\delta \right) - \mkd (\xkd-\x^{(i(k))})\;,
	\end{equation}
	where $i(k)$ is a random integer in $\set{0,1,\ldots,\no}$. The 2-step method can be defined similarly as in \autoref{eq:cUIRLI}.
\end{description}

\subsection{Supervised \IRLI}
The strategy consists of first learning an operator $P$ from the training pairs $\mathcal{S}$ with either one of the methods studied in \autoref{ch:op_learning}. The \emph{supervised} \IRLI
\index{Landweber-method!iteratively regularized!supervised} is defined as follows
\begin{equation}\label{eq:wsIRLI}
	\xkpd = \xkd - \opd{\xkd}^*\left( \op{\xkd}-\y^\delta \right) - \mkd P'[\xkd]^*\left( \lop{\xkd} -\y^\delta \right)\;.
\end{equation}
Or the according 2-step method, decomposing the method into a physics modeled and a data-driven iteration step:
\begin{equation} \label{eq:wsIRLI2} \begin{aligned}
		\xkppd &= \xkd - \opd{\xkd}^*\left( \op{\xkd}-\y^\delta \right), \\
		\xkpd &= \xkppd - \mkd P'[\xkppd] ^*\left( \lop{\xkppd} -\y^\delta \right)\;.
	\end{aligned}
\end{equation}
The supervised \IRLI \, has been analyzed in \cite{AspBanOekSch20}. The complication in the analysis in relation to \IRLI (see \cite{Sch98}) concerns the choice of the step-size of the data-driven update.

\subsection{Unsupervised \IRGN} \label{sec:wUIRGN}
Including unsupervised  expert information $\mathcal{U}=\set{\x^{(\ell)}:\ell=0,1,\ldots,\no}$ in the \IRGN-method defined in \autoref{eq:ir_gauss_iter} can be realized analogously as for \IRLI:
\begin{description}
	\item{The \emph{weighted unsupervised} \IRGN} reads as follows\index{Gauss-Newton-method!iteratively regularized!weighted unsupervised}
	\begin{equation} \label{eq:wUIRGN}
		\begin{aligned}
		\xkpd & = \xkd - \left( \mkd \id + \opd{\xkd}^*\opd{\xkd}\right)^{-1} \\
		& \qquad \left(\opd{\xkd}^*\left( \op{\xkd}-\y^\delta \right) + \mkd \left( \xkd-\frac{1}{\no+1}\sum_{\ell=0}^\no \x^{(\ell)} \right)\right)
		\end{aligned}
	\end{equation} 
    and its two-step variant
    \begin{equation*} \begin{aligned}
		\xkppd &= \xkd - \left(\mkd \id + \opd{\xkd}^*\opd{\xkd}\right)^{-1} \opd{\xkd}^*\left( \op{\xkd}-\y^\delta \right), \\
		\xkpd &= \xkppd -\mkd \left(\mkd \id + \opd{\xkd}^*\opd{\xkd}\right)^{-1} \left( \xkppd-\frac{1}{\no+1}\sum_{\ell=0}^\no \x^{(\ell)} \right)\;.
	\end{aligned}
\end{equation*}
\item{The \emph{cylic unsupervised} \IRGN}\index{Gauss-Newton-method!iteratively regularized!cyclic unsupervised} reads as follows
\begin{equation} \label{eq:cUIRGN}
	\begin{aligned}
		\xkpd & = \xkd - \left(\mkd \id + \opd{\xkd}^*\opd{\xkd}\right)^{-1} \\
		& \qquad \qquad \quad \left(\opd{\xkd}^*\left( \op{\xkd}-\y^\delta \right) + \mkd (\xkd-\x^{(r(k))})\right)\,,
	\end{aligned}
\end{equation}
and its two-step variant:
\begin{equation} \label{eq:cUIRGN2} \begin{aligned}
	\xkppd &= \xkd - \left(\mkd \id + \opd{\xkd}^*\opd{\xkd}\right)^{-1} \opd{\xkd}^*\left( \op{\xkd}-\y^\delta \right), \\
	\xkpd &= \xkppd -\mkd \left(\mkd \id + \opd{\xkd}^*\opd{\xkd}\right)^{-1} \left( \xkppd-\x^{(r(k))} \right)\;.
\end{aligned}
\end{equation}
Note that in the unsupervised \IRGN methods, we freeze $\left(\mkd \id + \opd{\xkd}^*\opd{\xkd}\right)^{-1}$, which means that we do not calculate an update of the inverse.

\item{The \emph{randomized unsupervised} \IRGN} \index{Gauss-Newton-method!iteratively regularized!randomized unsupervised} reads as follows
\begin{equation} \label{eq:rUIRGN}
	\begin{aligned}
\xkpd &= \xkd - \left(\mkd \id + \opd{\xkd}^*\opd{\xkd}\right)^{-1} \\
& \qquad \qquad \quad \left( \opd{\xkd}^*\left( \op{\xkd}-\y^\delta \right) + \mkd (\xkd-\x^{(i(k))}) \right)\;,
\end{aligned}
\end{equation}
where $i(k)$ is a random integer in $\set{1,\ldots,\no}$. The 2-step method can be defined similarly as in \autoref{eq:rUIRGN}.
\end{description}

\subsection{Supervised \IRGN}
As for supervised \IRLI, we learn an operator $P$ from the training pairs $\mathcal{S}$. Then \emph{supervised} \IRGN \index{Gauss-Newton-method!iteratively regularized!supervised} is defined as follows
\begin{equation}\label{eq:wsIRGN}
	\begin{aligned}
      \xkpd &= \xkd - \left( \opd{\xkd}^* \opd{\xkd} + \mkd P'[\xkd]^* P'[\xkd] \right)^{-1} \\
      &\qquad \left(\opd{\xkd}^*\left( \op{\xkd}-\y^\delta \right) + \mkd P'[\xkd] ^*\left( \lop{\xkd} -\y^\delta \right) \right)\;.
\end{aligned}
\end{equation}
Or according to the 2-step method
\begin{equation} \label{eq:wsIRGN2} \begin{aligned}
\xkppd &= \xkd - \left( \opd{\xkd}^* \opd{\xkd} + \mkd P'[\xkd]^* P'[\xkd] \right)^{-1} \opd{\xkd}^*\left( \op{\xkd}-\y^\delta \right), \\
\xkpd &= \xkppd - \mkd \left( \opd{\xkd}^* \opd{\xkd} + \mkd P'[\xkd]^* P'[\xkd] \right)^{-1}\\
& \qquad \qquad \qquad \quad P'[\xkppd] ^*\left( \lop{\xkppd}-\yd \right)\;.
\end{aligned}
\end{equation}
Of course, several other variants can be defined.
\begin{remark}
	In this section, we have considered iterative methods with the learned export operator $P$. All presented methods can be motivated to feature priors: In this case, we learn a feature operator $P$ first: The \emph{supervised feature} \IRLI
	\index{Landweber-method!iteratively regularized!supervised feature} is defined as follows
	\begin{equation}\label{eq:wsIRLI_f}
		\xkpd = \xkd - \opd{\xkd}^*\left( \op{\xkd}-\y^\delta \right) - \mkd P'[\xkd]^*\left( \lop{\xkd} -\z^\delta \right)\;.
	\end{equation}
	Or the according 2-step method, decomposing the method into a physics modeled and a data-driven iteration step:
	\begin{equation} \label{eq:wsIRLI2_f} \begin{aligned}
			\xkppd &= \xkd - \opd{\xkd}^*\left( \op{\xkd}-\y^\delta\right), \\
			\xkpd &= \xkppd - \mkd P'[\xkppd]^*\left( \lop{\xkppd}-\z^\delta \right)\;.
		\end{aligned}
	\end{equation}
    Note that $P'[\x]^*:\Z \to \X$ and the data are now $\z^\delta \in \Z$.	
\end{remark}

\begin{remark} \label{section: relation_to_vvr}{\bf (Relation to vector valued regression)}
	In \cite{AspBanOekSch20} we showed that nonlinear learned operators, such as, for instance, vector-valued reproducing kernel Hilbert spaces (vRKHS), are beneficial for convergence of hybrid iterative techniques: The implementation of supervised \IRLI and \IRGN, respectively, requires learning an operator $P$ from training pairs.
\end{remark}

\begin{remark}\label{section: relation_to_iss}{\bf(Relation to inverse scale spaces)}
From a variational perspective, these methods can be interpreted as discretizations of continuous gradient flows associated with coupled data fidelity terms. The learned feature term introduces an additional descent direction that complements the physics-based update and can be viewed as incorporating feature information into the evolution dynamics.
\end{remark}

\section{Model update learning}
In this section, we consider learning a model update. We are given information $\opo_0$ of $\opo$ such that
\begin{equation}\label{eq:modelupdate}
	\opo[\x] = \opo_0[\x] + N[\x]\;.
\end{equation}
For instance, we assume that $\opo_0$ is a linear operator and $N$ is a nonlinear operator. Instead of learning $\opo$ we are now learning $N$, which is a smaller perturbation. In this case we propose the following method for solving \autoref{eq:op} after we have learned $N$ with some method from \autoref{ss:nol}: 
	\begin{equation} \label{eq:wsIRLI2_u} \begin{aligned}
		\xkppd &= \xkd - \opo_0[\xkd]^*\left( \opo_0[\xkd] + N[\xkd] -\y^\delta\right), \\
		\xkpd &= \xkpd - N[\xkppd]^*\left( \opo_0[\xkppd] + N[\xkppd] -\y^\delta\right)\;.
	\end{aligned}
\end{equation}
If $\opo_0[\xkd]$ can be inverted, for instance if $\opo$ is linear, then it is possible to replace the
gradient iteration by Newton's method, yielding  
	\begin{equation} \label{eq:wsIRLI2_ub} \begin{aligned}
		\xkppd &= \xkd - \opo_0[\xkd]^{-1} \left( \opo_0[\xkd] + N[\xkd] -\y^\delta\right), \\
		\xkpd &= \xkpd - N'[\xkppd]^*\left( \opo_0[\xkppd] + N[\xkppd] -\y^\delta\right)\;.
	\end{aligned}
\end{equation}

\section{Open research questions}
\begin{opq} \label{opq:source}
	Does the trained operator $P$ allow for a simplification of the source condition? That is, is \autoref{eq:source_var_l} simpler to satisfy than \autoref{eq:source_var}. %\commentO{See \autoref{ex:source_c_explicit}, where the source condition has been explicitly calculated.}
   
   For the same example the learned source condition \autoref{eq:source_var_l}, which states
   \begin{equation*}
   	\xdag-\x^0 - P'[\xdag]^*(\lop{\xdag}-\y) = \opd{\xdag}^* \omega
   \end{equation*}
   is equivalent to
   \begin{equation}\label{eq:source_c_il}
   	\frac{\xdag-\x^0- P'[\xdag]^*(\lop{\xdag}-\y)}{\y} \in W^{2,2}(0,1) \cap W_0^{1,2}(0,1)\;.
   \end{equation}
   The research question is whether an operator $P$ can be constructed that encodes the zeros of $\y$? For instance, choosing $\x^0 \equiv 0$, the question becomes whether $P$ can be chosen such that $P'[\xdag]^*(\lop{\xdag}-\y)$ equals $\xdag$ at the zeros $\y$.
\end{opq}

\begin{opq}\label{opq:source_a} Same question as in \autoref{opq:source} but for the $a$-problem. % \commentO{See \autoref{ex:source_a_explicit}, where the source condition has been explicitly calculated.}
\end{opq}

\section{Further reading}
In the mathematical inverse problems context iterative machine learning methods have been considered for solving sparse and limited view tomography imaging problems (see \cite{AndKutOktPet22,BubKutLasMarSam19}), surrogate modeling \cite{HerSchwZec24}, nullspace learning \cite{SchwAntHal19}, optimizing iterative reconstruction methods \cite{AdlOkt18}, with neural network basis functions \cite{AdlOkt17,BurNeu03,DitKluMaaOte20,HalNgu23,LiSchwAntHal20,PinPet23}, post-processing of tomographic data \cite{JinMccFroUns17}, task adapted reconstructions \cite{AdlLunVerSchoOkt22}, parameter learning with NETT \cite{LiSchwAntHal20}, among others. They have also been used to speed up iterative methods and to improve parameter quality (see \cite{ArrMaaOktScho19,JinMccFroUns17}). Surveys of machine learning methods for solving mathematical inverse problems are summarized in \cite{ArrMaaOktScho19,HabHol23_report,MukHauOktPerScho23}.
Guaranteed convergence of learned reconstructions is the title of \cite{MukHauOktPerScho23}. However, this and papers with a similar topic are typically devoted to linear inverse problems, while here we consider nonlinear inverse problems. 
Recently generalized regularizers appeared under the name of \emph{plug and play} priors\index{prior!plug and play}; See, for instance, \cite{VenBouWoh13,MukHauOktPerScho23,LauBorAlmDelDur22}. The methodic ideas implement variable splitting and, therefore, are conceptually similar to the two-step hybrid methods considered here. However, their motivation is linear ill--posed problems, while we study nonlinear ill--posed problems here. A recent advance is PiLocNet \cite{LuAoWanPraCha24_report}, which is a hybrid Tikhonov regularization for linear ill-posed problems with sophisticated regularizers. A recent paper showing the relation between efficient diffusion posterior sampling and the Levenberg-Marquardt method is \cite{LiWan25}.

\chapter{Appendix}
Below we summarize basic notation and collect some results from the literature.

\section{Basic notation}\label{sec:not}
We try to obey the following notation throughout this book:
\begin{description}
	%\item{\bf Functions and random variables:}
	\item{\bf Functions and operators:}
\begin{enumerate}
	\item Elements of Euclidean space are denoted by $\vs \in \R^m$ or $\vt \in \R^n$. In dimension $m=1$ we also write $t \in \R$.
	\item Scalar functions defined on subsets of Euclidean space are written in monospace font, for example, $\x:\R^m \to \R, \, \vs \mapsto \x(\vs)$.
	\item Arguments of operators and functionals are written with square brackets $[\cdot]$, for instance, $\y = \op{\x}$.
	\item Parametrizations are denoted by $\vp$. For instance neural network functions $\x$ are parametrized by $\vp$. We write $\x[\vp](\vs)$ and mean that the neural network function obtained by parameters $\vp \in \R^\dimlimit$ is a function with arguments $\vs \in \R^m$  (see \autoref{se:nnf}).	
	\item $\norm{\cdot}_p$ denotes the $\ell^p$-norm on $\R^m$ or of sequences. %, in which case we might formally set $m=\infty$.
    \item $(\Z,\norm{\cdot}_\Z)$ denotes a normed space. If $\Z$ is a Hilbert space, then its inner product is denoted by $\inner{\z_1}{\z_2}_\Z$. The standard inner product on $\R^m$ is denoted by $\inner{\cdot}{\cdot}_2$.
	\item Lebesgue spaces are denoted by $L^p(\Omega)$, for $1 \leq p \leq \infty$.
    If necessary, we use $L^p(\Omega;\C)$ for complex valued functions.
\end{enumerate}
\item{\bf Indices and data driven notation:}
\begin{enumerate}	
	\item Training data are usually indexed by $\ell=1, \ldots,\no$ or $\ell=0,1, \ldots,\no$.
	\item $\hat{\x}$ denotes a learned function. That is a function, where the parameters have been trained by expert information.
	\item $\deep{\cdot}$ denotes an \deepONet.
	\item $\ttm$ and $\noc$ denote the number of components for approximation of a function: $\ttm$ is the dimension of a spline or wavelet, finite element space, respectively (see for instance \autoref{ex:a_recon_II}), and $\noc$ is the number of summands used for approximation of neural networks (see for instance \autoref{eq:classical_approximation}. That means $\x$ is approximated by $\x_\ttm$, $\x_\noc$, respectively. See for instance \autoref{eq:wiener_property} for neural networks and \autoref{eq:xfm} for wavelets.
	\item $\ttn$ denotes the order of approximation of an operator. That means $\opo$ is approximated by $\opo_\ttn$.
	\item Neural networks are parameterized by a vector $\vp$ of length $\noc_*$ (see for instance \autoref{eq:p}.
	\item For multi-indices $\gamma \in \N_0^m$ we define $\abs{\gamma} = \sum_{i=1}^m \gamma_i$.
	\index{multi-indices}
\end{enumerate}
\end{description}
\section{Probability theory} \label{ap:probability_theory}
We review the basic ideas of \emph{probability theory}\index{theory!probability} which are used at several places throughout the book: For instance we discuss the performance of regularization methods under the assumption that input data $\y$ depend on random events in a sample space $\Omega$ (see \autoref{sec: Polregres}) and for inverse problems related to Markov-models (see \autoref{ch:sequentiallearning}.

We start with the definition of measurable sets, which are defined by a $\sigma$-algebra (see for instance \cite{Els18}).\index{algebra!$\sigma$}
\begin{definition}({\bf $\sigma$-algebra})\label{de:ap:sigma_algebra}
	Let $\Omega$ be a set and consider $2^\Omega$ the set of all subsets of $\Omega$. A $\sigma$-algebra $\mathfrak{F}$ is a subset of $2^\Omega$, which satisfies
	\begin{enumerate}
		\item $\emptyset \in \mathfrak{F}$.
		\item If $A \in \mathfrak{F}$, then also the complement $\mathcal{C}A =\Omega \backslash A$ is an element of $\mathfrak{F}$.
		\item The countable union $\bigcup_{i \in \N} A_i \in \mathfrak{F}$ if $A_i \in \mathfrak{F}$ for all $i \in \N$.
	\end{enumerate}
	The elements of the $\sigma$-algebra are often called the \emph{measurable sets}\index{sets!measurable} of $\Omega$.
\end{definition}
\begin{example}
	\begin{enumerate}
		\item The \emph{Borel-$\sigma$-algebra} is the smallest $\sigma$-algebra that contains all open subsets of $\Omega \subseteq \R^n$. In particular this means that all open and closed sets are Borel-measurable.
		\item The $\sigma$-algebra of \emph{Lebesgue-measurable} sets on $\R^n$: A set is $\lambda_n$-measurable (Lebesgue-measurable) if it is the union of a countable set of \emph{finite measurable}\index{measurable!finite} Lebesgue sets.\index{measurable!$\lambda_n$}\index{measurable!Lebesgue} A set is finite measurable if it is the limit of a sequence of elementary sets. See for instance \cite{Els18}.
		\item There exist more Lebesgue-measurable sets than Borel-measurable sets. In fact the Borel-sets have cardinality $c$ and the Lebesgue-sets have cardinality $2^c$. 
		%\commentO{Found it on \url{https://math.stackexchange.com/questions/267991/differences-between-the-borel-measure-and-lebesgue-measure}}
		\item The Vitali-set (see \cite{EvaGar15}) is not Lebesgue-measurable and therefore also not Borel-measurable.% (see \cite{BroCiaSch26_planed}).
	\end{enumerate}
\end{example}

\begin{definition}({\bf Measure and probability space})\label{de:ap:prob_space}
	A \emph{measurable space}\index{space!probability} consists of a \emph{sample space}\index{space!sample} $\Omega$ equipped with a $\sigma$-algebra $\mathfrak{F}$ of subsets of $\Omega$. A \emph{measure space}\index{space!probability} is a measurable space endowed with a measure
	\begin{equation*}
		\mu : \mathfrak{F} \subseteq 2^{\Omega} \to \R_{\geq 0} \cup \set{+\infty}\;.
	\end{equation*}
	
	A \emph{probability space}\index{space!probability} is a measure space whose measure is a probability measure
	\begin{equation*}
		\mathbb{P} : \mathfrak{F} \subseteq 2^{\Omega} \to [0,1]\,,
	\end{equation*}
	such that $\mathbb{P}(\Omega) = 1$.
\end{definition}

\begin{definition} ({\bf Random variable})\label{de:ap:random_variable}
	Let $(\Alpha,\mathfrak{S})$ be a measurable space. An $\Alpha$-valued random variable $\textsf T$ is a measurable function ${\textsf T}: \Omega \to \Alpha$ from the underlying probability space $(\Omega, \mathfrak{F}, \mathbb{P})$ to the measurable space $(\Alpha, \mathfrak{S})$. A typical choice for $A$ is $\R^n$.
\end{definition}

\begin{example}\label{ex:ap:dice}
	Let $\Omega =\set{b,w}$ be the sample space of black and white balls.
	If we draw from the mug we can model this by the random variable
	\begin{equation*}
		\begin{aligned}
		{\textsf T}: \Omega &\to \R\,,\\
		               x &\mapsto \left\{\begin{array}{rcl}
		               	1 &\text{ if } & x=b\\
		               	2 &\text{ if } & x=w
		               	\end{array} \right.
		\end{aligned}
	\end{equation*}
	associated with the probability $\prob:\Omega \to [0,1]$. As a discrete random variable it is clearly measurable. It is common to use $b$, $w$ synonymously for $1$,$2$, respectively.
\end{example}

\begin{definition} ({\bf Distribution of random variable})\label{de:ap:random_variable_distribution}
	Let $(\Alpha,\mathfrak{S},\mu)$ be a measurable space and let ${\textsf T}: \Omega \mapsto \Alpha$ be a $\Alpha$-valued random variable. Then for any $\textit{Z} \in \mathfrak{S}$ the distribution or the probability law $\rho({\textsf T})$ of ${\textsf T}$ is defined as
	$$\rho(\set{{\textsf T} \in \textit{Z}}):= \mathbb{P}(\{\zeta \in \Omega: {\textsf T}(\zeta) \in \textit{Z} \}).$$
\end{definition}

\subsection{Bayes' theorem} \index{theorem!Bayes}
We recall the famous theorem of Bayes. For this we need the definition of \emph{joint} and \emph{conditional probabilities}\index{probability!joint} \index{probability!conditional}:

\begin{definition} \label{de:prob_joint} The \emph{joint probability} of two discrete random variables
	${\textsf T}_1: \Omega_1 \to \Alpha_1$, ${\textsf T}_2:\Omega_2 \to \Alpha_1$ is defined as follows
\begin{equation} \label{eq:prob_joint}
\prob_{{\textsf T}_1,{\textsf T}_2}(\alpha_1,\alpha_2) =
\prob\bigl( \set{   ({\textsf t}_1,{\textsf t}_2): ({\textsf T}_1,{\textsf T}_2)=(\alpha_1,\alpha_2)
           } \bigr)\,,		
\end{equation}
which is defined in the product space $({\textsf t}_1,{\textsf t}_2) \in \Omega_1 \times \Omega_2$.	
\end{definition}
\begin{remark}
	If the two discrete random variables
	${\textsf T}_1: \Omega_1 \to \Alpha_1$, ${\textsf T}_2:\Omega_2 \to \Alpha_1$ are independent \index{random variable!independent}, then 
	\begin{equation*} 
		\prob_{{\textsf T}_1,{\textsf T}_2}(\alpha_1,\alpha_2) = 
		\prob_{{\textsf T}_1}(\alpha_1) \prob_{{\textsf T}_2}(\alpha_2)\;.
	\end{equation*}
\end{remark}

\begin{definition} \label{de:prob_cond} The \emph{conditional probability} is defined as 	
	\begin{equation} \label{eq:prob_cond}
		\prob_{{\textsf T}_1|{\textsf T}_2}(\alpha_1|\alpha_2) = 
		\left\{ \begin{array}{ll}
			\frac{\prob_{{\textsf T}_1,{\textsf T}_2}(\alpha_1,\alpha_2)}{\prob_{{\textsf T}_2}(\alpha_2)}& \text{ if } \; \prob_{{\textsf T}_2}(\alpha_2)>0\,,\\
			0 & \text{ if } \; \prob_{{\textsf T}_2}(\alpha_2)=0\;.
		\end{array}
		\right.
	\end{equation}
\end{definition}

\begin{theorem}[Theorem of Bayes]\label{th:Bayes}
	Let $(\Omega,\mathfrak{F}=2^\Omega,\prob)$ be a probability space.
	Then the \emph{conditional probability} \index{probability!conditional}
	is given by
	\begin{equation} \label{eq:bayes}
		\prob_{\observ|\staterv}(\obse|\state) = \frac{\prob_{\staterv|\observ}(\state|\obse) \prob_\observ(\obse)}{\prob_\staterv(\state) }\;.
	\end{equation}		
\end{theorem}
%\begin{corollary}\label{co:bayes}
%	We have
%	\begin{equation} \label{eq:bayes_cor}
%	\prob_{\staterv,\observ}(\state,\obse) =  \prob_{\observ|\staterv}(\obse|\state) \prob_\staterv(\state)\;.
%	\end{equation}
%\end{corollary}

\section{Functional analysis} \label{sec:weak}
Functional analysis is a highly developed topic and there exist various great books such as \cite{Bre11b,Con07,DunSchw63,Rud73,Yos95}. Here we recall only a very limited amount of results, which are related to \emph{weak convergence} in separable Hilbert spaces. 

\begin{definition} Let $\X$ be a Hilbert space and let $(\x_k)_{k \in \N}$ be a sequence in $\X$. We say that $\x_k$ \emph{converges weakly to}\index{convergence!weak} $\tilde{\x}$, in symbols $\x_k \rightharpoonup \tilde{\x}$, if for all $\x \in \X$
	\begin{equation*}
		\inner{\x_k}{\x}_\X \to \inner{\tilde{\x}}{\x}_\X\;.
	\end{equation*}
\end{definition}
\begin{theorem}[Weak convergence is strong convergence in $\R^{\tt m}$] A sequence $(\x_k)_{k \in \N} \subset \R^{\tt m}$ converges weakly if and only if it converges strongly.
\end{theorem}

\begin{example}[Weak convergence is not strong convergence]\label{ex:onb}
	Let $(\x_k)_{k \in \N}$ be an orthonormal basis of $\X$ (in particular this means that $\X$ is infinite dimensional), then $(\x_k)$ converges weakly to $0$.
\end{example}

\begin{theorem} \label{th:weak_convergence}
	Every bounded sequence in $\X$ has a weakly convergent subsequence.	Conversely, every weakly convergent sequence is bounded.
\end{theorem}
\begin{remark}
	This generalizes the well-known result that every bounded sequence in $\R^{\tt m}$ has a convergent subsequence.
\end{remark}
%The next result requires weak lower semi-continuity of the norm:
\begin{definition}[Weak lower semi-continuity]
	A functional
	\begin{equation*}
		\mathcal{S} : \mathcal{D}({\mathcal S}) : \X \to \R \cup \set{+\infty}
	\end{equation*}
	is \emph{weak lower semi-continuous}\index{lower semi-continuous!weak} if for every sequence
	$(\x_k)_{k \in \N}$, which is weakly convergent in $\X$ to $\tilde{\x} \in \mathcal{D}({\mathcal S})$
	\begin{equation*}
		\mathcal{S}(\tilde{\x}) \leq \liminf \set{\mathcal{S}(\x_k) : k \in \N}\;.
	\end{equation*}
\end{definition}
\begin{theorem} \label{th:ls}
	The norm in a Hilbert-space is weakly lower semi continuous.
\end{theorem}

\section{Spectral decomposition}
is an important computational tool but also an important aspect in the analysis of regularization methods.
\begin{definition}[Spectral decomposition, \cite{EngHanNeu96}] \label{de:spectral} Let $\opo : \X \to \Y$ be a compact linear operator between real Hilbert spaces $\X$ and $\Y$. A \emph{singular system} $({\tt u}_k,{\tt v}_k;\gamma_k)_{k \in \N_0}$ is defined as follows:
	\begin{enumerate}
		\item $\gamma_k^2$, $k \in \N_0$ are the non-zero eigenvalues of the selfadjoint operator $\opo^*\opo$ (and also $\opo\opo^*$) written in decreasing order. We always take $\gamma_k > 0$.
		\item $({\tt u}_k \in \X)_{k \in \N_0}$ are a complete orthonormal system of eigenvectors of $\opo^*\opo$ (on the space $\overline{\mathcal{R}(\opo^*)} =\overline{\mathcal{R}(\opo^*\opo)}$.
		\item $\left({\tt v}_k := \frac{1}{\norm{\opo {\tt u}_k}_\Y} \opo {\tt u}_k \in \Y\right)_{k \in \N_0}$.
	\end{enumerate}
	$\left({\tt v}_k\right)_{k \in \N_0}$ is a complete orthonormal system of eigenvectors of $\opo \opo^*$ which span $\overline{\mathcal{R}(\opo)} = \overline{\mathcal{R}(\opo \opo^*)}$. Moreover, the following formulas hold:
	\begin{equation}\label{eq:spectral}
		\begin{aligned}
			\opo {\tt u}_k &= \gamma_k {\tt v}_k, \quad
			\opo^* {\tt v}_k = \gamma_k {\tt u}_k, \\
			\opo \x &= \sum_{k=0}^\infty \gamma_k \inner{\x}{{\tt u}_k}_\X {\tt v}_k \text{ for all } \x \in \X, \\
			\opo^* \y &= \sum_{k=0}^\infty \gamma_k \inner{\y}{{\tt v}_k}_\Y {\tt u}_k \text{ for all } \y \in \Y.
		\end{aligned}
	\end{equation}
\end{definition}
\autoref{de:spectral} generalizes the well-known spectral decomposition of matrices:
\begin{theorem}[Spectral theory: See Theorem 2.5.2 in \cite{GolVan96}]\label{th:svd} Let the operator $\opo: \R^m \to \R^n$ be linear. Then for every $\vx \in \R^m$
	
	\begin{equation} \label{eq:svd}
		\opo \vx = \sum_{j=1}^{\min \set{m,n}} \gamma_j \inner{\vx}{\vec{u}_j} \vec{v}_j \text{ and } \opo^T \vec{v}_j = \gamma_j \vec{u}_j,
		\opo \vec{u}_j = \gamma_j \vec{v}_j,
	\end{equation}
	where $\vec{u}_j \in \R^m$, $j=1,\ldots,m$ and $\vec{v}_j \in \R^n$, $j=1,\ldots,n$ are orthonormal, respectively, and
	\begin{equation*}
		0 \leq \gamma_1 \leq \gamma_2 \leq \cdots \leq \gamma_{\min \set{m,n}}.
	\end{equation*}
	In matrix form, this identity becomes more compact:
	\begin{equation} \label{eq:svdd}
		\opo = V D U^T \text{ with } U \in \R^{m \times m}, V \in \R^{n \times n},
	\end{equation}
	where $U$ and $V$ are orthogonal and
	\begin{equation*}
		\begin{aligned}
			D = \text{diag} (\gamma_1,\gamma_2, \cdots,\gamma_{\min \set{m,n}}) \in \R^{n \times m}.
		\end{aligned}
	\end{equation*}
\end{theorem}

\begin{definition}[Operator norm] \label{de:op_norm} Let $L:\X \to \Y$ be a linear operator between Hilbert spaces. Then the \emph{operator norm}\index{norm!operator} of $L$ is defined as
	\begin{equation*}
		\norm{L} := \sup_{\set{\x \in \X: \norm{\x}_\X=1}} \norm{L\x}_\Y.
	\end{equation*}
\end{definition}

\section{Distributions} \label{sec:distribution}
The distributions are needed to explicitly calculate the Fourier-transform of activation functions (see \autoref{ss:act}):
\begin{definition}[Continuously differentiable functions]
	For $k \in \N \cup \{\infty\}$ and $\Omega \subseteq \R^m$ let $C^k(\Omega;\C)$ be the space of $k$ times continuously differentiable complex-valued functions on $\Omega$. We denote by $C_0^k(\Omega;\C)$ the subspace of $C^k(\Omega;\C)$ consisting of functions with compact support in $\Omega$ and by $C_B^{k}(\Omega)$ the subspace of bounded functions.
	
	For real valued functions, the definition of $C^k(\R^m)$, $C_B^{l-1}$ and $C_0^\infty(\R^m)$ are analogous.
\end{definition}
\begin{definition}[Schwartz-functions]
	The elements of the set
	\begin{equation*}
		\mathcal{S} (\R^m;\C) := \set{\phi: \R^m \to \C: \forall_{\vec{i}, \vec{j} \in \N_0^m} \exists_{C>0} \sup_{\vs \in \R^m} \abs{\vs^{\vec{i}} D^{\vec{j}} \phi(\vs)} \leq C}
	\end{equation*}
	are called \emph{Schwartz-functions}.\index{function!Schwartz}
\end{definition}
\begin{definition}[Distributions]
	A \emph{distribution}\index{distribution} on $\Omega$ is a linear functional
	\begin{equation*}
		T: C_0^\infty(\Omega;\C) \to \C\,,
	\end{equation*}
	which is bounded in the following sense: For every set $K \subseteq \Omega$ which is compact in $\R^m$ there exist $k \in \N_0$ and a constant $C>0$ such that for all $\phi \in C_0^\infty(K;\C)$ it holds that
	\begin{equation*}
		\abs{T[\phi]} \leq C \sum_{\set{\vec{i} \in \N_0^m : \abs{\vec{i}} \leq k}}
		\sup_{\vs \in K} \abs{D^{\vec{i}} \phi(\vs)} \;.
	\end{equation*}
	%\item
	A \emph{tempered distribution}\index{distribution!tempered} on $\R^m$ is a linear functional
	\begin{equation} \label{eq:tempered_distribution}
		T: \mathcal{S}(\R^m;\C) \to \C\,,
	\end{equation}
	satisfying that there exists $k \in \N_0$ and $C >0$ such that for all $\phi \in \mathcal{S}(\R^m;\C)$
	\begin{equation*}
		\abs{T[\phi]} \leq C
		\max_{\set{\vec{i}, \vec{j} \in \N_0^m : \abs{\vec{i}} + \abs{\vec{j}} \leq k}} \sup_{\vs \in \R^m}\abs{\vs^{\vec{i}} D^{\vec{j}} \phi(\vs)} \;.
	\end{equation*}
	\item It is convenient to write (tempered) distributions as inner products: We write a (tempered) distribution $T[\phi]$
	as
	\begin{equation} \label{eq:ipdist}
		T[\phi] = \inner{T}{\phi}\;.
	\end{equation}
\end{definition}
In the following we summarize some results how we can work with distributions (see \cite{Yos95,DauLio88}):
\begin{theorem} \label{th:distribution_rechenregeeln}
	Let $T$ be a distribution or a tempered distributions on $\R$. with the notation \autoref{eq:ipdist} the derivative of $T$, $T'$, is defined by
	\begin{equation} \label{eq:disdiv}
		\inner{T'}{\phi} = -\inner{T}{\phi'} \text{ for all } \phi \in C_0^\infty(\R;\C), \mathcal{S}(\R;\C) \text{ respectively}.
	\end{equation}
	This is again a distribution, tempered distribution, respectively.
\end{theorem}

\begin{example}	The constant function
	\begin{equation*}
		T: s \mapsto 1
	\end{equation*}
	is a tempered distribution (see \cite{DauLio88}).
	
	The derivative of the $\delta$-distribution\index{distribution!$\delta$}
	\begin{equation*}
		\delta[\phi] = \inner{\delta}{\phi} = \phi(0) \text{ for all } \phi \in C_0^\infty(\R;\C)
	\end{equation*}
	is given by 	
	\begin{equation*}
		\delta'[\phi] = \inner{\delta'}{\phi} = -\phi'(0) \text{ for all } \phi \in C_0^\infty(\R;\C)\;.
	\end{equation*}
\end{example}

\begin{definition}\label{def:convergence_distribution}
	A sequence of distributions $T_k$ converges to a distribution $T$ (denoted by $T_k \rightarrow T$) if
	\begin{equation*}
		\lim_{k\rightarrow \infty}\inner{T_k}{\phi} = \inner{T}{\phi}
	\end{equation*}
	for any $\phi \in C_0^\infty(\R;\C)$.
\end{definition}
\begin{example}
	Consider the sequence of functions $\delta_k: \R \to \C$ defined by:
	\begin{equation*}
		\delta_k(x) = k \cdot \mathbf{1}_{\left[-\frac{1}{2k}, \frac{1}{2k}\right]}(x) =
		\begin{cases}
			k & \text{if } |x| \leq \frac{1}{2k} \\
			0 & \text{otherwise}
		\end{cases}
	\end{equation*}
	This sequence converges to $\delta$-distribution:
	\begin{equation*}
		\delta_k \rightarrow \delta.
	\end{equation*}
\end{example}

\section{Fourier-transform} \label{ch:fourier}
\begin{definition} \label{de:fourier} \index{Fourier-transform}
	\begin{itemize}
		\item The Fourier-transform of a function $f \in L^1(\R;\C)$ is defined as
		\begin{equation}
			\label{eq:fourier_l1}
			\omega \in \R \mapsto \mathcal{F}[f](\omega) = \frac{1}{\sqrt{2\pi}} \int_\R f(t) \exp^{-\i \omega t}\,dt\;.
		\end{equation}
		The inverse Fourier-transform of a function $g \in L^1(\R;\C)$ is defined as
		\begin{equation}
			\label{eq:fourier_l1_1}
			t \in \R \mapsto \mathcal{F}^{-1}[g](t) = \frac{1}{\sqrt{2\pi}} \int_\R g(\omega) \exp^{\i \omega t}\,dt\,.
		\end{equation}
		%for all $g$ in the range of the operator $\mathcal{F}$.
		\item The Fourier-transform of a function $f \in L^2(\R;\C)$ is defined as
		\begin{equation}
			\label{eq:fourier_l2}
			\omega \in \R \mapsto \mathcal{F}[f](\omega) = \frac{1}{\sqrt{2\pi}} \lim_{0 < A \to \infty} \int_{-A}^A f(t) \exp^{-\i \omega t}\,dt\;.
		\end{equation}
		The inverse of $\mathcal{F}$ is given by
		\begin{equation} \label{eq:fourier_l2_1}
			t \in \R \mapsto \mathcal{F}^{-1}[g](t) = \frac{1}{\sqrt{2\pi}} \lim_{0 < A \to \infty} \int_{-A}^A g(\omega) \exp^{\i \omega t}\,dt\,,
		\end{equation}
		for all $g \in L^2(\R)$.
		\item 	The Fourier-transform of a tempered distribution $f$ is the unique tempered distribution $\mathcal{F}[f]$ satisfying\footnote{$\mathcal{S}'$ can denote the dual space of real or complex valued function $\phi:\R \to \R$ or into $\C$ which are infinitely often differentiable, that is $\phi \in C^\infty(\R)$. We do not distinguish among them.}
		\begin{equation} \label{eq:f1}
			\mathcal{F}[f][\phi] = f[\mathcal{F}^{-1}[\phi]] \quad \text{ for all } \phi \in \mathcal{S}(\R) \text{ or } \mathcal{S}(\R;\C).
		\end{equation}
	\end{itemize}
\end{definition}
\begin{example}{\bf ($\delta$-distribution)}.
	The Fourier-transform of the constant function $s \mapsto 1$ is the \emph{$\delta$-distribution} \index{distribution!$\delta$}:
	\begin{equation} \label{eq:delta_dis}
		\mathcal{F}[s \mapsto 1](\omega) = \sqrt{2\pi} \delta(\omega)\;.
	\end{equation}
\end{example}
Here we show a few distributional Fourier-transforms: They are well-known but the scaling coefficients cause often some problems in calculation, such that we add them here for the sake of completeness.
\begin{table}[h]
	\centering
	\begin{tabular}
		{r||l}
		\hline
		$1$ & $\sqrt{2\pi} \delta$\\
		$s$ & $\sqrt{2\pi} \i \delta'$ \\
		$\chi_{[-c,c]}$ &  $\frac{1}{\sqrt{2\pi c}} \textnormal{sinc} (\omega/(2\pi c))$\\
		$\chi_{[0,\infty]}$ & $\sqrt{\frac{2}{\pi}}c \,\textnormal{sinc} (c\omega/\pi) $\\
		\hline	
	\end{tabular}
	\caption{\centering See for instance \cite{Zei13}. Here $\sinc (x) := \frac{\sin(\pi x)}{\pi x}$.}
\end{table}
The Fourier-transform of a derivative of a tempered distribution can be computed with the following theorem:
\begin{theorem} Let $f$ be a tempered distribution, then the Fourier-transform of the derivative of $f$, which is again a tempered distribution, reads as follows
	\begin{equation} \label{eq:deriv_dist}
		\mathcal{F}[f'] = \mathcal{F}[\omega \to \i \omega f(\omega)]\;.
	\end{equation}
\end{theorem}
An important theorem concerns \emph{convolution} \index{convolution} of functions, because the Fourier-transform maps it into a product:
\begin{definition}\label{de:convolution}
	Let $k \in L^1(\R)$ and $f \in L^2(\R)$. Then for all $s \in \R$
	\begin{equation} \label{eq:conv}
		(k \ast f)(s) := \int_\R k(s-t)f(t)\,dt\;.
	\end{equation}
	For $k \in \mathcal{S}'(\R)$ and $l \in \mathcal{S}'(\R)$ and $\phi \in \mathcal{S}(\R)$ we have
	\begin{equation} \label{eq:conv_dis}
		(k \ast l)[\phi] := (k(s) \otimes l(t))[\phi(s+t)]\;.
	\end{equation}
\end{definition}
The following result holds
\begin{theorem} \label{th:Fconv}
	If $k \in L^1(\R)$ and $f \in L^2(\R)$. Then
	\begin{equation} \label{eq:conv_est}
		\norm{k \ast f}_{L^2} \leq \norm{k}_{L^1} \norm{f}_{L^2}.
	\end{equation}
	Moreover,
	\begin{equation} \label{eq:conf_F}
		\mathcal{F}[k \ast f] = \sqrt{2\pi} \mathcal{F}[k] \mathcal{F}[f]\;.
	\end{equation}
\end{theorem}

The Fourier-transform has a series of useful properties:
\begin{theorem}
	The Fourier-transform is an isometry on $L^2(\R)$. In particular
	\begin{equation} \label{eq:fourier_isometry}
		\norm{\mathcal{F}[\x]}_{L^2} = \norm{\x}_{L^2}\;.
	\end{equation}
\end{theorem}

We frequently make use of the Paley-Wiener theorem: For this theorem we need to recall the definition of entire functions:
\begin{definition}
	A function ${\tt x}: \C \to \C$ in called \emph{entire}\index{function!entire}, if it is holomorphic on $\C$.
	A \emph{holomorphic} function on $\C$ \index{function!holomorphic} is a function, that has a complex derivative for every $z \in \C$ \index{derivative!complex}: That means that
	\begin{equation*}
		{\tt x}'(z_0) = \lim_{z \to z_0 \text{ in } \C} \frac{{\tt x}(z)-{\tt x}(z_0)}{z-z_0}
	\end{equation*}
	exists for all $\z \in \C$.
\end{definition}

\begin{theorem}[Paley-Wiener theorem] \label{th:PaleyWiener}
	The Fourier-transform $\mathcal{F}[\x]:\C \to \C$ of a distribution $\x:\R \to \C$ with compact support on $\R$ is an entire function.
\end{theorem}
For a proof see \cite{Hoe03}. The original theorem goes back to Paley-Wiener \cite{PalWie34}, which was formulated for $L^p$-functions. It has been generalized to distributions by Schwartz-\cite{Schw52}.

The theorem \autoref{th:PaleyWiener} above gives a necessary condition: compact support implies entire analyticity. However, it does not provide a growth estimate nor a converse. Yosida's \cite{Yos95} formulation of the Paley–Wiener theorem supplies both: an explicit exponential bound and the converse statement, thereby giving a complete characterization of the Fourier-transforms of compactly supported functions (or distributions) in terms of entire functions of exponential type.

\begin{theorem}[Paley--Wiener--Yosida]\label{th:PaleyWienerYosida}
	Let $\mathcal{F}[\x]: \C^n \to \C$ be the Fourier-transform of a distribution $\x : \R^n \to \C$ with compact support contained in the closed ball $\overline{B}(0, R)$. Then $\mathcal{F}[\x]$ is an entire function on $\C^n$ and there exists a constant $C > 0$ such that
	\begin{equation*}
	|\mathcal{F}[\x](z)| \leq C \, e^{2\pi R |\Im z|}, \qquad z \in \C^n.
	\end{equation*}	
	Conversely, if $F: \C^n \to \C$ is an entire function satisfying the growth condition 
	\begin{equation*}
	|F(z)| \leq C \, e^{2\pi R |\Im z|}
	\end{equation*}
	for some constants $C, R > 0$, and the restriction $F|_{\R^n}$ belongs to $L^2(\R^n)$, then there exists a unique distribution $\x$ with support contained in $\overline{B}(0,R)$ such that $F(z) = \mathcal{F}[\x](z)$ for all $z \in \C^n$.
\end{theorem}

For a proof see \cite{Yos95}. 

\section{Sobolev-spaces} \label{sob_def}
This is a summary of results from \cite{Ada75}, which show the relation between different Sobolev-spaces. Most of the proofs can be found there. We only concentrate on the essential results and ignore important topics such as \emph{distributional derivatives}. In most cases we can consider the distributional (generalized) derivatives as standard gradients. The section gives a hint on the role of generalized derivatives for working with Sobolev-functions, but it should be clear that generalized derivatives can be treated as differentiable functions.

\begin{definition} \label{de:sobspace}
	Let $\Omega \subset \R^m$ be open, ${\tt u} \in L_{\text loc}^1(\Omega)$, and $1\le i \le m$.
	Denote ${\tt g}_i = \partial^i {\tt u}$ 
	the distributional derivative of ${\tt u}$ with respect to the $i$-th component $s_i$ of $\vs \in \R^m$.
	If ${\tt g}_i \in L^1_{\text loc}(\Omega)$, then ${\tt g}_i$ is called
	the \emph{weak partial derivative}\index{derivative!weak partial}
	of ${\tt u}$ with respect to $s_i$.
\end{definition}

In other words, a weak derivative is a distributional derivative, which can be represented
by a locally summable function.
\smallskip

If all weak partial derivatives $\partial^i {\tt u}$, $1\le i \le m$, exist,
we define the \emph{weak gradient}\index{gradient}
of ${\tt u}$ as
\begin{equation*}
	\nabla {\tt u} := \left(\partial^1 {\tt u},\ldots,\partial^m {\tt u}\right)\;.
	\index{gradient!$\nabla {\tt u}$}
\end{equation*}
In the sequel we do not notationally distinguish between derivatives and weak derivatives
of functions and distributional derivatives.

\begin{definition}
	Let $1 \leq p \le \infty$ and $\Omega$ an open subset of $\R^m$.
	\begin{enumerate}
		\item The function ${\tt u}$ belongs to the \emph{Sobolev-space}\index{space!Sobolev}
		$W^{1,p}(\Omega)$,
		\index{Sobolev-space!$W^{1,p}(\Omega)$}
		if ${\tt u} \in L^p(\Omega)$ and the weak gradient
		$\nabla {\tt u}$ exists and belongs to $L^p(\Omega;\R^m)$.
		\item The function ${\tt u}$ belongs to $W_{\text loc}^{1,p}(\Omega)$
		\index{Sobolev-space!$W_{\text loc}^{1,p}(\Omega)$} 
		if ${\tt u}|_V \in W^{1,p}(V)$ for each open set $V$, which is compact in $\Omega$, in signs $V \subseteq_c \Omega$.
	\end{enumerate}
\end{definition}

For $1 \leq p < \infty$ and ${\tt u} \in W^{1,p}(\Omega)$, we define\index{norm!Sobolev}\footnote{Note that $\norm{\nabla {\tt u}}_2^p$ means the power $p$ of the the Euclidean-norm of the gradient and $\norm{\nabla {\tt u}}_\infty$ means the $L^\infty$ norm of the gradient.}
\begin{equation} \label{eq:sobolev1}
	\norm{{\tt u}}_{W^{1,p}} := \left( \int_\Omega \abs{{\tt u}(\vs)}^p +
	\norm{\nabla {\tt u}(\vs)}_2^p \, d \vs \right)^{1/p}\;.
\end{equation}
Moreover,
\begin{equation*}
	\norm{{\tt u}}_{W^{1,\infty}} := \max \set{\norm{{\tt u}}_{L^\infty}, \norm{\nabla {\tt u}}_{L^\infty}}\;.
\end{equation*}
The space $W^{1,p}(\Omega)$, $1\le p \le \infty$, with the norm $\norm{\cdot}_{W^{1,p}}$ is a Banach space.
If $1 < p < \infty$, then $W^{1,p}(\Omega)$ is reflexive.
Moreover, $W^{1,2}(\Omega)$ is a Hilbert space with the inner product
\begin{equation*}
\inner{{\tt u}}{{\tt v}}_{W^{1,2}} := \int_\Omega {\tt u}(\vs)\,{\tt v}(\vs) d\vs + \sum_i \int_\Omega \partial^i {\tt u}(\vs)\,\partial^i {\tt v}(\vs) d \vs\;.
\end{equation*}

\subsection*{Higher order Sobolev-spaces}

\begin{definition}
	Let $\Omega \subset \R^m$ be open, ${\tt u} \in L^1_{\text loc}(\Omega)$,
	$l \in \N$, and $1 \leq p \leq \infty$.
	
	\begin{enumerate}
		\item We say that ${\tt u} \in W^{l,p}(\Omega)$,
		if ${\tt u} \in L^p(\Omega)$, and for all multi-indices $\vec{\gamma}$
		with $\abs{\vec{\gamma}} \le l$ the distributional derivative $\partial^{\vec{\gamma}} {\tt u}$
		belongs to $L^p(\Omega)$.
		\item We say that ${\tt u} \in W^{l,p}_{\text loc}(\Omega)$,
		if for all open sets $V \subseteq_c \Omega$ we have ${\tt u}|_V \in W^{l,p}(V)$. 
		
		The elements of\/ $W^{l,p}_{\text loc}(\Omega)$ (and thus in particular the elements
		of $W^{l,p}(\Omega)$) are called \emph{Sobolev-functions}.\index{function!Sobolev}
	\end{enumerate}
\end{definition}

If ${\tt u} \in W^{l,p}(\Omega)$ and $0 \le k \le l$ we define
\begin{equation*}
\nabla^k {\tt u} = (\partial^{\vec{\gamma}} {\tt u})_{\abs{\vec{\gamma}}=k}
\end{equation*}
as the vector of all $k$-th order weak partial derivatives of ${\tt u}$.
In particular, $\nabla^0 {\tt u} := {\tt u}$.
From the definition of $W^{l,p}(\Omega)$ it follows that $\nabla^k {\tt u} \in L^p(\Omega;\R^{{\cal N}(k)})$
for all $0 \le k \le l$, where ${\cal N}(k)$ denotes the dimension of the vector containing all derivatives.
\smallskip

For $1 \le p < \infty$ a norm on $W^{l,p}(\Omega)$ is defined by\index{norm!Sobolev}
\begin{equation} \label{eq:sobolev_gen}
\norm{{\tt u}}_{W^{l,p}} := \Bigl(\sum_{k=0}^l \norms{\nabla^k
	{\tt u}}_{L^p}^p\Bigr)^{1/p}\,,
\qquad
{\tt u} \in W^{l,p}(\Omega)\;.
\index{norm!$W^{l,p}(\Omega)$}
\end{equation}
Here 
$$\norms{\nabla^k {\tt u}}_{L^p}^p = \int_\Omega \norms{\nabla^k {\tt u}(\vs)}_2^p d \vs$$
denotes the $L^p$-norm of the gradient. Recall that $\nabla^k {\tt u}(\vs)$ is a vector with ${\cal N}(k)$ components.

In the case $p = \infty$ we define
\begin{equation*}
\norm{{\tt u}}_{W^{l,\infty}} := \max_{0 \le k \le l} \norms{\nabla^k {\tt u}}_\infty\,,
\qquad
{\tt u} \in W^{l,\infty}(\Omega)\;.
\end{equation*}

\subsection*{Embedding theorems}\index{embedding!Sobolev}

In the following we summarize \emph{Sobolev's embedding} theorems.
The results are collected from~\cite[Chap.~V]{Ada75}.
\begin{theorem}[Sobolev-embedding]\label{th:ap:sob_emb}\index{theorem!Sobolev-embedding}
	Assume that $\Omega \subset \R^m$ is an open set with Lipschitz boundary.
	Let $j$, $l \in \N \cup \{0\}$, and $1 \leq p < \infty$.
	\begin{enumerate}
		\item If
		\begin{equation*}
			l p < m \quad\text{ and }\quad p \leq q \leq \frac{mp}{m-lp}\,,
		\end{equation*}
		then the embedding
		\begin{equation*}
			i:W^{j+l,p}(\Omega) \to W^{j,q}(\Omega)\,,
			\qquad j \in \N \cup \{0\}\,,
		\end{equation*}
		is bounded.
		In particular, the embedding
		\begin{equation*}
			i:W^{l,p}(\Omega) \to L^q(\Omega)
		\end{equation*}
		is bounded.
		\item If
		\begin{equation*}
			l p = m\quad \text{ and }\quad p \leq q < \infty\,,
		\end{equation*}
		then the embedding
		\begin{equation*}
			i:W^{j+l,p}(\Omega) \to W^{j,q}(\Omega)\,,
			\qquad j \in \N \cup \{0\}\,,
		\end{equation*}
		is bounded.
		In particular the embedding
		\begin{equation*}
			i:W^{l,p}(\Omega) \to L^q(\Omega)
		\end{equation*}
		is bounded.
		\item If
		\begin{equation*}
			l p > m\,,
		\end{equation*}
		then the embedding
		\begin{equation*}
			i:W^{j+l,p}(\Omega) \to C_B^j(\Omega)\,,
			\qquad j \in \N \cup \{0\}\,,
		\end{equation*}
		is bounded. $C_B^j (\Omega)$ is the space of functions, where derivatives up to order $j$ are continuous and bounded on $\Omega$.
	\end{enumerate}
\end{theorem}

We need the following subspaces of $W^{l,p}(\Omega)$:
\begin{enumerate}
	\item For $l \in \N$ and $1 \leq p < \infty$,
	\begin{equation} \label{eq:sob_hom}
		W_0^{l,p}(\Omega) := \overline{C_0^\infty (\Omega)}\,,
		\index{Sobolev-space!homogeneous}
	\end{equation}
	where the closure is taken with respect to $\norm{\cdot}_{W^{l,p}}$. The space $C_0^\infty (\Omega)$ consists of all functions, which are infinitely often differentiable and which have compact support in $\Omega$.
	
	The space $W_0^{l,p}(\Omega)$ is called \emph{homogeneous Sobolev-space} of $l$-th order.
	\item We then define homogeneous Sobolev-spaces:
	\begin{equation*}
		W_0^{l,\infty}(\Omega) := \bigl\{ {\tt u} \in W^{l,\infty}(\Omega)\cap C^{l-1}_B(\Omega) : {\tt u} = 0 \text{ on } \partial \Omega\bigr\}\;.
	\end{equation*}
	Note that because of \autoref{th:ap:sob_emb} below it follows that
	$W^{l,\infty}(\Omega)\subset C^{l-1}_B(\Omega)$ in case $\partial\Omega$ is Lipschitz.
\end{enumerate}

\begin{theorem}
	Let $l \in \N$.
	The spaces $W^{l,p}(\Omega)$, $W_0^{l,p}(\Omega)$,
	and $W^{l,p}_0(\mathbb{S}^1\times\Omega)$ are
	Banach spaces when equipped with the norm $\norm{\cdot}_{W^{l,p}}$.
	If $1 \le p < \infty$ then all above spaces are separable, and if $1 < p < \infty$, they
	are reflexive. If $p=2$, they are Hilbert spaces.
\end{theorem}

\begin{theorem}\label{th:ap:sob_0}
	Let $l \in \N_0$. Then the following hold:
	\begin{enumerate}
		\item If $1\le p \le \infty$, then
		$W^{0,p}(\Omega) = L^p(\Omega)$.
		\item If $1\le p < \infty$, then
		$W_0^{0,p}(\Omega) = L^p(\Omega)$.
		\item If $1 \le p < \infty$, then $W^{l,p}(\Omega) \cap C^\infty(\Omega)$
		is dense in $W^{l,p}(\Omega)$.
		\item If $1\le p < \infty$ and $\partial \Omega$ is Lipschitz, then
		$\{{\tt u}|_\Omega : {\tt u} \in C_0^\infty(\R^m)\}$ is dense in $W^{l,p}(\Omega)$.
		In particular this applies to the case $\Omega =\R^m$,
		which shows that $W^{l,p}(\R^m) = W_0^{l,p}(\R^m)$.
		\item If $\Omega \subset \R^m$ is bocL and $1 \le p \le q \le \infty$,
		then $W^{l,q}(\Omega)$ is a dense subset of $W^{l,p}(\Omega)$.
	\end{enumerate}
\end{theorem}

Compact embeddings are usually referred to as Rellich--Kondra{\v s}ov
embedding theorems.
\begin{theorem}({\bf Rellich--Kondra{\v s}ov)} \index{theorem!Rellich--Kondra{\v s}ov}\label{th:ap:rellich}
	Let $\Omega \subset \R^m$ be an open set with Lipschitz boundary,
	and let $\Omega_0 \subset \Omega$ be a \emph{bounded} subdomain.
	For $j \in \N\cup\{0\}$, $l \in \N$, and $1 \leq p < \infty$ the
	following embeddings are compact:
	\begin{enumerate}
		\item For $lp<m$ and $1 \leq q < mp/(m-lp)$ the embedding
		\begin{equation*}
			i:W^{j+l,p}(\Omega) \to W^{j,q}(\Omega_0)\;.
		\end{equation*}
		\item For $l p = m$ and $1 \leq q < \infty$ the embedding
		\begin{equation*}
			i:W^{j+l,p}(\Omega) \to W^{j,q}(\Omega_0)\;.
		\end{equation*}
		\item For $lp>m$ the embedding
		\begin{equation*}
		i:W^{j+l,p}(\Omega) \to C^j(\ol{\Omega}_0)\;.
		\end{equation*}
	\end{enumerate}
\end{theorem}

\begin{remark}\label{re:ap:rellich}
	The compactness of the embedding $i: W^{j+l,p}(\Omega) \to W^{j,q}(\Omega_0)$ in particular
	implies that whenever $({\tt u}_k) \subset W^{j+l,p}(\Omega)$ weakly converges to ${\tt u}$,
	then $\bigl(i({\tt u}_k)\bigr)$ strongly converges to $i({\tt u})$ (that is, $i$ is weak-strong
	sequentially continuous).
	
	Indeed, if $({\tt u}_k)$ converges weakly to ${\tt u}$, $({\tt u}_k) \rightharpoonup {\tt u}$, then the sequence is bounded. Therefore the compactness of $i$
	implies that $i({\tt u}_k)$ has a strongly convergent subsequence converging to some ${\tt v} \in W^{j,q}(\Omega_0)$.
	Since $i({\tt u}_k)$ weakly converges to $i({\tt u})$, it follows that ${\tt v} = i({\tt u})$.
	Thus the strong convergence of $\bigl(i({\tt u}_k)\bigr)$ follows.
\end{remark}

\section{Barron-vector spaces} \label{sec:barron}
Barron-vector spaces $\mathcal{B}_p$ (originally introduced in \cite{Bar93}) play a fundamental role in the theoretical analysis of neural networks, and can be defined as follows:
\begin{definition}[Barron-energy]
	Let $\x: \Omega \subseteq \R^m \to \R$ be a distribution which admits the representation
	\begin{equation}\label{eq:barron_rep}
		\x = \lim_{\noc \to \infty} \x_\noc \text{ with }
		\x_\noc(\vs):= \sum_{j=1}^\noc \alpha_j^\noc	\sigma(\vw_j^\noc{}^T\vs + \theta_j^\noc)\,,
	\end{equation}
    where the convergence is in a distributional setting (see \autoref{def:convergence_distribution}).

    If the activation function is differentiable and satisfies
    \begin{equation} \label{eq:sigma}
    \sigma(0) = 0,
    \end{equation}
    we define the Barron-energy as the limes inferior of all \ALNN{}s that converge to $\x$ in a distributional sense:
    \begin{equation} \label{eq:barron_norm}
    	\norms{\x}_{\mathcal{B}_p}^p := \inf_{\x_\noc \to \x} \sum_{j=1}^\noc \abss{\alpha_j^\noc}^p\left(\norms{\vw_j^\noc}_1+\abss{\theta_j^\noc}\right)^p\;.
    \end{equation}
    The Barron-vector space consists of all tempered distributions, for which the Barron-energy is finite.
\end{definition}
\begin{remark}
	\begin{itemize}
		\item
		\autoref{eq:sigma} is essential, because otherwise we could find vectors $\vw_j^\noc$ and numbers $\theta_j$, $j=1,\ldots,\noc$ which are non-zero and the right hand side of \autoref{eq:barron_norm} is not zero, but a function $\x_\noc$ is indeed zero.
		\item Note that this definition of Barron-energies is conceptually similar as we used them to define energies of Tauber-Wiener functions (see \autoref{re:tw}).
		\item It is not the standard definition of Barron-spaces and norms as given for instance in \cite{EMaWu22}. Note that their definition does not lead to a space and norm in a mathematical sense.
		We used here a definition, which is borrowed from measure theory (see \cite{EvaGar15}) based on coverings, which we used already for Tauber-Wiener functions (compare \autoref{eq:norm}).
		\item The definition of Barron-energies \autoref{eq:barron_norm} and \autoref{eq:barron_norm_TW}, \autoref{eq:barron_norm_TW2} for Tauber-Wiener functions, are implicit, via approximating sequences (see \autoref{eq:barron_rep}). It would be interesting to find representing/convergent $(\alpha_k,\vw_k,\theta_k)_{k \in \N}$, which represents  the norm \autoref{eq:barron_norm} of an arbitrary function $\x$.
	\end{itemize}
\end{remark}

In the following we outline the proof of the triangle inequality of the Barron-energy:

\begin{lemma} \label{re:barron}
	$\norms{\cdot}_{\mathcal{B}_p}$ satisfies the triangle inequality.
\end{lemma}
\begin{proof}
	\begin{enumerate}
	\item The first step consists in reducing the triangle inequality for distributions to network functions:
	In fact we have by definition
	\begin{equation*}
		\norms{\x+\z}_{\mathcal{B}_p}^p = \inf_{\rho_{\noc^+} \to \x+\z} \sum_{j=1}^{\noc^+} \abss{\alpha_j^{\noc^+}}^p\left(\norms{\vw_j^{\noc^+}}_1+\abss{\theta_j^{\noc^+}}\right)^p\;.
	\end{equation*}
    We first note that if $\rho_{\noc^+} \to \x+\z$ and $\x$, $\z$ are also approximateable, then for every sequence $\rho_{\noc^+}$ we have by assumption also convergence of every summand: $\x_{\noc} \to \x$ and $\z_{\noc'}:=\rho_{\noc_+}-\x_{\noc} \to \z$. This means that
    \begin{equation*}
    	\norms{\x+\z}_{\mathcal{B}_p}^p = \inf_{\x_\noc \to \x \& \z_{\noc'}+\z} \sum_{j=1}^{\noc^+} \abss{\alpha_j^{\noc^+}}^p\left(\norms{\vw_j^{\noc^+}}_1+\abss{\theta_j^{\noc^+}}\right)^p\,,
    \end{equation*}
    where $\noc^+ \leq \noc'+\noc''$ is the number of approximating network functions. Making two summands, representing $\x_\noc \to \x$ and $\z_{\noc'} \to \z$ we get
	\begin{equation*}
	\norms{\x+\z}_{\mathcal{B}_p} \leq \norms{\x}_{\mathcal{B}_p}+\norms{\z}_{\mathcal{B}_p}\;.
\end{equation*}
	Therefore it suffices to prove the triangle for network functions:
	\item Let
	\begin{align*}
	\x_{1, \noc}(\vs)&= \sum_{j=1}^\noc \alpha_j^{1, \noc}\sigma(\vw_j^{1, \noc}{}^T\vs + \theta_j^{1, \noc})\,,\\
	\x_{2, \noc'}(\vs)&= \sum_{j=1}^{\noc'} \alpha_j^{2, \noc'}\sigma(\vw_j^{2, \noc'}{}^T\vs + \theta_j^{2,\noc'})\,,
	\end{align*}
    such that
    \begin{equation*}
    \lim_{\noc \to \infty} \x_{1, \noc} = \x_1 \text{ and }
    \lim_{\noc' \to \infty}\x_{2, \noc'} = \x_2
    \end{equation*}
    in a distributional setting. Now, we find the two index sets, where
    \begin{equation*}
    	\begin{aligned}
    		I_1 &:= \set{j \in \set{1,\ldots,\noc} : \forall_{k=1,\ldots,\noc'}
    			(\vw_j^{1,\noc},\theta_j^{1,\noc}) \neq (\vw_k^{2,\noc'},\theta_k^{2,\noc'})}, \\	
    		I_2 &:= \set{k \in \set{1,\ldots,\noc'} : \forall_{j=1,\ldots,\noc}
    			(\vw_k^{1,\noc'},\theta_k^{2,\noc'}) \neq (\vw_j^{1,\noc},\theta_j^{1,\noc})}\;.
    	\end{aligned}
    \end{equation*}
    Then it follows that
    \begin{align*}
    \x_{1, \noc}(\vs) + \x_{2, \noc'}(\vs) = &
    \sum_{j \in I_1} \alpha_j^{1, \noc}\sigma(\vw_j^{1, \noc}{}^T\vs + \theta_j^{1, \noc}) \\
    \\ & +
    \sum_{j \notin I_1} (\alpha_j^{1, \noc} + \alpha_{k(j)}^{2,\noc'}) \sigma(\vw_j^{1, \noc}{}^T\vs + \theta_j^{1, \noc}) \\
    \\ & +
    \sum_{k \in I_2} \alpha_k^{2, \noc'}\sigma(\vw_k^{2, \noc'}{}^T\vs + \theta_k^{2,\noc'})\;.
    %\\
    %\rightarrow & \x_1(\vs)+ \x_2(\vs).
    \end{align*}
    From this we see from the Minkowski inequality and \autoref{eq:barron_norm} that
    \begin{equation*}
    	\begin{aligned}
    			& \norms{\x_{1} + \x_{2}}_{\mathcal{B}_p}^p\\
    			\leq &\inf \left(\sum_{j=1}^\noc \abss{\alpha_j^{1, \noc}}^p\left(\norms{\vw_j^{1, \noc}}_1+\abss{\theta_j^{1, \noc}}\right)^p + \sum_{j=\noc+1}^{\noc+\noc'} \abss{\alpha_j^{2, \noc'}}^p\left(\norms{\vw_j^{2, \noc'}}_1+\abss{\theta_j^{2, \noc'}}\right)^p\right)\\
    	\end{aligned}
    \end{equation*}
    Using the properties of the infimum, we get
    \begin{align*}
    & \norms{\x_{1} + \x_{2}}_{\mathcal{B}_p}^p\\
    \leq & \inf \left(\sum_{j=1}^\noc \abss{\alpha_j^{1, \noc}}^p\left(\norms{\vw_j^{1, \noc}}_1+\abss{\theta_j^{1, \noc}}\right)^p\right) \\
    & \qquad +
    \inf \left(\sum_{j=\noc+1}^{\noc+\noc'} \abss{\alpha_j^{2, \noc'}}^p\left(\norms{\vw_j^{2, \noc'}}_1+\abss{\theta_j^{2, \noc'}}\right)^p\right) \\
    = & \norms{\x_{1}}_{\mathcal{B}_p}^p + \norms{\x_{2}}_{\mathcal{B}_p}^p\;.
    \end{align*}
    In order to finally prove the triangle inequality we use the concavity of
    function $s \to f(s) = s^{1/p}$.
\end{enumerate}
\end{proof}
However, a Barron-energy might not satisfy that if $\norms{\x}_{\mathcal{B}_p} = 0$ that then $\x=0$. Therefore the norm properties are not proven, as shown in \cite{EMaWu22} the terminologies 'space' and 'norm' are used in a loose way.

\begin{remark}
In \autoref{ch:nn}, we introduced an energy \autoref{eq:norm} for Tauber-Wiener functions. compared to neural networks the situation is simpler because the approximation does not use the scaling terms $\vw_j^\noc$. We defined two different energies, the $\mathcal{L}^1$-norm (\autoref{eq:normve}) and the $\mathcal{B}^p$-energy (\autoref{eq:barron_norm_TW2}), which did not require the restriction $\sigma(0) = 0$. Therefore, more activation functions from \autoref{ss:act} can be included. Some of them have been discussed in the context of regularization in \cite{LiLuMatPer24}.

The definition of the Barron-energy \autoref{eq:barron_norm} is somehow different to the standard definition, as for instance given in \cite{EMaWu22}. It is not clear that they are equivalent. It would also be interesting to know whether we can impose constraints such that the Barron-energy is a norm?.
\end{remark}

\section{Besov spaces} \label{sec:besov}
Besov spaces are function spaces that generalize Sobolev, H\"{o}lder, and $L^p$-spaces. They 
are widely used in many mathematical areas such as harmonic analysis and wavelets theory, see \cite{Dau92,Tri83,Tri92, HarKerPicTsy98, ChoBar03}.

Besov spaces can be defined in various ways: We define them based on orthonormal \emph{wavelet systems}:
\begin{definition}
For a scaling function $\phi(\cdot)$ and wavelet function $\psi(\cdot)$. If the set of functions 
%\begin{align*}
%\set{\phi_{0,\vec{i}}(\cdot) &:= 2^j \phi(2^j \cdot - \vec{i}):\R^m \to \R: j \in \N_0,\vec{i} \in \ZN^m}	\\
%\set{\psi_{j,\vec{i}}(\cdot) &:= 2^j \psi(2^j \cdot - \vec{i}):\R^m \to \R: j \in \N_0,\vec{i} \in \ZN^m}
%\end{align*}
\begin{equation*}
\set{\phi_{j,\vec{i}}(\cdot) := 2^j \phi(2^j \cdot - \vec{i}):\R^m \to \R: j \in \N_0,\vec{i} \in \ZN^m}
\end{equation*}
\begin{equation*}
\set{\psi_{j,\vec{i}}(\cdot) := 2^j \psi(2^j \cdot - \vec{i}):\R^m \to \R: j \in \N_0,\vec{i} \in \ZN^m}
\end{equation*}

forms an orthonormal basis for $L^2(\R^m)$ then it is called \emph{orthonormal wavelet system}.\index{orthonormal wavelet system} 
\end{definition}
As a consequence, for every $\x \in L^2(\R^m)$, we have
\begin{equation*}
\x(\vs) = \sum_{\vec{i} \in \ZN^m} u_{0,\vec{i}} \phi_{0,\vec{i}}(\vs) + \sum_{j=1}^{\infty}\sum_{\vec{i}\in \ZN^m} \varpi_{j,\vec{i}} \psi_{j,\vec{i}}(\vs) \text{ for all } \vs \in \R^m.
\end{equation*}
Here $u_{j,\vec{i}} := \int_{\R^m} \x(\vs) \phi_{j,\vec{i}}(\vs) d\vs $ and $\varpi_{j,\vec{i}} := \int_{\R^m} \x(\vs) \psi_{j,\vec{i}}(\vs) d\vs $ are the expansion coefficients with respect to the orthonormal basis.

\begin{definition}[from \cite{ChoBar03}]
Let $s > 0$, $1 < p \leq \infty$, and $1 \leq q \leq \infty$. Then the \emph{Besov space} $B^s_{p,q}(\R^m)$ consists of all general functions $\x$ for which
%such that the Besov norm\index{norm!Besov}\index{space!Besov} for an orthonormal wavelet system 
%$$\set{\psi_{j,\vec{i}}: j\in \N, \vec{i}\in \ZN^m}$$ in $L^2(\R^m)$. The Besov norm $\norms{\x}_{B^s_{p,q}}$ is equivalent to a sequence norm of the wavelet coefficients:
\begin{equation*}\label{eq:besov_norm}
\norms{\x}_{B^s_{p,q}} = \norms{u_{0,\vec{i}}}_{p} + \left( \sum_{j=1}^\infty 2^{jq(s+1-2/p)} \left(\sum_{\vec{i}\in \ZN^m} \abss{\varpi_{j,\vec{i}}}^{p} \right)^{q/p}\right)^{1/q} < \infty\;.
\end{equation*}
\end{definition}

\section{Dual spaces} \label{sec:dualspace}
Dual spaces are an important concept, which has in particular applications in wavelet theory and for finite elements, what is relevant for this book. One differs between \emph{algebraic dual spaces}\index{dual space!algebraic} and \emph{continuous dual spaces}\index{dual space!continuous} (see for instance \cite{DunSchw63}). We focus on the latter ones:
\begin{definition} 
	Let $\X$ be a normed vector spaces (commonly a Hilbert or Banach space). The \emph{continuous dual} $\X'$ is the set of all linear functionals
	\begin{equation*}
		\set{L : \X \to \R : \norm{L}_{\X'}:=\sup \set{\abs{L\x} : \norm{\x}_\X \leq 1} < \infty}\;.
	\end{equation*}
    This is again a normed space. 
\end{definition}
\begin{remark}
$\X'$ is even a Banach space as long as $\X$ is complete. One can consider these functionals also as distributions (see \autoref{sec:distribution}).
\end{remark}
\begin{example} \label{ex:dual}
	The dual of the Sobolev space $W^{s,2}(\Omega)$ (see \autoref{sob_def}) is denoted by $W^{-s,2}(\Omega)$. In fact the linear functionals of $W^{-s,2}(\Omega)$ are commonly treated as functions and we write for $\x \in W^{s,2}(\Omega)$ and $\x' \in W^{-s,2}(\Omega)$:
	\begin{equation*}
		\inner{\x'}{\x} := \x'[\x]\;. 
	\end{equation*}
    In most cases this is actually the $L^2$-inner product.
\end{example}

\section{Spherical harmonics}
\emph{Spherical harmonics}\index{spherical harmonics} are special functions defined on the surface of the sphere, $\mathbb{S}^{m-1}$, of $\R^m$.
\begin{definition}[Spherical Harmonics] \label{de:spherical_harmonics} Let $2 \leq m \in \N$ and $l \in \N$. We define
	\begin{equation*}
		N(m,l) = \frac{(2l+m-2)(m+l-3)!}{l! (m-2)!},\; N(m,0)=1\;.
	\end{equation*}
    A \emph{spherical harmonic} $Y_{l,m}$ is the restriction to $\mathbb{S}^{m-1}$ of a
    harmonic and homogeneous polynomial of degree $l$ on $\R^m$.
\end{definition}
$N(m,l)$ denotes the number of independent spherical harmonics of degree $l$ on $\R^m$. For different $l$ spherical harmonics are orthogonal.
\begin{example}
	\begin{itemize}
	\item For $m=2$ and $l \in \N$ we have $N(m,l)= 2$. The spherical harmonics are given by
	\begin{equation} \label{eq:speherical_harmonics2}
		Y_{l,1}(\phi) = \cos(l\phi) \text{ and } Y_{l,-1}(\theta) = \sin(l\phi)  \text{ for } l \in \N\;.\footnote{With a slight abuse of notation we identify $\theta = \begin{pmatrix} \cos (\phi)\\
				\sin (\phi) \end{pmatrix}$. Note that the function on the right hand side is a function of $\theta$ because it lives on $\mathbb{S}^1$. The right hand side is in polar coordinates.}
	\end{equation}
For $m=2$ and $l=0$ we have $Y_{0}(\phi)=1$.
    \item For $m=3$ we have in spherical coordinates,
    \begin{equation} \label{eq:speherical_harmonics3}
	Y_{l,m}(\theta,\phi) = c_{l,m} \e^{\i m \phi} P_l^m(\cos(\theta))\,,
\end{equation}
where $P_l^m :[-1,1] \to \R$ are the \emph{Legendre polynomials} and $c_{lm}$ is a normalizing constant, such that $\int_{\mathbb{S}^2} Y_{l,m}(\theta,\phi) d\theta d \phi =1$.
\end{itemize}
\end{example}

\section{Differentiability of functionals and operators}

We recall the definitions of directional derivatives of an operator $\opo$.
For a survey on various concepts of differentiability we refer to~\cite{Cla90}.

\begin{definition}\label{de:ap:diff}
	Let $\opo: \X \to \Y$ be an operator between normed spaces $\X$ and $\Y$.
	\begin{enumerate}
		\item The operator $\opo$ admits a one-sided directional derivative $\opo'[\x,\tt h] \in \Y$
		at $\x \in \X$ in direction ${\tt h} \in \X$, if \index{derivative!one sided}
		\begin{equation}\label{eq:ap:dirderivative}
			\opo'[\x,{\tt h}] = \lim_{t \to 0^+} \frac{\op{\x+t{\tt h}}-\op{\x}}{t}\;.
		\end{equation}
		\item Let $\x \in \X$, and assume that $\opo'[\x,\tt h]$ exists for all ${\tt h} \in \X$.
		If there exists a bounded linear operator $\opd{\x}:\X \to \Y$ such that
		\begin{equation*}
			\opo'[\x,\tt h] = \opd{\x} {\tt h} \; \text{ for all } \; {\tt h} \in \X\,,
		\end{equation*}
		then $\opo$ is \emph{G\^ateaux-differentiable},\index{operator!G\^ateaux-differentiable}
		and $\opd{\x}$ is called the \emph{G\^ateaux-derivative}\index{derivative!G\^ateaux}
		of $\opo$ at $\x$.
		\item The operator $\opo$ is \emph{Fr\'echet-differentiable}\index{operator!Fr\'echet-differentiable}
		at $\x$, if it is G\^ateaux-differentiable
		and the convergence in \autoref{eq:ap:dirderivative} is uniform with respect to
		${\tt h} \in \mathcal{B}(0;\rho)$ for some $\rho > 0$.
	\end{enumerate}
\end{definition}

\section[Convergence of Newton's method]{Convergence of Newton's method and the Newton-Mysovskii conditions} \label{sec:NN}
In this section we review convergence conditions for Newton type methods.
We consider first Newton methods for solving functional equations \autoref{eq:op}.
Decomposition cases (see \autoref{se:decomp}) of the operator $N$ will be considered afterwards. This subsection is relevant for parameter learning as discussed in \autoref{ch:learning}.

\subsection{Newton's method in finite dimensions}
For Newton methods \emph{local convergence} is guaranteed under \emph{affine covariant} conditions.\index{condition!affine invariant}
Several variants of such conditions have been proposed in the literature (see for instance \cite{DeuHei79,DeuPot92,NasChe93}).
Analysis of Newton method has been an active research area, see for instance \cite{Ort68,Schw79}.

\begin{theorem}[Finite-dimensional Newton method] \label{th:deupot92} Let $\opo: \dom{\opo} \subseteq \R^m \to \R^m$ be continuously Fr\'echet-differentiable on a non-empty, open and convex set $\dom{\opo}$. Let $\vp^\dagger \in \dom{\opo}$ be a solution\footnote{Note that in comparison with \autoref{cha:intro}, here $\X=\Y=\R^m$ and the elements of $\X$ are denoted by vectors $\vp$. Moreover we specify $\norm{\cdot}_\X = \norm{\cdot}_2$.} of \autoref{eq:op}
	\begin{equation*}
		\opo(\vp) = \vec{y}\;.
	\end{equation*}
	Moreover, we assume that
	\begin{enumerate}
		\item $\opo'(\vp)$ is invertible for all $\vp \in \dom{\opo}$ and that
		\item the \emph{Newton-Mysovskii type condition} holds: That is, there exist some $C_F > 0$ such that
		\begin{equation} \label{eq:wrnmi}
			\begin{aligned}
				\norm{\opo'(\vq)^{-1}(\opo'(\vp+s(\vq-\vp))-\opo'(\vp))(\vq-\vp)}_2 \leq s C_F \norm{\vp-\vq}_2^2 \\
				\text{ for all } \vp, \vq \in \dom{\opo}, s \in [0,1].
			\end{aligned}
		\end{equation}
	\end{enumerate}	
	Let $\vp^0 \in \dom{\opo}$ which satisfies
	\begin{equation}\label{eq:h}
		\overline{\mathcal{B}(\vp^0;\rho)} \subseteq \dom{\opo} \text{ with }\rho := \norms{\vp^\dagger-\vp^0}_2
		\text{ and } \frac{\rho C_I C_L}{2} <1.
	\end{equation}
	Here $C_L$ denotes the Lipschitz constant of the inverse $\opo'(\vp)$ in $\overline{\mathcal{B}(\vp^0;\rho)}$ with respect to the spectral norm (which is the operator norm, \autoref{de:op_norm}), which is denoted here by  $\norm{\cdot}_2\norm{\opo'(\vp)^{-1} -\opo'(\vq)^{-1}}_2$ (note however it is the norm of operators represented as matrices),  and $C_I$ denotes the norm of the inverse: That is
	\begin{equation} \label{eq:cicp}
		\norm{\opo'(\vp)^{-1} -\opo'(\vq)^{-1}}_2 \leq C_L \norm{\vp-\vq}_2 \text{ and } \norm{\opo'(\vp)^{-1}}_2 \leq C_I\;.
	\end{equation}
	Then the Newton iteration with starting point $\vp^0$,
	\begin{equation} \label{eq:newton_invert} \begin{aligned}
			\vp^{k+1} = \vp^k - \opo'(\vp^k)^{-1}(F(\vp^k)-\by)
			\quad k \in \N_0,
		\end{aligned}
	\end{equation}
	satisfies that the iterates $\set{\vp^k: k=0,1,2,\ldots}$ belong to $\overline{\mathcal{B}(\vp^0;\rho)}$ and
	converge quadratically to $\vp^\dagger \in  \overline{\mathcal{B}(\vp^0;\rho)}$.
\end{theorem}
\begin{remark} \autoref{th:deupot92dg} can be generalized to abstract spaces. In finite dimensions, if $\opo: \R^m \to \R^m$, local convergence of Newton's method is often proven under Lipschitz continuity of the forward operator     
	\begin{equation} \label{eq:cicpL}
		\norm{\opo'(\vp) -\opo'(\vq)}_2 \leq C_L^+ \norm{\vp-\vq}
	\end{equation}
	and the linearization has full rank. Local convergence follows from the \emph{implicit function theorem} (see for instance \cite{Dei85b,KraPar13}). 
	
	Of course the local convergence regions $\overline{\mathcal{B}(\vp^0;\rho)}$, where the estimates \autoref{eq:cicp} and \autoref{eq:cicpL} hold, might be different in both situations.
\end{remark}

\subsection{Newton's method for functionals}
Now, we study the case of convergence of Gauss-Newton methods where $\opo: \bP = \R^{\dimlimit} \to \Y$, where $\Y$ is an infinite dimensional Hilbert-space. We use the notation of \autoref{se:nnf}.
In the case of a functional $\opo$, the inverse, used in a classical Newton method (see \autoref{eq:newton_invert}), is not well-defined, because $\opd{\vp}$ might not have full rank (see item (i) in \autoref{th:deupot92}). Since a standard Newton's method require the linearizations to have full rank, we call Newton's method which deal with non full rank, Gauss-Newton methods.

Before we phrase a convergence result for Gauss-Newton methods we recall and introduce some definitions:
\begin{definition}[Inner, outer and Moore-Penrose inverse] \label{not:inverse} (see \cite{Nas76,Nas87})\index{inverse!outer}\index{inverse!inner}\index{inverse!Moore-Penrose}
	Let $L: \bP \to \Y$ be a linear and bounded operator mapping between two vector spaces\footnote{The space $\bP$ does not necessarily be finite dimensional} $\bP$ and $\Y$. Then
	\begin{enumerate}
		\item the operator $B: \Y \to \bP$ is called a \emph{left inverse} to $L$ if
		$$ B L = I\;. $$
		\item $B:\Y \to \bP$ is called a \emph{right inverse} to $L$ if
		$$ L B = I\;. $$
		Left and right inverses are used in different context:
		\begin{itemize}
			\item For a left inverse the nullspace of $L$ has to be trivial, in contrast to $B$.
			\item For a right inverse the nullspace of $B$ has to be trivial.
		\end{itemize}
		\item $B:\bP \to \bP$ is called a \emph{inverse} to $L$ if $B$ is a right and a left inverse.
		\item $B:\bP \to \Y$ is an \emph{outer inverse} to $L$ if
		\begin{equation} \label{eq:outer}
			BLB = B.
		\end{equation}
		\item Let $\bP$ and $\Y$ be Hilbert-spaces and $L: \bP \to \Y$ be a linear bounded operator.  We denote the orthogonal projections
		$P$ and $Q$ onto $\mathcal{N}(L)$, the nullspace of $L$ (which is closed), and $\overline{\range{L}}$, the closure of the range of $L$:
		That is for all ${\tt p} \in \bP$ and $\y \in \Y$ we have
		\begin{equation} \label{eq:proj1}
			\begin{aligned}
				P{\tt p} &= \argmin \set{\norm{{\tt p}_1-{\tt p}}_\bP: {\tt p}_1 \in \mathcal{N}(L)} \text{ and }\\
				Q\y &= \argmin \set{\norm{\y_1-\y}_\Y: \y_1 \in \overline{\mathcal{R}(L)}}.
			\end{aligned}
		\end{equation}
		We therefore have
		\begin{equation*}
			\begin{aligned}
				P: \bP & \to \mathcal{N}(L) \dot{+} \mathcal{N}(L)^\bot\,\\
				{\tt p} &\mapsto P {\tt p} + 0
			\end{aligned} \quad \text{ and }\quad
			\begin{aligned}
				Q: \Y & \to \overline{\range{L}} \dot{+} \range{L}^\bot. \\
				\y &\to Q\y + 0
			\end{aligned}
		\end{equation*}
		$B:\dom{B} \subseteq \Y \to \bP$ with $\dom{B}:= \range{L} \dot{+} \range{L}^\bot$ is called the \emph{Moore-Penrose inverse} of $L$ if the following identities hold
		\begin{equation} \label{eq:MP}
			\begin{aligned}
				LBL &=L,\\
				BLB &=B,\\
				BL &=I-P,\\
				LB &= Q|_{\dom{B}}.
			\end{aligned}
		\end{equation}
	\end{enumerate}	
\end{definition}
\begin{remark}
	The range of a neural network operator $\Psi$ (defined in \autoref{eq:classical_approximation}, \autoref{eq:general}, \autoref{eq:quadratic_approximation_fixed} and other definitions in \autoref{ch:nn}) forms a submanifold in $\X$. 
	Next we show the relation between submanifolds and the Moore-Penrose inverse.
\end{remark}

\begin{definition}[Lipschitz differentiable immersion] \label{de:immersion} Let $\Psi: \dom{\Psi} \subseteq \bP =\R^{\dimlimit}\to \X$
	where $\dom{\Psi}$ is open, non-empty, convex and $\X$ is a seperable (potentially infinite dimensional) Hilbert-space.
	\begin{enumerate}
		\item \label{it1:immersion}
		We assume that $\mathcal{M}:=\Psi(\dom{\Psi})$ is a $\dimlimit$-dimensional \emph{submanifold} in $\X$: That is, the following conditions are satisfied:
		\begin{itemize}
			\item Let for all $\vp = (p_1,\ldots,p_{\dimlimit})^T \in \dom{\Psi}$ denote with $\Psi'$  the Fr\'echet-derivative of $\Psi$ (defined in \autoref{eq:classical_approximation}, \autoref{eq:general}, \autoref{eq:quadratic_approximation_fixed} and other definitions in \autoref{ch:nn}):
			\begin{equation*}
				\begin{aligned}
					\Psi'[\vp]: \bP & \to \X\;.\\
					\vq  &\mapsto 
					\sum_{i=1}^{\dimlimit} \partial_{p_i} \Psi[\vp] q_i
				\end{aligned}
			\end{equation*}
			Here $\begin{pmatrix} \partial_{p_i} \Psi[\vp] \end{pmatrix}_{i=1,\ldots,\dimlimit}$ denotes the vector of functions consisting of all partial derivatives with respect to $\vp$. In differential geometry notation this coincides with the \emph{tangential mapping} $T_{\vp} \Psi$. 
			
			\item The \emph{representation mapping} of the derivative
			\begin{equation*}
				\begin{aligned}
					\Psi': \dom{\Psi} & \to \X^{\dimlimit},\\
					\vp &\mapsto \begin{pmatrix}
						\partial_{p_i} \Psi[\vp]
					\end{pmatrix}_{i=1,\ldots,\dimlimit}.
				\end{aligned}
			\end{equation*}
			has always the same rank $\dimlimit$ in $\dom{\Psi}$, meaning that all elements of $\partial_{\vp} \Psi[\vp]$ are linearly independent.
			This assumption means, in particular, that $\Psi$ is an \emph{immersion} and $\mathcal{M}$ is a submanifold.
		\end{itemize}
		\item \label{it2:immersion} For given $\vp \in \bP$ we define
		\begin{equation} \label{eq:bP} \begin{aligned}
				P_\vp : \X &\to \X_\vp:=\spann\set{\partial_{p_i} \Psi[\vp]: i=1,\ldots,\dimlimit},\\
				\x &\mapsto P_\vp \x := \argmin \set{\norm{\x_1-\x}_\X : \x_1 \in \X_\vp}
		\end{aligned} \end{equation}
		as the projection from
		$$\X = \X_\vp \dot{+}
		\X_\vp\!\!\!{}^\bot$$
		onto $\X_\vp$, which is well-defined by the closedness of the finite dimensional subspace $\X_\vp$.
		
		Next we define the inverse of $\Psi'[\vp]$ on $\X_\vp$:
		\begin{equation*}
			\begin{aligned}
				\Psi'[\vp]^{-1} : \X_\vp &\to \bP,\\
				\x = \sum_{i=1}^{\dimlimit} x_i \partial_{p_i} \Psi[\vp] &\mapsto (x_i)_{i=1}^{\dimlimit}
			\end{aligned}
		\end{equation*}
		which is extended to $\X$ as a projection 
		\begin{equation} \label{eq:MP_Penrose}
			\begin{aligned}
				\Psi'[\vp]^\dagger : \X = \X_\vp \dot{+} \X_\vp\!\!\!{}^\bot &\to \bP,\\
				\x = (\x_1,\x_2) &\mapsto \Psi'[\vp]^{-1} \x_1
			\end{aligned}
		\end{equation}
		which are both well-defined because we assume that $\Psi$ is an immersion.
		Note that $x_i$, $i=1,\ldots,\dimlimit$ are coordinates with respect to the basis $\set{\partial_{p_i} \Psi[\vp]:i=1,\ldots,\dimlimit}$. We use here the notation of a Moore-Penrose inverse, which it actually is, which is proven below (see \autoref{le:MPN}).
		\item \label{it3:immersion} Finally, we assume that the operators $\Psi'[\vp]$ 
		are locally bounded and Lipschitz continuous in $\dom{\Psi}$:
		That is
		\begin{equation} \label{eq:cl}
			\begin{aligned}
				\norm{\Psi'[\vp]-\Psi'[\vq]}_{\bP \to \X}
				\leq C_L \norm{\vp-\vq}_{\bP} \quad 
				\norm{\Psi'[\vp]}_{\bP \to \X}
				\leq C_I \text{ for } \vp, \vq \in \dom{\Psi}.	
			\end{aligned}
		\end{equation}
		If $\Psi$ satisfies these three properties we call it a \emph{Lipschitz differentiable immersion}.
	\end{enumerate}
    Typically one takes $\norms{\vp}_{\bP} = \norms{\vp}_2$, the Euclidean norm or a Barron-type energy (see \autoref{eq:barron_norm})
    \begin{equation*}
    	\norms{\vp}_{\bP} = \abs{\alpha}^p \left( \norm{\vw}_1 + \abs{\theta} \right)\;.
    \end{equation*}
\end{definition}
The following lemma is proved by standard means but rather tedious calculations in \cite{SchHofNas23}: 
\begin{lemma}\label{le:MPN} For a Lipschitz differentiable immersion
	\begin{itemize}
		\item the function $\Psi'[\vp]^\dagger: \X \to \bP$ is in fact the Moore-Penrose inverse of $\Psi'[\vp]$ and
		\item for every point $\vp \in \dom{\Psi} \subseteq \bP$ there exists a non-empty closed neighborhood where $\Psi'[\vp]^\dagger$
		is uniformly bounded and it is Lipschitz continuous; That is		
		\begin{equation} \label{eq:clb}
			\begin{aligned}
				\norm{\Psi'[\vp]^\dagger-\Psi'[\vq]^\dagger}_{\X \to \bP}
				\leq C_L \norm{\vp-\vq}_{\bP}, \quad
				\norm{\Psi'[\vp]^\dagger}_{\X \to \bP}
				\leq C_I \text{ for } \vp, \vq \in \dom{\Psi}.	
			\end{aligned}
		\end{equation}
		\item Moreover, the operator $P_\vp$ from \autoref{eq:bP} is bounded.
	\end{itemize}
\end{lemma}

\subsection{Gauss-Newton method}
In the following we study a Gauss-Newton method for solving 
\begin{equation} \label{eq:ip_2}
	N(\vp)= \op{\Psi(\vp)} = \y,
\end{equation}
where $\Psi: \dom{\Psi} \subseteq \R^{\dimlimit} \to \X$ is a Lipschitz-continuous immersion (see \autoref{de:immersion}), 
$\opo:\X \to \Y$ is bounded and $N=\opo \circ \Psi$. That is, we are concerned with an operator $N$ that satisfies the 2nd decomposition case as discussed in \autoref{ss:decomp2}.
\begin{remark}
	A curious thing with Newton's method, \autoref{eq:newton_invert}, is, that it annihilates the linear operator during iteration: Formally we have 
	\begin{equation} \label{eq:Newton_sec}
		\begin{aligned}
		\vp_{k+1} &= \vp_k - N'[\vp_k]^{-1}(N[\vp_k]-\y) \\
		&=\vp_k - \Psi'[\vp_k]^{-1} \opo^{-1} (\opo \Psi[\vp_k]- \opo \Psi[\vp^\dagger])\\
		&= \vp_k - \Psi'[\vp_k]^\dagger(\Psi[\vp_k]-\Psi[\vp^\dagger]).
		\end{aligned}
	\end{equation}
Note that the ill--posedness introduced by the compact operator is dismissed. However, the representation only works if $\y$ is noiseless, which means that the data $\y$ is attainable by a parametrization is actually the basis for proving quadratic convergence of \autoref{eq:Newton_sec}.
\end{remark}

\begin{figure}[h]
	\begin{center}
		\includegraphics[width=.6\linewidth]{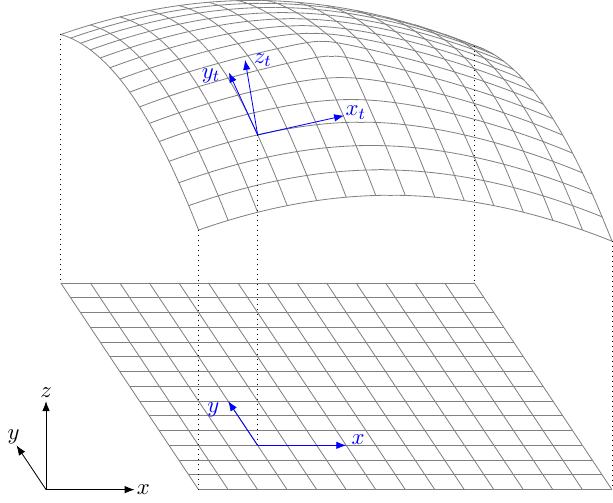}
		\caption{\label{fig:manifold} A Newton method on a manifold has to take into account that a point on the manifold lies in a higher-dimensional domain than the domain of definition (an atlas). So instead of the inverse in a Newton method, a Moore-Penrose inverse needs to be used. See \autoref{th:deupot92dg}.}
	\end{center}
\end{figure}

\begin{lemma} \label{le:dec} Let $F: \X \to \Y$ be linear, bounded, with trivial nullspace and dense range. Moreover, let $\Psi:\dom{\Psi} \subseteq \R^{\dimlimit} \to \X$ be a Lipschitz-differentiable immersion; see \autoref{de:immersion}, and let $N = F \circ \Psi$. 
	
	Note that by the definition of $N$, $\dom{N}=\dom{\Psi}$, and therefore for every 
	$\vp \in \dom{N}$ the derivative of the operator $N$ at a point $\vp$ has a Moore-Penrose inverse $N'[\vp]^\dagger$, 
	which satisfies:
	\begin{itemize}
		\item Decomposition property of the Moore-Penrose inverse:
		\begin{equation} \label{eq:MPa}
			N'[\vp]^\dagger \z = \Psi'[\vp]^\dagger \opo^{-1} \z \text{ for all } \vp \in \dom{N}, \z \in \range{F} \subseteq \Y.
		\end{equation}
		In particular this means that
		\begin{equation}\label{eq:MPb}
			N'[\vp]^\dagger N'[\vp] = I \text{ on } \R^{\dimlimit} \text{ and } 
			N'[\vp] N'[\vp]^\dagger = Q \text{ on } {\range{FP_\vp}},
		\end{equation}
		where 
		\begin{itemize} \item $I$ denotes the identity operator on $\R^{\dimlimit}$ and 
			\item $Q: \Y \to \overline{\range{FP_\vp}} \dot{+} \range{FP_\vp}^\bot$ is the orthogonal projection operator onto $\overline{\range{FP_\vp}}$.
			\end{itemize}
		\item Generalized Newton-Mysovskii condition:
		\begin{equation} \label{eq:wrnm} \begin{aligned}
				\norm{N'[\vp]^\dagger(N'[\vq+s(\vp-\vq)]-N'[\vq])(\vp-\vq)}_{\bP}  
				&\leq s C_I C_L \norm{\vp-\vq}_\bP^2\\
				& \quad \vp, \vq \in \dom{N}, s \in [0,1]\;.
			\end{aligned}
		\end{equation}
		We recall that the Lipschitz-constants $C_I$ and $C_L$ are defined in \autoref{eq:cl}.
	\end{itemize}
\end{lemma}

We have now all ingredients to state a local convergence rates result for a Gauss-Newton method, where the operator $N$ is the  composition of a linear bounded operator and a Lipschitz-differentiable immersions:
\begin{theorem} \label{th:deupot92dg} Let $F: \X \to \Y$ be linear, bounded, with trivial nullspace and dense range. Moreover, let 
	$\Psi: \dom{\Psi} \subseteq \bP \to \X$ be a Lipschitz-differentiable immersion with $\dom{\Psi}$ open, non-empty, and convex. Moreover, $N= \opo \circ \Psi : \dom{\Psi} \to \Y$. We assume that there exist $\vp^\dagger \in \dom{\Psi}$ that satisfies
	\begin{equation} \label{eq:sol}
		N[\vp^\dagger] = \y.
	\end{equation}
	Moreover, we assume that there exists $\vp^0 \in \dom{\Psi}$, which satisfies \autoref{eq:h}.
	Then, the iterates of the Gauss-Newton iteration,
	\begin{equation} \label{eq:newton} \begin{aligned}
			\vp_{k+1} = \vp_k - N'[\vp_k]^\dagger(N[\vp_k]-\y) 
			\quad k \in \N_0
		\end{aligned}
	\end{equation}
	are well-defined elements in $\overline{\mathcal{B}(\vp^0,\rho)}$ and converge quadratically to $\vp^\dagger$.
\end{theorem}
\begin{remark} \autoref{le:dec} shows that $\Psi(\vp)^\dagger \opo^{-1}$ is the Moore-Penrose inverse of the linearization of $N=\opo \Psi(\vp)$.
	In order to prove (quadratic) convergence of Gauss-Newton methods one requires an \emph{outer inverse} (see \autoref{not:inverse}). Following \cite{NasChe93} (see also \cite{Hae86}) the analysis of Gauss-Newton method could also be based on \emph{outer inverses} (compare \autoref{eq:outer} and \autoref{eq:MP}). 
	Expositions on Newton-Kantorovich-Mysovskii theory are \cite{KanAki64,OrtRhe70,Schw79} - here we replace the Newton-Kantorovich-Mysovskii conditions (see \autoref{eq:wrnmi}) by properties of an immersion, thus a differential geometry concept. For aspects related to Newton methods for singular points see \cite{DecKelKel83,Gri85}. For applications of generalized inverses in nonlinear analysis see 
	\cite{Nas87,Nas79}.
\end{remark}

%y\input{fe.tex}

%\begin{opq}\label{opq:barron_1}
%	We have provided two definitions of Barron-like energies for Tauber-Wiener functions in \autoref{eq:barron_norm_TW} and \autoref{eq:barron_norm_TW2}. We would like to know whether they are equivalent.
%\end{opq}

%
%\commentO{Check that $\mathcal{L}$ is always used for the same space. Once it appears in neural nets}
%
%

\backmatter
\ifdgruyter
  \printbibliography[env=bibnumeric]
\else
{
  \sloppy
  \chapter{Bibliography}
  \printbibliography[heading=none]

@BOOK{Ada75,
  AUTHOR = {Adams, R. A.},
  LOCATION = {New York},
  PUBLISHER = {Academic Press},
  DATE = {1975},
  ISBN = {9780080873817},
  NUMBER = {65},
  SERIES = {Pure and Applied Mathematics},
  TITLE = {{S}obolev Spaces},
}

@ARTICLE{AdlLunVerSchoOkt22,
  AUTHOR = {Adler, J. and Lunz, S. and Verdier, O. and Sch\"{o}nlieb, C.-B. and \"{O}ktem, O.},
  DATE = {2022},
  DOI = {10.1088/1361-6420/ac28ec},
  ISSN = {0266-5611},
  JOURNALTITLE = {Inverse Problems},
  NUMBER = {7},
  PAGES = {075006},
  SHORTJOURNAL = {Inverse Probl.},
  TITLE = {Task adapted reconstruction for inverse problems},
  VOLUME = {38},
}

@ARTICLE{AdlOkt18,
  AUTHOR = {Adler, J. and \"{O}ktem, O.},
  PUBLISHER = {IEEE},
  DATE = {2018},
  DOI = {10.1109/tmi.2018.2799231},
  ISSN = {0278-0062},
  JOURNALTITLE = {{IEEE} Transactions on Medical Imaging},
  NUMBER = {6},
  PAGES = {1322--1332},
  SHORTJOURNAL = {{IEEE} Trans. Med. Imag.},
  TITLE = {Learned Primal-Dual Reconstruction},
  VOLUME = {37},
}

@ARTICLE{AdlOkt17,
  AUTHOR = {Adler, J. and \"{O}ktem, O.},
  DATE = {2017},
  DOI = {10.1088/1361-6420/aa9581},
  ISSN = {0266-5611},
  JOURNALTITLE = {Inverse Problems},
  NUMBER = {12},
  PAGES = {124007},
  SHORTJOURNAL = {Inverse Probl.},
  TITLE = {Solving ill-posed inverse problems using iterative deep neural networks},
  VOLUME = {33},
}

@ARTICLE{AfkChuChu21,
  AUTHOR = {Afkham, B. M. and Chung, J. and Chung, M.},
  DATE = {2021},
  DOI = {10.1088/1361-6420/ac245d},
  ISSN = {0266-5611},
  JOURNALTITLE = {Inverse Problems},
  NUMBER = {10},
  PAGES = {105017},
  SHORTJOURNAL = {Inverse Probl.},
  TITLE = {Learning regularization parameters of inverse problems via deep neural networks},
  VOLUME = {37},
}

@ARTICLE{Afz25,
  AUTHOR = {Afzal Aghaei, A.},
  DATE = {2025},
  DOI = {10.1016/j.neucom.2025.129414},
  JOURNALTITLE = {Neurocomputing},
  PAGES = {129414},
  SHORTJOURNAL = {Neurocomputing},
  TITLE = {fKAN: Fractional Kolmogorov--Arnold Networks with trainable Jacobi basis functions},
  VOLUME = {623},
}

@ARTICLE{AghKilMil11,
  AUTHOR = {Aghasi, A. and Kilmer, M. and Miller, E. L.},
  DATE = {2011},
  DOI = {10.1137/100800208},
  ISSN = {1936-4954},
  JOURNALTITLE = {{SIAM} Journal on Imaging Sciences},
  NUMBER = {2},
  PAGES = {618--650},
  SHORTJOURNAL = {{SIAM} J. Imaging Sciences},
  TITLE = {Parametric Level Set Methods for Inverse Problems},
  VOLUME = {4},
}

@ARTICLE{AlfBibHamGaaGha23,
  AUTHOR = {Alfarra, M. and Bibi, A. and Hammoud, H. and Gaafar, M. and Ghanem, B.},
  DATE = {2023},
  DOI = {10.1109/tpami.2022.3201490},
  ISSN = {0162-8828},
  JOURNALTITLE = {{IEEE} Transactions on Pattern Analysis and Machine Intelligence},
  NUMBER = {4},
  PAGES = {5027--5037},
  SHORTJOURNAL = {{IEEE} Trans. Pattern Anal. Mach. Intell.},
  TITLE = {On the Decision Boundaries of Neural Networks: A Tropical Geometry Perspective},
  VOLUME = {45},
}

@ARTICLE{AlvGuiLioMor93,
  AUTHOR = {Alvarez, L. and Guichard, F. and Lions, P.-L. and Morel, J.-M.},
  LOCATION = {Berlin, Heidelberg},
  PUBLISHER = {Springer},
  DATE = {1993},
  ISSN = {0003-9527},
  JOURNALTITLE = {Archive for Rational Mechanics and Analysis},
  NUMBER = {3},
  PAGES = {199--257},
  SHORTJOURNAL = {Arch. Ration. Mech. Anal.},
  TITLE = {Axioms and fundamental equations of image processing},
  VOLUME = {123},
}

@ARTICLE{AlvLioMor92,
  AUTHOR = {Alvarez, L. and Lions, P.-L. and Morel, J.-M.},
  PUBLISHER = {SIAM},
  DATE = {1992},
  ISSN = {0036-1429},
  JOURNALTITLE = {{SIAM} Journal on Numerical Analysis},
  NUMBER = {3},
  PAGES = {845--866},
  SHORTJOURNAL = {{SIAM} J. Numer. Anal.},
  TITLE = {Image selective smoothing and edge detection by nonlinear diffusion. {II}},
  VOLUME = {29},
}

@ARTICLE{AmaHug91b,
  AUTHOR = {Amato, U. and Hughes, W.},
  DATE = {1991},
  DOI = {10.1088/0266-5611/7/6/004},
  ISSN = {0266-5611},
  JOURNALTITLE = {Inverse Problems},
  NUMBER = {6},
  PAGES = {793--808},
  SHORTJOURNAL = {Inverse Probl.},
  TITLE = {Maximum entropy regularization of Fredholm integral equations of the first kind},
  VOLUME = {7},
}

@INCOLLECTION{And86,
  AUTHOR = {Anderssen, R. S.},
  PUBLISHER = {Birkh\"{a}user, Basel},
  BOOKTITLE = {Inverse problems ({O}berwolfach, 1986)},
  DATE = {1986},
  PAGES = {11--30},
  SERIES = {Internat. Schriftenreihe Numer. Math.},
  TITLE = {The linear functional strategy for improperly posed problems},
  VOLUME = {77},
}

@ARTICLE{AndKutOktPet22,
  AUTHOR = {Andrade-Loarca, H. and Kutyniok, G. and \"{O}ktem, O. and Petersen, P.},
  DATE = {2022},
  DOI = {10.1016/j.acha.2021.12.007},
  ISSN = {1063-5203},
  JOURNALTITLE = {Applied and Computational Harmonic Analysis},
  PAGES = {155--197},
  SHORTJOURNAL = {Appl. Comput. Harmon. Anal.},
  TITLE = {Deep microlocal reconstruction for limited-angle tomography},
  VOLUME = {59},
}

@ARTICLE{AntHal21,
  AUTHOR = {Antholzer, S. and Haltmeier, M.},
  DATE = {2021},
  DOI = {10.3390/jimaging7110239},
  JOURNALTITLE = {Journal of Imaging},
  NUMBER = {11},
  PAGES = {239},
  SHORTJOURNAL = {J. Imaging},
  TITLE = {Discretization of Learned NETT Regularization for Solving Inverse Problems},
  VOLUME = {7},
}

@BOOK{Arn63,
  AUTHOR = {Arnol'd, V. I and et. al.},
  PUBLISHER = {American Mathematical Society},
  DATE = {1963},
  TITLE = {Translations. Ser. 2, 28, 16 papers on analysis},
}

@ARTICLE{ArnLog14,
  AUTHOR = {Arnold, D. N. and Logg, A.},
  URL = {https://www.siam.org/publications/siam-news/issues/volume-47-number-09-november-2014},
  DATE = {2014-11},
  JOURNALTITLE = {SIAM News},
  NUMBER = {9},
  TITLE = {Periodic Table of the Finite Elements},
  VOLUME = {47},
}

@ARTICLE{ArrMaaOktScho19,
  AUTHOR = {Arridge, S. and Maass, P. and \"{O}ktem, O. and Sch\"{o}nlieb, C.-B.},
  DATE = {2019},
  DOI = {10.1017/s0962492919000059},
  ISSN = {0962-4929},
  JOURNALTITLE = {Acta Numerica},
  PAGES = {1--174},
  SHORTJOURNAL = {Acta Numer.},
  TITLE = {Solving inverse problems using data-driven models},
  VOLUME = {28},
}

@ARTICLE{ArrSch12,
  AUTHOR = {Arridge, S. and Scherzer, O},
  PUBLISHER = {IOP Publishing},
  DATE = {2012},
  DOI = {10.1088/0266-5611/28/8/080201},
  ISSN = {0266-5611},
  JOURNALTITLE = {Inverse Problems},
  NUMBER = {8},
  PAGES = {080201},
  SHORTJOURNAL = {Inverse Probl.},
  TITLE = {Imaging from coupled physics},
  VOLUME = {28},
}

@ARTICLE{Arr99,
  AUTHOR = {Arridge, S. R.},
  DATE = {1999},
  DOI = {10.1088/0266-5611/15/2/022},
  ISSN = {0266-5611},
  JOURNALTITLE = {Inverse Problems},
  NUMBER = {2},
  PAGES = {R41--R93},
  SHORTJOURNAL = {Inverse Probl.},
  TITLE = {Optical tomography in medical imaging},
  VOLUME = {15},
}

@ARTICLE{ArrScho09,
  AUTHOR = {Arridge, S. R. and Schotland, J. C.},
  DATE = {2009},
  DOI = {10.1088/0266-5611/25/12/123010},
  ISSN = {0266-5611},
  JOURNALTITLE = {Inverse Problems},
  NUMBER = {12},
  PAGES = {123010},
  SHORTJOURNAL = {Inverse Probl.},
  TITLE = {Optical tomography: forward and inverse problems},
  VOLUME = {25},
}

@INPROCEEDINGS{Ash01,
  AUTHOR = {Ashikhmin, M.},
  BOOKTITLE = {Proceedings of the 2001 symposium on Interactive 3D graphics},
  DATE = {2001},
  DOI = {10.1145/364338.364405},
  PAGES = {217--226},
  TITLE = {Synthesizing natural textures},
}

@ARTICLE{AspBanOekSch20,
  AUTHOR = {Aspri, A. and Banert, S. and \"{O}ktem, O. and Scherzer, O.},
  DATE = {2020-03},
  DOI = {10.1080/01630563.2020.1740734},
  ISSN = {0163-0563},
  JOURNALTITLE = {Numerical Functional Analysis and Optimization},
  NUMBER = {10},
  PAGES = {1190--1227},
  SHORTJOURNAL = {Numer. Funct. Anal. Optim.},
  TITLE = {A Data-Driven Iteratively Regularized Landweber Iteration},
  VOLUME = {41},
}

@INCOLLECTION{AspFriKorSch21,
  AUTHOR = {Aspri, A. and Frischauf, L. and Korolev, Y. and Scherzer, O.},
  EDITOR = {Jadamba, B. and Khan, A. A. and Mig\'{o}rski, S. and Sama, M.},
  PUBLISHER = {CRC Press},
  BOOKTITLE = {Deterministic and Stochastic Optimal Control and Inverse Problems},
  DATE = {2021},
  DOI = {10.1201/9781003050575-13},
  PAGES = {303--318},
  TITLE = {Data Driven Reconstruction Using Frames and Riesz Bases},
}

@REPORT{AspFriSch24_preprint,
  AUTHOR = {Aspri, A. and Frischauf, L. and Scherzer, O.},
  URL = {https://arxiv.org/abs/2408.10690},
  DATE = {2024},
  NUMBER = {2408.10690},
  TITLE = {Spectral Function Space Learning and Numerical Linear Algebra Networks for Solving Linear Inverse Problems},
  TYPE = {Preprint on ArXiv},
}

@ARTICLE{AspFriSch25,
  AUTHOR = {Aspri, A. and Frischauf, L. and Scherzer, O.},
  DATE = {2025},
  DOI = {10.1553/etna_vol64s23},
  JOURNALTITLE = {Electronic Transactions on Numerical Analysis},
  PAGES = {23--44},
  SHORTJOURNAL = {Electron. Trans. Numer. Anal.},
  TITLE = {Spectral function space learning and numerical linear algebra networks for solving linear inverse problems},
  VOLUME = {64},
}

@ARTICLE{AspKorSch20,
  AUTHOR = {Aspri, A. and Korolev, Y. and Scherzer, O.},
  DATE = {2020-12},
  DOI = {10.1088/1361-6420/abb61b},
  ISSN = {0266-5611},
  JOURNALTITLE = {Inverse Problems},
  NUMBER = {12},
  PAGES = {125009},
  SHORTJOURNAL = {Inverse Probl.},
  TITLE = {Data driven regularization by projection},
  VOLUME = {36},
}

@ARTICLE{AspSch25,
  AUTHOR = {Aspri, A. and Scherzer, O.},
  EDITOR = {Bubba, T. A.},
  LOCATION = {Berlin, Boston},
  PUBLISHER = {De Gruyter},
  BOOKTITLE = {Data-driven Models in Inverse Problems},
  DATE = {2025},
  DOI = {10.1515/9783111251233-008},
  PAGES = {273--302},
  TITLE = {Analysis of generalized iteratively regularized Landweber iterations driven by data},
}

@ARTICLE{AttChbPeyRed18,
  AUTHOR = {Attouch, H. and Chbani, Z. and Peypouquet, J. and Redont, P.},
  DATE = {2018},
  DOI = {10.1007/s10107-016-0992-8},
  ISSN = {0025-5610},
  JOURNALTITLE = {Mathematical Programming},
  NUMBER = {1-2},
  PAGES = {123--175},
  SHORTJOURNAL = {Math. Program.},
  TITLE = {Fast convergence of inertial dynamics and algorithms with asymptotic vanishing viscosity},
  VOLUME = {168},
}

@ARTICLE{AttPeyRed16,
  AUTHOR = {Attouch, H. and Peypouquet, J. and Redont, P.},
  DATE = {2016},
  DOI = {10.1016/j.jde.2016.08.020},
  JOURNALTITLE = {Journal of Differential Equations},
  NUMBER = {10},
  PAGES = {5734--5783},
  SHORTJOURNAL = {J. Differential Equations},
  TITLE = {Fast convex optimization via inertial dynamics with {Hessian} driven damping},
  VOLUME = {261},
}

@ARTICLE{Aub67,
  AUTHOR = {Aubin, J. P.},
  PUBLISHER = {Scuola normale superiore},
  URL = {http://eudml.org/doc/83439},
  DATE = {1967},
  JOURNALTITLE = {Annali della Scuala Normale Superiore di Pisa. Classe di Scienze},
  NUMBER = {4},
  PAGES = {599--637},
  SHORTJOURNAL = {Ann. Sc. Norm. Super. Pisa Cl. Sci.},
  TITLE = {Behavior of the error of the approximate solutions of boundary value problems for linear elliptic operators by Galerkin's and finite difference methods},
  VOLUME = {21},
}

@INPROCEEDINGS{AujAubBlaCha03,
  AUTHOR = {Aujol, J.-F. and Aubert, G. and Blanc-F\'{e}raud, L. and Chambolle, A.},
  BOOKTITLE = {{\cite{GriLil03}}},
  DATE = {2003},
  PAGES = {297--312},
  TITLE = {Image decomposition application to {SAR} Images},
}

@ARTICLE{AujAubBlaCha05,
  AUTHOR = {Aujol, J.-F. and Aubert, G. and Blanc-F\'{e}raud, L. and Chambolle, A.},
  LOCATION = {Netherlands},
  PUBLISHER = {Springer},
  DATE = {2005},
  ISSN = {0924-9907},
  JOURNALTITLE = {Journal of Mathematical Imaging and Vision},
  NUMBER = {1},
  PAGES = {71--88},
  SHORTJOURNAL = {J. Math. Imaging Vision},
  TITLE = {Image decomposition into a bounded variation component and an oscillating component},
  VOLUME = {22},
}

@ARTICLE{AujCha05,
  AUTHOR = {Aujol, J.-F. and Chambolle, A.},
  LOCATION = {Netherlands},
  PUBLISHER = {Springer},
  DATE = {2005},
  ISSN = {0920-5691},
  JOURNALTITLE = {International Journal of Computer Vision},
  NUMBER = {1},
  PAGES = {85--104},
  SHORTJOURNAL = {Int. J. Comput. Vision},
  TITLE = {Dual Norms and Image Decomposition Models},
  VOLUME = {63},
}

@ARTICLE{AujGilChaOsh06,
  AUTHOR = {Aujol, J.-F. and Gilboa, G. and Chan, T. and Osher, S.},
  LOCATION = {Netherlands},
  PUBLISHER = {Springer},
  DATE = {2006},
  ISSN = {0920-5691},
  JOURNALTITLE = {International Journal of Computer Vision},
  NUMBER = {1},
  PAGES = {111--136},
  SHORTJOURNAL = {Int. J. Comput. Vision},
  TITLE = {Structure-Texture Image Decomposition---Modeling, Algorithms, and Parameter Selection},
  VOLUME = {67},
}

@ARTICLE{BahJelMer83,
  AUTHOR = {Bahl, L. R. and Jelinek, F. and Mercer, R. L.},
  DATE = {1983},
  DOI = {10.1109/tpami.1983.4767370},
  ISSN = {0162-8828},
  JOURNALTITLE = {{IEEE} Transactions on Pattern Analysis and Machine Intelligence},
  NUMBER = {2},
  PAGES = {179--190},
  SHORTJOURNAL = {{IEEE} Trans. Pattern Anal. Mach. Intell.},
  TITLE = {A Maximum Likelihood Approach to Continuous Speech Recognition},
  VOLUME = {PAMI-5},
}

@ARTICLE{Bak84,
  AUTHOR = {Bakushinskii, A. B.},
  DATE = {1984},
  JOURNALTITLE = {Zhurnal Vychislitelnoi Matematiki i Matematicheskoi Fiziki},
  PAGES = {1258--1259},
  SHORTJOURNAL = {Zh. Vychisl. Mat. Mat. Fiz.},
  TITLE = {Remarks on the choice of regularization parameter from quasioptimality and relation tests},
  VOLUME = {24},
}

@ARTICLE{Bak92,
  AUTHOR = {Bakushinskii, A. B.},
  DATE = {1992},
  JOURNALTITLE = {Computational Mathematics and Mathematical Physics},
  NUMBER = {9},
  PAGES = {1353--1359},
  SHORTJOURNAL = {Comput. Math. Math. Phys.},
  TITLE = {The problem of the convergence of the iteratively regularized {G}au\ss{}--{N}ewton method},
  VOLUME = {32},
}

@BOOK{BakGon94,
  AUTHOR = {Bakushinskii, A.B. and Goncharskii, A.V.},
  LOCATION = {Dordrecht, Boston, London},
  PUBLISHER = {Kluwer Academic Publishers},
  DATE = {1994},
  TITLE = {Ill-Posed Problems: Theory and Applications},
}

@BOOK{BakGon89,
  AUTHOR = {Bakushinskii, A.B. and Goncharskii, A.V.},
  LOCATION = {Moscow},
  PUBLISHER = {Nauka},
  DATE = {1989},
  TITLE = {Iterative Methods for the Solution of Incorrect Problems},
}

@BOOK{BakKok04,
  AUTHOR = {Bakushinsky, A. B. and Kokurin, M. Y.},
  LOCATION = {Dordrecht},
  PUBLISHER = {Springer},
  DATE = {2004},
  SERIES = {Mathematics and Its Applications},
  TITLE = {Iterative Methods for Approximate Solution of Inverse Problems},
  VOLUME = {577},
}

@INCOLLECTION{Bal12,
  AUTHOR = {Bal, G.},
  EDITOR = {Uhlmann, G.},
  LOCATION = {Cambridge},
  PUBLISHER = {Cambridge University Press},
  BOOKTITLE = {Inverse Problems and Applications: Inside Out {II}},
  DATE = {2012},
  PAGES = {325--368},
  SERIES = {Mathematical Sciences Research Institute Publications},
  TITLE = {Hybrid inverse problems and internal functionals},
  VOLUME = {60},
}

@ARTICLE{BalSch10,
  AUTHOR = {Bal, G. and Schotland, J. C.},
  DATE = {2010},
  JOURNALTITLE = {Physical Review Letters},
  PAGES = {043902},
  SHORTJOURNAL = {Phys. Rev. Lett.},
  TITLE = {Inverse scattering and acousto-optic imaging},
  VOLUME = {104},
}

@ARTICLE{BalUhl12,
  AUTHOR = {Bal, G. and Uhlmann, G.},
  DATE = {2012},
  DOI = {10.1016/j.aml.2012.03.005},
  JOURNALTITLE = {Applied Mathematics Letters},
  NUMBER = {7},
  PAGES = {1030--1033},
  SHORTJOURNAL = {Appl. Math. Lett.},
  TITLE = {Reconstructions for some coupled-physics inverse problems},
  VOLUME = {25},
}

@INPROCEEDINGS{BalBar18,
  AUTHOR = {Balestriero, R. and Baraniuk, R. B.},
  EDITOR = {Dy, J. and Krause, A.},
  PUBLISHER = {PMLR},
  URL = {https://proceedings.mlr.press/v80/balestriero18b.html},
  BOOKTITLE = {Proceedings of the 35th International Conference on Machine Learning},
  DATE = {2018},
  PAGES = {374--383},
  SERIES = {Proceedings of Machine Learning Research},
  TITLE = {A Spline Theory of Deep Networks},
  VOLUME = {80},
}

@ARTICLE{BalBerCas01,
  AUTHOR = {Ballester, C. and Bertalmio, M. and Caselles, V. and Sapiro, G. and Verdera, J.},
  PUBLISHER = {IEEE},
  DATE = {2001},
  ISSN = {1057-7149},
  JOURNALTITLE = {{IEEE} Transactions on Image Processing},
  NUMBER = {8},
  PAGES = {1200--1211},
  SHORTJOURNAL = {{IEEE} Trans. Image Process.},
  TITLE = {Filling-in by joint interpolation of vector fields and grey levels},
  VOLUME = {10},
}

@INCOLLECTION{BarBro90,
  AUTHOR = {Barber, D. C. and Brown, B. H.},
  LOCATION = {Philadelphia, PA},
  PUBLISHER = {SIAM},
  BOOKTITLE = {Inverse problems in partial differential equations (Arcata, CA, 1989)},
  DATE = {1990},
  PAGES = {151--164},
  TITLE = {Progress in electrical impedance tomography},
}

@ARTICLE{BarGok04,
  AUTHOR = {Barbone, P. E. and Gokhale, N. H.},
  DATE = {2004},
  DOI = {10.1088/0266-5611/20/1/017},
  ISSN = {0266-5611},
  JOURNALTITLE = {Inverse Problems},
  NUMBER = {1},
  PAGES = {283--296},
  SHORTJOURNAL = {Inverse Probl.},
  TITLE = {Elastic modulus imaging: on the uniqueness and nonuniqueness of the elastography inverse problem in two dimensions},
  VOLUME = {20},
}

@ARTICLE{BarObe07,
  AUTHOR = {Barbone, P. E. and Oberai, A. A.},
  PUBLISHER = {IOP Publishing Ltd},
  DATE = {2007},
  ISSN = {0031-9155},
  JOURNALTITLE = {Physics in Medicine {\&} Biology},
  NUMBER = {6},
  PAGES = {1577--1593},
  SHORTJOURNAL = {Phys. Med. Biol.},
  TITLE = {Elastic modulus imaging: some exact solutions of the compressible elastography inverse problem},
  VOLUME = {52},
}

@BOOK{BarSwi81,
  AUTHOR = {Barrett, H.H. and Swindell, W.},
  LOCATION = {New York},
  PUBLISHER = {Academic Press},
  DATE = {1981},
  PAGES = {384},
  TITLE = {Radiological Imaging: The Theory of Image Formation, Detection and Processing},
  VOLUME = {II},
}

@ARTICLE{Bar93,
  AUTHOR = {Barron, A. R.},
  DATE = {1993},
  DOI = {10.1109/18.256500},
  ISSN = {0018-9448},
  JOURNALTITLE = {{IEEE} Transactions on Information Theory},
  NUMBER = {3},
  PAGES = {930--945},
  SHORTJOURNAL = {{IEEE} Trans. Inf. Theory},
  TITLE = {Universal approximation bounds for superpositions of a sigmoidal function},
  VOLUME = {39},
}

@ARTICLE{BarCohDahDev08,
  AUTHOR = {Barron, A. R. and Cohen, A. and Dahmen, W. and DeVore, R. A.},
  DATE = {2008},
  DOI = {10.1214/009053607000000631},
  ISSN = {0090-5364},
  JOURNALTITLE = {Annals of Statistics},
  NUMBER = {1},
  SHORTJOURNAL = {Ann. Statist.},
  TITLE = {Approximation and learning by greedy algorithms},
  VOLUME = {36},
}

@ARTICLE{BasBoi98,
  AUTHOR = {Bastin, F. and Boigelot, C.},
  DATE = {1998},
  DOI = {10.1007/bf02479678},
  JOURNALTITLE = {The Journal of Fourier Analysis and Applications},
  NUMBER = {6},
  PAGES = {749--768},
  SHORTJOURNAL = {J. Fourier Anal. Appl.},
  TITLE = {Biorthogonal wavelets in $H^m(\mathbb{R})$},
  VOLUME = {4},
}

@ARTICLE{BecTeb09,
  AUTHOR = {Beck, A. and Teboulle, M.},
  PUBLISHER = {IEEE},
  DATE = {2009},
  DOI = {10.1109/tip.2009.2028250},
  ISSN = {1057-7149},
  JOURNALTITLE = {{IEEE} Transactions on Image Processing},
  NUMBER = {11},
  PAGES = {2419--2434},
  SHORTJOURNAL = {{IEEE} Trans. Image Process.},
  TITLE = {Fast gradient-based algorithms for constrained total variation image denoising and deblurring problems},
  VOLUME = {18},
}

@ARTICLE{Bel58b,
  AUTHOR = {Bellman, R.},
  DATE = {1958},
  DOI = {10.1090/qam/102435},
  JOURNALTITLE = {Quarterly of Applied Mathematics},
  NUMBER = {1},
  PAGES = {87--90},
  SHORTJOURNAL = {Q. Appl. Math.},
  TITLE = {On a routing problem},
  VOLUME = {16},
}

@ARTICLE{Bel57b,
  AUTHOR = {Bellman, Richard},
  DATE = {1957},
  JOURNALTITLE = {Journal of Mathematics and Mechanics},
  NUMBER = {5},
  PAGES = {679--684},
  TITLE = {A Markovian decision process},
  VOLUME = {6},
}

@ARTICLE{BenBubRatRic24,
  AUTHOR = {Benning, M. and Bubba, T. A. and Ratti, L. and Riccio, D.},
  DATE = {2024},
  DOI = {10.1093/imamat/hxae008},
  JOURNALTITLE = {{IMA} Journal of Applied Mathematics},
  NUMBER = {1},
  PAGES = {12--43},
  SHORTJOURNAL = {{IMA} J. Appl. Math.},
  TITLE = {Trust your source: quantifying source condition elements for variational regularisation methods},
  VOLUME = {89},
}

@ARTICLE{BerHooFauSch16,
  AUTHOR = {Beretta, E. and de Hoop, M. V. and Faucher, F. and Scherzer, O.},
  DATE = {2016},
  DOI = {10.1137/15m1043856},
  ISSN = {0036-1410},
  ISSUE = {6},
  JOURNALTITLE = {{SIAM} Journal on Mathematical Analysis},
  PAGES = {3962--3983},
  SHORTJOURNAL = {{SIAM} J. Math. Anal.},
  TITLE = {Inverse boundary value problem for the {H}elmholtz equation: quantitative conditional {L}ipschitz stability estimates},
  VOLUME = {48},
}

@ARTICLE{BerHooFraVesZha17,
  AUTHOR = {Beretta, E. and de Hoop, M. V. and Francini, E. and Vessella, S. and Zhai, J.},
  DATE = {2017},
  DOI = {10.1088/1361-6420/aa5bef},
  ISSN = {0266-5611},
  JOURNALTITLE = {Inverse Problems},
  NUMBER = {3},
  PAGES = {035013},
  SHORTJOURNAL = {Inverse Probl.},
  TITLE = {Uniqueness and Lipschitz stability of an inverse boundary value problem for time-harmonic elastic waves},
  VOLUME = {33},
}

@BOOK{BerBoc98,
  AUTHOR = {Bertero, M. and Boccacci, P.},
  LOCATION = {Bristol},
  PUBLISHER = {Institute of Physics Publishing},
  DATE = {1998},
  DOI = {10.1887/0750304359},
  ISBN = {0-7503-0439-1; 0-7503-0435-9},
  PAGETOTAL = {xii+351},
  TITLE = {Introduction to inverse problems in imaging},
}

@ARTICLE{BerPogTor88,
  AUTHOR = {Bertero, M. and Poggio, T. A. and Torre, V.},
  PUBLISHER = {IEEE},
  URL = {http://ieeexplore.ieee.org/xpls/abs\_all.jsp?arnumber=5962},
  DATE = {1988},
  ISSN = {0018-9219},
  JOURNALTITLE = {Proceedings of the {IEEE}},
  NUMBER = {8},
  PAGES = {869--889},
  SHORTJOURNAL = {Proc. {IEEE}},
  TITLE = {{Ill-posed problems in early vision}},
  VOLUME = {76},
}

@ARTICLE{Bis24,
  AUTHOR = {Bischoff, M.},
  URL = {https://www.scientificamerican.com/article/an-alternative-to-conventional-neural-networks-could-help-reveal-what-ai-is/},
  DATE = {2024},
  JOURNALTITLE = {Scientific American},
  SHORTJOURNAL = {Sci. Am.},
  TITLE = {Artificial intelligence neural network layers, conceptual illustration},
  VOLUME = May,
}

@BOOK{Bis95,
  AUTHOR = {Bishop, C. M.},
  LOCATION = {Oxford},
  PUBLISHER = {Clarendon Press},
  DATE = {1995},
  TITLE = {Neural Networks for Pattern Recognition},
}

@BOOK{Bla96,
  AUTHOR = {Blaschke, B.},
  LOCATION = {Linz, Austria},
  PUBLISHER = {Universit\"{a}tsverlag Rudolf Trauner},
  DATE = {1996},
  TITLE = {Some Newton Type Methods for the Regularization of Nonlinear Ill-Posed Problems},
}

@ARTICLE{BlaNeuSch97,
  AUTHOR = {Blaschke, B. and Neubauer, A. and Scherzer, O.},
  DATE = {1997},
  DOI = {10.1093/imanum/17.3.421},
  ISSN = {0272-4979},
  JOURNALTITLE = {{IMA} Journal of Numerical Analysis},
  NUMBER = {3},
  PAGES = {421--436},
  SHORTJOURNAL = {{IMA} J. Numer. Anal.},
  TITLE = {On convergence rates for the iteratively regularized {G}auss-{N}ewton method},
  VOLUME = {17},
}

@ARTICLE{BocFonKut23,
  AUTHOR = {Boche, H. and Fono, A. and Kutyniok, G.},
  DATE = {2023},
  DOI = {10.1109/tit.2023.3326879},
  ISSN = {0018-9448},
  JOURNALTITLE = {{IEEE} Transactions on Information Theory},
  NUMBER = {12},
  PAGES = {7887--7908},
  SHORTJOURNAL = {{IEEE} Trans. Inf. Theory},
  TITLE = {Limitations of Deep Learning for Inverse Problems on Digital Hardware},
  VOLUME = {69},
}

@ARTICLE{BonBreLorMaa07,
  AUTHOR = {Bonesky, T. and Bredies, K. and Lorenz, D. A. and Maass, P.},
  PUBLISHER = {Institute of Physics Publishing},
  DATE = {2007},
  DOI = {10.1088/0266-5611/23/5/014},
  ISSN = {0266-5611},
  JOURNALTITLE = {Inverse Problems},
  NUMBER = {5},
  PAGES = {2041--2058},
  SHORTJOURNAL = {Inverse Probl.},
  TITLE = {{A generalized conditional gradient method for nonlinear operator equations with sparsity constraints}},
  VOLUME = {23},
}

@ARTICLE{BonKazMaaSchoSchu08,
  AUTHOR = {Bonesky, T. and Kazimierski, K. S. and Maass, P. and Sch\"{o}pfer, F. and Schuster, T.},
  URL = {http://www.hindawi.com/GetArticle.aspx?doi=10.1155/2008/192679&e=ref},
  DATE = {2008},
  JOURNALTITLE = {Abstract and Applied Analysis},
  PAGES = {Art. ID 192679,19},
  SHORTJOURNAL = {Abstr. Appl. Anal.},
  TITLE = {Minimization of {T}ikhonov functionals in {B}anach spaces},
  VOLUME = {2008},
}

@ARTICLE{BotDonElbSch22,
  AUTHOR = {Bo\c{t}, R. and Dong, G. and Elbau, P. and Scherzer, O.},
  DATE = {2022},
  DOI = {10.1007/s10208-021-09536-6},
  JOURNALTITLE = {Foundations of Computational Mathematics},
  PAGES = {1567--1629},
  SHORTJOURNAL = {Found. Comput. Math.},
  TITLE = {Convergence Rates of First and Higher Order Dynamics for Solving Linear Ill-posed Problems},
  VOLUME = {22},
}

@BOOK{Boy20,
  AUTHOR = {Boyd, Robert W.},
  PUBLISHER = {Elsevier},
  URL = {https://shop.elsevier.com/books/nonlinear-optics/boyd/978-0-12-811002-7},
  DATE = {2020-07},
  EDITION = {4},
  ISBN = {978-0-12-811002-7},
  TITLE = {Nonlinear Optics},
}

@BOOK{BreSco08,
  AUTHOR = {Brenner, S. C. and Scott, L. R.},
  PUBLISHER = {Springer New York, NY},
  DATE = {2008},
  DOI = {10.1007/978-0-387-75934-0},
  TITLE = {The Mathematical Theory of Finite Element Methods},
}

@BOOK{Bre11b,
  AUTHOR = {Brezis, H.},
  DATE = {2011},
  DOI = {10.1007/978-0-387-70914-7},
  TITLE = {Functional Analysis, Sobolev Spaces and Partial Differential Equations},
}

@ARTICLE{BroBroZibAzh02,
  AUTHOR = {Bronstein, M. M. and Bronstein, A. M. and Zibulevsky, M. and Azhari, H.},
  PUBLISHER = {IEEE},
  DATE = {2002},
  ISSN = {0278-0062},
  JOURNALTITLE = {{IEEE} Transactions on Medical Imaging},
  NUMBER = {11},
  PAGES = {1395--1401},
  SHORTJOURNAL = {{IEEE} Trans. Med. Imag.},
  TITLE = {Reconstruction in diffraction ultrasound tomography using nonuniform FFT},
  VOLUME = {21},
}

@ARTICLE{Bro67,
  AUTHOR = {Browder, F. E.},
  LOCATION = {Berlin, Heidelberg},
  PUBLISHER = {Springer},
  DATE = {1967},
  DOI = {10.1007/bf00251595},
  ISSN = {0003-9527},
  JOURNALTITLE = {Archive for Rational Mechanics and Analysis},
  NUMBER = {1},
  PAGES = {82--90},
  SHORTJOURNAL = {Arch. Ration. Mech. Anal.},
  TITLE = {Convergence of approximants to fixed points of nonexpansive nonlinear mappings in banach spaces},
  VOLUME = {24},
}

@ARTICLE{BroPet67,
  AUTHOR = {Browder, F.E. and Petryshyn, W.V.},
  DATE = {1967},
  ISSN = {0022-247X},
  JOURNALTITLE = {Journal of Mathematical Analysis and Applications},
  PAGES = {197--228},
  SHORTJOURNAL = {J. Math. Anal. Appl.},
  TITLE = {Construction of fixed points of nonlinear mappings in {H}ilbert space},
  VOLUME = {20},
}

@INCOLLECTION{BroWei04,
  AUTHOR = {Brox, T. and Weickert, J.},
  BOOKTITLE = {Lecture Notes in Computer Science},
  DATE = {2004},
  DOI = {10.1007/978-3-540-28649-3_51},
  PAGES = {415--423},
  TITLE = {Level Set Based Image Segmentation with Multiple Regions},
}

@ARTICLE{BubKutLasMarSam19,
  AUTHOR = {Bubba, T. A. and Kutyniok, G. and Lassas, M. and M\"{a}rz, M. and Samek, W. and Siltanen, S. and Srinivasan, V.},
  DATE = {2019},
  DOI = {10.1088/1361-6420/ab10ca},
  ISSN = {0266-5611},
  JOURNALTITLE = {Inverse Problems},
  NUMBER = {6},
  PAGES = {064002},
  SHORTJOURNAL = {Inverse Probl.},
  TITLE = {Learning the invisible: a hybrid deep learning-shearlet framework for limited angle computed tomography},
  VOLUME = {35},
}

@BOOK{Buh03,
  AUTHOR = {Buhmann, M. D.},
  DATE = {2003},
  DOI = {10.1017/cbo9780511543241},
  TITLE = {Radial Basis Functions},
}

@ARTICLE{Bur01,
  AUTHOR = {Burger, M.},
  DATE = {2001},
  DOI = {10.1088/0266-5611/17/5/307},
  ISSN = {0266-5611},
  JOURNALTITLE = {Inverse Problems},
  NUMBER = {5},
  PAGES = {1327--1355},
  SHORTJOURNAL = {Inverse Probl.},
  TITLE = {A level set method for inverse problems},
  VOLUME = {17},
}

@ARTICLE{BurEng00,
  AUTHOR = {Burger, M. and Engl, H. W.},
  PUBLISHER = {Springer US},
  DATE = {2000},
  DOI = {10.1023/a:1016641629556},
  ISSN = {1019-7168},
  JOURNALTITLE = {Advances in Computational Mathematics},
  NUMBER = {4},
  PAGES = {335--354},
  SHORTJOURNAL = {Adv. Comput. Math.},
  TITLE = {Training neural networks with noisy data as an ill-posed problem},
  VOLUME = {13},
}

@ARTICLE{BurGilMoeEckCre16,
  AUTHOR = {Burger, M. and Gilboa, G. and Moeller, M. and Eckardt, L. and Cremers, D.},
  DATE = {2016},
  DOI = {10.1137/15m1054687},
  ISSN = {1936-4954},
  JOURNALTITLE = {{SIAM} Journal on Imaging Sciences},
  NUMBER = {3},
  PAGES = {1374--1408},
  SHORTJOURNAL = {{SIAM} J. Imaging Sciences},
  TITLE = {Spectral Decompositions Using One-Homogeneous Functionals},
  VOLUME = {9},
}

@ARTICLE{BurGilOshXu06,
  AUTHOR = {Burger, M. and Gilboa, G. and Osher, S. and Xu, J.},
  URL = {http://projecteuclid.org/euclid.cms/1145905942},
  DATE = {2006},
  JOURNALTITLE = {Communications in Mathematical Sciences},
  NUMBER = {1},
  PAGES = {179--212},
  SHORTJOURNAL = {Commun. Math. Sci.},
  TITLE = {Nonlinear inverse scale space methods for image restoration},
  VOLUME = {4},
}

@ARTICLE{BurNeu03,
  AUTHOR = {Burger, M. and Neubauer, A.},
  DATE = {2003},
  DOI = {10.1016/s0893-6080(02)00167-3},
  JOURNALTITLE = {Neural Networks},
  NUMBER = {1},
  PAGES = {79--90},
  SHORTJOURNAL = {Neural Netw.},
  TITLE = {Analysis of Tikhonov regularization for function approximation by neural networks},
  VOLUME = {16},
}

@INCOLLECTION{BurOshXuGil05,
  AUTHOR = {Burger, M. and Osher, S. and Xu, J. and Gilboa, G.},
  BOOKTITLE = {Lecture Notes in Computer Science},
  DATE = {2005},
  PAGES = {25--36},
  TITLE = {Nonlinear Inverse Scale Space Methods for Image Restoration},
}

@ARTICLE{BurOsh05,
  AUTHOR = {Burger, M. and Osher, S. J.},
  PUBLISHER = {Cambridge University Press},
  DATE = {2005},
  DOI = {10.1017/s0956792505006182},
  JOURNALTITLE = {European Journal of Applied Mathematics},
  NUMBER = {2},
  PAGES = {263--301},
  SHORTJOURNAL = {European J. Appl. Math.},
  TITLE = {A survey on level set methods for inverse problems and optimal design},
  VOLUME = {16},
}

@BOOK{CakCol14,
  AUTHOR = {Cakoni, F. and Colton, D.},
  LOCATION = {New York},
  PUBLISHER = {Springer US},
  URL = {https://www.springer.com/gp/book/9781461488262},
  DATE = {2014},
  ISBN = {978-1-4614-8826-2},
  SERIES = {188},
  TITLE = {A Qualitative Approach to Inverse Scattering Theory},
}

@BOOK{CakColHad16,
  AUTHOR = {Cakoni, F. and Colton, D. and Haddar, H.},
  PUBLISHER = {Society for Industrial and Applied Mathematics},
  DATE = {2016-11},
  DOI = {10.1137/1.9781611974461},
  TITLE = {Inverse Scattering Theory and Transmission Eigenvalues},
}

@INCOLLECTION{CalCaoCarSchoVal16,
  AUTHOR = {Calatroni, L. and Cao, Ch. and Carlos De los Reyes, J. and Sch\"{o}nlieb, C.-B. and Valkonen, T.},
  BOOKTITLE = {Variational Methods},
  DATE = {2016},
  DOI = {10.1515/9783110430394-008},
  PAGES = {252--290},
  TITLE = {8. Bilevel approaches for learning of variational imaging models},
}

@INCOLLECTION{Cal80,
  AUTHOR = {Calder\'{o}n, A. P.},
  LOCATION = {Rio de Janeiro},
  PUBLISHER = {Soc. Brasil. Mat.},
  BOOKTITLE = {Seminar on numerical analysis and its applications to continuum physics},
  DATE = {1980},
  PAGES = {65--73},
  TITLE = {On an inverse boundary value problem},
}

@BOOK{CalSom23,
  AUTHOR = {Calvetti, D. and Somersalo, E.},
  DATE = {2023},
  DOI = {10.1007/978-3-031-23824-6},
  TITLE = {Bayesian Scientific Computing},
}

@ARTICLE{CalSom18,
  AUTHOR = {Calvetti, D. and Somersalo, E.},
  DATE = {2018},
  DOI = {10.1002/wics.1427},
  JOURNALTITLE = {{WIREs} Computational Statistics},
  NUMBER = {3},
  SHORTJOURNAL = {Wiley Interdiscip. Rev. Comput. Stat.},
  TITLE = {Inverse problems: From regularization to Bayesian inference},
  VOLUME = {10},
}

@INCOLLECTION{CanDon99,
  AUTHOR = {Cand\`{e}s, D. and Donoho, D.},
  EDITOR = {Cohen, A. and Rabut, C. and Schumaker, L.},
  PUBLISHER = {Vanderbilt University Press},
  BOOKTITLE = {Curves and Surface Fitting},
  DATE = {1999},
  PAGES = {105--120},
  TITLE = {Curvelets -- a surprisingly effective nonadaptive representation for objects with edges},
}

@INCOLLECTION{Can06,
  AUTHOR = {Cand\`{e}s, E. J.},
  PUBLISHER = {Eur. Math. Soc., Z\"{u}rich},
  BOOKTITLE = {International {C}ongress of {M}athematicians. {V}ol. {III}},
  DATE = {2006},
  PAGES = {1433--1452},
  TITLE = {Compressive sampling},
}

@ARTICLE{Can08,
  AUTHOR = {Cand\`{e}s, E. J.},
  DATE = {2008},
  DOI = {10.1016/j.crma.2008.03.014},
  ISSN = {1631-073X},
  JOURNALTITLE = {Comptes Rendus Math\'{e}matique. Acad\'{e}mie des Sciences. Paris},
  NUMBER = {9-10},
  PAGES = {589--592},
  SHORTJOURNAL = {C. R. Math. Acad. Sci. Paris},
  TITLE = {The restricted isometry property and its implications for compressed sensing},
  VOLUME = {346},
}

@ARTICLE{CanRomTao06,
  AUTHOR = {Cand\`{e}s, E. J. and Romberg, J. K. and Tao, T.},
  DATE = {2006},
  DOI = {10.1109/tit.2005.862083},
  ISSN = {0018-9448},
  JOURNALTITLE = {{IEEE} Transactions on Information Theory},
  NUMBER = {2},
  PAGES = {489--509},
  SHORTJOURNAL = {{IEEE} Trans. Inf. Theory},
  TITLE = {Robust uncertainty principles: exact signal reconstruction from highly incomplete frequency information},
  VOLUME = {52},
}

@ARTICLE{CanRomTao06b,
  AUTHOR = {Cand\`{e}s, E. J. and Romberg, J. K. and Tao, T.},
  DATE = {2006},
  DOI = {10.1002/cpa.20124},
  ISSN = {0010-3640},
  JOURNALTITLE = {Communications on Pure and Applied Mathematics},
  NUMBER = {8},
  PAGES = {1207--1223},
  SHORTJOURNAL = {Comm. Pure Appl. Math.},
  TITLE = {Stable signal recovery from incomplete and inaccurate measurements},
  VOLUME = {59},
}

@ARTICLE{CarVitToi06,
  AUTHOR = {Carmeli, C. and De vito, E. and Toigo, A.},
  DATE = {2006},
  DOI = {10.1142/s0219530506000838},
  JOURNALTITLE = {Analysis and Applications},
  NUMBER = {4},
  PAGES = {377--408},
  SHORTJOURNAL = {Anal. Appl.},
  TITLE = {Vector valued reproducing kernel Hilbert spaces of integrable functions and Mercer theorem},
  VOLUME = {4},
}

@INPROCEEDINGS{CarDic89,
  AUTHOR = {Carroll, S. M. and Dickinson, B. W.},
  PUBLISHER = {IEEE},
  BOOKTITLE = {International Joint Conference on Neural Networks},
  DATE = {1989},
  DOI = {10.1109/ijcnn.1989.118639},
  TITLE = {Construction of neural nets using the radon transform},
}

@ARTICLE{CarLunRozTho22,
  AUTHOR = {Carson, E. and Lund, K. and Rozlo\v{z}n\'{\i}k, M. and Thomas, S.},
  DATE = {2022},
  DOI = {10.1016/j.laa.2021.12.017},
  ISSN = {0024-3795},
  JOURNALTITLE = {Linear Algebra and its Applications},
  PAGES = {150--195},
  SHORTJOURNAL = {Linear Algebra Appl.},
  TITLE = {Block Gram-Schmidt algorithms and their stability properties},
  VOLUME = {638},
}

@BOOK{CerCerKolMalPlo61,
  AUTHOR = {\v{C}ernikov, S. and \v{C}ernikova, N. and Kolmogorov, A. and Mal\x{2032}cev, A. and Plotkin, B. and Rimski\u{\i}-Korsakov, B. and Tihomirov, V. and Trohim\v{c}uk, Y. and Vilenkin, N. and Vitu\v{s}kin, A.},
  DATE = {1961},
  DOI = {10.1090/trans2/017},
  TITLE = {Twelve Papers on Algebra and Real Functions},
}

@REPORT{ChaMor25_report,
  AUTHOR = {Chambolle, A. and Morel, J-M.},
  URL = {https://hal.science/hal-04893423},
  DATE = {2025-01},
  PAGES = {75--94},
  SERIES = {Documents Math\'{e}matiques},
  TITLE = {{Image = Cartoon+Texture: How Yves Meyer's ''Oscillating patterns in image processing and in some nonlinear evolution equations'' ended up in a computer vision model}},
  TYPE = {Preprint},
}

@ARTICLE{ChaPoc11,
  AUTHOR = {Chambolle, A. and Pock, T.},
  LOCATION = {Netherlands},
  PUBLISHER = {Springer},
  DATE = {2011},
  DOI = {10.1007/s10851-010-0251-1},
  ISSN = {0924-9907},
  JOURNALTITLE = {Journal of Mathematical Imaging and Vision},
  NUMBER = {1},
  PAGES = {120--145},
  SHORTJOURNAL = {J. Math. Imaging Vision},
  TITLE = {A first-order primal-dual algorithm for convex problems with applications to imaging},
  VOLUME = {40},
}

@BOOK{ChaShe05,
  AUTHOR = {Chan, T. and Shen, J.},
  LOCATION = {Philadelphia},
  PUBLISHER = {SIAM},
  DATE = {2005},
  TITLE = {Image Processing and Analysis---Variational, PDE, Wavelet, and Stochastic Methods},
}

@ARTICLE{ChaSheVes03,
  AUTHOR = {Chan, T. and Shen, J. and Vese, L.},
  DATE = {2003},
  JOURNALTITLE = {Notices of the American Mathematical Society},
  NUMBER = {1},
  PAGES = {14--26},
  SHORTJOURNAL = {Notices Amer. Math. Soc.},
  TITLE = {Variational {PDE} models in image processing},
  VOLUME = {50},
}

@ARTICLE{ChaShe05b,
  AUTHOR = {Chan, T. F. and Shen, J.},
  DATE = {2005},
  DOI = {10.1002/cpa.20075},
  ISSN = {0010-3640},
  JOURNALTITLE = {Communications on Pure and Applied Mathematics},
  NUMBER = {5},
  PAGES = {579--619},
  SHORTJOURNAL = {Comm. Pure Appl. Math.},
  TITLE = {Variational image inpainting},
  VOLUME = {58},
}

@ARTICLE{ChaVes01,
  AUTHOR = {Chan, T. F. and Vese, L. A.},
  PUBLISHER = {IEEE},
  DATE = {2001},
  DOI = {10.1109/83.902291},
  ISSN = {1057-7149},
  JOURNALTITLE = {{IEEE} Transactions on Image Processing},
  NUMBER = {2},
  PAGES = {266--277},
  SHORTJOURNAL = {{IEEE} Trans. Image Process.},
  TITLE = {Active contours without edges},
  VOLUME = {10},
}

@ARTICLE{ChaKun94a,
  AUTHOR = {Chavent, G. and Kunisch, K.},
  DATE = {1994},
  ISSN = {0266-5611},
  JOURNALTITLE = {Inverse Problems},
  PAGES = {63--76},
  SHORTJOURNAL = {Inverse Probl.},
  TITLE = {Convergence of {T}ikhonov regularization for constrained ill-posed inverse problems},
  VOLUME = {10},
}

@ARTICLE{ChaKun93,
  AUTHOR = {Chavent, G. and Kunisch, K.},
  DATE = {1993},
  ISSN = {0764-583X},
  JOURNALTITLE = {Mathematical Modelling and Numerical Analysis},
  PAGES = {535--564},
  SHORTJOURNAL = {Math. Model. Numer. Anal.},
  TITLE = {Regularization in state space},
  VOLUME = {27},
}

@ARTICLE{ChePerXu15,
  AUTHOR = {Chen, J. and Pereverzyev Jr, S. and Xu, Y.},
  DATE = {2015},
  DOI = {10.1088/0266-5611/31/7/075005},
  ISSN = {0266-5611},
  JOURNALTITLE = {Inverse Problems},
  NUMBER = {7},
  PAGES = {075005},
  SHORTJOURNAL = {Inverse Probl.},
  TITLE = {Aggregation of regularized solutions from multiple observation models},
  VOLUME = {31},
}

@ARTICLE{TiaHon95,
  AUTHOR = {Chen, Tianping and Chen, Hong},
  DATE = {1995},
  DOI = {10.1109/72.392253},
  JOURNALTITLE = {{IEEE} Transactions on Neural Networks},
  NUMBER = {4},
  PAGES = {911--917},
  SHORTJOURNAL = {{IEEE} Trans. Neural Netw.},
  TITLE = {Universal approximation to nonlinear operators by neural networks with arbitrary activation functions and its application to dynamical systems},
  VOLUME = {6},
}

@ARTICLE{ChiVitMolRosVil24,
  AUTHOR = {Chirinos-Rodr\'{\i}guez, J. and De Vito, E. and Molinari, C. and Rosasco, L. and Villa, S.},
  DATE = {2024},
  DOI = {10.1088/1361-6420/ad8a84},
  ISSN = {0266-5611},
  JOURNALTITLE = {Inverse Problems},
  NUMBER = {12},
  PAGES = {125004},
  SHORTJOURNAL = {Inverse Probl.},
  TITLE = {On learning the optimal regularization parameter in inverse problems},
  VOLUME = {40},
}

@INCOLLECTION{ChoBar03,
  AUTHOR = {Choi, H. and Baraniuk, R. G.},
  BOOKTITLE = {Lecture Notes in Statistics},
  DATE = {2003},
  DOI = {10.1007/978-0-387-21579-2_2},
  PAGES = {9--29},
  TITLE = {Wavelet Statistical Models and Besov Spaces},
}

@ARTICLE{Cho56,
  AUTHOR = {Chomsky, N.},
  DATE = {1956},
  DOI = {10.1109/tit.1956.1056813},
  ISSN = {0018-9448},
  JOURNALTITLE = {{IEEE} Transactions on Information Theory},
  NUMBER = {3},
  PAGES = {113--124},
  SHORTJOURNAL = {{IEEE} Trans. Inf. Theory},
  TITLE = {Three models for the description of language},
  VOLUME = {2},
}

@BOOK{Chr16,
  AUTHOR = {Christensen, O.},
  PUBLISHER = {Springer International Publishing},
  DATE = {2016},
  DOI = {10.1007/978-3-319-25613-9},
  TITLE = {An Introduction to Frames and Riesz Bases},
}

@INCOLLECTION{ChuVes05,
  AUTHOR = {Chung, G. and Vese, L. A.},
  BOOKTITLE = {Lecture Notes in Computer Science},
  DATE = {2005},
  DOI = {10.1007/11585978_29},
  PAGES = {439--455},
  TITLE = {Energy Minimization Based Segmentation and Denoising Using a Multilayer Level Set Approach},
}

@ARTICLE{ChuVes09,
  AUTHOR = {Chung, G. and Vese, L. A.},
  DATE = {2009},
  DOI = {10.1007/s00791-008-0113-1},
  JOURNALTITLE = {Computing and Visualization in Science},
  NUMBER = {6},
  PAGES = {267--285},
  SHORTJOURNAL = {Comput. Vis. Sci.},
  TITLE = {Image segmentation using a multilayer level-set approach},
  VOLUME = {12},
}

@ARTICLE{ChuEsp17,
  AUTHOR = {Chung, J. and Espa\~{n}ol, M. I.},
  DATE = {2017},
  DOI = {10.1088/1361-6420/33/7/074004},
  ISSN = {0266-5611},
  JOURNALTITLE = {Inverse Problems},
  NUMBER = {7},
  PAGES = {074004},
  SHORTJOURNAL = {Inverse Probl.},
  TITLE = {Learning regularization parameters for general-form Tikhonov},
  VOLUME = {33},
}

@ARTICLE{ChyPauSchSchoZim01,
  AUTHOR = {Chyzak, F. and Paule, P. and Scherzer, O. and Schoisswohl, A. and Zimmermann, B.},
  URL = {http://projecteuclid.org/getRecord?id=euclid.em/999188421},
  DATE = {2001},
  ISSN = {1058-6458},
  JOURNALTITLE = {Experimental Mathematics},
  NUMBER = {1},
  PAGES = {67--86},
  SHORTJOURNAL = {Experiment. Math.},
  TITLE = {The construction of orthonormal wavelets using symbolic methods and a matrix analytical approach for wavelets on the interval},
  VOLUME = {10},
}

@BOOK{Cia78,
  AUTHOR = {Ciarlet, P. G.},
  LOCATION = {Amsterdam},
  PUBLISHER = {North-Holland},
  DATE = {1978},
  TITLE = {The Finite Element Method for Elliptic Problems},
}

@BOOK{Cia02,
  AUTHOR = {Ciarlet, P. G.},
  DATE = {2002},
  DOI = {10.1137/1.9780898719208},
  TITLE = {The Finite Element Method for Elliptic Problems},
}

@BOOK{Cla90,
  AUTHOR = {Clarke, F. H.},
  LOCATION = {Philadelphia, PA},
  PUBLISHER = {SIAM},
  DATE = {1990},
  EDITION = {2},
  SERIES = {Classics in Applied Mathematics},
  TITLE = {Optimization and Nonsmooth Analysis},
  VOLUME = {5},
}

@BOOK{Coh03,
  AUTHOR = {Cohen, A.},
  LOCATION = {Amsterdam},
  PUBLISHER = {North-Holland Publishing Co.},
  DATE = {2003},
  PAGETOTAL = {xviii+336},
  SERIES = {Studies in Mathematics and its Applications},
  TITLE = {Numerical Analysis of Wavelet Methods},
  VOLUME = {32},
}

@ARTICLE{CoiWei77,
  AUTHOR = {Coifman, R. R. and Weiss, G.},
  DATE = {1977},
  DOI = {10.1090/s0002-9904-1977-14325-5},
  ISSN = {0273-0979},
  JOURNALTITLE = {Bulletin of the American Mathematical Society},
  NUMBER = {4},
  PAGES = {569--645},
  SHORTJOURNAL = {Bull. Amer. Math. Soc.},
  TITLE = {Extensions of Hardy spaces and their use in analysis},
  VOLUME = {83},
}

@ARTICLE{ColKun86,
  AUTHOR = {Colonius, F. and Kunisch, K.},
  DATE = {1986},
  ISSN = {0075-4102},
  JOURNALTITLE = {Journal f\"{u}r die Reine und Angewandte Mathematik. [Crelle's Journal]},
  PAGES = {1--29},
  SHORTJOURNAL = {J. Reine Angew. Math.},
  TITLE = {Stability for parameter estimation in two point boundary value problems},
  VOLUME = {370},
}

@BOOK{ColKre19,
  AUTHOR = {Colton, D. and Kress, R.},
  PUBLISHER = {Springer},
  DATE = {2019},
  EDITION = {4},
  ISBN = {978-3-030-30350-1},
  NUMBER = {93},
  SERIES = {Applied Mathematical Sciences},
  TITLE = {Inverse Acoustic and Electromagnetic Scattering Theory},
}

@ARTICLE{ComPes07,
  AUTHOR = {Combettes, P. L. and Pesquet, J.-Ch.},
  DATE = {2007},
  DOI = {10.1137/060669498},
  ISSN = {1052-6234},
  JOURNALTITLE = {{SIAM} Journal on Optimization},
  NUMBER = {4},
  PAGES = {1351--1376},
  SHORTJOURNAL = {{SIAM} J. Optim.},
  TITLE = {Proximal thresholding algorithm for minimization over orthonormal bases},
  VOLUME = {18},
}

@ARTICLE{ConPhi12,
  AUTHOR = {Constantine, P. G. and Phipps, E. T.},
  DATE = {2012},
  DOI = {10.1016/j.amc.2012.05.009},
  JOURNALTITLE = {Applied Mathematics {\&} Computation},
  NUMBER = {24},
  PAGES = {11751--11762},
  SHORTJOURNAL = {Appl. Math. Comput.},
  TITLE = {A Lanczos method for approximating composite functions},
  VOLUME = {218},
}

@BOOK{Con07,
  AUTHOR = {Conway, J. B.},
  DATE = {2007},
  DOI = {10.1007/978-1-4757-4383-8},
  TITLE = {A Course in Functional Analysis},
}

@INCOLLECTION{Cor94,
  AUTHOR = {Cormack, A. M.},
  PUBLISHER = {Int. Press, Cambridge, MA},
  BOOKTITLE = {75 years of {R}adon transform ({V}ienna, 1992)},
  DATE = {1994},
  PAGES = {32--35},
  SERIES = {Conf. Proc. Lecture Notes Math. Phys., IV},
  TITLE = {My connection with the {R}adon transform},
}

@ARTICLE{Cor63,
  AUTHOR = {Cormack, A. M.},
  DATE = {1963},
  JOURNALTITLE = {Journal of Applied Physics},
  NUMBER = {9},
  PAGES = {2722--2727},
  SHORTJOURNAL = {J. App. Phys.},
  TITLE = {Representation of a Function by Its Line Integrals, with Some Radiological Applications},
  VOLUME = {34},
}

@ARTICLE{CraMas13,
  AUTHOR = {Crambes, Ch. and Mas, A.},
  DATE = {2013},
  DOI = {10.3150/12-bej469},
  ISSN = {1350-7265},
  JOURNALTITLE = {Bernoulli},
  NUMBER = {5B},
  SHORTJOURNAL = {Bernoulli},
  TITLE = {Asymptotics of prediction in functional linear regression with functional outputs},
  VOLUME = {19},
}

@ARTICLE{Cyb89,
  AUTHOR = {Cybenko, G.},
  DATE = {1989},
  DOI = {10.1007/bf02551274},
  JOURNALTITLE = {Mathematics of Control, Signals, and Systems},
  NUMBER = {4},
  PAGES = {303--314},
  SHORTJOURNAL = {Math. Control Signals Systems},
  TITLE = {Approximation by superpositions of a sigmoidal function},
  VOLUME = {2},
}

@ARTICLE{Dau88,
  AUTHOR = {Daubechies, I.},
  DATE = {1988},
  ISSN = {0010-3640},
  JOURNALTITLE = {Communications on Pure and Applied Mathematics},
  NUMBER = {7},
  PAGES = {909--996},
  SHORTJOURNAL = {Comm. Pure Appl. Math.},
  TITLE = {Orthonormal bases of compactly supported wavelets},
  VOLUME = {41},
}

@BOOK{Dau92,
  AUTHOR = {Daubechies, I.},
  LOCATION = {Philadelphia, PA},
  PUBLISHER = {SIAM},
  DATE = {1992},
  DOI = {10.1137/1.9781611970104},
  ISBN = {0-89871-274-2},
  PAGETOTAL = {xx+357},
  TITLE = {Ten Lectures on Wavelets},
}

@ARTICLE{DauDefMol04,
  AUTHOR = {Daubechies, I. and Defrise, M. and De Mol, C.},
  DATE = {2004},
  DOI = {10.1002/cpa.20042},
  ISSN = {0010-3640},
  JOURNALTITLE = {Communications on Pure and Applied Mathematics},
  NUMBER = {11},
  PAGES = {1413--1457},
  SHORTJOURNAL = {Comm. Pure Appl. Math.},
  TITLE = {An iterative thresholding algorithm for linear inverse problems with a sparsity constraint},
  VOLUME = {57},
}

@ARTICLE{DauDevDymFaiKov23,
  AUTHOR = {Daubechies, I. and DeVore, R. and Dym, N. and Faigenbaum-Golovin, S. and Kovalsky, S. Z. and Lin, K.-Ch. and Park, J. and Petrova, G. and Sober, B.},
  DATE = {2023},
  DOI = {10.1109/tit.2022.3199601},
  ISSN = {0018-9448},
  JOURNALTITLE = {{IEEE} Transactions on Information Theory},
  NUMBER = {1},
  PAGES = {482--495},
  SHORTJOURNAL = {{IEEE} Trans. Inf. Theory},
  TITLE = {Neural Network Approximation of Refinable Functions},
  VOLUME = {69},
}

@BOOK{DauLio88,
  AUTHOR = {Dautray, R. and Lions, J.-L.},
  LOCATION = {Berlin},
  PUBLISHER = {Springer-Verlag},
  DATE = {1988},
  ISBN = {3-540-19045-7},
  PAGETOTAL = {xvi+561},
  TITLE = {Mathematical analysis and numerical methods for science and technology. {V}ol. 2},
}

@ARTICLE{Dav81,
  AUTHOR = {Davison, M. E.},
  DATE = {1981},
  DOI = {10.1080/01630568108816093},
  ISSN = {0163-0563},
  JOURNALTITLE = {Numerical Functional Analysis and Optimization},
  NUMBER = {3},
  PAGES = {321--340},
  SHORTJOURNAL = {Numer. Funct. Anal. Optim.},
  TITLE = {A singular value decomposition for the radon transform in n-dimensional euclidean space},
  VOLUME = {3},
}

@BOOK{DeB78,
  AUTHOR = {De Boor, C.},
  LOCATION = {New York},
  PUBLISHER = {Springer Verlag},
  DATE = {1978},
  TITLE = {A Practical Guide to Splines},
}

@ARTICLE{LosSchoVal16,
  AUTHOR = {De Los Reyes, J. C. and Sch\"{o}nlieb, C.-B. and Valkonen, T.},
  DATE = {2016},
  DOI = {10.1016/j.jmaa.2015.09.023},
  ISSN = {0022-247X},
  JOURNALTITLE = {Journal of Mathematical Analysis and Applications},
  NUMBER = {1},
  PAGES = {464--500},
  SHORTJOURNAL = {J. Math. Anal. Appl.},
  TITLE = {The structure of optimal parameters for image restoration problems},
  VOLUME = {434},
}

@ARTICLE{LosSchoVal17,
  AUTHOR = {De los Reyes, J. C. and Sch\"{o}nlieb, C.-B. and Valkonen, T.},
  LOCATION = {Netherlands},
  PUBLISHER = {Springer},
  DATE = {2017},
  DOI = {10.1007/s10851-016-0662-8},
  ISSN = {0924-9907},
  JOURNALTITLE = {Journal of Mathematical Imaging and Vision},
  NUMBER = {1},
  PAGES = {1--25},
  SHORTJOURNAL = {J. Math. Imaging Vision},
  TITLE = {Bilevel Parameter Learning for Higher-Order Total Variation Regularisation Models},
  VOLUME = {57},
}

@ARTICLE{VitForNau22,
  AUTHOR = {De Vito, E. and Fornasier, M. and Naumova, V.},
  DATE = {2022},
  DOI = {10.1142/s0219530520500220},
  JOURNALTITLE = {Analysis and Applications},
  NUMBER = {2},
  PAGES = {353--400},
  SHORTJOURNAL = {Anal. Appl.},
  TITLE = {A machine learning approach to optimal Tikhonov regularization I: Affine manifolds},
  VOLUME = {20},
}

@ARTICLE{DecKelKel83,
  AUTHOR = {Decker, D. W. and Keller, H. B. and Kelley, C. T.},
  PUBLISHER = {SIAM},
  DATE = {1983},
  DOI = {10.1137/0720020},
  ISSN = {0036-1429},
  JOURNALTITLE = {{SIAM} Journal on Numerical Analysis},
  NUMBER = {2},
  PAGES = {296--314},
  SHORTJOURNAL = {{SIAM} J. Numer. Anal.},
  TITLE = {Convergence Rates for Newton's Method at Singular Points},
  VOLUME = {20},
}

@INCOLLECTION{DefMol87,
  AUTHOR = {Defrise, M. and De Mol, C.},
  PUBLISHER = {Academic Press, London},
  BOOKTITLE = {Inverse problems: an interdisciplinary study ({M}ontpellier, 1986)},
  DATE = {1987},
  ISBN = {0-12-014581-2},
  PAGES = {261--268},
  SERIES = {Adv. Electron. Electron Phys.},
  TITLE = {A note on stopping rules for iterative regularization methods and filtered {SVD}},
  VOLUME = {Suppl. 19},
}

@BOOK{Dei85b,
  AUTHOR = {Deimling, K.},
  DATE = {1985},
  DOI = {10.1007/978-3-662-00547-7},
  TITLE = {Nonlinear Functional Analysis},
}

@BOOK{DenHan09,
  AUTHOR = {Deng, D. and Han, Y.},
  PUBLISHER = {Springer Berlin Heidelberg},
  DATE = {2009},
  DOI = {10.1007/978-3-540-88745-4},
  TITLE = {Harmonic Analysis on Spaces of Homogeneous Type},
}

@ARTICLE{DeuEngSch98,
  AUTHOR = {Deuflhard, P. and Engl, H. W. and Scherzer, O.},
  DATE = {1998},
  DOI = {10.1088/0266-5611/14/5/002},
  ISSN = {0266-5611},
  JOURNALTITLE = {Inverse Problems},
  NUMBER = {5},
  PAGES = {1081--1106},
  SHORTJOURNAL = {Inverse Probl.},
  TITLE = {A convergence analysis of iterative methods for the solution of nonlinear ill-posed problems under affinely invariant conditions},
  VOLUME = {14},
}

@ARTICLE{DeuHei79,
  AUTHOR = {Deuflhard, P. and Heindl, G.},
  PUBLISHER = {SIAM},
  DATE = {1979},
  ISSN = {0036-1429},
  JOURNALTITLE = {{SIAM} Journal on Numerical Analysis},
  NUMBER = {1},
  PAGES = {1--10},
  SHORTJOURNAL = {{SIAM} J. Numer. Anal.},
  TITLE = {Affine invariant convergence theorems for {N}ewton's method and extensions to related methods},
  VOLUME = {16},
}

@BOOK{DeuHoh91,
  AUTHOR = {Deuflhard, P. and Hohmann, A.},
  LOCATION = {Berlin},
  PUBLISHER = {De Gruyter},
  DATE = {1991},
  DOI = {10.1515/9783110891997},
  TITLE = {Numerical Analysis. A First Course in Scientific Computation},
}

@BOOK{DeuHoh93,
  AUTHOR = {Deuflhard, P. and Hohmann, A.},
  LANGUAGE = {german},
  LOCATION = {Berlin},
  PUBLISHER = {De {G}ruyter},
  DATE = {1993},
  TITLE = {Numerische Mathematik I. Eine algorithmisch orientierte Einf\"{u}hrung},
}

@ARTICLE{DeuPot92,
  AUTHOR = {Deuflhard, P. and Potra, F. A.},
  PUBLISHER = {SIAM},
  DATE = {1992},
  DOI = {10.1137/0729080},
  ISSN = {0036-1429},
  JOURNALTITLE = {{SIAM} Journal on Numerical Analysis},
  NUMBER = {5},
  PAGES = {1395--1412},
  SHORTJOURNAL = {{SIAM} J. Numer. Anal.},
  TITLE = {Asymptotic Mesh Independence of Newton--Galerkin Methods via a Refined Mysovskii Theorem},
  VOLUME = {29},
}

@ARTICLE{DicMaa96,
  AUTHOR = {Dicken, V. and Maass, P.},
  DATE = {1996},
  ISSN = {0928-0219},
  JOURNALTITLE = {Journal of Inverse and Ill-Posed Problems},
  PAGES = {203--221},
  SHORTJOURNAL = {J. Inverse Ill-Posed Probl.},
  TITLE = {Wavelet-{G}alerkin methods for ill-posed problems},
  VOLUME = {4},
}

@ARTICLE{DitKluMaaOte20,
  AUTHOR = {Dittmer, S. and Kluth, T. and Maass, P. and Otero Baguer, D.},
  LOCATION = {Netherlands},
  PUBLISHER = {Springer},
  DATE = {2020},
  DOI = {10.1007/s10851-019-00923-x},
  ISSN = {0924-9907},
  JOURNALTITLE = {Journal of Mathematical Imaging and Vision},
  NUMBER = {3},
  PAGES = {456--470},
  SHORTJOURNAL = {J. Math. Imaging Vision},
  TITLE = {Regularization by Architecture: A Deep Prior Approach for Inverse Problems},
  VOLUME = {62},
}

@ARTICLE{Don06,
  AUTHOR = {Donoho, D. L.},
  DATE = {2006},
  DOI = {10.1109/tit.2006.871582},
  ISSN = {0018-9448},
  JOURNALTITLE = {{IEEE} Transactions on Information Theory},
  NUMBER = {4},
  PAGES = {1289--1306},
  SHORTJOURNAL = {{IEEE} Trans. Inf. Theory},
  TITLE = {Compressed sensing},
  VOLUME = {52},
}

@ARTICLE{Doy12,
  AUTHOR = {Doyley, M. M.},
  PUBLISHER = {IOP Publishing Ltd},
  DATE = {2012},
  DOI = {10.1088/0031-9155/57/3/r35},
  ISSN = {0031-9155},
  JOURNALTITLE = {Physics in Medicine {\&} Biology},
  NUMBER = {3},
  PAGES = {R35--R73},
  SHORTJOURNAL = {Phys. Med. Biol.},
  TITLE = {Model-based elastography: a survey of approaches to the inverse elasticity problem},
  VOLUME = {57},
}

@BOOK{DunSchw63,
  AUTHOR = {Dunford, N. and Schwartz, J.T.},
  LOCATION = {New York},
  PUBLISHER = {Wiley},
  DATE = {1963},
  TITLE = {Linear Operators {I},{II}},
}

@ARTICLE{DurLitBabChaAze05,
  AUTHOR = {Duric, N. and Littrup, P. and Babkin, A. and Chambers, D. and Azevedo, S. and Pevzner, R. and Tokarev, M. and Holsapple, E. and Rama, O and Duncan, R.},
  PUBLISHER = {American Association of Physicists in Medicine},
  DATE = {2005},
  JOURNALTITLE = {Medical Physics},
  PAGES = {1375--1386},
  SHORTJOURNAL = {Med. Phys.},
  TITLE = {Development of ultrasound tomography for breast imaging: technical assessment},
  VOLUME = {32},
}

@ARTICLE{EMaWu22,
  AUTHOR = {E, W. and Ma, Ch. and Wu, L.},
  DATE = {2022},
  DOI = {10.1007/s00365-021-09549-y},
  ISSN = {0176-4276},
  JOURNALTITLE = {Constructive Approximation},
  NUMBER = {1},
  PAGES = {369--406},
  SHORTJOURNAL = {Constr. Approx.},
  TITLE = {The Barron Space and the Flow-Induced Function Spaces for Neural Network Models},
  VOLUME = {55},
}

@INPROCEEDINGS{EfrLeu99b,
  AUTHOR = {Efros, A. A. and Leung, T. K.},
  BOOKTITLE = {Proceedings of the Seventh IEEE International Conference on Computer Vision},
  DATE = {1999},
  DOI = {10.1109/iccv.1999.790383},
  PAGES = {1033--1038 vol.2},
  TITLE = {Texture synthesis by non-parametric sampling},
}

@ARTICLE{Egg93,
  AUTHOR = {Eggermont, P. P. B.},
  DATE = {1993},
  ISSN = {0036-1410},
  JOURNALTITLE = {{SIAM} Journal on Mathematical Analysis},
  NUMBER = {6},
  PAGES = {1557--1576},
  SHORTJOURNAL = {{SIAM} J. Math. Anal.},
  TITLE = {Maximum entropy regularization for {F}redholm integral equations of the first kind},
  VOLUME = {24},
}

@ARTICLE{EggLarNas12,
  AUTHOR = {Eggermont, P. P. B. and LaRiccia, V. N. and Nashed, M. Z.},
  DATE = {2012},
  DOI = {10.1007/s13137-012-0037-2},
  JOURNALTITLE = {GEM - International Journal on Geomathematics},
  NUMBER = {2},
  PAGES = {155--178},
  SHORTJOURNAL = {GEM. Int. J. Geomath.},
  TITLE = {Moment discretization for ill-posed problems with discrete weakly bounded noise},
  VOLUME = {3},
}

@ARTICLE{EggLarNas09,
  AUTHOR = {Eggermont, P. P. B. and LaRiccia, V. N. and Nashed, M. Z.},
  DATE = {2009},
  DOI = {10.1088/0266-5611/25/11/115018},
  ISSN = {0266-5611},
  JOURNALTITLE = {Inverse Problems},
  NUMBER = {11},
  PAGES = {115018},
  SHORTJOURNAL = {Inverse Probl.},
  TITLE = {On weakly bounded noise in ill-posed problems},
  VOLUME = {25},
}

@ARTICLE{ElbSchShi17,
  AUTHOR = {Elbau, P. and Scherzer, O. and Shi, C.},
  DATE = {2017-11},
  DOI = {10.1016/j.jde.2017.06.018},
  JOURNALTITLE = {Journal of Differential Equations},
  NUMBER = {9},
  PAGES = {5330--5376},
  SHORTJOURNAL = {J. Differential Equations},
  TITLE = {Singular Values of the Attenuated Photoacoustic Imaging Operator},
  VOLUME = {263},
}

@ARTICLE{ElfUchDoy18,
  AUTHOR = {Elfwing, S. and Uchibe, E. and Doya, K.},
  DATE = {2018},
  DOI = {10.1016/j.neunet.2017.12.012},
  JOURNALTITLE = {Neural Networks},
  PAGES = {3--11},
  SHORTJOURNAL = {Neural Netw.},
  TITLE = {Sigmoid-weighted linear units for neural network function approximation in reinforcement learning},
  VOLUME = {107},
}

@BOOK{Els18,
  AUTHOR = {Elstrodt, J.},
  DATE = {2018},
  DOI = {10.1007/978-3-662-57939-8},
  TITLE = {Ma\ss{}- und Integrationstheorie},
}

@ARTICLE{EngGfr88,
  AUTHOR = {Engl, H. W. and Gfrerer, H.},
  DATE = {1988},
  JOURNALTITLE = {Applied Numerical Mathematics},
  NUMBER = {5},
  PAGES = {395--417},
  SHORTJOURNAL = {Appl. Numer. Math.},
  TITLE = {A posteriori parameter choice for general regularization methods for solving linear ill-posed problems},
  VOLUME = {4},
}

@ARTICLE{EngGre94,
  AUTHOR = {Engl, H. W. and Grever, W.},
  DATE = {1994},
  ISSN = {0029-599X},
  JOURNALTITLE = {Numerische Mathematik},
  NUMBER = {1},
  PAGES = {25--31},
  SHORTJOURNAL = {Numer. Math.},
  TITLE = {Using the {$L$}-curve for determining optimal regularization parameters},
  VOLUME = {69},
}

@BOOK{EngHanNeu96,
  AUTHOR = {Engl, H. W. and Hanke, M. and Neubauer, A.},
  LOCATION = {Dordrecht},
  PUBLISHER = {Kluwer Academic Publishers Group},
  URL = {https://link.springer.com/book/9780792341574},
  DATE = {1996},
  ISBN = {0-7923-4157-0},
  NUMBER = {375},
  PAGETOTAL = {viii+321},
  SERIES = {Mathematics and its Applications},
  TITLE = {Regularization of inverse problems},
}

@ARTICLE{EngHofKin05,
  AUTHOR = {Engl, H. W. and Hofinger, A. and Kindermann, S.},
  DATE = {2005},
  DOI = {10.1088/0266-5611/21/1/024},
  ISSN = {0266-5611},
  JOURNALTITLE = {Inverse Problems},
  NUMBER = {1},
  PAGES = {399--412},
  SHORTJOURNAL = {Inverse Probl.},
  TITLE = {Convergence rates in the {P}rokhorov metric for assessing uncertainty in ill-posed problems},
  VOLUME = {21},
}

@ARTICLE{EngHofZei93,
  AUTHOR = {Engl, H. W. and Hofmann, B. and Zeisel, H.},
  DATE = {1993},
  DOI = {10.1216/jiea/1181075772},
  ISSN = {0897-3962},
  JOURNALTITLE = {Journal of Integral Equations and Applications},
  NUMBER = {4},
  SHORTJOURNAL = {J. Integral Equations Appl.},
  TITLE = {A Decreasing Rearrangement Approach for a Class of Ill-Posed Nonlinear Integral Equations},
  VOLUME = {5},
}

@ARTICLE{EngKunNeu89,
  AUTHOR = {Engl, H. W. and Kunisch, K. and Neubauer, A.},
  URL = {http://stacks.iop.org/0266-5611/5/523},
  DATE = {1989},
  ISSN = {0266-5611},
  JOURNALTITLE = {Inverse Problems},
  NUMBER = {3},
  PAGES = {523--540},
  SHORTJOURNAL = {Inverse Probl.},
  TITLE = {Convergence rates for {T}ikhonov regularisation of nonlinear ill-posed problems},
  VOLUME = {5},
}

@ARTICLE{EngLan93,
  AUTHOR = {Engl, H. W. and Landl, G.},
  PUBLISHER = {SIAM},
  DATE = {1993},
  ISSN = {0036-1429},
  JOURNALTITLE = {{SIAM} Journal on Numerical Analysis},
  NUMBER = {5},
  PAGES = {1509--1536},
  SHORTJOURNAL = {{SIAM} J. Numer. Anal.},
  TITLE = {Convergence rates for maximum entropy regularization},
  VOLUME = {30},
}

@ARTICLE{EngZou00,
  AUTHOR = {Engl, H. W. and Zou, J.},
  DATE = {2000},
  DOI = {10.1088/0266-5611/16/6/319},
  ISSN = {0266-5611},
  JOURNALTITLE = {Inverse Problems},
  NUMBER = {6},
  PAGES = {1907--1923},
  SHORTJOURNAL = {Inverse Probl.},
  TITLE = {A new approach to convergence rate analysis of Tikhonov regularization for parameter identification in heat conduction},
  VOLUME = {16},
}

@BOOK{Eps07,
  EDITOR = {Epstein, Ch. L.},
  DATE = {2007},
  DOI = {10.1137/9780898717792},
  TITLE = {Introduction to the Mathematics of Medical Imaging},
}

@ARTICLE{EseShe02,
  AUTHOR = {Esedoglu, S. and Shen, J.},
  URL = {citeseer.ist.psu.edu/esedoglu02digital.html},
  DATE = {2002},
  JOURNALTITLE = {European Journal of Applied Mathematics},
  PAGES = {353--370},
  SHORTJOURNAL = {European J. Appl. Math.},
  TITLE = {Digital inpainting based on the {M}umford--{S}hah--{E}uler image model},
  VOLUME = {13},
}

@BOOK{EvaGar15,
  AUTHOR = {Evans, L. C. and Gariepy, R. F.},
  LOCATION = {Boca Raton, FL},
  PUBLISHER = {CRC Press},
  DATE = {2015},
  EDITION = {Revised},
  ISBN = {978-1-4822-4238-6},
  PAGETOTAL = {xiv+299},
  SERIES = {Textbooks in Mathematics},
  TITLE = {Measure theory and fine properties of functions},
}

@ARTICLE{FanXioWan20,
  AUTHOR = {Fan, F. and Xiong, J. and Wang, G.},
  DATE = {2020},
  DOI = {10.1016/j.neunet.2020.01.007},
  JOURNALTITLE = {Neural Networks},
  PAGES = {383--392},
  SHORTJOURNAL = {Neural Netw.},
  TITLE = {Universal approximation with quadratic deep networks},
  VOLUME = {124},
}

@ARTICLE{FauSch23,
  AUTHOR = {Faucher, F. and Scherzer, O.},
  DATE = {2023-01},
  DOI = {10.1016/j.jcp.2022.111685},
  ISSN = {0021-9991},
  JOURNALTITLE = {Journal of Computational Physics},
  PAGES = {111685},
  SHORTJOURNAL = {J. Comput. Phys.},
  TITLE = {Quantitative inverse problem in visco-acoustic media under attenuation model uncertainty},
  VOLUME = {472},
}

@BOOK{Fee10,
  AUTHOR = {Feeman, T. G.},
  LOCATION = {New York},
  PUBLISHER = {Springer},
  DATE = {2010},
  DOI = {10.1007/978-0-387-92712-1},
  SERIES = {Springer Undergraduate Texts in Mathematics and Technology},
  SUBTITLE = {A Beginner's Guide},
  TITLE = {The Mathematics of Medical Imaging},
}

@ARTICLE{FeiFucJueSchYan08,
  AUTHOR = {Feichtinger, R. and Fuchs, M. and J\"{u}ttler, B. and Scherzer, O. and Yang, H.},
  DATE = {2008},
  DOI = {10.1016/j.cad.2007.08.003},
  ISSN = {0010-4485},
  JOURNALTITLE = {Computer Aided Design},
  NUMBER = {1},
  PAGES = {13--24},
  SHORTJOURNAL = {Comput. Aided Design},
  TITLE = {Dual evolution of planar parametric spline curves and {T}-spline level sets},
  VOLUME = {40},
}

@ARTICLE{FesSut03b,
  AUTHOR = {Fessler, J. A. and Sutton, B. P.},
  DATE = {2003},
  DOI = {10.1109/tsp.2002.807005},
  ISSN = {1053-587X},
  JOURNALTITLE = {{IEEE} Transactions on Signal Processing},
  NUMBER = {2},
  PAGES = {560--574},
  SHORTJOURNAL = {{IEEE} Trans. Signal Process.},
  TITLE = {Nonuniform fast fourier transforms using min-max interpolation},
  VOLUME = {51},
}

@ARTICLE{Fil62,
  AUTHOR = {Filippov, A. F.},
  DATE = {1962},
  DOI = {10.1137/0301006},
  JOURNALTITLE = {Journal of the Society for Industrial and Applied Mathematics Series A Control},
  NUMBER = {1},
  PAGES = {76--84},
  SHORTJOURNAL = {J. Soc. Ind. Appl. Math. Ser. A Control},
  TITLE = {On Certain Questions in the Theory of Optimal Control},
  VOLUME = {1},
}

@THESIS{Fle11a,
  AUTHOR = {Flemming, J.},
  INSTITUTION = {Chemnitz University of Technology},
  DATE = {2011},
  TITLE = {Generalized Tikhonov regularization. Basic theory and comprehensive results on convergence rates},
  TYPE = {{PhD} thesis},
}

@ARTICLE{Fle11,
  AUTHOR = {Flemming, J.},
  DATE = {2011},
  DOI = {10.1080/00036811.2011.563736},
  ISSN = {0003-6811},
  JOURNALTITLE = {Applicable Analysis},
  NUMBER = {5},
  PAGES = {1029--1044},
  SHORTJOURNAL = {Appl. Anal.},
  TITLE = {Solution smoothness of ill-posed equations in {H}ilbert spaces: four concepts and their cross connections},
  VOLUME = {91},
}

@BOOK{Fle18,
  AUTHOR = {Flemming, J.},
  DATE = {2018},
  DOI = {10.1007/978-3-319-95264-2},
  TITLE = {Variational Source Conditions, Quadratic Inverse Problems, Sparsity Promoting Regularization},
}

@ARTICLE{FleHof10,
  AUTHOR = {Flemming, J. and Hofmann, B.},
  DATE = {2010},
  ISSN = {0163-0563},
  JOURNALTITLE = {Numerical Functional Analysis and Optimization},
  NUMBER = {3},
  PAGES = {245--284},
  SHORTJOURNAL = {Numer. Funct. Anal. Optim.},
  TITLE = {A New Approach to Source Conditions in Regularization with General Residual Term},
  VOLUME = {31},
}

@ARTICLE{FleHof11,
  AUTHOR = {Flemming, J. and Hofmann, B.},
  DATE = {2011},
  DOI = {10.1088/0266-5611/27/8/085001},
  ISSN = {0266-5611},
  JOURNALTITLE = {Inverse Problems},
  NUMBER = {8},
  PAGES = {085001,11},
  SHORTJOURNAL = {Inverse Probl.},
  TITLE = {Convergence rates in constrained {T}ikhonov regularization: equivalence of projected source conditions and variational inequalities},
  VOLUME = {27},
}

@ARTICLE{FleHofMat11,
  AUTHOR = {Flemming, J. and Hofmann, B. and Math\'{e}, P.},
  DATE = {2011},
  DOI = {10.1088/0266-5611/27/2/025006},
  ISSN = {0266-5611},
  JOURNALTITLE = {Inverse Problems},
  NUMBER = {2},
  PAGES = {025006},
  SHORTJOURNAL = {Inverse Probl.},
  TITLE = {Sharp converse results for the regularization error using distance functions},
  VOLUME = {27},
}

@ARTICLE{FreSchn98,
  AUTHOR = {Freeden, W. and Schneider, F.},
  DATE = {1998},
  DOI = {10.1088/0266-5611/14/2/002},
  ISSN = {0266-5611},
  JOURNALTITLE = {Inverse Problems},
  NUMBER = {2},
  PAGES = {225--243},
  SHORTJOURNAL = {Inverse Probl.},
  TITLE = {Regularization wavelets and multiresolution},
  VOLUME = {14},
}

@ARTICLE{FrePop69,
  AUTHOR = {Freud, G. and Popov, V.},
  DATE = {1969},
  JOURNALTITLE = {Proceedings of the Conference on Constructive Theory of Functions. (Approximation Theory.)},
  PAGES = {163--172},
  SHORTJOURNAL = {Proc. Conf. Construct. Theory Functions. (Approx. Theory.)},
  TITLE = {On approximation by spline functions},
}

@ARTICLE{Fri72,
  AUTHOR = {Frieden, B. R.},
  DATE = {1972},
  DOI = {10.1364/josa.62.000511},
  JOURNALTITLE = {Journal of the Optical Society of America},
  NUMBER = {4},
  PAGES = {511},
  SHORTJOURNAL = {J. Opt. Soc. Amer.},
  TITLE = {Restoring with Maximum Likelihood and Maximum Entropy*},
  VOLUME = {62},
}

@ARTICLE{FriBur72,
  AUTHOR = {Frieden, B. R. and Burke, J. J.},
  DATE = {1972},
  DOI = {10.1364/josa.62.001202},
  JOURNALTITLE = {Journal of the Optical Society of America},
  NUMBER = {10},
  PAGES = {1202},
  SHORTJOURNAL = {J. Opt. Soc. Amer.},
  TITLE = {Restoring with Maximum Entropy, II: Superresolution of Photographs of Diffraction-Blurred Impulses*},
  VOLUME = {62},
}

@ARTICLE{FriSchShi25,
  AUTHOR = {Frischauf, L. and Scherzer, O. and Shi, C.},
  EDITOR = {Bubba, T. A.},
  LOCATION = {Berlin, Boston},
  PUBLISHER = {De Gruyter},
  BOOKTITLE = {Data-driven Models in Inverse Problems},
  DATE = {2025},
  DOI = {10.1515/9783111251233-014},
  PAGES = {471--494},
  TITLE = {Classification with neural networks with quadratic decision functions},
}

@ARTICLE{FriSchShi24,
  AUTHOR = {Frischauf, L. and Scherzer, O. and Shi, C.},
  DATE = {2024},
  DOI = {10.1080/01630563.2024.2316580},
  ISSN = {0163-0563},
  ISSUE = {2},
  JOURNALTITLE = {Numerical Functional Analysis and Optimization},
  PAGES = {112--135},
  SHORTJOURNAL = {Numer. Funct. Anal. Optim.},
  TITLE = {Quadratic Neural Networks for Solving Inverse Problems},
}

@ARTICLE{FucJueSchYan09b,
  AUTHOR = {Fuchs, M. and J\"{u}ttler, B. and Scherzer, O. and Yang, H.},
  DATE = {2009},
  DOI = {10.1007/s00791-008-0110-4},
  ISSN = {1432-9360},
  JOURNALTITLE = {Computing and Visualization in Science},
  NUMBER = {6},
  PAGES = {287--295},
  SHORTJOURNAL = {Comput. Vis. Sci.},
  TITLE = {Combined evolution of level sets and {B}-spline curves for imaging},
  VOLUME = {12},
}

@ARTICLE{FulLueWer26,
  AUTHOR = {Fulsche, R. and Luef, F. and Werner, R. F.},
  DATE = {2026},
  DOI = {10.1016/j.jfa.2025.111265},
  JOURNALTITLE = {Journal of Functional Analysis},
  NUMBER = {4},
  PAGES = {111265},
  SHORTJOURNAL = {J. Funct. Anal.},
  TITLE = {Wiener's Tauberian theorem in classical and quantum harmonic analysis},
  VOLUME = {290},
}

@ARTICLE{Fun89,
  AUTHOR = {Funahashi, K.-I.},
  DATE = {1989},
  DOI = {10.1016/0893-6080(89)90003-8},
  JOURNALTITLE = {Neural Networks},
  NUMBER = {3},
  PAGES = {183--192},
  SHORTJOURNAL = {Neural Netw.},
  TITLE = {On the approximate realization of continuous mappings by neural networks},
  VOLUME = {2},
}

@ARTICLE{Fun13,
  AUTHOR = {Funk, P.},
  LOCATION = {Berlin, Heidelberg},
  PUBLISHER = {Springer},
  DATE = {1913},
  ISSN = {0025-5831},
  JOURNALTITLE = {Mathematische Annalen},
  PAGES = {278--300},
  SHORTJOURNAL = {Math. Ann.},
  TITLE = {\"{U}ber {F}l\"{a}chen mit lauter geschlossenen geod\"{a}tischen {L}inien},
  VOLUME = {74},
}

@ARTICLE{GilLevScho18,
  AUTHOR = {Gilbert, A. C. and Levinson, H. W. and Schotland, J. C.},
  DATE = {2018},
  DOI = {10.1364/ol.43.003005},
  JOURNALTITLE = {Optics Letters},
  NUMBER = {12},
  PAGES = {3005},
  SHORTJOURNAL = {Opt. Letters},
  TITLE = {Imaging from the inside out: inverse scattering with photoactivated internal sources},
  VOLUME = {43},
}

@BOOK{Gil18,
  AUTHOR = {Gilboa, G.},
  DATE = {2018},
  DOI = {10.1007/978-3-319-75847-3},
  TITLE = {Nonlinear Eigenproblems in Image Processing and Computer Vision},
}

@ARTICLE{Gil77,
  AUTHOR = {Gilyazov, S.F.},
  DATE = {1977},
  JOURNALTITLE = {Moscow University Computational Mathematics and Cybernetics},
  PAGES = {8--13},
  SHORTJOURNAL = {Moscow Univ. Comput. Math. Cybernet.},
  TITLE = {Iterative solution methods for inconsistent linear equations with non self-adjoint operators},
}

@ARTICLE{GirPog89,
  AUTHOR = {Girosi, F. and Poggio, T.},
  DATE = {1989},
  DOI = {10.1162/neco.1989.1.4.465},
  JOURNALTITLE = {Neural Computation},
  NUMBER = {4},
  PAGES = {465--469},
  SHORTJOURNAL = {Neural Comput.},
  TITLE = {Representation Properties of Networks: Kolmogorov's Theorem Is Irrelevant},
  VOLUME = {1},
}

@BOOK{GolVan96,
  AUTHOR = {Golub, G. H. and Van Loan, Ch. F.},
  LOCATION = {Baltimore},
  PUBLISHER = {The Johns Hopkins University Press},
  DATE = {1996},
  TITLE = {Matrix Computations},
}

@BOOK{GooBenCou16,
  AUTHOR = {Goodfellow, Ian and Bengio, Yoshua and Courville, Aaron},
  PUBLISHER = {MIT Press},
  DATE = {2016},
  TITLE = {Deep Learning},
}

@ARTICLE{GorHof94,
  AUTHOR = {Gorenflo, R. and Hofmann, B.},
  DATE = {1994},
  DOI = {10.1088/0266-5611/10/2/011},
  ISSN = {0266-5611},
  JOURNALTITLE = {Inverse Problems},
  NUMBER = {2},
  PAGES = {353--373},
  SHORTJOURNAL = {Inverse Probl.},
  TITLE = {On autoconvolution and regularization},
  VOLUME = {10},
}

@ARTICLE{Gra10b,
  AUTHOR = {Grasmair, M.},
  DATE = {2010-10},
  DOI = {10.1088/0266-5611/26/11/115014},
  ISSN = {0266-5611},
  JOURNALTITLE = {Inverse Problems},
  NUMBER = {11},
  PAGES = {115014},
  SHORTJOURNAL = {Inverse Probl.},
  TITLE = {Generalized {B}regman distances and convergence rates for non-convex regularization methods},
  VOLUME = {26},
}

@ARTICLE{Gra20,
  AUTHOR = {Grasmair, M.},
  DATE = {2020},
  DOI = {10.1080/01630563.2020.1772289},
  ISSN = {0163-0563},
  JOURNALTITLE = {Numerical Functional Analysis and Optimization},
  NUMBER = {11},
  PAGES = {1352--1372},
  SHORTJOURNAL = {Numer. Funct. Anal. Optim.},
  TITLE = {Source conditions for non-quadratic {T}ikhonov regularization},
  VOLUME = {41},
}

@ARTICLE{GraHalSch11,
  AUTHOR = {Grasmair, M. and Haltmeier, M. and Scherzer, O.},
  PUBLISHER = {Wiley Subscription Services, Inc., A Wiley Company},
  DATE = {2011},
  DOI = {10.1002/cpa.20350},
  ISSN = {0010-3640},
  JOURNALTITLE = {Communications on Pure and Applied Mathematics},
  NUMBER = {2},
  PAGES = {161--182},
  SHORTJOURNAL = {Comm. Pure Appl. Math.},
  TITLE = {Necessary and sufficient conditions for linear convergence of $l^1$-regularization},
  VOLUME = {64},
}

@ARTICLE{GraHalSch08,
  AUTHOR = {Grasmair, M. and Haltmeier, M. and Scherzer, O.},
  DATE = {2008},
  DOI = {10.1088/0266-5611/24/5/055020},
  ISSN = {0266-5611},
  JOURNALTITLE = {Inverse Problems},
  NUMBER = {5},
  PAGES = {055020,13},
  SHORTJOURNAL = {Inverse Probl.},
  TITLE = {Sparse regularization with {$l^q$} penalty term},
  VOLUME = {24},
}

@ARTICLE{Gre83,
  AUTHOR = {Greenleaf, J. F.},
  PUBLISHER = {IEEE},
  DATE = {1983},
  DOI = {10.1109/proc.1983.12591},
  ISSN = {0018-9219},
  JOURNALTITLE = {Proceedings of the {IEEE}},
  NUMBER = {3},
  PAGES = {330--337},
  SHORTJOURNAL = {Proc. {IEEE}},
  TITLE = {Computerized tomography with ultrasound},
  VOLUME = {71},
}

@BOOK{Gri85,
  AUTHOR = {Grisvard, P.},
  LOCATION = {Boston},
  PUBLISHER = {Pitman},
  DATE = {1985},
  TITLE = {Elliptic Problems in Nonsmooth Domains},
}

@INPROCEEDINGS{Gro83,
  AUTHOR = {Groetsch, C. W.},
  EDITOR = {H\"{a}mmerlin, G. and Hoffmann, K. H.},
  PUBLISHER = {Birkh\"{a}user, Basel},
  BOOKTITLE = {Improperly Posed Problems and Their Numerical Treatment},
  DATE = {1983},
  PAGES = {97--104},
  TITLE = {Comments on {M}orozov's {D}iscrepancy {P}rinciple},
}

@BOOK{Gro84,
  AUTHOR = {Groetsch, C. W.},
  LOCATION = {Boston},
  PUBLISHER = {Pitman},
  DATE = {1984},
  TITLE = {The Theory of Tikhonov Regularization for Fredholm Equations of the First Kind},
}

@ARTICLE{GroSch00,
  AUTHOR = {Groetsch, C. W. and Scherzer, O.},
  DATE = {2000},
  DOI = {10.1002/1099-1476(200010)23:15<1287::aid-mma165>3.3.co;2-e},
  ISSN = {0170-4214},
  JOURNALTITLE = {Mathematical Methods in the Applied Sciences},
  NUMBER = {15},
  PAGES = {1287--1300},
  SHORTJOURNAL = {Math. Methods Appl. Sci.},
  TITLE = {Non-stationary iterated {T}ikhonov-{M}orozov method and third-order differential equations for the evaluation of unbounded operators},
  VOLUME = {23},
}

@ARTICLE{GroNeu89,
  AUTHOR = {Groetsch, C.W. and Neubauer, A.},
  DATE = {1989},
  JOURNALTITLE = {Journal of Approximation Theory},
  PAGES = {184--200},
  SHORTJOURNAL = {J. Approx. Theory},
  TITLE = {Regularization of ill-posed problems: optimal parameter choice in finite dimensions},
  VOLUME = {58},
}

@INCOLLECTION{Gro11,
  AUTHOR = {Groetsch, Ch.},
  EDITOR = {Scherzer, O.},
  LOCATION = {New York},
  PUBLISHER = {Springer},
  BOOKTITLE = {\cite{Sch11}},
  DATE = {2011},
  DOI = {10.1007/978-0-387-92920-0_1},
  ISBN = {978-0-387-92920-0},
  PAGES = {3--41},
  TITLE = {Linear Inverse Problems},
}

@INCOLLECTION{Gro04,
  AUTHOR = {Grossauer, H.},
  EDITOR = {Pajdla, T. and Matas, J.},
  LOCATION = {Prague, Czech Republic},
  PUBLISHER = {Springer Berlin / Heidelberg},
  BOOKTITLE = {Computer Vision - ECCV 2004: 8th European Conference on Computer Vision},
  DATE = {2004},
  DOI = {10.1007/978-3-540-24671-8_17},
  ISBN = {978-3-540-21983-5},
  PAGES = {214--224},
  SERIES = {Lecture Notes in Computer Science},
  TITLE = {A Combined {PDE} and Texture Synthesis Approach to Inpainting},
  VOLUME = {3022},
}

@THESIS{Gro05,
  AUTHOR = {Grossauer, H.},
  INSTITUTION = {University of Innsbruck, Austria},
  LANGUAGE = {English},
  LOCATION = {Innsbruck},
  DATE = {2005-02},
  TITLE = {Completion of Images with Missing Data Regions},
  TYPE = {{PhD} thesis},
}

@INPROCEEDINGS{GroSch03,
  AUTHOR = {Grossauer, H. and Scherzer, O.},
  EDITOR = {Griffin, L. D. and Lillholm, M.},
  BOOKTITLE = {Scale Space Methods in Computer Vision, 4th International Conference, Scale-Space 2003},
  DATE = {2003},
  DOI = {10.1007/3-540-44935-3_16},
  ISBN = {3-540-40368-X},
  NUMBER = {2695},
  PAGES = {225--236},
  SERIES = {Lecture Notes in Computer Science},
  TITLE = {Using the complex {G}inzburg--{L}andau equation for digital inpainting in {2D} and {3D}},
}

@ARTICLE{GroKomLatScho24,
  AUTHOR = {Grossmann, T. G. and Komorowska, U. J. and Latz, J. and Sch\"{o}nlieb, C.-B.},
  DATE = {2024},
  DOI = {10.1093/imamat/hxae011},
  JOURNALTITLE = {{IMA} Journal of Applied Mathematics},
  NUMBER = {1},
  PAGES = {143--174},
  SHORTJOURNAL = {{IMA} J. Appl. Math.},
  TITLE = {Can physics-informed neural networks beat the finite element method?},
  VOLUME = {89},
}

@REPORT{GruHolLehHocZel24_report,
  AUTHOR = {Gruber, L. and Holzleitner, M. and Lehner, J. and Hochreiter, S. and Zellinger, W.},
  DATE = {2024},
  DOI = {10.48550/arxiv.2402.13891},
  NUMBER = {2402.13891},
  TITLE = {Overcoming Saturation in Density Ratio Estimation by Iterated Regularization},
  TYPE = {Preprint on ArXiv},
}

@ARTICLE{Haa10,
  AUTHOR = {Haar, A.},
  LOCATION = {Berlin, Heidelberg},
  PUBLISHER = {Springer},
  DATE = {1910},
  DOI = {10.1007/bf01456326},
  ISSN = {0025-5831},
  JOURNALTITLE = {Mathematische Annalen},
  NUMBER = {3},
  PAGES = {331--371},
  SHORTJOURNAL = {Math. Ann.},
  TITLE = {Zur {T}heorie der orthogonalen {F}unktionensysteme},
  VOLUME = {69},
}

@ARTICLE{HabTen03,
  AUTHOR = {Haber, E. and Tenorio, L.},
  DATE = {2003},
  DOI = {10.1088/0266-5611/19/3/309},
  ISSN = {0266-5611},
  JOURNALTITLE = {Inverse Problems},
  NUMBER = {3},
  PAGES = {611--626},
  SHORTJOURNAL = {Inverse Probl.},
  TITLE = {Learning regularization functionals a supervised training approach},
  VOLUME = {19},
}

@REPORT{HabHol23_report,
  AUTHOR = {Habring, A. and Holler, M.},
  DATE = {2023},
  DOI = {10.48550/arxiv.2312.14849},
  NUMBER = {2312.14849},
  TITLE = {Neural-network-based regularization methods for inverse problems in imaging},
  TYPE = {Preprint on ArXiv},
}

@BOOK{Hac18,
  AUTHOR = {Hackbusch, W.},
  DATE = {2017},
  DOI = {10.1007/978-3-662-54961-2},
  TITLE = {Elliptic Differential Equations},
}

@ARTICLE{Had07,
  AUTHOR = {Hadamard, M.},
  DATE = {1907},
  DOI = {10.1051/jphystap:019070060020200},
  JOURNALTITLE = {Journal of Physics: Theories and Applications},
  NUMBER = {1},
  PAGES = {202--241},
  SHORTJOURNAL = {J. Phys.: Theor. Appl.},
  TITLE = {Les probl\`{e}mes aux limites dans la th\'{e}orie des \'{e}quations aux d\'{e}riv\'{e}es partielles},
  VOLUME = {6},
}

@INCOLLECTION{HalNgu23,
  AUTHOR = {Haltmeier, M. and Nguyen, L.},
  BOOKTITLE = {Handbook of Mathematical Models and Algorithms in Computer Vision and Imaging},
  DATE = {2023},
  DOI = {10.1007/978-3-030-98661-2_81},
  PAGES = {1065--1093},
  TITLE = {Regularization of Inverse Problems by Neural Networks},
}

@INPROCEEDINGS{HanZan91,
  AUTHOR = {Hanafy, A. and Zanelli, C. I.},
  BOOKTITLE = {IEEE 1991 Ultrasonics Symposium},
  DATE = {1991},
  DOI = {10.1109/ultsym.1991.234310},
  PAGES = {1223--1227},
  TITLE = {Quantitative real-time pulsed Schlieren imaging of ultrasonic waves},
}

@ARTICLE{Han97,
  AUTHOR = {Hanke, M.},
  DATE = {1997},
  ISSN = {0266-5611},
  JOURNALTITLE = {Inverse Problems},
  NUMBER = {1},
  PAGES = {79--95},
  SHORTJOURNAL = {Inverse Probl.},
  TITLE = {A regularizing {L}evenberg--{M}arquardt scheme, with applications to inverse groundwater filtration problems},
  VOLUME = {13},
}

@BOOK{Han17,
  AUTHOR = {Hanke, M.},
  DATE = {2017},
  DOI = {10.1137/1.9781611974942},
  TITLE = {A Taste of Inverse Problems},
}

@BOOK{Han95,
  AUTHOR = {Hanke, M.},
  LOCATION = {Harlow},
  PUBLISHER = {Longman Scientific \& Technical},
  DATE = {1995},
  SERIES = {Pitman Research Notes in Mathematics Series},
  TITLE = {Conjugate Gradient Type Methods for Ill-Posed Problems},
  VOLUME = {327},
}

@BOOK{Han02,
  AUTHOR = {Hanke, M.},
  LOCATION = {Stuttgart, Leipzig, Wiesbaden},
  PUBLISHER = {Teubner},
  DATE = {2002},
  TITLE = {Grundlagen der Numerischen Mathematik und des Wissenschaftlichen Rechnens},
}

@ARTICLE{HanNeuSch95,
  AUTHOR = {Hanke, M. and Neubauer, A. and Scherzer, O.},
  DATE = {1995},
  DOI = {10.1007/s002110050158},
  ISSN = {0029-599X},
  JOURNALTITLE = {Numerische Mathematik},
  NUMBER = {1},
  PAGES = {21--37},
  SHORTJOURNAL = {Numer. Math.},
  TITLE = {A convergence analysis of the {L}andweber iteration for nonlinear ill-posed problems},
  VOLUME = {72},
}

@ARTICLE{HanRau96,
  AUTHOR = {Hanke, M. and Raus, T.},
  DATE = {1996},
  DOI = {10.1137/0917062},
  ISSN = {1064-8275},
  JOURNALTITLE = {{SIAM} Journal on Scientific Computing},
  NUMBER = {4},
  PAGES = {956--972},
  SHORTJOURNAL = {{SIAM} J. Sci. Comput.},
  TITLE = {A General Heuristic for Choosing the Regularization Parameter in Ill-Posed Problems},
  VOLUME = {17},
}

@ARTICLE{HanSch01,
  AUTHOR = {Hanke, M. and Scherzer, O.},
  DATE = {2001},
  DOI = {10.2307/2695705},
  ISSN = {0002-9890},
  JOURNALTITLE = {The American Mathematical Monthly},
  NUMBER = {6},
  PAGES = {512--521},
  SHORTJOURNAL = {Amer. Math. Mon.},
  TITLE = {Inverse problems light: numerical differentiation},
  VOLUME = {108},
}

@ARTICLE{Han92,
  AUTHOR = {Hansen, P. Ch.},
  DATE = {1992},
  DOI = {10.1137/1034115},
  ISSN = {0036-1445},
  JOURNALTITLE = {{SIAM} Review},
  NUMBER = {4},
  PAGES = {561--580},
  SHORTJOURNAL = {{SIAM} Rev.},
  TITLE = {Analysis of Discrete Ill-Posed Problems by Means of the L-Curve},
  VOLUME = {34},
}

@BOOK{Han10,
  AUTHOR = {Hansen, P. Ch.},
  LOCATION = {Philadelphia, PA},
  PUBLISHER = {SIAM},
  DATE = {2010},
  DOI = {10.1137/1.9780898718836},
  SERIES = {Fundamentals of Algorithms},
  TITLE = {Discrete inverse problems},
  VOLUME = {7},
}

@ARTICLE{HanOle93,
  AUTHOR = {Hansen, P. Ch. and O'Leary, D. P.},
  DATE = {1993},
  DOI = {10.1137/0914086},
  ISSN = {1064-8275},
  JOURNALTITLE = {{SIAM} Journal on Scientific Computing},
  NUMBER = {6},
  PAGES = {1487--1503},
  SHORTJOURNAL = {{SIAM} J. Sci. Comput.},
  TITLE = {The Use of the L-Curve in the Regularization of Discrete Ill-Posed Problems},
  VOLUME = {14},
}

@INCOLLECTION{HarKerPicTsy98,
  AUTHOR = {H\"{a}rdle, W. and Kerkyacharian, G. and Picard, D. and Tsybakov, A.},
  BOOKTITLE = {Lecture Notes in Statistics},
  DATE = {1998},
  DOI = {10.1007/978-1-4612-2222-4_9},
  PAGES = {101--124},
  TITLE = {Wavelets and Besov Spaces},
}

@BOOK{Har49,
  AUTHOR = {Hardy, G. H.},
  PUBLISHER = {Clarendon Press},
  DATE = {1949},
  TITLE = {Divergent series},
}

@ARTICLE{Har31,
  AUTHOR = {Hardy, G. H.},
  DATE = {1931},
  DOI = {10.1093/qmath/os-2.1.107},
  JOURNALTITLE = {The Quarterly Journal of Mathematics},
  NUMBER = {1},
  PAGES = {107--112},
  SHORTJOURNAL = {Q. J. Math.},
  TITLE = {The Summability of a Fourier Series by Logarithmic Means},
  VOLUME = {os-2},
}

@ARTICLE{HarLit21,
  AUTHOR = {Hardy, G. H. and Littlewood, J. E.},
  DATE = {1921},
  DOI = {10.1112/plms/s2-19.1.21},
  JOURNALTITLE = {Proceedings of the London Mathematical Society},
  NUMBER = {1},
  PAGES = {21--29},
  SHORTJOURNAL = {Proc. Lond. Math. Soc.},
  TITLE = {On a Tauberian Theorem for Lambert's Series, and Some Fundamental Theorems in the Analytic Theory of Numbers},
  VOLUME = {s2-19},
}

@ARTICLE{Hae86,
  AUTHOR = {H\"{a}ussler, W. M.},
  DATE = {1986},
  ISSN = {0029-599X},
  JOURNALTITLE = {Numerische Mathematik},
  PAGES = {119--125},
  SHORTJOURNAL = {Numer. Math.},
  TITLE = {A {K}antorovich-type convergence analysis for the {G}auss-{N}ewton-method},
  VOLUME = {48},
}

@ARTICLE{HeiHof09,
  AUTHOR = {Hein, T. and Hofmann, B.},
  DATE = {2009},
  DOI = {10.1088/0266-5611/25/3/035003},
  EID = {035003},
  ISSN = {0266-5611},
  JOURNALTITLE = {Inverse Problems},
  NUMBER = {3},
  PAGES = {035003},
  SHORTJOURNAL = {Inverse Probl.},
  TITLE = {Approximate source conditions for nonlinear ill-posed problems---chances and limitations},
  VOLUME = {25},
}

@ARTICLE{HerSchwZec24,
  AUTHOR = {Herrmann, L. and Schwab, Ch. and Zech, J.},
  PUBLISHER = {Springer US},
  DATE = {2024},
  DOI = {10.1007/s10444-024-10171-2},
  ISSN = {1019-7168},
  JOURNALTITLE = {Advances in Computational Mathematics},
  NUMBER = {4},
  SHORTJOURNAL = {Adv. Comput. Math.},
  TITLE = {Neural and spectral operator surrogates: unified construction and expression rate bounds},
  VOLUME = {50},
}

@ARTICLE{HinDenYuDahMoh12,
  AUTHOR = {Hinton, G. and Deng, L. and Yu, D. and Dahl, G. and Mohamed, A.-r. and Jaitly, N. and Senior, A. and Vanhoucke, V. and Nguyen, P. and Sainath, T. and Kingsbury, B.},
  DATE = {2012},
  DOI = {10.1109/msp.2012.2205597},
  JOURNALTITLE = {{IEEE} Signal Processing Magazine},
  NUMBER = {6},
  PAGES = {82--97},
  SHORTJOURNAL = {{IEEE} Signal Process. Mag.},
  TITLE = {Deep Neural Networks for Acoustic Modeling in Speech Recognition: The Shared Views of Four Research Groups},
  VOLUME = {29},
}

@ARTICLE{HocHonOst09,
  AUTHOR = {Hochbruck, M. and H\"{o}nig, M. and Ostermann, A.},
  DATE = {2009},
  DOI = {10.1051/m2an/2009021},
  JOURNALTITLE = {{ESAIM}: Mathematical Modelling and Numerical Analysis},
  NUMBER = {4},
  PAGES = {709--720},
  SHORTJOURNAL = {{ESAIM} Math. Model. Numer. Anal.},
  TITLE = {Regularization of nonlinear ill-posed problems by exponential integrators},
  VOLUME = {43},
}

@THESIS{Hoc91,
  AUTHOR = {Hochreiter, S.},
  INSTITUTION = {Technische Universit\"{a}t M\"{u}nchen},
  DATE = {1991},
  PAGES = {1--150},
  TITLE = {Untersuchungen zu dynamischen neuronalen {Netzen}},
  TYPE = {Master thesis},
  VOLUME = {1},
}

@ARTICLE{HocSchm97,
  AUTHOR = {Hochreiter, S. and Schmidhuber, J.},
  DATE = {1997},
  DOI = {10.1162/neco.1997.9.8.1735},
  JOURNALTITLE = {Neural Computation},
  NUMBER = {8},
  PAGES = {1735--1780},
  SHORTJOURNAL = {Neural Comput.},
  TITLE = {Long Short-Term Memory},
  VOLUME = {9},
}

@ARTICLE{Hof06,
  AUTHOR = {Hofmann, B.},
  DATE = {2006},
  DOI = {10.1002/mma.686},
  ISSN = {0170-4214},
  JOURNALTITLE = {Mathematical Methods in the Applied Sciences},
  NUMBER = {3},
  PAGES = {351--371},
  SHORTJOURNAL = {Math. Methods Appl. Sci.},
  TITLE = {Approximate source conditions in {T}ikhonov-{P}hillips regularization and consequences for inverse problems with multiplication operators},
  VOLUME = {29},
}

@BOOK{Hof99,
  AUTHOR = {Hofmann, B.},
  LOCATION = {Stuttgart},
  PUBLISHER = {Teubner},
  DATE = {1999},
  TITLE = {Mathematik Inverser Probleme},
}

@ARTICLE{Hof94,
  AUTHOR = {Hofmann, B.},
  DATE = {1994},
  ISSN = {0928-0219},
  JOURNALTITLE = {Journal of Inverse and Ill-Posed Problems},
  PAGES = {61--76},
  SHORTJOURNAL = {J. Inverse Ill-Posed Probl.},
  TITLE = {On the degree of ill-posedness for nonlinear problems},
  VOLUME = {2},
}

@BOOK{Hof86,
  EDITOR = {Hofmann, B.},
  LOCATION = {Leipzig},
  DATE = {1986},
  DOI = {10.1007/978-3-322-93034-7},
  TITLE = {Regularization for Applied Inverse and III-Posed Problems},
}

@INPROCEEDINGS{Hof93,
  AUTHOR = {Hofmann, B.},
  EDITOR = {Anger et al., G.},
  LOCATION = {Berlin},
  PUBLISHER = {Akademie Verlag},
  BOOKTITLE = {Inverse problems: principles and applications in geophysics, technology, and medicine},
  DATE = {1993},
  PAGES = {174--188},
  TITLE = {Regularization of nonlinear problems and the degree of ill-posedness},
}

@INCOLLECTION{HofHofMatPla21,
  AUTHOR = {Hofmann, B. and Hofmann, Ch. and Math\'{e}, P. and Plato, R.},
  BOOKTITLE = {Deterministic and Stochastic Optimal Control and Inverse Problems},
  DATE = {2021},
  DOI = {10.1201/9781003050575-4},
  PAGES = {79--111},
  TITLE = {Nonlinear Tikhonov Regularization in Hilbert Scales with Oversmoothing Penalty: Inspecting Balancing Principles},
}

@ARTICLE{HofKalPoeSch07,
  AUTHOR = {Hofmann, B. and Kaltenbacher, B. and P\"{o}schl, C. and Scherzer, O.},
  DATE = {2007},
  DOI = {10.1088/0266-5611/23/3/009},
  ISSN = {0266-5611},
  JOURNALTITLE = {Inverse Problems},
  NUMBER = {3},
  PAGES = {987--1010},
  SHORTJOURNAL = {Inverse Probl.},
  TITLE = {A convergence rates result for {T}ikhonov regularization in {B}anach spaces with non-smooth operators},
  VOLUME = {23},
}

@ARTICLE{HofKinMat19,
  AUTHOR = {Hofmann, B. and Kindermann, S. and Math\'{e}, P.},
  DATE = {2019},
  DOI = {10.1515/jiip-2018-0039},
  ISSN = {0928-0219},
  JOURNALTITLE = {Journal of Inverse and Ill-Posed Problems},
  NUMBER = {2},
  PAGES = {283--300},
  SHORTJOURNAL = {J. Inverse Ill-Posed Probl.},
  TITLE = {Penalty-based smoothness conditions in convex variational regularization},
  VOLUME = {27},
}

@INCOLLECTION{HofMat20,
  AUTHOR = {Hofmann, B. and Math\'{e}, P.},
  PUBLISHER = {Springer Singapore},
  BOOKTITLE = {Springer Proceedings in Mathematics {\&} Statistics},
  DATE = {2020},
  DOI = {10.1007/978-981-15-1592-7_8},
  PAGES = {169--176},
  TITLE = {A Priori Parameter Choice in Tikhonov Regularization with Oversmoothing Penalty for Non-linear Ill-Posed Problems},
}

@ARTICLE{HofMat07,
  AUTHOR = {Hofmann, B. and Math\'{e}, P.},
  PUBLISHER = {SIAM},
  DATE = {2007},
  DOI = {10.1137/060654530},
  ISSN = {0036-1429},
  JOURNALTITLE = {{SIAM} Journal on Numerical Analysis},
  NUMBER = {3},
  PAGES = {1122--1141},
  SHORTJOURNAL = {{SIAM} J. Numer. Anal.},
  TITLE = {Analysis of Profile Functions for General Linear Regularization Methods},
  VOLUME = {45},
}

@ARTICLE{HofSch94,
  AUTHOR = {Hofmann, B. and Scherzer, O.},
  DATE = {1994},
  DOI = {10.1088/0266-5611/10/6/007},
  ISSN = {0266-5611},
  JOURNALTITLE = {Inverse Problems},
  NUMBER = {6},
  PAGES = {1277--1297},
  SHORTJOURNAL = {Inverse Probl.},
  TITLE = {Factors influencing the ill-posedness of nonlinear problems},
  VOLUME = {10},
}

@ARTICLE{HofSch98,
  AUTHOR = {Hofmann, B. and Scherzer, O.},
  DATE = {1998},
  DOI = {10.1088/0266-5611/14/5/007},
  ISSN = {0266-5611},
  JOURNALTITLE = {Inverse Problems},
  NUMBER = {5},
  PAGES = {1189--1206},
  SHORTJOURNAL = {Inverse Probl.},
  TITLE = {Local ill-posedness and source conditions of operator equations in {H}ilbert spaces},
  VOLUME = {14},
}

@ARTICLE{HofTau97,
  AUTHOR = {Hofmann, B. and Tautenhahn, U.},
  DATE = {1997},
  DOI = {10.4171/zaa/800},
  JOURNALTITLE = {Zeitschrift f\"{u}r Analysis und ihre Anwendungen. Journal of Analysis and its Applications},
  NUMBER = {4},
  PAGES = {979--1000},
  SHORTJOURNAL = {Z. Anal. Anwend.},
  TITLE = {On Ill-Posedness Measures and Space Change in Sobolev Scales},
  VOLUME = {16},
}

@ARTICLE{HofYam10,
  AUTHOR = {Hofmann, B. and Yamamoto, M.},
  DATE = {2010},
  DOI = {10.1080/00036810903208148},
  ISSN = {0003-6811},
  JOURNALTITLE = {Applicable Analysis},
  NUMBER = {11},
  PAGES = {1705--1727},
  SHORTJOURNAL = {Appl. Anal.},
  TITLE = {On the interplay of source conditions and variational inequalities for nonlinear ill-posed problems},
  VOLUME = {89},
}

@ARTICLE{Hoh00,
  AUTHOR = {Hohage, T.},
  DATE = {2000},
  DOI = {10.1080/01630560008816965},
  ISSN = {0163-0563},
  JOURNALTITLE = {Numerical Functional Analysis and Optimization},
  NUMBER = {3-4},
  PAGES = {439--464},
  SHORTJOURNAL = {Numer. Funct. Anal. Optim.},
  TITLE = {Regularization of exponentially ill-posed problems},
  VOLUME = {21},
}

@ARTICLE{HohWei17a,
  AUTHOR = {Hohage, T. and Weidling, F.},
  PUBLISHER = {SIAM},
  DATE = {2017},
  DOI = {10.1137/16m1067445},
  ISSN = {0036-1429},
  JOURNALTITLE = {{SIAM} Journal on Numerical Analysis},
  NUMBER = {2},
  PAGES = {598--620},
  SHORTJOURNAL = {{SIAM} J. Numer. Anal.},
  TITLE = {Characterizations of Variational Source Conditions, Converse Results, and Maxisets of Spectral Regularization Methods},
  VOLUME = {55},
}

@ARTICLE{HolKunBar18,
  AUTHOR = {Holler, G. and Kunisch, K. and Barnard, R. C.},
  DATE = {2018},
  DOI = {10.1088/1361-6420/aade77},
  ISSN = {0266-5611},
  JOURNALTITLE = {Inverse Problems},
  NUMBER = {11},
  PAGES = {115012},
  SHORTJOURNAL = {Inverse Probl.},
  TITLE = {A bilevel approach for parameter learning in inverse problems},
  VOLUME = {34},
}

@ARTICLE{HolPer24,
  AUTHOR = {Holzleitner, M. and Pereverzyev, S. V.},
  DATE = {2024},
  DOI = {10.1016/j.jco.2024.101853},
  JOURNALTITLE = {Journal of Complexity},
  PAGES = {101853},
  SHORTJOURNAL = {J. Complex},
  TITLE = {On regularized polynomial functional regression},
  VOLUME = {83},
}

@ARTICLE{HooQiuSch15,
  AUTHOR = {de Hoop, M. V. and Qiu, L. and Scherzer, O.},
  DATE = {2015-01},
  DOI = {10.1007/s00211-014-0629-x},
  ISSN = {0029-599X},
  JOURNALTITLE = {Numerische Mathematik},
  NUMBER = {1},
  PAGES = {127--148},
  SHORTJOURNAL = {Numer. Math.},
  TITLE = {An analysis of a multi-level projected steepest descent iteration for nonlinear inverse problems in Banach spaces subject to stability constraints},
  VOLUME = {129},
}

@ARTICLE{HooQiuSch12,
  AUTHOR = {de Hoop, M. V. and Qiu, L. and Scherzer, O.},
  PUBLISHER = {IOP Publishing},
  DATE = {2012-03},
  DOI = {10.1088/0266-5611/28/4/045001},
  ISSN = {0266-5611},
  JOURNALTITLE = {Inverse Problems},
  NUMBER = {4},
  PAGES = {16pp},
  SHORTJOURNAL = {Inverse Probl.},
  TITLE = {Local analysis of inverse problems: {H}\"{o}lder stability and iterative reconstruction},
  VOLUME = {28},
}

@BOOK{Hoe03,
  AUTHOR = {H\"{o}rmander, L.},
  LOCATION = {New York},
  PUBLISHER = {Springer Verlag},
  DATE = {2003},
  EDITION = {2},
  TITLE = {The Analysis of Linear Partial Differential Operators I},
}

@ARTICLE{HorKid15,
  AUTHOR = {H\"{o}rmann, S. and Kidzi\'{n}ski, {\L{}}.},
  DATE = {2015},
  DOI = {10.1111/sjos.12094},
  JOURNALTITLE = {Scandinavian Journal of Statistics},
  NUMBER = {1},
  PAGES = {43--62},
  SHORTJOURNAL = {Scand. J. Stat.},
  TITLE = {A Note on Estimation in Hilbertian Linear Models},
  VOLUME = {42},
}

@ARTICLE{Hor91,
  AUTHOR = {Hornik, K.},
  DATE = {1991},
  DOI = {10.1016/0893-6080(91)90009-t},
  JOURNALTITLE = {Neural Networks},
  NUMBER = {2},
  PAGES = {251--257},
  SHORTJOURNAL = {Neural Netw.},
  TITLE = {Approximation capabilities of multilayer feedforward networks},
  VOLUME = {4},
}

@ARTICLE{HorStiWhi89,
  AUTHOR = {Hornik, K. and Stinchcombe, M. and White, H.},
  DATE = {1989},
  JOURNALTITLE = {Neural Networks},
  PAGES = {359--366},
  SHORTJOURNAL = {Neural Netw.},
  TITLE = {Multilayer Feedforward Networks are Universal Approximators},
  VOLUME = {2},
}

@ARTICLE{Hou73,
  AUTHOR = {Hounsfield, G. N.},
  DATE = {1973},
  JOURNALTITLE = {British Journal of Radiology},
  NUMBER = {552},
  PAGES = {1016--1022},
  SHORTJOURNAL = {Brit. J. Radiology},
  TITLE = {Computerised transverse axial scanning (tomography). {P}art 1: Description of system},
  VOLUME = {46},
}

@ARTICLE{Gua03,
  AUTHOR = {Huang, Guang-Bin},
  DATE = {2003},
  DOI = {10.1109/tnn.2003.809401},
  JOURNALTITLE = {{IEEE} Transactions on Neural Networks},
  NUMBER = {2},
  PAGES = {274--281},
  SHORTJOURNAL = {{IEEE} Trans. Neural Netw.},
  TITLE = {Learning capability and storage capacity of two-hidden-layer feedforward networks},
  VOLUME = {14},
}

@ARTICLE{HuaJinLuZha25,
  AUTHOR = {Huang, J. and Jin, Q. and Lu, X. and Zhang, L.},
  DATE = {2025},
  DOI = {10.1007/s00211-025-01458-7},
  ISSN = {0029-599X},
  JOURNALTITLE = {Numerische Mathematik},
  NUMBER = {2},
  PAGES = {539--571},
  SHORTJOURNAL = {Numer. Math.},
  TITLE = {On early stopping of stochastic mirror descent method for ill-posed inverse problems},
  VOLUME = {157},
}

@ARTICLE{HubSheNeuSch18,
  AUTHOR = {Hubmer, S. and Sherina, E. and Neubauer, A. and Scherzer, O.},
  DATE = {2018},
  DOI = {10.1137/17m1154461},
  ISSN = {1936-4954},
  JOURNALTITLE = {{SIAM} Journal on Imaging Sciences},
  NUMBER = {2},
  PAGES = {1268--1293},
  SHORTJOURNAL = {{SIAM} J. Imaging Sciences},
  TITLE = {Lam\'{e} Parameter Estimation from Static Displacement Field Measurements in the Framework of Nonlinear Inverse Problems},
  VOLUME = {11},
}

@ARTICLE{ItoKun99,
  AUTHOR = {Ito, K. and Kunisch, K.},
  DATE = {1999},
  ISSN = {0764-583X},
  JOURNALTITLE = {Mathematical Modelling and Numerical Analysis},
  NUMBER = {1},
  PAGES = {1--21},
  SHORTJOURNAL = {Math. Model. Numer. Anal.},
  TITLE = {An active set strategy based on the augmented {L}agrangian formulation for image restoration},
  VOLUME = {33},
}

@ARTICLE{ItoKun00,
  AUTHOR = {Ito, K. and Kunisch, K.},
  DATE = {2000},
  ISSN = {0362-546X},
  JOURNALTITLE = {Nonlinear Analysis: Theory, Methods {\&} Applications},
  PAGES = {591--616},
  SHORTJOURNAL = {Nonlinear Anal.},
  TITLE = {Augmented {L}agrangian methods for nonsmooth, convex optimization in {H}ilbert spaces},
  VOLUME = {41A},
}

@ARTICLE{Jel76,
  AUTHOR = {Jelinek, F.},
  PUBLISHER = {IEEE},
  DATE = {1976},
  DOI = {10.1109/proc.1976.10159},
  ISSN = {0018-9219},
  JOURNALTITLE = {Proceedings of the {IEEE}},
  NUMBER = {4},
  PAGES = {532--556},
  SHORTJOURNAL = {Proc. {IEEE}},
  TITLE = {Continuous speech recognition by statistical methods},
  VOLUME = {64},
}

@ARTICLE{JinMccFroUns17,
  AUTHOR = {Jin, K. H. and McCann, M. T. and Froustey, E. and Unser, M.},
  PUBLISHER = {IEEE},
  DATE = {2017},
  DOI = {10.1109/tip.2017.2713099},
  ISSN = {1057-7149},
  JOURNALTITLE = {{IEEE} Transactions on Image Processing},
  NUMBER = {9},
  PAGES = {4509--4522},
  SHORTJOURNAL = {{IEEE} Trans. Image Process.},
  TITLE = {Deep Convolutional Neural Network for Inverse Problems in Imaging},
  VOLUME = {26},
}

@ARTICLE{Jin25,
  AUTHOR = {Jin, Q.},
  DATE = {2025},
  DOI = {10.1088/1361-6420/adc598},
  ISSN = {0266-5611},
  JOURNALTITLE = {Inverse Problems},
  NUMBER = {4},
  PAGES = {045011},
  SHORTJOURNAL = {Inverse Probl.},
  TITLE = {Convergence rates of Landweber-type methods for inverse problems in Banach spaces},
  VOLUME = {41},
}

@ARTICLE{JinHua24,
  AUTHOR = {Jin, Q. and Huang, Q.},
  DATE = {2024},
  DOI = {10.1137/24m1651721},
  ISSN = {1936-4954},
  JOURNALTITLE = {{SIAM} Journal on Imaging Sciences},
  NUMBER = {4},
  PAGES = {2212--2241},
  SHORTJOURNAL = {{SIAM} J. Imaging Sciences},
  TITLE = {An Adaptive Heavy Ball Method for Ill-Posed Inverse Problems},
  VOLUME = {17},
}

@ARTICLE{JinLu14,
  AUTHOR = {Jin, Q. and Lu, X.},
  DATE = {2014},
  DOI = {10.1088/0266-5611/30/4/045012},
  ISSN = {0266-5611},
  JOURNALTITLE = {Inverse Problems},
  NUMBER = {4},
  PAGES = {045012},
  SHORTJOURNAL = {Inverse Probl.},
  TITLE = {A fast nonstationary iterative method with convex penalty for inverse problems in Hilbert spaces},
  VOLUME = {30},
}

@ARTICLE{JinWan18,
  AUTHOR = {Jin, Q. and Wang, W.},
  DATE = {2018},
  DOI = {10.1088/1361-6420/aaa0fb},
  ISSN = {0266-5611},
  JOURNALTITLE = {Inverse Problems},
  NUMBER = {3},
  PAGES = {035001},
  SHORTJOURNAL = {Inverse Probl.},
  TITLE = {Analysis of the iteratively regularized Gauss--Newton method under a heuristic rule},
  VOLUME = {34},
}

@ARTICLE{AlElr19,
  AUTHOR = {Al-johania, N. and Elrefaei, L.},
  DATE = {2019},
  DOI = {10.22266/ijies2019.0630.19},
  JOURNALTITLE = {International Journal of Intelligent Engineering and Systems},
  NUMBER = {3},
  PAGES = {178--191},
  SHORTJOURNAL = {Int. J. Int. Eng. Sys.},
  TITLE = {Dorsal Hand Vein Recognition by Convolutional Neural Networks: Feature Learning and Transfer Learning Approaches},
  VOLUME = {12},
}

@ARTICLE{Jon90,
  AUTHOR = {Jones, L. K.},
  PUBLISHER = {IEEE},
  DATE = {1990},
  DOI = {10.1109/5.58342},
  ISSN = {0018-9219},
  JOURNALTITLE = {Proceedings of the {IEEE}},
  NUMBER = {10},
  PAGES = {1586--1589},
  SHORTJOURNAL = {Proc. {IEEE}},
  TITLE = {Constructive approximations for neural networks by sigmoidal functions},
  VOLUME = {78},
}

@BOOK{Kab12,
  AUTHOR = {Kabanikhin, S. I.},
  LOCATION = {Berlin, New York},
  PUBLISHER = {De Gruyter},
  DATE = {2012},
  DOI = {10.1515/9783110224016},
  TITLE = {Inverse and Ill-Posed Problems. Theory and Applications.},
}

@BOOK{KaiSom05,
  AUTHOR = {Kaipio, J. and Somersalo, E.},
  LOCATION = {New York},
  PUBLISHER = {Springer Verlag},
  DATE = {2005},
  DOI = {10.1007/b138659},
  ISBN = {978-0-387-27132-3},
  SERIES = {Applied Mathematical Sciences},
  TITLE = {Statistical and Computational Inverse Problems},
  VOLUME = {160},
}

@ARTICLE{Kak79,
  AUTHOR = {Kak, A. C.},
  PUBLISHER = {IEEE},
  DATE = {1979},
  DOI = {10.1109/proc.1979.11440},
  ISSN = {0018-9219},
  JOURNALTITLE = {Proceedings of the {IEEE}},
  NUMBER = {9},
  PAGES = {1245--1272},
  SHORTJOURNAL = {Proc. {IEEE}},
  TITLE = {Computerized tomography with X-ray, emission, and ultrasound sources},
  VOLUME = {67},
}

@BOOK{KakSla01,
  AUTHOR = {Kak, A. C. and Slaney, M.},
  LOCATION = {Philadelphia, PA},
  PUBLISHER = {Society for Industrial and Applied Mathematics (SIAM)},
  DATE = {2001},
  DOI = {10.1137/1.9780898719277},
  SERIES = {Classics in Applied Mathematics},
  TITLE = {Principles of Computerized Tomographic Imaging},
  VOLUME = {33},
}

@ARTICLE{Kal08,
  AUTHOR = {Kaltenbacher, B.},
  DATE = {2008},
  ISSN = {0897-3962},
  JOURNALTITLE = {Journal of Integral Equations and Applications},
  NUMBER = {2},
  PAGES = {201--228},
  SHORTJOURNAL = {J. Integral Equations Appl.},
  TITLE = {Convergence rates of a multilevel method for the regularization of nonlinear ill-posed problems},
  VOLUME = {20},
}

@ARTICLE{Kal98b,
  AUTHOR = {Kaltenbacher, B.},
  DATE = {1998},
  ISSN = {0163-0563},
  JOURNALTITLE = {Numerical Functional Analysis and Optimization},
  PAGES = {807--833},
  SHORTJOURNAL = {Numer. Funct. Anal. Optim.},
  TITLE = {On {B}royden's method for nonlinear ill-posed problems},
  VOLUME = {19},
}

@ARTICLE{Kal16,
  AUTHOR = {Kaltenbacher, B.},
  PUBLISHER = {SIAM},
  DATE = {2016},
  DOI = {10.1137/16m1060984},
  ISSN = {0036-1429},
  JOURNALTITLE = {{SIAM} Journal on Numerical Analysis},
  NUMBER = {4},
  PAGES = {2594--2618},
  SHORTJOURNAL = {{SIAM} J. Numer. Anal.},
  TITLE = {Regularization Based on All-At-Once Formulations for Inverse Problems},
  VOLUME = {54},
}

@ARTICLE{Kal97,
  AUTHOR = {Kaltenbacher, B.},
  DATE = {1997},
  ISSN = {0266-5611},
  JOURNALTITLE = {Inverse Problems},
  PAGES = {729--753},
  SHORTJOURNAL = {Inverse Probl.},
  TITLE = {Some {N}ewton-type methods for the regularization of nonlinear ill-posed problems},
  VOLUME = {13},
}

@ARTICLE{Kal06,
  AUTHOR = {Kaltenbacher, B.},
  DATE = {2006},
  ISSN = {0163-0563},
  JOURNALTITLE = {Numerical Functional Analysis and Optimization},
  NUMBER = {5-6},
  PAGES = {637--665},
  SHORTJOURNAL = {Numer. Funct. Anal. Optim.},
  TITLE = {Toward global convergence for strongly nonlinear ill-posed problems via a regularizing multilevel approach},
  VOLUME = {27},
}

@ARTICLE{KalNeuRam02,
  AUTHOR = {Kaltenbacher, B. and Neubauer, A. and Ramm, A. G.},
  DATE = {2002},
  DOI = {10.1515/jiip.2002.10.3.261},
  ISSN = {0928-0219},
  JOURNALTITLE = {Journal of Inverse and Ill-Posed Problems},
  NUMBER = {3},
  PAGES = {261--280},
  SHORTJOURNAL = {J. Inverse Ill-Posed Probl.},
  TITLE = {Convergence rates of the continuous regularized Gauss---Newton method},
  VOLUME = {10},
}

@BOOK{KalNeuSch08,
  AUTHOR = {Kaltenbacher, B. and Neubauer, A. and Scherzer, O.},
  LOCATION = {Berlin},
  PUBLISHER = {Walter de Gruyter},
  DATE = {2008},
  DOI = {10.1515/9783110208276},
  ISBN = {978-3-11-020420-9},
  SERIES = {Radon Series on Computational and Applied Mathematics},
  TITLE = {Iterative regularization methods for nonlinear ill-posed problems},
  VOLUME = {6},
}

@INCOLLECTION{KalNguSch21,
  AUTHOR = {Kaltenbacher, B. and Nguyen, T. T. N. and Scherzer, O.},
  EDITOR = {Kaltenbacher, B. and Schuster, T. and Wald, A.},
  PUBLISHER = {Springer, Cham},
  BOOKTITLE = {Time-dependent Problems in Imaging and Parameter Identification},
  DATE = {2021},
  DOI = {10.1007/978-3-030-57784-1_5},
  PAGES = {121--163},
  TITLE = {The Tangential Cone Condition for Some Coefficient Identification Model Problems in Parabolic PDEs},
}

@ARTICLE{KanMar14,
  AUTHOR = {Kang, S. H. and March, R.},
  DATE = {2014},
  DOI = {10.1016/j.jvcir.2014.04.008},
  JOURNALTITLE = {Journal of Visual Communication and Image Representation},
  NUMBER = {6},
  PAGES = {1446--1459},
  SHORTJOURNAL = {J. Vis. Commun. Image Represent.},
  TITLE = {Multiphase image segmentation via equally distanced multiple well potential},
  VOLUME = {25},
}

@ARTICLE{KanSanYip11,
  AUTHOR = {Kang, S. H. and Sandberg, B. and Yip, A. M.},
  DATE = {2011},
  DOI = {10.3934/ipi.2011.5.407},
  ISSN = {1930-8337},
  JOURNALTITLE = {Inverse Problems {\&} Imaging},
  NUMBER = {2},
  PAGES = {407--429},
  SHORTJOURNAL = {Inverse Probl. Imaging},
  TITLE = {A regularized k-means and multiphase scale segmentation},
  VOLUME = {5},
}
}
\fi

\printindex

\end{document}